\documentclass[a4paper,10pt]{article}
\usepackage{amsthm}
\usepackage{amsmath}
\usepackage{amsxtra}
\usepackage{amssymb}
\usepackage{comment}
\usepackage{enumitem}
\usepackage{accents}
\usepackage{mathrsfs}
\usepackage{turnstile}
\usepackage{thmtools}
\usepackage{makeidx}
\usepackage{IEEEtrantools}
\usepackage[colorlinks=true,linkcolor=blue]{hyperref}
\usepackage{tikz}
\usepackage{geometry}
\usepackage{url}

 \newcommand{\PWO}{\mathrm{PWO}}

\newcommand{\CC}{\mathbb C}

\newcommand{\RR}{\mathbb R}

\newcommand{\BB}{\mathbb B}
\newcommand{\sub}{\subseteq}
\newcommand{\cross}{\times}
\newcommand{\all}{\forall}
\newcommand{\ex}{\exists}

\newcommand{\inter}{\cap}
\renewcommand{\int}{\inter}

\newcommand{\om}{\omega}
\newcommand{\pow}{\mathcal{P}}
\newcommand{\OR}{\mathrm{OR}}

\newcommand{\Card}{\mathrm{Card}}

\newcommand{\Hull}{\mathrm{Hull}}

\newcommand{\cut}{\backslash}

\newcommand{\Tt}{\mathcal{T}}
\newcommand{\Ss}{\mathcal{S}}
\newcommand{\Uu}{\mathcal{U}}
\newcommand{\Vv}{\mathcal{V}}
\newcommand{\Ww}{\mathcal{W}}
\newcommand{\Ll}{\mathcal{L}}

\newcommand{\Uubar}{{\bar{\Uu}}}
\newcommand{\Ttbar}{{\bar{\Tt}}}

\newcommand{\rg}{\mathrm{rg}}
\newcommand{\dom}{\mathrm{dom}}

\newcommand{\ins}{\trianglelefteq}
\newcommand{\nins}{\ntrianglelefteq}
\newcommand{\pins}{\triangleleft}
\newcommand{\npins}{\ntriangleleft}
\newcommand{\crit}{\mathrm{cr}}

\newcommand{\union}{\cup}
\newcommand{\rest}{\!\upharpoonright\!}
\newcommand{\com}{\circ}

\newcommand{\range}{\rg}

\newcommand{\lh}{\mathrm{lh}}
\newcommand{\Ult}{\mathrm{Ult}}

\newcommand{\sats}{\models}
\newcommand{\elem}{\preccurlyeq}
\newcommand{\J}{\mathcal{J}}

\newcommand{\HC}{\mathrm{HC}}
\newcommand{\ZFC}{\mathsf{ZFC}}
\newcommand{\ZF}{\mathsf{ZF}}

\newcommand{\es}{\mathbb{E}}

\newcommand{\mubar}{{\bar{\mu}}}

\newcommand{\betabar}{{\bar{\beta}}}

\newcommand{\etabar}{{\bar{\eta}}}

\newcommand{\eps}{\varepsilon}

\newcommand{\Qbar}{{\bar{Q}}}
\newcommand{\Ubar}{{\bar{U}}}

\newcommand{\Ttvec}{{\vec{\Tt}}}

\newcommand{\Nbar}{{\bar{N}}}

\newcommand{\core}{\mathfrak{C}}

\newcommand{\pred}{\mathrm{pred}}

\newcommand{\dirlim}{\mathrm{dir lim}}
\newcommand{\un}{\union}

\newcommand{\id}{\mathrm{id}}

\newcommand{\sq}{\mathrm{sq}}
\newcommand{\nth}{{\textrm{th}}}
\newcommand{\conc}{\ \widehat{\ }\ }

\newcommand{\forces}{\dststile{}{}}
\newcommand{\bfPi}{\undertilde{\Pi}}
\newcommand{\bfSigma}{\undertilde{\Sigma}}
\newcommand{\bfDelta}{\undertilde{\Delta}}
\newcommand{\rSigma}{\mathrm{r}\Sigma}

\newcommand{\rPi}{\mathrm{r}\Pi}

\newcommand{\rDelta}{\mathrm{r}\Delta}

\newcommand{\supp}{\mathbb{S}}

\newcommand{\Baire}{{^\om}\om}

\DeclareMathOperator{\Th}{Th}

\DeclareMathOperator{\card}{card}
\DeclareMathOperator{\cof}{cof}
\DeclareMathOperator{\wfp}{wfp}

\newcommand{\xvec}{\vec{x}}

\newcommand{\bfrSigma}{\undertilde{\rSigma}}

\newcommand{\psub}{\subsetneq}

\newcommand{\Yy}{\mathcal{Y}}

\newcommand{\Xx}{\mathcal{X}}

\newcommand{\alphavec}{\vec{\alpha}}
\newcommand{\cHull}{\mathrm{cHull}}

\newcommand{\Avec}{\vec{A}}
\newcommand{\DD}{\mathbb{D}}
\newcommand{\unsq}{\mathrm{unsq}}

\newcommand{\lpole}{\left\lfloor}
\newcommand{\rpole}{\right\rfloor}

\newcommand{\univ}[1]{\lpole #1\rpole}

\newcommand{\tu}{\textup}
\newcommand{\myurl}{https://sites.google.com/site/schlutzenberg/home-1/research/papers-and-preprints}

\swapnumbers

\declaretheoremstyle[bodyfont=\it]{slanted}
\declaretheoremstyle[bodyfont=\normalfont]{normal}

\declaretheorem[name=Definition,style=normal,qed=$\dashv$,
numberwithin=section]{dfn}
\declaretheorem[name=Definition,style=normal,numbered=no,qed=$\dashv$]{dfn*}
\declaretheorem[name=Definition,style=normal,numbered=no]{dfnnoqed*}
\declaretheorem[name=Example,style=definition,sibling=dfn]{exm}
\declaretheorem[name=Theorem,style=slanted,sibling=dfn]{tm}
\declaretheorem[name=Theorem,style=slanted,numbered=no]{tm*}

\declaretheorem[name=Lemma,style=slanted,sibling=dfn]{lem}

\declaretheorem[name=Korollar,style=slanted,numbered=no]{kor*}
\declaretheorem[name=Corollary,style=slanted,sibling=dfn]{cor}
\declaretheorem[name=Corollary,style=slanted,numbered=no]{cor*}
\declaretheorem[name=Remark,style=definition,sibling=dfn]{rem}

\declaretheorem[name=Fact,style=definition,sibling=dfn]{fact}

\swapnumbers
\declaretheoremstyle[headfont=\scshape]{claimstyle}
\declaretheorem[name=Claim,style=claimstyle]{clm}

\declaretheorem[name=Claim,style=claimstyle,numbered=no]{clm*}
\declaretheorem[name=Subclaim,style=claimstyle,numberwithin=clm]{sclm}

\declaretheorem[name=Subclaim,style=claimstyle,numbered=no]{sclm*}
\declaretheorem[name=Subsubclaim,style=claimstyle,numberwithin=sclm]{ssclm}

\declaretheorem[name=Subsubclaim,style=claimstyle,numbered=no]{ssclm*}

\declaretheoremstyle[headfont=\scshape]{casestyle}
\declaretheorem[name=Assumption,style=casestyle]{ass}

\declaretheorem[name=Case,style=casestyle]{case}

\declaretheorem[name=Subcase,style=casestyle,numberwithin=case]{scase}

\declaretheorem[name=Subsubcase,style=casestyle,numberwithin=scase]{sscase}

\newcommand{\rpm}{\mathrm{rpm}}
\newcommand{\Mmm}{\mathscr{M}}
\newcommand{\eqdef}{=_{\mathrm{def}}}

\newcommand{\lgcd}{\mathrm{lgcd}}

\renewcommand{\pm}{\mathrm{pm}}

\newcommand{\passive}{\mathrm{pv}}

\renewcommand{\card}{\mathrm{card}}
\newcommand{\pvec}{\vec{p}}

\newcommand{\params}{\mathrm{par}}
\newcommand{\dis}{\mathrm{dis}}

\newcommand{\prodstage}{\mathrm{prod}{\cdot}\mathrm{stage}}
\newcommand{\lifttree}{\mathrm{lift}}
\newcommand{\lex}{\mathrm{lex}}
\newcommand{\dcd}{\mathrm{dcd}}
\newcommand{\exit}{\mathrm{exit}}
\renewcommand{\deg}{\mathrm{deg}}
\newcommand{\omdeg}{\om\text{-}\deg}

\newcommand{\dropset}{\mathscr{D}}

\newcommand{\FF}{\mathbb{F}}

\newcommand{\bvec}{\vec{b}}
\newcommand{\tvec}{\vec{t}}

\newcommand{\udotvec}{\vec{\dot{u}}}
\newcommand{\avec}{\vec{a}}
\newcommand{\dec}{\mathrm{dec}}
\newcommand{\undec}{\mathrm{undec}}
\newcommand{\Bvec}{\vec{B}}

\newcommand{\uev}{\mathrm{uev}}
\newcommand{\vecdotx}{\vec{\dot{x}}}

\newcommand{\copymap}{\mathrm{copy}{\cdot}\mathrm{map}}

\newcommand{\rep}{\mathrm{rep}}

\newcommand{\simple}{\mathrm{s}}
\newcommand{\extcopy}{\mathrm{copy}}
\newcommand{\Rr}{\mathcal{R}}
\newcommand{\frakL}{\mathfrak{L}}
\newcommand{\pre}{\mathrm{pre}}

\renewcommand{\prod}{\mathrm{prod}}

\newcommand{\res}{\mathrm{res}}
\newcommand{\copyseg}{\mathrm{copy}}
\newcommand{\Scale}{\mathrm{Scale}}
\newcommand{\Lim}{\mathrm{Lim}}
\newcommand{\stack}{\mathrm{stack}}
\newcommand{\kappavec}{\vec{\kappa}}
\newcommand{\emdd}{\mathrm{edd}}
\newcommand{\redd}{\mathrm{redd}}
\newcommand{\segs}{\mathrm{segs}}
\newcommand{\prodseg}{\mathrm{prod}{\cdot}\mathrm{seg}}
\newcommand{\cardprojpropsegs}{\mathrm{cpp}{\cdot}\mathrm{segs}}
\newcommand{\modres}{\mathrm{mod}{\cdot}\mathrm{res}}
\newcommand{\segdegs}{\mathrm{seg}{\cdot}\mathrm{degs}}
\newcommand{\projdeg}{\mathrm{proj}{\cdot}\mathrm{deg}}
\newcommand{\restree}{\mathrm{res}{\cdot}\mathrm{tree}}
\newcommand{\modresmap}{\mathrm{mod}{\cdot}\mathrm{res}{\cdot}\mathrm{map}}
\newcommand{\resprodstage}{\mathrm{res}{\cdot}\mathrm{prod}{\cdot}\mathrm{stage}}
\newcommand{\resmap}{\mathrm{res}{\cdot}\mathrm{map}}
\newcommand{\resl}{\mathrm{res}{\cdot}\mathrm{l}}
\newcommand{\psegdeg}{\mathrm{p}{\cdot}\mathrm{seg}{\cdot}\mathrm{deg}}
\newcommand{\phalroot}{\mathrm{root}}
\newcommand{\exitcopymap}{\mathrm{exit}{\cdot}\mathrm{copy}{\cdot}\mathrm{map}}
\newcommand{\exitcopyseg}{\mathrm{exit}{\cdot}\mathrm{copy}}
\newcommand{\copymapredd}{\mathrm{copy}{\cdot}\mathrm{map}{\cdot}\mathrm{redd}}
\newcommand{\copyredd}{\mathrm{copy}{\cdot}\mathrm{redd}}

\newcommand{\exitresadd}{\mathrm{exit}{\cdot}\mathrm{res}{\cdot}\mathrm{add}}
\newcommand{\wt}{\widetilde}
\newcommand{\critres}{\mathrm{crit}{\cdot}\mathrm{res}}
\newcommand{\critresmap}{\mathrm{crit}{\cdot}\mathrm{res}{\cdot}\mathrm{map}}
\newcommand{\critresprodstage}{\mathrm{crit}{\cdot}\mathrm{res}{\cdot}\mathrm{prod}{\cdot}\mathrm{stage}}
\newcommand{\critresl}{\mathrm{crit}{\cdot}\mathrm{res}{\cdot}\mathrm{l}}
\newcommand{\critrestree}{\mathrm{crit}{\cdot}\mathrm{res}{\cdot}\mathrm{tree}}
\newcommand{\abrestree}{\mathrm{ab}{\cdot}\mathrm{res}{\cdot}\mathrm{tree}}
\newcommand{\abresl}{\mathrm{ab}{\cdot}\mathrm{res}{\cdot}\mathrm{l}}
\newcommand{\abresprodstage}{\mathrm{ab}{\cdot}\mathrm{res}{\cdot}\mathrm{prod}{\cdot}\mathrm{stage}}
\newcommand{\abresprodseg}{\mathrm{ab}{\cdot}\mathrm{res}{\cdot}\mathrm{prod}{\cdot}\mathrm{seg}}
\newcommand{\abdirectprodseg}{\mathrm{ab}{\cdot}\mathrm{lift}{\cdot}\mathrm{prod}{\cdot}\mathrm{seg}}

\newcommand{\abliftprodstage}
{\mathrm{ab}{\cdot}{\mathrm{lift}}{\cdot}\mathrm{prod}{\cdot}\mathrm{stage}}

\newcommand{\seq}{{{^{<\om}}\om}}
\newcommand{\Det}{\mathrm{Det}}

\newcommand{\sse}{\mathrm{sse}}
\newcommand{\tpcopy}{\mathrm{tpcopy}}

\newcommand{\Yback}{{P}}
\newcommand{\Zback}{{Q}}

\newcommand{\successor}{\mathrm{succ}}
\newcommand{\liftdom}{\mathrm{liftdom}}
\newcommand{\wtpi}{{\widetilde{\pi}}}
\newcommand{\wtalpha}{{\widetilde{\alpha}}}
\newcommand{\wtpsi}{{\widetilde{\psi}}}
\newcommand{\abres}{\mathrm{ar}}
\newcommand{\uvec}{\vec{u}}
\newcommand{\xibar}{\bar{\xi}}
\newcommand{\betavec}{\vec{\beta}}
\newcommand{\eenum}{\end{enumerate}}
\newcommand{\benum}{\begin{enumerate}}
\newcommand{\benumdd}{\benum[label=--]}

\newcommand{\dual}[1]{
  \mathrel{\vbox{\offinterlineskip\ialign{
    \hfil##\hfil\cr
    $\scriptscriptstyle\smile$\cr
    \noalign{\kern0.1ex}
    $#1$\cr
}}}}
\newcommand{\widedual}[1]{
  \mathrel{\vbox{\offinterlineskip\ialign{
    \hfil##\hfil\cr
    $\scriptscriptstyle\longsmile$\cr
    \noalign{\kern0.1ex}
    $#1$\cr
}}}}

\makeindex
\title{Mouse scales (preliminary draft v3)}
\author{Farmer Schlutzenberg\footnote{Gef\"ordert durch die Deutsche Forschungsgemeinschaft (DFG) im Rahmen der Exzellenzstrategie des Bundes und der L\"ander EXC 2044--390685587, Mathematik M\"unster: Dynamik--Geometrie--Struktur.
Gef\"ordert durch die Deutsche Forschungsgemeinschaft (DFG) -- Projektnummer 445387776.
Funded by the Deutsche Forschungsgemeinschaft (DFG, German Research Foundation) -- project number 445387776.
This work was funded in part by the Austrian Science Fund (FWF) [grant Y 1498].
}\footnote{afirstname dot a lastname at  tuwien dot ac dot at}\\
TU Wien}
\begin{document}
\maketitle

\begin{abstract}
 We give a construction of scales
 (in the descriptive set theoretic sense)
 directly from mouse existence hypotheses,
 without using any determinacy arguments.  The construction is related to the Martin-Solovay construction for scales on $\Pi^1_2$ sets. The prewellorders of the scales compare reals $x,y$ by comparing features of certain kinds of fully backgrounded $L[\es,x]$- and $L[\es,y]$-constructions executed in mice $P$ with $x,y\in P$. In this way we reach an new proof of the scale property for many pointclasses, for which the scale property was classically established using determinacy arguments (for example, $\Pi^1_3$). Moreover, it also yields many further pointclasses with the scale property, for example intermediate between $\Pi^1_{2n+1}$ and $\Sigma^1_{2n+2}$, and also instances of complexity well beyond projective.
 
 The reader should take note of the following footnote:\footnote{This is an early draft of this paper. Although it contains essentially the entire construction and almost all the proofs, it is not yet in a form which is fully ready for reading. At certain points in the current draft, there are  missing components and inconsistencies in setup. These issues presently occur mainly within \S\S\ref{sec:iterability_M_infty},\ref{sec:lower_semi}. In particular, the proof of lower semicontinuity is only given for a special case (however, the general case is an elaboration of that one, incorporating methods used elsewhere in the paper, which takes some work, but most of the ideas are already contained in the current version in some form). And the section on iterability of $M_\infty$ was mostly written 
 some time ago, prior to changes being made elsewhere, and probably needs updating. Moreover,
 the paper needs general proof-reading and cleaning up, so there are probably also
 errors/omissions/ill-definedness to be sorted out throughout.
 The author plans to improve these things in the reasonably near future. Although there will be changes, these should be at the detailed level;  the essential ideas are expected to be stable. (The material in \S\S\ref{sec:Q-mice},\ref{sec:P-construction} has been significantly improved over an earlier version.)
 
 The reason the draft is being made available at this early stage is in order to make some material available to the talk participants for a talk I am giving on the topic. I had hoped to have a better version available by this time, but it seems that making an imperfect version available is better than nothing.}
\end{abstract}

\tableofcontents

\section{Introduction}
\label{sec:introduction}
Let $X\sub\RR^2$ be a subset of the plane.
Suppose we would like to assign to each $x\in \RR$ such that
\begin{equation}\label{eqn:exists_y_(x,y)_in_X} \exists y\ [(x,y)\in X] \end{equation}
some specific $y\in\RR$ such that
\[ (x,y)\in X.\]
By the Axiom of Choice (AC), there is some function $f$ which achieves this; i.e.
\[ (x,f(x))\in X \]
for each $x$ as in line (\ref{eqn:exists_y_(x,y)_in_X}). Assuming that also the domain of $f$ is exactly
the set of all such $x$, then $f$ is called a \emph{uniformization} of $X$.
So AC guarantees that  uniformizations exist. But often appeals to AC can be avoided, by supplying some concrete or explicit procedure for selecting a particular $y$ for each $x$.
AC itself does not guarantee
that there is such an explicit procedure.\footnote{AC does ensure the existence of a uniformization $f$, and $f$ itself can be viewed as some sort of abstract procedure for selecting $y$s for $x$s. But we want the procedure to be (intuitively) concrete (say, explicitly specified or defined in some way), and AC does not say anything about this. Probably because of this intuitively non-constructive nature, AC is often viewed as less obviously valid than the other axioms of set theory. However, it is included in the mostly commonly accepted axioms for mathematics (the ZFC axioms).}
Relatedly, we might want a uniformization $f$
which is definable from the given set $X$, or a uniformization 
which is somehow of optimal complexity.

For example, a set $Z\sub\RR^2$ 
is called \emph{coanalytic} or $\bfPi^1_1$
if there is a Borel set $Y\sub\RR^3$ such that for all $x,y\in\RR$, we have
\[ (x,y)\in Z\ \iff\ \all z\in\RR\ [(x,y,z)\in Y].\]
It was shown by Lusin and Sierpinski \cite{lusin_1930a}, \cite{sierpinski_1930} that every Borel set $X\sub\RR^2$ can be uniformized by some $\bfPi^1_1$ set $f\sub X$.
It was asked by Sierpinski
whether
every $\bfPi^1_1$ set $X\sub\RR^2$
can also be uniformized by a $\bfPi^1_1$ set $f\sub X$,
and this became a key open question at the time.
Eventually Kond\^o \cite{kondo_1939} resolved it affirmatively, improving on a partial result in this direction due to Novikoff \cite{lusin_novikoff_1935}.
We will include a proof of this fact in  \S\ref{subsec:Pi^1_1}. (Kond\^o's result is also optimal in this case.)

The $\bfPi^1_1$ subsets of $\RR^n$ are defined similarly, starting from a Borel subset of $\RR^{n+1}$.
A $\bfSigma^1_1$ (also called \emph{analytic}) subset of $\RR^n$ is the complement of a $\bfPi^1_1$ subset of $\RR^n$; equivalently, $X\sub\RR^n$ is $\bfSigma^1_1$ iff there is a Borel set $Y\sub\RR^{n+1}$ such that for all $\vec{x}\in\RR^n$, we have
\[ \vec{x}\in X\iff\exists y\in\RR\ [(\vec{x},y)\in Y].\]

In contrast to the case with $\bfPi^1_1$,
it was shown in \cite{novikoff_1931} that there are $\bfSigma^1_1$ sets $X\sub\RR^2$
for which there is \emph{no} $\bfSigma^1_1$ uniformization (in fact there are \emph{closed} subsets of $\RR^2$ for which there is no $\bfSigma^1_1$ uniformization).

We say here the the pointclass $\bfPi^1_1$
(that is, the collection of all $\bfPi^1_1$ sets) has the \emph{uniformization property}, whereas $\bfSigma^1_1$ does not.
These two results are theorems of ZFC.

One defines the $\bfSigma^1_{k+1}$ and $\bfPi^1_{k+1}$ subsets of  $\RR^n$, for $k\geq 1$, by recursion on $k$,
as follows: A set $X\sub\RR^n$ is $\bfSigma^1_{k+1}$ iff there is a $\bfPi^1_k$ set $Y\sub\RR^{n+1}$ such that
\[ \vec{x}\in X\iff\exists y\in\RR\ [(\vec{x},y)\in Y],\]
and $X$ is $\bfPi^1_{k+1}$
iff it is the complement of a $\bfSigma^1_{k+1}$ set, or equivalently, iff there is a $\bfSigma^1_k$ set $Y\sub\RR^{n+1}$ such that
\[ \vec{x}\in X\iff\all y\in \RR\ [(\vec{x},y)\in Y].\]
It follows that if $m\geq 1$ is odd then
$X\sub\RR^n$ is $\bfSigma^1_m$
iff there is a Borel set $Z\sub\RR^{n+m}$
such that
\[ \vec{x}\in X\iff\exists y_1\ \all y_2\ \exists y_3\ \all y_4\ \ldots\ \exists y_{m}\ [(\vec{x},\vec{y})\in Y] \]
where $\vec{y}=(y_1,\ldots,y_m)$ and all quantifiers above range over $\RR$.
And if $m>1$ is even it is likewise,
but the last quantifier is ``$\all y_m$''
instead of ``$\exists y_m$''.
Since $\bfPi^1_m$ sets are complements of $\bfSigma^1_m$ sets,  if $m\geq 1$ is odd then $X\sub\RR^n$
is $\bfPi^1_m$ iff for some Borel $Y\sub\RR^{n+m}$,
\[ \vec{x}\in X\iff\all y_1\ \exists y_2\ \all y_3\ \exists y_4\ \ldots\ \all y_m\ [(\vec{x},\vec{y})\in Y],\]
and if $m$ is even then we change the last quantifier to ``$\exists y_m$''.

(One also defines $\bfDelta^1_m=\bfSigma^1_m\cap\bfPi^1_m$. That is,
a set $A\sub\RR^n$ is $\bfDelta^1_m$ iff it is
both $\bfSigma^1_m$ and $\bfPi^1_m$.)

So which of the pointclasses $\bfSigma^1_m$ or $\bfPi^1_m$ have the uniformization property? As an easy consequence of Kond\^o's result that $\bfPi^1_1$ has the uniformization property, 
$\bfSigma^1_2$ also  has it. But ZFC proves that for each $n\geq 1$, at most one of $\bfSigma^1_n$ and $\bfPi^1_n$ can have the uniformization property (see Fact \ref{fact:Gamma_and_dual_not_both_unif} for the version with $\RR$ replaced by Baire space), so $\bfPi^1_2$ does not.

What about $\bfPi^1_3$ or $\bfSigma^1_3$?
It turns out that which one of the two has the uniformization property, if either at all, is independent of ZFC.

We write $V$ for the class of all sets (and we assume that $V$ models ZFC). A \emph{proper class transitive inner model of ZF} is a class $N\sub V$ such that:
\begin{enumerate}[label=--]
 \item $N$ models ZF (here ZF denotes the Zermelo-Fr\"ankel axioms, which are the usual axioms of set theory excluding AC),
 \item $N$ is transitive (that is,
 for each $x\in N$ and each $y\in x$, we have $x\in N$), and
 \item every ordinal is in $N$.
\end{enumerate}

G\"odel's constructible universe
$L$ is the minimal 
proper class transitive inner model of ZF: that is, $L$ is such an inner model, $L\sub N$ for every such inner model $N$. 
Each set in $L$ has a very explicit description in terms of some ordinal.
Moreover, $L$ models ZFC (in particular, AC),
GCH (the Generalized Continuum Hypothesis),
and the statement ``$\bfSigma^1_n$ has the uniformization property  for each $n\geq 2$''. In particular, ZFC does not disprove the uniformization property for $\bfSigma^1_3$ (at least, unless ZFC is inconsistent).
$L$ has many other  nice properties, and is a highly canonical model of ZFC. Given this, one might presume that $\bfSigma^1_3$ ``should'' have the uniformization property, and hence $\bfPi^1_3$ not.

In contrast, other natural set theoretic hypotheses have been discovered, which actually yield the uniformization property for $\bfPi^1_3$.
\emph{Large cardinal axioms} posit the existence of infinite sets with strong reflection properties.
For example, \emph{inaccessible cardinals}, which are low in the large cardinal hierarchy, are
cardinals $\kappa$ such that:
\begin{enumerate}[label=--]
 \item for every ordinal $\alpha<\kappa$, the power set $\pow(\alpha)$ has cardinality $<\kappa$, and
 \item for every ordinal $\lambda<\kappa$ and function $f:\lambda\to\kappa$, the range of $f$ is bounded in $\kappa$.
\end{enumerate}
The set theoretic universe $V$ is stratified by the cumulative hierarchy  $\left<V_\alpha\right>_{\alpha\in\OR}$, defined as follows: $V_0=\emptyset$, $V_{\alpha+1}=\pow(V_\alpha)$, and $V_\lambda=\bigcup_{\alpha<\lambda}$ for limit ordinals $\lambda$.
Every set belongs to some $V_\alpha$. If $\kappa$ is inaccessible then $V_\kappa\models\ZFC$, and there are unboundedly many $\alpha<\kappa$ such that $V_\alpha\models\ZFC$.
So inaccessibility transcends ZFC in this manner.

Although large cardinals may appear distant from considerations relevant to real numbers, it turns out that they have a profound influence on the uniformization questions mentioned above.
For example, ZFC +  sufficient large cardinals proves that $\bfPi^1_3$ has the uniformization property -- although this question is not resolved by ZFC alone.
In fact it is enough to have cardinals $\delta$ and $\kappa$
such that $\delta$ is a Woodin cardinal, $\kappa$ is measurable and $\delta<\kappa$. Every measurable cardinal is inaccessible and a limit of inaccessible cardinals, and every Woodin is a limit of measurable cardinals (though need not be measurable itself).
Extending this, ZFC + ``there are infinitely many Woodin cardinals'' proves that for every $n\geq 0$, $\bfPi^1_{2n+1}$ and $\bfSigma^1_{2n+2}$ have the uniformization property (so $\bfSigma^1_{2n+1}$ and $\bfPi^1_{2n+2}$ do not).\footnote{Therefore, $L$ cannot satisfy these large cardinal axioms. But actually,
there is a much simpler and more direct proof that $L$ has no measurable cardinals, a fact due to Scott.} This phenomenon also extends well beyond the pointclasses $\bfSigma^1_n$ and $\bfPi^1_n$.

We write $\mathbb{N}$ or $\om$
for the set $\{0,1,2,\ldots\}$ of all non-negative integers. (We also formally consider the non-no negative integers themselves as ordinals, so $0=\emptyset$ and for $n\in\om$, $n+1=\{0,1,\ldots,n\}$, and $\om$ is also the least infinite ordinal.)
Although we have so far talked about $\mathbb{R}$, meaning the (conventional) real numbers,
for the purposes of descriptive set theory it is generally more convenient to consider \emph{Baire space} ${^\om}\om$, which is the topological space whose elements are the functions $f:\om\to\om$,
under the following natural topology. For $k\in\om$,
let ${^k}\om$
denote the set of all functions $f$ such that $f:k\to\om$.
Let ${^{<\om}\om}$ denote $\bigcup_{k\in\om}{^k\om}$.
For $s\in{^{<\om}\om}$, let
\[ N_s=\{f\in{^\om}\om\bigm|s\sub f\};\]
in other words, for $f\in{^\om}\om$, we have $f\in N_s$ iff $f\rest\dom(s)=s$ (where $\dom(s)$ denotes the domain of $s$).
Then Baire space has the topology generated by the basis
\[ \mathscr{B}=\{N_s\bigm|s\in{^{<\om}\om}\}.\]
Baire space is homeomorphic to the irrationals under the subspace topology induced by $\RR$.

From now on, instead of $\RR$, we will only work directly with Baire space ${^\om}\om$ and product spaces of the form $\Xx_{mn}=\om^m\cross({^\om}\om)^n$ for $m,n<\om$ with $0<m+n$, where $\om$ is endowed with the discrete topology. (Let $m,n,m',n'<\om$
with $0<m+n$ and $0<m'+n'$.
If $0<n,n'$, then $\Xx_{mn}$ is homeomorphic to $\Xx_{m'n'}$. If $0=n=n'$, then $\Xx_{mn}=\om^m$ is homeomorphic to $\Xx_{m'n'}=\om^{m'}$.)
The definitions of $\bfSigma^1_k,\bfPi^1_k$,
for $k\geq 1$, extend directly to subsets of $\Xx_{mn}$, starting with a Borel subset of $\Xx_{mn}\cross({^\om}\om)^k$ and using $k$ alternating quantifiers  ``$\exists x\in{^\om}\om$'' and ``$\all x\in {^\om}\om$''. And  $\bfDelta^1_k=\bfSigma^1_k\cap\bfPi^1_k$ as before. (If $n=0$, so $\Xx_{m0}=\om^m$,
then \emph{all} subsets of $\Xx_{m0}$ are $\bfSigma^1_k$. But if $n>0$ then the notion is of interest.) There is a finer variant of these pointclasses, $\Sigma^1_k,\Pi^1_k,\Delta^1_k$
(typeset without the ``$\undertilde{\ }$'') which are defined in \ref{dfn:Sigma^1_n}.

Now the classical proof that $\bfPi^1_3$
has the uniformization property
involves crucially the concept of \emph{determinacy} of infinite games.
Given a set $A\sub{^\om}\om$,
the game $\mathscr{G}_A$ is defined as follows. There are two players, I and II. The players alternate playing integers $x_0,x_1,x_2,\ldots$, with player I starting.
So player I plays $x_n$ for all even $n$, and player II plays $x_n$ for all odd $n$. The game runs for infinitely many rounds, producing the sequence $\left<x_n\right>_{n\in\om}$,
which we identify with the function $f:\om\to\om$ where $f(n)=x_n$.
Given this run of the game (producing this $f$),
we declare that player I has won the run iff $f\in A$.

A \emph{strategy} is a function $\sigma:{^{<\om}\om}\to\om$. A function $f\in{^\om}\om$ is \emph{according to $\sigma$ for player I} iff for all even $n\in\om$, we have
\[ f(n)=\sigma(f\rest n).\]
Likewise \emph{for player II},
but replacing ``even'' with ``odd''. We say that $\sigma$ is \emph{winning for player I}
(with respect to $A$)
iff $f\in A$ for every $f$ which is according to $\sigma$ for player I. Likewise \emph{winning for player II}.
We say that $A$ (or $\mathscr{G}_A$) is \emph{determined} iff there is a winning strategy for either player I or player II.

Now using a wellorder of $\RR$, one can construct (in ZFC) sets $A\sub{^\om}\om$ which are not determined. But what about simply definable sets?
It was shown by Gale and Stewart \cite{gale_stewart_1953} that every open $A\sub{^\om}\om$ is determined. A famous result of Martin \cite{martin_borel_det} strengthened this, showing in ZFC that all Borel sets $A\sub{^\om}\om$ are determined. But this falls short of the $\bfSigma^1_1$ or the $\bfPi^1_1$ sets. (In fact, Lusin \cite{lusin_1927} showed that a set is Borel iff it is both $\bfSigma^1_1$ and $\bfPi^1_1$.) Martin also showed that if there is a measurable cardinal then all $\bfSigma^1_1$, and all $\bfPi^1_1$, subsets of ${^\om}\om$ are determined (these complexity classes are defined completely analogously those for subsets of $\RR$). Martin and Steel \cite{projdet} later showed that if there are $n$ Woodin cardinals and a measurable cardinal above them, then all $\bfSigma^1_{n+1}$ and $\bfPi^1_{n+1}$ sets are determined.
This phenomenon is now known to extend much further into the large cardinal and determinacy hierarchies.

Now the classical proof of the uniformization property for (for example)  $\bfPi^1_3$, assuming sufficient large cardinals,
 breaks into two natural and quite separate pieces:
\begin{enumerate}[label=(\roman*)]
 \item\label{item:LC_to_det} By the Martin-Steel theorem mentioned above, all $\bfSigma^1_2$ (and $\bfPi^1_2$) sets are determined (assuming a Woodin cardinal and a measurable above).
 \item\label{item:unif_from_det} Uniformization for $\bfPi^1_3$ follows from  the determinacy of all $\bfSigma^1_2$ sets (see \cite{mosch}).
\end{enumerate}
There is an analogous structure for $\bfPi^1_{2n+1}$ in general, and for many higher pointclasses.

In this paper, we will find another route to proving the uniformization property for (for example) $\bfPi^1_3$ (amongst many other pointclasses), without considering determinacy, but instead, finding another path from large cardinals to uniformization. 
A key tool used considered in part \ref{item:unif_from_det}
above is that of a \emph{scale},
and we will also employ scales
in our proof, but they will be constructed in a manner
different to the classical one.

There are in fact three central methods that have been used for constructing scales: 
 (a) using winning strategies for games (including \ref{item:unif_from_det} above), 
(b) the Martin-Solovay construction, and (c) directed systems of mice. \footnote{A \emph{mouse} is a particular kind of inner model, generalizing G\"odel's constructible universe $L$, but capable of exhibiting many more large cardinals than is $L$. In this paper we deal extensively with mice. Giving the precise definition is, however, not within the scope of this work; one should refer to \cite{outline} and \cite{fsit} for a general introduction, though the precise formulation of \emph{mouse} that we use is explained in \S\ref{sec:notation}.}
These take place under appropriate determinacy/large cardinal hypotheses.
We will describe
a method of scale construction
which is essentially a generalization
of the Martin-Solovay construction to mice of higher complexity (the role of sharps in the Martin-Solovay construction is replaced with higher mice). We deal exclusively with mice and iteration strategies for them; no determinacy arguments are used. We will in particular include a proof
of the scale properties for $\Pi^1_{2n+1}$ and $\Sigma^1_{2n+2}$, but also for many higher pointclasses, and also for many pointclasses at intermediate levels of the projective hierarchy (defined in terms of mice).

Aside from  the Martin-Solovay construction, 
the scale construction we will give was also motivated by an observation credited in \cite[between Theorems 3.5 and 3.6]{steel_games_and_scales} to Woodin, that (under appropriate mouse existence hypotheses) one can prove the prewellordering property for $\Pi^1_3$ directly with the methods of inner model theory.

As an example of the kinds of pointclasses
for which we will define scales, suppose that for each real $x$ there is an $(0,\om_1+1)$-iterable active premouse over $x$ satisfying ``there is a Woodin limit of Woodin cardinals''. Let $M_x$ denote the least such sound premouse. For a $\Sigma_1$ formula $\psi$ in the language of real premice (this is the usual language of premice augmented with a constant symbol $\dot{x}$ interpreted as the base real $x$, and recall that the language includes symbols which refer to the internal extender sequence and the active extender) let $A_\psi=\{x\in\RR\bigm|M_x\sats\psi\}$. Define the pointclass
\[ \Gamma=\{A_\psi\bigm|\psi\text{ is as above}\}.\]
We will show that $\Gamma$ has the scale property. This particular choice of $\Gamma$
is quite arbitrary; there are many such examples.

In order to better motivate the ideas,
we will begin by proving some of the classical facts, using the methods which we will generalize. This will include the scale property for $\Pi^1_1$ in \S\ref{subsec:Pi^1_1},
that $\Pi^1_2$ sets have $\Delta^1_3$ scales in  \S\ref{sec:MS}, 
and the prewellordering property for $\Pi^1_3$ in \S\ref{sec:Pi^1_3}.
The scale calculations here are  thinly
veiled versions of the standard proofs,
but in a somewhat different language than usual. That for the prewellordering property for $\Pi^1_3$ is certainly different from the standard proof; it is an argument directly from mice, instead of from determinacy (the proof we give was found by the author; we do not know how close it is to Woodin's original argument). Although this beginning does not produce new results, it seems to be the best expository path, since it provides a reasonable intuition for the kind of direction we will want to head in. After that introduction, we have to spend a fair while developing some fine structure and notation, etc, for background constructions, before we can really begin the scale construction in \S\ref{sec:the_scale}.

The basic ideas for the definition of the scale (in \S\ref{sec:the_scale}), and the proof that $M_\infty$
is iterable (in \S\ref{sec:iterability_M_infty})
were found by the author in 2010/11.\footnote{Early on, the author thought that he could prove the semiscale property for $\Pi^1_3$,
and announced a talk on this topic for the \emph{2nd conference on
the core model induction and hod mice} in M\"unster, Germany in 2011. However,
shortly after announcing the talk, he noticed that
he did not know how to show that the norms of the putative scale were in fact norms, and thus, changed the topic of the talk, although as of this writing, the conference website \url{https://ivv5hpp.uni-muenster.de/u/rds/core_model_induction_and_hod_mice_2.html} has retained the original talk title.} But at that time, the author did not see how to resolve the remaining issues,
and stopped thinking further about the problem until early 2017. The proof in \S\ref{sec:lifting_norms_are_norms}, that the putative norms of the scale are in fact norms,
came early-mid 2017.
Lower semicontinuity was somewhat more troublesome. The author spent significant
time attempting, without success, to more or less directly adapt the realization method of the proof for $\Pi^1_2/\Delta^1_3$ (given in \S\ref{sec:MS}) to higher levels.
The key idea which enables something like that to be done, and which we use in \S\ref{sec:lower_semi}, came in mid 2017, and the remaining details were essentially sorted out by about mid 2018.

A scale construction
directly from mice was in fact announced by Itay Neeman some time around the year 2000,
as alluded to in \cite[between Theorems 3.5, 3.6]{steel_games_and_scales}.
In 2010/11, after working out the basic ideas for  the definition of the scales we give in \S\ref{sec:the_scale}
and the proof of iterability
of $M_\infty$ in \S\ref{sec:iterability_M_infty},
the author communicated with Neeman on the topic, and Neeman then shared 
his unpublished notes \cite{neeman_lightening} and slides \cite{neeman_talk}
 with the author. Neeman's scale definition makes use of his \emph{lightening} construction
(a fine structural variant of $L[\es,x]$-construction);
his notes \cite{neeman_lightening} describe the
lightening construction
and develop some of its basic properties,
but do not discuss scales or prewellorders.
The slides \cite{neeman_talk}
contain a brief outline of Neeman's lightening construction and scale definition.
Based on the slides, one can see some initial
similarities between Neeman's  construction and the one we present here.
Both define prewellorders by comparing reals via comparing features of background constructions (in Neeman's case, his lightening construction) over those reals in appropriate background mice. The depth 1 ``theory prewellorders''  and ``ordinal prewellorders'' used here
have clear analogues in \cite{neeman_talk}. However,
the higher depth  prewellorders which we use here
(involving non-trivial finite iteration trees)
do not seem to have an analogue in Neeman's approach, at least
not one that is mentioned in \cite{neeman_talk}. Also, the prewellorders
we define relating to indices of \emph{background} extenders (which represent much tighter connection to the background universe, and do not seem to be specified by Skolem terms over the mice being constructed) do not seem to have any clear analogue in Neeman's slides, and neither do other similar kinds of norms that we employ.
Incorporating these kinds of information into our prewellorders is essential 
in our proof, even in order to verify that we get a semiscale.
Further, the methods of proof we use here focus on comparison via iteration trees, and this does not seem to feature in Neeman's approach.
However, none of Neeman's work on the topic 
has  been published, and \cite{neeman_talk} is very sparse in detail, and as mentioned, \cite{neeman_lightening}  only deals with the lightening construction, so a meaningful comparison of the two approaches is difficult to make.

\subsection{Background: the prewellordering and scale properties}\label{sec:background}
In this section we review some background on scales and prewellorders. The material here is all standard and well known; for more including attribution see \cite{mosch}, \cite{kechris_cdst}.
\begin{dfn}
The set of all functions $f:X\to Y$ is denoted $^XY$, and $^{<\om}Y$ denotes $\bigcup_{n<\om}{^n}Y$. Let $s,t\in{^{<\om}}Y$.
Then the \emph{length} $\lh(s)$ of $s$ is just the domain of $s$. Note that $s\sub t$ iff $s=t\rest\lh(s)$.

For $s\in{^{<\om}\om}$ let
\[ N_s=\{f\in{^\om}\om\bigm|s\sub f\}.\]
 \emph{Baire space} 
 is the topological space on the set ${^\om}\om$ with  topology generated by $\{N_s\bigm|s\in{^{<\om}\om}\}$.
In this paper, a \emph{real}
is formally an element of Baire space.
\end{dfn}

\begin{dfn}
 Let $X$ be a set.
 A \emph{tree} on $X$
 is a set $T\sub {^{<\om}}X$
 which is closed under initial segment.
 
  Let $T$ be a tree. For $s\in T$,
  $T/s$ denotes the sub-tree compatible with $s$; that is,
  \[ T/s=\{t\in T\bigm|s\sub t\text{ or }t\sub s\}.\]
  A \emph{branch through $T$} is a sequence $f$ with domain $\om$ such that $f\rest n\in T$ for all $n<\om$.
  We write $[T]$ for the
  set of all branches through $T$.
  
 Suppose $T$ is a tree on $X\cross Y$.
 Then we identify $T$ with
 the set of pairs $(t,u)$
 of functions $t\in{^{<\om}X}$
 and $u\in{^{<\om}}Y$ such that
 $\dom(t)=\dom(u)$ and $f\in T$ where $\dom(f)=\dom(t)$ and $f(i)=(t(i),u(i))$ for all $i\in\dom(f)$. Likewise, we identify $[T]$
 with the set of pairs $(x,y)$
 such that $x\in{^\om}X$ and $y\in{^\om}Y$
 and $(x\rest n,y\rest n)\in T$ for all $n<\om$.
 The \emph{projection} $p[T]$ of $T$ is
 \[ p[T]=\Big\{x\in{^\om}X\Bigm|\exists y\in{^\om}Y\ \big((x,y)\in[T]\big)\Big\}.\] 
 Given $x\in{^\om}X$,
 $T_x$ denotes the tree on $Y$ defined by:
 \[ u\in T_x\iff (x\rest\lh(u),u)\in T.\]
It follows that $x\in p[T]$
iff $[T_x]\neq\emptyset$.
 
 Let $T$ be a tree on $X$ and $<_X$
a wellorder on $X$.
 Suppose $[T]\neq\emptyset$.
 Then the \emph{$<_X$-left-most branch}
 through $[T]$ is the branch $f\in[T]$ such that for all $s\in{^{<\om}X}$ with $s\not\sub f$, if $s(n)<_Xf(n)$ where $n$ is least such that $s(n)\neq f(n)$,
 then $T/s$ has no branch.
If $X=\om\cross\lambda$ for some ordinal $\lambda$ then
the \emph{left-most-branch}
is just the $<_{\om\cross\lambda}$-left-most-branch
where $<_{\om\cross\lambda}$ is the lexicographic order on $\om\cross\lambda$.
\end{dfn}

\begin{rem}
Let $A\sub X\cross Y$. A \emph{uniformization} of $A$
is a set $U\sub A$ such that
for all $x\in\dom(A)$ (so $x\in X$)
there is a unique $y\in Y$ such that $(x,y)\in U$ (hence $(x,y)\in A$). Equivalently, $U\sub A$ and $U$ is a function 
with $\dom(U)=\dom(A)$.

Note that the existence of a uniformization
for all such $X,Y,A$ is equivalent
to the Axiom of Choice. We will be interested in the case that
$X=Y={^\om}\om$, in which case such existence is a certain choice principle.

Suppose $A\sub{^\om}\om\cross{^\om}\om$ and $A=p[T]$
where $T$ is a tree on $\om\cross\om\cross\lambda$, for some ordinal $\lambda$.
Then we can define
a uniformization $U$ of $A$ as follows: Given $x\in\dom(A)$,
let $U(y)$ be the first component
of the left-most-branch of $T_x$.
\end{rem}

Given a set $A\sub{^\om}\om$,
a \emph{scale} on $A$ gives an indirect description of a tree $T$
projecting to $A$,
 one which is very useful and natural.

\begin{dfn}\label{dfn:prewellorder}
 Let $A\sub{^\om}\om$.
 A \emph{norm} on $A$ is a function $\pi:A\to\OR$,
 and a norm is $\pi$ is \emph{regular}
 if 
 $\rg(\pi)$ is an ordinal.
 
 A \emph{prewellorder} on $A$
 is a partial order $\leq$ on $A$
 which is total (i.e. $x\leq y\vee y\leq x$ for all $x,y\in A$)
 and whose strict part is wellfounded (so defining an equivalence relation on $A$ by $x\approx y\iff x\leq y\leq x$,
 and defining the order $\widetilde{<}$ as the order induced on the classes $[x]$
 by the strict part of $\leq$,
 then the ordertype of $\widetilde{<}$ is some ordinal).
 \end{dfn}
 
 \begin{dfn}
  Let $\pi:A\to\OR$ be a norm on $A$. We define an associated prewellorder $\leq_{\pi}$ on $A$ by
  \[ x\leq_{\pi} y\iff\pi(x)\leq\pi(y).\]
  
  Conversely, let $\leq$ be any prewellorder on $A$.
  Then we define an associated norm  $\pi_{\leq}:A\to\OR$
   by
   \[ \pi(x)=\text{ the rank of 
  }[x]\text{ in }\widetilde{<}, \] where $[x]$ and $\widetilde{<}$ are defined as in Definition \ref{dfn:prewellorder}.
 \end{dfn}

 The following lemma is easy to verify:
 \begin{lem}
Let $A\sub{^\om}\om$. Then:
 \begin{enumerate}
  \item 
For any norm $\pi$ on $A$,
  $\leq_\pi$ is a prewellorder on $A$.
  \item For any prewellorder $\leq$ on $A$,
  $\pi_{\leq}$ is a regular norm on $A$.
\item For any regular norm $\pi$ on $A$,
$\pi_{\leq_\pi}=\pi$.
\item For any prewellorder $\leq$ on $A$, $\leq_{\pi_\leq}={\leq}$.
\end{enumerate}  
 \end{lem}
 \begin{rem}
  Because of this relationship between norms, regular norms and prewellorders, and because we will only be interested in the regular version
  $\pi_{\leq_\varphi}$ of a given norm $\varphi$,
  we will identify (regular) norms with prewellorders.
 \end{rem}

\begin{dfn}
Let $A\sub{^\om}\om$.
Let $\vec{\varphi}=\left<\varphi_n\right>_{n<\om}$ be a sequence of norms on $A$.
Let $\left<x_n\right>_{n<\om}\sub A$ and let $x\in{^\om}\om$. Then we say that $\left<x_n\right>_{n<\om}$ converges to $x$ \emph{modulo $\vec{\varphi}$}, written
\[ x_n\to x \mod\vec{\varphi}, \]
iff $x_n\to x$ topologically \tu{(}that is, $\all i<\om\ \exists k<\om\ \all m\in[k,\om)
\ [x_m(i)=x(i)]$\tu{)}, and for all $j<\om$ the sequence
$\left<\varphi_j(x_n)\right>_{n<\om}$
is eventually constant.

Let $\vec{\varphi}=\left<\varphi_n\right>_{n<\om}$ be a sequence of norms on $A$. We say that
$\vec{\varphi}$ is:
\begin{enumerate}[label=--]
\item 
 a \emph{semiscale
 on $A$} iff for all sequences $\left<x_n\right>_{n<\om}\sub A$,
 if
$x_n\to x\mod\vec{\varphi}$
 then $x\in A$.
 \item \emph{lower semicontinuous on $A$}
 iff  for all sequences $\left<x_n\right>_{n<\om}\sub A$,
 if
 $x_n\to x\mod\vec{\varphi}$
 then for all $j<\om$,
 \[ \varphi_j(x)\leq\lim_{n\to\om}\varphi_j(x_n).\]
 \item a \emph{scale on $A$}
 if it is a semiscale on $A$ which is lower semicontinuous on $A$.\qedhere
 \end{enumerate}
\end{dfn}

\begin{dfn}
 Let $T$ be a tree on $\om\cross\lambda$, for some ordinal $\lambda$, and $A=p[T]$.
 We define a scale $\vec{\varphi}_T=\left<\varphi_{nT}\right>_{n<\om}$ on $A$ as follows. Given $x\in A$ let $\ell_x$ be the left-most-branch through $T_x$. Define $\varphi_{nT}$  by:
 $\varphi_{nT}(x)$ is the rank of
 \[ (x(0),\ell_x(0),x(1),\ell_x(1),\ldots,x(n-1),\ell_x(n-1)) \]
 in the lexicographic ordering on all $2n$-tuples
 \[ (y(0),\ell_y(0),\ldots,y(n-1),\ell_y(n-1)) \]
 where $y\in A$.
\end{dfn}

Conversely:
\begin{dfn}
Let $\vec{\varphi}$ be a semiscale on $A\sub{^\om}\om$. We define the tree $T_{\vec{\varphi}}$ of $\vec{\varphi}$ as follows:
\[ T=\Big\{\Big(x\rest n,\left<\varphi_i(x)\right>_{i<n}\Big)\Bigm|x\in A\text{ and }n<\om\Big\}.\qedhere`\]
\end{dfn}
A straightforward application of  the definitions gives:
\begin{fact}We have:
 \begin{enumerate}
  \item 
Let $T$ be a tree on $\om\cross\lambda$, for some ordinal $\lambda$. Then $\vec{\varphi}_T$ is a scale on $p[T]$.
\item Let $\vec{\varphi}$
be a semi-scale on $A\sub{^\om}\om$. Then $T_{\vec{\varphi}}$ is a tree with $p[T_{\vec{\varphi}}]=A$.
 \end{enumerate}
 \end{fact}
 
 Thus, one can always pass from a semiscale to a scale,
 going via the tree of the semiscale.
 However, it is important to obtain scales of optimal complexity -- this helps in obtaining uniformizations of optimal complexity, for example.
 It seems that passing from a semiscale to a scale in the above manner might increase the complexity in an undesired manner.
 \begin{dfn}
  Given a prewellorder $\leq$ on a set $A\sub{^\om}\om$, define  relations $\leq^*$ and $<^*$ on ${^\om}\om$ by:
  \[ x\leq^* y\iff x\in A\text{ and if }y\in A\text{ then }x\leq y, \]
  \[ x<^*y\iff x\in A\text{ and if }y\in A\text{ then }x\leq y\not\leq x.\qedhere\]
 \end{dfn}

 \begin{dfn}
 A \emph{standard space}
 is a space of the form $\om^m\cross({^\om}\om)^n$, for some $m,n<\om$, endowed with the product topology
 (where $\om$ is endowed with the discrete topology).
 
   A \emph{pointclass}
  is a set $\Gamma\sub\bigcup_{m,n<\om}\pow(\om^m\cross({^\om}\om)^n)$.
  
  Given $A\sub\Xx$ where $\Xx$ is a standard space,
  $A^c$ denotes the complement $\Xx\cut A$ of $A$ in $\Xx$. (If $A=\emptyset$,
  then this is ambiguous, but this should not cause a problem.)
  
   Let $\Gamma$ be a pointclass.
   
   The \emph{dual} of $\Gamma$,
   denoted $\stackrel{\smile}{\Gamma}$,
   is the pointclass $\{A^c|A\in\Gamma\}$,
   where we automatically include each standard space $\Xx\in\stackrel{\smile}{\Gamma}$
   if $\emptyset\in\Gamma$.
   
  Let $y\in{^\om}\om$. We define the \emph{relativized} pointless $\Gamma(y)$.
  Given a standard space $\Xx$
  and $A\sub\Xx$, we put $A\in\Gamma(y)$
  iff
   there is a set $B\in\Gamma$
   with $B\sub\Xx\cross{^\om}\om$
   such that for all $x\in\Xx$,
   we have $x\in A$ iff $(x,y)\in B$.
   
   We define $\Delta_\Gamma=\Gamma\cap\stackrel{\smile}{\Gamma}$.
 \end{dfn}
 
 \begin{rem}
  Most pointclasses of interest have some basic closure properties, like closure under continuous preimage, closure under pairwise conjunction, etc.
 \end{rem}

\begin{dfn}\label{dfn:scale_prop}
  Let $\Gamma$ be a pointclass.
  
For $A\in\Gamma$, a \emph{$\Gamma$-\tu{(}semi\tu{)}scale} on $A$
is a (semi)scale $\vec{\varphi}=\left<{\leq_n}\right>_{n<\om}$ on $A$ such that
  the ternary relations $R,S$ of arguments $(n,x,x')$ defined by
  \[ R(n,x,x')\iff x\leq^*_n x', \]
  \[ S(n,x,x')\iff x<^*_n x' \]
  are both in $\Gamma$.

 $\Gamma$ has the \emph{\tu{(}semi\tu{)}scale property}, denoted $\Scale(\Gamma)$, iff for every $A\in\Gamma$ there is a $\Gamma$-(semi)scale on $A$.
  
  $\Gamma$ has the \emph{prewellordering property},
  denoted $\PWO(\Gamma)$,
  if for every $A\in\Gamma$ there is a prewellorder $\leq$ on $A$ such that ${\leq^*}\in\Gamma$ and ${<^*}\in\Gamma$.
  
  $\Gamma$ has the \emph{uniformization property}
  iff for every $A\in\Gamma$ such that $A\sub{{^\om}\om}\cross{{^\om}\om}$
  there is a $U\in\Gamma$ which uniformizes $A$.
  
  $\Gamma$ has the \emph{reduction property}
  iff for all $A,B\in\Gamma$ there are $A',B'\in\Gamma$ such that $A'\sub A$, $B'\sub B$, $A\cup B=A'\cup B'$, and $A'\cap B'=\emptyset$.
  
  $\Gamma$ has the \emph{separation property}
  iff for all standard spaces $\Xx$ and all $A,B\in\Gamma$ with $A,B\sub\Xx$,
  if $A\cap B=\emptyset$ then
  there are $C,D\in\Gamma$
  such that $A\sub C$, $B\sub D$, $C\cap D=\emptyset$, and $C\cup D=\Xx$.
  
\end{dfn}

Note here that we do not demand
that each individual $\leq^*_n$ is in $\Gamma$, but  that the sequence $\left<{\leq^*_n}\right>_{n<\om}$ is in $\Gamma$
(in the obvious codes),
and likewise for the $<^*_n$. (For typically considered pointclasses, it will follow
that each individual $\leq^*_n$ and $<^*_n$ also belongs to $\Gamma$.)

 \begin{dfn}\label{dfn:R'}
  Let $\Gamma$ be a pointclass and let $\vec{\varphi}=\left<{\leq_n}\right>_{n<\om}$
  be a $\Gamma$-scale on some $A\in\Gamma$, as witnessed by $R,S$.
 For $n<\om$ define the relation $\equiv_{<n}^*$ on ${^\om}\om$ by
 \[ x\equiv_{<n}^* x' \iff n=0\ \vee\Big(x,x'\in A\text{ and }x\rest n=x'\rest n\text{ and }\all k<n\ \big[x\leq^*_k x'\leq^*_k x\big]\Big).\]

 For $n<\om$ define 
 $R'\sub\om\cross{^\om}\om\cross{^\om}\om$ by setting $R'(n,x,x')$ iff either
 \begin{enumerate}[label=(\roman*)]
\item  $x\equiv_{<n}^*x'$, or
\item  there
 is $k<n$ such that
 \[ x\equiv_{<k}^* x'\text{ and either }x<^*_k x'\text{ or }\big(x\leq^*_k x'\text{ and }x(k)<x'(k)\big).\qedhere\]
 \end{enumerate}
\end{dfn}

\begin{dfn}
 The \emph{cumulative hierarchy} $\left<V_\alpha\right>_{\alpha\in\OR}$
 is defined by transfinite recursion on ordinals $\alpha$ as follows. We set $V_0=\emptyset$,
 $V_{\alpha+1}=\pow(V_\alpha)$, and $V_\lambda=\bigcup_{\alpha<\lambda}V_\alpha$
 for limit ordinals $\lambda$.
\end{dfn}

In the following definition we are interested in $V_{\omega+1}$. Note that ${^\om}\om\sub V_{\om+1}$.  We defined the pointclasses $\bfSigma^1_k,\bfPi^1_k,\bfDelta^1_k$, for $k\geq 1$, in the introduction. We now define the pointclass $\Sigma^1_1$, which is a finer variant.

\begin{dfn}\label{dfn:Sigma^1_n}
Let $\Xx$ be a standard space.
For $x\in\Xx$, let $x^*\in V_{\omega+1}$ be the following natural encoding of $x$. If $\Xx=\om^m$ for some $m<\om$
then $x^*=x$. Now suppose $\Xx=({^\om}\om)^n$ where $0<n<\om$. Let $x\in\Xx$, so $x:n\to{^\om}\om$.
Let $x^*:\om\to V_\om$ be the function \[x^*(k)=(x(0)\rest k,x(1)\rest k,\ldots,x(n-1)\rest k).\]
Now suppose $\Xx=\om^m\cross({^\om}\om)^n$ where $0<m,n<\om$. Let $x\in\Xx$;
say $x=(y,z)$ where $y\in\om^m$
and $z\in({^\om}\om)^n$. Then $x^*=(y,z^*)$.
So note that $x,x^*$ are easily inter-definable,
and $x^*\sub V_\om$, so $x^*\in V_{\om+1}$.

Let $\Ll$ be the language of set theory (that is, $\Ll$ has the binary relation symbols $\in$ and $=$).  An $\Ll$-formula is called \emph{$\Sigma_0$} iff for each quantifier appearing in $\varphi$, there are variables $x,y$ such that
the quantifier either has form
``$\all x\in y$'' or form ``$\exists x\in y$''.

 The pointclass $\Sigma^1_1$ consists of all sets $A$ such that for some standard space $\Xx$, we have $A\sub\Xx$
  and there is a $\Sigma_0$ $\Ll$-formula $\varphi(u,v,w)$ in the free variables $u,v,w$ such that for all $x\in\Xx$, we have
  \[ x\in A\iff \exists y\in {^\om}\om\ \Big[V_{\om+1}\sats\varphi(x^*,y,V_\om)\Big].\]
  
  Given $\Sigma^1_k$ where $k\geq 1$, the pointclass $\Pi^1_1$ is the dual $\stackrel{\smile}{\Sigma^1_k}$ of $\Sigma^1_k$.
  
  Given $\Pi^1_k$ where $k\geq 1$,
  the pointclass $\Sigma^1_{k+1}$
  is $\exists^{{^\om}\om}\ \Pi^1_k$; that is,
  if $A\sub\Xx$ where $\Xx$ is a standard space then
  $A\in\Sigma^1_{k+1}$ iff there is $B\in\Pi^1_k$ with $B\sub\Xx\cross{^\om}\om$ and for all $x\in\Xx$, we have
  \[ x\in A\iff\exists y\in{^\om}\om\ [(x,y)\in B].\]
  Define $\Delta^1_k=\Delta_{\Sigma^1_k}$.
  \end{dfn}

  We can now state the basic relation between the lightface pointclasses $\Sigma^1_k$, $\Pi^1_k$
  and their boldface counterparts $\bfSigma^1_k$, $\bfPi^1_k$:
\begin{fact}
 Let $k\geq 1$. Then $\bfSigma^1_k=\bigcup_{y\in{^\om}\om}\Sigma^1_k(y)$ and $\bfPi^1_k=\bigcup_{y\in{^\om}\om}\Pi^1_k(y)$.
\end{fact}

In Facts \ref{fact:scale_implies_unif}--\ref{fact:Gamma_and_dual_not_both_unif} below, we give proofs only for the lightface versions ($\Sigma^1_k,\Pi^1_k,\Delta^1_k$); the boldface versions ($\bfSigma^1_k,\bfPi^1_k,\bfDelta^1_k$) are similar.

\begin{fact}\label{fact:scale_implies_unif}Let $k\geq 1$ and suppose that  $\Pi^1_k$ has the scale property. Then $\Pi^1_k$ has the uniformization property. Likewise for $\bfPi^1_k$.
\end{fact}
\begin{proof}
 Fix a $\Pi^1_k$-scale $\vec{\varphi}$ on $A$,
 as witnessed by the $\Pi^1_k$ relations $R,S$ as in Definition \ref{dfn:scale_prop}.
Let  $R'$ be as in Definition \ref{dfn:R'}.
Note that $R'$ is also $\Pi^1_k$.
Note that for all $n<\om$,
\[ R'_n\rest A=\{((x,y),(x',y'))\in A\bigm|R'(n,(x,y),(x',y'))\} \] is a prewellorder  on $A$.
Note also that for $m<n<\om$,
$R'_{n+1}$ refines $R'_n$;
that is, \[\text{ if }R'(n+1,(x,y),(x',y'))\text{
then }R'(n,(x,y),(x',y')).\]

Now define
 \[ U(x,y)\iff A(x,y)\wedge\all n<\om\ \all y'\ R'(n,(x,y),(x,y')).\]
 Since $R$ is $\Pi^1_k$, so is $U$.

 We claim that $U$ uniformizes $A$, as desired. For let $\pi_n:A\to\OR$ be the norm of the prewellorder $R'_n\rest A$. Let $x\in\dom(A)$. Let $\left<y_n\right>_{n<\om}$ be such that \[(x,y_n)\in A\text{ and }\pi_n(x,y_n)=\min(\{\pi_n(x,y)|(x,y)\in A\}).\]
 Note that $y_n\to y$ for some $y$, and that
\[ (x,y_n)\to(x,y)\mod\vec{\varphi}.\]
So $(x,y)\in A$, and by lower semicontinuity,
\[ \varphi_n(x,y)\leq\lim_{m\to\om}\varphi_n(x,y_m), \]
and note then that
\[ \all n<\om\ \all y'\ R'(n,(x,y),(x,y')). \]
So $U(x,y)$ holds, and it is straightforward to see that if $U(x,y')$ then $y=y'$.
\end{proof}

\begin{fact}
 Let $k\geq 1$ and suppose that $\Pi^1_k$ has the uniformization property. Then $\Sigma^1_{k+1}$ has the uniformization property.
 Likewise for $\bfPi^1_k,\bfSigma^1_{k+1}$.
\end{fact}
\begin{proof}
 Let $A\sub\Xx\cross\Yy$ be $\Sigma^1_{k+1}$.
 Let $B\sub\Xx\cross\Yy\cross{^\om}\om$ be $\Pi^1_k$
 and such that for all $(x,y)\in\Xx\cross\Yy$ we have
 \[ (x,y)\in A\iff\exists z\in{^\om}\om\ [(x,y,z)\in B].\]
 By uniformization for $\Pi^1_k$, we can fix $C\in\Pi^1_k$ which uniformizes $B$
 in $\Yy\cross{^\om}\om$; that is, $C\sub B$ and \[C:S\to\Yy\cross{^\om}\om \]
 where $S=\{x\in\Xx\bigm|\exists y,z\ [(x,y,z)\in B]\}$.
 Now define $U$ by
 \[ (x,y)\in U
\iff\exists z\ [(x,y,z)\in C] 
 \]
 and note that $U\in\Sigma^1_{k+1}$ and $U$ uniformizes $A$.
\end{proof}

\begin{fact}
 Let $k\geq 1$ and let $\Gamma\in\{\Pi^1_k,\Sigma^1_k,\bfPi^1_k,\bfSigma^1_k\}$. Then:
 \begin{enumerate}\item\label{item:unif_implies_red} if $\Gamma$ has the uniformization property then $\Gamma$ has the reduction property; and
  \item\label{item:red_implies_sep_for_dual} if $\Gamma$ has the reduction property then 
 $\stackrel{\smile}{\Gamma}$ has the separation property.
  \end{enumerate}
\end{fact}
\begin{proof}
 Part \ref{item:unif_implies_red}: Let $A,B\in\Gamma$ such that $A,B\sub\Xx$ for some standard space $\Xx$. Let $A'=A\cross\{0\}$
 and $B'=B\cross\{1\}$.
 Let $C=(A\cross\{0\})\cup(B\cross\{1\})$
 (so $C\sub\Xx\cross\om$)
 and let $D$ uniformize $C$ in the last coordinate. Now let $A'=\{x\in\Xx\bigm|(x,0)\in D\}$ and $B'=\{x\in\Xx\bigm|(x,1)\in D\}$,
 and note that $(A',B')$ reduces $(A,B)$, as desired.
 
Part \ref{item:red_implies_sep_for_dual}: This is easy.
\end{proof}

\begin{dfn}
 Let $\Gamma$ be a pointclass and $\Yy,\Xx$ be  standard spaces.
 A set $U\sub\Yy\cross\Xx$
 is called \emph{$\Gamma$-universal}
 (for subsets of $\Xx$)
 iff $U\in\Gamma$ and
 \[ \{U_y\bigm|y\in\Yy\} \]
 is exactly the collection of all subsets of $\Xx$ which are in $\Gamma$. (Here $U_y=\{x\in\Xx\bigm|(y,x)\in U\}$.)
 We also say that the $\Gamma$-universal set is \emph{$\Yy$-indexed}.
\end{dfn}

A well known construction gives:
\begin{fact}
 Let $k\geq 1$. For each standard space $\Xx$,
 there is an $\om$-indexed $\Sigma^1_k$-universal set and an $\om$-indexed $\Pi^1_k$-universal set,
 and there is an ${^\om}\om$-indexed $\bfSigma^1_k$-universal set and an ${^\om}\om$-indexed $\bfPi^1_k$-universal set.
\end{fact}

\begin{fact}\label{fact:no_universal_Delta-set}
 Let $k\geq 1$. Then there is no $\om$-indexed $U$ which is $\Delta^1_k$-universal for subsets of ${^\om}\om$, and no ${^\om}\om$-indexed
 $U$ which is $\bfDelta^1_k$-universal for subsets of ${^\om}\om$.
\end{fact}
\begin{proof}
 Suppose that $U\sub\om\cross{^\om}\om$ is $\Delta^1_k$-universal. 
 Define $U'\sub{^\om}\om$
 by putting $x\in U'$ iff $(x(0),x)\notin U$.
 Then $U'\in\Delta^1_k$, so  by the universality of $U$, there is $d<\om$
 such that \[ (d,x)\in U\iff x\in U'\iff (x(0),x)\notin U \]
 for all $x\in{^\om}\om$.  But then letting $x\in{^\om}\om$ be any element such that $x(0)=d$, we get
 \[ (d,x)\in U\iff (d,x)\notin U,\]
 a contradiction.
 \end{proof}
 
\begin{fact}\label{fact:Gamma_and_dual_not_both_unif}
 Let $k\geq 1$. Then it is not the case that both $\Pi^1_k,\Sigma^1_k$ have the reduction property.
 Therefore they do not both have the uniformization property. Likewise for $\bfPi^1_k,\bfSigma^1_k$.
\end{fact}
\begin{proof}
 Suppose otherwise.
 Then both also have the separation property.
 Let $U\sub\om\cross{^\om}\om$ be a $\Sigma^1_k$-universal for  subsets of ${^\om}\om$.
 Let $\sigma:\om\to\om\cross\om$
 be a recursive bijection.
 For $(m,n)\in\om\cross\om$
 let $(m,n)_0=m$ and $(m,n)_1=n$.
 
 Let $(m,x)\in A$ iff $(\sigma(m)_0,x)\in U$
 and let $(m,x)\in B$ iff $(\sigma(m)_1,x)\in U$.
 So $A,B\in\Sigma^1_k$.
 Let $A',B'\in\Sigma^1_k$ be such that $(A',B')$ reduces $(A,B)$.
 So $A'\cap B'=\emptyset$.
 Let $A'',B''\in\Sigma^1_k$ be such that $(A'',B'')$ separates $(A',B')$.
 So $A''=\om\cross{^\om}\om\cut B''$,
 so $A'',B''\in\Delta^1_k$. Now note that $A''$ is $\Delta^1_k$-universal for subsets of ${^\om}\om$, contradicting Fact \ref{fact:no_universal_Delta-set}.
 \end{proof}

\subsection{Notation}
\label{sec:notation}

Descriptive set theoretic terminology and notation is mostly as in \cite{mosch}, and as described in \S\ref{sec:introduction} and \S\ref{sec:background}.

Most inner model theoretic terminology and notation is as described in \cite[\S1.1]{iter_for_stacks} (and further details are covered in \cite[\S1.2]{extmax} 
and \cite[\S1.1]{premouse_inheriting}).
In this section we describe some further things not covered or which might differ slightly from those sources.

\index{premouse}\index{$\Ll_\pm$}\index{language of premice}
We will be dealing with \emph{real-premice},
by which we mean premice $M$ as in \cite{outline},
with Mitchell-Steel indexing,
except that:
\begin{enumerate}[label=(\roman*)]\item the first segment $M|\om$ of  $M$
is of the form $(V_\om,x)$ where $x\in\RR$,
\item\label{item:superstrong_exts}
 we allow extenders of superstrong type on the extender sequence of $M$, 
 \item\label{item:condensing_pm} we demand that every proper segment of $M$ is not only $\om$-sound but \emph{satisfies condensation}, as defined  below.\footnote{For iterable structures, this makes no difference. It is included for convenience.}
 \end{enumerate}
 We also make analogous modifications to the definitions for all structures related to premice, 
such as pseudo-premice, protomice, bicephali, phalanxes, etc.
We say that
$M$ as above is an \emph{$x$-premouse} or is \emph{over $x$},
and write $x^M=x$.
We usually omit the prefix \emph{real-}, and often omit explicit mention of $x$.
The language $\Ll_\pm$ of (real-)premice $M$ is the standard premouse language augmented with a constant symbol 
$\dot{x}$ interpreted as $x^M$, and all fine structure for real-premice is defined 
using this language.
Property \ref{item:superstrong_exts} above introduces a small change in the details of iteration trees,
as explained in \cite[Remark 2.48]{operator_mice_v3}.
We abbreviate \emph{premouse} with \emph{pm}.

(Occasionally we will mention premice $M$ \emph{over $X$},
where $X$ is some transitive set/structure; in this case
the fine structure of $M$ uses the language with constant
symbols for every element of the transitive closure of $X\cup\{X\}$,
and the extenders in the extender sequence of $M$ have critical points $>\OR^X$.
But this is only an exception, and when not explicitly mentioned otherwise,
all premice will be over a real $x$.)

 Let $N$ be a $(k+1)$-sound pm.
 We say $N$ \emph{top-satisfies $(k+1)$-condensation}
 iff condensation holds with respect to $k$-lifting embeddings
 $\pi:M\to N$ with $(k+1)$-sound $M$ and $\crit(\pi)\geq\rho_{k+1}^M$
 (see \cite{premouse_inheriting} for definitions).
 We say that $N$ satisfies \emph{$(k+1)$-condensation}\index{$(k+1)$-condensation}
 iff $N'$ satisfies $(k'+1)$-condensation for all $(N',k')\leq_\lex(N,k)$.
 If $N$ is $\om$-sound, we say that $N$ \emph{satisfies condensation}
 iff $N$ satisfies $(k+1)$-condensation for all $k<\om$. (This is the notion required of all proper segments of $M$ in condition \ref{item:condensing_pm} above.)
 Note that if $N$ satisfies condensation and $\pi:M\to N$
 is fully elementary with $\crit(\pi)=\rho_\om^M$,
 then condensation holds with respect to $\pi$.

 We work with Mitchell-Steel fine structure (see \cite{outline}, \cite{fsit}), simplified by dropping the parameters $u_n$ as explained in \cite[\S5]{V=HODX_pub}.
So for premice $N$ and for $k<\om$,
this specifies the notions \emph{$k$-soundness}, \emph{$(k+1)$-universality}, \emph{$(k+1)$-solidity}, and the $(k+1)^{\nth}$ core $\core_{k+1}(N)$ of $N$. We now define when $N$ is \emph{$k$-good},\index{$k$-good}
for $k\leq\om$, by induction on $k$. If $N$ is $k$-good, then $\core_j(N)$ will be defined and 
$j$-sound for each $j\leq k$, and hence $\core_{k+1}(N)$ will also be defined. We say that $N$ is 
$0$-good. Given $k<\om$, $N$ is $(k+1)$-good iff $N$ is $k$-good, $\core_k(N)$ is 
$(k+1)$-universal, and $\core_{k+1}(N)$ is $(k+1)$-solid (hence $\core_k(N)$ is also 
$(k+1)$-solid). And $N$ is $\om$-good iff $N$ is $k$-good for all $k$.

An \emph{$\om$-premouse} is an $\om$-sound premouse
which projects to $\om$.
An \emph{$\om$-mouse} is an $(\om,\om_1+1)$-iterable $\om$-premouse,
except for in \S\ref{sec:Pi^1_3},
where the definition is modified to $(\om,\om_1)$-iterability.
If $M$ is an $\om$-mouse then $\omdeg(M)$ is the least $k$ such that $\rho_{k+1}^M=\om$.

If $M,N$ are premice,
we write $M\ins^*N$ to mean
that either $M\ins N$ or $M=N||\alpha$
for some $\alpha\leq\OR^N$.

If $\Tt$ is an iteration tree on a phalanx
consisting of a $\lambda$-sequence of models,
and $\alpha<\lh(\Tt)$,
then $\phalroot^\Tt(\alpha)$ denotes
the root of $\alpha$ in $\Tt$
(this is the unique $\beta<\lambda$
such that $\beta\leq^\Tt\alpha$).
For $Y$ a premouse and $\kappa\in\OR^Y$
such that $\kappa$ is $Y$-measurable via a ($Y$-total) extender in $\es_+^Y$,
$D_{\kappa}^Y$ denotes the order $0$ measure on $\kappa$ in $\es_+^Y$.

An iteration tree $\Tt$ is \emph{sse-essentially $m$-maximal}
if it follows the rules for $m$-maximality,
except that if $E^\Tt_\alpha$
is of superstrong type (that is,
$\lambda(E^\Tt_\alpha)=\nu(E^\Tt_\alpha)$), then we only demand that $\nu(E^\Tt_\alpha)<\lh(E^\Tt_\beta)$ for $\alpha<\beta$
instead of demanding $\lh(E^\Tt_\alpha)\leq\lh(E^\Tt_\beta)$.
We say that $\alpha<\lh(\Tt)$
is \emph{$\Tt$-stable} iff $\lh(E^\Tt_\alpha)\leq\lh(E^\Tt_\beta)$ for all $\beta>\alpha$.

 We take it that the degree $n$ of an $n$-maximal (or sse-essentially $n$-maximal) iteration tree is encoded explicitly into $\Tt$, as $\deg^\Tt_0$.

A premouse $M$ is called \emph{non-small} iff there is $\delta<\OR^M$ such that $M\sats$``$\delta$ is Woodin'', and $M$ is called \emph{small} otherwise.\index{small}\index{non-small}\index{small segment}\index{non-small segment} If $M$ is non-small we write $\delta^M$ for the least Woodin of $M$.
We say that a normal iteration tree $\Tt$ is \emph{non-small} iff there is $\alpha+1<\lh(\Tt)$
such that $\exit^\Tt_\alpha$ is non-small, and say $\Tt$ is \emph{small} otherwise.
Note that if $\Tt$ is $k$-maximal and non-small and $\alpha$ is least such that $\exit^\Tt_\alpha$ is 
non-small, then $\Tt$
is equivalent to $\Tt\rest(\alpha+1)\conc\Tt'$ where $\Tt'$ is $q$-maximal and on $Q$
where $Q\ins M^\Tt_\alpha$ and either:
\begin{enumerate}[label=--]
 \item $Q=M^\Tt_\alpha$ and $\delta^{\exit^\Tt_\alpha}$ is Woodin in $M^\Tt_\alpha$ and $q=\deg^\Tt_\alpha$, or
 \item $Q\pins M^\Tt_\alpha$ is the Q-structure for $\delta^{\exit^\Tt_\alpha}$
 and $\rho_{q+1}^{Q}\leq\delta^{\exit^\Tt_\alpha}<\rho_q^Q$.
\end{enumerate}
(This is because $\delta^{\exit^\Tt_\alpha}$ is a strong cutpoint of $\exit^\Tt_\alpha$,
and hence of $\Tt$.) Note then that for every $\beta+1<\lh(\Tt)$,
if $\alpha\leq\beta$ then $\exit^\Tt_\beta$ is non-small. We say that $\Tt$ is the \emph{small 
segment} of $\Tt$, and $\Tt'$ the \emph{non-small segment}.

If $\delta$ is a Woodin cardinal,
we write $\BB_\delta$ for the $\delta$-generator extender algebra at $\delta$.
(See \cite[\S7.2]{outline} for the $\om$-generator version. The $\delta$-generator version is defined in the same manner,
except that, in  notation as there, there are propositional variables $A_\alpha$ for each $\alpha<\delta$.)

If $M$ is a structure of some fragment
of set theory modelling ``there is a largest cardinal'', then $\lgcd(M)$ denotes the largest cardinal of $M$.

\section{Motivating examples}

Before we start our work in earnest,
we will give some basic examples of the kind of scale (and prewellordering) constructions  we will be developing, ones which are easy to describe. 
This should give the reader some  intuition for what we want to achieve. 

\subsection{The scale property for $\Pi^1_1$}\label{subsec:Pi^1_1}

The basic approach to constructing scales that we will develop
is easily illustrated with a slight variant of the usual proof
that $\Pi^1_1$ has the scale property. We start with this.

Fix a $\Pi^1_1$ set $A\sub\Baire$ and a recursive tree $T$ on $\om\cross\om$
such that for all reals $x$,
\[ x\in A\iff T_x\text{ is wellfounded.}\]
For $x\in A$ let $M_x=\J_\alpha(x)$ where $\alpha$ is least such that $\J_\alpha(x)\sats$``There is $\beta\in\OR$
and a rank function $\varrho:T_x\to\beta$''.
Here $\varrho\in M_x$, so this statement is $\Sigma_1$ (in the language $\Ll_{\rpm}$ of real premice,
which includes a constant symbol $\dot{x}$ interpreted as $x$),
and
\[ M_x=\Hull_1^{M_x}(\emptyset) \]
(the hull is computed using $\Ll_{\rpm}$).

For each sentence $\psi\in\Ll_{\rpm}$, we define a norm
\[ \varphi_\psi:A\to\{0,1\} \]
by setting $\varphi_\psi(x)=1$ iff $M_x\sats\psi$.

Suppose that $\{x_n\}_{n<\om}\sub A$ and $x_n\to x\mod\vec{\varphi}$,
where $\vec{\varphi}$ enumerates these norms. Let
\[ T_\infty=\{\psi\mid 1=\lim_{n\to\om}\varphi_\psi(x_n)\}. \]
Clearly $T_\infty$ is a complete consistent theory. We claim there is a unique (up to isomorphism)  structure $M_\infty$
for the language $\Ll_\rpm$ such that
\[ \Th^{M_\infty}=T_\infty\text{ and }M_\infty=\Hull^{M_\infty}(\emptyset).\]

Uniqueness is clear. We verify existence.
Let $\psi(v)$ be a formula in $\Ll_{\rpm}$ of one free variable $v$.
Let $x\in A$. If $M_x\sats\exists y\psi(y)$ then let
\[ t_\psi^x=\text{ the }<_{L[x]}\text{-least }y\in M_x\text{ such that }M_x\sats\psi(y),\]
and otherwise let $t_\psi^x=0$.
Define $f_\psi:\om\to V$
by
 \[ f_\psi(n)=t_\psi^{x_n}. \]

We set $M_\infty$ to be the ultraproduct of the structures $M_{x_n}$
modulo the cofinite filter, formed using only the functions $f_\psi$
Write $[f]$ for the equivalence class of $f$.
The reader will happily verify that L\'os Theorem goes through
 as usual,
and the desired properties hold of $M_\infty$.
In particular, $M_\infty\sats T_\infty$,
so $M_\infty\sats$``$V=\J_\OR(\dot{x})$, there is a rank function $\varrho:\widetilde{T}\to\OR$,
and no proper segment of me has such a rank function'',
where $\widetilde{T}\in M_\infty$ is what $M_\infty$ sees as the
tree associated to its base real $\dot{x}^{M_\infty}$
(according to the recursive procedure which converts $x$ to $T_x$).

However, $M_\infty$ might be illfounded -- even its $\om$ might be.
We need to introduce extra norms to ensure wellfoundedness.
For each formula $\tau\in\Ll_{\rpm}$ of one free variable $v$,
let
\[ \psi_\tau(v)\iff v\in\OR\wedge\tau(v).\]

For each $\tau$, we define a norm
\[ \varphi^\OR_\tau:A\to\OR \]
on $A$, by setting
\[ \varphi^\OR_\tau(x)=t^x_{\psi_\tau}.\]

Now suppose that $x_n\to x_\infty$ mod $\vec{\varphi}$,
where $\vec{\varphi}$ now enumerates all of the preceding norms (of both types).

We claim that $M_\infty$ is wellfounded. For note that the ``ordinals'' $\OR^{M_\infty}$
of $M_\infty$ are all of the form $[f_{\psi_\tau}]$ for some $\tau$ as above.

Define the map
\[ \pi:\OR^{M_\infty}\to\OR \]
by
\[ \pi([f_{\psi_\tau}])=\lambda_\tau=\lim_{n\to\om}\varphi^\OR_\tau(x_n). \]
It is easy to see that $\pi$ is order-preserving,
so $M_\infty$ is wellfounded.

It follows that $x^{M_\infty}=x_\infty\in\RR$ and $M_\infty=\J_\beta(x_\infty)$
for some $\beta$, and because $M_\infty\sats T_\infty$,
that therefore $x_\infty\in A$ and $M_\infty=M_{x_\infty}$,
so we have a semi-scale.

In order to calibrate the definability appropriately, we
replace the preceding norms $\varphi$ with the norms
$\varphi^*$, where $\varphi^*(x)$
is the lexicographic rank of the pair $(\OR^{M_x},\varphi(x))$
(note here that $\varphi(x)<\OR^{M_x}$). We assume from now that
$x_n\to x$ mod the new norms.

Lower semi-continuity follows from the properties of $\pi$.
For note that because $\pi$ is order-preserving,
\[ \OR^{M_\infty}\leq\lim_{n\to\om}\OR^{M_{x_n}}.\]
Lower-semi-continuity with respect to the theory norm components (the $0/1$-valued norms) is clear.
And for the ordinal norms, note that
\[ \varphi^\OR_\tau(x)=[f_{\psi_\tau}],\text{ so } \pi(\varphi^\OR_\tau(x))=\lambda_\tau,\]
and since $\pi$ is order-preserving, therefore $\varphi^\OR_\tau(x)\leq\lambda_\tau$.

Finally we need to see that we have a $\Pi^1_1$ scale.
So fix $\varphi$, either one of the theory norms or ordinal norms above, and consider $\varphi^*$.
Given $x\in A$, let $\theta_x=\OR^{M_x}$.
Let $x,y\in\RR$. Then $x\leq^*_{\varphi^*}y$ iff
\begin{equation}\label{eqn:leq^*_equiv} M_x\text{ exists  and if }M_y\text{ exists then }[\theta_x\leq\theta_y\text{ and if }\theta_x=\theta_y\text{ then }\varphi(x)\leq\varphi(y)].\end{equation}
But if $x\in A$ (equivalently,  $M_x$ exists),
then  all first-order structures for the language of set theory
with wellfounded $\om$ and $x,y\in\RR^N$,
and which satisfy ``Foundation holds and $V=L[x,y]$ and either $M_x$ exists or $M_y$ exists'', are
 correct about the truth of (\ref{eqn:leq^*_equiv}).
This is by overspill for $\Pi^1_1$, 
see \cite[8.11, 8.16]{barwise}:
\begin{lem}[Overspill for $\Pi^1_1$, Ville]\label{lem:overspill_Pi^1_1}
 Let $M$ be a structure in the language of set theory, 
 which is illfounded, but $\om^M=\om$ is wellfounded.
 Suppose that $M$ satisfies Foundation  and there is $x\in\RR^M$ such that $M\sats$``$V=L[x]$''.
 Let $\beta$ be the ordertype of the wellfounded part of $M$.
 Then $L_\beta[x]$ is admissible.
\end{lem}
\begin{proof}
 Let $\varphi$ be $\Sigma_1$
 and $p,d\in L_\beta[x]$
 and suppose that $L_\beta[x]\models\psi(d,z,p)$, the statement asserting ``for all $y\in d$ there is $z$ such that $\varphi(p,y,z)$''.
 Taking any $\alpha\in\OR^M$ such that $\alpha$ is in the illfounded part of $M$, we have $M\sats$``$L_\alpha[x]\sats\psi(p,y,z)$''. Let $\alpha_0\in\OR^M$ be such that $M\sats$``$\alpha_0$ is the least ordinal $\alpha$ such that $L_\alpha[x]\models\psi(p,y,z)$''.
 It follows that $\alpha_0$ is not in the illfounded part, so $\alpha_0<\beta$ is a real ordinal, and $L_{\alpha_0}\models\psi(p,y,z)$, so we are done.
\end{proof}

Note that by the lemma, if $N$ (as described just before the lemma) is illfounded then $\om_1^{\mathrm{CK},(x,y)}\sub\wfp(N)$,
and therefore $\theta_x\in\wfp(N)$,
and therefore $N$ is correct about (\ref{eqn:leq^*_equiv}). It easily follows that (\ref{eqn:leq^*_equiv}) is $\Pi^1_1$ (in $V$), as desired.

\subsection{$\Delta^1_3$ scales on $\Pi^1_2$ sets}\label{sec:MS}

We now consider a second example: the theorem of Martin and Solovay, that
assuming every real has a sharp,
for every $\Pi^1_2$ set $A$,
there is a $\Delta^1_3$ scale on $A$.
(There are more refined versions of this theorem, but this version is sufficient to illustrate what we want.) 
We will prove this in the style of the foregoing construction for $\Pi^1_1$. The proof is also just a slight variant of the Martin-Solovay construction, but it will continue to motivate our main construction to come.

Actually, the same construction will work more generally.
Recall that a set $A\sub\RR$ is $\Game{{<\omega^2}}\text{-}\Pi^1_1$ iff there is $n<\om$ and a formula $\varphi$ of the language of set theory such that for all $x\in\RR$, we have
\[ x\in A\iff L[x]\sats\varphi(x,\kappa^x_{0},\ldots,\kappa^x_{n-1}), \]
where $\left<\kappa^x_\alpha\right>_{\alpha\in\OR}$
enumerates the class of Silver indiscernibles for $L[x]$ in increasing order. Fix such an $A,n,\varphi$.
We will construct a $\Delta^1_3$ scale on $A$.

For $x\in A$, let $M_x=x^\#$, represented as an active mouse over $x$. We have $M_x=\Hull_1^{M_x}(\emptyset)$ (the hull is computed  using the language $\Ll_{\rpm}$ of real premice, which  includes a predicate for the active extender). 

We define theory norms like in \S\ref{subsec:Pi^1_1};
so for each sentence $\psi\in\Ll_{\rpm}$, we have a $0/1$-valued norm $\varphi_\psi$,
with $\varphi_\psi(x)=1$ iff $M_x\sats\psi$.
Again, if $x_n\to x$ modulo these norms, and $T_\infty$ is as before, then we can define a unique limit model $M_\infty$
for the language $\Ll_{\rpm}$
such that $\Th^{M_\infty}=T_\infty$
and $M_\infty=\Hull^{M_\infty}(\emptyset)$;
in fact, we get $M_\infty=\Hull_1^{M_\infty}(\emptyset)$,
because $M_x=\Hull_1^{M_x}(\emptyset)$
for all $x$.\footnote{Here ``$\Hull(X)$'' (with no subscript) denotes the full definable hull of finite tuples taken from $X$, and ``$\Hull_1(X)$'' the $\Sigma_1$-hull of such tuples.} To ensure the wellfoundedness of $M_\infty$, we again introduce countably many ordinal norms $\varphi^{\OR}_\tau$ as before. But this time, this (does not appear to be) enough, because we also need to know that the limit model $M_\infty$ is iterable. To ensure iterability, we introduce further norms. These norms 
are somewhat subtler than those used  so far, and it will not be immediate that they are well-defined.
Given $z\in\RR$  and $n<\om$
let $z^\#_n=(z^\#)_n$ be the $n$th iterate of $z^\#$ (as an active $z$-premouse).
Given $x\in\RR$ and an active premouse $Q$ with $x^\#\in Q$,
let  $P^{Q}_x$ be the $\crit(F^Q)$th iterate of $x^\#$;
equivalently, $P^{Q}_x$ is the unique active $x$-premouse $P$ such that
$\es^P=\emptyset$ (so $F^P\neq\emptyset$ is the first extender in the $P$-sequence), and letting $\alpha=\OR^P$, then
 $P=(L_\alpha[x],F)$
and $F\rest\nu(F)\sub F^{Q}$. Note that $P^Q_x$ is definable over $Q$ from $x$,
uniformly in such $(Q,x)$.

Now for each $n<\om$ and each ordinal term $\tau$ of $n$ arguments
which is $\Sigma_1$ in $\Ll_{\rpm}$, we define a prewellorder $\leq_{\tau}$ on $A$ as follows: for $x,y\in A$ set
\[ x\leq_{\tau}y\iff \all z\in{^\om}\om\ [\text{if }(x,y)^\#\leq_Tz\text{ then }\tau^{P^{(z^\#)_n}_x}(\vec{\xi})\leq\tau^{P^{(z^\#)_n}_y}(\vec{\xi})]\]
where $\vec{\xi}=\vec{\kappa}^z_n=(\kappa^z_0,\ldots,\kappa^z_{n-1})$ is the increasing enumeration of the first $n$ $z$-indiscernibles;
equivalently, is the sequence of critical points
of measures used in the iteration leading from $z^\#$ to $z^\#_n$.

We we will show that these $\leq_\tau$ are  indeed prewellorders,
and that by adjoining them, we get a scale.
The following claim is the key to seeing that $\leq_{\tau}$ is a prewellorder:

\begin{clm}\label{clm:MS_norm_independence}
 Let $w,z$ be such that $(x,y)^\#\leq_T w$ and $(x,y)^\#\leq_T z$.
 Then
 \[ \tau^{P^{w^\#_n}_x}(\vec{\kappa}^w_n)\leq \tau^{P^{w^\#_n}_y}(\vec{\kappa}^w_n)\iff\tau^{P^{z^\#_n}_x}(\vec{\kappa}^z_n)\leq\tau^{P^{z^\#_n}_y}(\vec{\kappa}^z_n).\]
 \end{clm}
\begin{proof}
Let us consider the case that $n=1$; the general case is very similar.
Let $\mu$ be any $w$-indiscernible
and $Q_w$ be an iterate of $w^\#$
whose active extender has critical point $>\mu$. Let $\kappa=\crit(F^{w^\#})$.

\begin{sclm}We have
 \[ \tau^{P^{w^\#_1}_x}(\kappa)\leq
  \tau^{P^{w^\#_1}_y}(\kappa)\iff\tau^{P^{Q_w}_x}(\mu)\leq\tau^{P^{Q_w}_y}(\mu).
 \]\end{sclm}
 \begin{proof}
Since $\tau$ is an $\rSigma_1$ term
(of $\Ll_{\rpm}$),
we may easily fix an $\rSigma_1$ formula
 $\psi$ of $\Ll_{\rpm}$ such that
 \[ \tau^{P^{w^\#_1}_x}(\kappa)\leq
  \tau^{P^{w^\#_1}_y}(\kappa)\iff w^\#_1\sats\psi(\kappa)\]
  and
 \[ \tau^{P^{Q_w}_x}(\mu)\leq\tau^{P^{Q_w}_y}(\mu)
 \iff Q_w\sats\psi(\mu).
 \]
But then there is an $\rSigma_1$ formula $\psi'$ of no free variables such that
\[ w_1^\#\sats\psi(\kappa)\iff w^\#\sats\psi' \]
and letting $Q'_w$ be the iterate of $w^\#$ whose active extender has critical point $\mu$, then
\[ Q_w\sats\psi(\mu)\iff Q'_w\sats\psi'.\]
But the iteration  map $i:w^\#\to Q_w'$ is $\rSigma_1$-elementary, so
\[ w^\#\sats\psi'\iff Q'_w\sats\psi',\]
which suffices.
\end{proof}

It is likewise with $w$ replaced by $z$ and $Q_w$ replaced by any such $Q_z$, and so we can take $\mu$ and $\crit(F^{Q_w})=\crit(F^{Q_z})$ to be both $w$- and $z$-indiscernibles. But then $P^{Q_w}_x=P^{Q_z}_x$ and $P^{Q_w}_y=P^{Q_z}_y$, and since $\mu$ was in common, the claim follows.
\end{proof}

\begin{clm}
Each $\leq_\tau$ is a prewellorder on $A$.
\end{clm}
\begin{proof}
 It suffices to see that for each countable $\bar{A}\sub A$,
 $\leq_\tau\rest\bar{A}$ is a prewellorder.
 But note that this follows from Claim \ref{clm:MS_norm_independence},
 by taking $z\geq_T x^\#$ for some $x$ which codes all reals in $\bar{A}$.
\end{proof}

Now let $x_n\to x\mod\vec{\varphi}$,
where $\vec{\varphi}$ includes all the norms we have introduced.
This produces a limit model $M_\infty$, which is wellfounded,
so $M_\infty$ is an $x$-premouse
for some real $x$, is a putative $x^\#$ (satisfies the obvious first order properties of sharps), and satisfies ``$x\in A$''
(note this can be expressed naturally as a first order statement over sharps).

\begin{clm}\label{clm:M_infty=x^sharp}
 $M_\infty=x^\#$.
\end{clm}
\begin{proof}
 We need to see that $M_\infty$ is iterable. Let $N$ be some iterate of $M_\infty$, via the iteration of length $\lambda$; we show that $N$ is wellfounded.
 Let $\left<\kappa_\alpha\right>_{\alpha<\lambda}$ be the sequence of critical points of the iteration;
 then $N=\Hull_1^N(\{\kappa_\alpha\}_{\alpha<\lambda})$.
 
 Let $z$ be some real
 such that $z\geq x^\#$ and $z\geq\left<x_n\right>_{n<\om}^\#$. Let $z^\#_\lambda$ be the $\lambda$th iterate of $z^\#$,
 and $\left<\mu_\alpha\right>_{\alpha<\lambda}$ be the enumeration of the critical points used in the iteration. We define a map
 \[\pi:\OR^N\to\OR^{(z^\#)_\lambda}, \]as follows.
For each $\Sigma_1$ ordinal term $\tau$ of arity $n$,
 and each $\{\alpha_0,\ldots,\alpha_{n-1}\}\in[\lambda]^n$,
 letting $\vec{\kappa}=(\kappa_{\alpha_0},\ldots,\kappa_{\alpha_{n-1}})$, 
 set
 \[ \pi(\tau^N(\vec{\kappa}))=
\lim_{k\to\om}  \tau^{P^{(z^\#)_\lambda}_{x_k}}(\vec{\mu}),
 \]
 where $\vec{\mu}=(\mu_{\alpha_0},\ldots,\mu_{\alpha_{n-1}})$. Because we incorporated the prewellorders $\leq_\tau$ for each $\tau$, and by indiscernibility, it is straightforward to see that the limit above exists in the sense that the sequence $\left<\tau^{z^\#_\lambda}_{x_k}(\vec{\mu})\right>_{k<\om}$ is eventually constant. This map is well-defined and order-preserving. For let $\tau_0,\tau_1$ be $\Sigma_1$ ordinal terms
 respectively of arities $n_0,n_1$,
 and let $a_i\in[\lambda]^{n_i}$ for $i=0,1$.
 Let $a=a_0\cup a_1$.
 Let $a_0=\{\alpha_{00},\ldots,\alpha_{0,k_0-1}\}$
  with $\alpha_{0j}<\alpha_{0,j+1}$,
  and $a_1=\{\alpha_{10},\ldots,\alpha_{1,k_1-1}\}$ likewise.
 Suppose that
 \[ \tau^N(\kappa_{\alpha_{00}},\ldots,\kappa_{\alpha_{0,i_0-1}})\leq\tau^N(\kappa_{\alpha_{10}},\ldots,\kappa_{\alpha_{1,k_1-1}}).\]
 Let $b=|a|$ and let $\sigma:b\to a$
 be the order-preserving bijection
 and $b_i=\sigma^{-1}``a_i$, for $i=0,1$.
 Let $\sigma(\beta_{0j})=\alpha_{0j}$
 and $\sigma(\beta_{1j})=\alpha_{1j}$.
 
 Write $(M_\infty)_k$ for the $k$th iterate of $M_\infty$, and for $\beta<k$, let $\nu_\beta=\crit(F^{(M_\infty)_\beta})$.
 Then $T_\infty$ contains the natural statement
$\psi$ asserting that
 \[ \tau^{(M_\infty)_{b}}(\nu_{\beta_{00}},\ldots,\nu_{\beta_{0,i_0-1}})\leq\tau^{(M_\infty)_b}(\nu_{\beta_{10}},\ldots,\nu_{\beta_{1,i_1-1}}).
 \]
This statement $\psi$ in fact $\Sigma_1$,
since $\tau$ is a $\Sigma_1$ term.
So for all sufficiently large $n$,
$(x_n)^\#\sats\psi$. Therefore and likewise, for all such $n$,
for all ordinals $\xi\geq b$ and all $c\in[\xi]^{b}$, letting $\varsigma:b\to c$ be the order-preserving bijection
and $\gamma_{ij}=\varsigma(\beta_{ij})$
and $\eta_\gamma=\crit(F^{(x_n^\#)_{\gamma}})$ for $\gamma<\xi$, we have
 \[ \tau^{(x_n^\#)_{\xi}}(\eta_{\gamma_{00}},\ldots,\eta_{\gamma_{0,i_0-1}})\leq\tau^{(x_n^\#)_\xi}(\eta_{\gamma_{10}},\ldots,\eta_{\gamma_{1,i_1-1}}).
 \]
 It follows that
 \[ 
\lim_{k\to\om}  \tau_0^{P^{z^\#_\lambda}_{x_k}}(\mu_{\alpha_{00}},\ldots,\mu_{\alpha_{0,k_0-1}})\leq
\lim_{k\to\om}  \tau_1^{P^{z^\#_\lambda}_{x_k}}(\mu_{\alpha_{10}},\ldots,\mu_{\alpha_{1,k_1-1}}),\]
as desired.
\end{proof}

Since $M_\infty=x^\#$ and $M_\infty\sats$``$x\in A$'', we have $x\in A$, verifying that we have a semiscale. It remains to verify:

\begin{clm}
 The semiscale is lower semicontinuous.
\end{clm}
\begin{proof}
For the theory norms this is immediate (and like in the $\Pi^1_1$ case). So let $\tau$ be a $\Sigma_1$ ordinal term of arity $n$;
we must see that $x\leq_{\tau}x_k$ for all sufficiently large $k$, so considering the definition of $\leq_\tau$ and Claim \ref{clm:MS_norm_independence}, it suffices to fix some $z\in{^\om}\om$ with $\left<x_k\right>_{k<\om}^\#\leq_T z$ and $x^\#\leq_Tz$ and show that
 \begin{equation}\label{eqn:xi-vec_goal} \tau^{P^{(z^{\#\#})_n}_x}(\vec{\xi})\leq\lim_{k\to\om}\tau^{P^{(z^{\#\#})_n}_{x_k}}(\vec{\xi}),\end{equation}
 where $\vec{\xi}=\vec{\kappa}^{z^\#}_n=(\kappa^{z^\#}_0,\ldots,\kappa^{z^\#}_{n-1})$ is the increasing enumeration of the first $n$ $z^\#$-indiscernibles (as opposed to just $z$-indiscernibles). But let $\lambda=\kappa_n^{z^\#}=\crit((z^{\#\#})_n)$, the $n$th $z^\#$-indiscernible. So the $\lambda$th iterate $N$ of $M_\infty=x^\#$ is just $P^{(z^{\#\#})_n}_x$, and the $\lambda$th iterate $((x_k)^\#)_\lambda$ of $(x_k)^\#$
 is just $P^{(z^{\#\#})_n}_{x_k}$. Let $\pi:\OR^N\to\OR^{(z^{\#})_\lambda}$
 be the map defined in the proof of Claim \ref{clm:M_infty=x^sharp} (this, as there, is working with the $\lambda$th iterate of $z^\#$,
 not that of $z^{\#\#}$; note that $(z^\#)_\lambda\in(z^{\#\#})_n$,
 and moreover, $(z^\#)_\lambda$ has cardinality $\lambda$ in $(z^{\#\#})_n$).
 Let $\left<\kappa_\alpha\right>_{\alpha<\lambda}$
 and $\left<\mu_\alpha\right>_{\alpha<\lambda}$ also be as there. 
 Since $\pi$ is order-preserving, for all $\{\alpha_0,\ldots,\alpha_{n-1}\}\in[\lambda]^n$, we have
 \begin{equation}\label{eqn:kappa-vec,mu-vec_inequality}\tau^{P^{(z^{\#\#})_n}_x}(\vec{\kappa})=\tau^N(\vec{\kappa})\leq\pi(\tau^N(\vec{\kappa}))=\lim_{k\to\om}\tau^{P^{(z^\#)_\lambda}_{x_k}}(\vec{\mu})=\lim_{k\to\om}\tau^{P^{(z^{\#\#})_n}_{x_k}}(\vec{\mu}) \end{equation}
 where $\vec{\kappa}=(\kappa_0,\ldots,\kappa_{\alpha_{n-1}})$
 and $\vec{\mu}=(\mu_{\alpha_0},\ldots,\mu_{\alpha_{n-1}})$.  (The first equality here holds because as mentioned earlier,
 we have $P^{(z^{\#\#})_n}_x=N$,
 and the last equality because  $P^{(z^\#)_\lambda}_{x_k}=P^{(z^{\#\#})_n}_{x_k}$.) But now $\vec{\xi}=(\kappa_0^{z^\#},\ldots,\kappa_{n-1}^{z^\#})$,
 and note that since  $x^\#\leq_T z$
 and the $\kappa_\alpha$'s are the critical points given by iterating $x^\#$,
 we have $\kappa_{\kappa_i^{z^\#}}=\kappa_i^{z^\#}$, and since the $\mu_\alpha$'s
 are the critical points given by iterating $z^\#$ (that is, the $\kappa_\beta^{z}$s),
 we also have $\mu_{\kappa_i^{z^\#}}=\kappa_i^{z^\#}$. So taking $\{\alpha_0,\ldots,\alpha_{n-1}\}=\{\kappa^{z^\#}_0,\ldots,\kappa_{n-1}^{z^\#}\}$,
 we get $\vec{\kappa}=(\kappa_0^{z^\#},\ldots,\kappa_{n-1}^{z^\#})=\vec{\xi}=\vec{\mu}$. Therefore line (\ref{eqn:xi-vec_goal})
 follows from line (\ref{eqn:kappa-vec,mu-vec_inequality}). 
\end{proof}

This completes the proof that we have defined a scale. The fact that it is $\Delta^1_3$-definable is because $\{x^\#\}$ is a $\Pi^1_2$-singleton, uniformly in $x$.

\subsection{The prewellordering property for $\Pi^1_3$, $1$-small mice and $M_1$}\label{sec:Pi^1_3}
The prewellordering property for $\Pi^1_3$ (that is, $\mathrm{PWO}(\Pi^1_3)$) was proved under determinacy assumptions ($\Det(\bfDelta^1_2)$) independently by Martin and Moschovakis.\footnote{
In fact they proved a much more general fact, the First Periodicity Theorem \cite[Theorem 6B.1]{mosch}. This was developed from earlier ideas due to Blackwell, Addison and Martin;
see \cite[Footnote 6H.9]{mosch}.}
The next key hint towards our scale construction comes from an alternate proof of $\PWO(\Pi^1_3)$,
which instead of any determinacy assumptions/arguments, uses mouse existence assumptions and an analysis of $L[\es,x]$-constructions. We will describe this proof in this section. Woodin first proved $\PWO(\Pi^1_3)$ using inner model theoretic techniques, but his argument has not been published (it is mentioned in \cite[between Theorems 3.5 and 3.6]{steel_games_and_scales}); the argument we give here was found in approximately 2011 by the author.\footnote{After developing some of the material in the paper, including the construction in this section, the author asked Neeman about his work on the topic. Neeman then sent the slides to a talk dated 1999. The $0$th norm mentioned there is the direct analogue of the norm we use in the present section.} We will also develop some more descriptive set theory
for $\Pi^1_3$ using only inner model theoretic techniques. The material discussed here will also be used in our eventual proof of the scale property for $\Pi^1_3$
(that is, $\Scale(\Pi^1_3)$).

For this subsection  we make the following assumption:

\begin{ass}\label{ass:om_1^L[x]} Throughout this subsection (\S\ref{sec:Pi^1_3}) we assume that $\om_1^{L[x]}<\om_1$ for all 
reals $x$.
\end{ass}

We will restrict our attention to $1$-small premice in this subsection.
Because of this and by Assumption \ref{ass:om_1^L[x]}, $(n,\omega_1)$-iterability will be enough to ensure
termination of comparisons, where more generally $(n,\omega_1+1)$-iterability would be required. (This takes a standard  short argument,
like that in the proof of Lemma \ref{lem:1-small_comparison} below.)

\subsubsection{$\Pi^1_3$ and $1$-small mice}\label{subsubsec:Pi^1_3_and_1-small}

We will begin with some  material regarding the complexity of $1$-small mice and their iteration strategies.
The arguments in this subsubsection (\S\ref{subsubsec:Pi^1_3_and_1-small}) are standard and well-known
(see especially \cite{projwell}), but will be important in the proof of $\mathrm{PWO}(\Pi^1_3)$,
and later, of $\mathrm{Scale}(\Pi^1_3)$, so for self-containment, we discuss them.

\begin{dfn} Let $\Tt$ be a countable putative $n$-maximal iteration tree on a countable $n$-sound premouse. We say that
$\Tt$ is \emph{uniquely-$\Pi^1_1$-guided}
iff for every limit $\lambda<\lh(\Tt)$, there is $Q\ins M^\Tt_\lambda$
such that $Q=\J_\alpha(M(\Tt))$ for some ordinal $\alpha$ and $Q$ is a Q-structure for 
$M(\Tt\rest\lambda)$ and there is no $\Tt\rest\lambda$-cofinal branch $b\neq[0,\lambda)_\Tt$ with 
$M^{\Tt\rest\lambda}_b\ins Q$. 

For $\alpha\in\OR$ we say a structure $P$ for the language of set theory is 
\emph{$\alpha$-wellfounded} iff either $\OR^P$ is wellfounded or $\alpha$
is a subset of $\wfp(\OR^P)$, the wellfounded part of $\OR^P$.

Let $P$ be a countable $n$-sound $1$-small premouse.

We say that $P$ is \emph{$\Pi^1_2$-$n$-iterable}
iff for every $\alpha<\om_1$ and for every countable putative $n$-maximal iteration tree $\Tt$ on 
$P$ such that $\Tt$
is uniquely-$\Pi^1_1$-guided, (i) if $\Tt$ has successor length then $M^\Tt_\infty$
is wellfounded, and (ii)
if $\Tt$ has limit length then there is a $\Tt$-cofinal branch $b$
such that $M^\Tt_b$ is $\alpha$-wellfounded.

Define \emph{small-$\Pi^1_3$-$n$-iterable}
by replacing ``$\alpha$-wellfounded'' above with ``wellfounded''
(and dropping the ``for all $\alpha<\om_1$'').

If $P$ is $\om$-sound, we say that $P$ is \emph{$\Pi^1_2$-iterable}
iff $P$ is $\Pi^1_2$-$\om$-iterable, and likewise small-$\Pi^1_3$-iterable.
\end{dfn}
\begin{rem}
It is easy to see that unique-$\Pi^1_1$-guidedness is $\Pi^1_1$ and $\Pi^1_2$-$n$-iterability is 
$\Pi^1_2$, and small-$\Pi^1_3$-$n$-iterability is 
$\Pi^1_3$ (in natural codes).
And small-$\Pi^1_3$-$n$-iterability trivially implies $\Pi^1_2$-$n$-iterability.
\end{rem}
\begin{lem}
Let $P$ be a countable  $(n,\om_1)$-iterable $n$-sound 1-small premouse.
Then $P$ is small-$\Pi^1_3$-$n$-iterable. In fact, fixing any $(n,\om_1)$-iteration strategy 
$\Sigma$
for $P$, then every countable putative $n$-maximal uniquely-$\Pi^1_1$-guided tree $\Tt$
on $P$ is via $\Sigma$.
\end{lem}
\begin{proof}
 Let $\Tt$ be as described 
 but not via $\Sigma$. Let $\lambda<\lh(\Tt)$ be least such that 
\[ b=\Sigma(\Tt\rest\lambda)\neq[0,\lambda)_\Tt=c \]
and $Q\ins M^\Tt_c$ be the Q-structure for $M(\Tt)$ given by
unique-$\Pi^1_1$-guidedness.
Then $M^\Tt_b$ is wellfounded, and $M^\Tt_b\nins Q$ by
unique-$\Pi^1_1$-guidedness. It follows that $Q\pins M^\Tt_b$, but this contradicts the 
Zipper Lemma.
\end{proof}
\begin{lem}\label{lem:limit_trees_uniquely-Pi^1_1-guided}
Let $P$ be a countable  small-$\Pi^1_3$-$n$-iterable  $n$-sound $1$-small premouse.
Then $P$ is $(n,\om_1)$-iterable. Moreover, if $\Tt$ is any limit length $n$-maximal tree on $P$ 
then $\Tt$ is uniquely-$\Pi^1_1$-guided.
\end{lem}
\begin{proof}
Define an $(n,\om_1)$-strategy $\Sigma$ for $P$ by:
Given $\Tt$ on $P$ of limit length, let $\Sigma(\Tt)$ be some $\Tt$-cofinal branch
$b$ which minimizes $\OR(M^\Tt_b)$. (If there is some $b$
such that $M^\Tt_b\sats$``$\delta(\Tt)$ is Woodin'',
then $\Sigma(\Tt)$ will choose some such $b$, and note that 
$\Tt\conc b$ cannot be properly continued. Otherwise, there is a unique $b$
such that the Q-structure $Q$ for $M(\Tt)$ has $Q\pins M^\Tt_b$.)
\end{proof}
\begin{rem}\label{rem:1-small_it_Pi^1_3}
 So for countable $1$-small $n$-sound premice $P$,
 $(n,\om_1)$-iterability is $\Pi^1_3$.
\end{rem}
\begin{lem}\label{lem:1-small_comparison}
 Let $P$ be a countable, $p$-sound, $(p,\om_1)$-iterable premouse,
 and $\Sigma$ a $(p,\om_1)$-strategy for $P$.
 Let $Q,q,\Gamma$ be likewise, and suppose that $Q$ is $1$-small. Then there is a successful 
comparison $(\Tt,\Uu)$ of $(P,Q)$
 via $(\Sigma,\Gamma)$, of countable length.
\end{lem}
\begin{proof}
 Suppose not. Then we get a comparison $(\Tt,\Uu)$ as in the statement of the lemma, but of length $\om_1$.
 Because $Q$ is $1$-small and by Lemma \ref{lem:limit_trees_uniquely-Pi^1_1-guided},
 $\Uu$ is uniquely-$\Pi^1_1$-guided.
 In fact, $\Tt$ is also, because at each limit 
stage $\eta$ of $\Tt$ we have $M(\Tt\rest\eta)\ins M^\Uu_\eta$, and because there is some 
$\alpha\geq\eta$ such that $E^\Uu_\alpha\neq\emptyset$, and $M^\Uu_\eta$ is $1$-small, there must 
be a $1$-small Q-structure $Q$ for $M(\Tt\rest\eta)$ with $Q\pins M^\Uu_\eta$,
which implies that $Q\pins M^\Tt_\eta$ also.

But now consider forming this comparison inside $L[x]$ where $x$ is a real with $P,Q\in\HC^{L[x]}$.
This can be done because the comparison process is determined by the models,
and $L[x]$ has the ordinals to compute the Q-structures $Q$ at limit stages,
and an easy absoluteness argument then gives that the (unique) branches are in 
$L[x]$. But because $\om_1^{L[x]}<\om_1$ by Assumption \ref{ass:om_1^L[x]}, this contradicts the fact
that comparison terminates. 
\end{proof}
\begin{rem}
 Comparison type arguments need this sort of component throughout this section,
 and we leave it to the reader to supply it. It is likewise in the following:
\end{rem}

\begin{lem}
 Let $P$ be an $(\om,\om_1)$-iterable
 $1$-small $\om$-premouse. Then there is a unique $(\om,\om_1)$-strategy for $P$,
denoted $\Sigma_P$.
\end{lem}

\begin{lem}[Overspill for $\Pi^1_3$]\label{lem:overspill_mouse}
 Let $P$ be a $\Pi^1_2$-iterable $1$-small $\om$-premouse, which is not $(\om,\om_1)$-iterable.
 Then $Q\pins P$ for every $1$-small $\om$-mouse $Q$.
\end{lem}
\begin{proof}
 Suppose not. Then $Q\nins P\nins Q$. Consider the comparison $(\Tt,\Uu)$ of $(P,Q)$, using 
$\Sigma_Q$ to form $\Uu$, and taking $\Tt$ to be uniquely-$\Pi^1_1$-guided,
as far as possible. If we have produced $(\Tt,\Uu)\rest\lambda$,
and either $\lambda=\om_1$, or $\Tt\rest\lambda$ cannot be properly extended as a 
uniquely-$\Pi^1_1$-guided tree,
then we stop the process.
 The comparison cannot succeed by standard fine structure.
 
 Suppose we produce $(\Tt,\Uu)$ of length $\lambda<\om_1$, and $\Tt\rest\lambda$ cannot be properly 
extended as a uniquely-$\Pi^1_1$-guided tree. Then
$\Tt\rest\lambda$ has limit length ignoring padding, and for simplicity we assume
that $\Uu\rest\lambda$ is likewise. Since $Q$ is iterable,
we get a wellfounded $M^\Uu_\lambda$.
Because $Q$ is an $\om$-mouse, there is a Q-structure $Q\ins M^\Uu_\lambda$ for $M(\Tt,\Uu)$.
Let $\alpha=\OR^Q$. By $\Pi^1_2$-iterability there is a $\Tt\rest\lambda$-cofinal
branch $b$ such that $M^\Tt_b$ is $\alpha$-wellfounded.
If $Q\pins M^\Tt_b$ then we could properly extend $\Tt$, contradiction.
So $M^\Tt_b\ins Q$. Since $P$ is an $\om$-premouse,
then $M^\Tt_b=Q=M^\Uu_\lambda$. But this contradicts standard fine structure.

So the comparison must produce $(\Tt,\Uu)$ of length $\om_1$.
But because both trees are uniquely-$\Pi^1_1$-guided, we can argue as in 
\ref{lem:1-small_comparison} for a contradiction.
\end{proof}

\begin{dfn}\label{dfn:varphi(w)-witness}
Let $\psi$ be a $\Pi^1_1$ formula, and $\varphi$ be the $\Pi^1_3$ formula
\[ \varphi(w)\ \iff\ \all x\ \ex y\ \psi(w,x,y).\]
Let $T$ be the tree on $\om^4$ for $\neg\psi$.
Given reals $w,x,y$, let $T_{wxy}$ be the
 section of $T$ at $(w,x,y)$, so $\psi(w,x,y)$ iff $T_{wxy}$ is wellfounded.
 
 Given $w\in\RR$, a \emph{$\varphi(w)$-prewitness} is a  $1$-small $w$-premouse such that
 for some $\delta<\eta<\OR^P$,
 $P\sats$``$\delta$ is Woodin
 and the extender algebra $\BB_{\delta}$ at $\delta$
 forces that, letting $x$ be the generic real, 
there is $y\in\RR$ 
and a rank function $\rho:T_{wxy}\to\eta$''.

 A \emph{$\varphi(w)$-witness} is a countable $(0,\om_1)$-iterable $\varphi(w)$-prewitness.
\end{dfn}

 Note that being a $\varphi(w)$-prewitness is preserved by non-dropping $0$-maximal iterations.
 
\begin{lem}\label{lem:varphi(w)-witness}
Let $\varphi$ be a $\Pi^1_3$ formula. Then:

\begin{enumerate}
\item\label{item:if_varphi(w)-witness_exists_then_varphi(w)} If there is a $\varphi(w)$-witness then:
\begin{enumerate}[label=--] \item $\varphi(w)$ holds,
 and
 \item there is a sound $\varphi(w)$-witness $P$ such that $\rho_1^P=\om$.
 \end{enumerate}
 
 \item\label{item:if_M_1^sharp(w)_exists_then_varphi(w)-witness} If $M_1^\#(w)$ exists \tu{(}is $(\om,\om_1)$-iterable\tu{)} and $\varphi(w)$ holds then there is a 
$\varphi(w)$-witness.
\end{enumerate}
\end{lem}
\begin{proof}
For simplicity assume $w=0$.

Part \ref{item:if_varphi(w)-witness_exists_then_varphi(w)}:
Suppose there is a $\varphi$-witness $P$,
and $\psi,\delta,\eta$ are as in \ref{dfn:varphi(w)-witness}. Let $x\in\RR$. Because $P$ is iterable, we can 
form a genericity iteration on $P$
 making $x$ generic for the extender algebra. (The genericity iteration terminates by the same 
argument as in the proof of \ref{lem:1-small_comparison}.) Let $Q$ be the resulting iterate of $P$. 
Then $Q[x]\sats$``There is $y\in\RR$ 
such that $T_{x,y}$ is rankable'' and since $Q$ is wellfounded, this verifies that $\psi(x,y)$ 
holds. So $\varphi$ holds.

Now let us verify that we can obtain a sound $\varphi(w)$-witness $\bar{P}$
such that $\rho_1^{\bar{P}}=\om$. For this, let
\[ \bar{P}=\cHull_1^P(\{\delta\})\]
and $\pi:\bar{P}\to P$ be the uncollapse.
Then clearly $\rho_1^{\bar{P}}=\om$ as $\bar{P}=\Hull_1^{\bar{P}}(\{\bar{\delta}\})$ where $\pi(\bar{\delta})=\delta$.
Letting $\eta<\OR^P$ be least as in \ref{dfn:varphi(w)-witness},
note that $\eta\in\rg(\pi)$ (we have $\RR^{P[x]}\sub L_{\delta}[P|\delta,x]$ whenever $x$ is $(P,\BB_\delta)$-generic), and letting $\pi(\bar{\eta})=\eta$, that $\bar{\eta}$ witnesses that $\bar{P}$ is 
a $\varphi(w)$-prewitness. Finally, note that 
$\Hull_1^{\bar{P}}(\bar{\delta})=\bar{P}|\bar{\delta}$ by condensation. Therefore $p_1^{\bar{P}}=\{\delta^{\bar{P}}\}$ and $\bar{P}$ is $1$-sound,
as desired.

Part \ref{item:if_M_1^sharp(w)_exists_then_varphi(w)-witness}:
Assume that $M_1^\#$ exists and $\varphi$ holds. It suffices to find some 
$N\pins M_1$
which is a $\varphi$-prewitness, as clearly the least such is countable.
But if there is no such $N$ then letting $\alpha$ be a large ordinal,
we have that $M_1|\alpha\sats$ ``It is forceable with $\BB_{\delta^{M_1}}$
that letting $x$ be the generic real, for all $y\in\RR$, $T_{0xy}$ is illfounded''.
So let $G$ be $M_1$-generic for the extender algebra, with $p\in G$ forcing this.
Then we get $x\in M_1[G]$ and $M_1[G]\sats$ ``$\all y\in\RR\ \neg\psi(x,y)$''.
But this statement is $\Pi^1_2$, and $M_1[G]$ is $\Pi^1_2$-correct,
so $\neg\varphi$, contradiction.
\end{proof}

\begin{dfn}
We say \emph{$\Pi^1_3$ is witnessed by mice}
iff for every $\Pi^1_3$ formula $\varphi$ and $x\in\RR$,
$\varphi(x)$ holds iff there is a $\varphi(x)$-witness.
\end{dfn}
\begin{rem}
Of course, by \ref{lem:varphi(w)-witness}, one direction of \emph{$\Pi^1_3$ is witnessed by mice} 
is already true (this uses Assumption \ref{ass:om_1^L[x]}).
Moreover, if
 $M_1^\#(x)$ exists for every $x\in\RR$,
then $\Pi^1_3$ is indeed witnessed by mice. We do not know whether the converse of the latter holds.
\end{rem}
\subsubsection{The construction $\DD$ and comparisons thereof}

\begin{dfn}\index{$\DD$}\index{$N_\alpha$}\index{$F_\alpha$}\index{$t_\alpha$}\label{dfn:DD}
 Let $P$ be a $1$-small premouse, $x\in\RR^P$ and $\lambda\leq\OR^P$ be a limit. We define
 \[ \DD=\DD_x^P=\left<N_\alpha,F_\alpha\right>_{\alpha\in\Lim\inter(0,\lambda]}, \]
consisting of  $x$-premice $N_\alpha$ and background extenders $F_\alpha$.
We set $F_\alpha=\emptyset$  
except where we 
explicitly say otherwise.
We define $\CC\rest(\alpha+1)$ by induction on limits $\alpha$.
We will 
state certain assumptions throughout the construction,
and if any of those assumptions fail at any stage, then we stop the construction.
Note that $P$ need not model $\ZFC$. In order to define $N_\alpha$, we first make the assumption
\begin{enumerate}[label=\tu{(}A\arabic*\tu{)}]
 \item\label{DDconass:N_is_P-con}  $\aleph_\beta^P<\OR^P$ for every $\beta<\alpha$.
\end{enumerate}

\begin{case} $\alpha=\om$

Set $N_\om=(V_\om,x)$.
\end{case}

\begin{case} $\alpha$ is a limit of limits.

Set $N_\alpha=\liminf_{\beta<\alpha}N_\beta$.\footnote{This is defined just like in traditional $L[\es,x]$-constructions,
such as in \cite[\S11]{fsit}.}
\end{case}

\begin{case} $\alpha=\beta+\om$ where $\beta$ is a limit.

\begin{scase} $N_\beta$ is passive and there is $F$ such that there is $E$ such that:
\begin{enumerate}[label=--]
\item $F\in\es^P$ and $\nu_F=\aleph_\beta^P$ (so $F$ is $P$-total),
\item $(N_\beta,E)$ is an active $1$-small premouse (so $E\neq\emptyset$), and
\item $E\rest\nu_E\sub F$.
\end{enumerate}
Let $F$ witness this with $\lh(F)$ minimal, and $E$ be induced by $F$. Then 
\[ N_{\beta+\om}=(N_\beta,E)\text{ and }F_{\beta+\om}=F\text{ and }t_{\beta+\om}=1.\]
\end{scase}

\begin{scase}Otherwise.

Then we assume
\begin{enumerate}[label=\tu{(}A2\tu{)}]
\item\label{DDconass:segs_good}  $N_\beta$ is $\om$-good and $\core_\om(N_\beta)$ satisfies condensation,
\footnote{These notions were defined in \S\ref{sec:notation}.}
 \end{enumerate}
and set $N_{\beta+\om}=\J(\core_\om(N_\beta))$.
\end{scase}
\end{case}

This completes the definition of $\DD$.

We write $N^P_{x\alpha}=N^\DD_\alpha=N_\alpha\text{ and }F^P_{x\alpha}=F^\DD_\alpha=F_\alpha\text{ 
and 
}t^P_{x\alpha}=t^\DD_\alpha=t_\alpha$.

We say that $N^P_{x\alpha}$ \emph{exists} iff it is defined (that is, all assumptions required to 
define $N^P_{x\alpha'}$ for all limits $\alpha'\leq\alpha$, hold).\index{$N_{\alpha}$ exists}

Suppose $N^P_{x\alpha}$ exists. We write $N^P_{x\alpha 0}=N^P_{x\alpha}$,
and if $N^P_{x\alpha}$ is $n$-good and $k\leq\min(\om,n+1)$,
write
\[ N^P_{x\alpha k}=\core_k(N^P_{x\alpha})^\unsq,\]
and if also $0\leq j\leq k$
then\[\tau^P_{x\alpha kj}:(N^P_{x\alpha k})^\sq\to(N^P_{x\alpha j})^\sq \]
denotes the core embedding. We say that $N^P_{x\alpha 0}$ \emph{exists},\index{$N_{\alpha n}$ exists}
say that $N^P_{x\alpha,n+1}$ \emph{exists} iff $N^P_{x\alpha}$ is $n$-good,
and say that $N^P_{x\alpha\om}$ \emph{exists} iff $N^P_{x\alpha}$ is $\om$-good.
\end{dfn}

\begin{dfn}
 Let $Y,Y'$ be $1$-small premice and $x\in\RR^Y\inter\RR^{Y'}$.
 
 Suppose that $N^Y_{\lambda x}$ and $N^{Y'}_{\lambda x}$ exist. Let 
$\DD^Y_x=\left<N_\alpha,F_\alpha\right>_{\alpha\leq\lambda}$ and $\DD^{Y'}_x=\left<N'_\alpha,F'_\alpha\right>_{\alpha\leq\lambda}$.
Say that $Y,Y'$ are $(\lambda,x)$-\emph{compatible}
iff for every $\beta+\om\leq\lambda$, we have  $F_{\beta+\om}\inter X=F'_{\beta+\om}\inter X$ where
$\theta=\min(\lh(F_{\beta+\om}),\lh(F'_{\beta+\om}))$ and $X=(Y\inter 
Y')\cross[\theta]^{<\om}$.\footnote{Note that 
$\theta=\min(\lh(F_{\beta+\om}),\lh(F'_{\beta+\om}))$, not $\min(\nu(F_{\beta+\om}),\nu(F'_{\beta+\om}))$.}

Say that $Y,Y'$ are \emph{$x$-compatible} iff they are $(\lambda,x)$-compatible for every $\lambda$
such that $N^Y_{x\lambda},N^{Y'}_{x\lambda}$ both exist.
\end{dfn}

\begin{rem}
Note that $(\lambda,x)$-compatibility
implies in particular that for each $\beta+\om\leq\lambda$, we have $F_{\beta+\om}\neq\emptyset$ iff $F'_{\beta+\om}\neq\emptyset$.
 An easy induction gives that if $Y,Y'$ are $(\lambda,x)$-compatible 
then $N_\alpha=N'_\alpha$ for all $\alpha\leq\lambda$.
The following fact is proven with standard methods; see \cite[\S12]{fsit}. We will moreover give a detailed proof of a variant, in Lemma \ref{lem:iterability}.
\end{rem}

\begin{fact}
 Let $P$ be a $(0,\omega_1)$-iterable $1$-small premouse. Let $\lambda\in\Lim$
 be such that $\aleph_\lambda^P\leq\OR^P$.
 Let $x\in\RR^P$.
 Then $N^P_{x\lambda}$ exists,
 and is $(0,\omega_1)$-iterable.
\end{fact}

The following fact is also proved via standard methods; see \cite[\S11]{fsit}. We state a more detailed variant later in Lemma \ref{lem:projections_to_cardinals}.
\begin{fact}\label{fact:DD_stages_proj_to_card_of_N_alpha}
 Let $P$ be a premouse, $x\in\RR^P$
 and $\lambda\in\Lim$ be such that $N^P_{x\lambda}$ exists.
 Let $\rho$ be a cardinal of $N^P_{x\lambda}$.
 Then for each $M\pins N^P_{x\lambda}$
 such that $\rho_\om^M=\rho$,
 there is a unique $\xi<\lambda$
 such that $M=\core_\om(N^P_{x\xi})$.
\end{fact}

\begin{lem}\label{lem:add_ext_non-ZFC}
 Let $P$ be a $1$-small premouse, $x\in\RR^P$
 and $\lambda\in\Lim$ be such that $N^P_{x,\lambda+\om}$ exists
 and is active. 
 Then $P|\aleph_\lambda^P\not\sats\ZFC$, and therefore neither $\lambda$ nor $\aleph_\lambda^P$ are measurable in $P$.
\end{lem}
\begin{proof}
We may assume that $\lambda=\aleph_\lambda^P$.
 Note that $\CC^P_x\rest(\lambda+1)$
 and hence also $N_{\lambda x}^P$ are both definable
 over $P|\lambda=P|\aleph_\lambda^P$.
 Since $N_{\lambda+\om,x}^P$ exists,
 $\aleph_{\lambda+\om}^P\leq\OR^P$,
 so $N_{\lambda x}^P\in P$.
 And since $N_{\lambda+\om,x}^P$
 is active and its passivization is just $N_{\lambda x}^P$, these have
 a largest cardinal $\rho$.
 Therefore $N_{\lambda x}^P\in P|\lambda=P|\aleph_\lambda^P$.
 But by Fact \ref{fact:DD_stages_proj_to_card_of_N_alpha},
 $\lambda$ is the supremum of the ordinals $\xi<\lambda$
 such that for some $M\pins N_{x\lambda}^P$ with $\rho_\om^M=\rho$, we have $M=\core_\om(N_{x\xi}^P)$.
 Also by the uniqueness in that fact, we get a cofinal function $\OR(N^P_{x\lambda})\to\lambda$
 definable over $P|\lambda$, which suffices.
\end{proof}

\begin{lem}\label{lem:first_comparison}
Let $R$ be a countable $1$-small
premouse and $\Sigma$ be a $(0,\om_1)$-iteration strategy for $R$.
Let $S,\Lambda$ be likewise.
Then there are iteration trees $\Tt,\Uu$ on $R,S$,
via $\Sigma,\Lambda$ respectively (in particular, $0$-maximal trees), of countable 
successor length, nowhere dropping in model or degree, such that $M^\Tt_\infty,M^\Uu_\infty$
are $x$-compatible for every $x\in\RR\inter R\cap S$.
\end{lem}
\begin{proof}
We produce $\Tt,\Uu$ by a kind of comparison. The trees we literally construct might be padded, but equivalent to trees as required for the lemma.
Suppose we have defined $(\Tt,\Uu)\rest(\gamma+1)$, with both trees nowhere dropping,
 but they do not witness the lemma. Let $Y=M^\Tt_\gamma$ and $Z=M^\Uu_\gamma$.
 For $x\in\RR\inter R\inter S$ let
 \[ \DD^Y_x=\left<M_{x,\alpha},F_{x,\alpha}\right>_{\alpha\leq\Upsilon}, \]
 \[ \DD^Z_x=\left<N_{x,\alpha},G_{x,\alpha}\right>_{\alpha\leq\Psi},\]
with $\Upsilon$ the largest limit such that $\aleph_{\Upsilon}^Y\leq\OR^Y$, and $\Psi$ likewise.

 Let $X_0^Y$ be the set of all $x\in\RR\inter R\inter S$ such that for some 
$\alpha<\min(\Upsilon,\Psi)$, $Y,Z$ are 
$(\alpha,x)$-compatible but not $(\alpha+\om,x)$-compatible,
and  $F_{x,\alpha+\om}\neq\emptyset$. For 
$x\in X_0^Y$, let $\alpha^Y_x$ be the corresponding  $\alpha$. For $x\in\RR\inter R\inter S\cut 
X_0^Y$, 
let $\alpha^Y_x=\infty$. Define $X_0^Z$ and $\alpha^Z_x$ likewise. Note that $X_0^Y\un 
X_0^Z\neq\emptyset$ (because $\Tt,\Uu$ do not witness the lemma). Let 
$\alpha_x=\min(\alpha^Y_x,\alpha^Z_x)$,
and $\alpha=\min_{x\in\RR\inter R\inter S}\alpha_x$. So $\alpha<\infty$.
Let $X_1^Y=\{x\in X_0^Y\mid\alpha^Y_x=\alpha\}$, and $X_1^Z$ likewise. So $X_1^Y\un 
X_1^Z\neq\emptyset$.

If $X_1^Y\neq\emptyset$ then let $x\in X_1^Y$ be such that $\lh(F_{x,\alpha+\om})$ is least possible,
and we set $E^\Tt_\gamma=F_{x,\alpha+\om}$. Also let $x^\Tt_\gamma=$ some such $x$. Otherwise set
$E^\Tt_\gamma=x^\Tt_\gamma=\emptyset$.
We define $E^\Uu_\gamma$ and $x^\Uu_\gamma$ likewise from $X^Z_1$. Define
$\alpha_\gamma=\alpha$.

This completes the definition of the comparison (since at limit stages, we use $\Sigma,\Lambda$ to continue).

The first claim below is immediate from the definitions:
\begin{clm}\label{clm:nu(E^Tt_gamma)}
For every $\gamma+1<\lh(\Tt,\Uu)$,
if $E^\Tt_\gamma\neq\emptyset$
then $\nu(E^\Tt_\gamma)=\aleph_{\alpha_\gamma}^{M^\Tt_\gamma}$, and likewise for $\Uu$. Therefore
 $\Tt,\Uu$ do not drop anywhere.
\end{clm}

\begin{clm}\label{clm:simultaneous_incompatible}
 Suppose $E^\Tt_\gamma\neq\emptyset\neq E^\Uu_\gamma$. Let
$\theta=\min(\lh(E^\Tt_\gamma),\lh(E^\Uu_\gamma))$ and
$A= M^\Tt_\gamma\inter M^\Uu_\gamma$.
 Then
 \[ E^\Tt_\gamma\rest(A\cross[\theta]^{<\om})\neq E^\Uu_\gamma\rest(A\cross[\theta]^{<\om}).\]
\end{clm}
\begin{proof}
 Suppose not. Let $x=x^\Tt_\gamma$ and $y=x^\Uu_\gamma$ and $Y=M^\Tt_\gamma$ and $Z=M^\Uu_\gamma$ 
and $\CC=\DD^Y_x$ and $\DD=\DD^Z_x$
 and $N=N^Y_{\alpha_\gamma x}$ and $(N,E)=N^Y_{\alpha_\gamma+\om,x}$. We have
$E^\Tt_\gamma=F^\CC_{\alpha_\gamma+\om}$, so $\nu(E^\Tt_\gamma)=\aleph_{\alpha_\gamma}^Y$.

Because $Y,Z$ are 
$(x,\alpha_\gamma)$-compatible, we have $N_{\alpha_\gamma}^\DD=N$. In particular, $N\sub 
M^\Tt_\gamma\inter M^\Uu_\gamma$.
Now
\[ E\rest\nu_E\sub 
E^\Tt_\gamma\rest\nu(E^\Tt_\gamma)=F^\CC_{\alpha_\gamma+\om}\rest\aleph_{\alpha_\gamma}^Y, \]
\[ 
\OR^N\leq\alpha_\gamma\leq\min(\aleph_{\alpha_\gamma}^Y,\aleph_{\alpha_\gamma}
^Z)=\min(\nu(E^\Tt_\gamma),\nu(E^\Uu_\gamma))<\theta. \]
But then because $N\sub M^\Tt_\gamma\inter M^\Uu_\gamma$ and by the contradictory hypothesis, we have
\[ E^\Tt_\gamma\inter (N\cross[\theta]^{<\om})=E^\Uu_\gamma\inter(N\cross[\theta]^{<\om}), \]
and therefore $E^\Uu_\gamma$ works as a background extender for $(N,E)$ in $\DD$.
Therefore $F_{\alpha_\gamma+\om}^\DD\neq\emptyset$, so $x\in X^{M^\Uu_\gamma}_1$. But because $Y,Z$ are not 
$(x,\alpha+\om)$-compatible and again by the contradictory hypothesis, we have $F_{\alpha_\gamma+\om}^\DD\neq E^\Uu_\gamma$,
and so by the minimality of $\lh(F_{\alpha_\gamma+\om}^\DD)$ (in the definition of $\DD$), 
$\lh(F_{\alpha_\gamma+\om}^\DD)<\lh(E^\Uu_\gamma)$. This contradicts the minimality of the choice 
of $E^\Uu_\gamma$.
\end{proof}

\begin{clm}\label{clm:alpha_gamma<alpha_eps}
 $\alpha_\gamma<\alpha_\eps$ for all
$\gamma+1<\eps+1<\lh(\Tt,\Uu)$.
Therefore (by Claim \ref{clm:nu(E^Tt_gamma)})
$\Tt,\Uu$ have the increasing length condition.
\end{clm}
\begin{proof}
We consider the case that $\eps=\gamma+1$; the general case is then an easy adaptation.
Let $Y=M^\Tt_\gamma$ and $Z=M^\Uu_\gamma$.

Suppose first that $E^\Tt_\gamma\neq\emptyset$ but $E^\Uu_\gamma=\emptyset$.
Because $\nu(E^\Tt_\gamma)=\aleph_{\alpha_\gamma}^{Y}$,
for each $x$, $\DD_x^{Y}$ and $\DD_x^{M^\Tt_{\gamma+1}}$ agree through all stages 
$\leq\alpha_\gamma$, because these only involve extenders with strengths 
$<\aleph_{\alpha_\gamma}^{Y}$, over which $Y$ and $M^\Tt_{\gamma+1}$ agree.
Also, we have $M^\Uu_{\gamma+1}=Z$, and therefore $M^\Tt_{\gamma+1},M^\Uu_{\gamma+1}$ 
are $(x,\alpha_\gamma)$-compatible for all $x$. So $\alpha_\gamma\leq\alpha_{\gamma+1}$.

Now
$\aleph_{\alpha_\gamma}^{Y}=\aleph_{\alpha_\gamma}^{M^\Tt_{\gamma+1}}$
and
$\lh(E^\Tt_\gamma)=\aleph_{\alpha_\gamma+\om}^{M^\Tt_{\gamma+1}}$.
So for any $x$, if
\[ E^*\eqdef F_{\alpha_\gamma+\om}^{\DD_x(M^\Tt_{\gamma+1})}\neq\emptyset,\]
then $\lh(E^*)<\lh(E^\Tt_\gamma)$, so $E^*\in\es^Y$. It follows that 
$F_{\alpha_\gamma+\om}^{\DD_x^Y}=E^*$ and, since $\lh(E^*)<\lh(E^\Tt_\gamma)$,
that $Y,Z$ are $(x,\alpha_\gamma+\om)$-compatible,
but then $M^\Tt_{\gamma+1},Z$ are also $(x,\alpha_\gamma+\om)$-compatible.

Conversely, suppose $F^*\eqdef F_{\alpha_\gamma+\om}^{\CC_x^Z}\neq\emptyset$.
Since $E^\Uu_\gamma=\emptyset$, we have that $Y,Z$ are 
$(x,\alpha_\gamma+\om)$-compatible. So
$E^*\eqdef F^{\CC_x^Y}_{\alpha_\gamma+\om}\neq\emptyset$ and
\begin{equation}\label{eqn:compatibility} 
E^*\inter(A\cross[\theta]^{<\om})=F^*\inter(A\cross[\theta]^{<\om}) \end{equation}
where $A=Y\inter Z$ and $\theta=\min(\lh(E^*),\lh(F^*))$. Note
\begin{equation}\label{eqn:alpha_gamma_small} 
\alpha_\gamma\leq\min(\aleph_{\alpha_\gamma}^Y,\aleph_{\alpha_\gamma}^Z)=\min(\nu(E^*),
\nu(F^*))<\theta.\end{equation}

We claim that $\lh(E^*)<\lh(E^\Tt_\gamma)$, which clearly suffices. So suppose 
$\lh(E^*)\geq\lh(E^\Tt_\gamma)$. Let $y=x^\Tt_\gamma$.
If $E^*=E^\Tt_\gamma$ then note that by line (\ref{eqn:compatibility}), 
$F^{\DD_y^Z}_{\alpha_\gamma+\om}\neq\emptyset$,
but then $E^\Uu_\gamma\neq\emptyset$, contradiction.
So $\lh(E^*)>\lh(E^\Tt_\gamma)$.

Let $\kappa=\crit(F^*)=\crit(E^*)$. Let $X$ be the set of all 
$\beta<\kappa$ such that $F^{\DD_y^Y}_{\beta+\om}\neq\emptyset$;
 equivalently, $F^{\DD_y^Z}_{\beta+\om}\neq\emptyset$. So $X\in 
Y\inter Z$. Let $U=\Ult(Y,E^*)$.
Then $E^\Tt_\gamma\in\es^U$ and note that $E^\Tt_\gamma=F^{\DD^U_y}_{\alpha_\gamma+\om}$.
Therefore $\alpha_\gamma\in i^{Y}_{E^*}(X)$. By lines (\ref{eqn:compatibility}) and 
(\ref{eqn:alpha_gamma_small}) then, $\alpha_\gamma\in i^{Z}_{F^*}(X)$. So letting $W=\Ult(Z,F^*)$, we have
 $F^{\DD^{W}_y}_{\alpha_\gamma+\om}\neq\emptyset$.
So by coherence, $F^{\DD^{Z}_y}_{\alpha_\gamma+\om}=F^{\DD^{W}_y}_{\alpha_\gamma+\om}\neq\emptyset$. 
But then  $E^\Uu_\gamma\neq\emptyset$, contradiction.

The case that $E^\Tt_\gamma\neq\emptyset\neq E^\Uu_\gamma$ is very similar, so we leave it to the reader.
\end{proof}

The following claim is an easy consequence of the preceding claims and Lemma \ref{lem:add_ext_non-ZFC}.
\begin{clm}\label{clm:nu_not_measurable}
 Let $\gamma+1<\lh(\Tt,\Uu)$ be such that $E^\Tt_\gamma\neq\emptyset$. Let 
$\kappa=\nu(E^\Tt_\gamma)$.
 Then for all $\beta>\gamma$,
 \[ \kappa=\aleph_{\alpha_\gamma}^{M^\Tt_\gamma}=\aleph_{\alpha_\gamma}^{M^\Tt_\beta}, \]
 \[ \lh(E^\Tt_\gamma)=(\kappa^+)^{M^\Tt_\beta}=\aleph_{\alpha_\gamma+1}^{M^\Tt_\beta}, \]
 and $M^\Tt_\beta\sats$``$\kappa$ is not measurable''. Therefore if 
$\beta+1<\lh(\Tt,\Uu)$ and $E^\Tt_\beta\neq\emptyset$ then $\kappa\neq\crit(E^\Tt_\beta)$.
 Likewise for $\Uu$.
\end{clm}

\begin{clm}\label{clm:first_comparison_termination}
 The comparison terminates in countably many stages.
\end{clm}
\begin{proof}
 Suppose not. Then we get $\lh(\Tt,\Uu)=\om_1$. Let $a\in\RR$ be such that  $R,S\in\HC^{L[a]}$.
 Then since $R,S$ are $1$-small
 and $\Tt,\Uu$ are $0$-maximal trees
 of limit length, they are guided by Q-structures, and therefore
  $\Tt,\Uu\in L[a]$.
 By Assumption \ref{ass:om_1^L[x]},
 we have $\om_1^{L[a]}<\lh(\Tt,\Uu)$.
 
 Now work in $V=L[a]$, so now $\om_1=\om_1^{L[a]}<\lh(\Tt,\Uu)$.
 As usual,  let $\pi:H\to V_\eta$ be elementary 
with $H$ countable transitive, and everything relevant in $\rg(\pi)$. Let $\kappa=\crit(\pi)$.
The usual argument shows that
\begin{equation}\label{eqn:it_map_agmt} i^\Tt_{\kappa,\om_1}\rest A = \pi\rest 
A=i^\Uu_{\kappa,\om_1}\rest A\end{equation}
where
$A=M^\Tt_\kappa\inter M^\Uu_\kappa$.
Of course,
\begin{equation}\label{eqn:pow(kappa)_stable}
\all\alpha\in[\kappa,\om_1]\left[M^\Tt_\alpha\inter\pow([\kappa]^{<\om})=
M^\Tt_\kappa\inter\pow([\kappa]^{<\om})\right], \end{equation}
and likewise for $\Uu$.

Let $\gamma+1=\min((\kappa,\om_1]_\Tt)$ and $\zeta+1=\min((\kappa,\om_1]_\Uu)$.
We may assume  $\gamma\leq\zeta$.
By Claim \ref{clm:nu_not_measurable}
and the normality of $\Tt,\Uu$, we have
$\crit(i^\Tt_{\gamma+1,\om_1})>\lh(E^\Tt_\gamma)$ and 
$\crit(i^\Uu_{\zeta+1,\om_1})>\lh(E^\Uu_\zeta)$.
Together with line (\ref{eqn:it_map_agmt}), it follows that
\begin{equation}\label{eqn:compat} 
E^\Tt_\gamma\rest(A\cross[\theta]^{<\om})=E^\Uu_\zeta\rest(A\cross[\theta]^{<\om}) \end{equation}
where $\theta=\min(\lh(E^\Tt_\gamma),\lh(E^\Uu_\zeta))$.
 So by Claim \ref{clm:simultaneous_incompatible}, $\gamma<\zeta$. So by Claim \ref{clm:alpha_gamma<alpha_eps} and 
the 
ISC, we have
$E^\Uu_\zeta\rest\aleph_{\alpha_\gamma}^{M^\Uu_\zeta}\in\es(M^\Uu_\zeta)$,
$\aleph_{\alpha_\gamma}^{M^\Uu_\gamma}=\aleph_{\alpha_\gamma}^{M^\Uu_\zeta}$, and
$E^*\eqdef E^\Uu_\zeta\rest\aleph_{\alpha_\gamma}^{M^\Uu_\gamma}\in\es(M^\Uu_\gamma)$.

Let $x=x^\Tt_\gamma$.
Write $\CC_{\gamma x}=\DD^{M^\Tt_\gamma}_x$ and 
$\DD_{\gamma x}=\DD^{M^\Uu_\gamma}_x$
and
$N=N^{\CC_{\gamma x}}_{\alpha_\gamma}=N^{\DD_{\gamma x}}_{\alpha_\gamma}$
(the equality holds by $(x,\alpha_\gamma)$-compatibility).
We have $N\sub M^\Tt_\gamma\inter 
N^\Tt_\gamma$. Clearly
\[ 
\OR^N\leq\alpha_\gamma\leq\min(\aleph_{\alpha_\gamma}^{M^\Tt_\gamma},\aleph_{\alpha_\gamma}^{
M^\Uu_\gamma})\leq\theta.\]
Let $E=F(N_{\alpha_\gamma+\om}^{\CC_{\gamma x}})$. So
$E\rest\nu_E\sub 
E^\Tt_\gamma\rest\nu(E^\Tt_\gamma)=E^\Tt_\gamma\rest\aleph_{\alpha_\gamma}^{M^\Tt_\gamma}´$. But then by the preceding remarks and the 
compatibility of $E^\Tt_\gamma$ with $E^\Uu_\zeta$,
\[ E\rest\nu_E\sub E^*\rest\aleph_{\alpha_\gamma}^{M^\Uu_\gamma}=E^*\rest\nu(E^*).\]
So $E^*$ is a suitable 
candidate for a background extender for $\DD_{\gamma x}$ at stage $\alpha_\gamma+\om$. So 
$F^{\DD_{\gamma x}}_{\alpha_\gamma+\om}\neq\emptyset$ and
$\lh(F^{\DD_{\gamma,x}}_{\alpha_\gamma+1})\leq\lh(E^*)$.
Since $M^\Tt_\gamma,M^\Uu_\gamma$ are not $(x,\alpha_\gamma+\om)$-compatible,
it follows that 
$E^\Uu_\gamma\neq\emptyset$ and $\lh(E^\Uu_\gamma)\leq\lh(E^*)$. But then
$E^*\notin\es(M^\Uu_\zeta)$,
contradiction.
\end{proof}
Since the comparison only terminates with trees that witness the statement of the lemma, we are done.
\end{proof}

\subsubsection{The prewellordering property for $\Pi^1_3$}
We can now give the promised alternate proof of $\mathrm{PWO}(\Pi^1_3)$:

\begin{tm}\label{tm:PWO(Pi^1_3)}Suppose that $M_1^\#(x)$ exists \tu{(}and is $(\om,\om_1)$-iterable\tu{)}
 for every $x\in\RR$. Then $\mathrm{PWO}(\Pi^1_3)$ holds.
\end{tm}
\begin{proof}
 Let $\psi$ be a $\Pi^1_3$ formula and $A=\{x\in\RR\mid\psi(x)\}$.
 
 Say a premouse $P$ is \emph{pre-relevant}
 if it is countable,  $1$-small, and has a (unique) Woodin $\delta^P$.
 Say a premouse $P$ is \emph{relevant} if it is pre-relevant and $(0,\om_1)$-iterable.

 Let $P$ be pre-relevant.
Let $A_*^P$ be the set of all $x\in\RR$ such that there is $\alpha<\delta^P$
such that $\core_\om(N_{x\alpha}^P)$ (is defined and) is a $\psi(x)$-prewitness.
Define a norm $\varphi^P$ on $A_*^P$ by setting $\varphi^P(x)=$ the least such $\alpha$.
Write $x\leq^P_* y$ iff $x\in A_*^P$ and if $y\in A_*^P$ then $\varphi^P(x)\leq\varphi^P(y)$.

\begin{clm}\label{clm:relevant_pms_agree_on_order}
Let $P,Q$  be relevant premice.
For $x,y\in\RR^P\inter\RR^Q$
with $x\in A_*^P$ and $y\in A_*^Q$,
we have $x\leq^P_*y$ iff $x\leq^Q_*y$.
\end{clm}
\begin{proof}
Let $P',Q'$ be iterates of $P,Q$
witnessing Lemma \ref{lem:first_comparison}. It is easy to see that
$x\leq^{P'}_*y$ iff $x\leq^{Q'}_*y$.
But 
 by elementarity,
$x\leq^P_*y$ iff $x\leq^{P'}_*y$,
and likewise for $Q,Q'$. This proves the claim.
\end{proof}

\begin{clm}\label{clm:M_1(y)_builds_psi(x)-witnesses}
 Let $x,y\in\RR$ with $x\in A\cap M_1(y)$. Let $P=M_1(y)|(\delta^{M_1(y)})^{+M_1(y)}$ (so $P$ is relevant). Then $x\in A_*^P$,
 and $\varphi^P(x)<\delta^P=\delta^{M_1(y)}$.
\end{clm}
\begin{proof}
Since $x\in A$
and by Lemma \ref{lem:varphi(w)-witness}, there is a $\psi(x)$-witness $W$ such that $W$ is sound and $\rho_1^W=\om$. Then $W\pins M_1(x)$. But standard calculations show that $\DD_{x}^{M_1(y)}$ reaches a final model $N$ which coiterates with $M_1(x)$, with no dropping on the main branch of either side,
and $\delta^N\leq\delta^{M_1(y)}$ (in fact $\delta^N=\delta^{M_1(y)}$). Therefore $W\pins N$.
But then $W=\core_\om(N^{M_1(y)}_{x\alpha})$ for some $\alpha<\delta^{M_1(y)}$,
so $\varphi^{P}(x)\leq\alpha<\delta^{P}=\delta^{M_1(y)}$.
\end{proof}

We now define a prewellorder $\leq$ on $A$, by setting $x\leq y$
 iff there is some relevant premouse $P$
 with $x,y\in P$
 and $x\leq^P_*y$; equivalently, for all relevant $P$ with $x,y\in P$,
 if $y\in A_*^P$ then $x\leq^P_*y$
 (the equivalence is by  Claims
 \ref{clm:relevant_pms_agree_on_order} and \ref{clm:M_1(y)_builds_psi(x)-witnesses}).

 \begin{clm}$\leq$ is a
prewellorder on $A$.\end{clm}
\begin{proof}Given a countable set $\bar{A}\sub A$,
  pick some $z\in\RR$ coding all elements of $\bar{A}$.
 Let $P=M_1(z)|(\delta^{M_1(z)})^{+M_1(z)}$.
 By Claim \ref{clm:M_1(y)_builds_psi(x)-witnesses},
 for each $x\in\bar{A}$ there is $\alpha<\delta^P$
 such that $\core_\om(N_{x\alpha}^P)$ is a $\psi(x)$-witness,
 demonstrating that the restriction of $\leq$ to $\bar{A}$ is a prewellorder.\end{proof}

\begin{clm} $\leq^*$ and $<^*$ are $\Pi^1_3$.\end{clm}
\begin{proof}
 Given $x,y\in\RR$ and a pre-relevant $(x,y)$-premouse $P$, say that $P$
 is \emph{$(x,y)$-minimal} iff $P$ is sound with $\rho_1^P=\om$ and either $x\in A_*^P$ or $y\in A_*^P$
 but there is no pre-relevant $P'\pins P$ such that $x\in A_*^{P'}$ or $y\in A_*^{P'}$.
 Note that if $x\in A$ or $y\in A$, then like in the proof of Lemma \ref{lem:varphi(w)-witness} and by condensation, the least $Q\pins M_1(x,y)$ which is pre-relevant with either $x\in A_*^Q$ or $y\in A_*^Q$, 
 is $(x,y)$-minimal.
 (If $\wt{Q}\pins M_1(x,y)$
 is pre-relevant with either $x\in A_*^{\wt{Q}}$ or $y\in A_*^{\wt{Q}}$, then observe that $Q=\cHull_1^{\wt{Q}}(\{\delta^{\wt{Q}}\})\pins M_1(x,y)$.)
 Moreover, $Q$ is the unique $(x,y)$-minimal $Q'\pins M_1(x,y)$.
 \begin{sclm}
 The following are equivalent:\begin{enumerate}[label=(\roman*)]\item $x\leq^*y$,
 \item $x\in A$ and $x\leq^P_*y$ for every $\Pi^1_2$-iterable $(x,y)$-minimal premouse $P$ (and note this condition is $\Pi^1_3$).
 \end{enumerate}
 \end{sclm}
 \begin{proof}
 If $x\leq^*y$ then let $Q\pins M_1(x,y)$ be $(x,y)$-minimal;
 so $x\leq^Q_*y$. Now let $P$ be $\Pi^1_2$-iterable $(x,y)$-minimal. Then $P=Q$, by \ref{lem:overspill_mouse} and $(x,y)$-minimality,
 which suffices. The other direction is clear.\end{proof}

 This shows that $\leq^*$ is $\Pi^1_3$.
 The definability for $<^*$ is likewise.
\end{proof}
 
 This completes the proof of the theorem.
\end{proof}

\subsubsection{More on $\Pi^1_3$}

 (***This section to be added.)

\section{Projecting mice}
\subsection{Some fine structure}

In the following lemma, \emph{$(k+1)$-condensation}
refers to the version of condensation asserted in the conclusion of  \cite[Theorem 5.2]{premouse_inheriting}.

\begin{lem}\label{lem:p_1_for_passive}
 Let $M$ be a passive premouse such that all proper segments of $M$ satisfy $(k+1)$-condensation for each $k<\om$. Then:
 \begin{enumerate}
  \item\label{item:cardinal_1-hull}  For each $M$-cardinal $\rho\in(\om,\OR^M)$, we have $\Hull_1^M(\rho)=M||\rho$.
  \item\label{item:when_rho_1<OR} If $\rho^M_1<\OR^M$ then:
  \begin{enumerate}[label=\tu{(}\roman*\tu{)},ref=\tu{(}\ref{item:when_rho_1<OR}\alph*\tu{)}]
  \item\label{item:lgcd_exists} $M$ has a largest cardinal $\theta$ (so $\rho_1^M\leq\theta$),
  \item\label{item:p_1_high} if $\om<\theta$ then $p_1^M\neq\emptyset$ and $\theta\leq\max(p_1^M)$, and
 \item\label{item:hull_max_p_1} Suppose $p_1^M\neq\emptyset$ and let $\gamma=\max(p_1^M)$.
 Then
 $\Hull_1^M(\gamma)=M||\gamma$
 and if $\gamma>\theta$ then $\lgcd(M||\gamma)=\theta$.
  \end{enumerate}
  \end{enumerate}
\end{lem}
\begin{proof}
Part \ref{item:cardinal_1-hull}: Let $\rho$ be an $M$-cardinal. Suppose $\rho\leq\alpha<\OR^M$ and $\alpha\in\Hull_1^M(\rho)$ and
$\alpha$ is a limit ordinal. Then we may fix a formula $\varphi$ and $\gamma<\rho$
such that $M|\alpha\sats\varphi(\gamma)$ but for all $\alpha'<\alpha$, we have $M|\alpha'\sats\neg\varphi(\gamma)$.
Say $\varphi$ is $\rSigma_{n+1}$.

Suppose that $\rho=(\theta^+)^M$ where $\theta$ is an $M$-cardinal.
By condensation we may find $\xi<\rho$ such that $\gamma<\xi$
and letting
\[ H=\cHull_{n+1}^{M|\alpha}(\theta\cup\{\xi\}) \]
and $\pi:H\to M|\alpha$ be the uncollapse, then $H\pins M|\rho$.
For we easily get $H\in M|\rho$ since $M|\alpha\in M$. So we just have to avoid conclusions 2(b), 2(c) and 2(d)
of \cite{premouse_inheriting}[Theorem 4.2***]. Its conclusion 2(b) is impossible because $H$ projects to $\theta$.
To avoid 2(c) and 2(d), choose $\xi$ such that $\xi>\theta$ and
\[ \xi\notin\Hull_{n+1}^{M|\alpha}(\xi).\]
But then $H\sats\varphi(\gamma)$, a contradiction.

If instead $\rho$ is a limit cardinal of $M$, it is easier, as we may take
\[ H=\cHull_{n+1}^{M|\alpha}(\theta) \]
for some $M$-cardinal $\theta$ with $\gamma<\theta$, and then $H\pins M$, etc.

Parts \ref{item:lgcd_exists} and \ref{item:p_1_high} follow immediately from part \ref{item:cardinal_1-hull}.

Part \ref{item:hull_max_p_1}:
By \ref{item:p_1_high}, we have $\theta\leq\gamma$.
Let $H=\Hull_1^M(\gamma)$.
Then $H$ is transitive, because if $\alpha\in H$ and $\alpha\geq\theta$,
then $M$ has a surjection $f:\theta\to\alpha$, but then so does $H$.
So if $H\neq M||\gamma$ then $\gamma\in H$,
which contradicts the minimality of $p_1^M$.
And if $\gamma>\theta$ then since $H\elem_1 M$
and $\lgcd(M)=\theta$, we get $\lgcd(H)=\theta$.
\end{proof}

\subsection{$\om$-mice}\label{sec:om-mice}
\begin{dfn}
An \emph{$\om$-premouse} is a sound premouse $M$ with $\rho_\om^M=\om$.
An 
\emph{$\om$-mouse} is an $(\om,\om_1+1)$-iterable
$\om$-premouse.\footnote{However, in \S\ref{sec:Pi^1_3} we will
use the definition differently, just assuming $(\om,\om_1)$-iterability there.}
\end{dfn}

We will make do throughout with normal iterability (that is, $(k,\om_1+1)$-iterability for 
$k$-sound premice $P$). One fact that helps here is the following, which is by \cite[Corollary 9.4]{iter_for_stacks} and (for the ``moreover'' clause) its proof:

\begin{fact}
 Every $\om$-mouse is $(\om,\om_1,\om_1+1)^*$-iterable. Moreover, if $\Sigma$ is the unique 
$(\om,\om_1+1)$-strategy for $M$, then there is an $(\om,\om_1,\om_1+1)^*$-strategy $\Gamma$ for $M$ 
such that $\Gamma\rest\HC$ is projective in $\Sigma\rest\HC$.
\end{fact}

\begin{dfn}
 Let $\Tt$ be either a $k$-maximal iteration tree
 or a $k$-maximal stack.
 We say that $\Tt$ is \emph{terminally-non-dropping} iff $\Tt$ has a final model
 and
 $b^\Tt$ does not drop in model or degree.
\end{dfn}

It is shown in  \cite{fsfni_v4}
that every $(k,\om_1+1)$-iterable $k$-sound premouse $M$ is $(k+1)$-solid and $(k+1)$-universal.
But assuming also that $\rho_{k+1}^M=\om$,
there is a proof of this fact which, modulo already published
results, is much shorter:
\begin{tm}
 Let $M$ be a $k$-sound, $(k,\om_1+1)$-iterable premouse such that $\rho_{k+1}^M=\om$.
 Then $M$ is $(k,\om_1,\om_1+1)^*$-iterable, so $M$ is $(k+1)$-solid and $(k+1)$-universal.
\end{tm}
\begin{proof}
 By \cite[Theorem 9.6]{iter_for_stacks}, $M$ is $(k,\om,\om+1)^*$-iterable,
 and moreover, for each $k$-maximal finite stack of finite normal trees on $M$,
 the resulting iterate is normally iterable (at the degree of that iterate). So we can fix a finite terminally-non-dropping $k$-maximal stack $\Ttvec=\left<\Tt_i\right>_{i<j}$
 on $M$ such that for every finite $k$-maximal terminally-non-dropping 
tree on $N=M^\Ttvec_\infty$,
we have $i^\Tt(p_{k+1}^N)=p_{k+1}^{M^\Tt_\infty}$. (Clearly for finite
terminally-non-dropping $k$-maximal trees $\Uu$ where $M^\Uu_0$ is $k$-sound and $\rho_{k+1}(M^\Uu_0)=\om$,
we have $p_{k+1}^{M^\Uu_\infty}\leq i^\Uu(p_{k+1}^{M^\Uu_0})$. So if the statement fails then we 
would get a length $\om$ stack with illfounded direct limit.) But then for \emph{any} terminally-non-dropping $k$-maximal 
tree $\Uu$ on $N$, we have
$i^\Uu(p_{k+1}^N)=p_{k+1}(M^\Uu_\infty)$, because otherwise we can factor off a finite support tree 
contradicting our choice of $N$. We have that $N$ is $(k,\om_1+1)$-iterable as just mentioned.

Let $C=\core_{k+1}(N)$ and $\pi:C\to N$ the core map. A comparison of $C$ with $N$ shows that
$N$ is $(k+1)$-universal (using only that $\rho_{k+1}^N=\om$; we don't need the preservation of 
$p_{k+1}$ for this). So $\pi(p_{k+1}^C)=p_{k+1}^N$. And for every terminally-non-dropping tree $\Tt$ on $C$,
$i^\Tt$ preserves $p_{k+1}$, because as above, otherwise we can take $\Tt$ to be 
finite, but then lifting to $\pi\Tt$ on $N$, we get a contradiction (the first copy map is $\pi$,
a $k$-embedding, so by \cite{fs_tame}, the copy map $\pi_j:M^\Tt_j\to M^{\pi\Tt}_j$ is a near $\deg^\Tt(j)$-embedding,
and applying this with the ultimate copy map $\pi_j$, one observes that $p_{k+1}^{M^{\pi\Tt}_j}\leq\pi_j(p_{k+1}^{M^\Tt_j})$).

Now $C|\om_1^C=N|\om_1^N=M|\om_1^M$, so $C,N,M$ compare to common models with terminally-non-dropping $k$-maximal trees. So it suffices to show that $C$ is $(k,\om_1,\om_1+1)^*$-iterable.
To see this, it suffices to see that our strategy $\Sigma$ for $C$ has inflation condensation.
To see this, it suffices to see that $\Sigma$ is the unique $(k,\om_1+1)$-strategy for $C$.

So suppose $\Gamma\neq\Sigma$ is such a strategy for $C$.
Let $\Vv$ be a tree of limit length via both $\Sigma,\Gamma$,
such that $b=\Sigma(\Vv)\neq\Gamma(\Vv)=c$.
Let
\[(\Tt,\Uu)=(\Vv\conc b\conc\Tt',\Vv\conc c\conc\Uu')\]
be the successful comparison of the phalanxes $\Phi(\Vv\conc b)$ and $\Phi(\Vv\conc c)$,
considered as a pair of $k$-maximal trees on $C$.
One now reaches a contradiction almost as under the added assumption that $C$ is also $\om$-sound.
The only difference lies in the case that $M^\Tt_\infty=M^\Uu_\infty$ and both $\Tt,\Uu$ are terminally-non-dropping,
so assume this. Then $i^\Tt=i^\Uu$ because $C=\Hull_{k+1}^C(\pvec_{k+1}^C)$
and $i^\Tt,i^\Uu$ preserve $p_{k+1}$ (because every such tree on $C$ preserves 
$p_{k+1}$); note that one would usually use $(k+1)$-solidity in order to deduce that $i^\Tt=i^\Uu$.
This leads to contradiction as usual.
\end{proof}

\begin{rem}
We will be interested in certain $L[\es,x]$-constructions formed inside background universes 
$P$ which themselves are premice. The main case of interest will be when $P$ is
a Woodin-$\om$-mouse. However, we will need to develop some things somewhat more generally,
and for this, we will be mostly interested in the case that $P$ is either an $\om$-mouse,
or a $k$-Woodin mouse, for some $k<\om$.
\end{rem}

\subsection{Q-structures}\label{sec:Q-mice}

In this paper we will often deal with mice $M$ which are Q-structures for the least Woodin cardinal $\delta^M$ of $M$; that is,
there is a failure of Woodinness of $\delta^M$ definable from parameters over $M$, in fact definable in a nice way.
These feature very often in inner model theory. We collect in this section some basic facts about such Q-structures.
The material in this section would be to some extent folklore, but as far as the author is aware, much of it has not appeared in print elsewhere. (Moreover, formally we reduce iterability hypotheses to just normal iterability, whereas other sources would have probably assumed iterability for stacks.) We probably don't use all of the material in the paper, but it seems useful to collect the basic facts here. For a fair portion of the material,
we can work just under the hypothesis that $\delta$ is Woodin but not a limit of Woodins of $M$, instead of demanding that $\delta$ be the least Woodin $\delta^M$ of $M$.

\begin{dfn}
 Let $Q$ be a $q$-sound premouse and $\delta\leq\rho_q^Q$. We say that $\delta$ is \emph{$\bfrSigma_q^Q$-singular}
 iff there is $\alpha<\delta$
 and a $\bfrSigma_q^Q$ function $f:\alpha\to\delta$ with $\range(f)$ cofinal in $\delta$. We say that $\delta$ is \emph{$\bfrSigma_q^Q$-singular} otherwise.
\end{dfn}

\begin{rem}
Note that because we assume $\delta\leq\rho_q^Q$, it would make no difference to the definition, if we had also considered partial functions (with domain $\sub\alpha<\delta$). Note that if $q>0$, $\bfrSigma_q^Q$-regularity (in the above context) is the same as the assumption that
\[\delta\cap\Hull_{q}^Q(\alpha\cup\{x\}) \text{ is bounded in }\delta \]
for each $x\in Q$ and $\alpha<\delta$.
And in case $q=0$,
it is the same as requiring that if $\delta\in\OR^Q$ then  $Q\sats$``$\delta$ is regular''.\end{rem}

A \emph{Q-prepair}, defined next, is basically a Q-structure for a  successor Woodin cardinal $\delta$:

\begin{dfn}\label{dfn:Q-prepair}
Let $Y$ be a $q$-sound premouse and $\delta\leq\OR^Y$.
We say that $\delta$ is \emph{$\bfrSigma_q^Y$-Woodin}
iff $\delta\leq\rho_q^Y$ and for every $A\sub\delta$ such that $A$ is $\bfrSigma_q^Y$,
there is $\kappa<\OR^Y$ such that $\kappa$ is $({<\delta},A)$-reflecting as witnessed by $\es^Y$.
(In the case that $\delta=\OR^Y$ then, $\delta$ is $\bfrSigma_0^Y$-Woodin
iff $\delta\leq\rho_0^Y$ and $Y\sats$``There are unboundedly many strong cardinals as witnessed by $\es$''.)

Let $Y$ be a premouse and $\delta\leq\OR^Y$.
We say that $(Y,\delta)$ is a \emph{Q-prepair} iff for some $q<\om$:
\begin{enumerate}[label=--]
 \item $Y\sats$``$\delta$ is not a limit of Woodin cardinals''.
 \item $Y$ is $q$-sound.
\item $\delta\leq\rho_q^Y$ and $\delta$ is $\bfrSigma_q^Y$-regular.
\item Either:
\begin{enumerate}[label=\tu{(}\roman*\tu{)}]
 \item\label{item:singular_type} $\delta$ is $\bfrSigma_q^Y$-Woodin 
and $\bfrSigma_{q+1}^Y$-singular, or
 \item $q>0$ and $\delta$ is $\bfrSigma_{q-1}^Y$-Woodin but not $\bfrSigma_q^Y$-Woodin.
\end{enumerate}
\end{enumerate}
If $(Y,\delta)$ is a Q-prepair, clearly the witnessing $q$ is unique.
We write $q^{Y,\delta}=q$, and call $q^{Y,\delta}$ the \emph{Q-degree} of $(Y,\delta)$.
We say that $(Y,\delta)$ is \emph{Q-singular} iff \ref{item:singular_type} above holds,
and otherwise \emph{Q-regular}.

A \emph{Q-pair} is a
$(q^{Y,\delta},\om_1+1)$-iterable Q-prepair $(Y,\delta)$.

For the main application,
we will just be interested in the case
that $\delta=\delta^Y$ (the least Woodin of $Y$), in which case we will say that $Y$ is 
a \emph{Q-premouse} (\emph{Q-mouse} in the iterable case),  say
$q$ is the \emph{Q-degree of $Y$},
and write $q^Y=q$.
\end{dfn}

We did not explicitly mention the case that $\rho_{q+1}^Y<\delta\leq\rho_q^Y$ and $\delta$ is 
$\bfrSigma_q^Y$-Woodin, because given iterability, this yields Q-singularity:

\begin{lem}\label{lem:hull_cof_in_delta}
 Let $k<\om$, let $M$ be $(k,\om_1+1)$-iterable and $\delta\leq\OR^M$ be $\bfrSigma_k^M$-Woodin,
 and suppose $M\sats$``$\delta$ is not a limit of Woodin cardinals''.
 Let
\[ H=\Hull_{k+1}^M(\rho_{k+1}^M\cup\{\delta,\pvec_{k+1}^M\}).\]
Then $H$ is unbounded in $\delta$.
Therefore if $\rho_{k+1}^M<\delta$ then $(M,\delta)$ is a Q-singular Q-pair of Q-degree $k$.
\end{lem}
\begin{proof}
We may assume $\rho^M_{k+1}<\delta<\rho_k^M$,
as $\Hull_{k+1}^M(\rho_{k+1}^M\cup\pvec_{k+1}^M)$ is cofinal in $\rho_k^M$. In particular, $\delta<\rho_0^M$.
Let $\eta=\sup(H\inter\delta)$ and suppose $\eta<\delta$.
 Let
 \[ J=\cHull_{k+1}^M(\eta\un\{\pvec_{k+1}^M,\delta\}) \]
 and $\pi:J\to M$ be the uncollapse. Clearly $J\notin M$ and $\rg(\pi)$ is cofinal in $\rho_k^M$.
We have $\delta\in\rg(\pi)$ by definition. Because $\delta$ is regular in $M$,
 the usual argument then shows that $\crit(\pi)=\eta$ and $\pi(\eta)=\delta$.
 
 Now we claim that in fact,
\begin{equation}\label{eqn:delta_in_for_free} \delta\in\Hull_{k+1}^M(\eta\cup\{\pvec_{k+1}\})=\rg(\pi),\end{equation}
so
$J=\cHull_{k+1}^M(\eta\cup\{\pvec_{k+1}^M\})$
and $\pi$ is the corresponding uncollapse map. For
 let $\gamma$ be the supremum of $0$ and all Woodin cardinals of $M$ which are ${<\delta}$ (so $\gamma<\delta$ by hypothesis).
 Since $\delta\in\rg(\pi)$,
 so is $\gamma$, so $\gamma<\eta$.
 But then since $\delta$ is the least Woodin of $M$ which is $>\gamma$, 
 line (\ref{eqn:delta_in_for_free}) is clear
 unless $k=0$ and $M$ is passive,
 so suppose this is the case. Then since $\delta<\OR^M$ is an $M$-cardinal,
 we have $\delta\leq\max(p_1^M)$,
 and since $\gamma<\eta$, this again
 easily gives line (\ref{eqn:delta_in_for_free}).

 Also $J|\eta=M|\eta\pins M|\delta$
 and $\gamma<\eta$, so $\eta$ is not Woodin in $M$, 
 so we can fix $K\pins M|\delta$ which is a Q-structure for $\eta$.

 Now $J$ is $k$-sound and (using line (\ref{eqn:delta_in_for_free}))
 is  $\eta$-sound
 with $\rho_{k+1}^J=\rho_{k+1}^M\leq\eta$.
 And $K$ is $\om$-sound and $\rho_\om^K=\eta$
 (note that $\eta$ is an $M$-cardinal).
 Since $J\notin M$, we have $J\nins K$,
 and since $K$ is a Q-structure for $\eta$
 and $K\neq J$,
 we have $K\nins J$. Let $(J_0,K_0)$
 be the least pair $(J',K')$
 such that $J'\ins J$ and $K'\ins K$
 and $J'||\eta^{+J'}=K'||\eta^{+K'}$
 but $J'\neq K'$. Let $j_0,k_0\in\{-1\}\cup\om$
 be such that $\rho_{j_0+1}^{J_0}\leq\eta<\rho_{j_0}^{J_0}$
 and $\rho_{k_0+1}^{K_0}=\eta<\rho_{k_0}^{K_0}$.
 Then $(J_0,K_0,\eta)$ is a non-trivial
 $(j_0,k_0,\om_1+1)$-iterable bicephalus
 (see \cite{premouse_inheriting}; because of the agreement between $J_0,K_0$
 and that $\pi\rest\eta=\id$,
 we can lift trees on the bicephalus to trees on $M$ as in \cite{premouse_inheriting}).
 Since $K_0$ is fully sound, by \cite[Theorem 4.3]{premouse_inheriting}, 
 $J_0$ is non-$(j_0+1)$-sound, so $J_0=J$
 and $j_0=k$ and $\rho_{k+1}^J<\eta$.
 Also by \cite[Theorem 4.3]{premouse_inheriting} and since $\eta$ is a limit cardinal of $M$, $K_0$ is active type 3  with largest cardinal $\eta$, and letting $\kappa=\crit(F^{K_0})$, then $J$ has a $(k,\eta)$-good core at $\kappa$ (see \cite[Definition 4.1]{premouse_inheriting}).
 But then $\rho_{k+1}^M=\rho_{k+1}^J\leq\kappa<\eta$ and\[\eta\cap\Hull^J_{k+1}(\kappa\cup\pvec_{k+1}^J)=\kappa<\eta,\]
 contradicting the fact that $\Hull_{k+1}^J(\rho_{k+1}^M\cup\pvec_{k+1}^J)$ is cofinal in $\eta$.
\end{proof}
\begin{lem}\label{lem:Q-mouse_core}
 Let $(Y,\delta)$ be a Q-pair and $q=q^Y$.
 Let $k\geq 1$ be such that $C=\core_{q+k}(Y)$ exists \tu{(}this certainly holds if $k=1$\tu{)}. Let $\pi:C\to Y$ be the core embedding and suppose that $\delta\in\rg(\pi)$. Let $\pi(\bar{\delta})=\delta$
 \tu{(}it is easy to see that this is automatically true if $\delta=\delta^Y$\tu{)}.
 Then:
 \begin{enumerate}[label=--]\item $(C,\bar{\delta})$ is a Q-mouse with $q^{C,\bar{\delta}}=q$. \item  $(C,\bar{\delta})$ is Q-regular \tu{(}Q-singular\tu{)} iff $(Y,\delta)$ is Q-regular \tu{(}Q-singular\tu{)}
 \item $\bar{\delta}<\rho_q^C$  iff $\delta<\rho_q^Y$,
  \item $\rho_{q+1}^{C}<\bar{\delta}$ iff $\rho_{q+1}^Q<\delta$.
 \end{enumerate}
\end{lem}

\begin{proof}
We assume $k=1$ and $\delta^Y<\rho_0^Y$ and leave the remaining cases to the reader.
If $Y$ is $(q+1)$-sound then $C=Y$, so everything is trivial. So suppose otherwise,
so $\rho_{q+1}^Y<\rho_q^Y$.
Let $\pi:C\to Y$ be the core embedding.
So $C$ is $q$-sound 
and non-small, $\bar{\delta}\leq\rho_q^C$, and $\rg(\pi)$ is cofinal in both 
$\rho_q^Y$ and $\delta$ (cofinality in $\delta$ is by \ref{lem:hull_cof_in_delta} if $Y$ is 
Q-singular; if $Y$ is Q-regular then $\delta=\rho_q^Y$).
Because $\delta$ is $\bfrSigma_q^Y$-regular, the elementarity of $\pi$ and continuity of $\pi$ at 
$\bar{\delta}$ imply that $\bar{\delta}$ is $\bfrSigma_q^C$-regular.

\begin{case}\label{case:Y_Q-reg} $Y$ is Q-regular.

So $q>0$ and (as $Y$ is non-$(q+1)$-sound) $\rho_{q+1}^Y<\rho_q^Y=\delta$, so $\rho_{q+1}^C<\rho_q^C=\bar{\delta}$.

We claim that $\bar{\delta}$ is $\bfrSigma_{q-1}^C$-Woodin.
If $q=1$ this is immediate, so suppose $q>1$ and $\rho_{q-1}^C=\bar{\delta}$.
Because $C$ is $(q-1)$-sound, every $\bfrSigma_{q-1}^C$ subset of $\bar{\delta}$
is therefore coded easily into
\[ t^C=t_{q-1}^C=\Th_{\rSigma_{q-1}}^C(\bar{\delta}\un\pvec_{q-1}^C),\]
and it suffices to see that  there are cofinally many $\mu<\bar{\delta}$ such that $C\sats$``$\mu$ is $({<\bar{\delta}},t^C)$-reflecting''
(note $t^C\notin C$ though).
Let $t^Y=t_{q-1}^Y$ be defined likewise over $Y$. Fix $\theta<\bar{\delta}$ and let $\kappa\in(\pi(\theta),\delta)$ be least such that 
$Y\sats$``$\kappa$ 
is $({<\delta},t^Y)$-reflecting''. Then $\kappa\in\rg(\pi)$ and $\mu=\pi^{-1}(\kappa)$ 
works, because the statement
``$\kappa$ is $(<\delta,t^Y)$-reflecting'' is 
$\rPi_q^Y(\{\kappa,\delta,\pvec_{q-1}^Y\})$, the verification of which we leave to the reader.

It remains to see that $\bar{\delta}$ is not  $\bfrSigma_q^C$-Woodin.
Let $T^C=t_q^C$ (as above, but with $q$ replacing $q-1$) and
$T^Y=t_q^Y$. 
So we can fix $\theta<\bar{\delta}$ such that
($*$) $Y\sats$``no $\kappa\in(\pi(\theta),\delta)$ is $({<\delta},T^Y)$-reflecting''.
For each 
$\alpha<\bar{\delta}$ we have $T^C\rest\alpha\in C$ and
\begin{equation}\label{eqn:T-pres}\pi(T^C\rest\alpha)=T^Y\rest\pi(\alpha). \end{equation}
Now suppose that $\bar{\delta}$ is $\bfrSigma_q^C$-Woodin, and fix $\mu\in(\theta,\bar{\delta})$
such that $C\sats$``$\mu$ is $(<\bar{\delta},T^C)$-reflecting''.
Using line (\ref{eqn:T-pres}) it easily follows that $Y\sats$``$\pi(\mu)$ is 
$(<\delta,T^Y)$-reflecting'', contradicting ($*$).
\end{case}

\begin{case} $Y$ is Q-singular and $\rho_{q+1}^Y=\delta$.

So $\rho_{q+1}^C=\bar{\delta}=\delta=\rho_{q+1}^Y<\rho_q^Y$, so $\bar{\delta}<\rho_q^C$,
so $\bar{\delta}$ is $\bfrSigma_q^C$-Woodin. So we just need to see that $\bar{\delta}$ is 
$\bfrSigma_{q+1}^C$-singular. But by
\cite[Corollary 1.5]{fsfni_v4},\footnote{Note we are only assuming $(q,\om_1+1)$-iterability, which is enough for \cite{fsfni_v4}. If we were assuming $(q,\om_1,\om_1+1)^*$-iterability,
then one could have used more traditional arguments here.}
every $\bfrSigma_{q+1}^Y$-definable subset of $\rho_{q+1}^Y=\delta$ is also $\bfrSigma_{q+1}^C$, which suffices.
\end{case}
\begin{case}\label{case:Y_Q-singular_and_rho_q+1<delta} $Y$ is Q-singular and $\rho_{q+1}^Y<\delta$.

So $\delta$ is $\bfrSigma_q^Y$-Woodin but $\bfrSigma_{q+1}^Y$-singular, and $\bar{\delta}$ 
is $\bfrSigma_{q+1}^C$-singular by \ref{lem:hull_cof_in_delta}.
As stated earlier, $\bar{\delta}$ is $\bfrSigma_q^C$-regular,
and
as in Case \ref{case:Y_Q-reg},  if $q>0$ then $\delta^C$ is $\bfrSigma_{q-1}^C$-Woodin.
So if $\bar{\delta}$ is $\bfrSigma_q^C$-Woodin then we are done.
So suppose otherwise, so $q>0$ and $\bar{\delta}=\rho_q^C$ and 
$\delta=\rho_q^Y$. Let $T^C=\Th_{\rSigma_q}^C(\delta\cup\{\pvec_q^C\})$ and $T^Y$  likewise. Then for each $\alpha<\bar{\delta}$, we have
\[ \pi(T^C\rest\alpha)=T^Y\rest\alpha.\]
Let $\theta<\bar{\delta}$.
We want to see that there is $\kappa\in(\theta,\bar{\delta})$
such that $\kappa$ is $({<\bar{\delta}},T^C)$-reflecting. Suppose not. There is $\mu\in(\pi(\theta),\delta)$ which is $({<\delta},T^Y)$-reflecting,
so fix the least such $\mu$.
Note $\mu\notin\rg(\pi)$. Let $\gamma$
be the least $\gamma'<\bar{\delta}$
such that $\pi(\gamma')\geq\mu$;
so $\pi``\gamma\sub\mu<\pi(\gamma)$.
Given $\kappa\in(\theta,\gamma)$,
let $\eta_\kappa<\delta$
be the least $\eta$ such that there is no $E\in\es^C$ with $\crit(E)=\kappa$
and $E$ cohering $T^C$ through $\eta$
and $\nu(E)\geq\eta$.
For $\xi\in[\theta,\gamma)$,
let $\eta_{\geq\xi}=\sup_{\kappa\in[\xi,\gamma)}\eta_\kappa$. Then $\eta_{\geq\xi}=\delta$,
because of the existence of $\mu\in Y$.
So we can fix $\kappa_0\in(\theta,\gamma)$
with $\eta_{\kappa_0}\geq\gamma$
and also fix $\kappa_1\in(\kappa_0,\gamma)$
with $\eta_{\kappa_1}>\eta'$,
where $\eta'$ is the least inaccessible of $M$ which is $>\eta_{\kappa_0}$.
Let $E_0\in\es^C$ be such that $\crit(E_0)=\kappa_0$ and $E_0$ coheres $T^C$ through $\kappa_1$ and $\nu(E_0)=\kappa_1$.
Let $E_1\in\es^C$ be such that $\crit(E_1)=\kappa_1$ and $E_1$ coheres $T_C$ through $\eta'$ 
and $\nu(E_1)=\eta'$ (where $\eta'$ is as above). Let $F=i_{E_1}(E_0)\rest\eta'$.
Then $F\in\es^C$
and $\crit(F)=\kappa_0$
and $F$ coheres $T^C$ through $\eta'$
and $\nu(F)=\eta'$, contradicting that $\eta_{\kappa_0}<\eta'$.
\qedhere
\end{case}
\end{proof}

\begin{lem}\label{lem:iterates_maintain_def_regularity}
 Let $Y$ be $y$-sound,  $\delta\leq\rho_y^Y$,
 be an $\bfrSigma_y^Y$-regular $Y$-cardinal,
 $\Tt$ be a successor length $y$-maximal tree on $Y$ which is based on $M|\delta$,
 such that $b^\Tt$ does not drop in model.
 Let $Z=M^\Tt_\infty$ and $\delta'=i^\Tt(\delta)$.
 Then
 \begin{enumerate}[label=(\roman*)] \item\label{item:i^Tt_continuous_at_delta} 
 $i^\Tt$ is continuous at $\delta$, \item\label{item:delta'<=rho_y^Z} $\delta'\leq\rho_y^{Z}$, and \item\label{item:delta'_def_reg}$\delta'$ is $\bfrSigma_y^{Z}$-regular.\end{enumerate}
 Therefore  $b^\Tt$ does not drop in degree.
\end{lem}
\begin{proof} This is by induction along the nodes of $b^\Tt$.
Given all properties hold with $Z$ replaced by $M^\Tt_\alpha$ for some $\alpha\in b^\Tt$,
then properties \ref{item:i^Tt_continuous_at_delta} and \ref{item:delta'<=rho_y^Z} follow routinely at $M^\Tt_{\beta+1}$ where $\pred^\Tt(\beta+1)=\alpha$.
So suppose also $Z=M^\Tt_{\beta+1}$,
and let
us verify \ref{item:delta'_def_reg}. This is immediate if $\delta<\rho_q^Y$ or $q=0$,
so suppose $q>0$ and $\delta=\rho_q^Y$, so $\delta'=\rho_q^{Z}$ (by continuity). Because $\delta$ is 
$\bfrSigma_q^Y$-regular, for each $\alpha<\rho_q^Y$
there is $\gamma_\alpha<\delta$ such that
\[ \delta\cap\Hull_q^Y(\alpha\un\{\pvec_k^Y\})\sub\gamma_\alpha. \]
But because $i^\Tt$ preserves elements of $T_q$ (the $\rSigma_k$-theory predicate),
it easily follows that
\[ \Hull_q^Z(i^\Tt(\alpha)\un\{\pvec_k^Z\})\inter\delta^Z\sub i^\Tt(\gamma_\alpha), \]
which, since $\delta'=\sup i^\Tt``\delta$,
gives that $\delta'$ is $\bfrSigma_q^Z$-regular, as required.\end{proof}

\begin{lem}\label{lem:Q-mouse_iteration}
Let $(Y,\delta)$ be a Q-prepair of Q-degree $q$. Let $\Tt$ be a successor length $q$-maximal tree on $Y$, 
based on $Y|\delta$, such that $b^\Tt$ does not drop in model \tu{(}hence nor in degree\tu{)}. Let $Z=M^\Tt_\infty$. If $\delta<\rho_0^Y$ let $\delta'=i^\Tt(\delta)$,
and if $\delta=\rho_0^Y$ let $\delta'=\rho_0^Z$.
Then $(Z,\delta')$ is a 
Q-prepair of Q-degree $q$, and $\delta'=\sup i^\Tt``\delta$.
Moreover, $(Z,\delta')$ is Q-regular \tu{(}Q-singular\tu{)}
iff $Y$ is Q-regular \tu{(}Q-singular\tu{)}.
\end{lem}
\begin{proof}
We assume that $\delta<\rho_0^Y$ and leave the other case to the reader. 
Clearly $Z$ is $q$-sound.
Note that Lemma  \ref{lem:iterates_maintain_def_regularity}
applies, so we get the conclusions stated there.

\begin{case} $(Y,\delta)$ is Q-singular.
 
So $\delta$ is $\bfrSigma_q^Y$-Woodin but 
$\bfrSigma_{q+1}^Y$-singular.
The elementarity and continuity of $i^\Tt$
at $\delta$ gives that $\delta'$ is $\bfrSigma^Z_{q+1}$-singular.
So it suffices to see that $\delta'$ is $\bfrSigma_q^Z$-Woodin.
But because $Y,Z$ are $q$-sound, this is proven as in the proof of \ref{lem:Q-mouse_core}.
\end{case}

\begin{case} $Y$ is Q-regular.

So $q>0$ and $Y$ is $\bfrSigma_{q-1}^Y$-Woodin but not $\bfrSigma_q^Y$-Woodin,
and $\delta=\rho_q^Y$.
The fact that $\delta'$ is $\bfrSigma_{q-1}^Z$-Woodin is as before,
and by the argument used in Case \ref{case:Y_Q-singular_and_rho_q+1<delta} of the proof of
Lemma \ref{lem:Q-mouse_core},
$\delta'$ is non-$\bfrSigma_{q}^Z$-Woodin.\qedhere
\end{case}
\end{proof}
In the following lemma,
recall that $Y$ is a Q-mouse
if $\delta^Y<\OR^Y$ and $(Y,\delta^Y)$
is a Q-pair; and recall that $\delta^Y$ is the \emph{least} Woodin of $Y$.
\begin{lem}
 Let $Y$ be a Q-mouse and $q=q^Y$. Then $Y$ is $(q,\om_1,\om_1+1)^*$-iterable.
\end{lem}
\begin{proof}
We have $\rho_{q+1}^Y\leq\delta=\delta^Y$. Let $C$ be the $\delta$-core of $Y$
and $\pi:C\to Y$ be the core map.

\begin{clm} There is a unique $(q,\om_1+1)$-strategy for $C$,
and therefore $C$ is $(q,\om_1,\om_1+1)^*$-iterable.\end{clm}

The ``therefore'' clause follows from \cite[Lemma 4.45/Theorem 4.47, Theorem 1.2]{iter_for_stacks}.
Using the claim, we  complete the proof.
Using the normal part of our $(q,\om_1,\om_1+1)^*$-strategy for $C$, and any 
$(q,\om_1+1)$-strategy for $Y$ (recall it is the least Woodin of $Y$),
we compare $C$ with $Y$. Because $\delta$ 
is a strong cutpoint of both $C$ and $Y$,
the comparison leads to a common model, with no dropping in model or 
degree on either side. But then $Y$ inherits the tail strategy for $C$,
so $Y$ is $(q,\om_1,\om_1+1)^*$-iterable.

\begin{proof}[Proof of Claim]
$Y$ is $(q+1)$-universal above $\delta$;
that is, $C|\delta^{+C}=Y|\delta^{+Y}$ (by
the comparison argument just mentioned). It follows 
that $\pi(p_{q+1}^C\cut\delta)=p_{q+1}^Y\cut\delta$.
And by relativizing the results of \cite{fsfni_v4} above the strong cutpoint $\delta$, $p_{q+1}^C\cut\delta$
is $(q+1)$-solid for $C$, so $C$ is $\delta$-sound.

Now suppose that $\Sigma,\Gamma$ are two distinct $(q,\om_1+1)$-strategies for $C$.
Let $\Vv$ be countable, of limit length, via $\Sigma,\Gamma$, such that 
$b=\Sigma(\Vv)\neq\Gamma(\Vv)=c$. Let $\Tt=\Vv\conc b\conc\Tt'$ and $\Uu=\Vv\conc c\conc\Uu'$,
via $\Sigma,\Gamma$ respectively, be the successful comparison of the phalanxes $\Phi(\Vv\conc b)$, 
$\Phi(\Vv\conc c)$ respectively.

Let $\alpha<\lh(\Vv)$ be such that $[0,\alpha]_\Vv\inter\dropset_\deg^\Vv=\emptyset$.
By Lemma \ref{lem:Q-mouse_iteration}, $i^\Vv_{0\alpha}$ is continuous at $\delta$.
Because $C$ is $\delta$-sound and $\rho_{q+1}^C\leq\delta$ and $\delta$ is a strong cutpoint of $C$,
$M^\Vv_\alpha$ is $i^\Vv_{0\alpha}(\delta)$-sound and $\rho_{q+1}(M^\Vv_\alpha)\leq i^\Vv(\delta)$
and $i^\Vv_{0\alpha}(\delta)$ is a strong cutpoint of $M^\Vv_\alpha$.
Thus, there is at most one above-$i^\Vv(\alpha)$-$(q,\om_1+1)$-strategy for $M^\Vv_\alpha$.
So we may assume that $\Vv$ is based on $C|\delta$.

If $M(\Vv)$ is non-small and $\eta=\delta^{M(\Vv)}$ (the least Woodin of $M(\Vv)$)
then letting $Q$ be the $\eta$-sound Q-structure for $M(\Vv)|\eta$,
we have $Q\ins M^\Vv_\alpha$ where $\alpha$ is least such that $\eta<\lh(E^\Vv_\alpha)$,
and $\Vv\rest[\alpha,\lh(\Vv))$ is based on $Q$ and is above $\eta$,
and is via an above-$\eta$-$(k,\om_1+1)$-strategy for $Q$, where $k$ is least such that $\rho_{k+1}^Q\leq\eta$.
But there is a unique such above-$\eta$ strategy for $Q$,
giving a
contradiction.

So $M(\Vv)$ is small. Let $Q_b\ins M^\Vv_b$ and $Q_c\ins M^\Vv_c$
be the Q-structures for $M(\Vv)$; note that these exist and are $\delta(\Vv)$-sound
and project to $\delta(\Vv)$, which is a strong cutpoint of both (as $M(\Vv)$ is small), and $Q_b$ 
is above-$\delta(\Vv)$-$(k_b,\om_1+1)$-iterable,
where $k_b$ is such that $\rho_{k_b+1}^{Q_b}\leq\delta(\Vv)<\rho_{k_b}^{Q_b}$,
and likewise for $Q_c,k_c$. So standard fine structure gives that $Q_b=Q_c$, so $k_b=k_c$.
Write $Q=Q_b$. As usual by the Zipper Lemma (\cite[Theorem 6.10]{outline} and its variants) we can't have $Q\pins M^\Tt_b$ and $Q\pins 
M^\Tt_c$.
So by symmetry, we may assume $Q=M^\Tt_b$.

Suppose $b\inter\dropset_\deg^\Vv\neq\emptyset$. Then $Q$ is 
unsound, so $Q=M^\Tt_c$, and $\delta(\Vv)$ is $\bfrSigma_{k+1}^Q$-singular
where $k=\deg^\Tt(b)$ (this uses that $M(\Vv)$ is small, 
hence iteration maps along $b$ are eventually continuous at the preimage of $\delta(\Vv)$),
and note that $\rho_{k+1}^Q<\delta(\Vv)\leq\rho_k^Q$.
So by symmetry and the Zipper Lemma, $c\inter\dropset_\deg^\Vv=\emptyset$.
So $q=\deg^\Vv(c)$ and $\delta^Q=\delta(\Vv)\leq\rho_q^Q$, so $q\leq k$. Again by the Zipper Lemma, $q<k$.
By \ref{lem:Q-mouse_iteration}, $Q=M^\Vv_c$ is a Q-premouse and $q^Q=q$.
If $Q$ is Q-singular then $\delta(\Vv)$ is $\bfrSigma_{q+1}^Q$-singular, but as $q+1\leq k$,
this again contradicts the Zipper Lemma. So $Q$ is Q-regular,
so $q>0$ and $\delta(\Vv)$ is not $\bfrSigma_q^Q$-Woodin.
Therefore there is $\theta<\delta^Q$
such that there is no $\kappa\in(\theta,\delta^Q)$ which is $({<\delta^Q},T)$-reflecting where 
$T=t_q^Q$ (the theory defined as in the proof of \ref{lem:Q-mouse_core}).
But letting $\alpha\in c$ and $\kappa=\crit(i^\Vv_c)$
and $\nu=\nu(E^\Vv_{\beta+1})$ where $\beta+1=\min((\alpha,c])$,
we have $T\rest\kappa=t_q^{M^\Vv_\alpha}\rest\kappa$
and $i^\Vv_c(T\rest\kappa)\rest\nu=T\rest\nu$. Likewise for $b$ and $i^\Vv_{\alpha'b}$
when $\alpha'\in b$ and $(\alpha',b]\inter\dropset_\deg^\Vv=\emptyset$,
because $q\leq k$ (in fact $q<k$). Combining these remarks with the argument of the Zipper Lemma, 
it is straightforward to show that $M(\Vv)$ is Woodin with respect to $T$, a contradiction.

So  $b\inter\dropset_\deg^\Vv=\emptyset$.
Since $Q=M^\Tt_b$,
we have $i^\Tt_b(\delta^Y)=\delta^Q$
and $Q$ is a Q-premouse with $q^Q=q=\deg^\Tt_b$.

Suppose $Q\pins M^\Vv_c$.
Since $\delta(\Vv)$ is a cardinal of $M^\Vv_c$,
then $\rho_\om^Q=\delta(\Vv)$. So $i^\Vv_{\alpha c}$ is fully elementary
on the preimage of $Q$ and continuous at the preimage of $\delta(\Vv)$ when $\alpha\in c$ 
is sufficiently large, so again an easy adaptation of the preceding argument works.

So $M^\Vv_b=Q=M^\Vv_c$. So by the symmetry of earlier calculations,
also $c\cap\mathscr{D}^\Vv_{\mathrm{def}}=\emptyset$, and
$\delta(\Vv)=\delta^{M^\Vv_b}=\delta^{M^\Vv_c}$.
But then another adaptation of the preceding argument again gives a contradiction.\qedhere
\end{proof}As mentioned earlier, this proves the lemma.
\end{proof}

\begin{lem}\label{lem:Q_extra_it}
 Let $Q$ be an $(m+1)$-sound $(m+1,\om_1+1)$-iterable  premouse such that 
$\delta=\rho_{m+1}^Q$ is $\bfrSigma_{m+1}^Q$-regular.
Let $k\leq\om$ be largest such that $\delta$ is $\bfrSigma_k$-regular.
Then $Q$ is $(k,\om_1+1)$-iterable,
and in fact, $k$-maximal trees on $Q$ are equivalent to $(m+1)$-maximal trees on $Q$.
\end{lem}
\begin{proof}
We may assume $m+1<k$.
Note that $\rho_k^Q=\rho_{m+1}^Q=\delta$
and $Q$ is $k$-sound. So it is easily enough to see that for every $\kappa<\delta$ and every $f:\kappa\to Q$ which is $\bfrSigma_k^Q$, $f$
is in fact $\bfrSigma_{m+1}^Q$.
So fix such $\kappa,f$.
Since $Q=\Hull_{m+1}^Q(\delta\cup\{\pvec_{m+1}^Q\})$, we can fix a surjective $\bfrSigma_{m+1}^Q$  partial function $h:_{\mathrm{p}}\delta\to Q$. Note that there is then an $\bfrSigma_k^Q$ function $g:\kappa\to\delta$ such that $g(\alpha)\in\dom(h)$ and $f(\alpha)=h(g(\alpha))$
for all $\alpha<\kappa$.
But by the $\bfrSigma_k^Q$-regularity of $\delta$, $g$ is bounded in $\delta$,
and since $\rho_k^Q=\delta$, therefore $g\in Q$. So $f$ is in fact $\bfrSigma_{m+1}^Q$.
\end{proof}
\begin{rem}\label{rem:get_Q-mouse}
Let $Q$ be an $(m,\om_1+1)$-iterable non-small $\delta^Q$-sound $m$-sound premouse
with $\rho_{m+1}^Q\leq\delta^Q<\rho_m^Q$, and suppose that
if $Q$ is $\om$-sound then $\J(Q)\sats$``$\delta^Q$ is not Woodin''. Note that by Lemma \ref{lem:hull_cof_in_delta}, either
\begin{enumerate}[label=\tu{(}\alph*\tu{)}]
 \item\label{item:sing} $\delta$ is $\bfrSigma_{m+1}^Q$-singular, or
 \item\label{item:reg} $\rho_{m+1}^Q=\delta^Q$ and $Q$ is $(m+1)$-sound and $\delta^Q$ is 
$\bfrSigma_{m+1}^Q$-regular.
\end{enumerate}

If \ref{item:sing} holds 
then $Q$ is a Q-singular Q-mouse of degree $m$.

If \ref{item:reg} holds and $Q$ is $(m+1,\om_1+1)$-iterable then
by \ref{lem:Q_extra_it}, $Q$ is a Q-mouse.
\end{rem}
\subsection{P-construction}\label{sec:P-construction}

\begin{dfn}
 Let $m<\om$ and $M$ be an $m$-sound premouse.
 We say that $M$ is \emph{$m$-condensation-standard} iff for all $(N,n+1)\ins(M,m)$, we have:
 \begin{enumerate}
  \item 
 for all $n$-sound
  premice $H$ with $\rho_{n+1}^H<\rho_{n+1}^N$
  and all $n$-lifting embeddings $\pi:H\to N$ with $\rho=\rho_{n+1}^H\leq\crit(\pi)$, we have:
 \begin{enumerate}[label=--]
  \item if $H$ is $(n+1)$-sound
  then either:
  \begin{enumerate}[label=--]\item $H\pins N$, or
   \item 
$N|\rho$ is active and $H\pins\Ult(N|\rho,F^{N|\rho})$,
\end{enumerate}
and
  \item 
  if $H=\Hull_{n+1}^H(\rho\cup\{x\})$
  for some $x\in H$ then letting $C=\core_{n+1}(H)$,
   either:
  \begin{enumerate}[label=--]
  \item $C\pins N$, or
  \item $N|\rho$ is active and $C\pins\Ult(N|\rho,F^{N|\rho})$, 
\end{enumerate}
and $H$ is an iterate of $C$, via an above-$\rho$ $n$-maximal iteration tree $\Tt$ on $C$ of finite length.
 \end{enumerate}
 \item\label{item:m-cond-stand_parameter} for all $\alpha\in p_{n+1}^N$,
 letting  $q=p_{n+1}^N\cut(\alpha+1)$,
 \[ H=\cHull_{n+1}^{N}(\alpha\cup\{q,\pvec_n^N\}),\]
  $\pi:H\to N$  be the uncollapse, $C=\core_{n+1}(H)$
 and $\rho=\rho_{n+1}^C=\rho_{n+1}^H$,
 then:
 \begin{enumerate}\item\label{item:cond_for_C_in_param} either:
 \begin{enumerate}[label=--]
  \item $C\pins N$, or
  \item $N|\rho$ is active and $C\pins\Ult(N|\rho,F^{N|\rho})$, 
 \end{enumerate}\item there is $x\in H$ such that $H=\Hull_{n+1}^H(\rho\cup\{x\})$,
 \item $H$ is an iterate of $C$,
 via an above-$\rho$ $n$-maximal
 iteration tree $\Tt$ on $C$ of finite length,
 
 \item\label{item:what_x_is_param}either:
 \begin{enumerate}
 \item $\rho=\alpha=\crit(\pi)$
 (so we can take $x=\pi^{-1}(\{q,\pvec_n^N\})$), or
 \item $\rho<\alpha=\rho^{+H}=\crit(\pi)$,
 and there is $\xi\in(\rho,\alpha)$
 such that we can take $x=\{\xi\}\cup\pi^{-1}(\{q,\pvec_n^N\})$.\qedhere
 \end{enumerate}
 \end{enumerate}
 \end{enumerate}
 
 Let also $\delta<\OR^M$ with $\rho_{m+1}^M\leq\delta<\rho_m^M$. Suppose that $M$
 is $\delta$-sound. We say that $M$ is \emph{$\delta$-condensation-standard}
 iff for all $\alpha\in p_{m+1}^M$
 with $\alpha\geq\delta$,
 letting $q=p_{m+1}^M\cut(\alpha+1)$,
  $H=\cHull_{m+1}^M(\alpha\cup\{q,\pvec_m^M\})$,
  $\pi:H\to M$ be the uncollapse,
  $C=\core_{m+1}(H)$ and $\rho=\rho_{m+1}^C=\rho_{m+1}^H$, then conclusions
  \ref{item:cond_for_C_in_param}--\ref{item:what_x_is_param}
  of part \ref{item:m-cond-stand_parameter} above hold,
  but with $n$ replaced by $m$ throughout.
\end{dfn}

The following fact is by \cite[***]{fsfni_v4}:
\begin{fact} Let $m<\om$. We have:
\begin{enumerate}[label=--]
 \item Every $(0,\om_1+1)$-iterable premouse is $0$-condensation-standard,
 \item Every $(m,\omega_1+1)$-iterable $(m+1)$-sound premouse is $(m+1)$-condensation-standard.
 \item Let $M$ be an  $(m,\om_1+1)$-iterable $m$-sound premouse.
 Let $\delta<\OR^M$ and suppose $\rho_{m+1}^M\leq\delta<\rho_m^M$ and $M$ is $\delta$-sound. Then $M$ is $\delta$-condensation-standard.\footnote{However,
 this doesn't quite give the full picture.
 We would also want to consider the case that $M$ is an $m$-maximal iterate of some $P$ which was $(m+1)$-sound. But the iteration map $j:P\to M$ might not preserve the $(m+1)$-solidity witnesses.
 So we could have a witness $W\in P$ for some $\alpha\in p_{m+1}^P$ such that $j(W)$ is beyond the corresponding witness $W'$ for $M$.
 Let $C\pins P$ be $C=\core_{m+1}(W)$.
 So we should get a proper segment $C'$ of $j(C)$ such that $W'$ is a finite iterate of $C'$.}
 \end{enumerate}
\end{fact}

\begin{dfn}
 Let $Y$ be a premouse
 and $\delta\in\OR^Y$ be a strong cutpoint of $Y$.
 Let $M\in Y$ with $\OR^M=\delta$.
 We say that $M$ is a \emph{P-base of $Y$
 (at $\delta$)}
 iff
 \begin{enumerate}[label=--]
  \item $M$ is a premouse satisfying ZFC
  and $\J(M)\sats$``$\delta$ is Woodin but not a limit of Woodins'',\footnote{The requirement that $\delta$ not be a limit of Woodins of $\J(M)$ could be relaxed, but
  it appears that in the vicinity
  of´ measurable Woodins there are further subtleties to handle.}
  \item $M\sub Y|\delta$ and $M$ is definable from parameters over $Y|\delta$,
  \item $Y|\delta$ is $(\J(M),\BB_{\delta}^Y)$-generic.
 \end{enumerate}
 Suppose $M$ is a P-base of $Y$ at $\delta$.
 The \emph{P-construction} $\mathscr{P}^Y(M)$ of $Y$ over $M$ is the unique premouse $P$ such that:
 \begin{enumerate}
  \item $M\ins P$,
  \item $\OR^P\leq\OR^Y$,
  \item $P\sats$``$\delta$ is Woodin'',
  \item for $\xi\in(\delta,\OR^P]$, we have
  $F^{P|\xi}=F^{Y|\xi}\rest(P||\xi)$,
  \item either $\OR^P=\OR^Y$ or $P$ projects ${<\delta}$ or $P$ is a Q-structure for $\delta$.\qedhere
 \end{enumerate}
\end{dfn}
The following lemma captures key properties of P-constructions. The basic idea of it was due to Steel, and versions of it have been mentioned or appeared in, for example, \cite{twms}, \cite{ScalesK(R)}, \cite{sile}, \cite{odle_v2}, \cite{mueller_sargsyan_2021}, \cite{vm2_v2}, \cite{vmom_v2}, \cite{gaps_as_derived_models_v2}. The version we present here includes aspects examining the precise Q-degree of the P-construction, which, as far as the author knows, have not appeared elsewhere, though they use an argument similar to that in the proof of
\cite[Lemma 9, Theorem 10]{conjecture_mouse_order_for_weasels}.
\begin{lem}\label{lem:P-con_basic_props}
There is a recursive sequence $\left<\psi_k,\varrho_k\right>_{k<\om}$ of formulas $\psi_k,\varrho_k$ such that the following holds.
 Let $y<\om$ and let $Y$ be a $y$-condensation-standard $y$-sound premouse.
 Let $M$ be a P-base of $Y$ at $\delta\in\OR^Y$. Let $G$ be the $\BB_\delta^{\J(M)}$-generic
 determined by $Y|\delta$. Suppose that
 for all $(Y',y')\ins(Y,y)$ with $\delta<\OR^{Y'}$,
 \[\text{if }\big[P'=\mathscr{P}^{Y'}(M)\text{ exists and is }y''\text{-sound for all }y''<y'\text{ and }\rho_{y'}^{Y'}\leq\delta\big]\text{ then }\big[\delta\leq\rho_{y'}^{P'}\text{ and }\delta\text{ is }\bfrSigma_{y'}^{P'}\text{-Woodin.}\big]\]
Then:
 \begin{enumerate}\item\label{item:delta_reg_in_Y} $\delta$ is a regular cardinal of $Y$ and $\delta\leq\rho_y^Y$,\item\label{item:P-con_exists} $P=\mathscr{P}^Y(M)$ exists and $\OR^P=\OR^Y$
 \tu{(}and $P$ is unique\tu{)}. 
  \item\label{item:unif_def_of_P-con} For all $\xi\in(\delta,OR^P]$,
  we have:
  \begin{enumerate}[label=--]\item $P||\xi\sub Y||\xi$
  \item $P||\xi$ is $\Delta_1^{Y||\xi}(\{M\})$, uniformly in $\xi$,
  \item $P|\xi$ is $\Delta_1^{Y|\xi}(\{M\})$,
  uniformly in $\xi$,
  \end{enumerate}
 
\item\label{item:F^P|xi_and_F^Y|xi} Let $\xi\in(\delta,\OR^P]$ be such that $Y|\xi$ is active. Let $\kappa=\crit(F^{Y|\xi})=\crit(F^{P|\xi})$. Then
$F^{P|\xi}$ (which is a premouse extender) has the same generators as has $F^{Y|\xi}$. Moreover, for all $a\in[\lh(F^{Y|\xi})]^{<\om}$ and all $f\in Y|\xi$
with $f:[\kappa]^{|a|}\to P|\xi$
there is $g\in P|\xi$
such that
\[ \{u\in[\kappa]^{|a|}\bigm|f(u)=g(u)\}\in (F^{Y|\xi})_a.\]

\item\label{item:P|xi[Y|delta]} For all $\xi\in(\delta,\OR^P]$,
$Y|\delta$ is $(P|\delta,\BB_{\delta}^{P|\xi})$-generic,
and the generic extension $(P|\xi)[Y|\delta]$, considered as a premouse $N$ in the natural way, is just $Y|\xi$. Here the ``natural way'' means that $\es^{N}$
is computed by setting $N|\delta=Y|\delta$,
and for $\beta\in(\delta,\xi]$,
$F^{N|\beta}$ is just the
 canonical extension of $F^{P|\beta}$  to the generic extension; that is, for all $x\in[\lh(F^{P|\beta})]^{<\om}$, $(F^{P|\beta})_a$
is a dense subset of $(F^{N|\beta})_a$
with respect to $\sub$.
 \item\label{item:fs_matchup_proper_segs}
 for all $(\xi,n)\leq(\OR^P,y)$
 with $\xi>\delta$, we have:
 \begin{enumerate}
  \item\label{item:inductive_P|xi_fs} $P|\xi$ is $n$-sound with $\delta\leq\rho_n^{P|\xi}=\rho_{n}^{Y|\xi}$ and $p_n^{P|\xi}=p_n^{Y|\xi}$.
 \item\label{item:inductive_P|xi_when_delta_not_in_p} If $n\neq 1$ or $P|\xi$ is active or $\delta$ is not the largest cardinal of $P|\xi$ then $\delta\notin p_n^{P|\xi}$.
 \item\label{item:inductive_P|xi_when_delta_in_p} Suppose $n=1$  and $P|\xi$ is passive with largest cardinal $\delta$. Then $p_1^{P|\xi}=\{\gamma\}=p_1^{Y|\xi}$ for some $\gamma\geq\delta$ 
 \tu{(}possibly $\gamma=\delta$\tu{)}.
\item\label{item:rSigma_n+1_sat_of_P_def} the $\rSigma_{n+1}$ satisfaction relation $P|\xi\sats_{n+1}$ for $P|\xi$ is $\rSigma_{n+1}^{Y|\xi}(\{M\})$; in fact, it is defined by $\varrho_{n+1}(M,\cdot)$ over $Y|\xi$.
  \item\label{item:rSigma_n+1_forcing_rel_def} there is an $\rSigma_{n+1}^{P|\xi}(\{\delta\})$ relation $\forces^{\mathrm{strong},\xi}_{n+1}$
  \tu{(}the \emph{strong $\rSigma_{n+1}$ forcing relation}\tu{)} such that for all $\rSigma_{n+1}$ formulas $\varphi(v_0,\ldots,v_{k-1})$ and all $\BB_\delta$-names $\tau_0,\ldots,\tau_{k-1}\in P|\xi$, 
  \[ Y|\xi\sats\varphi(\tau_{0G},\ldots,\tau_{k-1,G}) \iff\exists p\in G\ \Big[p\forces^{\mathrm{strong},\xi}_{n+1}\varphi(\tau_0,\ldots,\tau_{k-1})\Big];\]
  in fact, $\forces^{\mathrm{strong},\xi}_{n+1}$
  is defined by $\psi_{n+1}(\delta,\cdot)$ over $P|\xi$,
  
  \item\label{item:Hull_matchup} for all $X\sub\xi$, we have
  \[ \xi\cap\Hull_{n+1}^{P|\xi}((\delta+1)\cup X)=\xi\cap\Hull_{n+1}^{Y|\xi}((\delta+1)\cup X),\]
  and if $n>0$ then $\delta\in\Hull_{2}^{P|\xi}(\delta)$ and so
  \[ \delta\in\xi\cap\Hull_{n+1}^{P|\xi}(\delta\cup X)=\xi\cap\Hull_{n+1}^{Y|\xi}((\delta+1)\cup X).\]
  \end{enumerate}
  \item\label{item:rho_y+1,p_y+1^P} Suppose $Y$ is $\delta$-condensation-standard $\delta$-sound with $\rho_{y+1}^Y\leq\delta$. Then:
  \begin{enumerate}\item $P$ is $\delta$-sound with  $\rho_{y+1}^P\leq\delta$
  and $p_{y+1}^P\cut\delta=p_{y+1}^Y\cut\delta$.
  \item If $y\neq 0$ or $P$ is active or $\delta$ is not the largest cardinal of $P$ then $\delta\notin p_{y+1}^{P}$.
  \item Suppose $y=0$ and $P$ is passive with largest cardinal $\delta$. Then $p_1^{P}\cut\delta=\{\gamma\}=p_1^Y\cut\delta$ for some $\gamma\geq\delta$ \tu{(}possibly $\gamma=\delta$\tu{)}.
\end{enumerate}

\item\label{item:approximate_functions_degree_n} Let $\zeta\in(\delta,\OR^P]$
be such that $Y|\zeta$ is active.
Let $\kappa=\crit(F^{Y|\zeta})$, so $\delta<\kappa$. Let $(\zeta,0)\leq(\xi,n)\leq(\OR^P,y)$
with $\kappa^{+Y|\xi}=\kappa^{+Y|\zeta}$
and $\kappa<\rho_n^{Y|\xi}$
\tu{(}so $P|\xi$ is $n$-sound and $\kappa<\rho_n^{P|\xi}=\rho_n^{Y|\xi}$\tu{)}.
Let $a\in[\lh(F^{Y|\zeta})]^{<\om}$
and $f:[\kappa]^{|a|}\to P|\xi$
with $f$ being $\bfrSigma_n^{Y|\xi}$-definable.
Then
there is $X\in P|\xi$
with $X\in(F^{P|\xi})_a$
such that $f\rest X$ is $\bfrSigma_n^{P|\xi}$-definable.
\end{enumerate}
 \end{lem}
\begin{proof}[Proof sketch]
The formula $\psi_n$ will be used to define the strong $\rSigma_n$ forcing relation for forcing with $\BB_\delta$ over the models $P|\xi$
(where $\xi>\delta$, from the parameter $\delta$), and the formula $\varrho_n$ to define the $\rSigma_n$ satisfaction relation $P|\xi\sats$
for $P|\xi$ over $Y|\xi$ (where $\xi>\delta$, and from parameter $M$). The explicit formulation of these formulas will be left to the reader, but it comes out of the uniformity of the proof.

 The proof of the conditions claimed to hold
 regarding ordinals $\xi$ are proved by induction on $\xi$, with a sub-induction on $n<\om$.
 Parts \ref{item:delta_reg_in_Y} and \ref{item:P-con_exists} will then follow at the end (note that the existence  of $\mathscr{P}^Y(M)$ is certainly non-trivial, because part of it is that $P$ is in fact a premouse,
 which of course makes significant demands on its extender sequence, and all of its proper segments must be fully sound, etc);
 but we also maintain by induction on $\xi$
 that $P|\xi$ is a premouse and that $\delta$ is a regular cardinal in $Y|\xi$.
  Part \ref{item:unif_def_of_P-con} is easy.
 Parts \ref{item:F^P|xi_and_F^Y|xi} 
 and \ref{item:P|xi[Y|delta]} are then straightforward to verify,
 using the inductive hypotheses,
 and the fact that $Y|\delta$ is $(P|\xi,\BB_\delta^{P|\xi})$-generic
 (this is because $M$ is a $P$-base for $Y$ at $\delta$ and because there is no $\gamma\in[\delta,\xi)$
 such that $\rho_\om^{P|\gamma}<\delta$,
 and  $P|\xi\sats$``$\delta$ is Woodin''). 
 
 Part \ref{item:fs_matchup_proper_segs}: Let $\delta<\xi\leq\OR^P$;
 note that $\delta<\rho_0^{P|\xi}=\rho_0^{N|\xi}$ (using part \ref{item:F^P|xi_and_F^Y|xi}).  Parts \ref{item:inductive_P|xi_fs},
 \ref{item:inductive_P|xi_when_delta_not_in_p},
 \ref{item:inductive_P|xi_when_delta_in_p} for $n=0$ are trivial, and part \ref{item:rSigma_n+1_sat_of_P_def} for $n=0$
 is an easy consequence of part \ref{item:unif_def_of_P-con}.
 Part \ref{item:rSigma_n+1_forcing_rel_def} for $n=0$:
 first compute (using the inductive hypotheses) that the $\rSigma_0$-forcing relation $\forces_0^{\xi}$ of $\BB_\delta$ (for forcing over $P|\xi$) is $\rDelta_1^{P|\xi}(\{\delta\})$,
 and that the $\rSigma_0$ forcing theorem holds with respect to the extension $Y|\xi = (P|\xi)[G]$.
 Infer that the strong $\rSigma_1$ forcing relation  $\forces^{\mathrm{strong},\xi}_{1}$ is $\rSigma_1^{P|\xi}(\{\delta\})$ (this is the relation of tuples $(p,\varphi(\vec{v}),\vec{\tau})$, where
 $p\in\BB_\delta$, $\varphi(\vec{v})=\varphi(v_0,\ldots,v_{k-1})$ is an $\rSigma_1$ formula  of form ``$\exists y\ \psi(y,v_0,\ldots,v_{k-1})$'',
 where $\psi$ is $\rSigma_0$, and $\tau_0,\ldots,\tau_{k-1}\in P|\xi$ are $\BB_\delta$-names,
 which asserts   ``there is some $\BB_\delta$-name $\sigma\in P|\xi$
 such that $p\forces_0^\xi\psi(\sigma,\tau_0,\ldots,\tau_{k-1})$''). Since $Y|\xi=(P|\xi)[G]$, this yields part \ref{item:rSigma_n+1_forcing_rel_def}.
Part \ref{item:Hull_matchup} for $n=0$:
This follows readily  from parts \ref{item:rSigma_n+1_forcing_rel_def}
and \ref{item:rSigma_n+1_sat_of_P_def}
(note that since $M$ is definable over $Y|\delta$ from parameters, we have $M\in\Hull_1^{Y|\xi}(\delta+1)$).

Now suppose $n=1$;
we first aim to establish parts \ref{item:inductive_P|xi_fs},
\ref{item:inductive_P|xi_when_delta_not_in_p}, \ref{item:inductive_P|xi_when_delta_in_p}
in this case. This will be achieved by Claims \ref{clm:rho_1_p_1_matchup_if_delta<rho_1}
and \ref{clm:p_matchup_when_rho_1<=delta} below:

\begin{clm}\label{clm:rho_1_p_1_matchup_if_delta<rho_1}
If $\delta<\rho_1^{Y|\xi}$
then:
\begin{enumerate}\item\label{item:rho_1^P|xi=rho_1^Y|xi} $\rho_1^{P|\xi}=\rho_1^{Y|\xi}$,\item\label{item:p_1^P|xi=p_1^Y|xi}$p_1^{P|\xi}=p_1^{Y|\xi}$,
 \item\label{item:P|xi_is_1-sound} $P|\xi$ is $1$-sound (in particular, $1$-solid).
\end{enumerate}
\end{clm}
\begin{proof}
Suppose $\delta<\rho_1^{Y|\xi}$.

Part \ref{item:rho_1^P|xi=rho_1^Y|xi}:
By part \ref{item:Hull_matchup}
and since $Y|\xi$ is $1$-sound, we get
\begin{equation}\label{eqn:P|xi_is_Hull_of_rho_1_etc} P|\xi=\Hull_1^{P|\xi}(\rho_1^{Y|\xi}\cup\{p_1^{Y|\xi}\}),\end{equation}
and therefore $\rho_1^{P|\xi}\leq\rho_1^{Y|\xi}$.
So suppose $\rho_1^{P|\xi}<\rho_1^{Y|\xi}$ and let $\rho=\max(\rho_1^{P|\xi},\delta)$.  Let $\bar{P}=\cHull_1^{P|\xi}(\rho\cup\{p_1^{P|\xi}\})$ and $\pi:\bar{P}\to P|\xi$ the uncollapse.

We have $\delta\in\rg(\pi)$.
To see this, it is enough to see that there is some $\gamma\geq\delta$ with $\gamma\in\rg(\pi)$, since then if $\gamma>\delta$ then there is $\eta<\delta$
such that $\delta$ is the least Woodin of $P|\gamma$ which is $>\eta$, putting $\delta\in\rg(\pi)$. So suppose $\rg(\pi)=\delta$. So $p_1^{P|\xi}\sub\delta$. Since $\delta$ is a (in fact strong) cutpoint of $P|\xi$,
it clearly follows that $P|\xi$ is passive.
But then since $\rg(\pi)=\delta$,
we have $P|\delta\preccurlyeq_{1} P|\xi$.
But then $\Th_{\rSigma_1}^{P|\xi}(\rho_1^{P|\xi}\cup\{p_1^{P|\xi}\})=\Th_{\rSigma_1}^{P|\delta}(\rho_1^{P|\xi}\cup\{p_1^{P|\xi}\})$,
so this theory is in $P|\xi$, a contradiction.

So by part \ref{item:Hull_matchup} at $n=0$,
\[ \xi\cap\rg(\pi)=\xi\cap\Hull_{1}^{Y|\xi}((\delta+1)\cup p_1^{P|\xi}).\]
Let $\bar{Y}=\cHull_1^{Y|\xi}((\delta+1)\cup p_1^{P|\xi})$ and $\pi^+:\bar{Y}\to Y$
the uncollapse map. So $\xi\cap\rg(\pi)=\xi\cap\rg(\pi^+)$,
and so in fact $\rg(\pi)=(P|\xi)\cap\rg(\pi^+)$.
Since $\delta$ is Woodin in $P|\xi$
and by genericity, $\delta$ is a cardinal in $Y|\xi$. And $\delta<\rho_1^{Y|\xi}$.
So $\rho_1^{\bar{Y}}=\delta$
and $\bar{Y}=\Hull_1^{\bar{Y}}(\delta\cup\{x\})$ for some $x\in\bar{Y}$.
And $\sigma$ is $0$-lifting (in fact a near $0$-embedding). So because $Y|\xi$
is $1$-condensation-standard (recall $\xi<\OR^P\leq\OR^Y$), letting $C=\core_1(\bar{Y})$, we get that $\bar{Y}$
is an iterate of $C$, via an above-$\delta$ $0$-maximal tree $\Tt^+$ on $C$,
and that $C\pins Y|\xi$.
Let $D=\mathscr{P}^C(M)$;
so $D\pins P|\xi$.
By induction, $\Tt^+$ translates
to a tree $\Tt$ on $D$,
which is also above-$\delta$ $0$-maximal,
and $M^{\Tt}_\infty=\mathscr{P}^{M^{\Tt^+}_\infty}(M)=\mathscr{P}^{\bar{Y}}(M)$.
But $\mathscr{P}^{\bar{Y}}(M)=\bar{P}$,
by the internal definability of P-construction and that $\rg(\pi)=(P|\xi)\cap\rg(\pi^+)$.
So $\bar{P}=M^\Tt_\infty$.
But since $D\in P|\xi$
and $\Tt$ is finite, therefore $\bar{P}\in P|\xi$, a contradiction.

Part \ref{item:p_1^P|xi=p_1^Y|xi}: By part \ref{item:rho_1^P|xi=rho_1^Y|xi} and line (\ref{eqn:P|xi_is_Hull_of_rho_1_etc}),
we have $p_1^{P|\xi}\leq p_1^{Y|\xi}$.
But if $p_1^{P|\xi}<p_1^{Y|\xi}$, then we can argue much as before for a contradiction,
 again using the fact that $Y|\xi$
is $1$-condensation-standard
(but this time using the clause regarding  $p_1^{Y|\xi}$).

Part \ref{item:P|xi_is_1-sound}: Since $p_1^{P|\xi}=p_1^{Y|\xi}$,
the fact that $P|\xi$ is $1$-solid
is proved using the proof of part  \ref{item:p_1^P|xi=p_1^Y|xi}. And
 by line (\ref{eqn:P|xi_is_Hull_of_rho_1_etc}) (and by parts \ref{item:rho_1^P|xi=rho_1^Y|xi}
 and \ref{item:p_1^P|xi=p_1^Y|xi}), it follows that $P|\xi$ is $1$-sound.
\end{proof}

\begin{clm}\label{clm:p_matchup_when_rho_1<=delta}
 If $\rho_1^{Y|\xi}=\delta$ then:
 \begin{enumerate}
  \item\label{item:rho_match_when_rho_1^Y|xi<=delta} $\rho_1^{P|\xi}=\delta$,
  \item\label{item:delta_in_p_1_iff_delta=max(p_1)} If $P|\xi$ is active or $\delta$ is not the largest cardinal of $P|\xi$ then $\delta\notin p_1^{P|\xi}$,
  \item\label{item:if_P|xi_passive_with_lgcd=delta_p_1} If $P|\xi$ is passive and $\delta$ is the largest cardinal of $P|\xi$,
  then $p_1^{P|\xi}=\{\gamma\}$ for some $\gamma\geq\delta$,
  \item\label{item:rho_1=delta_P|xi_1-solid}$P|\xi$ is $1$-sound (in particular, $P|\xi$ is $1$-solid)
  with $p_1^{P|\xi}=p_1^{Y|\xi}$.
 \end{enumerate}
\end{clm}
\begin{proof}
Part \ref{item:rho_match_when_rho_1^Y|xi<=delta}:
 $\rho_1^{P|\xi}\leq\delta$
 like in the proof of Claim \ref{clm:rho_1_p_1_matchup_if_delta<rho_1},
and $\delta\leq\rho_1^{P|\xi}$ by hypothesis (of the lemma).

Part \ref{item:delta_in_p_1_iff_delta=max(p_1)}:
If $P|\xi$ is active then $\crit(F^{P|\xi})\in\Hull_1^{P|\xi}(\emptyset)$, and $\delta<\kappa$ as $\delta$ is a strong cutpoint of $Y$;
since $\delta$ is a successor Woodin of $P|\xi$, therefore $\delta\in\Hull_1^{P|\xi}(\delta)=\Hull_1^{P|\xi}(\rho_1^{P|\xi})$,
so $\delta\notin p_1^{P|\xi}$. If $P|\xi$ is passive then it has a largest cardinal $\theta\geq\delta$, and $p_1^{P|\xi}\not\sub\theta$, by Lemma \ref{lem:p_1_for_passive}. But then letting $\gamma=\max(p_1^{P|\xi})$, like before,
we have $\delta\in\Hull_1^{P|\xi}(\delta\cup\{\gamma\})$.

Part \ref{item:if_P|xi_passive_with_lgcd=delta_p_1} is by Lemma \ref{lem:p_1_for_passive}.

Part \ref{item:rho_1=delta_P|xi_1-solid}:
We have \begin{equation}\label{eqn:delta_in_Hull_1^P|xi(delta_cup_p_1)}\delta\in\Hull_1^{P|\xi}(\delta\cup\{p_1^{Y|\xi}\})\end{equation} (note we refer here to $p_1^{Y|\xi}$, not $p_1^{P|\xi}$).
For otherwise
 $P|\xi$ is passive, and  by induction, $P|\xi$ and $Y|\xi$ have the same cardinals $\geq\delta$, but then $p_1^{Y|\xi}\cut\delta\neq\emptyset$,
 and so since $\delta$ is a successor Woodin of $P|\xi$, we get line (\ref{eqn:delta_in_Hull_1^P|xi(delta_cup_p_1)}).
 Like in the proof of Claim \ref{clm:rho_1_p_1_matchup_if_delta<rho_1}, it follows that
 \[ P|\xi=\Hull_1^{P|\xi}(\delta\cup\{p_1^{Y|\xi}\}). \] So $p_1^{P|\xi}\leq p_1^{Y|\xi}$ and
  it suffices to see that
 $p_1^{Y|\xi}$ is $1$-solid for $P|\xi$.
If $\delta\notin p_1^{Y|\xi}$, this is also like in the proof of Claim \ref{clm:rho_1_p_1_matchup_if_delta<rho_1}. So suppose $\delta\in p_1^{Y|\xi}$.
Then analogously to the proofs of
parts \ref{item:delta_in_p_1_iff_delta=max(p_1)}
and \ref{item:if_P|xi_passive_with_lgcd=delta_p_1},
$Y|\xi$ is passive with largest cardinal $\delta$
and $p_1^{Y|\xi}=\{\delta\}$,
and so we only need to see that $\cHull_1^{P|\xi}(\delta)\in P|\xi$. But $\cHull_1^{P|\xi}(\delta)=P|\delta$, by condensation as in the proof of Lemma \ref{lem:p_1_for_passive}, which suffices.
\end{proof}

Recall that an $\rSigma_2$ formula $\varphi(\vec{v})$ (in free variables $\vec{v}$) has the form
\[ \varphi(\vec{v})\iff \exists q,\alpha,t\ \Big[T_1(q,\alpha,t)\wedge\psi(q,\alpha,t,\vec{v})\Big]\]
where $\psi$ is $\rSigma_1$,
and $T_1(q,\alpha,t)$ asserts that $\alpha<\rho_1$ and $t=\Th_{\rSigma_1}(\alpha\cup\{q\})$.
Since $\rho_1^{P|\xi}=\rho_1^{Y|\xi}$,
the variable $\alpha$
then automatically ranges
over the same ordinals,
whether we are interpreting
$\rSigma_2$ formulas over $P|\xi$
or over $Y|\xi$.

\begin{clm}
Part \ref{item:rSigma_n+1_sat_of_P_def} 
holds for $\xi$ and $n=1$.\end{clm}
\begin{proof} Since the $\rSigma_1$ satisfaction relation $P|\xi\sats$ is $\rSigma_1^{Y|\xi}(\{M\})$,
and since $P|\xi$ is $\rDelta_1^{Y|\xi}(\{M\})$, given $\alpha<\rho_1^{P|\xi}$
and $q\in P|\xi$,
we can easily recover $\Th_{\rSigma_1}^{P|\xi}(\alpha\cup\{q\})$ from $\Th_{\rSigma_1}^{Y|\xi}(\alpha\cup\{q,M\})$. It easily follows
that the $\rSigma_2$ satisfaction relation of $P|\xi$ is definable over $Y|\xi$ in the desired manner.\end{proof}

\begin{clm}
 Part \ref{item:rSigma_n+1_forcing_rel_def}
 holds for $\xi$ and $n=1$.
\end{clm}
\begin{proof} Recall that $\delta\leq\rho_1^{P|\xi}=\rho_1^{Y|\xi}$. We consider two cases:
\begin{case}\label{case:delta<rho_1_def_strong_rSigma_2_forcing_rel} $\delta<\rho_1^{P|\xi}=\rho_1^{Y|\xi}$.

Define the \emph{strong $\rSigma_2$ forcing relation} $p\forces^{\mathrm{strong},\xi}_2\varphi(\sigma)$ of $P|\xi$, where $p\in\BB_\delta$,
$\varphi(v)$ is  an $\rSigma_2$ formula of
form
\[ \varphi(v)\iff\exists q,\alpha,t \Big[T_1(q,\alpha,t)\wedge\psi(q,\alpha,t,v)\Big]\]
where $\psi$ is $\rSigma_1$,
and $\sigma$ is a $\BB_\delta$-name
in $P|\xi$,
to assert ``there are $r,\beta,u\ [T_1(r,\beta,u)$ and $r$ has form $(\dot{q},\delta)$ for some $\BB_\delta$-name $\dot{q}$,
 and letting $\dot{t}$ be the canonical
 name for the $\rSigma_1$-theory of the parameters $\check{\beta}\cup\{\dot{q}\}$
 of the $\BB_\delta$-generic extension
 determined by $u$, then $p\forces^\xi_{1}\psi(\dot{q},\check{\beta},\dot{t},\sigma)$''. Here the \emph{canonical name}
 $\dot{t}$ is just the natural name for
 the set of all pairs $(\varrho,\vec{x})$ with $\varrho$
 an $\rSigma_1$ formula and
 $\vec{x}\in(\check{\beta}\cup\{\dot{q}\})^{<\om}$, such that there is $s\in\dot{G}$ such that the formula ``$s\forces_1\varrho(\vec{x})$'' is in $\check{u}$. By induction,
 it is straightforward to see  this works.
\end{case}

\begin{case}$\delta=\rho_1^{P|\xi}=\rho_1^{Y|\xi}$.

In this case we argue much as in the proof of \cite[Lemma 9, Theorem 10]{conjecture_mouse_order_for_weasels}.
 
 \begin{sclm}\label{sclm:Boolean_value_exists}
 Let $\tau\in P|\xi$ be a $\BB_\delta$-name and let $\varphi$ be an $\rSigma_1$ formula. Then the Boolean value $\|\varphi(\tau)\|$ is an element of $\BB_\delta^{P|\xi}$.\end{sclm}

 That is, there is $p\in \BB_\delta$ such that for every $q\leq p$ there is $r\leq q$ such that $r\forces^{\mathrm{strong}}_1\varphi(\tau)$,
 and there is no $q\leq 1-p$ such that $q\forces^{\mathrm{strong}}_1\varphi(\tau)$. Clearly this $p$ is uniquely determined, and we define $\|\varphi(\tau)\|=p$.
 \begin{proof}
 Let $\rho=\rho_0^{P|\xi}$. (Recall
 that if $P|\xi$, or equivalently $Y|\xi$,
 is type 3, then $\rho_0^{P|\xi}<\xi$,
 but $\delta<\xi$ also in this case, since $\delta$ is a strong cutpoint of $Y|\xi$.)
 
 We define an ordinal $\lambda\leq\delta$ and sequences $\left<p_\alpha\right>_{\alpha<\lambda}\sub\BB_\delta$ and $\left<\eta_\alpha\right>_{\alpha<\lambda}\sub\rho$, with the sequences specified
recursively as follows. Set $p_0=0$ and $\eta_0=0$.
Suppose we have $\left<p_\alpha,\eta_\alpha\right>_{\alpha<\gamma}$ where $0<\gamma\leq\delta$. If $\gamma=\delta$ we stop the construction (and set $\lambda=\delta$). Otherwise let $\eta=\sup_{\alpha<\gamma}\eta_\alpha$.
If $\eta<\rho$ and there is $\eta'\in(\eta,\rho)$
such that\footnote{The notation $\wr$ is discussed in Definition \ref{dfn:unif_cofinal} *** (should move this back).  It is the usual stratification of $\rSigma_1^{P|\xi}$ into a hierarchy of length $\rho_0$. The version for $\rSigma_2$ has length $\rho_1$.} \[\|((P|\xi)\wr\eta')\sats\varphi(\tau)\|\neq\|((P|\xi)\wr\eta)\sats\varphi(\tau)\|,\]
then let $\eta_\gamma$ be the least such $\eta$ and let $p_\gamma=\|((P|\xi)\wr\eta_\gamma)\sats\varphi(\tau)\|$.
Otherwise we stop the construction and set $\lambda=\gamma$.
This completes the recursion. So $\lambda\leq\delta$ by definition.

Note that $\left<p_\alpha\right>_{\alpha<\lambda}$ is $\bfrSigma_1^{P|\xi}$
and $p_\beta$ is strictly weaker than $p_\alpha$ for $\alpha<\beta$. So
because $\delta$ is $\bfrSigma_1^{P|\xi}$-Woodin, and in particular we get the $\delta$-cc with respect to $\bfrSigma_1^{P|\xi}$-definable antichains, it follows that $\lambda<\delta$
and there is $\beta<\delta$
such that $\left<p_\alpha\right>_{\alpha<\lambda}\in P|\beta$ (the latter as there is no $\bfrSigma_1^{P|\xi}$ singularization of $\delta$).
Since $\rho_1^{P|\xi}\geq\delta$,
therefore $\left<p_\alpha\right>_{\alpha<\lambda}\in P|\delta$. So $p=\sum_{\alpha<\lambda}p_\alpha\in\BB_\delta$,
and note that $p=\|\varphi(\tau)\|$, as desired.
 \end{proof}
 \begin{sclm}\label{sclm:rho_1=delta_theory_names_exist}
  Let $\tau\in P|\xi$ be a $\BB_\delta$-name and let $\gamma<\delta$.
  Let $\mathscr{I}$ be the set of pairs $(\varphi,\vec{x})$ where $\varphi$
  is an $\rSigma_1$ formula
  and $\vec{x}=(\tau,\check{\gamma}_0,\ldots,\check{\gamma}_{m-1})$
  for some $m<\om$ and some $(\gamma_0,\ldots,\gamma_{m-1})\in\gamma^{m}$.
  Then there is $\beta<\delta$
  such that:
  \begin{enumerate}[label=--]\item $\|\varphi(\vec{x})\|\in P|\beta$ for all $(\varphi,\vec{x})\in\mathscr{I}$, and
   \item the function $(\varphi,\vec{x})\mapsto\|\varphi(\vec{x})\|$, with domain $\mathscr{I}$, is in $P|\beta$.
  \end{enumerate}
 \end{sclm}
\begin{proof} This is an elaboration of the proof of Subclaim \ref{sclm:Boolean_value_exists}.
Define a sequence $\left<\vec{p}_\alpha,\eta_\alpha\right>_{\alpha<\lambda}$,
for some $\lambda\leq\delta$, with $\eta_\alpha<\rho=\rho_0^{P|\xi}$ and $\vec{p}_\alpha:\mathscr{I}\to\BB_\delta$ and each $\vec{p}_\alpha\in P|\delta$, and $\left<\vec{p}_\alpha,\eta_\alpha\right>_{\alpha<\zeta}\in P|\rho$ for each $\zeta<\lambda$, by recursion as follows. Set $\eta_0=0$ and $\vec{p}_0(\varphi,\vec{x})=0$ for all $(\varphi,\vec{x})$.
Now let $\zeta<\delta$ be given and suppose we have defined $\left<\vec{p}_\alpha,\eta_\alpha\right>_{\alpha<\zeta}$.
Let $\eta=\sup_{\alpha<\zeta}\eta_\alpha$
and suppose there is $\eta'\in(\eta,\rho)$
and $(\varphi,\vec{x})\in\mathscr{I}$ such that
\[ \|((P|\xi)\wr\eta')\sats\varphi(\vec{x})\|\neq\|((P|\xi)\wr\eta)\sats\varphi(\vec{x})\|.\]
Then let $\eta_\zeta$ be the least such $\eta'$, and for each $(\varphi,\vec{x})\in\mathscr{I}$ set
\[ \vec{p}_\zeta(\varphi,\vec{x})=\|((P|\xi)\wr\eta_\zeta)\sats\varphi(\vec{x})\|.\]
This completes the recursion; note that it is $\bfrSigma_1^{P|\xi}$.

We have $\lambda\leq\delta$ by construction. We want to see that $\lambda<\delta$, so suppose otherwise. For $\alpha<\delta$, let $\mathscr{I}_\alpha$ be the set of all
$(\varphi,\vec{x})\in\mathscr{I}$ such that
\[ \|((P|\xi)\wr\eta_\alpha)\sats\varphi(\vec{x})\|\neq\|((P|\xi)\wr\eta')\sats\varphi(\vec{x})\|\]
where $\eta'=\sup_{\beta<\alpha}\eta_\beta$. Note that $\emptyset\neq\mathscr{I}_\alpha\in P|\delta$ (and $\mathscr{I}\in P|\delta$) and the map $\alpha\mapsto\mathscr{I}_\alpha$ is $\bfrSigma_1^{P|\xi}$. But then an application of $\bfrSigma_1^{P|\xi}$-Woodinness gives that there is some $\mathscr{J}\sub\mathscr{I}$ such that for
cofinally many $\alpha<\delta$,
we have $\mathscr{J}=\mathscr{I}_\alpha$.
Fixing $(\varphi,\vec{x})\in\mathscr{J}$,
we now get an $\bfrSigma_1^{P|\xi}$-definable antichain corresponding to $(\varphi,\vec{x})$,
a contradiction.  So $\lambda<\delta$.

Since for each $\alpha<\lambda$,
$\vec{p}_\alpha\in P|\delta$,
a further application of $\bfrSigma_1^{P|\xi}$-Woodinness gives that there is $\beta<\delta$ with $\vec{p}_\alpha\in P|\beta$ for all $\alpha<\lambda$.
So since $\rho_1^{P|\xi}=\delta$,
therefore $\left<\vec{p}_\alpha\right>_{\alpha<\lambda}\in P|\delta$, which easily suffices.
\end{proof}
\end{case}

Given $\tau,\gamma$ as in Subclaim \ref{sclm:rho_1=delta_theory_names_exist},
let $t_{\tau,\gamma}$ be the function $(\varphi,\vec{x})\mapsto\|\varphi(\vec{x})\|$ shown to exist in $P|\delta$ there.

\begin{sclm}\label{sclm:(tau,gamma)_mapsto_t_(tau,gamma)_def}
 The function $(\tau,\gamma)\mapsto t_{\tau,\gamma}$ is $\rSigma_2^{P|\xi}(\{\delta\})$-definable.
\end{sclm}
\begin{proof}
Given $(\tau,\gamma)$ and the corresponding domain $\mathscr{I}$, note that $t_{\tau\gamma}$ is the unique function $\left<p_{\varphi,\vec{x}}\right>_{(\varphi,\vec{x})\in\mathscr{I}}$
such that there is a function $\left<A_{\varphi,\vec{x}}\right>_{(\varphi,\vec{x})\in\mathscr{I}}\in P|\delta$
such that for each $(\varphi,\vec{x})\in\mathscr{I}$, we have:
\begin{enumerate}[label=--]
 \item $A_{\varphi,\vec{x}}\sub\BB_\delta$
 and $p_{\varphi,\vec{x}}=\sum A_{\varphi,\vec{x}}$,
 \item  $q\forces^{\mathrm{strong}}_{\BB_\delta}\varphi(\vec{x})$
 for each $q\in A_{\varphi,\vec{x}}$,
 \item there is no $q\leq (1-p_{\varphi,\vec{x}})$ such that $q\forces^{\mathrm{strong}}_{\BB_\delta}\varphi(\vec{x})$.
\end{enumerate}
(The existence of such a function $\left<A_{\varphi,\vec{x}}\right>_{(\varphi,\vec{x})\in\mathscr{I}}$ follows
directly from the proof of Subclaim \ref{sclm:rho_1=delta_theory_names_exist}.)
Note that, moreover, these properties (of $\left<p_{\varphi,\vec{x}}\right>_{(\varphi,\vec{x})\in\mathscr{I}},\left<A_{\varphi,\vec{x}}\right>_{(\varphi,\vec{x})\in\mathscr{I}}$)
are determined by $\Th_{\rSigma_1}^{P|\xi}(\beta\cup\{y\})$
for an appropriate parameter $y\in P|\xi$ and ordinal $\beta<\delta$. This yields the subclaim.
\end{proof}
Finally note that $t_{\tau,\gamma}$
is easily converted into a $\BB_\delta$-name for $\Th_{\rSigma_1}^{P|\xi[\dot{G}]}(\tau_{\dot{G}}\cup\gamma)$.
Using this to substitute for the ``canonical name for the $\rSigma_1$ theory'' used in Case \ref{case:delta<rho_1_def_strong_rSigma_2_forcing_rel},
it is easy to define an appropriate strong $\rSigma_2$ forcing relation with the right properties.
\end{proof}

\begin{clm}
 Part \ref{item:Hull_matchup} for $\xi$ and $n=1$.\end{clm}
 \begin{proof}This follows directly from parts \ref{item:rSigma_n+1_sat_of_P_def}, \ref{item:rSigma_n+1_forcing_rel_def}, much as in the $n=0$ case;
 we have $\delta\in\Hull_{2}^{P|\xi}(\delta)$  since $\delta$ is a successor Woodin cardinal of $P|\xi$.\end{proof}
 
 The $n=2$ version is very similar to that for $n=1$. 
 We have $\delta\in\Hull_2^{P|\xi}(\delta)$ since $\delta$ is a successor Woodin of $P|\xi$.
For the definability of the strong $\rSigma_3$ forcing relation, in case $\rho_2^{P|\xi}=\rho_2^{Y|\xi}=\delta$,
one uses the standard stratification of $\rSigma_2$
through ordinals $\alpha<\rho_1$
(with at least the parameter $p_1^{P|\xi}=p_1^{Y|\xi}$ in the theories considered)
as opposed to the  stratification of $\rSigma_1$ through ordinals $<\rho_0^{P|\xi}$; see \cite[\S2]{fsit} for details (the \emph{min-terms} of \cite[\S5]{V=HODX_pub} also use this stratification). (The proof of Subclaim \ref{sclm:(tau,gamma)_mapsto_t_(tau,gamma)_def} generalizes directly, though,
without needing to refer to $p_1^{P|\xi}$,
so we only need the parameter $\delta$
to define the $\rSigma_3$ strong forcing relation.)

For $n>2$ it is then a direct generalization of $n=2$; we leave the remaining details to the reader.

Part \ref{item:rho_y+1,p_y+1^P}:
Because we have already established
part \ref{item:fs_matchup_proper_segs}, this is like part of its proof.

Part \ref{item:approximate_functions_degree_n}:
By part \ref{item:F^P|xi_and_F^Y|xi},
we may assume $n>0$ and that $f:[\kappa]^{|a|}\to\xi$. Let $\psi$ be an $\rSigma_n$ formula and $\gamma<\xi$
be such that $f(u)=\beta$ iff $Y|\xi\sats\psi(u,\beta,\gamma)$, for all $u\in[\kappa]^{|a|}$ and $\beta<\xi$.
By part \ref{item:fs_matchup_proper_segs}, the strong $\rSigma_n$ forcing relation of $P|\xi$ is $\bfrSigma_n^{P|\xi}$.

Let $A$ be the set of all $p\in\BB_\delta$
such that there are $u\in[\kappa]^{|a|}$
and $\beta_0<\beta_1<\xi$ such that \[(P|\xi)\sats\text{``}p\forces^{\mathrm{strong}}_n\psi(u,\beta_0,\gamma)\wedge\psi(u,\beta_1,\gamma)\text{''}.\]
Since $\kappa<\rho_n^{Y|\xi}=\rho_n^{P|\xi}$, we have $A\in P|\xi$.
By genericity then, there is $p_0\in G$
such that no $q\leq p_0$ is in $A$.

Now applying the $\rSigma_n$ forcing theorem in $Y|\xi=(P|\xi)[G]$ (with the strong $\rSigma_n$ forcing relation), where we can also refer to the parameter $G$,
we can define a function $g:[\kappa]^{|a|}\to\BB_\delta$, in an $\bfrSigma_n^{Y|\xi}$ manner, such that for each $u$, $g(u)\in G$
and $g(u)\leq p_0$ and there is $\beta$ such that $P|\xi\sats \text{``}g(u)\forces^{\mathrm{strong},\xi}_n\psi(u,\beta,\gamma)\text{''}$. 
Since $\kappa<\rho_n^{Y|\xi}$,
we therefore actually get $g\in Y|\xi$,
so $g\in Y|\zeta$.
So there is $X\in(F^{Y|\zeta})_a$
such that $g\rest X$ is constant with constant value $p_1$.
We can then refine $X$ to some $X'\in (F^{P|\zeta})_a$.
Now note that  $f\rest X'$ is $\bfrSigma_n^{P|\xi}$,
since for $u\in X'$, we have $f(u)=\beta$
iff $P|\xi\sats\text{``}\forces^{\mathrm{strong}}_n\psi(u,\beta,\gamma)\text{''}$ (this uses that $p_1\notin A$).
 \end{proof}
 
 \begin{dfn}\label{dfn:translation_pair}
  Let $Y$ be a premouse and $M$ be a P-base of $Y$ at $\delta\in\OR^Y$. Let $P=\mathscr{P}^Y(M)$. Suppose $\OR^P=\OR^Y$. Let $n<\om$
  and suppose that $Y$ is $n$-sound
  with $\delta<\rho_n^{Y}$. 
  
 Let $\Tt$ be an above-$\delta$, $n$-maximal putative iteration tree on $Y$.
 Let $\Tt'$ be an above-$\delta$, $n$-maximal putative
 iteration tree on $P$.
 We say that $(\Tt,\Tt')$
 is a \emph{translation pair}
 iff $\lh(\Tt)=\lh(\Tt')$,
 ${<^\Tt}={<^{\Tt'}}$
 and $\lh(E^\Tt_\alpha)=\lh(E^{\Tt'})_\alpha$
 for all $\alpha+1<\lh(\Tt)$.
 \end{dfn}

 \begin{lem}\label{lem:P-construction_tree_translation}
  Let $Y,M,\delta,P,n$ be as in Definition \ref{dfn:translation_pair}.
  Let $\Tt$ be an above-$\delta$, $n$-maximal
  putative tree on $Y$.
  Then:
  \begin{enumerate}\item\label{item:unique_Tt'_translation_pair_exists} There is a unique $\Tt'$
  such that $(\Tt,\Tt')$
  is a translation pair.
  \item For each $\alpha<\lh(\Tt)$, we have
  $M^{\Tt'}_\alpha=\mathscr{P}^{M^\Tt_\alpha}(M)$ and this has ordinal height $\OR(M^\Tt_\alpha)$.
  \item For all $\alpha+1<\lh(\Tt)$,
  we have $\alpha+1\in\mathscr{D}^{\Tt'}$ iff $\alpha+1\in\mathscr{D}^\Tt$.
  \item For all $\alpha+1<\lh(\Tt)$,
  we have $M^{*\Tt'}_{\alpha+1}=\mathscr{P}^{M^{*\Tt}_{\alpha+1}}(M)$ and this has ordinal height $\OR(M^{*\Tt}_{\alpha+1})$.
  \item  For all $\alpha+1<\lh(\Tt)$,
  we have $\deg^{\Tt'}_{\alpha+1}=\deg^\Tt_{\alpha+1}$.
  \item\label{item:translation_pair_iteration_maps_agree} For all $\alpha<\lh(\Tt)$
  and $\beta<^\Tt\alpha$, if $(\beta,\alpha]^\Tt\cap\dropset^{\Tt}=\emptyset$
  then $i^{\Tt'}_{\beta\alpha}=i^\Tt_{\beta\alpha}\rest M^{\Tt'}_\beta$.
  \end{enumerate}
  Likewise conversely, given any above-$\delta$, $n$-maximal putative tree $\Tt'$ on $P$.
  
  Therefore for any $\theta,\chi\in\OR$,
  \begin{enumerate}[label=(\alph*)]
   \item\label{item:P-con_n-max_strategy_translation}
$P$ is above-$\delta$ $(n,\theta)$-iterable   iff $Y$ is above-$\delta$
$(n,\theta)$-iterable.
\item\label{item:P-con_n-max_stack_strategy_translation} $P$ is above-$\delta$ $(n,\theta,\chi)^*$-iterable iff $Y$ is above-$\delta$
$(n,\theta,\chi)^*$-iterable,
 \end{enumerate}
 and in both cases \ref{item:P-con_n-max_strategy_translation}
 and \ref{item:P-con_n-max_stack_strategy_translation}, the witnessing iteration strategies occur in pairs $(\Sigma,\Sigma')$, with  $\Sigma$ a strategy for $Y$ and $\Sigma'$ one for $P$, such that $\Tt$ is via $\Sigma$ iff $\Tt'$ is via $\Sigma'$, for any translation pair $(\Tt,\Tt')$ on $(Y,P)$ (or sequence of translation pairs in case \ref{item:P-con_n-max_stack_strategy_translation}).
 \end{lem}
\begin{proof}
Fix $\Tt$ as stated. The uniqueness of  a corresponding $\Tt'$ is clear.
 The proof of its existence and properties \ref{item:unique_Tt'_translation_pair_exists}--\ref{item:translation_pair_iteration_maps_agree} is by induction on 
 on $\lh(\Tt)$. It follows readily from 
 Lemma \ref{lem:P-con_basic_props}
 and a slight generalization of its part \ref{item:approximate_functions_degree_n}. (That is,
 given $\alpha+1<\lh(\Tt')=\lh(\Tt)$
 and $d=\deg^{\Tt'}_{\alpha+1}=\deg^\Tt_{\alpha+1}$, we need to know that the $\bfrSigma_d^{M^{*\Tt'}_{\alpha+1}}$ functions with codomain $M^{*\Tt'}_{\alpha+1}$  are equivalent to the $\bfrSigma_d^{M^{*\Tt}_{\alpha+1}}$ functions with that codomain, for the purposes of forming the ultrapowers.
 But this is just a slight generalization
 of part \ref{item:approximate_functions_degree_n}.) of Lemma \ref{lem:P-con_basic_props}.
\end{proof}

\subsection{Mouse operators (should be obsolete)}\label{sec:mouse_ops}
(THIS SECTION \S\ref{sec:mouse_ops}
SHOULD HAVE BEEN SUPERSEDED BY \S\ref{sec:pointclasses}, so the reader should be safe ignoring it.)
\begin{dfn}
A \emph{mouse operator} is a function $F:\RR\to\HC$ such that $F(x)$ is an $\om$-mouse over $x$ 
for each $x\in\RR$.

A \emph{Q-operator} is a mouse operator $F$ such that for some $q<\om$,
for each $x\in\RR$, $F(x)$ is a Q-mouse of Q-degree $q$ such that $\rho_{q+1}^Q=\om$.
We say $q$ is the \emph{Q-degree} of $F$ and write $q=q^F$.

A Q-operator $F$ of Q-degree $q$ is \emph{strongly rebuilding} iff for every $x\in\RR$ and $z\leq_T 
x$ and $z$-premouse $M$, if $M$ is $(F(x),q)$-P-good and $N=\mathscr{P}^Y(M)$ is $(q,\om_1+1)$-iterable 
and $\delta^Y$-sound and $\rho_{q+1}$then either $F(z)\pins N$ or $\core_{q+1}(N)=F(z)$.
\end{dfn}

\begin{dfn}Let $F$ be a mouse operator and $\varphi\in\Ll_\pm$.\footnote{We will interpret 
$\varphi$ 
over relativized premice.} We 
say that $F$ is \emph{Woodin+$\varphi$-minimal} iff
all countable transitive swo'd sets $X$, $F(X)$ is the least Woodin $X$-mouse $M$ such that 
$M\sats\varphi(\delta^M)$. We write $F_{\mathrm{W}+\varphi}$ for $F$ in this case.

Let $F$ be Woodin+$\varphi$-minimal.
We say that $F$ \emph{relativizes above the Woodin} iff for all $X\in\dom(F)$:
\begin{enumerate}
\item Let $N$ be an $n$-sound, $(n,\om_1+1)$-iterable Woodin premouse over $X$.
Let $R=\core_{n+1}(N)$ (note $R$ is a Woodin premouse).
Then $N\sats\varphi(\delta^N)$ iff $R\sats\varphi(\delta^R)$.
 \item Let $M\ins F(X)$ be 
a Woodin premouse such that $\delta=\delta^M$ is a cardinal of $F(X)$. Let $A\in 
M|\delta^M$ be transitive swo'd. Let $N$ be an $A$-premouse which is 
$\Delta_2^{M|\delta}(\{A\})$, with $\OR^N=\delta$, 
and such that 
$\J(N)$ is a Woodin premouse with $\delta^{\J(N)}=\delta$, and $M|\delta$ is generic 
over $\J(N)$ for the $\delta$-generator extender algebra determined by extenders in $\es^N$. 
Suppose that for every $R$ with $M|\delta\ins R\pins M$,
letting $R'$ be the P-construction of $R$ above $N$, then $\delta\leq\rho_\om^{R'}$ and
$\J(R')$ is a Woodin premouse with $\delta^{\J(R')}=\delta$
(therefore $M'$ is a well-defined Woodin premouse with $\delta^{M'}=\delta$).
Then $M\sats\varphi(\delta)$ iff $M'\sats\varphi(\delta)$.
\end{enumerate}
\end{dfn}

Clearly there are many large cardinal properties $\varphi$ ``above the Woodin'' which are 
such that $F=F_{\mathrm{W}+\varphi}$ relativizes above the 
Woodin (if $F$ exists). This includes many fine large cardinal properties, such as the existence of 
an extender in $\es$ with more than 1 generator, etc.

We will show that if $F$ relativizes above the Woodin then $F$ 
``reconstructs itself'' in a certain manner. We now define the background construction we use 
for this.

\begin{dfn}
 Let $R$ be pre-relevant. A pseudo-iteration tree $\Tt$ on $M$ is \emph{small} iff for every 
$\alpha+1<\lh(\Tt)$, $M^\Tt_\alpha|\lh(E^\Tt_\alpha)\sats$``There is no Woodin cardinal''.
\end{dfn}

The pseudo-iteration trees on pre-relevant structures in which we're interested break into two 
parts, the first being small, and the second based on a Q-structure, above its least Woodin:

\begin{lem}
 Let $R$ be pre-relevant and $n=q^R$. Let $\Vv$ be a non-small $n$-maximal pseudo-iteration tree on 
$R$.
Let $\alpha$ be 
least such that $M^\Tt_\alpha|\lh(E^\Tt_\alpha)\sats$``There is a Woodin'', and  
$\delta=\delta^{M^\Tt_\alpha|\lh(E^\Tt_\alpha)}$. Then there is $Q\ins M^\Tt_\alpha$ which is a 
$\delta$-sound Q-structure for $\delta$, and letting $k$ be least such that 
$\rho_{k+1}^Q\leq\delta$, then $\Vv\rest[\alpha,\lh(\Vv))$ is an above-$\delta$, $k$-maximal 
pseudo-iteration tree on $Q$.
\end{lem}

\subsection{Silver mice}\label{sec:silver_mice}

Let $P,M$ be active type 1 premice. Say that $F^P$ is a \emph{submeasure} of $F^M$
iff $\crit(F^P)=\crit(F^M)$ and $F^P\rest\nu(F^P)\sub F^M\rest\nu(F^P)$.
Write $P\pins^{\text{almost}}M$ iff either $P\pins M$, or letting $\kappa=\crit(F^P)$
and $\nu=(\kappa^+)^P$, then $M|\nu$ is active and $P\pins\Ult(M|\nu,F^{M|\nu})$.
We write $P\ins^{\mathrm{almost}} M$
iff either $P\pins^{\mathrm{almost}}M$ or $P=M$.

\newcommand{\Ccc}{\mathscr{C}}
\newcommand{\Nnn}{\mathscr{N}}

Let $M$ be an active $\infty$-iterable mouse. Then $\Nnn^M$ denotes the proper class premouse given 
by iterating away the active extender of $M$, and $\Ccc^M$ denotes the class of critical points 
arising in the iteration.

\begin{dfn}
A \emph{Silver mouse} is a sound mouse $M$ such that $\rho_1^M=\om$, $p_1^M=\emptyset$,
$M$ is active type 1, and there is no type 1 $P\pins^{\text{almost}}M$ such that $F^P$ is a 
submeasure of $F^M$.
\end{dfn}

\begin{tm}Let $M$ be an $\infty$-iterable active mouse. Then
\begin{equation}\label{eqn:C_gens} \Nnn^M=\Hull^{\Nnn^M}(\Ccc^M) \end{equation}
iff $M$ is a Silver mouse.\end{tm}

\begin{proof} Suppose line (\ref{eqn:C_gens}) holds. We show that $M$ is a Silver mouse. Let 
$\kappa=\crit(F^M)$.

If $M$ is not type 1, then 
let $P$ be the type 1 proper segment of $M$. Note that $\Nnn^P\neq\Nnn^M$ because 
$(\kappa^{++})^{\Nnn^P}<(\kappa^{++})^{\Nnn^M}$. Let $\pi:\Nnn^P\to\Nnn^M$ be the elementary 
embedding that is just the limit of the copying maps given by copying the length $\infty$ iteration 
of $F^P$ to the length $\infty$ iteration of $F^M$. Note then that $\pi``\Ccc^P=\Ccc^M$. By line 
(\ref{eqn:C_gens}), $\pi$ is a surjection, so $\Nnn^P=\Nnn^M$, contradiction.

So $M$ is type 1. 
Suppose now that $P$ is an $\infty$-iterable active mouse such that $\kappa=\crit(F^P)=\crit(F^M)$ 
and $P|(\kappa^+)^P=M||(\kappa^+)^P$ and $F^P\rest(\kappa^+)^P=F^M\rest(\kappa^+)^P$.
Then we can again define $\pi:\Nnn^P\to\Nnn^M$ such that $\pi``\Ccc^P=\Ccc^M$ but 
$\rg(\pi)\neq\Ccc^M$, the latter because
\[ (\kappa^+)^{\Nnn^P}=(\kappa^+)^P<(\kappa^+)^M=(\kappa^+)^{\Nnn^M}, \]
again a contradiction. So there is no such $P$, from which it follows that there is no 
$P\pins^{\text{almost}}M$ such that $F^P$ is a submeasure of $F^M$.

Finally if either $\rho_1^M>\om$ or 
$p_1^M\neq 0$ or $M$ is not sound then we can find a premouse $H\in M$ and a $\Sigma_1$-elementary 
$j:H\to M$, with $\rg(j)\neq M$ and $H^\passive\neq M^\passive$. So $\Nnn^H\neq\Nnn^M$.
But we get an elementary $\pi:\Nnn^H\to\Nnn^M$ with $\pi``\Ccc^H=\Ccc^M$, again a contradiction.

So $M$ is a Silver mouse.

Now suppose that $M$ is a Silver mouse, and we show that line (\ref{eqn:C_gens}) holds.
Write $N=\Nnn^M$, $C=\Ccc^M$, $H=\Hull^N(C)$ 
(the uncollapsed hull) and let $\left<\kappa_\alpha\right>_{\alpha\in\OR}$ enumerate $C$ in 
increasing order. It is routine to show:
\begin{clm}
 $C$ is a club of indiscernibles for $N$, in 
the model theoretic sense.
\end{clm}

The main work is the following claim:

\begin{clm}
$H\inter(\kappa_\alpha^+)^N$ is cofinal in $(\kappa_\alpha^+)^N$ for every $\alpha$.
\end{clm}
\begin{proof}
Consider the case that $\alpha=0$; the only difference is otherwise is discussed at the end. Let 
$\lambda=\sup H\inter(\kappa^+)^M$ and suppose that $\lambda<(\kappa^+)^M$. We aim to show that 
there is $P\pins^{\text{almost}} M$ such that 
$F^P$ is a submeasure of $F^M$, with $(\kappa^+)^P=\lambda$.

Let $H'=\Hull^N(\kappa\un C)$. Then standard calculations show 
that $H'\inter(\kappa^+)^N=\lambda$.
Let $\bar{N}$ be the transitive collapse of $H'$. Let $\pi:\Nbar\to N$ be the uncollapse map.
So $\crit(\pi)=\lambda=(\kappa^+)^\Nbar$ and $\pi((\kappa^+)^\Nbar)=(\kappa^+)^N$ and 
$\Nbar|\lambda=N||\lambda$. Let $\mu$ be the normal measure derived from $F^M$
and $\mubar=\mu\inter\Nbar$. Let $\Qbar=\Nbar|\lambda$ and $Q=N|(\kappa^+)^N$. 
$\Ubar=\Ult(\Qbar,\mubar)$ and $U=\Ult(Q,\mu)$ and $j:\Ubar\to U$ the natural factor map; that is,
\[ j([f]^{\Qbar}_{\mubar})=[f]^Q_\mu.\]
By condensation we have $\Ubar||\lambda=\Qbar$.

We claim that $\mubar$ is weakly amenable to $\Nbar$, or equivalently, to $\Qbar$.
For let $\alpha<\lambda$. We want to see that $\mubar\inter\Nbar|\alpha\in\Nbar$. Let $f:\kappa\to 
\Nbar|\alpha=N|\alpha$ be the $N$-least surjection, so $f\in H$. Let 
$\left<\kappa_\beta\right>_{\beta<\OR}$ enumerate $\Ccc^N$. Let $t$ be a term and 
$\alphavec\in[\kappa]^{\om}$ and $n<\om$ be such that
\[ f=t^N(\alphavec,\kappa_0,\kappa_1,\ldots,\kappa_n); \]
we have $\kappa=\kappa_0$. Let
\[ f^*=t^N(\alphavec,\kappa_1,\kappa_2,\ldots,\kappa_{n+1}),\]
\[ X=\{\alpha<\kappa|\kappa\in f^*(\alpha)\}, \]
\[ \mu'=f``X. \]
Then $\mu'=\mu\inter N|\alpha=\mubar\inter\Nbar|\alpha$ and $\mu'\in H$, so $\mu'\in\Nbar$, as 
required.

So by weak amenability,
\[ \Ubar|(\kappa^+)^\Ubar=\Qbar=\Nbar|(\kappa^+)^\Nbar.\]
Let
\[ P=(\Ubar|(\kappa^{++})^\Ubar),\tilde{F}) \]
where $\tilde{F}$ is the amenable code for the extender of $\mubar$.
Note that $P$ is a premouse, and letting $k=j\rest P$, we have
\[ k:P\to M \]
is $\Sigma_0$-elementary and $\crit(k)=(\kappa^+)^P$. (This is a standard argument; by the 
commutativity $j\com i^{\Qbar}_{\mubar}=i^Q_\mu$, we get that $j,k$ move fragments of $F^P$ (in the 
code $\tilde{F}$) to fragments of $F^M$.) Moreover, $F^P$ is a submeasure of $F^M$.

But then $P\ins^{\text{almost}}M$ by Theorem \ref{thm:ISC_for_submeasures}, which suffices.

So $M$ is not a Silver mouse, contradiction. This completes the proof of the claim for $\alpha=0$.

When $\alpha>0$, run the proof above with the $\alpha^\nth$ iterate $M_\alpha$ of $M$ replacing $M$,
producing a premouse $P$ such that $F^P$ is a submeasure of $F^{M_\alpha}$. If we have that 
$M_\alpha$ is itself fully iterable, then we could argue as before. But there is another method. 
Note that $\mubar,P\in M_\alpha$. There is a $\Sigma_1$ sentence $\varphi$ in the language of type 
1 premice asserting
''there is a premouse $S$ which is a submeasure of me and letting 
$\mu=\crit(F^S)=\crit(F^{L[\es]})$, we have $S|(\mu^+)^S=L[\es]||(\mu^+)^S$''. Since 
$M_\alpha\sats\varphi$, therefore $M\sats\varphi$, but now we can run the last part of the 
argument again to show that $S\pins^{\text{almost}}M$, a contradiction.
\end{proof}
Given an active premouse $R$ let $F^R_\alpha=F\inter(R|\alpha)$ for $\alpha<\OR^R$,
where $F$ is the amenable predicate for $F^R$.

\begin{clm}
Let $\alpha\in\OR$ and $\lambda=\OR^{M_\alpha}=(\kappa_\alpha^{++})^{\Nnn^M}$.
Then there are cofinally many $\beta<\lambda$ such that $F^{M_\alpha}_\beta\in H$.
\end{clm}
\begin{proof}
 Let $\xi<\nu=(\kappa_\alpha^+)^{M_\alpha}=(\kappa_\alpha^+)^{\Nnn^M}$ be such that $\xi\in H$ and 
$\xi$ is large enough that
\[ E=F^M\rest(M_\alpha|\xi)\cross[\nu]^{<\om}\notin M_\alpha|\beta.\]
We have $E\in H$, much as in our proof that $\mubar$ was weakly amenable to $\Nbar$ earlier.
This gives the claim.
\end{proof}

\begin{clm}
Let $\alpha\in\OR$. Then $\univ{M_\alpha}\sub H$.
\end{clm}
\begin{proof}Let $x\in M_\alpha$. Then because $M$ is sound with $\rho_1^M=\om$ and 
$p_1^M=\emptyset$, there is a $\Sigma_1$ formula $\varphi$ and there is 
$\kappavec\in(C\inter\kappa_\alpha)^{<\om}$ such that such that $x$ is the unique $y\in 
M_\alpha$ such that $M\sats\varphi(y,\kappavec)$. It easily follows that $x$ is definable over 
$\Nnn^M|(\kappa_\alpha^{++})^{\Nnn^M}$ from the parameter $F^{M_\alpha}_\beta$, for sufficiently 
large $\beta<\OR(M_\alpha)$. Therefore $x\in H$, as required.
\end{proof}
\end{proof}

\section{The $L[\es,x]$-construction $\CC$}\label{sec:CC}
In this section we describe the kind of $L[\es,x]$-construction we will consider.
 The construction $\CC$ will be produced inside a premouse $P$ as background universe.
There are three special features of the construction. Two of these 
are directly significant in our scale construction,
or more specifically, in the eventual proof that we do in fact construct a scale. These are (i) 
the requirement that the strengths of background extenders are strictly increasing with stages of 
the construction, and (ii) the locality of the construction in the vicinity of 
measurable cardinals. The most obvious consequence of property (i) is that if at stage $\alpha$,
we define $N_{\alpha+1}=(N_\alpha,E)$ with $E\neq\emptyset$, where $N_\alpha$ is the premouse produced by stage $\alpha$ 
of $\CC$, backgrounding $E$ with an extender $F$, then $F$ coheres 
$\CC\rest(\alpha+1)$. When selecting such an $F$ we will always take $F$ to be of minimal index, 
and 
so in fact, $N_\alpha^{\Ult(V,F)}=N_\alpha$, but at stage $\alpha+1$ in $\Ult(V,F)$ 
we do not add an extender, setting instead $N_{\alpha+1}=\J(N_\alpha)$. Thus, we 
have a coherence property analogous to that of the active extenders of premice. While property (i) 
slows down the construction somewhat, property (ii) accelerates it in the vicinity of measurable 
cardinals. Both properties are important in our proof of (a) lower semi-continuity, 
and property (i) is also important in the proof that (b) the norms we define are indeed norms (both 
of these proofs involve a kind of comparison, though (a) is less overtly so).
In property (ii), the construction is made more local by taking some background extenders 
from ultrapowers of $V$; if our mice had Jensen indexing then these extenders would just be certain 
partial extenders from the extender sequence. If one goes further and allows more such partial 
extenders, then our proof of (b) runs into problems. Property (ii) is 
formulated as it is because it gives sufficient locality for the proof of (a), but can still be 
handled in our proof of (b). We also remark that property (ii) is a simplification of
what is used in \cite{premouse_inheriting}.
Using constructions which omit option (ii), one can define semiscales, by a simplification of the 
arguments here, but the author does not know whether lower semi-continuity holds for the resulting semiscale. (That was indeed our original construction, and the one presented at
the 4th M\"unster conference on
inner model theory, 2017.
Of course a semiscale does in turn yield a scale, but it is not clear that the resulting scale has optimal complexity.)

If one could produce scales from more local background 
constructions, then one might be able to optimally define scales in contexts where the construction 
we give does not. However, the third special feature of the construction does ensure that
we produce optimally definable scales for pointclasses appropriately corresponding to mouse 
operators $x\mapsto M_x$ in which $M_x$ is \emph{minimal above a Woodin cardinal},
for some appropriate notion of minimality (these statements are made precise in \ref{?}).
That is, our construction has two ``modes'',
 one which applies below local Woodin cardinals and one above local Woodins. The parts below local Woodins are 
formed as already described above. The parts above local Woodins will 
be P-construction. Note that one can think of P-construction as a special case of 
$L[\es,x]$-construction, in which we allow \emph{all} extenders from the extender sequence of the 
background universe as background extenders (and in which coring is automatically
trivial because the models produced are fully sound, until reaching a Q-structure for the relevant Woodin cardinal).
(The construction below local Woodins has non-trivial coring as usual.)
Because we use P-constructions above local Woodins, our construction always inherits
standard large cardinal properties from the background universe above such local Woodins,
and in various situations this ensures the optimality of
definability of the scales we construct. Beyond this it is not particularly integral to the scale 
construction, however, and one could omit it, though possibly at the cost of increasing the 
complexity of the scale. Its presence also does not have a strong impact on the overall argument;
in fact nearly all of the ideas show up already in the case of $\Scale(\Pi^1_3)$,
for which the mice are $1$-small, and so in this case no extenders are ever added above a local Woodin 
anyway.

\begin{dfn}\index{$M$-cardinal-large}
 Let $P$ be a premouse. We say that $P$ is \emph{large} iff $P$ satisfies ``there is a Woodin cardinal'', and \emph{small} otherwise. Let $\lambda\leq\OR^P$ be a limit.
 We say that $\lambda$ is \emph{$P$-cardinal-large} iff there is $\delta<\lambda$
 such that $\delta$ is a cardinal of $P$ and $P|\lambda\sats$``$\delta$ is Woodin''.
Say that $P|\lambda$ is $P$-cardinal-large iff $\lambda$ is.
\end{dfn}

\begin{dfn}
 Let $P$ be a premouse
 and $\kappa$ be such that
 there is a $P$-total extender
 $E\in\es_+^P$ with $\crit(E)=\kappa$. Then $D^P_{\kappa0}$ denotes the least such $E$.
\end{dfn}

We  now define the $L[\es,x]$-construction $\CC$ which will be the basis for our scale 
construction:

\begin{dfn}\index{$\CC$}\index{Q-local $L[\es]$-construction}\index{$N_\alpha$}\index{$F_\alpha$}\index{$t_\alpha$}\label{dfn:CC}
 Let $P$ be a premouse, $x\in\RR^P$ and $\lambda\leq\OR^P$ be a limit. We define
 \[ \CC=\CC_x^P=\left<N_\alpha,F_\alpha,t_\alpha\right>_{\alpha\in\Lim\inter(0,\lambda]}, \]
consisting of  $x$-premice $N_\alpha$, integers $t_\alpha\in[0,3]$, and extenders $F_\alpha$.
We set $F_\alpha=\emptyset$ and 
$t_\alpha=0$ except where we 
explicitly say otherwise.
We define $\CC\rest(\alpha+1)$ by induction on limits $\alpha$.
We will 
state certain assumptions throughout the construction,
and if any of those assumptions fail at any stage, then we stop the construction.
Note that $P$ need not model $\ZFC$. In order to define $N_\alpha$, we first make the assumptions
\begin{enumerate}[label=\tu{(}A\arabic*\tu{)}]
 \item\label{conass:aleph_alpha_existence} Either (i) $\aleph_\beta^P<\OR^P$ for every $\beta<\alpha$, or
(ii) $\alpha$ is $P$-cardinal large, and
 \item\label{conass:Q-structures} For each limit $\beta<\alpha$ such that
 $\beta$ is $P$-cardinal large, $P|\beta$ is a Q-structure for $\delta^{P|\beta}$ iff
 $N_\beta$ is a Q-structure for $\delta^{P|\beta}$.
\end{enumerate}

If $\kappa$ is $P$-measurable via some extender in $\es^P$ and $\kappa\leq\alpha$ and $P|\kappa$ is small then
letting $D=D^P_{\kappa 0}$, we write
\index{$\chi^P_{x\kappa}$}
\[ \chi^P_{x\kappa}=(\kappa^+)^{i^P_D(N_\kappa)}. \]
We say that $\alpha$ is \emph{$(P,x)$-near-a-measurable}\index{near-a-measurable}
iff $\kappa=\card^P(\alpha)$ is measurable in $P$ via $\es^P$ and
$\kappa<\alpha<\chi^P_{x\kappa}$. 

\begin{case} $\alpha$ is not $P$-cardinal-large and not $(P,x)$-near-a-measurable.
\begin{scase} $\alpha=\om$

Set $N_\om=(V_\om,x)$.
\end{scase}

\begin{scase} $\alpha$ is a limit of limits.

Set $N_\alpha=\liminf_{\beta<\alpha}N_\beta$.
\end{scase}

\begin{scase} $\alpha=\beta+\om$.

\begin{sscase} $N_\beta$ is passive and there is $F$ such that there is $E$ such that:
\begin{enumerate}[label=--]
\item $F\in\es^P$ and $\nu_F=\aleph_\beta^P$ (so $F$ is $P$-total),
\item $(N_\beta,E)$ is an active premouse (so $E\neq\emptyset$), and
\item $E\rest\nu_E\sub F$.
\end{enumerate}
Let $F$ witness this with $\lh(F)$ minimal, and $E$ be induced by $F$. Then 
\[ N_{\beta+\om}=(N_\beta,E)\text{ and }F_{\beta+\om}=F\text{ and }t_{\beta+\om}=1.\]
\end{sscase}

\begin{sscase}Otherwise.

Then we assume
\begin{enumerate}[label=\tu{(}A3\tu{)}]
\item\label{conass:segs_good}  $N_\beta$ is $\om$-good and $\core_\om(N_\beta)$ satisfies condensation,
\footnote{These notions were defined in \S\ref{sec:notation}.}
 \end{enumerate}
and set $N_{\beta+\om}=\J(\core_\om(N_\beta))$.
\end{sscase}
\end{scase}
\end{case}
\begin{case}$\alpha$ is $(P,x)$-near-a-measurable but not $P$-cardinal-large.

Let $F=D^P_{\kappa 0}\in\es^P$. Set
\[ N_\alpha=i_F^P(N_\kappa)|\alpha\text{ and }F_\alpha=F\text{ and }t_\alpha=2.\]
\end{case}

\begin{case} $\alpha$ is $P$-cardinal-large.

Let $\delta=\delta^{P|\alpha}$.
Then $\CC\rest(\delta,\alpha]$ is the P-construction of $P|\alpha$ above $N_\delta$,
where we assume that
\begin{enumerate}[label=\tu{(}A4\tu{)}]
 \item\label{conass:N_is_P-con} $\OR^{N_\delta}=\delta$ and the P-construction as above works,
 giving $N_\alpha\sim_\delta P|\alpha$.
\end{enumerate}
If $P|\alpha$ is active then we set $F_\alpha=F^{N|\alpha}$ and $t_\alpha=3$.
\end{case}

This completes the definition of $\CC$.

We write $N^P_{x\alpha}=N^\CC_\alpha=N_\alpha\text{ and }F^P_{x\alpha}=F^\CC_\alpha=F_\alpha\text{ 
and 
}t^P_{x\alpha}=t^\CC_\alpha=t_\alpha$.

We say that $N^P_{x\alpha}$ \emph{exists} iff it is defined (that is, all assumptions required to 
define $N^P_{x\alpha'}$ for all limits $\alpha'\leq\alpha$, hold).\index{$N_{\alpha}$ exists}

Suppose $N^P_{x\alpha}$ exists. We write $N^P_{x\alpha 0}=N^P_{x\alpha}$,
and if $N^P_{x\alpha}$ is $n$-good and $k\leq\min(\om,n+1)$,
write
\[ N^P_{x\alpha k}=\core_k(N^P_{x\alpha})^\unsq,\]
and if also $0\leq j\leq k$
then\[\tau^P_{x\alpha kj}:(N^P_{x\alpha k})^\sq\to(N^P_{x\alpha j})^\sq \]
denotes the core embedding. We say that $N^P_{x\alpha 0}$ \emph{exists},\index{$N_{\alpha n}$ exists}
say that $N^P_{x\alpha,n+1}$ \emph{exists} iff $N^P_{x\alpha}$ is $n$-good,
and say that $N^P_{x\alpha\om}$ \emph{exists} iff $N^P_{x\alpha}$ is $\om$-good.
\end{dfn}

\begin{rem}\label{rem:omit_sub_super}
When the real $x$ or background universe $P$ or construction $\CC$
are clear from context, then we often drop the subscript $x$ and/or superscripts $P$ and $\CC$ from 
the notation above.

We will begin by establishing various facts about the construction, and later prove that the models $N_\alpha$ are iterable,
assuming that the background $P$ is. The first lemma is a simple observation on the nature of the construction close to measurables.
\end{rem}
\begin{lem}\label{lem:measurable_stack}
Suppose $t_\alpha=2$ and $\kappa=\crit(F_\alpha)$,
so $F_\alpha=D=D^P_{\kappa0}$.
For  $G\in\es^P$ any $P$-total extender with $\crit(G)=\kappa$,
let $R'_G=i^P_G(N_\kappa)$ and $\xi_G=\xi=\kappa^{+R'_G}$ and $R_G=R'_G|\xi$.
Then
\[ R_G=R_D\sats\ZF^-+\text{``}\kappa\text{ is the largest cardinal''}.\]
Therefore $R_G|\alpha=N_\alpha$ and if $N_\xi$ is defined then $R_G=N_\xi$.
\end{lem}

\begin{proof}
Let $S$ be the stack of all sound premice $A\in P$ such that $N_\kappa\pins A$,
$\rho_\om^A=\kappa$, and $A$ satisfies condensation. Then $R_G=S=R_D$. For all segments of $N_\kappa$ satisfy 
condensation, so $R'_G$ and $R_G$ are likewise, so $R_G\ins S$. But if $A\pins S$ and $\rho_\om^A=\kappa$
then $i^P_G(A)$ satisfies condensation, which implies $A\pins i^P_G(A)$, so $A\pins R_G$, as required.
Hence, $R_G=S=R_D$.

Clearly  $R_G\sats\ZF^-$+``$
\kappa$ is the largest cardinal'',
since $R_G=R'_G|\xi^{+R'_G}$.
\end{proof}

The following lemma will not be needed until we reach the proof of lower semicontinuity for mouse scales,
but we include it now as it fits well with the previous lemma, of which it is an easy corollary:

 \begin{lem}\label{lem:iterate_up_to_N_x,kappa}
  Suppose 
  $\kappa$ is measurable in $P$,
  as witnessed by $\es^P$,
  and $\CC_x\rest\chi$ is defined,
  where $\chi=\chi_{x\kappa}^P$.
  Let $M\in P|\kappa$ be an
  $x$-premouse.
  Let $n\leq\om$ and let $\Tt\in P$ an $n$-maximal
  tree on $M$ of length $\kappa+1$.
  Suppose that $M^\Tt_\kappa|\kappa=N_{x\kappa}^P$.
  Then $M^\Tt_\kappa|\kappa^{+M^\Tt_\kappa}=N_{x\chi}^P$.
 \end{lem}
\begin{proof}
Let $j:P\to\Ult(P,F^P_{\kappa 0})$
be the ultrapower map.
Note that by Lemma \ref{lem:measurable_stack},
and since $M^\Tt_\kappa|\kappa=N_{x\kappa}^P$,
\[ j(M^\Tt_\kappa)|\chi=N_{x\chi}^P\text{ and }
 \chi=\kappa^{+j(M^\Tt_\kappa)}.\]
But $j(M^\Tt_\kappa)=M^{j(\Tt)}_{j(\kappa)}$, and $j(\Tt)\rest(\kappa+1)=\Tt$,
Moreover, $[0,j(\kappa))^{j(\Tt)}\cap\dropset^{j(\Tt)}\sub\kappa$,
and $i^{j(\Tt)}_{\kappa,j(\kappa)}\sub j$,
so $\crit(i^{j(\Tt)}_{\kappa,j(\kappa)})=\kappa$. Therefore
\[ M^\Tt_{\kappa}|\kappa^{+M^\Tt_\kappa}=M^{j(\Tt)}_{j(\kappa)}|\kappa^{+M^{j(\Tt)}_{j(\kappa)}}=j(M^\Tt_\kappa)|\kappa^{+j(M^\Tt_\kappa)}=N_{x\chi}^P, \]
as desired.
\end{proof}

Note that in the following lemma,
which follows immediately from the definitions,
$\gamma+\om$ is defined independently of $x$:

\begin{lem}\label{lem:least_t=2}Let $Y$ be a premouse,
$\kappa$ be measurable in $Y$, with $Y|\kappa$ small (hence $\kappa$ is not Woodin in $Y$).  Let $Y|\gamma$ be the Q-structure for $Y|\kappa$. 
Let $x\in\RR^Y$ and suppose
that $\CC\rest(\gamma+\om+1)$ is defined
\tu{(}that is, $\CC^Y_x\rest(\gamma+\om+1)$; see \ref{rem:omit_sub_super}\tu{)}.
Then $\gamma+\om$ is the least $\alpha$ such that $t_{\alpha x}^Y=2$
and $\kappa=\card^Y(\alpha)$.
\end{lem}

The next lemma follows routinely from the universality of the standard parameter
(cf.~\cite[where?]{fsit}, which generalizes easily to the present construction):

\begin{lem}\label{lem:projections_to_cardinals}
 Let $Y$ be a premouse, $x\in\RR^Y$, $\alpha$ a limit such that $\CC\rest\alpha$ is defined.
Then:
\begin{enumerate}
 \item If $\beta<\gamma<\alpha$ and $\rho=\rho_\om^{N_\beta}\leq\rho_\om^{N_\xi}$ 
for all $\xi\in(\beta,\gamma)$ then either:
\begin{enumerate}[label=--]
 \item $\rho<\OR^{N_\beta}$ and
\begin{enumerate}[label=--]
\item 
$\core_\om(N_\beta)||(\rho^+)^{\core_\om(N_\beta)}=N_\beta||(\rho^+)^{N_\beta}=
N_\gamma||(\rho^+)^{N_\beta}$,
\item $\core_\om(N_\beta)\pins N_\gamma$,
\item $\rho$ is a cardinal of $N_\gamma$,
\end{enumerate}
or
\item $\rho=\OR^{N_\beta}$ and $\gamma=\beta+\om$ and 
$N_{\beta+\om}$ is active \tu{(}so $N_{\beta+\om}^\passive=N_\beta$\tu{)}, or
\item $\rho=\OR^{N_\beta}$ and $\core_\om(N_\beta)=N_\beta\pins_\card N_\gamma$.
\end{enumerate}
\item Let $\beta\leq\xi<\gamma<\alpha$ be such that
$\rho=\rho_\om^{N_\beta}=\rho_\om^{N_\xi}$
and $N_\beta||\rho=N_\xi||\rho$ and $(\rho^+)^{N_\beta}=(\rho^+)^{N_\xi}$.
Then $\beta=\xi$.
\item Let $\gamma<\alpha$. Then $\rho\in N_\gamma\inter\card^{N_\gamma}\cut\om$ iff there
is $\beta<\gamma$ such that
\begin{enumerate}[label=--]
 \item $\rho=\rho_\om^{N_\beta}\leq\rho_\om^{N_\xi}$ for all $\xi\in(\beta,\gamma)$, and
 \item if $\beta+\om=\gamma$ then $N_\gamma$ is passive \tu{(}hence $N_\gamma=\J(\core_\om(N_\beta))$\tu{)}.
\end{enumerate}
\item\label{item:prod_stage} Let $\gamma<\alpha$ and $\rho\in 
N_\gamma\inter\card^{N_\gamma}\cut\om$. Let $A\pins N_\gamma$ be such that $\rho_\om^A=\rho$.
Then there is a unique $\beta<\gamma$ such that $\core_\om(N_\beta)=A$.
Moreover, $\rho\leq\rho_\om^{N_\xi}$ for all $\xi\in(\beta,\gamma)$.

\item\label{item:N_gamma_lgcd} Let $\gamma<\alpha$ be a limit of limits.
Suppose there is $\beta<\gamma$ such that 
\[ \rho=\rho_\om^{N_\beta}=\liminf_{\beta'<\gamma}\rho_\om^{N_\beta'}.\]
Then $\rho=\lgcd(N_\gamma)$ and letting
$B=\{\beta'\in[\beta,\gamma)\ |\ \rho_\om^{N_{\beta'}}=\rho\}$, we have
\begin{enumerate}[label=--]
\item $B$ is cofinal in $\gamma$,
\item if $\beta_0,\beta_1\in B$ and $\beta_0<\beta_1$ then 
$\core_\om(N_{\beta_0})\pins\core_\om(N_{\beta_1})$, and
\item $N_\gamma=\stack_{\beta'\in B}\core_\om(N_{\beta'})$.
\end{enumerate}
\item Let $\gamma<\alpha$ be a limit of limits. If there is no $\beta<\gamma$ as in
part \ref{item:N_gamma_lgcd} then $N_\gamma$ has no largest cardinal.
\item\label{item:limit_of_limits} If $\alpha$ is a limit of limits then $\CC\rest(\alpha+1)$ is defined.
\end{enumerate}
\end{lem}

\begin{lem}\label{lem:lifting_t=2_proj_meas}
 Let $Y$ be a premouse,  $x\in\RR^Y$,
 $\alpha$ a limit such that $\CC\rest(\alpha+1)$
 is defined and $t^Y_\alpha\neq 2$,
 and let $\beta<\alpha$ be such that
 $t^Y_\beta=2$ and
 $R=N^Y_\beta\pins N=N^Y_\alpha$ and $\rho_\om^R$ is a cardinal in $N$.
 Then:
 \begin{enumerate}[label=(\alph*)]\item\label{item:card_seg_t=2}  $\kappa=\rho_\om^R$ is measurable in $Y$, as witnessed by $\es^Y$ (so $\kappa=\card^Y(\beta)$ is the measurable associated to $N^Y_\beta$),
 \item\label{item:no_project_leq_meas_while_seg_remains} letting $\chi=\chi_{x\kappa}^Y$, there is no $\xi\in[\chi,\alpha)$ such that $\rho_\om^{N_\xi}\leq\kappa$, so $N_\chi=N||\chi$, \item\label{item:card_seg_t=2_type_3} if $N$ is type 3 then $\OR^R<\nu^N$.
 \end{enumerate}
\end{lem}
\begin{proof}
 Part \ref{item:card_seg_t=2}: Assume not and let $\kappa=\card^Y(\beta)=\card^Y(\rho_\om^R)<\rho_\om^R$,
 so $\kappa$ is measurable in $Y$, as witnessed by $\es^Y$, and letting $F=F_{\kappa 0}^P$ and $R^*=i^P_F(R)$
 and $S=R^*|(\kappa^+)^{R^*}$, then $R\pins S$ but $N\npins S$.
 Let $R'\pins S$ be least such that $R\ins R'$ and $\rho_\om^{R'}=\kappa$.
 Then because $\rho_\om^R$ is an $N$-cardinal, $N\pins R'$.
 But then $t_\alpha^P=2$, contradiction.
 
 Part \ref{item:no_project_leq_meas_while_seg_remains}: 
 Suppose otherwise and let $\xi$ be least such.
 By Lemma \ref{lem:projections_to_cardinals}, then $\rho_\om^{N_\xi}=\kappa$. But since $\xi<\alpha$,
 $N_\xi\in Y$,
 $\core_\om(N_\xi)\in Y$ and $\core_\om(N_\xi)$ is sound and satisfies condensation,
 and so by the proof of Lemma \ref{lem:measurable_stack},
 $\core_\om(N_\xi)\pins N_\chi$,
 a contradiction.
 
 Part \ref{item:card_seg_t=2_type_3}: By part \ref{item:no_project_leq_meas_while_seg_remains},
 if $\chi<\OR^N$ then $\chi=\kappa^{+N}\leq\nu(F^N)$, so $\kappa=\rho_\om^R<\nu(F^N)$,
 so $\OR^R<\nu(F^N)$. So suppose
$\chi+\om=\alpha$, $t^Y_{\alpha}=1$ and
 and $N^{\passive}=N_\chi$.
 It follows that $N\in Y$.
 Since $N$ is type 3, therefore $i_F^Y(N)$ is a type 3 premouse, but note that by the ISC,
 $N\pins i_F^Y(N)$, so $N$ is also sound and satisfies condensation, so again by the proof of Lemma \ref{lem:measurable_stack}, $N\pins N_\chi$,
 a contradiction.
\end{proof}

\begin{dfn}\label{dfn:gamma^Y_kappa}
  Let $Y$ be a premouse.
 Let $\delta<\OR^Y$ be such that $\delta$ is not Woodin in $Y$.
 Then $\gamma^Y_\delta$ denotes
 the $\gamma$ such that $Y|\gamma$ is a Q-structure for $Y|\delta$.
\end{dfn}

\begin{dfn}\label{dfn:xi^Y_kappa}
Let $Y$ be a premouse,
 $x\in\RR^Y$, and
 $\kappa<\alpha<\OR^Y$ be such that $\CC^Y_x\rest(\alpha+1)$ is defined
 and $t^Y_{\alpha x}=2$ and $\card^Y(\alpha)=\kappa$. 
Then
 $\xi^Y_{\kappa x}$ denotes the least $\xi\leq\alpha$ such that $t^Y_{\xi x}=2$ and $\card^Y(\xi)=\kappa$.
\end{dfn}
\begin{lem}\label{lem:meas_implies_least_t=2_exists}Let $Y$ be a premouse,
$x\in\RR^Y$ and $\kappa$ measurable in $Y$, via $\es^Y$, such that no $\delta<\kappa$ is Woodin  in $Y$ (hence neither is $\kappa$). Suppose that $\CC^Y\rest\kappa$ is defined (so $\CC^Y\rest(\kappa+1)$ is also defined). Let $D=F^Y_{\kappa 0}$. Suppose that $\Ult(Y,D)$ is wellfounded.  Let $\gamma=\gamma^Y_\kappa$ and $\xi=\gamma+\om$.  Then:
\begin{enumerate}[label=--]
 \item $\CC^Y\rest(\xi+1)$ is defined,
 $t^Y_{\xi}=2$ and
  $\xi=\xi^Y_{\kappa}$,
  \item 
  $N^Y_\gamma$ is fully sound,
  is a Q-structure for $N^Y_\gamma|\kappa=N^Y_\kappa$,
  and 
  is given by the P-construction of $Y|\gamma$
  over $N^Y_\kappa$,
  \item  $N^Y_{\xi}=\J(N^Y_\gamma)$ is also fully sound,
 \item
letting $\chi=\kappa^{+i^Y_D(N)}$
where $N=N^Y_{\kappa}$,
then $\CC^Y\rest(\chi+1)$ is defined
and $t^Y_{\chi}=0$.
\end{enumerate}
\end{lem}
\begin{proof}
Let $Y'=\Ult(Y,D)$ and $j:Y\to Y'$ be the ultrapower map.
 Since $\CC^Y\rest(\kappa+1)$
 is defined, so is $\CC^{Y'}\rest (j(\kappa)+1)$. Since $Y|\gamma=Y'|\gamma$ is the Q-structure for $Y|\kappa$,
 it follows that $N^{Y'}_{\gamma}$ is the Q-structure for $N^{Y'}_\kappa$ (and is produced by P-construction over $N^{Y'}_\kappa$). But $N^{Y'}_{\kappa}=N^{Y'}_{j(\kappa)}|\kappa=N^Y_{\kappa}$, by elementarity.
 So $N^Y_{\gamma}$ is well-defined
 and $N^Y_{\gamma}=N^{Y'}_{\gamma}$.
 Also since $N^{Y'}_{j(\kappa)}|\kappa=N^{Y'}_{\kappa}=N^Y_{\kappa}$,
 it follows that $\rho_\om^{N^Y_{\gamma}}=\kappa$, so $N^Y_\gamma$ is fully sound.
 Likewise, $\J(N^Y_{\gamma})=N^{Y'}_{\gamma+\om}$ is fully sound. And also by reflection (with $j$), these structures satisfy condensation. Since $\J(N^Y_{\gamma})\pins j(N^Y_\kappa)$, therefore
 $N^{Y}_{\gamma+\om}=\J(N^Y_\gamma)$ is well-defined
 and $t^{Y}_{\gamma+\om}=2$
 and $\card^Y(\gamma+\om)=\kappa$.
 The rest now follows easily.
\end{proof}
The following lemma follows easily from the definition:

\begin{lem}\label{lem:con_locality}
 Let $Y$ be a premouse, $x\in\RR^Y$, $\alpha$ a limit such that $\aleph_\alpha^Y<\OR^Y$,
 $Y|\aleph_\alpha^Y$ is small and $\CC^Y\rest\alpha$ is defined.
 Let $Y'\ins^*Y$\footnote{See \S\ref{sec:notation}.} with $\aleph_\alpha^Y\leq\OR^{Y'}$.
Then $\CC^{Y'}\rest\alpha=\CC^Y\rest\alpha$,
and if $\CC^Y\rest(\alpha+1)$ is defined then $\CC^{Y'}\rest(\alpha+1)=\CC^Y\rest(\alpha+1)$.
\end{lem}

The following lemma is now an easy corollary of the previous two:
\begin{lem}\label{lem:ZFC_stage}
 Let $Y$ be a premouse, $x\in\RR^Y$, $\kappa$ a cardinal of $Y$ such that $Y|\kappa$ is
 small\footnote{Smallness is not actually necessary.}
 and $Y|\kappa\sats\ZFC$. Then $\CC^Y\rest\kappa=\CC^{Y|\kappa}$,
and if $\CC^Y\rest\kappa$ has length $\kappa$ then $N_\kappa^Y$ is a class of 
$Y|\kappa$, $\OR(N_\kappa^Y)=\kappa$ and $N_\kappa^Y\sats\ZFC$.
\end{lem}

\begin{lem}\label{lem:con_coherence}
 Let $Y$ be a premouse, $x\in\RR^Y$,  
$\CC^Y\rest(\alpha+\om+1)$ defined, $N^Y_{\alpha+\om}$ active \tu{(}so 
$t^Y_{\alpha+\om}=1$\tu{)}, $F=F^Y_{\alpha+\om}$ and $U=\Ult(Y,F)$ wellfounded. Then:
\begin{enumerate}
\item\label{item:coherence_of_CC} $\CC^U\rest(\alpha+1)=\CC^Y\rest(\alpha+1)$.
\item\label{item:no_ext} $F_{\alpha+\om}^U=\emptyset$, so 
$N_{\alpha+\om}^U=\J(N_{\alpha}^U)=\J(N_\alpha^Y)$.
\item\label{item:no_ZFC} $Y|\aleph_\alpha^Y\not\sats\ZFC$.
\item\label{item:no_meas} There is no $E\in\es^U$ such that
$\nu(N_{\alpha+\om}^Y)\leq\crit(E)\leq\aleph_\alpha^Y\leq\nu_E$.
\end{enumerate}
\end{lem}
\begin{proof}
We have $\nu_F=\aleph_\alpha^Y$, so $Y|\aleph_\alpha^Y=U|\aleph_\alpha^U$
and $U|\aleph_{\alpha+1}^U=Y||\lh(F)$. So
part \ref{item:coherence_of_CC} follows from \ref{lem:con_locality}. Therefore 
$N_{\alpha}^U=N_{\alpha}^Y$.
But if $F'=F_{\alpha+\om}^U\neq\emptyset$ then $\nu_{F'}=\aleph_\alpha^U=\aleph_\alpha^Y$, so by 
coherence, $F'\in\es^Y$
and $\lh(F')<\lh(F)$, but $F'$ backgrounds an active extender over $N^U_\alpha=N_{\alpha}^Y$,
contradicting the minimality of $\lh(F)$ (in the choice of 
$F=F_{\alpha}^Y$). This gives part \ref{item:no_ext}.

Because $N_{\alpha+\om}^Y$ is active,
$(N_{\alpha+\om}^Y)^\passive=N_\alpha^Y$ has a largest cardinal,
so by \ref{lem:ZFC_stage}, $Y|\aleph_\alpha^Y\not\sats\ZFC$, giving part \ref{item:no_ZFC}.
Now suppose there is $E$ as in part \ref{item:no_meas}, and let $E$ be such with minimal index.
Let $\kappa=\crit(E)$.
Since $Y|\aleph_\alpha^Y\not\sats\ZFC$,
\[\crit(F)<\nu(N_{\alpha+\om}^Y)\leq\kappa=\crit(E)<\aleph_\alpha^Y=\nu_E<\lh(E)<\aleph_{\alpha+1}
^U=\lh(F).\]
So $E$ is $Y$-total.
Let $G=i^Y_E(F)\rest\aleph_\alpha^Y$.
Note that $G\in\es^Y$ and $\nu_G=\aleph_\alpha^Y$ and $\lh(G)<\lh(E)<\lh(F)$.
But $G\rest\kappa=F\rest\kappa$ and $\nu(N_{\alpha+\om}^Y)\leq\kappa$,
so $G$ backgrounds $N_{\alpha+\om}^Y$. This contradicts the minimality of $\lh(F)$.
\end{proof}
\begin{rem} Lemma \ref{lem:con_coherence} assumes that $t^Y_{\alpha+\om}=1$. We can also consider the question of coherence of $\CC$ under $F_\alpha$ when $t_\alpha\in\{2,3\}$.

Suppose $t^Y_\alpha=3$
and $F=F^Y_\alpha\neq\emptyset$
and $U=\Ult_k(N,F)$ is wellfounded
where $N\ins Y$ and $k\leq\om$ is such that this makes sense.
Then easily $\CC^U\rest(\alpha+1)=\CC^Y\rest(\alpha+1)$ and $t^U_\alpha=3$ but $F^U_\alpha=\emptyset$.

But in contrast to this and to Lemma \ref{lem:con_coherence},
suppose $t^Y_\alpha=2$ and let $F=F^Y_\alpha$
and $\kappa=\crit(F)$, so $F=D_{\kappa0}^Y$, and suppose $U=\Ult(Y,F)$ is wellfounded.
We have $N^U_\kappa=N^Y_\kappa$.
Let $\alpha_0\geq\kappa$
be least such that  $Y|\alpha_0$ is the Q-structure for $\kappa$.
If $\alpha>\alpha_0+\om$ then
$F$ does not cohere $\CC\rest\alpha$,
as $t^U_\xi\neq 2=t^Y_\xi$
for each limit $\xi\in(\alpha_0,\alpha)$.
Moreover, we can have 
$\xi\in(\alpha_0,\alpha)$ such that $N^Y_\xi$ is active,
but $N^U_\xi=\J_\xi(N^U_{\alpha_0})$, because $\kappa$ is not 
measurable in $U$, and so the first $\eta>\alpha_0$ such that $N_\eta^U$ is 
active, has $\eta>\mu>\alpha_0$ where $\mu$ is measurable in $U$,
whereas
\[ \xi<(\kappa^+)^{i^Y_F(N_\kappa)}\leq(\kappa^+)^U=(\kappa^+)^Y.\]
\end{rem}

\begin{lem}\label{lem:limit_proj_across_meas}
 Let $Y$ be a premouse, $x\in\RR^Y$, $\CC\rest(\alpha+\om+1)$ exists,
 $F=F_{\alpha+\om}\neq\emptyset$ and $U=\Ult(Y,F)$ wellfounded.
Let $\kappa=\crit(F)<\alpha$ and $\widetilde{Y}=i^Y_F(Y|\kappa)$.
 
Then $N^{Y|\kappa}=N^Y_\kappa=N^Y_\alpha|\kappa=N^{\widetilde{Y}}|\kappa$.
\end{lem}
\begin{proof}
By \ref{lem:ZFC_stage}, $N^Y_\kappa=N^{Y|\kappa}$ and $\OR(N^{Y|\kappa})=\kappa$ and 
$N^{Y|\kappa}\sats\ZFC$ and $\CC^{Y|\kappa},N^{Y|\kappa}$ are definable from $x$ over $Y|\kappa$.
The same definitions over $\widetilde{Y}$ give $\CC^{\widetilde{Y}},N^{\widetilde{Y}}$.
By elementarity then, $N^{Y|\kappa}=N^{\widetilde{Y}}|\kappa$,
and $\kappa$ is a limit cardinal of $N^{\widetilde{Y}}$.
So by \ref{lem:projections_to_cardinals}, $\rho_\om(N^{\widetilde{Y}}_\xi)\geq\kappa$
for all $\xi\in[\kappa,\OR^{\widetilde{Y}})$.
But by \ref{lem:con_coherence} and \ref{lem:con_locality},
\[ \CC^Y\rest(\alpha+1)=\CC^U\rest(\alpha+1)=\CC^{\widetilde{Y}}\rest(\alpha+1),\]
and it follows that $N^Y_\alpha|\kappa=N^Y_\kappa$, as required.
\end{proof}
\begin{lem}\label{lem:meas_limits_of_con_require_no_proj_across}
 Let $Y$ be a premouse, $x\in\RR^Y$, $N=N^Y_\alpha$ defined, $t^Y_\alpha\neq 2$, $\delta<\OR^N$,
$N|\delta$ small, and 
$N\sats$``$\delta$ is a limit of measurables via $\es$''. Then
$\delta$ is a limit cardinal of $Y$,
$Y\sats$``$\delta$ is a limit of measurables via $\es$'' and 
$N|\delta=N^Y_\delta$.\footnote{The assumption that $t_\alpha\neq 2$ is necessary.
For otherwise we could have
$\mu<\delta<(\mu^+)^Y$ for some measurable $\mu$ of $Y$.}
\end{lem}
\begin{proof}
By \ref{lem:projections_to_cardinals}, $N|\delta=N_\beta$ for some $\beta<\alpha$,
and because $t_\alpha\neq 2$ and $\delta$ is a limit cardinal of $N$ and $N|\delta$ is small,
it follows that $t_\beta=0$. Now we claim that if $\kappa<\delta$ and $N\sats$``$\kappa$ is measurable via 
$\es$'',
then $Y\sats$``$\kappa$ is measurable via $\es$''. The proof of this here is just slightly 
different to usual, because our backgrounding condition is different.
Let $E\in\es^N$ be the order $0$ measure on $\kappa$.
Because $\rho_\om(N|\lh(E))=\kappa^{+N}$ is a cardinal of $N$,
there is $\xi<\beta$
such that $N|\lh(E)=\core_\om(N_\xi)$. But because $N|\lh(E)$ is type 1, it easily follows that 
$N_\xi$ is sound, so $N_\xi=N|\lh(E)$. We have $t_\xi\neq 3$ because $N|\delta$ and hence $N|\lh(E)$ are small. And $t_\xi\neq 2$ because otherwise, letting $\mu=\card^Y(\xi)$ be 
the corresponding measurable of $Y$, and $\zeta>\xi$ be least such that $t_\zeta\neq 2$,
so $\zeta<\beta$,
then $N_\xi\pins N_\zeta\sats$``$\mu$ is the largest cardinal'', but this implies 
that $\mu<\kappa^{+N}$, and this contradicts the fact that $N|\kappa^{+N}\pins N_\xi$.
So $t_\xi=1$, so $E$ is backgrounded, so $Y\sats$``$\kappa$ is measurable via $\es$''.

So suppose $N|\delta\neq 
N^Y_\delta$ and let $\gamma\in[\delta,\alpha)$ be least such that 
$\rho=\rho_\om^{N_\gamma}<\delta$.
Fix
$\kappa\in(\rho,\delta)$ such that $N\sats$``$\kappa$ is measurable via $\es$''.
Let $E\in\es^N$ be the order $0$ measure on $\kappa$.
Then as in the previous paragraph, there is $\xi$ such that $N|\lh(E)=N_\xi$
and $t_\xi=1$, and note that $\gamma<\xi$. But then $\rho<\kappa<\gamma<\xi$,
which contradicts \ref{lem:limit_proj_across_meas}.
\end{proof}

\begin{lem}\label{lem:extenders_restrictions_from_bkgd}
 Let $Y$ be a premouse, $x\in\RR^Y$, $\CC\rest(\alpha+1)$ is defined,
 $t_\alpha\in\{0,1\}$, and $E\in\es_+^{N_\alpha}$ such that $\nu_E$ is a cardinal of $N_\alpha$.
 Then there is a $Y$-total $E^*\in\es^Y$ such that $E\rest\nu_E\sub E^*$.
\end{lem}
\begin{proof} If $t_\alpha=1$ and $E=F^{N_\alpha}$
use $E^*=F_{\alpha}$ (recall that then $\aleph_{\alpha}^Y\leq\OR^Y$
and $\alpha=\beta+\om$ for some $\beta$, where $\nu(F_{\alpha})=\aleph_\beta^{Y}$,
so $F_{\alpha}\neq F^Y$). So note we may assume $t_\alpha=0$.
Let $N=N_\alpha$ and $\nu=\nu_E$. We have $\rho_\om^{N|\lh(E)}=\nu$.
So there is $\xi<\alpha$ such that $N|\lh(E)=\core_\om(N_\xi)$.
But because $E$ is type 1/3 and $\rho_\om^{N|\lh(E)}=\nu$, it follows that $N_\xi$ is sound and 
$N|\lh(E)=N_\xi$. (If $N_\xi$ were not sound that note that because $\rho_\om^{N_\xi}=\nu$,
$E$ must be a proper segment of $F^{N_\xi}$. So $E\in\es^{N_\xi}$ by the ISC.
 But $\pow(
\nu)\cap N_\xi=\pow(\nu)\cap\core_\om(N_\xi)$ by fine structural universality, contradiction.)
Now suppose $t_\xi=3$ and let $\delta=\delta^{Y|\xi}$. Because $t_\alpha=0$
and by Assumption \ref{conass:Q-structures} there is $\beta\in[\xi,\alpha)$
such that $N_\beta$ is a Q-structure for $\delta$, and so $\rho_\om^{N_\beta}\leq\delta$.
Therefore $\nu=\delta$. But $\delta$ is the least
Woodin of $Y|\xi$, which implies $\delta<\crit(E)$, contradiction. So $t_\xi\in\{1,2\}$.
If $t_\xi=1$ then $E^*=F^Y_\xi$ suffices.
If $t_\xi=2$ then because $\nu$ is a cardinal of $N_\alpha$ and $t_\alpha=0$,
we must have $\nu=\card^Y(\xi)$ is $Y$-measurable, and letting $D\in\es^Y$ be the order $0$ measure 
on $\nu$, then $N|\lh(E)\pins i^Y_D(N|\nu)$. Now in $U=\Ult(Y,D)$,
$\nu$ is not measurable, so arguing as above, $N|\lh(E)=N_{\xi'}^U$ for some $\xi'$,
so $E$ is backgrounded by some $E^*\in\es^U$. But then $E^*\rest\nu\in\es^Y$,
so $E^*\rest\nu$ suffices.
\end{proof}

\begin{dfn}
A \emph{generalized-pseudo-premouse}\index{generalized-pseudo-premouse} is a structure $M=(N,G)$ 
which satisfies the axioms for an active premouse with active extender $G$,
excluding the ISC, $N$ has largest cardinal $\delta$, and either:
\begin{enumerate}[label=--]
 \item $G\rest\delta\in\es^N$, or
 \item $\delta=\rho^{+N}$ where $\rho$ is an $N$-cardinal and $G\rest\rho\in\es^N$.
\end{enumerate}
Note that when $G\rest\delta\in\es^N$, $M$ is just a pseudo-premouse
as in \cite{fsit} (except that we also allow superstrong extenders on $\es$,
and all proper segments of $M$ are assumed to satisfy condensation).

We also need to consider bicephali.
Recall from \cite{fsit}, as generalized in \cite{extmax},
that a \emph{bicephalus}\index{bicephalus} is a structure $B=(N,F,G)$
 where both $(N,F)$ and $(N,G)$ are active premice.
 \footnote{We are actually generalizing both \cite{fsit} and \cite{extmax}
 here (except that we as usual demand condensation of all proper segments),
 as in those papers only short extender premice are considered,
which is not the case here. The restriction to short extender premice
is used significantly in the bicephalus comparison proof of \cite{extmax},
so we need to be somewhat careful in generalizing the results there.}
\footnote{In \cite{fsit}, the case that, for example, $(N,F)$ is type 2 but $(N,G)$ is type 3,
 is not considered, whereas this is considered in \cite{extmax}.
 We need the more general notion here, because when setting 
$N_{\alpha+\om}^\CC=(N_\alpha,E)$, we make our choice of $E$
just by minimizing $\lh(F_\alpha^\CC)$,
without first minimizing $\nu(E)$. We do not know whether adaptations of our arguments here might
go through if 
one minimized $\nu(E)$ first, but doing so would at least complicate things, so it seems 
smoother to just use the more general bicephalus arguments.}

\emph{Iteration trees}\index{iteration tree on bicephalus} on bicephali are defined basically as in \cite{extmax},
and on generalized-pseudo-premice as in \cite{fsit};
the one difference is that here our premice can have superstrong extenders,
whereas in those papers all premice are below superstrongs,
and thus one must incorporate the modifications as described in \cite[Remark 2.48]{operator_mice_v3} in the
obvious manner.\footnote{However, 
we only need consider the use of superstrong extenders $E^\Tt_\alpha$ in trees $\Tt$ on bicephali 
in which such a drop occurs that $M^{*\Tt}_{\alpha+1}$ is a premouse (not a bicephalus),
and in fact, for each $\beta>\alpha$, $M^{*\Tt}_{\beta+1}$ is a premouse. If we had limited 
ourselves to a more restricted version of \ref{thm:pseudo-premice}, then we could have made a 
similar restriction to the trees on generalized-pseudo-premice we consider, but such a restriction is not 
important or particularly helpful in this case.}

Fix a premouse $Y$ and $x\in\RR^Y$.

Let $(N,G)$ be a generalized-pseudo-premouse over $x$.
Say $(N,G)$ is a \emph{creature of $\CC^Y$}\index{creature of $\CC$ (gen-pseudo-pm)} iff for some $\alpha$,
$N^Y_{\alpha}$ exists and equals $N$, $t_\alpha^Y=0$,
$\aleph_{\alpha+\om}^Y\leq\OR^Y$, $G\in Y$ and there is 
$E\in\es^Y$ such 
that $\nu_E=\aleph_\alpha^Y$ and $G\rest\nu_G\sub E$.\footnote{
We just require that $\nu_E=\aleph_\alpha^Y$ and $\aleph_{\alpha+\om}^Y\leq\OR^Y$ for simplicity,
in that this agrees with the requirements in $\CC$.}

Let
$B=(N,F_0,F_1)$ be a bicephalus over $x$. Say $B$ is a \emph{creature of 
$\CC^Y$}\index{creature of $\CC$ (bicephalus)}
iff for some $\alpha$, $N^Y_\alpha$ exists and equals $N$,
$t_\alpha^Y=0$, $\aleph_{\alpha+\om}^Y\leq\OR^Y$, $F_0,F_1\in Y$ and there are 
$E_0,E_1\in\es^Y$ 
such that $\nu_{E_i}=\aleph_\alpha^Y$ and $F_i\rest\nu_{F_i}\sub E_i$ for each $i$.
\end{dfn}

The following is a corollary to arguments in \cite{fsit} and \cite[Theorem 5.3]{extmax}.
We defer the proof to \S\ref{sec:ISC_etc}; see \ref{rem:proof_bicephalus_comp}:
\begin{tm}\label{thm:bicephali} Let $B=(N,F,G)$ be a $(0,\om_1+1)$-iterable bicephalus such that 
$N$ is small.
Then $F=G$.
\end{tm}

The main ideas for the proof of the following fact are given in \cite[\S10]{fsit},
together with \cite{deconstructing} and a tweak in \cite{recon_con}.

\begin{fact}
 Every $(0,\om_1,\om_1+1)^*$-iterable pseudo-premouse is a premouse.
\end{fact}
We will give a proof of this fact in \S\ref{sec:ISC_etc} (see \ref{rem:proof_gen-pseudo-pm}), but with a somewhat 
different argument than that mentioned above,
and one which is slightly more general and uses only normal iterability:

\begin{tm}\label{thm:pseudo-premice}
Let $M$ be a $(0,\om_1+1)$-iterable generalized-pseudo-premouse. Then either
$M$ is a premouse or
\tu{[}$\nu(F^M)<\lgcd(M)=\delta$ and $(M|\delta,F^M\rest\delta)$ is a premouse\tu{]}.
\end{tm}

\begin{dfn}
 A premouse $Y$ is $(x,\alpha)$-iterability-good
iff $\CC\rest\alpha$ is defined
 and for all $\beta<\alpha$ we have:
 \begin{enumerate}[label=--]
  \item 
For every $G$ such that
$M=(N_\beta,G)$ is a generalized-pseudo-premouse and creature of $\CC$,
the conclusion of \ref{thm:pseudo-premice} holds of $M$;
 \item For all $F_0,F_1$ such that
 $(N_\beta,F_0,F_1)$ is a bicephalus and creature of $\CC$, we have $F_0=F_1$.\qedhere
 \end{enumerate}
\end{dfn}

\begin{lem}[Woodin exactness]\label{lem:Woodin_exactness}
 Let $Y$ be a premouse, $x\in\RR^Y$,
 $\delta=\aleph_\delta^Y<\OR^Y$,
 $Y|\delta$ small, and 
 let $Q\ins Y$ such that either $Q=Y|\delta$
 or $Q\sats$``$\delta$ is Woodin'' \tu{(}so if $Y|\delta\pins Q$
 then $Q$ is $Y$-cardinal large and $\delta=\delta^Q$\tu{)}.
Suppose that $Y$ is $(x,\delta)$-iterability-good.
Let $M=N^Y_\delta$. Suppose that $\OR^M=\delta$.
Let $\eta=\OR^Q$. Then:
\begin{enumerate}
 \item\label{item:M_small}$M$ is small.
 \item\label{item:P-con_defined} $\CC^Y\rest(\eta+1)$ is defined, so $\CC^Y\rest(\delta,\eta]$ is 
the P-construction of $Q$ over $M$,
 $N^Y_\eta\sim_\delta Q$,
 and if $Q$ is $\delta$-sound then so is $N^Y_\eta$,
\item\label{item:eta+1_iter_good} $Y$ is $(x,\eta+1)$-iterability good.
 \item\label{item:no_proj_across_delta} $\rho_\om(N^Y_\alpha)\geq\delta$ for all 
$\alpha\in[\delta,\eta)$, so $N^Y_\eta|\delta=M$.
 \item\label{item:Woodin_inherited} if $\eta>\delta$ then $N^Y_\eta\sats$``$\delta$ is Woodin''.
 \item\label{item:background_gen} if $\eta>\delta$ then $Y|\delta$ is $N^Y_\eta$-generic for the $\delta$-generator 
extender algebra of $N^Y_\eta$ at $\delta$.
 \item\label{item:Q-structure_match} If $Q\pins Y$ then $N^Y_\eta$ is a Q-structure for $\delta$ 
iff 
$Q$ is a Q-structure 
for $\delta$.
\end{enumerate}
\end{lem}
\begin{proof}
The proof is by induction on
$\delta$.

Part \ref{item:M_small}:
Suppose not and let $\bar{\delta}=\delta^M$.
We have $t_\delta=0$, so
by \ref{lem:meas_limits_of_con_require_no_proj_across}, $\bar{\delta}$ is a limit of measurables of $Y$
(hence $\bar{\delta}=\aleph_{\bar{\delta}}^Y$)
and $M|\bar{\delta}=N_{\bar{\delta}}$ (hence $\OR^{N_{\bar{\delta}}}=\bar{\delta}$).
Let $\bar{\eta}<\delta$
be such that $Y|\bar{\eta}$ is the Q-structure
for $\bar{\delta}$.
By induction, $N_{\bar{\eta}}$ is the P-construction
of $Y|\bar{\eta}$ over $M|\bar{\delta}$,
and $N_{\bar{\eta}}$ is a Q-structure for $\bar{\delta}$.
It follows that $\rho_\om^{N_{\bar{\eta}}}=\bar{\delta}$,
so $N_{\bar{\eta}}\pins M$ and $M\sats$``$\bar{\delta}$ is not Woodin'',
contradiction.

Part \ref{item:no_proj_across_delta}: Suppose not.
Let $\alpha\in[\delta,\eta)$
be least such that for some $n<\om$,
we have
$\rho_{n+1}(N_\alpha^Y)<\delta$,
 let $n<\om$ be least such,
 and let $\rho=\rho_{n+1}(N_\alpha^Y)$.
 Note that since $\alpha$ exists,
 we have $\delta<\eta$,
 so $\delta$ is Woodin
 in $Y|\eta$,
 so $Y|\eta\sats\ZFC$,
 so $M\sats\ZFC$, and therefore
 $\delta<\alpha$.
Let $R=Y|\alpha\pins Q$. An application of condensation
to $R$ easily yields that $\rho_{n+5}^R=\delta$.
Working inside $Q$, where $\delta$ is Woodin, we can find $\bar{\delta}<\delta$ such that, letting
\[ \bar{R}=\cHull_{n+5}^R(\bar{\delta}\un\pvec_{n+5}^R) \]
and 
$\pi:\bar{R}\to R$ be the uncollapse, then  $\bar{R}$ is $(n+5)$-sound and 
$\rho_{n+5}^{\bar{R}}=\bar{\delta}$ and $\rho<\crit(\pi)=\bar{\delta}$ and 
$\pi(\bar{\delta})=\delta$; by condensation we have $\bar{R}\pins Y|\delta$.
Note that $\bar{\delta}$ is a $Y$-cardinal
and $\bar{R}$ is $Y$-cardinal-large with $\bar{\delta}=\delta^{\bar{R}}$.

Let $(\bar{\CC},\bar{N})$ be defined over $\bar{R}$ as 
$(\CC^R,N^R)=(\CC^Y\rest(\alpha+1),N_\alpha^Y)$ is defined over $R$.
Note here that $\CC^R$ ends with the P-construction of $R$ above $M$, with output $N^R$,
so $\bar{\CC}$ ends with the P-construction of $\bar{R}$ above $\bar{M}=\pi^{-1}(M)$, with output 
$\bar{N}$. 
Note that
$\bar{\CC}\rest(\bar{\delta}+1)\sub\CC^Y$
and $\bar{M}\pins_\card M$ and 
\[  N^Y_\alpha|\bar{\delta}=N_{\bar{\delta}}^Y=\bar{M}=\bar{N}|\bar{\delta}. \]
It follows that
$\bar{\CC}=\CC^Y\rest(\OR^{\bar{R}}+1)$ and
$\bar{N}=N^Y_{\OR^{\bar{R}}}$.
But $\pi$ is sufficiently elementary that
$\rho_{n+1}^{\bar{N}}=\rho<\bar{\delta}$,
contradicting the fact that 
$\bar{M}\pins_\card M$.

Part \ref{item:Woodin_inherited}: This is almost by the argument of \cite[\S11]{fsit},
but some things are slightly different due to our 
modified background conditions.

Let $N=N^Y_\eta$ and $A\in\pow(\delta)\inter N$, so $A\in Q$. Let $\kappa<\delta$ be 
such that $Q\sats$``$\kappa$ is 
$({<\delta},(A,\es^N\rest\delta))$-reflecting as witnessed by $\es$''. It suffices to see that $N\sats$``$\kappa$ is $({<\delta},A)$-reflecting as witnessed by $\es$''.

Fix $\lambda\in(\kappa,\delta)$ with $\lambda$ inaccessible 
in $Y$ and $N_\lambda=N_\delta|\lambda$. Let $G^*\in\es^Y$ with 
$\crit(G)=\kappa$ and 
\[ i^Y_{G^*}(A,\es^N)\rest\lambda=(A,\es^N)\rest\lambda.\]
Let $G=G^*\rest N\cross[\lambda]^{<\om}$
and for $\rho<\lambda$ let
$G_\rho=G\rest\rho$. We will prove by induction on $N$-cardinals 
$\rho\in(\kappa,\lambda)$ that $G_\rho\in\es^N$,
completing the proof.

First consider $\rho=\kappa^{+N}$.
Let $\gamma=\rho^{+U}$ where $U=\Ult(N,G_\rho)$.
By condensation applied to the factor
\[ k:U\to\Ult(N,G) \]
and because $G$ coheres $\es^N\rest\lambda$, there is $\xi$ such that 
$N_\xi=U||\gamma=N||\gamma$. Note that $t_\xi=0$ and $N_\xi$ is small (as $M$ is small). Because $(N_\xi,G_\rho)$ is a premouse and $G^*$ 
is sufficiently strong and $G_\rho\rest\rho\sub G^*$, we have $F_{\xi+\om}\neq\emptyset$.
Let $N_{\xi+\om}=(N_\xi,E)$ and $B=(N_\xi,G_\rho,E)$. Then
$B$ is a bicephalus and a creature of $\CC^Y$, so because $Y$ is $(x,\delta)$-iterability-good,
we have $G_\rho=E$. It also follows that $N_{\xi+\om}\pins N$, so $G_\rho\in\es^N$ as required.

Now fix an $N$-cardinal $\rho\in(\kappa,\lambda)$ and suppose we have
$G_\rho\in\es^N$.

\begin{clm} There are unboundedly many $\gamma<\rho^{+N}$ such that
	$G_\gamma\in\es^N$.\end{clm}
\begin{proof}
$G$ has some generator in $(\rho,\rho^{+N})$,
because otherwise
\[ \rho^{+N}=\rho^{+\Ult(N,G_\rho)}=\lh(G_\rho),\]
contradicting the fact that $G_\rho\in N$.
And if $\rho$ is the greatest generator of $G$ in $(\rho,\rho^{+N})$,
then $G_{\rho+1}$ is non-type Z, for the same reason.
Now let $\gamma\in(\rho,\rho^{+N})$ be any generator of $G$ such that
$E=G_{\gamma+1}$
is non-type Z. Let
 $U=\Ult(N,E)$ and $\xi=\gamma^{+U}$, and note that the factor
 $k:U\to\Ult(N,G)$
  has $\crit(k)>\gamma$ and $\card^U(\gamma)=\rho$, so either 
\begin{enumerate}[label=(\roman*)]
 \item  $\crit(k)=\xi=\rho^{+U}<\rho^{+N}$, or
  \item $\xi=\rho^{+U}=\rho^{+N}<\crit(k)$.
 \end{enumerate}
So $U|\xi=N||\xi$. Note that $\lh(G_\rho)<\xi$ and $G_\rho\in\es^U$, because $G_\rho\in\es^N$ and 
$E$ is non-type Z.
So $P=(N||\xi,E)$ is a pseudo-premouse,
and a creature of $\CC^Y$. We have $N|\xi=N_\chi$ for some $\chi$ with $t_\chi\neq3$ as $N|\xi$ is 
small (otherwise $N|\kappa\sats$``There is a proper class of Woodins'').
So because $Y$ is $(x,\delta)$-iterability-good,
$P$ is a premouse. So if (ii) holds then 
$U|\xi=N|\rho^{+N}=N_\chi$, so $t_\chi=0$, but $E\rest\nu_E\sub G^*$, so then
$F_{\chi+\om}\neq\emptyset$, contradicting that $\xi=\rho^{+N}$. So (i) holds.

Suppose for the moment that the generators of $G$ are unbounded in $\rho^{+N}$.
Then the preceding 
discussion applies with arbitrarily large generators $\gamma<\rho^{+N}$. Since $N||\xi=U|\xi$ 
where $\xi$ is as above, and each $P$ as above satisfies the ISC,
it follows that all relevant proper segments of $G_{\rho^{+N}}$
are in $\es^N$, or are an ultrapower away, and the claim
follows from this.

So let $\nu$ be the strict sup of the generators of $G$ below $\rho^{+N}$ and suppose $\nu<\rho^{+N}$.
We have $N|\rho^{+N}=U|\rho^{+U}$ where $U=\Ult(N,G_\nu)$. 
Suppose $G_\nu$ is non-type Z. Then by the previous paragraph,
$\nu$ is a limit of generators. But then $\nu$ is a cardinal of 
$U$, so $\rho^{+U}=\nu<\rho^{+N}$, contradiction.
So $G_\nu$ is type Z, so $\nu=\xi+1$ and $\xi$ is a limit of generators,
and letting $U'=\Ult(N,G_\xi)$, we have $(\xi^+)^{U'}=\rho^{+N}$,
and like before, $\xi=(\rho^+)^{U'}$. The factor map $k:U'\to U$ has $\crit(k)=\xi$ and $k(\xi)=\rho^{+N}$.
But now condensation gives a contradiction: either $N|\xi$ is passive
and $U'|(\xi^+)^{U'}=N|\rho^{+N}$, which is impossible as $\xi$ is not a cardinal in $N$,
or $N|\xi$ is active and $(\xi^+)^{U'}\leq(\xi^+)^{\Ult(N|\xi,F^{N|\xi})}$,
which is impossible as $(\xi^+)^{U'}=\rho^{+N}$.
\end{proof}

Now shifting notation, let $\rho>\kappa^{+N}$ be an $N$-cardinal,
and suppose that for unboundedly many $\gamma<\rho$,
we have $G_\gamma\in\es^N$. It suffices to see:
\begin{clm} $G_\rho\in\es^N$. \end{clm}
\begin{proof}
	Let $\xi=(\rho^+)^{\Ult(N,G_\rho)}$. Suppose
\begin{equation}\label{eqn:t_chi=0}\text{there is }\chi\text{ such that 
}N_\chi=N||\xi\text{ and }t_\chi=0.\end{equation}
Again $N_\chi$ is small, so a bicephalus argument as in the case of $G_{\kappa^{+N}}$ shows that 
$G_\rho\in\es^N$.
Note that we use \ref{lem:iterability} for mixed type bicephali here, because 
$G_\rho$ is of type 3, and there might ostensibly be some competing 
type 2 (or type 1) extender. 

Finally, suppose (\ref{eqn:t_chi=0}) fails. Then note that $N|\xi=N_\chi$ 
for some $\chi$
such that $t_\chi=2$ (as $N||\xi$ is small). Therefore $\rho$ is $Y$-measurable and $N|\rho=N_\rho$.
Let $F_0\in\es^Y$ be the order $0$ measure on $\rho$.
In $Y$ we have all relevant proper segments of $G_\rho$
in $\es^N$. This fact lifts into $\Ult(Y,F_0)$, by $i^Y_{F_0}$.
But
$G_\rho=i^Y_{F_0}(G_\rho)\rest\rho$
is one of the relevant segments
of $i^Y_{F_0}(G_\rho)$, so $G_\rho\in\es^{i^Y_{F_0}(N)}$,
but then $G_\rho\in\es^N$.\end{proof}

Part \ref{item:background_gen}: This is standard: One shows that all 
the 
$\delta$-generator extender algebra axioms of $M$ hold for $\es^{Y|\delta}$, which suffices since $\delta$ is Woodin
in $N_\eta$.
This follows as usual from \ref{lem:extenders_restrictions_from_bkgd}: if $\nu$ 
is a cardinal of $M$ and $E\in\es^M$ and $\nu=\nu(E)$, then by 
\ref{lem:extenders_restrictions_from_bkgd}, $E$ is just the restriction of some 
$E^*\in\es^{Y|\delta}$, and $E^*$ coheres $\es^{Y|\delta}$ through $\nu$.

Parts \ref{item:P-con_defined}, \ref{item:eta+1_iter_good}:
By the previous parts and since $M=N^{Y|\delta}$ is definable
from parameters (in fact, $x$) over $Y|\delta$, the P-construction works, so part \ref{item:P-con_defined} holds,
which trivially yields part \ref{item:eta+1_iter_good},
as $M$ has no largest cardinal,
and $t^Y_\beta=3$ for all limits $\beta\in(\delta,\eta]$.

Part \ref{item:Q-structure_match}:
Note that we assume $Q\pins Y$ for this part, so $\J(Q)\ins Y$.

If $Q$ is not a Q-structure for $\delta$, so $\J(Q)\ins Y\sats$``$\delta$ is Woodin'',
then by the preceding parts applied to $\J(Q)$, $\J(N)\sats$``$\delta$ is Woodin'', so $N$ is also not a Q-structure for $\delta$.

Now suppose that $Q$ is a Q-structure,
but $N$ is not. Consider first the case that $\delta<\eta$.
Since $N$ is $\delta$-sound (as $Q\pins Y$, so $Q$ is $\omega$-sound), therefore $\delta\leq\rho_\om^N$. But letting $m<\om$ be such that $\rho_{m+1}^Q=\delta<\rho_m^Q$ (noting that $\delta<\rho_0^Q$ as $\delta$ is the least Woodin of $Q$), note that $\rho_{m+1}^N\leq\delta<\rho_m^N$, so $\rho_{m+1}^N=\delta$.

So $\J(N)$ is a premouse with $\delta$ Woodin,
 and $Y|\delta$ is $\J(N)$-generic for the extender algebra of at $\delta$.
We will show that $\J(Q)\sats$``$\delta$ is Woodin'', a contradiction.

We have that $\J(N)[Y|\delta]$ has universe $\J(Q)$,
because the extender algebra is small forcing relative to each of the extenders in $\es^N$
above $\delta$, so those in $\es^Q$ are just their small forcing extensions. 
Now let $A\in\pow(\delta)\inter\J(Q)$. Let $\tau\in\J(N)$ be a nice name for
$A$. That is, let $\tau:\delta\to\BB_\delta^{\J(N)}$ with $\tau\in\J(N)$  and
 $\alpha\in A$ iff $\tau(\alpha)\in G$, where $G$ is the extender algebra generic determined by 
$\es^Q\rest\delta$.
Let $f:\delta\to\delta$ be a function in $N$
such that for each $\beta<\delta$, $\tau``\beta\sub N|{f(\beta)}$.
Let $\kappa<\delta$ be such that
\begin{equation}\label{eqn:reflection_in_N} \J(N)\sats\text{``}\kappa\text{ is 
}(<\delta,(\tau,f))\text{-reflecting,
as witnessed by }\es\text{''.}\end{equation}

We reach a contradiction by verifying that
\begin{equation}\label{eqn:reflection_in_Q} \J(Q)\sats\text{``}\kappa\text{ is 
}(<\delta,A)\text{-reflecting,
as witnessed by }\es\text{''.}\end{equation}
To see this, let $\lambda\in(\kappa,\delta)$ be any $N$-cardinal such that 
$f``\lambda\sub\lambda$. Let $E\in\es^N$ be such that $\lambda=\nu_E$
and $E$ witnesses line (\ref{eqn:reflection_in_N}) through $\lambda$.
Because $\lambda=\nu_E$ is an $N$-cardinal, by Lemma \ref{lem:extenders_restrictions_from_bkgd} `there is a $Q$-total $E^*\in\es^Q$
such that $\nu_E\leq\nu_{E^*}$ and $E\rest\nu_E\sub E^*$.
We have $f``\lambda\sub\lambda$, so $\tau``\lambda\sub N|\lambda$.
Note that
\[ i^Q_{E^*}(f,\tau,\es^Q)\inter V_\lambda=(f,\tau,\es^Q)\inter V_\lambda, \]
and ($*$) for $\alpha<\kappa$, we have $\alpha\in A$ iff $\es^Q\sats\tau(\alpha)$,
($\tau(\alpha)$ is some Boolean formula).

Lifting ($*$) with $j=i^Q_{E^*}$, for $\alpha<j(\kappa)$, we have
$\alpha\in j(A)$ iff $j(\es^Q)\sats j(\tau)(\alpha)$. But for $\alpha<\lambda$,
$j(\tau)(\alpha)=\tau(\alpha)\in V_\lambda$, and $j(\es^Q)\rest \lambda=\es^Q\rest \lambda$.
But then $A\inter\lambda=j(A)\inter\lambda$ as required.

Now suppose instead that $\delta=\eta$. So we are assuming that
$M\sats\ZFC$ and $\J(M)\sats$``$\delta$ is Woodin'',
but $Q=Y|\delta$ is a Q-structure for $\delta$. But then just like in
part \ref{item:background_gen}, $Y|\delta$ is $\J(M)$-generic for the $\delta$-generator extender algebra of $\J(M)$ at $\delta$,
and we can argue as in the previous case for a contradiction.
\end{proof}

\begin{lem} (***to add, with setup of previous lemma): Suppose that $\delta$ is $\bfrSigma_1^Q$-singular. We show that $\delta$ is
not $\bfrSigma_1^N$-Woodin (***to finish...)\end{lem}

From \ref{lem:meas_limits_of_con_require_no_proj_across} and \ref{lem:Woodin_exactness} 
together, we easily get a characterization of all local Woodin cardinals
of models of $\CC$:

\begin{cor}\label{cor:local_Woodins}
 Let $Y$ be a premouse, $x\in\RR^Y$,
$N=N^Y_\alpha$ defined,  $t^Y_\alpha\neq 2$,
$\delta<\alpha$, and $Y$ be $(x,\delta)$-iterability-good.
Then $\delta\in N$ and $N\sats$``$\delta$ is the least Woodin cardinal''
 iff $Y|\alpha$ is $Y$-cardinal-large
 and $\delta=\delta^{Y|\alpha}$.
 Therefore $\delta\in N$ and $N\sats$``$\delta$ is Woodin''
 iff $Y|\alpha$ is $Y$-cardinal-large
 and $Y|\alpha\sats$``$\delta$ is Woodin''.
\end{cor}

See
\cite{fsfni_v4},
\cite[Theorem 5.2]{premouse_inheriting} and
\cite[Lemma 2.3]{V=HODX_pub} for proofs of the following facts
(note that we only assume $(y,\om_1+1)$-iterability,
not $(y,\om_1,\om_1+1)^*$):

\begin{fact}\label{fact:condensation_from_norm_it}
Let $y>0$ and  $Y$ be a $y$-sound, $(y,\om_1+1)$-iterable premouse.
Then:
\begin{enumerate}[label=\tu{(}\roman*\tu{)}]
 \item $Y$ is $(y+1)$-solid and $(y+1)$-universal, and satisfies $(y+1)$-condensation
 (see \cite[Theorem 5.2]{premouse_inheriting}),
 \item\label{item:self-solid} Suppose $\om<\rho_y^Y$ and let $\theta<\rho_y^Y$
be a $Y$-cardinal and $x\in Y$. Then there is $q\in[\OR^Y]^{<\om}$
such that $q\cut\min(p_y^Y)=p_y^Y$ and letting
\[ C=\cHull_y^Y(\theta\cup\{q,\vec{p}_{y-1}^Y\}) \]
and $\pi:C\to Y$ the uncollapse, then $C\pins Y$ (hence $C$ is fully sound),
$\rho_\om^C=\theta$, and $\pi(p_y^C)=q$ (and note $\pi$ is a near $(y-1)$-embedding).
\end{enumerate}

\end{fact}

\begin{lem}\label{lem:rSigma_red_N_to_Y}
 Let $Y$ be a non-small premouse, $\delta=\delta^Y$,
 $Y$ is $(x,\delta)$-iterability-good,
 $N=N^Y$ \tu{(}$N$ exists and is non-small with $\delta=\delta^N$,
 by \ref{lem:Woodin_exactness}\tu{)}.
 Suppose $Y,N$ are $y$-sound and
 $\delta\leq\rho_z^N=\rho_z^Y$ and $p_z^N=p_z^Y$ for all $z\leq y$.
 Then
 \begin{enumerate}
 \item\label{item:rSigma_1_red} $\rSigma_1^N\sub\rSigma_1^Y(\{\delta,x\})$
 and  $\bfrSigma_1^N\sub\bfrSigma_1^Y$,
  \item\label{item:rSigma_y+1_red} if $y>0$ then $\rSigma_{y+1}^N\sub\rSigma_{y+1}^Y(\{x\})$
 and $\bfrSigma_{y+1}^N\sub\bfrSigma_{y+1}^Y$.
\end{enumerate}
Moreover, there is a recursive function $\varphi\mapsto\varphi'$
of formulas $\varphi$
computing these reductions
uniformly in $Y,y,x$.
\end{lem}
\begin{proof}
 Part \ref{item:rSigma_1_red}: This follows easily from the fact that $N$ is the P-construction of $Y$
over $N|\delta$, and that $N|\delta$ is definable over $Y|\delta$.

Part \ref{item:rSigma_y+1_red}: By induction on $y$. Suppose $\rSigma_y^N\sub\rSigma_y^Y(\{\delta,x\})$,
and that we have the uniform reduction for this complexity. Because $y>0$, 
$\{\delta\}$ is $\rSigma_{y+1}^Y$, so it follows easily
from the hypotheses of the lemma
and the inductive hypothesis that
$\rSigma_{y+1}^N\sub\rSigma_{y+1}^Y(\{x\})$
and $\bfrSigma_{y+1}^N\sub\bfrSigma_{y+1}^Y$, and that the reductions are uniform.
\end{proof}

\begin{lem}\label{lem:Q-local_fs_match}
 Let $Y$ be a non-small premouse, $\delta=\delta^Y$,
 $Y$ is $(x,\delta)$-iterability-good,
 $N=N^Y$.
 Suppose $Y$ is $y$-sound, $(y,\om_1+1)$-iterable,
$\delta\leq\rho_y^Y$, and $\delta^Y$ is $\bfrSigma_y^Y$-regular.
Then:
\begin{enumerate}
 \item\label{item:fs_match} $N$ is $y$-sound and $\rho_z^N=\rho_z^Y$
and $p_z^N=p_z^Y$ for each $z\leq y$ \tu{(}so \ref{lem:rSigma_red_N_to_Y} applies\tu{)},
 \item\label{item:N_it_above_delta} $N$ is above-$\delta$, $(y,\om_1+1)^*$-iterable,
 and if $Y$ is above-$\delta$, $(y,\om_1,\om_1+1)^*$-iterable then so is $N$.
 \end{enumerate}  
\end{lem}

\begin{proof}
The proof is by induction on $y$.

Note that we do not assume that $\delta$ is $\bfrSigma_y^Y$-Woodin.
But of course if $\delta$ is not $\bfrSigma_y^Y$-Woodin
then $\rho_y^Y=\delta$, and note that if $\rho_{y'}^N=\delta$ where $y'<y$,
then for part \ref{item:fs_match}, we just need to see that $\rho_y^N\geq\delta$,
and part \ref{item:N_it_above_delta} is trivial (modulo induction).

Part \ref{item:fs_match}: We may assume $y>0$.
Suppose $\rho_y^N<\delta=\rho_y^Y$.
Let $\theta<\delta$
 be such that
 \[ \delta,\vec{w}_y^Y,\pvec_y^N\in\Hull^Y_y(\theta\un\{\pvec_y^Y\}).\]
 where $\vec{w}_y^Y$ is the set of $y$-solidity witnesses for $Y$.
 Now because $\delta$ is $\bfrSigma_y^Y$-regular,
\[ H=\Hull_y^Y(\rho_y^N\un\theta\un\{\pvec_y^Y\})\inter\delta \]
is bounded in $\delta$. Let $\eta=\sup(H\inter\delta)$.
Note that $\eta$ is a limit cardinal of $Y$. Let
\[ C=\cHull_y^Y(\eta\un\{\pvec_y^N\}) \]
and $\pi:C\to Y$ be the uncollapse. Note that $\crit(\pi)=\eta$ and $\pi(\eta)=\delta$
and $\pvec_y^N\in\rg(\pi)$, $C$ is sound and $\rho_y^C=\eta$, so by condensation (with \ref{fact:condensation_from_norm_it}),
$C\pins Y$. Also, $\CC^C\rest\eta=\CC^Y\rest\eta$
and $\eta=\delta^C$, so $N_{\OR^C}^Y=N^C$ is the P-construction of $C$ above $N_\eta^C=N_\eta^Y$.
By induction  and \ref{lem:rSigma_red_N_to_Y},
\[ \pi\rest N^C:N^C\to N \]
is $\rSigma_y$-elementary,
so $\core_y(N)$ is computable from $N^C$, so $N^C\notin N$.
But $N^C||\eta=N||\eta$ is a cardinal segment of $N$ and $N^C$ projects $\leq\eta$,
so $N^C\in N$, contradiction.

If instead $\eta'=\max(\delta,\rho_y^N)<\rho_y^Y$, then one can argue similarly for a contradiction,
but replacing $\eta$ above with $\eta'$,
and using \ref{fact:condensation_from_norm_it}\ref{item:self-solid} to obtain a sound hull of $Y$.

So $\rho_y^N\geq\rho_y^Y\geq\delta$. Because
$Y$ is $y$-sound and is extender algebra generic over $N$
(and by the local definability of the forcing relation),
we therefore have $\rho_y^N=\rho_y^Y$ and $p_y^N\leq p_y^Y$.
Let $\rho=\rho_y^N=\rho_y^Y$. We must still verify that $p_y^N\geq p_y^Y$
(hence equality holds)
and that $N$ is $y$-sound.

Since $Y$ is above-$\delta$,
$(y-1,\om_1+1)$-iterable,
so is $N$,
so by \cite{fsfni_v4},
$N$ is $y$-universal and $y$-solid.
(Clearly $N$ considered as an $N|\delta$-premouse is $y$-universal and $y$-solid. As $\delta\leq\rho_y^N$, it readily follows
that $N$ itself is $y$-universal and $y$-solid,
except in the case that $y=1$ and $\rho_1^N=\delta$.
But in this case, we have $\delta\in\Hull_1^N(\{p_1^N\})$,
which again ensures that $N$ itself is $y$-universal and $y$-solid.
For if $p_1^N\neq\emptyset$ or $N$ is active this is trivial,
as by Woodin exactness \ref{lem:Woodin_exactness},
$\delta$ is the least Woodin of $N$.
But if $N$ is passive then $\Hull_1^N(\delta)=N|\delta$ by condensation,
so $p_1^N\neq\emptyset$.) By $y$-universality and level-by-level correspondence,
\[ (\rho^+)^Y=\rho^{+N}\sub\cHull_y^N(\rho\cup\pvec_y^N), \]
and by induction and \ref{lem:rSigma_red_N_to_Y},
therefore
$(\rho^+)^Y\sub\cHull_y^Y(\rho\cup\pvec_y^N)$,
and since $Y$ is $y$-solid, therefore  $p_y^N\geq p_y^Y$.
The $y$-soundness of $N$ now follows by genericity (or is trivial if $\rho_{y-1}^Y=\delta$).

Part \ref{item:N_it_above_delta}: This follows from standard
considerations with P-constructions. (Here is a sketch.
One establishes
a 1--1 correspondence between
above-$\delta$, $y$-maximal trees $\Tt$ on $N$, and above-$\delta$, 
$y$-maximal trees $\Uu$ on $Y$. The correspondence
is such that $\lh(\Tt)=\lh(\Uu)$, for each $\alpha<\lh(\Tt)$,
we have $M^\Tt_\alpha=N(M^\Uu_\alpha)$, and these models
correspond fine structurally much as do $N,Y$, 
$\dropset^\Tt_\deg=\dropset^\Uu_\deg$ and $\deg^\Tt=\deg^\Uu$,
and when $i^\Tt_{\alpha\beta}$ is defined then $i^\Tt_{\alpha\beta}\sub i^\Uu_{\alpha\beta}$,
and likewise for $i^{*\Tt}_{\alpha\beta}$ and $i^{*\Uu}_{\alpha\beta}$.
One also maintains that $\rSigma_m^{M^\Tt_\alpha}\sub\rSigma_m^{M^\Uu_\alpha}(\{\delta\})$ when $m=\deg^\Tt(\alpha)>0$.)
\end{proof}

\begin{lem}\label{lem:iterates_pres_fs_match}
 Let $Y,y,N$ satisfy the hypotheses of \ref{lem:Q-local_fs_match}
 excluding iterability, and satisfy part \ref{item:fs_match}
 of its conclusion.
 Let $\Tt$ be a successor length $y$-maximal tree on $Y$ which is based on $Y|\delta$,
 such that $b^\Tt$ does not drop in model.
  Let $Y'=M^\Tt_\infty$ and $\delta'=i^\Tt(\delta)$.
  Then  $Y'$ is $(x,\delta')$-iterability-good.
 Let $N'=N^{Y'}$.
 Then part \ref{item:fs_match} of \ref{lem:Q-local_fs_match}
 holds of $Y',y$, and therefore 
 \ref{lem:rSigma_red_N_to_Y} applies
 and
 \[ i^\Tt\rest N:N\to N' \]
 is a $y$-embedding.
\end{lem}
\begin{proof}
The fact that $Y'$ is $(x,\delta')$-iterability-good
is by the elementarity of $i^\Tt$.

Note that \ref{lem:iterates_maintain_def_regularity} applies to $\Tt,\delta$,
so $b^\Tt$ does not drop in degree,
$\delta^{Y'}\leq\rho_y^{Y'}$ and $\delta^{Y'}$ is $\bfrSigma_y^{Y'}$-regular.
The proof of part \ref{item:fs_match} of \ref{lem:Q-local_fs_match}
is now by induction on $z\leq y$.
It is trivial if $z=0$,
so suppose $z>0$.
So Lemma \ref{lem:rSigma_red_N_to_Y}
applies at $z-1$.

Suppose first that $z<y$,
so $\rho_z^{Y'}=i^\Tt(\rho_z^Y)$
and $p_z^{Y'}=i^\Tt(p_z^Y)$.
We have $p_z^N=p_z^Y$,
and $N$ is $z$-solid.
Recall that ``$w$ is a generalized $z$-solidity witness'' is $\rPi_z(\{\pvec_{z-1}\})$. But by Lemma \ref{lem:rSigma_red_N_to_Y} applied at $z-1$,
$\rSigma_z^N$ is expressible in $\rSigma_z^Y(\{\delta\})$,
and hence the $\rSigma_z$-elementarity
of $i^\Tt$ preserves this statement.
Somewhat similarly, but with higher complexity,
we have that $\all \alpha<\rho_z^N=\rho_z^Y$ there is $t\in N$ such that $N\sats$``$t$ is a generalized $z$-solidity witness
for $\alpha\cup\{\pvec_z^N\}$''.
But by Lemma \ref{lem:rSigma_red_N_to_Y}
at $z-1$, note that this statement
is $\rSigma_{z+2}^Y(\{\delta\})$,
hence also preserved by $i^\Tt$
(using that $z<y$ and $i^\Tt$ is $\rSigma_{y+1}$-elementary).
But conversely, because of the genericity of $Y'$ over $N'$,
we have $N'=\Hull^{N'}(\rho_z^{Y'}\cup\{\pvec_z^{Y'}\})$.
Putting these things together
gives that $\rho_z^{N'}=\rho_z^{Y'}$
and $p_z^{N'}=p_z^{Y'}$,
and $N'$ is $z$-sound.

Now suppose that $z=y$.
Then we can argue mostly as above,
except that now we get $\rho_z^{N'}=\sup i^\Tt``\rho_z^{N'}$;
this uses that the generalized $z$-solidity  witnesses for $\alpha\cup\{p_z^N\}$ for each $\alpha<\rho_z^N$
are preserved by $i^\Tt$.
Otherwise it is as before.

Since we have $\rho_y^{N'}=\rho_y^{Y'}$,
and by Lemma \ref{lem:rSigma_red_N_to_Y},
it now follows that $i^\Tt\rest N:N\to N'$
is a $y$-embedding.
\end{proof}

\begin{lem}
 Let $Y$ be a premouse, $x\in\RR^Y$,
 $\alpha$ a limit such that $\CC\rest(\alpha+1)$ is defined
 and $t_\alpha\in\{1,2\}$. Then $N_\alpha^Y,F_\alpha^Y\in Y|\rho_0^Y$
 (so in particular, if $Y$ is active type 3,
 then $N_\alpha^Y,F_\alpha^Y\in Y^\sq$).
 \end{lem}
 \begin{proof}
We have $\aleph_\alpha^Y\leq\OR^Y$, or more precisely, for each $\beta<\alpha$,
 $Y\sats$``$\aleph_\beta$ exists'',
 so $\aleph_\beta^Y<\OR^Y$.
 
 Suppose $t_\alpha=1$. So $\alpha=\gamma+\om$ for some $\gamma$,
 and $N_\alpha=(N_\gamma,F)$ is an active premouse. Moreover,
 $\nu(F_\alpha^Y)=\aleph_\gamma^Y<\lh(F_\alpha^Y)<\aleph_{\gamma+1}^Y<\OR^Y$, so $F_\alpha^Y\in Y$ ($F_\alpha^Y$ is not the active extender of $Y$).
 Note that $N_\alpha$ is definable over $Y|\lh(F_\alpha^Y)$, so $N_\alpha\in Y$.
 Since $N_\alpha$ is coded by a subset of its largest cardinal $\rho$, therefore in fact $N_\alpha\in Y|\rho^{+Y}$. Also since $\aleph_\alpha^Y\leq\OR^Y$,
 we have $Y|\lh(F_\alpha^Y)<\rho_0^Y$.
 
 If $t_\alpha=2$,
 then $\kappa+\om\leq\alpha$,
 so $\aleph_{\kappa+\om}^Y\leq\OR^Y$,
 so $F_\alpha^Y$ (an order 0 measure)
 is in $\es^Y$ and in fact, 
 $\lh(F_\alpha^Y)<\rho_0^Y$
 (and $N_\alpha$ is again definable over $Y|\lh(F_\alpha^Y)$).
\end{proof}

\begin{rem}
 Of course if $Y$ is active type 3 and satisfies ``there is a Woodin cardinal'' and $\alpha>\nu(F^Y)$
 with $Y|\alpha$ active, then $F^Y_\alpha$
 is just the restriction of $F^{Y|\alpha}$,
 and hence is not in $Y^\sq$.
 If $t_\alpha=0$, then $\alpha\leq\aleph_\alpha^Y\leq\rho_0^Y$,
 and $N_\alpha$ is definable over $Y|\aleph_\alpha^Y$. (If $Y$ is active type 3 then we must have $\aleph_\alpha^Y\leq\rho_0^Y$, since $\aleph_\alpha^Y$ is a limit of $Y$-cardinals. It follows that $N_\alpha^Y$ is definable over $Y|\rho_0^Y$,
 hence $N_\alpha^Y\in Y$,
 even if $\OR(N_\alpha^Y)=\alpha=\aleph_\alpha^Y=\rho_0^Y$. If $Y$ is passive, we might have that $\OR(N_\alpha)=\alpha=\aleph_\alpha^Y=\OR^Y$.)
\end{rem}

\section{Copying and $\nu$-preservation}\label{sec:nu-pres}
\begin{dfn}\label{dfn:nu-pres,neat}
Let $Q,R$ be premice and $\pi:Q^\sq\to R^\sq$ be 
$\Sigma_0$-elementary.
Recall that $\psi_\pi$ denotes $\pi$ if $Q,R$ are passive, and otherwise letting $\mu=\crit(F^Q)$,
\[ \psi_\pi:\Ult(Q|(\mu^+)^Q,F^Q)\to\Ult(R|(\pi(\mu)^+)^R,F^R) \]
is the map induced by $\pi$ via the Shift Lemma. So $\pi\sub\psi_\pi$.
Recall that $\pi$ is:
\begin{enumerate}[label=--]
 \item \emph{$\nu$-preserving} iff, if $Q$ is type 3 then $\psi_\pi(\nu(Q))=\nu(R)$,
 \item \emph{$\nu$-high} iff $Q$ is type 3 and $\psi_\pi(\nu(Q))>\nu(R)$,
 \item \emph{$\nu$-low} iff $Q$ is type 3 and $\psi_\pi(\nu(Q))<\nu(R)$.\qedhere
\end{enumerate}
\end{dfn}

The next lemma is straightforward:
\begin{lem}\label{lem:nu-pres_elem}
Let $\pi:M^\sq\to N^\sq$ be a weak $m$-embedding.
Then:
\begin{enumerate}[label=--]
 \item If $M,N$ are not type 3 then $\pi$ is $\nu$-preserving.
 \item If $\pi$ is $\rSigma_1$-elementary then $\pi$ is non-$\nu$-low.
 \item If $\pi$ is $\rSigma_2$-elementary then $\pi$ is $\nu$-preserving.
\end{enumerate}
\end{lem}
\begin{lem}\label{lem:nu-pres_it_map}
 Let $M$ be an $m$-sound premouse and $\Tt$ an $m$-maximal tree on $M$.
 Let $\alpha\leq_\Tt\beta$ be such that $(\alpha,\beta]_\Tt\inter\dropset^\Tt=\emptyset$
 and $M^\Tt_\beta$ is type 3. Then:
 \begin{enumerate}
  \item\label{item:non-nu-low} $i^\Tt_{\alpha\beta}$ is non-$\nu$-low,
 and if $\alpha$ is a successor then $i^{*\Tt}_{\alpha\beta}$ is non-$\nu$-low.
 \item\label{item:when_nu-high} $i^\Tt_{\alpha\beta}$ is 
$\nu$-high
 iff $\cof^{M^\Tt_\alpha}(\nu(M^\Tt_\alpha))=\kappa$ is $M^\Tt_\alpha$-measurable
 and there are $\xi,\gamma+1$ such that
 \[ \alpha\leq_\Tt\xi<_\Tt\gamma+1\leq_\Tt\beta,\]
 $\pred^\Tt(\gamma+1)=\xi$ and $\deg^\Tt(\gamma+1)=0$ and
$\crit(i^\Tt_{\xi\beta})=i^\Tt_{\alpha\xi}(\kappa)$.
Similarly for $i^{*\Tt}_{\alpha\beta}$ if $\alpha$ is a successor
\tu{(}but allowing $\xi\in[\pred^\Tt(\alpha),\beta)_\Tt$\tu{)}.
\item\label{item:cof_nu(M_beta)} Suppose $i^\Tt_{\alpha\beta}$ is $\nu$-high and let $\xi$ be least 
as above. Then $i^\Tt_{\alpha\xi}$ is $\nu$-preserving
and
\[ \cof^{M^\Tt_\xi}(\nu(M^\Tt_\xi))=i^\Tt_{\alpha\xi}(\kappa)=
\cof^{M^\Tt_\beta}(\nu(M^\Tt_\beta))<i^\Tt_{\alpha\beta}(\kappa)<\nu(M^\Tt_\beta).\]
Similarly for $i^{*\Tt}_{\alpha\beta}$ if $\alpha$ is a successor.
In particular, $\nu(M^\Tt_\beta)$ is $M^\Tt_\beta$-singular.
\item\label{item:nu_reg} Suppose $\nu(M^\Tt_\alpha)$ is $M^\Tt_\alpha$-regular. Then $i^\Tt_{\alpha\beta}$ is 
$\nu$-preserving and $\nu(M^\Tt_\beta)$ is $M^\Tt_\beta$-regular. Likewise for $i^{*\Tt}_{\alpha\beta}$, $M^{*\Tt}_\alpha$ and 
$\nu(M^{*\Tt}_\alpha)$ if $\alpha$ is a successor.
\item\label{item:nu_sing} Suppose $\nu(M^\Tt_\alpha)$ has $M^\Tt_\alpha$-cofinality $\mu<\nu(M^\Tt_\alpha)$
and $i^\Tt_{\alpha\beta}$ is $\nu$-preserving. Then $\nu(M^\Tt_\beta)$ has $M^\Tt_\beta$-cofinality $i^\Tt_{\alpha\beta}(\mu)$.
Likewise for $i^{*\Tt}_{\alpha\beta}$, $M^{*\Tt}_\alpha$ and 
$\nu(M^{*\Tt}_\alpha)$ if $\alpha$ is a successor.
\end{enumerate}
\end{lem}
\begin{proof}
We have $\nu$-preservation for iteration maps of degree $\geq 1$ by \ref{lem:nu-pres_elem},
and likewise non-$\nu$-lowness for all iteration maps. This gives part \ref{item:non-nu-low}.
Note that parts \ref{item:nu_reg} and \ref{item:nu_sing}
follow easily from the other parts, considering the elementarity
of $\psi_{j}$  and $\psi_{j^*}$
where $j=i^\Tt_{\alpha\beta}$ and $j^*=i^{*\Tt}_{\alpha\beta}$.

Now consider the case that $M$ is a type 3 premouse and $E$ is an extender over $M$
and $U=\Ult_0(M,E)$ and $i^\sq:M^\sq\to U^\sq$ is the ultrapower map.
(So literally $U=\Ult(M^\sq,E)^\unsq$.)
Let $\nu=\nu^M$ and $\kappa=\crit(E)$.
Let $U'=\Ult(M,E)$, formed without squashing,
shifting the predicate as with type 2 extenders. Let $i':M\to U'$ be the ultrapower map.
So $i^\sq=i'\rest M^\sq$, and $i'$ is continuous at $\nu$ iff $\cof^M(\nu)\neq\kappa$.

Clearly $F^{U'}\rest\nu^U=F^U\rest\nu^U$, and by \cite[\S9]{fsit}, $\nu(F^{U'})=\nu^U$.
Letting $\mu=\crit(F^U)$ and $P=U|(\mu^+)^U$, we have $P=U'|(\mu^+)^{U'}$.
Let $\nu'=i'(\nu)=\lgcd(U')$.
Let $f:[\kappa]^{<\om}\to\kappa$ with $f\in M$ and $a\in[\nu]^{<\om}$ be such that
$\nu=[a,f]^M_{F^M}$.
We have $a,f\in M^\sq$.
But by standard calculations,
\[ 
\psi_{i^\sq}(\nu)=
[i^\sq(a),i^\sq(f)]^P_{F^{U}}=
[i'(a),i'(f)]^P_{F^{U'}}=\nu'.\]
So $i^\sq$ is $\nu$-high iff $i'$ is discontinuous at $\nu^M$ iff $\cof^M(\nu)=\kappa$.
Part \ref{item:when_nu-high} easily follows.

Now suppose that $i^\sq$ is $\nu$-high and let $f:\kappa\to\nu^M$ be cofinal
with $f\in M$. Then $i'(f)\in U'$, so $g\eqdef i'(f)\rest\kappa\in U'$,
but
\[ g=i'\com f=i^\sq\com f, \]
so $g$ is cofinal in $\nu^U$ and $g:\kappa\to\nu^U$, but $U^\passive=U'||((\nu^U)^+)^{U'}$, so $g\in U$.
So $\cof^U(\nu^U)=\kappa$.

Now consider part \ref{item:cof_nu(M_beta)}. By part \ref{item:when_nu-high}, $i^\Tt_{\alpha\xi}$ 
is $\nu$-preserving. So letting $\psi=\psi_{i^\Tt_{\alpha\xi}}$,
\[ \psi\rest(M^\Tt_\alpha)^\passive:(M^\Tt_\alpha)^\passive\to (M^\Tt_\xi)^\passive \]
is fully elementary and $i^\Tt_{\alpha\xi}\sub\psi$, and therefore
$\cof^{M^\Tt_\xi}(\nu(M^\Tt_\xi))=i^\Tt_{\alpha\xi}(\kappa)$. Let $\gamma+1\leq_\Tt\beta$ be such 
that $\pred^\Tt(\gamma+1)=\xi$. By the previous paragraph,
\[ \cof^{M^\Tt_{\gamma+1}}(\nu(M^\Tt_{\gamma+1}))=i^\Tt_{\alpha\xi}(\kappa). \]
So if $\gamma+1=\beta$ we are done. But if $\gamma+1<\beta$ then
$i^\Tt_{\alpha\xi}(\kappa)<\crit(i^\Tt_{\gamma+1,\beta})$,
so $i^\Tt_{\gamma+1,\beta}$ is $\nu$-preserving, so part \ref{item:cof_nu(M_beta)} 
follows by considering $\psi_{i^\Tt_{\gamma,\beta+1}}$.
\end{proof}

\begin{rem}\label{rem:nu-pres}
Letting $R$ be any type 3 premouse, $\kappa=\crit(F^R)$,
and $\gamma\in[(\kappa^+)^R,\nu^R)$ be such that $(\gamma^+)^R<\nu^R$,
and letting $M\pins R$ be such that $F^M=F^R\rest(\gamma^+)^R$,
then the inclusion map $M^\sq\to R^\sq$ is $\Sigma_0$-elementary and $\nu$-low.
But this map is not a weak $0$-embedding; see \ref{lem:nu-low_weak-0}.
However, there are also examples of weak $0$-embeddings which are $\nu$-low, and examples of core 
embeddings $\core_1(N)\to\core_0(N)$ which are $\nu$-high.
\end{rem}
\begin{lem}\label{lem:nu-low_weak-0}
 Let $\sigma:M^\sq\to N^\sq$ be a $\nu$-low weak $n$-embedding
and $\widetilde{\nu}=\psi_\sigma(\nu^M)$, so $\wt{\nu}<\nu^N$.
 Then:
 \begin{enumerate}[label=--]
 \item $n=0$.
  \item $\nu^M$ is a limit cardinal of $M$, so $\widetilde{\nu}$ is a limit cardinal of 
$N$.
 \item There is a type 3 $\widetilde{N}\pins N$ such that 
$F^{\widetilde{N}}=F^N\rest\widetilde{\nu}$, and 
$\rho_\om^{\widetilde{N}}=\widetilde{\nu}=\nu^{\widetilde{N}}$.
 \item $\wt{N}\in N^\sq$ and $\rg(\sigma)\sub\wt{N}^\sq\elem_0 N^\sq$.
 \item Letting $\wt{\sigma}:M^\sq\to\wt{N}^\sq$ have the same graph as $\sigma$, then $\wt{\sigma}$ 
is a $\nu$-preserving weak $0$-embedding.
\end{enumerate}
\end{lem}
\begin{proof}
The fact that $n=0$ is by \ref{lem:nu-pres_elem}.
Suppose that
$\nu^M=(\gamma^+)^M$ where $\gamma=\lgcd(M)$. We have $\widetilde{\nu}<\nu^N$ so
\[ \wt{\nu}=(\sigma(\gamma)^+)^{N^\sq}\in N^\sq. \]
Because $\sigma$ is a weak $0$-embedding, $\sigma$ is $\Sigma_1$-elementary in the parameter 
$\gamma$.
It is a $\Sigma_1^{N^\sq}(\{\sigma(\gamma)\})$ assertion
that $\sigma(\gamma)^+$ exists (as the assertion can refer to $F^N$).
This pulls back to $M^\sq$, a contradiction.

Because $\nu^M$ is a limit cardinal of $M$, most of the 
rest follows easily from the ISC. Note that $\wt{\sigma}$ is $\nu$-preserving by construction.
The fact $\wt{\sigma}$ is a weak $0$-embedding
is because $\sigma$ is such and $\rg(\sigma)\sub\wt{N}^\sq\elem_0 
N^\sq$. \end{proof}

\begin{dfn}Let $N$ be a premouse.

Define $\segs(N)=\{M\mid M\ins N\}$\index{$\segs$} and\index{$\cardprojpropsegs$}
\[ \cardprojpropsegs(N)=\{M\mid M\pins N\ \&\ \rho_\om^M\in\card^N\} \]
(here \emph{segs} abbreviates \emph{segments} and \emph{cpp} abbreviates \emph{cardinal projecting proper}).
If $N$ is $n$-sound we write $(M,m)\ins(N,n)$ iff $M\ins N$ and $m\leq\om$
and either $M\pins N$ or $m\leq n$. We write $(M,m)\pins(N,n)$ iff $(M,m)\ins(N,n)$ but 
$(M,m)\neq(N,n)$.
Define (\emph{segments-degrees})\index{$\segdegs$}
\[ \segdegs(N,n)=\{(M,m)\mid (M,m)\ins(N,n)\}. \]

Suppose $N$ is $n$-sound and let $e\ins N$ be active and $\kappa<\nu(F^e)$.
Then
$\psegdeg^{N,n}(e,\kappa)$\index{$\psegdeg$}
(\emph{projecting segment-degree}) denotes the least $(S,s)\ins(N,n)$ such that $(e,0)\ins(S,s)$ and either $(S,s)=(N,n)$ or [$s<\om$ 
and $\rho_{s+1}^S\leq\kappa$].

 Suppose $N$ is $\om$-sound. We say that $N$ is \emph{properly projecting}\index{properly projecting}
iff $\rho_\om^N<\OR^N$. In general
(independent of whether $N$ is properly projecting), the \emph{projection degree}\index{projection degree}
$\projdeg(N)$\index{$\projdeg$} of $N$
is the least $q\in[-1,\om)$ such that $\rho_{q+1}^N=\rho_\om^N$
(so $q=-1$ iff either $N$ is non-properly-projecting or $N$ is type 3 with $\nu^N=\rho_\om^N$).

Let $M\ins N$.
 The \emph{extended model dropdown sequence}
 \index{extended model dropdown sequence}
 \index{model dropdown sequence}\index{dropdown sequence} $\emdd^N(M)$\index{$\emdd$} of $(N,M)$ is the sequence 
$\left<M_i\right>_{i\leq k}$,
where (i) $M_0=M$ and $M_k=N$, (ii) $k>0$ iff $M\pins N$, (iii) if $i<k$ then $i+1<k$ iff 
$\rho_\om^{M_i}$ is not an $N$-cardinal, and
(iv) if $i+1<k$ then $M_{i+1}$ is the least $R$ such that $M_i\pins R\pins N$
and $\rho_\om^R<\rho_\om^{M_i}$.
The \emph{reverse extended model dropdown sequence}\index{reverse extended model dropdown sequence} $\redd^N(M)$\index{$\redd$} of $(N,M)$ is 
$\left<M_{k-i}\right>_{i\leq k}$
where $\left<M_i\right>_{i\leq k}=\emdd^N(M)$.
\end{dfn}

\begin{rem}
 Let $\left<M_i\right>_{i\leq k}=\redd^N(M)$ where $M\pins N$.
 Let $0<i<k$, so $M\pins M_i\pins N$.
 Clearly 
 $M_i$ is properly projecting and $\projdeg(M_i)\geq 0$.
Suppose $N$ is $n$-sound and $M$ is 
active and $\kappa<\nu(F^M)$. Let $(R,r)=\psegdeg^{N,n}(M,\kappa)$.
Note that $R=M_i$ for some $i\leq k$, $\kappa<\rho_0^M$, $0\leq r$, $\kappa<\rho_r^R$ 
and $R||(\kappa^+)^R=M||(\kappa^+)^M$; moreover, if $r=\om$ then $R=N$ and $n=\om$.
\end{rem}

\begin{dfn}\label{dfn:q-neat}
Let $Q,R$ be premice,  $\varphi:Q^\sq\to R^\sq$ and $q\in\{-1\}\un\om$. We say $\varphi$ is
\emph{$q$-neat}\index{neat}\index{$q$-neat}
iff
\begin{enumerate}[label=--]
 \item $Q,R$ are $\om$-sound,
 \item $\projdeg(Q)=\projdeg(R)=q$,
 \item if $q\geq 0$ then
$\varphi$ is a $\nu$-preserving near $q$-embedding, $\varphi(\rho_\om^Q)\geq\rho_\om^R$
and $\varphi``\rho_\om^Q\sub\rho_\om^R$ (note $\rho_\om^Q<\rho_0^Q$ because $q\geq 
0$, so $\rho_\om^Q\in\dom(\varphi)$),
\item if $q=-1$ then $\varphi$ is non-$\nu$-low $\Sigma_0$-elementary.\qedhere
\end{enumerate}
\end{dfn}

\begin{dfn}\label{dfn:copyseg}
Let 
$M,N$ be premice and $\pi:M\to N$ be a weak $0$-embedding.
For $Q\ins M$ we define
\index{$\copyseg$}\index{$\copymap$}\index{$\exitcopyseg$}\index{$\exitcopymap$}
\begin{IEEEeqnarray*}{rClcl}
    R&=&\copyseg^{\pi}(Q)&\text{ such that }&R\ins N,\\
 \sigma&=&\copymap^{\pi}(Q)&\text{ s.t. }&\sigma:Q^\sq\to R^\sq,\\
 R'&=&\exitcopyseg^{\pi}(Q)&\text{ s.t. }&R'\ins R,\\
 \sigma'&=&\exitcopymap^{\pi}(Q)&\text{ s.t. }&\sigma':Q^\sq\to(R')^\sq;\end{IEEEeqnarray*}
\noindent moreover, $\sigma,\sigma'$ have the same graphs,
and $R'\pins R$ iff $Q=M$ and $\pi$ is $\nu$-low.

\begin{case} $Q=M$.
 
Then $R=N$ and $\sigma=\sigma'=\pi$.
If $\pi$ is non-$\nu$-low then $R'=N$.
If $\pi$ is $\nu$-low then let
$R'\pins R$ be as in \ref{lem:nu-low_weak-0} (so $F^{R'}=F^R\rest\psi_\pi(\nu^M)$).
\end{case}

\begin{case} $Q\pins M$.

We will set $R'=R$ and $\sigma'=\sigma$, so it suffices to define $R,\sigma$.

\begin{scase}\label{scase:Q_in_cppsegs(M)} $Q\in\cardprojpropsegs(M)$.
 
Let $q=\projdeg(Q)$. We will choose $R\in\cardprojpropsegs(N)$
and
$\sigma:Q^\sq\to R^\sq$
a $\Sigma_1$-elementary $q$-neat embedding with $\pi\rest\rho_\om^Q\sub\sigma$
(note that $\rho_\om^Q\sub\dom(\pi)$).
Moreover, we will have $\projdeg(R)=q$,
and  $Q$ is properly projecting iff
$R$ is properly projecting.

Let $R^+=\psi_\pi(Q)$ and $\sigma^+=\psi_\pi\rest Q^\sq$,
so $\sigma^+:Q^\sq\to R^\sq$ is fully elementary and $\pi\rest\rho_\om^Q\sub\sigma^+$.

\begin{sscase} $Q\pins M^\sq$ or $\pi$ is non-$\nu$-high.

Equivalently, $R^+\pins N$. We set $R=R^+\pins N$ and $\sigma=\sigma^+$.
\end{sscase}

\begin{sscase}\label{sscase:R_neq_R^+,q=-1} $Q\npins M^\sq$,
$\pi$ is $\nu$-high, $q=-1$ and $Q$ is properly projecting (so $Q$ is type 3).

Then $M$ is also type 3,
and $\nu^Q=\rho_\om^Q=\nu^M$. Let $R\pins N$
be such that $F^R=F^{R^+}\rest\nu^N$;
this exists by the ISC and because $\nu^N$ is an $N$-cardinal.
Note that $\nu(F^R)=\nu^N$
and $\sigma^+``Q^\sq\sub R^\sq$.
Let $\sigma:Q^\sq\to R^\sq$ have graph $\sigma=\sigma^+$.
Note that $\sigma$ is $\Sigma_1$-elementary, as $\sigma^+$ is fully elementary
and $R^\sq\elem_0(R^+)^\sq$.
So $\sigma$ is not $\nu$-low. So $R,\sigma$ are as desired.
\end{sscase}
\begin{sscase}
 $Q\npins M^\sq$,
 $\pi$ is $\nu$-high, $q=-1$ and $Q$ is non-properly-projecting.

So $M$ is type 3 and $Q=M|\nu^M$.
We set $R=N|\nu^N$ and $\sigma:Q\to R$
is just $\sigma=\pi\rest Q=\psi_\pi\rest Q$.
Much as in Subsubcase \ref{sscase:R_neq_R^+,q=-1} above,
$\sigma$ is $\Sigma_1$-elementary
and everything else is clear.
\end{sscase}
\begin{sscase}\label{sscase:R_neq_R^+,q_geq_0} $Q\npins M^\sq$, $\pi$ is $\nu$-high, $q\geq 0$.

So $Q$ is properly projecting. Set
\[ R=\cHull_{q+1}^{R^+}(\nu^N\un\pvec_{q+1}^{R^+}) \] 
and $\varphi:R^\sq\to(R^+)^\sq$ be the uncollapse.
Because $\sigma^+$ is fully elementary, it preserves all standard parameters and solidity 
witnesses. So by choice of $q$, it follows that $\rg(\sigma^+)\sub\rg(\varphi)$.
And note that  $R$ is fully sound, $\projdeg(R)=q$,
\[ \rho_\om^{R}=\rho_{q+1}^R=\nu^N<\rho_q^R,\]
$\varphi$ is a near $q$-embedding which preserves $p_{q+1}$, and $\varphi\rest\nu^N=\id$.
So by condensation,
$R\pins R^+$, so $R\pins N$.
Now set $\sigma=\varphi^{-1}\com\sigma^+$.
Then $\sigma$ is $\nu$-preserving because
$\sigma^+=\varphi\com\sigma$, and
$\varphi,\sigma$ are near $0$-embeddings, hence not $\nu$-low, and $\sigma^+$ is fully 
elementary, hence $\nu$-preserving. The rest is clear. This completes this subsubcase and subcase.
\end{sscase}
\end{scase}

\begin{scase}
 $Q\notin\cardprojpropsegs(M)$.

Let $\left<M_i\right>_{i\leq k}=\redd^M(Q)$, so $k\geq 2$, $M_0=M$ and $M_k=Q$. We have 
$M_1\in\cardprojpropsegs(M)$. Let $R_1=\copyseg(M_1,\pi)$
and $\sigma_1=\copymap(M_1,\pi)$, so $R_1\pins N$. Let $q=\projdeg(M_1)$. Note that $q\geq 0$
and
$\sigma_1:M_1^\sq\to R_1^\sq$
is $q$-neat, so $\sigma_1$ is a $\nu$-preserving near $0$-embedding.
We have $Q\pins M_1$. Set
\[ R=\copyseg^{\sigma_1}(Q)=\psi_{\sigma_1}(Q)\text{ and }
\sigma=\copymap^{\sigma_1}(Q)=\psi_{\sigma_1}\rest Q^\sq.\qedhere\]
\end{scase}
\end{case}

We also write $\copyseg(\pi,Q)=\copyseg^\pi(Q)$
and define $\copyseg^\pi(F^Q)=F^{\copyseg^\pi(Q)}$, etc.
\end{dfn}

\begin{dfn}\label{dfn:copy_redd}
 Let $M,N$ be premice and $\pi:M^\sq\to N^\sq$ a weak $0$-embedding.
 Let $e\ins M$ be active. Let $\left<M_i\right>_{i\leq k}=\redd^M(e)$.
 Define\index{$\copyredd$}\index{$\copyseg$}\index{$\copymapredd$}
 \[ \copyredd^{\pi}(e)=
 \left<\copyseg^{\pi}(M_i)\right>_{i\leq k}, \]
 \[ \copymapredd^{\pi}(e)=\left<\copymap^{\pi}(M_i)\right>_{i\leq k}.\qedhere\]
\end{dfn}

\begin{lem}
 Let $M,N$ be premice, $\pi:M\to N$ be a non-$\nu$-high weak $0$-embedding.
 Let $Q\pins M$. Then $\copyseg^{\pi}(Q)=\psi_\pi(Q)$ and
 \[ \copymap^{\pi}(Q)=\psi_\pi\rest Q^\sq:Q^\sq\to\psi_\pi(Q)^\sq \]
 is fully elementary.
\end{lem}
\begin{lem}
 Let $M,N$ be premice, $\pi:M\to N$ be a weak $0$-embedding, and $e\pins M$ be active.
 Let
 \[ \left<M_i\right>_{i\leq k}=\redd^M(e), \]
 \[ \left<R_i\right>_{i\leq k}=\copyredd^\pi(e),\]
 \[ \left<\sigma_i\right>_{i\leq k}=\copymapredd^\pi(e).\]
 Then:
\begin{enumerate}[label=--]
\item $\left<R_i\right>_{i\leq k}=\redd^N(R_k)$ and $\sigma_0=\pi$.
\item $\all i>0$, $M_i,R_i$
properly project and $\projdeg(M_i)=\projdeg(R_i)\eqdef m_i$.
\item $\all i>0$, $\sigma_i:M_i^\sq\to R_i^\sq$ is $m_i$-neat and $\Sigma_1$-elementary.
\item $\sigma_i\rest\rho=\sigma_{i+1}\rest\rho$ where 
$\rho=\rho_\om^{M_{i+1}}=\rho_{m_{i+1}+1}^{M_{i+1}}$.
\tu{(}Note that $M_{i+1}\pins M_i$ 
and $\rho$ is a cardinal of $M_i$, so
$\rho\sub\dom(\sigma_i)$.\tu{)}
\item $\exitcopyseg^{\pi}(e)=\copyseg^{\pi}(e)=R_k$, $\exitcopymap^{\pi}(e)
=\copymap^{\pi}(e)=\sigma_k$.
\end{enumerate}
\end{lem}

\begin{dfn}\label{dfn:neat_copy}
 Let $\pi:M\to N$ be a weak $m$-embedding.
 
 Let $\Sigma$ be a (partial) iteration strategy for $N$.
 Then the \emph{neat $\pi$-pullback} of $\Sigma$
 is the (partial) strategy given by copying
 trees $\Tt$ on $M$ to trees $\Uu$ on $N$ via $\Sigma$,
 with copy maps $\pi_\alpha:M^\Tt_\alpha\to M^\Uu_\alpha$,
 starting with $\pi_0=\pi$, produced via the Shift Lemma as usual,
 and with $E^\Uu_\alpha=\exitcopyseg(\pi_\alpha,E^\Tt_\alpha)$.

 We write $\copyseg(\pi,\Tt)$ for the putative tree $\Uu$
 given by copying a tree $\Tt$ as above, if this exists.
\end{dfn}
\begin{rem}
 It is straightforward to see that this gives a sensible copying process.
 The trees $\Tt$ and $\Uu$ always have the same length and tree order
 (unlike the copying construction for Mitchell-Steel indexing, when
 one handles $\nu$-high embeddings
 in the manner described in \cite{mim}
 and sketched in \cite{outline}).
\end{rem}

\begin{lem}\label{lem:nu-pres_copy_map}
Let $\pi:M\to N$ be a weak $m$-embedding.
Let $\Tt,\Uu$ be $m$-maximal on $M,N$ respectively
with  $\Uu=\copyseg(\pi,\Tt)$.
Let $\pi_\alpha:M^\Tt_\alpha\to M^\Uu_\alpha$ be the resulting copy map.
Then:
\begin{enumerate}
\item\label{item:pi_alpha_is_weak_deg-embedding} $\pi_\alpha$ is a weak $\deg^\Tt(\alpha)$-embedding.
 \item\label{item:drop_nu-pres} If $[0,\alpha]_\Tt\inter\dropset_\Tt\neq\emptyset$
 then $\pi_\alpha$ is a $\nu$-preserving near $\deg^\Tt(\alpha)$-embedding.
 \item\label{item:nu-preserving_propagates} If $\pi$ is $\nu$-preserving
 then $\pi_\alpha$ is $\nu$-preserving.
  \item\label{item:near_embedding_propagates} If $\pi$ is a near $m$-embedding
 then $\pi_\alpha$ is a near $\deg^\Tt(\alpha)$-embedding.
 \item\label{item:reg_nu-pres_etc} Suppose $M$ is type 3, $\nu^M$ is $M$-regular
 and $[0,\alpha]_\Tt\inter\dropset^\Tt=\emptyset$. Then
  $\pi_\alpha$ is $\nu$-preserving/high/low
 iff $\pi$ is $\nu$-preserving/high/low.
 \item\label{item:when_pi_nu-low} Suppose $M$ is type 3 and that if $\pi$ is $\nu$-low
 then $\nu^M$ is $M$-singular. Then for all $\alpha$,
 if $\pi_\alpha$ is $\nu$-low then $\nu^{M_\alpha}$ is $M_\alpha$-singular.
\end{enumerate}
\end{lem}
\begin{proof}
Part \ref{item:pi_alpha_is_weak_deg-embedding}:
This is essentially as usual.

 Parts \ref{item:drop_nu-pres}--\ref{item:near_embedding_propagates}: For $\nu$-preservation, see the proof of \cite[3.13]{recon_con} ***could transfer that proof here.
 For near embeddings, note that
 the proof of \cite{fs_tame}
 goes through with the modified copying algorithm (that is, using $\Uu=\copyseg(\pi,\Tt)$).

 Part \ref{item:reg_nu-pres_etc}: By \ref{lem:nu-pres_it_map}, $i^\Tt_{0\alpha}$ is $\nu$-preserving.
 
 Suppose $\pi$ is $\nu$-preserving.
 Then $\nu^N$ is also $N$-regular,
 so by \ref{lem:nu-pres_it_map}, $i^\Tt_{0\alpha},i^\Uu_{0\alpha}$ are both $\nu$-preserving,
 so by commutativity, $\pi_\alpha$ is also.
 
 If $\pi$ is $\nu$-low then since $i^\Tt_{0\alpha}$ is $\nu$-preserving
 and by commutativity, $\pi_\alpha$ is $\nu$-low.
 
 If $\pi$ is $\nu$-high then since $i^\Uu_{0\alpha}$ is non-$\nu$-low,
 $\pi_\alpha$ is $\nu$-high.

 Part \ref{item:when_pi_nu-low}: Let $\alpha<\lh(\Tt)$; by part \ref{item:drop_nu-pres},
 we may assume that $[0,\alpha]^\Tt\cap\dropset^\Tt=\emptyset$.
 So if $\nu^M$ is $M$-singular then by Lemma \ref{lem:nu-pres_it_map}, $\nu^{M_\alpha}$ is $M_\alpha$-singular. So we may assume $\nu^M$ is $M$-regular. So by hypothesis, $\pi$ is non-$\nu$-low. But then by part \ref{item:reg_nu-pres_etc}, $\pi_\alpha$ is also non-$\nu$-low, which suffices.
\end{proof}

\begin{rem}
 Note that in the context of the lemma above, if $\pi$ is $\nu$-preserving,
 then because we maintain that every $\pi_\alpha$ is $\nu$-preserving,
 $\copyseg(\pi,\Tt)$
 is actually produced by the usual algorithm for the copying construction;
 that is, if $E^\Tt_\alpha=F^{M^\Tt_\alpha}$ then $E^\Uu_\alpha=F^{M^\Uu_\alpha}$,
 and otherwise $E^\Uu_\alpha=\psi_{\pi_\alpha}(E^\Tt_\alpha)$.
\end{rem}

\section{Resurrection and iterability for $\CC$}\label{sec:CC_iterability}
The iterability proof for models $N_\alpha$ of $\CC$ is an adaptation of the usual one;
we must of course deal with extenders whose ancestor
is added at a stage $\alpha$ with $t_\alpha\in\{1,2,3\}$.
For $t_\alpha=1,3$ it is the usual process;
with respect to $t_\alpha=3$, one just has to observe that
things work out fine structurally.
For $t_\alpha=2$ we need to insert extra ultrapowers
into the lift tree, in the style of \cite[\S5]{premouse_inheriting}
(though here things are slightly simpler).
This will verify that the models $N_\alpha$ are iterable (assuming the background premouse $Y$ is).
The algorithm is just a slight variant
of that in \cite{premouse_inheriting}
(which is a variant of the usual one),
but we will give the algorithm in detail,
not just for completeness but more
because we need to set up notation and terminology
to describe the details,
because they will be fundamental in the construction of scales to come.

\subsection{Lifting and resurrection for $m$-maximal trees}

\begin{rem}[Lifting and resurrection process]

A tree $\Tt$ on $N_\alpha$ will lift to a padded tree $\Uu$ on $Y$, as follows.
The nodes of $\Uu$ are a subset of $\lh(\Tt)\cross\om$,
such that whenever $(\alpha,n+1)$ is a node, then so is $(\alpha,n)$,
but for each $\alpha$ there are only finitely many $n$ such that $(\alpha,n)$ is a node.
If $\alpha<_\Tt\beta$ then $(\alpha,n)<_\Uu(\beta,0)$ for some $n$,
and if $n>0$ and $\alpha=\pred^\Tt(\beta)$ then $\Tt$ drops in model at $\beta$.
Fixing some $\beta<\lh(\Tt)$, let $n$ be largest such that $(\beta,n)$ is a node.
If $i+1\leq n$ then $(\beta,i)=\pred^\Uu(\beta,i+1)$,
and if $E^\Uu_{\beta i}\neq\emptyset$ then $E^{\Uu}_{\beta i}$ is an order $0$ measure
 used as a ``$t=2$'' background extender in $\CC=\CC^{M^\Uu_{(\beta,i)}}$ (that is, for some $\gamma$
we have $t^\CC_\gamma=2$ and $E^\Uu_{\beta i}=F^\CC_\gamma$).
The ultrapowers by these order $0$ measures facilitate resurrection.
And $E^\Uu_{\beta n}$ is the lift of $E^\Tt_\beta$.
We will have $\crit(E^\Uu_{\beta i})<\crit(E^\Uu_{\beta,i+1})$ for $i+1<n$.
\end{rem}

\begin{dfn}
 A premouse $Y$ is \emph{$(x,\alpha,n)$-good}\index{$(x,\alpha,n)$-good}
 iff:
 \begin{enumerate}[label=--]
  \item $x\in\RR^Y$, $\alpha$ is a limit ordinal such that $\CC^Y\rest(\alpha+1)$ exists and $t^Y_\alpha\neq 2$,
    \item $Y$ is $(x,\beta)$-iterability-good for all $\beta<\alpha$,
\item $N^Y_\alpha$ is $n$-good, and
\item if $\alpha=\OR^Y$ and $n>0$ then $Y$ is non-small and $n$-sound
 and $\rho_n^Y\geq\delta^Y$ and $N|\delta^Y$ is $(Y,n)$-P-good. ***?
 (Should mean that $N|\delta^Y$ is an P-base for $Y$,
 $N=\mathscr{P}^Y(N|\delta^Y)$ exists, $N$ is $n$-sound and for $i\leq n$, 
$\rho_i^N=\rho_i^Y$ and $p_i^N=p_i^Y$,
 and have the appropriate translation of fine structure. So everything we want before we consider 
$\rSigma_{n+1}$. I think these things should now be following from the lemmas giving such properties...)
\item  $Y$ is $(0,\om)$-iterable (that is, for finite length $0$-maximal trees).\qedhere
\end{enumerate}
\end{dfn}
\begin{dfn}\label{dfn:prodseg}Let $Y$ be a $(\xi,y,0)$-good premouse.

Let $S\in\cardprojpropsegs(N^Y_\xi)$.
Then $\prodstage^Y(S)$\index{$\prodstage$} (\emph{production stage}) denotes the unique $\beta<\xi$
with $\core_\om(N^Y_\beta)=\core_\om(S)$ (cf.~\ref{lem:projections_to_cardinals}(\ref{item:prod_stage})).

We define the \emph{production segment} 
$\prodseg^Y_y(N_\alpha)=\prodseg^Y_y(\alpha)$ of $N_\alpha$, for $\alpha\leq\xi$, by setting $\prodseg^Y_y(\alpha)=$
\begin{enumerate}[label=--]
\item $Y|\om$ if $\alpha=\om$,
\item $\J(\prodseg^Y_y(\beta))$ if $\alpha=\beta+\om$ and $t_\alpha=0$,
\item $\stack_{\beta<\alpha}\prodseg^Y_y(\beta)$ if $\alpha$ is a limit of limits and $t_\alpha=0$,
\item $Y|\lh(F^Y_\alpha)$ if $t_\alpha=1$,
\item $Y|\lh(D_{\kappa0}^Y)$ where $\kappa=\card^Y(\alpha)$ if $t_\alpha=2$,
(***this is probably unnatural...should probably instead be
the production segment of the resurrected production stage as computed in $\Ult(Y,D_{\kappa 0}^Y)$;
thus production segments would not all be segments of $Y$
but segments of finite iterates).
\item $Y|\alpha$ if $t_\alpha=3$.
\end{enumerate}
Note that if $\alpha<\beta$ then $\prodseg^Y_y(\alpha)\ins^*\prodseg^Y_y(\beta)$,
and $\prodseg^Y_y(\alpha)=\prodseg^Y_y(\beta)$ iff
[$t_\alpha=2$ and $\beta\leq\chi^Y_\kappa$ where
$\kappa=\card(\alpha)$]. Note that $\alpha\leq\OR(\prodseg^Y_y(\alpha))$ for all $\alpha$.
\end{dfn}
\begin{dfn}\label{dfn:modres}
Let $Y$ be $(x,\alpha,n)$-good and $N=N^Y_{\alpha n}$.
 Let $M$ be $n$-sound and $\pi:M^\sq\to N^\sq$ be a weak $n$-embedding.
 We define the \emph{model resurrection} function\index{$\modres$}
\[ \modres=\modres^Y_{\alpha n\pi}:\segs(M)\to V \]
associated to $(Y,\alpha,n,\pi)$ (and implicitly associated also to $M$).

Let $S\ins M$. Then $\modres(S)$ will specify, relative to $\pi$, a 
natural ancestor $A$ for $S$ in $\CC^Y$, 
together with a resurrection map $\psi:S^\sq\to A^\sq$, and further associated objects.
Because the backgrounding for $\CC^Y$ uses extenders from $\es^{Y'}$ for iterates $Y'$
of $Y$, the resurrection process will involve forming such iterates.
(For the present we will take $A$ to be the core ($n$-core if $S=M$; $\om$-core if $S\pins M$)
of some $N^{Y'}_{\alpha'}$, but we will later also consider $N^{Y'}_{\alpha'm}$ for various 
$m$.)

Let $\left<M_i\right>_{i\leq k}=\redd^M(S)$. 
So $M_0=M$ and $M_k=S$.
We will set\index{$\Psi_i$}
 \[ \modres(S)=(\Vv,\left<\Psi_i\right>_{i\leq k}) \]
where $\Vv$ is a padded $0$-maximal length $k+1$ nowhere-dropping linear iteration tree on $Y$
which uses only order $0$ measures, and writing $Y_i=M^\Vv_i$, then $\Psi_i$ will specify an 
ancestor $A_i$
for $M_i$, with $A_i\in Y_i$ or $A_i=N$ (the latter only when $i=0$), together with an associated lifting 
map $\psi_i:M_i^\sq\to A_i^\sq$.
For $i\leq k$ we will have
\[ \Psi_i=(\beta_i,A_i,\psi_i,\alpha_i,d_i). \]

We set $\beta_0=\emptyset$ (or ``undefined''), $\alpha_0=\alpha$, $d_0=n$, $A_0=N=N^{Y_0}_{\alpha_0d_0}$ and
\[ \psi_0=\pi:M_0^\sq\to A_0^\sq. \]
Note that if $S=M$ then $k=0$ and we are done.

Now suppose we have defined $\Psi_i$ where $i<k$.

If $0<i$ we will have $d_i=\om$ and
$\beta_i<\lh(\CC^{Y_{i-1}})$ and $\alpha_i<\lh(\CC^{Y_i})$
and $\beta_i<\alpha_{i-1}$ and $\alpha_i<i^\Vv_{i-1,i}(\alpha_{i-1})$ and $t_{\alpha_i}^{Y_i}\neq 
2$ and
\[ A_i\eqdef N_{\beta_{i}\om}^{Y_{i-1}}=N_{\alpha_{i}\om}^{Y_{i}} \]
and
$\psi_{i}:M_{i}^\sq\to A_{i}^\sq$
is a $\Sigma_1$-elementary $m_i$-neat embedding,
where $m_i=\projdeg(M_i)$, and letting
\[ \psi_{i-1}^*=\tau^{Y_{i-1}}_{\alpha_{i-1}d_{i-1}0}\com\psi_{i-1}:M_{i-1}^\sq\to 
(N_{\alpha_{i-1}0}^{Y_{i-1}})^\sq, \]
we will have
$\psi_{i-1}^*\rest\rho_\om^{M_{i}}=\psi_{i}\rest\rho_\om^{M_{i}}$,
and note that $\psi_{i-1}^*$ is a weak $0$-embedding (near $0$-embedding if $i>1$ or $n>0$).

Set $d_{i+1}=\om$ and
$\psi_i^*=\tau^{Y_i}_{\alpha_i d_i0}\com\psi_i$, so $\psi_i^*:M_i^\sq\to(N_{\alpha_i 0}^{Y_i})^\sq$ is
a weak $0$-embedding (near $0$-embedding if $i>0$ or $n>0$).
Now $M_{i+1}\in\cardprojpropsegs(M_i)$.
Set
\[ A_{i+1}=\copyseg(M_{i+1},\psi_i^*), \]
\[ \psi_{i+1}=\copymap(M_{i+1},\psi_i^*). \]
So $\psi_i^*\rest\rho_\om^{M_{i+1}}=\psi_{i+1}\rest\rho_\om^{M_{i+1}}$. We have 
$A_{i+1}\in\cardprojpropsegs(N^{Y_i}_{\alpha_i0})$. Set
\[ \beta_{i+1}=\prodstage^{Y_i}_{\alpha_i}(A_{i+1}), \]
so $\beta_{i+1}<\alpha_i$ and $A_{i+1}=N_{\beta_{i+1}\om}^{Y_i}$.

If $t_{\beta_{i+1}}^{Y_i}\neq 2$ then we set $E^\Vv_i=\emptyset$ and $\alpha_{i+1}=\beta_{i+1}$.
If $t_{\beta_{i+1}}^{Y_i}=2$ then we set $E^\Vv_i=D_{\kappa0}^{Y_i}$ where 
$\kappa=\card^{Y_i}(\beta_{i+1})$, and set
\[ \alpha_{i+1}=\prodstage^{Y_{i+1}}_{i^\Vv_{i,i+1}(\alpha_i)}(A_{i+1}). \]
(Note that
$N_\kappa^{Y_i}=N_{\alpha_i}^{Y_i}|\kappa\pins_\card N_{\alpha_i}^{Y_i}$
and since $t^{Y_i}_{\alpha_i}\neq 2$,
letting $\chi=\chi_{\kappa}^{Y_i}=\kappa^{+i^{Y_i}_{D}(N_\kappa^{Y_i})}$
where $D=D_{\kappa0}^{Y_i}$,
then $\chi=\kappa^{+N^{Y_i}_{\alpha_i}}$
and $N^{Y_i}_{\alpha_i}|\chi=i^{Y_i}_{D}(N^{Y_i}_\kappa)|\chi$,
and $A_{i+1}\pins N^{Y_i}_{\alpha_i}|\chi$, and since $A_{i+1}\in\cardprojpropsegs(N^{Y_i}_{\alpha_i})$,
therefore $\rho_\om^{A_{i+1}}=\kappa$.
Also 
\[ i^\Vv_{i,i+1}(N_\kappa^{Y_i})=i^\Vv_{i,i+1}(N_{\alpha_i}^{Y_i})|i^\Vv_{i,i+1}(\kappa), \] 
so $A_{i+1}\in\cardprojpropsegs(N^{Y_{i+1}}_{i^\Vv_{i,i+1}(\alpha_i)})$.)
Note that independent of $t_{\beta_{i+1}}^{Y_i}$, we have
\[ A_{i+1}=N^{Y_{i+1}}_{\alpha_{i+1}\om}\text{ and } t_{\alpha_{i+1}}^{Y_{i+1}}\neq 2\text{ and 
}\alpha_{i+1}\leq i^\Vv_{i,i+1}(\beta_{i+1})<i^\Vv_{i,i+1}(\alpha_i).\]
This completes the definition.

Define also the \emph{resurrection tree}, \emph{model resurrection map}, \emph{resurrection production stage} and
\emph{resurrection length} functions by:
\begin{enumerate}[label=--]
 \item $\restree^Y_{\alpha n\pi}(S)=\Vv$,\index{$\restree$}
\item $\modresmap^Y_{\alpha n\pi}(S)=\psi_k$,\index{$\modresmap$}
\item $\resprodstage^Y_{\alpha n\pi}(S)=\alpha_k$\index{$\resprodstage$} and
\item $\resl^Y_{\alpha n\pi}(S)=k$.\index{$\resl$}
\end{enumerate}

Also if $Y$ is $z$-sound and $(z,\om)$-iterable then\index{$\restree^{Y,z}$}
\[ \restree^{Yz}_{\alpha n\pi}(S)\] denotes
the $z$-maximal tree $\Vv'$ on $Y$ which uses the same extenders as does $\Vv$;
we point out below that this is well-defined.
Analogously,
\[ \modres^{Yz}_{\alpha n\pi}(S)=(\Vv',\left<\Psi_i\right>_{i\leq k}).\qedhere\]
\end{dfn}
The following lemma is straightforward to verify:

\begin{lem}
Let $Y$ be $(x,\alpha,n)$-good,
$N=N^Y_{\alpha n}$ and
\[\pi:M^\sq\to N^\sq\] a weak $n$-embedding,
$S\ins M$ and $\left<M_i\right>_{i\leq k}=\redd^M(S)$
and $\Vv=\restree^Y_{\alpha n\pi}(S)$. Then:
\begin{enumerate}[label=--]
 \item $\Vv$ is a padded, length $k+1$, $0$-maximal, nowhere dropping, linear iteration tree on $Y$ 
which uses only order $0$ measures.
\item For each $i+1\leq j+1\leq k$ with $E^\Vv_i\neq\emptyset\neq E^\Vv_j$,
we have
\[ \crit(E^\Vv_{i})<\crit(E^\Vv_j)<i^\Vv_{ij}(\crit(E^\Vv_i)).\]
\item $\modres^Y_{\alpha n\pi}(M_j)=(\Vv\rest(j+1),\left<\Psi_i\right>_{i\leq j})$
for each $j\leq k$.
\item Suppose $Y$ is $z$-sound and $(z,\om)$-iterable. Then $\Vv'=\restree^{Y,z}_{\alpha n\pi}(S)$ is 
well-defined. In fact, suppose $E^\Vv_i\neq\emptyset$ and $i$ is least such. 
Let $Y'_l=M^{\Vv'}_l$.
Let $y\leq z$ be largest such that 
$\kappa=\crit(E^\Vv_i)<\rho_y^Y$, so $y=\deg^{\Vv'}(i+1)$.
Let $i<l\leq k$. Then $y=\deg^{\Vv'}(l)$,
\[ \lambda\eqdef i^\Vv_{il}(\kappa)=i^{\Vv'}_{il}(\kappa)<\rho_y^{Y'_l},\]
\[ Y_l|(\lambda^+)^{Y_l}=Y'_l|(\lambda^+)^{Y'_l}, \]
and $i^\Vv_{il}\rest\pow(\kappa)=i^{\Vv'}_{il}\rest\pow(\kappa)$,
and in particular, $\CC^{Y_l}\rest(\lambda+1)=\CC^{Y'_l}\rest(\lambda+1)$.
\end{enumerate}
\end{lem}

\begin{dfn}\label{dfn:degres}
 Let $Y$ be $(x,\alpha,m)$-good, $N=N^Y_{\alpha m}$,
$\pi:M^\sq\to N^\sq$ a weak $m$-embedding.
 We define the \emph{resurrection}\index{resurrection function} function
 $\res^Y_{\alpha m\pi}$,\index{$\res$} with 
 \[ \dom(\res^Y_{\alpha m\pi})=\segdegs(M,m).\]
 
Let $(S,s)\ins(M,m)$. Then
\[ \res^Y_{\alpha m\pi}(S,s)=(\Vv,\left<\Psi_i\right>_{i\leq k},\psi) \]
where $(\Vv,\left<\Psi_i\right>_{i\leq k})=\modres^Y_{\alpha 
m\pi}(S)$, $\Psi_k=(\beta_k,A_k,\psi_k,\alpha_k,d_k)$, $Y_k=M^\Vv_k$,
and
\[ \psi=\tau^{Y_k}_{\alpha_k d_k s}\com\psi_k:S^\sq\to(N^{Y_k}_{\alpha_k s})^\sq.\]
Define $\resmap^Y_{\alpha m\pi}(S,s)=\psi$.\index{$\resmap$}

 We define the \emph{critical point resurrection} function
 $\critres^Y_{\alpha m\pi}$,\index{$\critres$} with
 \[ \dom(\critres^Y_{\alpha m\pi})=\{(e,\kappa)\mid e\ins M\ \&\ e\text{ is active }
 \&\ \kappa<\nu(F^e)\},\]
 and also the \emph{critical point resurrection map} function
  $\critresmap^Y_{\alpha m\pi}$,\index{$\critresmap$} with the same domain.

Let $e\ins M$ be active and $\kappa<\nu(F^e)$. Then
\[ \critres^Y_{\alpha m\pi}(e,\kappa)=\res^Y_{\alpha m\pi}(S,s)\]
and
 \[ \critresmap^Y_{\alpha m\pi}(e,\kappa)=\resmap^Y_{\alpha m\pi}(S,s)\]
where $(S,s)=\psegdeg^{M,m}(e,\kappa))$ (recall $(S,s)\ins(M,m)$ is the first level 
after $e$ projecting to $\kappa$ (so $\rho^S_{s+1}\leq\kappa<\rho_s^S$), or $(S,s)=(M.m)$ if there 
is no such).

Define the functions\index{$\critrestree$}\index{$\critrestree^{Y,z}$}\index{$\critresprodstage$}\index{$\critresl$}
\[ \critrestree^Y_{\alpha m\pi},\ \critrestree^{Y,z}_{\alpha m\pi},\ \critresprodstage^Y_{\alpha m\pi},\ \critresl^Y_{\alpha m\pi},\]
also with the same domain as $\critres^Y_{\alpha m\pi}$, analogously. So for example
\[ \critrestree^Y_{\alpha m\pi}(e,\kappa)=\restree^Y_{\alpha m\pi}(S) \]
where $(S,s)=\psegdeg^{M,m}(e,\kappa)$.
\end{dfn}

\begin{lem}\label{lem:when_psi_nu-low}
 Let $Y$ be $(x,\alpha,m)$-good, $N=N^Y_{\alpha m}$,
$\pi:M^\sq\to N^\sq$ a weak $m$-embedding.
Let $(S,s)\ins(M,m)$ and
$(\Vv,\left<\Psi_i\right>_{i\leq k},\psi)=\res^Y_{\alpha m\pi}(S,s)$.
Then:
\begin{enumerate}[label=--]
\item $t_\eta^{M^\Vv_k}\neq 2$ where $\eta=\resprodstage^Y_{\alpha m\pi}(S,s)$ and
 \item $\psi$ is $\nu$-low iff $s=m=0$, $S=M$ and 
$\pi$ is $\nu$-low \tu{(}so $k=0$ and $\psi=\pi$\tu{)}.
\end{enumerate}
\end{lem}
\begin{proof}
Because core embeddings are $\Sigma_1$-elementary, hence non-$\nu$-low,
and $m$-neat embeddings are non-$\nu$-low, this follows readily from the definitions.
\end{proof}

\begin{dfn}\label{dfn:exitresadd}
 Let $Y$ be $(x,\alpha,m)$-good, $N=N^Y_{\alpha m}$,
$\pi:M^\sq\to N^\sq$ a weak $m$-embedding.
We define the \emph{exit resurrection addendum} function\index{$\exitresadd$}
\[ \exitresadd=\exitresadd^Y_{\alpha m\pi},\]
with
$\dom(\exitresadd)=\{e\mid e\ins M\ \&\ e\text{ is active}\}$. Let $e\ins M$ be active.
Let
\[ \res^Y_{\alpha m\pi}(e,0)=(\Vv,\left<\Psi_i\right>_{i\leq k},\psi)\]
and $\Psi_k=(\cdot,\cdot,\cdot,\alpha_k,\cdot)$.
We will set $\exitresadd(e)=(\widetilde{\alpha},\widetilde{\psi})$ where 
$\widetilde{\alpha}\leq\alpha_k$, and
\[ \widetilde{\psi}:e^\sq\to(N^{Y_k}_{\widetilde{\alpha}0})^\sq \]
is a non-$\nu$-low weak $0$-embedding such that $\widetilde{\psi},\psi$
have the same graphs (but they possibly have different codomains).

If $\psi$ is non-$\nu$-low then
we set $\widetilde{\alpha}=\alpha_k$ and $\widetilde{\psi}=\psi$.

Now suppose $\psi$ is $\nu$-low, so by \ref{lem:when_psi_nu-low}, $e=M$ is type 3, $m=0$,
$\psi=\pi$, $k=0$. Let $\widetilde{\nu}=\psi_\pi(\nu^M)$,
so $\widetilde{\nu}<\nu^N$ and $\widetilde{\nu}$ is a cardinal of $N$.
So there is $\widetilde{N}\pins N$ with 
$F^{\widetilde{N}}=F^N\rest\widetilde{\nu}$,
and note $\widetilde{N}\in\cardprojpropsegs(N)$.
Since $N=N_{\alpha 0}$ we can therefore set
\[ \widetilde{\alpha}=\prodstage^Y_{\alpha}(\widetilde{N})<\alpha.\]
(And $\widetilde{\psi}$ has the same graph as does $\psi$,
but note they have different codomains. We verify in \ref{lem:exitresadd} below that 
$(\widetilde{\alpha},\widetilde{\psi})$ have the stated properties.)\end{dfn}

\begin{lem}\label{lem:exitresadd}
 Let $Y$ be $(x,\alpha,m)$-good, $N=N^Y_{\alpha m}$, $\pi:M^\sq\to N^\sq$ a weak 
$m$-embedding, and suppose that if $\pi$ is $\nu$-low then $\nu^M$ is $M$-singular. Let $e\ins M$ be 
active
 and $\kappa<\nu(F^e)$. Let
 \[ (S,s)=\psegdeg^{M,m}(e,\kappa) \]
 \tu{(}so $e\ins S\ins M$ and $\kappa<\rho_s^S$ and $\kappa<\rho_0^e$ and 
$e||(\kappa^+)^e=S||(\kappa^+)^S$\tu{)}.
 Let
 \[ (\Vv,\left<\Psi_i\right>_{i\leq k},\psi)=\res^Y_{\alpha m\pi}(e,0),\]
  \[ \alpha_k=\resprodstage^Y_{\alpha m\pi}(e,0),\]
  \[ (\wt{\alpha},\wt{\psi})=\exitresadd^Y_{\alpha m\pi}(e),\]
\[\varphi=\critresmap^Y_{\alpha m\pi}(e,\kappa),\]
\[ l=\critresl^Y_{\alpha m\pi}(e,\kappa)\]
and  $Y_i=M^\Vv_i$.
Then:
\begin{enumerate}
 \item $\psi:e^\sq\to(N_{\alpha_k0}^{Y_k})^\sq$ and $\wt{\psi}:e^\sq\to(N_{\wt{\alpha}0}^{Y_k})^\sq$
are weak $0$-embeddings with the same graph, and $\wt{\psi}$ is non-$\nu$-low.
\item $\varphi:S^\sq\to(\wt{N}_{\alpha_l s}^{Y_l})^\sq$ is a weak $s$-embedding.
\item $\wt{\psi}\rest(\kappa^+)^e=\varphi\rest(\kappa^+)^S$.
\item $t^{Y_k}_{\alpha_k},t^{Y_k}_{\wt{\alpha}}\in\{1,3\}$ and $t^{Y_l}_{\alpha_l}\neq 2$.
\end{enumerate} 
\end{lem}
\begin{proof} We just verify that $\wt{\psi}:e^\sq\to(N_{\wt{\alpha}0}^{Y_k})^\sq$ is a non-$\nu$-low weak 
$0$-embedding, and that $t^{Y_k}_{\wt{\alpha}}\in\{1,3\}$,
and omit the rest.

We have $t^{Y_k}_{\alpha_k}\neq 2$ and $Y_k=M^\Vv_k$, and $t^{Y_k}_{\alpha_k}\neq 0$ 
because $N^{Y_k}_{\alpha_k}$ is active.
So we may assume that $\psi$ is $\nu$-low. So by \ref{lem:when_psi_nu-low},
$e=M$ is type 3, $m=k=0$, and $\psi=\pi$.
Let
$\widetilde{N}=N^Y_{\widetilde{\alpha}\om}$ and
$\widetilde{\nu}=\nu^{\widetilde{N}}$.
We have $\widetilde{N}=N^Y_{\widetilde{\alpha}0}$
by the ISC
and because $\rho_\om^{\widetilde{N}}=\nu^{\widetilde{N}}$.
And because $\nu^M$ is $M$-singular, $\widetilde{\nu}$
is $N$-singular, hence $Y$-singular, hence non-$Y$-measurable, so 
$t^Y_{\wt{\alpha}}\neq 2$. And $t^Y_{\wt{\alpha}}\neq 0$ because $\wt{N}$ is active.
By \ref{lem:nu-low_weak-0}, $\widetilde{\psi}$ is a $\nu$-preserving weak $0$-embedding.
\end{proof}

\begin{lem}\label{lem:res_comm}
 Let $Y$ be $(x,\zeta,m)$-good, $N=N^Y_{\zeta m}$, $\pi:M^\sq\to N^\sq$ a weak 
$m$-embedding. Let $Y',\zeta',N',M',\pi'$ be likewise \tu{(}with the same $x,m$\tu{)}. Let
$\varrho:M\to M'$ be a $\nu$-preserving near $m$-embedding.
Suppose $Y,Y'$ are $y$-sound and let $\varsigma:Y\to Y'$ be a $\nu$-preserving near $y$-embedding.
Suppose
$\pi'\com\varrho=\varsigma\com\pi$.
Let $e\ins M$ be active and $e'=\copyseg(\varrho,e)$.

Let
\begin{enumerate}[label=--]
 \item  $\left<M_i\right>_{i\leq k}=\redd^{M}(e)$,
 \item $(\Vv,\left<\Psi_j\right>_{j\leq k},\psi)=\res^{Y y}_{\zeta m\pi}(e,0)$,
 \item $\Psi_j=(\beta_j,A_j,\psi_j,\alpha_j,d_j)$ and $Y_j=M^\Vv_j$ for $j\leq k$,
 \item $(\wtalpha,\wtpsi)=\exitresadd^{Yy}_{\zeta m\pi}(e)$.
\end{enumerate}
Define $\left<M'_j\right>_{j\leq k'}$, $\Vv'$, $\Psi'_j$, etc,
 analogously from $M',e',Y'$, etc.

 Suppose further that
 \begin{enumerate}[label=(\roman*)]
 \item\label{item:zeta<rho_0^Y}$\zeta<\rho_0^Y$ and $\zeta'<\rho_0^{Y'}$ and $\varsigma(\zeta)=\zeta'$.
 \end{enumerate}
 Then:
 \begin{enumerate}
  \item\label{item:redd_pres} $k=k'$,
  $M_0=M$, $M'_0=M'$, and $M'_j=\psi_{\varrho}(M_j)$ for each $j\in(0,k]$,
  \item\label{item:Vv'_is_copy_of_Vv} $\Vv'\rest(j+1)=\copyseg(\varsigma,\Vv\rest(j+1))$ for each $j\leq k$;
  let $\varsigma_j:Y_j\to Y'_j$ be the resulting copy map \tu{(}so $\varsigma_0=\varsigma$\tu{)},
  \item $\varsigma_j:Y_j\to Y'_j$ is a $\nu$-preserving near $\deg^\Vv(j)$-embedding,
  \item\label{item:varsigma_0_commutes_with_core_maps} For each $n\leq m$, we have $\varsigma_0\com\tau^{Y_0}_{\alpha_0 mn}=\tau^{Y'_0}_{\alpha_0'mn}\com\varsigma_0$,
 so
  \[ \varsigma_0\com\tau^{Y_0}_{\alpha_0 mn}\com\pi=\tau^{Y'_0}_{\alpha'_0mn}\com\pi'\com\varrho. \]
  \item\label{item:beta',A'_pres_sigma_j} $(\beta'_{j+1},A'_{j+1})=\varsigma_j(\beta_{j+1},A_{j+1})$ and $t^{Y'_j}_{\beta'_{j+1}}=t^{Y_j}_{\beta_{j+1}}$ for $j+1\leq k$,
  \item\label{item:psi'_comm_sigma_j} 
$\psi'_{j+1}\com\psi_\varrho\rest\core_0(M_{j+1})=\varsigma_{j}\com\psi_{j+1}$
   for $j+1\leq k$ (also cf.~part \ref{item:psi'_comm_sigma_j+1} below),
  \item\label{item:alpha',A'_pres_sigma_j+1} $(\alpha'_{j+1},A'_{j+1})=\varsigma_{j+1}(\alpha_{j+1},A_{j+1})$ for $j+1\leq k$,
  \item\label{item:psi'_comm_sigma_j+1} $\psi'_{j+1}\com\psi_\varrho\rest\core_0(M_{j+1})=\varsigma_{j+1}\com\psi_{j+1}$
  (note that the $\varsigma_j$
  in part \ref{item:psi'_comm_sigma_j}
  has been replaced by $\varsigma_{j+1}$, but otherwise the equation is the same),
  \item\label{item:core_sigma_j+1_comm} For each $n\leq\omega$, we have $\varsigma_{j+1}\com\tau^{Y_{j+1}}_{\alpha_{j+1}\om n}=\tau^{Y'_{j+1}}_{\alpha'_{j+1}\om n}\com\varsigma_{j+1}$,
  and hence,
  \[ \varsigma_{j+1}\com\tau^{Y_{j+1}}_{\alpha_{j+1}\om n}\com\psi_{j+1}=\tau^{Y'_{j+1}}_{\alpha'_{j+1}\om n}\com\psi'_{j+1}\com\varrho\rest \core_0(M_{j+1}). \]
  \item\label{item:psitilde_comm} Either:
  \begin{enumerate}[label=--]
  \item $\zeta<\rho_0^Y$ and $\wtalpha<\rho_0^{Y_k}$ and $\wtalpha'<\rho_0^{Y'_k}$ and $\varsigma_k(\wtalpha)=\wtalpha'$
  and
  \[ \varsigma_k\com\wtpsi=\wtpsi'\com\psi_\varrho\rest e^\sq,\]
  or
  \item $\zeta=\OR^Y$, $e=M$, $\wtpsi=\pi$, $\wtpsi'=\pi'$  and either  
  \begin{enumerate}[label=--]
  \item $\pi,\pi'$ are non-$\nu$-low, $\wtalpha=\zeta$, $\wtalpha'=\zeta'$, or
 \item   $\pi,\pi'$ are $\nu$-low, $\wtalpha<\rho_0^Y$ and $\wtalpha'<\rho_0^{Y'}$ and $\varsigma(\wtalpha)=\wtalpha'$,
 \end{enumerate}
  or
  \item $\zeta=\OR^Y$, $e\pins M$, $\wtalpha=\alpha_k$, $\wtalpha'=\alpha_k'$, $\wtpsi=\psi_k$ and $\wtpsi'=\psi'_k$.
  \end{enumerate}
 \end{enumerate}
 
 Now suppose instead of \ref{item:zeta<rho_0^Y} that
 \begin{enumerate}[label=(\roman*)]
 \setcounter{enumi}{1}
 \item\label{item:Y,Y'_large} $M,M',Y,Y'$ are large, $\zeta=\OR^Y$, $\zeta'=\OR^{Y'}$, $m=y$, and $\delta^M<\OR^e$.
\end{enumerate}
Then parts \ref{item:redd_pres}--\ref{item:psitilde_comm}
 hold after replacing $\varsigma=\varsigma_0,\ldots,\varsigma_k$ with $\psi_{\varsigma_0},\ldots,\psi_{\varsigma_k}$,
 but many things simplify;
 in particular, $\Vv,\Vv'$ are trivial,
 $\varsigma_0=\ldots=\varsigma_k$,
 $\left<A_i\right>_{i\leq k}=\redd^{N}(\copyseg(\pi,e))$
 and $\left<A'_i\right>_{i\leq k}=\redd^{N'}(\copyseg(\pi',e'))$,
 the core maps in parts \ref{item:varsigma_0_commutes_with_core_maps}  and \ref{item:core_sigma_j+1_comm}
 are identity maps,  and
 the $t$-values in part \ref{item:beta',A'_pres_sigma_j} all $=3$.
\end{lem}
\begin{proof}
Part \ref{item:redd_pres} follows immediately from the fact that $\varrho$ is a $\nu$-preserving
near $0$-embedding.

\begin{case} Hypothesis \ref{item:zeta<rho_0^Y} holds.
 
Parts \ref{item:Vv'_is_copy_of_Vv}--\ref{item:core_sigma_j+1_comm} are proved by, roughly, simultaneous induction on $j$
(we leave it to the reader to calibrate exactly what is proved at each induction step and in what order).
Most of the details are immediate, so we just make some remarks concerning some finer points.

The fact that $\varsigma_j$ is a $\nu$-preserving near $\deg^\Vv(j)$-embedding
follows from \cite{fs_tame} and \ref{lem:nu-pres_copy_map}, given 
part \ref{item:Vv'_is_copy_of_Vv}
at $j$.

Consider parts \ref{item:beta',A'_pres_sigma_j} and \ref{item:psi'_comm_sigma_j}.
For notational simplicity we assume $j=0$, but the general case is likewise.
We have $\alpha_0=\zeta<\rho_0^Y$ and $\alpha'_0=\zeta'=\varsigma(\alpha_0)<\rho_0^{Y'}$.
So, recalling $N=N^Y_{\alpha_0 m}$
and $N'=N^{Y'}_{\alpha_0 m}$,
\[ \varsigma_0\rest N^Y_{\alpha_0m}:N^Y_{\alpha_0m}\to N^{Y'}_{\alpha_0m}\text{ and }\varsigma_0\rest N^Y_{\alpha_0}:N^Y_{\alpha_0}\to N^{Y'}_{\alpha_0'} \]
are fully elementary, and in particular, $\nu$-preserving.
Let
\[ \wtpi=\tau^Y_{\alpha_0 m0}\com\pi\text{ and }\wtpi'=\tau^{Y'}_{\alpha'_0 m0}\com\pi'.\]
Then $\wtpi$ is $\nu$-preserving iff $\wtpi'$ is $\nu$-preserving,
by commutativity and because $\varrho$ and $\varsigma\rest N$ are $\nu$-preserving.
So if $\wtpi$ is $\nu$-preserving or $M_1\pins M^\sq$ 
then parts \ref{item:beta',A'_pres_sigma_j} and \ref{item:psi'_comm_sigma_j} easily follow.

So suppose that $\wtpi$ is $\nu$-high and $M_1\npins M^\sq$.
Note that $\wtpi'$ is also $\nu$-high, and $M'_1=\psi_\varrho(M_1)\npins(M')^\sq$ (as $\varrho$ is $\nu$-preserving).
We have $A_1=\copyseg^{\wtpi}(M_1)$ and $A'_1=\copyseg^{\wtpi'}(M'_1)$.
Note that Subsubcase \ref{sscase:R_neq_R^+,q=-1} or \ref{sscase:R_neq_R^+,q_geq_0}
of \ref{dfn:copyseg} attains. We assume \ref{sscase:R_neq_R^+,q_geq_0} attains;
\ref{sscase:R_neq_R^+,q=-1} is similar. So let $R^+,\varphi,q$ be defined as there with respect to $M_1,M,\wtpi,N^Y_{\alpha_0}$
(and note that $R$ there is $A_1$ and $\sigma$ there is $\psi_1$).
Let $(R')^+,\varphi',q'$ be likewise with respect to $M'_1$, etc. Clearly $q=q'$,
and $\psi_\varrho,\varsigma_0,\psi_{\wtpi},\psi_{\wtpi'}$ restrict to give fully elementary, commuting maps
between $M_1,M'_1,R^+,(R')^+$, and $\varsigma_0(R^+)=((R')^+)$ and $R^+=\psi_\wtpi(M_1)$, etc.
Also, $\varsigma_0(\nu^{N^Y_{\alpha_0}})=\nu^{N^{Y'}_{\alpha_0'}}$. But then clearly $\varsigma_0(R,\varphi)=(R',\varphi')$,
so
\[ \varsigma_0\com\psi_1=\psi_1'\com\psi_\varrho\rest M_1.\]
Since $\varsigma_0(R)=R'$, it also follows that $\varsigma_0(\beta_1)=\beta'_1$,
and therefore that $t^{Y}_{\beta_1}=t^{Y'}_{\beta'_1}$.

We have verified parts \ref{item:beta',A'_pres_sigma_j} and \ref{item:psi'_comm_sigma_j} in the example situation,
but in general it is very similar.

Now consider parts \ref{item:alpha',A'_pres_sigma_j+1} and \ref{item:psi'_comm_sigma_j+1};
again we assume $j=0$.
By parts \ref{item:beta',A'_pres_sigma_j} and \ref{item:psi'_comm_sigma_j}
we may assume that $t^{Y_0}_{\beta_1}=2$ (and hence $t^{Y'_0}_{\beta'_1}=2$).
Since $\varsigma_0(\beta_1)=\beta'_1$,
we have $\varsigma(E^\Vv_0)=E^{\Vv'}_0$, and letting $\kappa=\crit(E^\Vv_0)$,
we have $\kappa=\rho_\om^{A_1}$ and
\begin{equation}\label{eqn:varsigma_1,varsigma_0_agmt} \varsigma_1\rest\pow(\kappa)=\varsigma_0\rest\pow(\kappa)\text{, hence }
\varsigma_1(A_1)=A'_1\text{ and }\varsigma_1\rest A_1=\varsigma_0\rest A_1.\end{equation}
It follows that $\varsigma_1(\alpha_1)=\alpha'_1$, giving part \ref{item:alpha',A'_pres_sigma_j+1},
and part \ref{item:psi'_comm_sigma_j+1} follows from part \ref{item:psi'_comm_sigma_j} and line (\ref{eqn:varsigma_1,varsigma_0_agmt}).

This completes our remarks regarding parts \ref{item:Vv'_is_copy_of_Vv}--\ref{item:core_sigma_j+1_comm}.
Part \ref{item:psitilde_comm} follows easily using that $\pi$ is $\nu$-low iff $\pi'$ is $\nu$-low,
and that
$\psi_\varsigma(\psi_\pi(\nu^M))=\psi_{\pi'}(\nu^{M'})$
as $\varrho$ is $\nu$-preserving.
\end{case}
\begin{case} Hypothesis \ref{item:Y,Y'_large} holds.
 
Note that in this case,
we have $t^Y_\zeta=3=t^{Y'}_{\zeta'}$,
and since $\delta^M<\lh(e)$,
the entire sequences $\left<M_j\right>_{j\leq k}$
and $\left<M'_j\right>_{j\leq k}$
are above $\delta^M$ and $\delta^{M'}$ respectively. So everything under consideration is within P-construction parts of $\CC^Y$ and
and $\CC^{Y'}$, and so many things trivialize. We leave the details to the reader.\qedhere
\end{case}
\end{proof}

\begin{dfn}
 Let $M,N$ be $m$-sound and $\pi:M^\sq\to N^\sq$.
 We say that $\pi$ is \emph{$m$-okay}\index{-okay} iff $\pi$ is a weak $m$-embedding and:
 \begin{enumerate}
  \item Suppose $\pi$ is non-$\nu$-preserving. Then $m=0$. Let $\theta=\cof^M(\nu^M)$. Then $\theta<\nu^M$ and
  if $\pi$ is $\nu$-high then $\sup\pi``\theta\leq\cof^N(\nu^N)<\pi(\theta)$. 
  \item Suppose $m>0$. Let $\wt{\rho}$ be least such that either $\wt{\rho}=\rho_{m-1}^M$
  or $\pi(\wt{\rho})\geq\rho_{m-1}^N$. 
  Let $\kappa^M=\cof^{\bfrSigma_m^M}(\rho_{m-1}^M)$ and $\kappa^N=\cof^{\bfrSigma_m^N}(\rho_{m-1}^N)$.
   Then:
   \begin{enumerate}[label=--]
   \item $\kappa^M\notin[\rho_m^M,\wt{\rho})$.
   \item If $\kappa^M=\rho_{m-1}^M$ then $\kappa^N=\rho_{m-1}^N$.\footnote{Of course if $\rho_{m-1}^M<\rho_0^M$
   then $\pi(\rho_{m-1}^M)=\rho_{m-1}^N$ because $\pi$ is a weak $m$-embedding.}
   \item If $\kappa^M<\rho_{m-1}^M$ then $\pi(\kappa^M)=\kappa^N$.\qedhere
  \end{enumerate}
 \end{enumerate}
\end{dfn}

\begin{lem}\label{lem:tpcopy} ***This is incomplete; where is it used?

 Let $M,N$ be $m$-sound and $\pi:M\to N$ be an $m$-okay embedding.
 Let $\Tt$ be a putative sse-$m$-maximal tree on $M$.
 Then there are unique $\lambda$, $\Uu$, $\left<\pi_\alpha\right>_{\alpha<\lambda}$ and
 $\left<\sigma_\alpha,\pi^*_{\alpha+1}\right>_{\alpha+1<\lh(\Tt)}$ such that:
 \begin{enumerate}
  \item $\lambda\leq\lh(\Tt)$ and $\Uu$ is a putative sse-$m$-maximal tree on $N$ of length $\lambda$.
  \item  If $\lambda<\lh(\Tt)$ then $\Uu$ has successor length and $M^\Uu_{\lambda-1}$ is illfounded.
  \item $\Tt\rest\lambda$ and $\Uu\rest\lambda$ have the same tree and drop structure.
  \item\label{item:degree_not_match}If $\deg^\Tt(\alpha)\neq\deg^\Uu(\alpha)$ then $[0,\alpha)_\Uu$ does not drop in model or degree,
  \tu{(}so $[0,\alpha]_\Tt$ does not drop in model and $\deg^\Uu(\alpha)=m$\tu{)} but  $m>0$
  and $\deg^\Tt(\alpha)=m-1$.
  \item $\pi_\alpha$ is $\in$-preserving.
  \item\label{item:okay_emb} If $M^\Tt_\alpha,M^\Uu_\alpha$ are wellfounded then $\pi_\alpha:M^\Tt_\alpha\to M^\Uu_\alpha$ is a $\deg^\Tt(\alpha)$-okay embedding,
  and if $[0,\alpha]_\Tt$ drops in model or degree then $\pi_\alpha$ is a $\nu$-preserving near $\deg^\Tt(\alpha)$-embedding.
  \item $\exit^\Uu_\alpha=\exitcopyseg^{\pi_\alpha}(\exit^\Tt_\alpha)$.
  \item $\sigma_\alpha=\exitcopymap^{\pi_\alpha}(\exit^\Tt_\alpha):\exit^\Tt_\alpha\to\exit^\Uu_\alpha$.
  \item Letting $\beta=\pred^\Tt(\alpha+1)$ we have
  \begin{enumerate}[label=--]
  \item $M^{*\Uu}_{\alpha+1}=\copyseg^{\pi_\beta}(M^{*\Tt}_{\alpha+1})$ and
  \item $\pi^*_{\alpha+1}=\copymap^{\pi_\beta}(M^{*\Tt}_{\alpha+1}):(M^{*\Tt}_{\alpha+1})^\sq\to(M^{*\Uu}_{\alpha+1})^\sq$.
  \end{enumerate}
  \item $\pi_{\alpha+1}:(M^\Tt_{\alpha+1})^\sq\to(M^\Uu_{\alpha+1})^\sq$ is induced by $(\sigma_\alpha,\pi^*_{\alpha+1})$ via Shift Lemma.
  \item $\pi_\lambda\com i^\Tt_{\alpha\lambda}=i^\Uu_{\alpha\lambda}\com\pi_\alpha$.
  for $\alpha<_\Tt\lambda$ such that $(\alpha,\lambda]_\Tt$ does not drop.
 \end{enumerate}
 \end{lem}
\begin{proof}
This is of course mostly the usual copying construction,
and uniqueness is straightforward. So we omit the standard calculations;
the main new issues are conditions \ref{item:degree_not_match} and \ref{item:okay_emb}. Everything is proved by induction on $\alpha$.

We let $\bar{\Uu}$ be the putative tree on $N$ of length $\lambda$
defined as is  $\Uu$, except that $\deg^{\bar{\Uu}}(\alpha)=\deg^\Tt(\alpha)$ for all $\alpha<\lh(\Uu)$.
So we get sequences $\left<\bar{\pi}_\alpha\right>_{\alpha<\lambda}$ and $\left<\bar{\sigma}_\alpha,\bar{\pi}^*_{\alpha+1}\right>_{\alpha+1<\lh(\Tt)}$
corresponding to $\bar{\Uu}$ analogous to those for $\Uu$,
and
\[ \varrho_\alpha:(M^{\bar{\Uu}}_\alpha)^\sq\to(M^\Uu_\alpha)^\sq,
\]
\[ \varrho^{*}_{\alpha+1}:(M^{*\bar{\Uu}}_{\alpha+1})^\sq\to(M^{*\Uu}_{\alpha+1})^\sq\]
are the standard factor maps.

The fact that $\pi_\alpha$ is a near $\deg^\Tt(\alpha)$-embedding when $[0,\alpha]_\Tt$ drops in model or degree
follows by a slight adaptation of the proof of \cite{fs_tame}.
So we just need to verify that $\pi_\alpha$ is a $\deg^\Tt(\alpha)$-okay embedding.

Consider then $\pi_{\alpha+1}$. 
\end{proof}

\begin{dfn}
 Let $M,N$ be $m$-sound and $\pi:M\to N$ be an $m$-okay embedding.
 Let $\Tt$ be a putative sse-$m$-maximal tree on $M$. 
 Then $\tpcopy^\pi(\Tt)$\index{$\tpcopy$} denotes the unique tree $\Uu$ on $N$
 as in \ref{lem:tpcopy}.
\end{dfn}

The following lemma summarizes all the iterability facts we will need. One could certainly prove 
more such facts, but we content ourselves with those stated here.  We have made certain restrictions 
which are more for convenience than 
anything else, such as that $t_\xi^Y\neq 2$. The details of the proof are 
straightforward and mostly routine (in light of \cite{premouse_inheriting}),
but we will give a complete account,
both because, as discussed at the beginning of this section,
there are some wrinkles,
and because we need the details in the rest of the paper.

\begin{dfn}
An \emph{$X$-weak $m$-embedding} $\pi:M\to N$ is just a weak $m$-embedding $\pi$
 as witnessed by the cofinal set $X\sub\rho_m^M$.
\end{dfn}

\begin{lem}\label{lem:iterability}\index{iterability of $N_\alpha$}
Let $Y$ be $(x,\xi,z)$-good, $N=N^Y_{\xi z}$. Then:
 \begin{enumerate}
  \item\label{item:general_mouse_it} Suppose $\aleph_\xi^Y<\rho_0^Y$ and $t^Y_\xi\in\{0,1\}$ and $Y$ is 
$(0,\om_1+1)$-iterable. Then $N$ is 
$(z,\om_1+1)$-iterable. Likewise if we replace $(0,\om_1+1)$ with $(0,\om_1,\om_1+1)^*$
and $(z,\om_1+1)$ with $(z,\om_1,\om_1+1)^*$.
 \item\label{item:Q-mouse_it} Suppose that:
 \begin{enumerate}[label=--]
 \item $Y$ is $y$-sound and $(y,\om_1+1)$-iterable,
 \item $Y$ is non-small, $\delta^Q\leq\rho_y^Q$ and $\delta^Q$ is $\bfrSigma^Q_y$-regular,
 \item either:
 \begin{enumerate}[label=--]
 \item $\xi=\OR^Y$ and $z\leq y$, or
 \item $\xi<\delta^Y$.
 \end{enumerate}
 \end{enumerate}
Then $N$ is $(z,\om_1+1)$-iterable. Likewise if we replace $(y,\om_1+1)$ with $(y,\om_1,\om_1+1)^*$
and $(z,\om_1+1)$ with $(z,\om_1,\om_1+1)^*$.
  \item\label{item:pseudo-premouse_it} If $t_\xi^Y=0$ and $M=(N,G)$ is 
a generalized-pseudo-premouse and creature of $\CC^Y$ then $M$ is $(0,\om_1+1)$-iterable.
 \item\label{item:bicephalus_it} If $t_\xi^Y=0$ and $M=(N,F,G)$ is a 
bicephalus and creature of $\CC^Y$ then $M$ is $(0,\om_1+1)$-iterable.
 \end{enumerate}
\end{lem}
\begin{proof}
The proof of each part is almost the same.
It is the details of parts \ref{item:general_mouse_it} and \ref{item:Q-mouse_it} which feature in 
the scale construction. (In fact, only their restrictions to $(z,\om_1+1)$ and $(y,\om_1+1)$.)
We focus on the proof of part \ref{item:Q-mouse_it}; the proof of part 
\ref{item:general_mouse_it} is a simplification thereof.

Note that
 $N$ is non-small with $\delta^N=\delta=\delta^Y$, by \ref{lem:Woodin_exactness}.
Let $\pi,X$ be such that $\pi:M^\sq\to N^\sq$ is an $X$-weak $z$-embedding such that:
\begin{enumerate}
\item\label{item:if_pi_nu-low_then_nu_sing} If $\pi$ is $\nu$-low then $\nu^M$ is $M$-singular.
 \item\label{item:if_OR^Y=xi_then_delta_reg} If $\xi=\OR^Y$ then:
 \begin{enumerate}[label=--]
 \item $\delta\in\rg(\pi)$
 (so $M$ is non-small and 
$\pi(\delta^M)=\delta$)\footnote{Note that, for example, 
$Y|\delta^Y\elem_1 Y|(\delta^Y+\om)$, so this condition does make a 
demand if $z=0$.}
\item  $\delta^M\leq\rho_z^M$ and $\delta^M$ is 
$\bfrSigma_z^M$-regular.
\end{enumerate}
\end{enumerate}
(It is easy to see that $\id:N\to N$ has these properties.)

First fix a $(y,\om_1+1)$-iteration 
strategy $\Sigma_Y$ for $Y$.
We define a $(z,\om_1+1)$-strategy $\Sigma_M$ for $M$, lifting trees $\Tt$ via $\Sigma_M$ to 
trees $\Uu$ via $\Sigma_Y$.
One can then iterate this process to produce a $(z,\om_1,\om_1+1)^*$-strategy for $M$.
(We define the normal strategy for $M$ instead of just for $N$, to assist when we extend to a 
$(z,\om_1,\om_1+1)^*$-strategy later.)

The copying process splits naturally into two stages: the first is for the small segment of 
$\Tt$, and the second for the non-small segment. The non-small segment
is simpler, as for that stage we can just use standard copying.
But it is more natural to discuss the small 
segment 
first, so we do so. We list below many objects we will keep track of,
and many conditions that we maintain regarding them. Once we have these lists made,
it is pretty straightforward to verify that we can indeed maintain them, and hence complete the 
proof.

Suppose that $\Tt$ via $\Sigma_M$ is small, of length $\lambda$, and we have so far produced the 
corresponding lift $\Uu$ via $\Sigma_Y$.
We will have sequences
\[ 
\left<M_\alpha,m_\alpha,Y_\alpha,\xi_\alpha,N_\alpha,\pi_\alpha,X_\alpha\right>_{\alpha<\lambda},\]
\[ \left<S_\alpha,k_\alpha,\left<M_{\alpha i},Y_{\alpha 
i},\varrho_{\alpha i}\right>_{i\leq k_\alpha},
\eta_\alpha,
T_\alpha,\sigma_\alpha,M^*_{\alpha+1},\xi^*_{\alpha+1},N^*_{\alpha+1},\pi^*_{\alpha+1}
\right>_{\alpha+1<\lambda }
\]
such that (\emph{L} in \emph{L1}, etc, abbreviates \emph{Lift}):
\begin{enumerate}[label=\tu{(}L\arabic*\tu{)}]
\item\label{item:T_is_small_z-max} $\Tt$ is small $z$-maximal on $M$ (of length $\lambda$), based on $M|\delta^M$.
\item For $\alpha<\lambda$: $M_\alpha=M^\Tt_\alpha$ and $m_\alpha=\deg^\Tt(\alpha)$.
\item\label{item:S_alpha_is_exit_etc} For $\alpha+1<\lambda$: $S_\alpha=\exit^\Tt_\alpha$,
$\left<M_{\alpha i}\right>_{i\leq k_\alpha}=\redd^{M_\alpha}(S_\alpha)$ and 
$M^*_{\alpha+1}=M^{*\Tt}_{\alpha+1}$.
\item\label{item:Uu_nodes} $\Uu$ is padded small $y$-maximal on $Y$, via $\Sigma_Y$, based on 
$Y|\delta^Y$, nowhere 
dropping in model or degree.
The nodes of $\Uu$ form a set $I\sub\lambda\cross\om$,
ordered lexicographically,
and for each $\alpha+1\leq\lambda$, letting
\[ I_\alpha=\{i\mid (\alpha,i)\in I\}, \]
if $\alpha+1=\lambda$ then $I_\alpha=[0,0]$, and if $\alpha+1<\lambda$ then 
$I_\alpha=[0,k_\alpha]$. 
For each $(\alpha,i)$, if $E^\Uu_{\alpha i}\neq\emptyset$ then $\nu(E^\Uu_{\alpha i})$
is a cardinal of $M^\Uu_{\alpha i}$.
\item\label{item:Y_alpha=} For $\alpha<\lambda$: $Y_\alpha=M^\Uu_{\alpha0}$.
\item\label{item:U_alpha_i_objects} For $\alpha+1<\lambda$ and $i\leq k_\alpha$:
\begin{enumerate}[label=\tu{(}\roman*\tu{)}]
 \item $Y_{\alpha i}=M^\Uu_{\alpha i}$ and $\varrho_{\alpha i}\in\OR^{Y_{\alpha i}}$
 \item\label{item:varrho_order} $\varrho_{\beta j}\leq\varrho_{\alpha i}$ for 
$(\beta,j)<(\alpha,i)$ (with $j\leq k_\beta$).
 \item\label{item:E_alpha_k_alpha} $E^\Uu_{\alpha k_\alpha}\neq\emptyset$ and $\varrho_{\alpha 
k_\alpha}<\nu(E^\Uu_{\alpha k_\alpha})$ and there is no $E\in\es(M^\Uu_{\alpha+1,0})$
with
\[ \varrho_{\alpha k_\alpha}\leq\crit(E)\leq\nu(E^\Uu_{\alpha k_\alpha})\leq\nu(E). \]
 \item\label{item:E_alpha_i} For each $i<k_\alpha$, if $E=E^\Uu_{\alpha i}\neq\emptyset$ then
 $E$ is an order $0$ measure and $\varrho_{\alpha i}=\crit(E)$
 and $\nu(E^\Uu_{\beta k_\beta})<\crit(E)$ for all $\beta<\alpha$.
 \item For each 
$i<k_\alpha$, $\pred^\Uu(\alpha,i+1)=(\alpha,i)$.
\item\label{item:U_order} Let $(\beta,j)=\pred^\Uu(\alpha+1,0)$. Then:
\begin{enumerate}[label=--]
\item $(\beta,j)$ is least such that $\crit(E^\Uu_{\alpha k_\alpha})<\varrho_{\beta j}$.
\item $\beta=\pred^\Tt(\alpha+1)$ and $M^*_{\alpha+1}=M_{\beta j}$ (so $j>0$ iff $\alpha+1\in\dropset^\Tt$).
\end{enumerate}
\end{enumerate}
\item\label{item:branch_projection} For each $\alpha<\lambda$, we have 
$[0,\alpha)_\Tt=\{\beta<\alpha\mid (\beta,0)<_\Uu(\alpha,0)\}$.
 \item For limits $\alpha<\lambda$, there is $\beta<\alpha$ such that:
 \begin{enumerate}[label=--]
  \item $(\beta,0)<_\Uu(\alpha,0)$ and $\beta<_\Tt\alpha$,
  \item for all $\gamma,i$ with $(\beta,0)<_\Uu(\gamma,i)<_\Uu(\alpha,0)$,
 we have $i=0$,
 \item $(\beta,\alpha)_\Tt=\{\gamma\mid (\beta,0)<_\Uu(\gamma,0)<_\Uu(\alpha,0)\}$.
 \end{enumerate}
\item\label{item:xi_when_no_drop} If $(0,\alpha]_\Tt\inter\dropset^\Tt=\emptyset$
then $m_\alpha=z$ and either:
\begin{enumerate}
\item $\xi=\OR^Y$ and $\xi_\alpha=\OR^{Y_\alpha}$, or
\item $\xi<\rho_0^Y$ and $\xi_\alpha=i^\Uu_{(0,0),(\alpha,0)}(\xi)$.
\end{enumerate}
\item\label{item:xi_when_drop} If $(0,\alpha]_\Tt\inter\dropset^\Tt\neq\emptyset$ then 
$\xi_\alpha<\delta^{Y_\alpha}$.
\item\label{item:pres_xi_or_drop_xi} Let $\beta<_\Tt\alpha$ be such that $\xi_\beta<\rho_0^{Y_\beta}$.
Then:
\begin{enumerate}
\item If $(\beta,\alpha]_\Tt\inter\dropset^\Tt=\emptyset$
then $\xi_\alpha=i^\Uu_{(\beta,0),(\alpha,0)}(\xi_\beta)$.
\item If $(\beta,\alpha]_\Tt\inter\dropset^\Tt\neq\emptyset$
then $\xi_\alpha<i^\Uu_{(\beta,0),(\alpha,0)}(\xi_\beta)$.
\end{enumerate}
\item\label{item:N_alpha} $N_\alpha=N_{\xi_\alpha m_\alpha}^{Y_\alpha}$ and $t_{\xi_\alpha}^{Y_\alpha}\neq 2$.
 \item\label{item:pi_alpha_weak} $\pi_\alpha:M_\alpha^\sq\to N_\alpha^\sq$ is an $X_\alpha$-weak 
$m_\alpha$-embedding.
\item\label{item:pi_alpha_nu-low} If $\pi_\alpha$ is $\nu$-low then $\nu^{M_\alpha}$ is 
$M_\alpha$-singular.\footnote{In most circumstances we 
will also know that $\pi_\alpha$
is non-$\nu$-high. But suppose $P$ is non-small type 3 (so $N^P$ is non-small type 3
and $\nu=\nu(N^P)=\nu(P)$), $\mu\in N^P$ is 
measurable in $N^P$ via $\es(N^P)$, and $\cof^P(\nu)=\mu<\delta^P$
but $\cof^{N^P}(\nu)\neq\mu$ (so $\mu<\cof^{N^P}(\nu)<\delta^P$, by the $\delta^P$-cc of the 
extender algebra). Let $E\in\es(N^P)$ be $N^P$-total with $\crit(E)=\mu$.
Let $\Tt=\left<E\right>$, a tree on $N^P$. Let $\Uu=\left<F\right>$ be the liftup tree on $P$.
Then $\pi_1:(M^\Tt_1)^\sq\to(M^\Uu_1)^\sq$ is $\nu$-high.}
\item $\xi_0=\xi$ and $m_0=z$ and $\pi_0=\pi:M^\sq\to N^\sq$ and $X_0=X$.
\item\label{item:X_beta_is_ptwise_image_non-dropping} If $\alpha\leq_\Tt\beta$ and $(\alpha,\beta]_\Tt\inter\dropset_{\deg}^\Tt=\emptyset$
then $X_\beta=i^\Tt_{\alpha\beta}``X_\alpha$.
\item
If $\alpha+1\in\dropset_{\deg}^\Tt$ then $X_{\alpha+1}=i^{*\Tt}_{\alpha+1}``\rho_{m_{\alpha+1}}(M^*_{\alpha+1})$.
\item\label{item:delta_when_no_drop}
$\delta^{Y_\alpha}\leq\rho_y^{Y_\alpha}$, $\delta^{Y_\alpha}$ is $\bfrSigma_y^{Y_\alpha}$-regular,
and $i^\Uu_{(0,0),(\alpha,0)}$ is continuous at $\delta^Y$.
If $\xi=\OR^Y$ and $(0,\alpha]_\Tt\inter\dropset^\Tt=\emptyset$
then likewise,
$\delta^{M_\alpha}\leq\rho_z^{M_\alpha}$,
$\delta^{M_\alpha}$
is $\bfrSigma_z^{M_\alpha}$-regular,
$i^\Tt_{0\alpha}$ is continuous at $\delta^M$,
and $\pi_\alpha(\delta^{M_\alpha})=\delta^{N_\alpha}=\delta^{Y_\alpha}$.
\item\label{item:U_rest_I_alpha} if $\alpha+1<\lh(\Tt)$ then $\Uu\rest I_\alpha=\restree^{Y_\alpha y}_{\xi_\alpha 
m_\alpha\pi_\alpha}(S_\alpha)$
\item\label{item:eta_alpha,sigma_alpha} $(\eta_\alpha,\sigma_\alpha)=\exitresadd^{Y_\alpha}_{\xi_\alpha 
m_\alpha\pi_\alpha}(S_\alpha)$ and $T_\alpha=N^{Y_{\alpha k_\alpha}}_{\eta_\alpha}$ (so
$\sigma_\alpha:S_\alpha^\sq\to T_\alpha^\sq$
is a non-$\nu$-low weak $0$-embedding, by \ref{lem:exitresadd}).
 \item\label{item:t_eta_alpha} $\eta_\alpha=\chi+\om$ for some $\chi$ with 
$t^{Y_{\alpha 
k_\alpha}}_\chi=0$ and
 $t^{Y_{\alpha k_\alpha}}_{\eta_\alpha}=1$.
\item\label{item:bkgd_ext} $E^\Uu_{\alpha k_\alpha}=F^{Y_{\alpha k_\alpha}}_{\eta_\alpha}$
(so $\nu(E^\Uu_{\alpha k_\alpha})=\aleph_{\eta_\alpha-\om}^{Y_{\alpha k_\alpha}}$)
\item\label{item:varrho_alpha_i} Let $(\Vv,\left<\Psi_i\right>_{i\leq 
k_\alpha})=\modres^{Y_\alpha}_{\xi_\alpha 
m_\alpha\pi_\alpha}(S_\alpha)$ and $\Psi_i=(\beta_i,A_i,\psi_i,\alpha_i,d_i)$.
Then $\varrho_{\alpha i}=\rho_\om^{A_{i+1}}$
for $i<k_\alpha$ (this agrees with \ref{item:U_alpha_i_objects}\ref{item:E_alpha_i})
and
$\varrho_{\alpha k_\alpha}=\nu(T_\alpha)$.
\item\label{item:CC_coherence_iterability} We have:
\begin{enumerate}[label=(\roman*)]
\item $\CC^{Y_{\alpha+1}}\rest\eta_\alpha=\CC^{Y_{\alpha 
k_\alpha}}\rest\eta_\alpha$ (so $N_{\eta_\alpha-\om}^{Y_{\alpha+1}}=T_\alpha^\passive$)
 and 
 \item\label{item:t=0_after_Ult} $t^{Y_{\alpha+1}}_{\eta_\alpha}=0$.
 \end{enumerate}(However, \ref{item:t=0_after_Ult} is not used in the proof of parts \ref{item:general_mouse_it}
 and \ref{item:Q-mouse_it}
 of the lemma (\ref{lem:iterability}),
 and we will not claim the analogue of \ref{item:t=0_after_Ult} in the proofs of parts \ref{item:pseudo-premouse_it}
and \ref{item:bicephalus_it},  that is, the iterability for  generalized-pseudo-premice and bicephali.)
 \item\label{item:model_coherence_card} If $\lgcd(T_\alpha)$ is an $N_{\alpha+1}$-cardinal then 
$T_\alpha^\passive=N_{\alpha+1}||\OR^{T_\alpha}$.
\item\label{item:model_coherence_non-card} If $\lgcd(T_\alpha)$ is not an $N_{\alpha+1}$-cardinal 
then there is $\rho$ such that 
\begin{enumerate}[label=--]
 \item $\lgcd(T_\alpha)=\nu(T_\alpha)=(\rho^+)^{T_\alpha}$ (so $S_\alpha,T_\alpha$ are type 1 or 3),
 \item $N_{\alpha+1}||\nu(T_\alpha)=T_\alpha|\nu(T_\alpha)$,
\end{enumerate}
and either
\begin{enumerate}
 \item $T_\alpha^\passive=N_{\alpha+1}||\OR^{T_\alpha}$, or
  \item $\emptyset\neq G=F^{N_{\alpha+1}|\nu(T_\alpha)}$ is type 2
  and $T_\alpha^\passive=\Ult(N_{\alpha+1},G)||\OR^{T_\alpha}$.
\end{enumerate}
 \item\label{item:sigma_gamma_pi_alpha_agmt} If $\gamma<\alpha<\lambda$ then 
$\sigma_\gamma\sub\pi_\alpha$ and $\pi_\alpha(\nu(S_\gamma))\geq\nu(T_\gamma)=\varrho_{\gamma 
k_\gamma}$.
\item\label{item:pi_alpha_sigma_alpha_agmt} If $\gamma<\alpha<\alpha+1<\lambda$
then $\sigma_\alpha\rest S_\gamma=\pi_\alpha\rest S_\gamma$ (so if $\alpha<\beta<\lambda$
then by 
\ref{item:sigma_gamma_pi_alpha_agmt} and since $S_\gamma\sub S_\alpha^\sq$, 
we have $\pi_\beta\rest S_\gamma=\pi_\alpha\rest S_\gamma$).
 \item\label{item:T=} $T\eqdef\pi_{\alpha+1}(S_\alpha^\passive)\pins_\card N_{\alpha+1}$
 and $\OR^{T_\alpha}\leq\OR^T<\rho_{m_{\alpha+1}}^{N_{\alpha+1}}$ and 
$(\eta_\alpha-\om)\leq\prodstage^{Y_{\alpha+1}}_{\xi_{\alpha+1}}(T)$.\footnote{Since $\Tt$ is small, it does not use any superstrong extenders, so
we do have $S^{\passive}_\alpha\pins M_{\alpha+1}^{\sq}$ in the first place, so $S^{\passive}_\alpha\in\dom(\pi_{\alpha+1})$.}
\item\label{item:lift_it_comm} If $\alpha\leq_\Tt\beta$ and 
$(\alpha,\beta]_\Tt\inter\dropset^\Tt=\emptyset$
then
\[ \pi_\beta\com i^\Tt_{\alpha\beta}=
i^\Uu_{(\alpha,0),(\beta,0)}\com
\tau_{\xi_\alpha m_\alpha m_\beta}^{Y_\alpha}
\com\pi_\alpha=
\tau_{\xi_\beta m_\alpha m_\beta}^{Y_\beta}\com
i^\Uu_{(\alpha,0),(\beta,0)}\com\pi_\alpha.\]

\item\label{item:successor_step} Let $\gamma=\pred^\Tt(\alpha+1)$ and $\kappa=\crit(E^\Tt_\alpha)$ 
and
$(S,s)=\psegdeg^{M_\gamma,m_\gamma}(S_\gamma,\kappa)$
and
\[ l=\critresl^{Y_\gamma}_{\xi_\gamma m_\gamma\pi_\gamma}(S_\gamma,\kappa),
\]
so $S=M^*_{\alpha+1}=M_{\gamma l}$ and $s=\deg^\Tt(\alpha+1)=m_{\alpha+1}$.
Then
\begin{enumerate}[label=--]
\item $\xi^*_{\alpha+1}=\critresprodstage^{Y_\gamma}_{\xi_\gamma 
m_\gamma\pi_\gamma}(S_\gamma,\kappa)$,
\item $N^*_{\alpha+1}=N^{Y_{\gamma l}}_{\xi^*_{\alpha+1}s}$,
\item $(\gamma,l)=\pred^\Uu(\alpha+1,0)$,
\item $\pi^*_{\alpha+1}=\critresmap^{Y_\gamma}_{\xi_\gamma 
m_\gamma\pi_\gamma}(S_\gamma,\kappa)$ and $\pi^*_{\alpha+1}:(M^*_{\alpha+1})^\sq\to(N^*_{\alpha+1})^\sq$
is a weak $s$-embedding such that
\[ \pi^*_{\alpha+1}\rest(\kappa^+)^{M^*_{\alpha+1}}=\sigma_\alpha\rest(\kappa^+)^{S_\alpha}.\]
\end{enumerate}

If $(0,\alpha+1]_\Tt\inter\dropset^\Tt=\emptyset$ and $\xi=\OR^Y$ then (note $l=0$)
\[ \xi^*_{\alpha+1}=\OR^{Y_{\gamma}}=\xi_\gamma\ \&\ 
\xi_{\alpha+1}=\OR^{Y_{\alpha+1}};\]
otherwise $\xi_{\alpha+1}=i^\Uu_{(\gamma,l),(\alpha+1,0)}(\xi^*_{\alpha+1})$.
If $\alpha+1\in\dropset^\Tt$ then
\[ \xi^*_{\alpha+1}<i^\Uu_{(\gamma,0),(\gamma,l)}(\xi_\gamma)\text{ and } 
\xi^*_{\alpha+1}<\delta^{Y_{\gamma l}},\]
whereas if $\alpha+1\notin\dropset^\Tt$ then ($l=0$ and) 
$\xi^*_{\alpha+1}=\xi_\gamma$; with condition \ref{item:N_alpha}, this determines 
$N_{\alpha+1}$. Finally,
\[ \pi_{\alpha+1}:M_{\alpha+1}^\sq\to N_{\alpha+1}^\sq \]
is given much as in the Shift Lemma
from $\pi^*_{\alpha+1},\sigma_\alpha$
(but using the larger ultrapower at the upper level; this is discussed further below), and
\[ \pi_{\alpha+1}\com 
i^{*\Tt}_{\alpha+1}=i^\Uu_{(\gamma,l),(\alpha+1,0)}\com\pi^*_{\alpha+1}.\]
\end{enumerate}

This completes the list of conditions. The verification that we can maintain these conditions is 
pretty routine, but we discuss the main points.

First suppose that $\lambda$ is a limit ordinal, and we wish to extend 
everything to $\Tt\rest\lambda+1$ and $\Uu\rest(\lambda+1,1)$.
Note that the non-padded tree equivalent to $\Uu\rest\lambda\cross\om$
has limit ordertype, as
 $E^\Uu_{\alpha k_\alpha}\neq\emptyset$ for each $\alpha<\lambda$.
 
Let $b=\Sigma(\Uu\rest(\lambda,0))$. Set
$[0,\lambda)_\Tt=b'=\{\alpha\mid (\alpha,0)\in b\}$.
By conditions \ref{item:U_alpha_i_objects}, \ref{item:branch_projection}, $b'$ is a $(\Tt\rest\lambda)$-cofinal branch.
By \ref{item:xi_when_drop},  \ref{item:pres_xi_or_drop_xi}, $b'\inter\dropset^\Tt$ is finite,
and \ref{item:xi_when_no_drop}, \ref{item:pres_xi_or_drop_xi} together determine 
$\xi_\lambda$. This  determines $M_\lambda,m_\lambda,Y_\lambda,N_\lambda$, and $X_\lambda$ is determined by \ref{item:X_beta_is_ptwise_image_non-dropping}, and  the $X_\lambda$-weak $m_\lambda$-embedding
\[ \pi_\lambda:M_\lambda^\sq\to N_\lambda^\sq \]
by  \ref{item:lift_it_comm}. It is easy to verify that the conditions are 
maintained at $\lambda+1$.

Now suppose that $\lambda=\alpha+1$ and we have $\Tt\rest\alpha+1$
and $\Uu\rest(\alpha,1)$ (that is, $M^\Uu_{\alpha 0}$ is defined, etc,
but not $M^\Uu_{\alpha 1}$), satisfying the conditions, and wish to extend to 
$\Tt\rest\alpha+2$ and $\Uu\rest(\alpha+1,1)$.
We will discuss the propagation of
\ref{item:Uu_nodes},
\ref{item:U_alpha_i_objects}, \ref{item:xi_when_no_drop}, \ref{item:pi_alpha_weak}, 
\ref{item:pi_alpha_nu-low}, 
\ref{item:delta_when_no_drop}, \ref{item:U_rest_I_alpha}, \ref{item:t_eta_alpha}, 
\ref{item:bkgd_ext}, \ref{item:CC_coherence_iterability}--\ref{item:T=} and 
\ref{item:successor_step}. Within \ref{item:U_alpha_i_objects}
we just discuss parts \ref{item:varrho_order}--\ref{item:E_alpha_i} and
\ref{item:U_order}. The remaining conditions are either by definition,  immediate or routine, and we leave them to the reader. 

Conditions \ref{item:T_is_small_z-max}--\ref{item:S_alpha_is_exit_etc}: By definition or assumption.

Conditions \ref{item:Uu_nodes}, \ref{item:U_alpha_i_objects}\ref{item:varrho_order},\ref{item:E_alpha_i}, 
\ref{item:U_rest_I_alpha}--\ref{item:varrho_alpha_i}, \ref{item:pi_alpha_sigma_alpha_agmt}:
We set
\begin{enumerate}[label=--]
 \item $\Uu\rest I_\alpha=\Vv=\restree^{Y_\alpha y}_{\xi_\alpha 
m_\alpha\pi_\alpha}(S_\alpha)$,
\item $(\eta_\alpha,\sigma_\alpha)=\exitresadd^{Y_\alpha}_{\xi_\alpha 
m_\alpha\pi_\alpha}(S_\alpha)$ and $T_\alpha=N^{Y_{\alpha k_\alpha}}_{\eta_\alpha}$,
\item $E^\Uu_{\alpha k_\alpha}=F^{Y_{\alpha k_\alpha}}_{\eta_\alpha}$; also write $F_\alpha=E^\Uu_{\alpha k_\alpha}$,
\item $\pred^\Uu(\alpha+1,0)$ and $\deg^\Uu(\alpha+1,0)$ as demanded by $y$-maximality,\footnote{Since $\Uu$ is padded, this maybe doesn't fully determine $\pred^\Uu(\alpha+1,0)$,
depending on conventions. Actually we will want $\pred^\Uu(\alpha+1,0)$
to be the least $(\beta,j)$ such that 
$\crit(F_\alpha)<\varrho_{\beta j}$.
Certainly this requirement uniquely
determines $\pred^\Uu(\alpha+1,0)$.
We will verify later that this is commensurate with $\Uu$ being $y$-maximal.}
\item $\varrho_{\alpha i}=\rho_\om^{A_{i+1}}$
and $\varrho_{\alpha k_\alpha}=\nu(N^{Y_{\alpha k_\alpha}}_{\eta_\alpha})$,
where $A_{i+1}$ is as in
\ref{item:varrho_alpha_i};
\end{enumerate}
these things are demanded by \ref{item:Uu_nodes} (for $y$-maximality of $\Uu$),
\ref{item:U_rest_I_alpha},
 \ref{item:eta_alpha,sigma_alpha}, \ref{item:bkgd_ext}, and \ref{item:varrho_alpha_i},
and this yields \ref{item:U_rest_I_alpha},   \ref{item:eta_alpha,sigma_alpha}, \ref{item:bkgd_ext},
\ref{item:varrho_alpha_i}.
Condition \ref{item:t_eta_alpha} will follow because
$\Tt$ is small.

\begin{clm}The above specifications  yield a $y$-maximal
$\Uu\rest(\alpha+1,1)$ (whose last extender is $F_\alpha$), and also conditions 
\ref{item:U_alpha_i_objects}\ref{item:varrho_order},\ref{item:E_alpha_i}
and \ref{item:pi_alpha_sigma_alpha_agmt}.\end{clm}
\begin{proof}
As in \ref{item:varrho_alpha_i}, let $\Psi_i=(\beta_i,A_i,\psi_i,\alpha_i,d_i)$ where
\[ (\Vv,\left<\Psi_i\right>_{i\leq k_\alpha})=\modres^{Y_\alpha y}_{\xi_\alpha 
m_\alpha\pi_\alpha}(S_\alpha).\] For $i\geq 1$ we have
$\psi_i:M_{\alpha i}^\sq\to A_i^\sq$ is a neat
embedding, and if $1\leq i<k_\alpha$ then $\rho_\om^{M_{\alpha i}}<\rho_\om^{M_{\alpha,i+1}}\leq\rho_0^{M_{\alpha i}}$ 
and
\[ \psi_i(\rho_\om^{M_{\alpha i}})\geq\rho_\om^{A_i}=\varrho_{\alpha,i-1} \]
 which easily gives that 
$\varrho_{\alpha,i-1}<\varrho_{\alpha i}$.
Also if $k_\alpha>0$ then easily $\varrho_{\alpha,k_\alpha-1}\leq\varrho_{\alpha k_\alpha}$.
So for \ref{item:U_alpha_i_objects}\ref{item:varrho_order} it suffices to see that 
$\varrho_{\beta j}\leq\varrho_{\alpha 0}$ for $\beta<\alpha$.

From now on we assume that $\alpha=\beta+1$, and leave the limit case to the reader.

\begin{case}
 $S_\alpha\pins M_\alpha$, so $k_\alpha>0$.
 
Then $S_\beta^\passive\pins_\card M_{\beta+1}=M_{\alpha}$
and $\OR^{S_\beta}<\rho_{m_\alpha}^{M_\alpha}$
and $\OR^{S_\beta}<\OR^{S_\alpha}$, so $\OR^{S_\beta}\leq\rho_\om^{M_{\alpha 1}}$.
By \ref{item:T=},
\[ T\eqdef\pi_\alpha(S_\beta^{\passive})\pins_\card N_\alpha|\rho_{m_\alpha}^{N_\alpha},\]
\[ \varrho_{\beta k_\beta}=\nu(T_\beta)<\OR^{T_\beta}\leq\OR^T \]
 and since $\tau^{Y_\alpha}_{\xi_\alpha m_\alpha 0}\rest\rho_{m_\alpha}^{N_\alpha}=\id$, note that therefore
\[ \varrho_{\beta k_\beta}<\OR^T=\pi_\alpha(\OR^{S_\beta})\leq\rho_\om^{A_1}=\varrho_{\alpha 0},
\]
giving \ref{item:U_alpha_i_objects}\ref{item:varrho_order}.
It easily follows that $\sigma_\alpha\rest S_\beta=\pi_\alpha\rest S_\beta$,
giving \ref{item:pi_alpha_sigma_alpha_agmt}.
Now suppose that $k<k_\alpha$ is least such that $E=E^\Vv_k\neq\emptyset$
and let $\kappa=\crit(E)=\varrho_{\alpha k}$.
For \ref{item:U_alpha_i_objects}\ref{item:E_alpha_i},
we need to see that $\nu(F_\beta)<\kappa$.
We have $A_{k+1}=N_\eta^{Y_\alpha}$ for some $\eta$ such that $t_\eta^{Y_\alpha}=2$,
and $\kappa=\rho_\om^{A_{k+1}}$, and $N_\kappa^{Y_\alpha}=A_{k+1}|\kappa$.
By \ref{item:T=}, \[\eta_\beta-\om\leq\prodstage^{Y_\alpha}_{\xi_\alpha}(T),\] and 
$T=N_{\zeta}^{Y_\alpha}$ for some 
$\zeta\geq\eta_\beta-\om$. We have $T\ins_\card A_1|\rho_\om^{A_1}$,
so
\[ T\ins_\card A_{k+1}|\rho_\om^{A_{k+1}}=A_{k+1}|\kappa=N_\kappa^{Y_\alpha}.\]
Actually then $T\pins N_\kappa^{Y_\alpha}$ as $T$ has a largest cardinal.
So $\prodstage^{Y_\alpha}_{\xi_\alpha}(T)<\kappa$, 
so $\eta_\beta-\om<\kappa$, but $\kappa$ is $Y_\alpha$-measurable,
so
$\nu(F_\beta)=\aleph_{\eta_\beta-\om}^{Y_\alpha}<\kappa$,
as required.

Now if $\Vv$ uses a measure, then note that
$\kappa'\leq\rho_{\alpha k_\alpha}=\nu(F^{T_\alpha})<\nu(F_\alpha)$
for each critical point $\kappa'$
used in $\Vv$
(the fact that $\nu(F^{T_\alpha})<\nu(F_\alpha)$
is because $\eta_\alpha-\om\geq\OR^{T_\alpha}>\nu(F^{T_\alpha})$, and $\nu(F_\alpha)=\aleph_{\eta_\alpha-\om}^{Y_{\alpha k_\alpha}}$).
If $\Vv$ uses no measure (but still $k_\alpha>0$),
then we just need to verify that $\lh(F_\beta)<\lh(F_\alpha)$.
But note that $T\pins_\card N^{Y_\alpha}_{\eta_\alpha}$,
so $\eta_\beta-\om<\eta_\alpha$, which easily suffices.
So we have verified the claim assuming $S_\alpha\pins M_\alpha$.\end{case}

\begin{case} $S_\alpha=M_\alpha$.
 
 Then $k_\alpha=0$ and $\Vv$ is trivial,
and because $\OR^{S_\beta}<\nu(M_\alpha)$ and
$T\pins_\card 
N_\alpha|\rho_{m_\alpha}^{N_\alpha}$,
one can argue much as before. We omit further details.
\end{case}

This completes the proof of the claim.
\end{proof}

For \ref{item:Uu_nodes}, it just
remains to verify that $\Uu$ is small.
By induction (with \ref{item:Uu_nodes}),
$\Uu\rest(\alpha,1)$ is based on $Y|\delta^{Y}$
and is nowhere dropping in model or degree.
By \ref{item:delta_when_no_drop}, $\delta^{Y_\alpha}$ is $\bfrSigma_y^{Y_\alpha}$-regular,
and $\Vv$ is based on $Y_\alpha|\delta^{Y_\alpha}$
and uses only total measures, so 
by \ref{lem:iterates_maintain_def_regularity},
is nowhere dropping in model or degree.
Also, $\nu(F_\alpha)$ is a cardinal of $Y_{\alpha k_\alpha}$
and $F_\alpha\in Y_{\alpha k_\alpha}|\delta^{Y_{\alpha k_\alpha}}$
and $\Uu\rest(\alpha+1,1)$ is small,
because $\Tt$ is small
and $F_\alpha=F_{\eta_\alpha}^{Y_{\alpha k_\alpha}}$ and by Woodin exactness
\ref{lem:Woodin_exactness} (by which
$t_{\eta_\alpha}^{Y_{\alpha 
k_\alpha}}=1$).

Condition \ref{item:Y_alpha=}: just a definition.

Conditions \ref{item:U_alpha_i_objects}\ref{item:E_alpha_k_alpha}, 
\ref{item:CC_coherence_iterability}: We have
$\varrho_{\alpha k_\alpha}=\nu(T_\alpha)<\nu(F_\alpha)=\aleph_{\eta_\alpha-\om}^{Y_{\alpha k_\alpha}}$,
and the rest of these conditions follows from \ref{lem:con_coherence}.\footnote{But note that in 
the proof of iterability for pseudo-premice $(N,G)$ and bicephali $(N,F,G)$,
when we use a background extender $F^*$ for $F$ or $G$ (or one of their images),
then the analogue of \ref{item:CC_coherence_iterability}(ii) can fail, because
there is no minimality condition on $\lh(F^*)$. It is important that we drop the minimality 
condition here, as we will not be in a position where we can guarantee such minimality.}

Conditions \ref{item:U_alpha_i_objects}\ref{item:U_order},
\ref{item:xi_when_no_drop},
\ref{item:pi_alpha_weak}, 
\ref{item:successor_step}:

Let $\gamma=\pred^\Tt(\alpha+1)$ and $\kappa,S,s,l$
be as in \ref{item:successor_step};
we have $\xi^*_{\alpha+1}$, $N^*_{\alpha+1}$, $\pi^*_{\alpha+1}$, $\xi_{\alpha+1}$,
$N_{\alpha+1}$ 
defined as there.
We consider the case that $\gamma<\alpha$ and leave the contrary case to the reader.

\ref{item:xi_when_no_drop} follows easily from \ref{item:xi_when_no_drop},
 \ref{item:delta_when_no_drop} at $\gamma$ and the smallness of $\Tt$ (for then if $(0,\alpha+1]_\Tt\inter\dropset^\Tt=\emptyset$
and $\xi=\OR^Y$ then $m_{\alpha+1}=z$).

As usual
$M^*_{\alpha+1}||\kappa^{+M^*_{\alpha+1}}=S_\gamma||\kappa^{+S_\gamma}=
S_\alpha||\kappa^{+S_\alpha}$.
Let $\zeta=\kappa^{+S_\gamma}$. Then
$\pi^*_{\alpha+1}\rest\zeta=\sigma_\gamma\rest\zeta$ 
by \ref{lem:exitresadd},
$\sigma_\gamma\rest\zeta=\pi_\alpha\rest\zeta$
because $\zeta\sub S_\gamma^\sq$ and by
\ref{item:sigma_gamma_pi_alpha_agmt},
and $\pi_\alpha\rest\zeta=\sigma_\alpha\rest\zeta$ by
\ref{item:pi_alpha_sigma_alpha_agmt}. So $\pi^*_{\alpha+1}\rest\zeta=\sigma_\alpha\rest\zeta$.
The fact that $\pi^*_{\alpha+1}$ is a weak $s$-embedding is also by \ref{lem:exitresadd}.

Since $\kappa<\nu(E^\Tt_\gamma)$, we have
$\crit(F_\alpha)=\sigma_\alpha(\kappa)=\sigma_\gamma(\kappa)<\varrho_{\gamma 
k_\gamma}=\nu(T_\gamma)<\nu(F_\gamma)$,
so $\pred^\Uu(\alpha+1,0)\leq(\gamma,k_\gamma)$. And given $\beta<\gamma$,
we have $\nu(E^\Tt_\beta)\leq\kappa$, and because $\sigma_\alpha\rest 
S_\beta=\pi_\gamma\rest S_\beta$ and by \ref{item:sigma_gamma_pi_alpha_agmt},
\[ 
\crit(F_\alpha)=
\sigma_\alpha(\kappa)\geq\sigma_\alpha(\nu(E^\Tt_\beta))=
\pi_\gamma(\nu(E^\Tt_\beta))\geq\varrho_{\beta k_\beta}.\]
But then since $\nu(F_\beta)\leq\nu(F_\alpha)$
(in fact $<$), \ref{item:U_alpha_i_objects}\ref{item:E_alpha_k_alpha}
gives
$\nu(F_\beta)<\crit(F_\alpha)$,
so $\pred^\Uu(\alpha+1,0)>(\beta,k_\beta)$.

So $(\gamma,0)\leq\pred^\Uu(\alpha+1,0)\leq(\gamma,k_\gamma)$.
But if $0<k_\gamma$ then for $i+1\leq k_\gamma$,
\[ \kappa<\rho_\om^{M_{\gamma,i+1}}\ \iff\
\crit(F_\alpha)=\sigma_\gamma(\kappa)<\varrho_{\gamma i},\]
and moreover, if $E^\Uu_{\gamma i}\neq\emptyset$, then
$\crit(F_\alpha)\neq\varrho_{\gamma i}$. The latter is just because $E^\Uu_{\gamma i}$
is order $0$ and $\varrho_{\gamma i}=\crit(E^\Uu_{\gamma i})$. It follows that by setting 
$\pred^\Uu(\alpha+1,0)$ to be the least $(\beta,j)$
such that $\crit(F_\alpha)<\varrho_{\beta j}$, we get a normal tree
(that is, for each $(\beta,j)$ such that $E^\Uu_{\beta j}\neq\emptyset$,
we have $\crit(F_\alpha)<\nu(E^\Uu_{\beta j})$ iff $\pred^\Uu(\alpha+1,0)\leq(\beta,j)$),
and also $\pred^\Uu(\alpha+1,0)=(\gamma,l)$ (recall $M^*_{\alpha+1}=M_{\gamma l}$).
This gives \ref{item:U_alpha_i_objects}\ref{item:U_order}. 

If $\alpha+1\in\dropset^\Tt$ but $(0,\gamma]\inter\dropset^\Tt=\emptyset$,
then $\xi^*_{\alpha+1}<\delta^{Y_{\gamma l}}$. For $M_\gamma$ is non-small and 
$\pi_\gamma(\delta^{M_\gamma})=\delta^{Y_\gamma}$,
and because 
$\Tt$ is small, $\lh(E^\Tt_\gamma)<\delta^{M_\gamma}$,
so $M^{*}_{\alpha+1}\pins M_\gamma|\delta^{M_\gamma}$, but $\delta^{Y_\gamma}$ 
is a closure point of $\CC^{Y_\gamma}$, which suffices.

Finally, consider $\pi_{\alpha+1}:M_{\alpha+1}^\sq\to 
N_{\alpha+1}^\sq$, that is, recalling $s=m_{\alpha+1}=\deg^\Tt(\alpha+1)$,
\[ \pi_{\alpha+1}:\Ult_s(M^*_{\alpha+1},E^\Tt_\alpha)^\sq\to 
(N^{Y_{\alpha+1}}_{\xi_{\alpha+1} s})^\sq. \]
We define $\pi_{\alpha+1}$ as usual:\footnote{$[a,f]^{M^*,k}_E$
denotes the element of $\Ult_k(M^*,E)$ represented by the pair $(a,f)$.}
\[ \pi_{\alpha+1}([a,f_{t,q}]^{M^*_{\alpha+1},s}_{E^\Tt_\alpha})=
[\sigma_\alpha(a),f'_{t,\pi^*_{\alpha+1}(q)}]^{Y_{\gamma l},y}_{F_\alpha}, \]
where $q\in(M^*_{\alpha+1})^\sq$ and
\[ f=f_{t,q}:[\kappa]^{<\om}\to(M^*_{\alpha+1})^\sq \]
and either (i) $s=0$ and $f=q$, or (ii) $s>0$ and $t$ is an $\rSigma_s$ term and \[ f(u)=t^{M^*_{\alpha+1}}(q,u),\] and
\[ g\eqdef f'_{t,\pi^*_{\alpha+1}(q)}:[\crit(F_\alpha)]^{<\om}\to(N^*_{\alpha+1})^\sq \]
is likewise. Since $\Uu$ is small $y$-maximal with no drops, etc, $y=\deg^\Uu(\alpha+1,0)$. Let us observe that $g$ is indeed used in 
forming  
$\Ult_y(Y_{\gamma l},F_\alpha)$. If $N^*_{\alpha+1}\in Y_{\gamma l}$ then $g\in Y_{\gamma l}$, which 
suffices.
So suppose $(0,\alpha+1]_\Tt\inter\dropset^\Tt=\emptyset$ and $\xi=\OR^Y$.
Then $l=0$ and $\xi^*_{\alpha+1}=\OR^{Y_{\gamma}}$ and $N^*_{\alpha+1}=N^{Y_\gamma}$, so 
maybe $g\notin Y_{\gamma}$.
But in this case $s=m_\gamma=z\leq y$, and by Lemmas \ref{lem:rSigma_red_N_to_Y} and \ref{lem:iterates_pres_fs_match},
$\bfrSigma_z^{N_\gamma}\sub\bfrSigma_z^{Y_\gamma}$, so $g$ is indeed used
in forming the ultrapower.

The fact that $\pi_{\alpha+1}$ is an $X_{\alpha+1}$-weak $m_{\alpha+1}$-embedding is then as usual,
except that if $(0,\alpha+1]_\Tt\inter\dropset^\Tt=\emptyset$ and $\xi=\OR^Y$ then we again use the fact
that $\bfrSigma_z^{N_\gamma}\sub\bfrSigma_z^{Y_\gamma}$ to
verify that $\pi_{\alpha+1}$ is $\rSigma_z$-elementary,
and that $\bfrSigma_{z+1}^{N_\gamma}\sub\bfrSigma_{z+1}^{Y_\gamma}$
for the $\rSigma_{z+1}$-elementarity on $\Hull_{z+1}^{M_{\alpha+1}}(X_{\alpha+1}\cup\pvec_{z+1}^{M_{\alpha+1}})$.
This completes \ref{item:pi_alpha_weak} and 
\ref{item:successor_step}.

\ref{item:model_coherence_card}, \ref{item:model_coherence_non-card}, 
\ref{item:sigma_gamma_pi_alpha_agmt}:
We have $\sigma_\alpha\sub\pi_{\alpha+1}$.
Let $\mu=\crit(F_\alpha)$
and $\mu'=i_{F_\alpha}(\mu)$. Note $\mu'<\rho_{m_{\alpha+1}}^{N_{\alpha+1}}$.
Let
$T^*=i^{Y_{\alpha k_\alpha}}_{F_\alpha}(T_\alpha)$,
so \[T^*|\mu'\pins_\card N_{\alpha+1}|\rho_{m_{\alpha+1}}^{N_{\alpha+1}}.\] Let
$k:\Ult(T_\alpha,F^{T_\alpha})\to T^*$ be the factor and
$j:\Ult(S_\alpha,F^{S_\alpha})\to\Ult(T_\alpha,F^{T_\alpha})$
be induced by $\sigma_\alpha$, so $k\rest\nu(T_\alpha)=\id$ and $j(\nu(S_\alpha))\geq\nu(T_\alpha)$
because $\sigma_\alpha$ is non-$\nu$-low. But 
$\pi_{\alpha+1}(\nu(S_\alpha))=k(j(\nu(S_\alpha)))$, so 
$\pi_{\alpha+1}(\nu(S_\alpha))\geq\nu(T_\alpha)$. This gives 
\ref{item:sigma_gamma_pi_alpha_agmt}.

Let $\chi=\lgcd(T_\alpha)$.
If $\chi$ is an $N_{\alpha+1}$-cardinal, equivalently a $T^*$-cardinal,
then by condensation with $k$,
\[ T_\alpha^\passive=T^*||\OR^{T_\alpha}=N_{\alpha+1}||\OR^{T_\alpha},\]
giving \ref{item:model_coherence_card}.
If $\chi$ is not a $T^*$-cardinal, then clearly $\chi=(\rho^+)^{T_\alpha}$ and 
$\crit(k)=\chi=\nu(T_\alpha)$, and condensation gives \ref{item:model_coherence_non-card},
other than the fact that $N_{\alpha+1}|\chi$ is not active type 1 or 3. But suppose it is type 1 or 
3. We may assume that $\rho$ is a $T_\alpha$-cardinal. Then 
$\rho_\om(N_{\alpha+1}|\chi)=\rho$ is a $T^*$-cardinal,
so there is $\zeta$ such that $N_{\alpha+1}|\chi=N_{\zeta\om}^{Y_{\alpha+1}}$.
We have $\zeta>\eta_\alpha$ because
$\CC^{Y_{\alpha+1}}\rest\eta_\alpha=\CC^{Y_{\alpha k_\alpha}}\rest\eta_\alpha$.
But then $\chi$ is an $N_\zeta^{Y_{\alpha+1}}$-cardinal,
so $\chi\leq\nu(N_\zeta^{Y_{\alpha+1}})$, and note that
\[ F(N_\zeta^{Y_{\alpha+1}})\rest\rho=F(N_{\alpha+1}|\chi)\rest\rho, \]
so by the ISC, $F(N_{\alpha+1}|\chi)\rest\rho\in N_\zeta^{Y_{\alpha+1}}$, a contradiction.

\ref{item:T=}: Let $T=\pi_{\alpha+1}(S_\alpha^\passive)$.
Note $T\pins_\card N_{\alpha+1}|\rho_{m_{\alpha+1}}^{N_{\alpha+1}}$.
Now
$N_{\eta_\alpha-\om}^{Y_{\alpha+1}}=T_\alpha^\passive$
and $\OR^T>\pi_{\alpha+1}(\nu(S_\alpha))\geq\nu(T_\alpha)$.
So if $T_\alpha^\passive=N_{\alpha+1}||\OR^{T_\alpha}$,
then clearly $\OR^{T_\alpha}\leq\OR^T$, and
it follows that
$\prodstage^{Y_{\alpha+1}}_{\xi_{\alpha+1}}(T)\geq\eta_\alpha-\om$.
But otherwise by \ref{item:model_coherence_card}, \ref{item:model_coherence_non-card}, $T|\nu(T_\alpha)=N_{\alpha+1}|\nu(T_\alpha)$ is active
but $T_\alpha|\nu(T_\alpha)$ is passive, which implies
$\prodstage^{Y_{\alpha+1}}_{\xi_{\alpha+1}}(T)>\eta_\alpha-\om$.

\ref{item:pi_alpha_nu-low}: Suppose $M_{\alpha+1}$ is type 3
and $\nu(M_{\alpha+1})$ is $M_{\alpha+1}$-regular.
Then $\nu(M^*_{\alpha+1})$ is $M^*_{\alpha+1}$-regular,
by \ref{lem:nu-pres_elem}.
Let $(\gamma,\ell)=\pred^\Uu(\alpha+1,0)$.
Now $\pi^*_{\alpha+1}$ is non-$\nu$-low.
If $\alpha+1\in\dropset^\Tt$ this is by \ref{lem:when_psi_nu-low}.
So suppose $\alpha+1\notin\dropset^\Tt$.
Then $M^*_{\alpha+1}=M_\gamma$,
so $\nu(M_\gamma)$ is $M_\gamma$-regular, so by induction, $\pi_\gamma$ is non-$\nu$-low,
and $\pi^*_{\alpha+1}=\tau\com\pi_\gamma$ where $\tau$ is a $\Sigma_1$-elementary core embedding, which suffices.

As $\nu(M^*_{\alpha+1})$ is
$M^*_{\alpha+1}$-regular, $i^{*\Tt}_{\alpha+1}$ is $\nu$-preserving,
by \ref{lem:nu-pres_elem}. But
\[ i^\Uu_{(\gamma,l),(\alpha+1,0)}\rest(N^*_{\alpha+1})^\sq:(N^*_{\alpha+1})^\sq\to 
N_{\alpha+1}^\sq\]
is $\Sigma_1$-elementary, hence non-$\nu$-low, and by commutativity, it follows that
$\pi_{\alpha+1}$ is non-$\nu$-low.

Finally, \ref{item:delta_when_no_drop}: By \ref{lem:iterates_maintain_def_regularity}.

This completes our discussion of the successor stage, and hence, of the lifting process for the 
small part $\Tt_0=\Tt\rest(\alpha+1)$ of $\Tt$.
Now suppose that $\Tt$ is non-small, and
$\alpha$ is least such that $\exit^\Tt_\alpha$ is non-small.
Let $\Uu_0=\Uu\rest(\alpha+1,1)$, etc, be given as above, including
a weak $m_\alpha=\deg^\Tt(\alpha)$-embedding
\[ \pi_\alpha:M_\alpha\to N^{Y_\alpha}_{\xi_\alpha m_\alpha}.\]

\begin{case}\label{case:non-small_no_drop}
$\xi=\OR^Y$ and $[0,\alpha)_\Tt\inter\dropset^\Tt=\emptyset$
and $\delta^{M_\alpha}<\lh(E^\Tt_\alpha)$.

Let $Y'=Y_\alpha$ and $N'=N^{Y'}$.
So  $[0,\alpha)_\Tt$ and $[(0,0),(\alpha,0))_\Uu$ do not drop in model/degree, $m_\alpha=z\leq y$,
\[ \pi_\alpha(\delta^{M_\alpha})=\delta'=\delta^{N'}=\delta^{Y'}\leq\rho_y^{Y'} \]
and $\delta'$ is $\bfrSigma_{y}^{Y'}$-regular.
As $\Uu\rest(\alpha,1)$ is based on $Y|\delta^Y$
and $\delta'$ is a strong cutpoint of $Y'$,
the normal continuation of $\Uu$ via $\Sigma_Y$ yields an above-$\delta'$,
$(y,\om_1+1)$-strategy $\Sigma_{Y'}$ for $Y'$.
Note that \ref{lem:iterates_pres_fs_match} applies,
so $N'$ and $Y'$
are generically equivalent above $\delta'$ and
have matching $\rho_z,p_z$ for $z\leq y$,
and by \ref{lem:P-construction_tree_translation},
above-$\delta'$, $y$-maximal trees on $Y'$
correspond to those on $N'$. So let $\Sigma_{N'}$
be the above-$\delta'$ strategy induced by $\Sigma_{Y'}$. 
Let $\Sigma_{M_\alpha}$ be the neat $\pi_\alpha$-pullback
of $\Sigma_{N'}$ (see \ref{dfn:neat_copy}; here we lift $z$-maximal
trees on $M_\alpha$ to $y$-maximal on $N'$). We use $\Sigma_{M_\alpha}$ to
iterate $M_\alpha$ above $\delta^{M_\alpha}$; clearly this suffices.
\end{case}

\begin{case}Otherwise.

We split into two subcases (though one could have treated them as one);
we don't give all the details as they are similar to those in the previous case.

\begin{scase}\label{case:non-small_no_res}
$M_\alpha$ is non-small and $\delta^{M_\alpha}<\lh(E^\Tt_\alpha)$.

Let $s\leq m_\alpha$ be largest such that $\delta^{M_\alpha}<\rho_s^{M_\alpha}$.
We define an above-$\delta^{M_\alpha}$, $(s,\om_1+1)$-strategy for $M_\alpha$.
Let $P=N^{Y_\alpha}_{\xi_\alpha s}$ and $\tau=\tau^{Y_\alpha}_{\xi_\alpha m_\alpha s}$. So
\[ \psi\eqdef\tau\com\pi_\alpha:M_\alpha^\sq\to P^\sq\]
is a weak $s$-embedding, $P$ is non-small
and $\psi(\delta^{M_\alpha})=\delta^P<\rho_s^P$.
By \ref{cor:local_Woodins}, $t_{\xi_\alpha}^{Y_{\alpha}}=3$, so $\delta^P=\delta^{Y_\alpha|\xi_\alpha}$ is a $Y_\alpha$-cardinal,
and $P=\mathscr{P}^{Y_\alpha|\xi_\alpha}(P|\delta^P)$.
We now lift above-$\delta^{M_\alpha}$ $s$-maximal trees $\Tt_1$ on $M_\alpha$
to above-$\delta^P$ $s$-maximal trees $\Uu_1'$ on $P$,
which translate to above-$\delta^{Y_\alpha}$ $s$-maximal trees $\Uu_1$ on $Y_\alpha$,
with $\Uu=\Uu_0\conc\Uu_1$ via $\Sigma_Y$.
(Here we use the neat $\pi_\alpha$-pullback for $\Tt_1$.)
Clearly $\Tt=\Tt_0\conc\Tt_1$ is $z$-maximal and $\Uu=\Uu_0\conc\Uu_1$
is $y$-maximal.
\end{scase}

\begin{scase}\label{case:non-small_res}
Otherwise (either $M_\alpha$ is small or $\lh(E^\Tt_\alpha)<\delta^{M_\alpha}$).

Since $S_\alpha=\exit^\Tt_\alpha$ is non-small, $S_\alpha\pins M_\alpha$.
By the case hypothesis there is $Q\pins M_\alpha$ which is a Q-structure for $\delta'=\delta^{S_\alpha}$.
So letting $(S,s)=\psegdeg^{M_\alpha}(S_\alpha,\delta')$,
we have $S_\alpha\ins S\pins M_\alpha$ and $\rho_{s+1}^S\leq\delta'<\rho_s^S$ (here $\delta'<\rho_0^S$ as $\delta'$ is the least Woodin of $S$).
We define an above-$\delta'$, $(s,\om_1+1)$-strategy for $S$.
Let
\[ \Vv=\critrestree^{Y_\alpha y}_{\xi_\alpha m_\alpha\pi_\alpha}(S_\alpha,\delta'), \]
\[ l=\critresl^{Y_\alpha}_{\xi_\alpha m_\alpha\pi_\alpha}(S_\alpha,\delta') \]
\[ \eta=\critresprodstage^{Y_\alpha}_{\xi_\alpha m_\alpha\pi_\alpha}(S_\alpha,\delta') \]
\[ \psi=\critresmap^{Y_\alpha}_{\xi_\alpha m_\alpha\pi_\alpha}(S_\alpha,\delta') \]
and $Y_{\alpha i}=M^\Vv_i$. So $l>0$ and
and $t^{Y_{\alpha l}}_{\eta}=3$ and letting $P=N^{Y_{\alpha l}}_{\eta s}$, 
$\psi:S^\sq\to P^\sq$
is a weak $s$-embedding and $\delta^P=\psi(\delta^S)<\rho_s^P$.
We now proceed like before,
lifting above-$\delta^{S}$ $s$-maximal trees $\Tt_1$ on $S$
to trees $\Uu_1'$  on $P$ and their translation $\Uu_1$ on $Y_{\alpha l}|\eta$.
We set $\Uu=\Uu_0\conc\Vv\conc\Uu_1$.
Here $\Uu$ is $y$-maximal;
we have $\lh(E^\Vv_i)<\delta^P$.
 \end{scase}
\end{case}

This completes the construction of the $(z,\om_1+1)$-strategy $\Sigma_M$ for $M$.

Now suppose that $Y$ is $(y,\om_1,\om_1+1)^*$-iterable, via strategy $\Sigma'_Y$,
and $\Sigma_Y$ is its first round. We want to construct a $(z,\om_1,\om_1+1)^*$-strategy
$\Sigma'_M$ for $M$ extending $\Sigma_M$.
So suppose we have $\Tt$ produced as above (possibly non-small), of successor length $<\om_1$,
and $\Uu$ on $Y$ is its lift. Let
\[\pi_\infty:M_\infty\to N_\infty=N_{\xi_\infty m_\infty}^{Y_\infty} \]
be the final lifting map, so $\pi_\infty$ is an $X_\infty$-weak $m_\infty=\deg^\Tt(\infty)$-embedding.
Let $y_\infty=\deg^\Uu(\infty)$.
Note that $(M_\infty,Y_\infty,\xi_\infty,m_\infty,\pi_\infty,X_\infty)$
satisfies the conditions
\ref{item:if_pi_nu-low_then_nu_sing}
and  \ref{item:if_OR^Y=xi_then_delta_reg} we placed on $(M,Y,\xi,m,\pi,X)$ at the outset.
So we form the next round of $\Sigma_M'$ by lifting to the next round of $\Sigma_Y'$.
This extends through transfinitely many rounds, and completes the proof of part
\ref{item:Q-mouse_it} of the lemma.

The proof of part \ref{item:general_mouse_it} is a simplification, with the only difference that 
there might not be a Woodin cardinal in the original model, but this is not important
as we assume that $\aleph_\xi^Y<\rho_0^Y$, and hence $\CC^Y\rest(\xi+1)\in Y$.

For part \ref{item:pseudo-premouse_it},
we have a generalized-pseudo-pm $M=(N,G)$ where $N=N_\xi^Y$, $t^Y_\xi=0$,
 $\aleph_{\xi+\om}^Y\leq\OR^Y$ and
$G^*\in\es^Y$ is such that $G\rest\nu_G\sub G^*$ and $\nu_{G^*}=\aleph_\xi^Y$.
Take $G^*$ with $\lh(G^*)$ minimal such that $G^*$ is as above 
(inducing the given $G$).
We only want a $(0,\om_1+1)$-strategy for $M$.
Recall  we do not squash $M$. We proceed basically as before. But if
$(0,\alpha]_\Tt\inter\dropset^\Tt=\emptyset$, we will have
\[ \pi_\alpha:M_\alpha\to N_\alpha=i^\Uu_{(0,0),(0,\alpha)}(M),\]
and $\pi_\alpha$ is a weak $0$-embedding and $\pi_\alpha(\lgcd(M_\alpha))=\lgcd(N_\alpha)$.
If $E^\Tt_\alpha=F^{M_\alpha}$ then we set $k_\alpha=0$ and 
$E^\Uu_{00}=i^\Uu_{(0,0),(\alpha,0)}(G^*)$. In this situation,
the minimality of $\lh(G^*)$ specified above ensures condition
\ref{item:U_alpha_i_objects}\ref{item:E_alpha_k_alpha}
(not literally by \ref{lem:con_coherence}, but by its proof).
If instead $\exit^\Tt_\alpha\pins M^\Tt_\alpha$ then we resurrect
an ancestor as before.

For part \ref{item:bicephalus_it}, we have a bicephalus $B=(N,F^0,F^1)$,
where $N=N_\xi^Y$ and $t_\xi^Y=0$, and background extenders $G^0,G^1$ for $F^0,F^1$ with $\nu_{G_i}=\aleph_\xi^Y$,
and $\lh(G^0),\lh(G^1)$ minimal with respect to these properties.
Recall we are allowing mixed type bicephali (for example, $(N,F)$ is type 2 and $(N,G)$ is type 1 or 3).
In this case we iterate $B$ at the unsquashed level, forming ultrapowers and shifting the active 
predicates just as for type 2 extenders. Let us consider such unsquashed iterations.
Recall the rules for iteration specified in \cite{extmax}:
If $E^\Tt_\alpha=F^0_\alpha$ then we set
\[ \nu^\Tt_\alpha=\max(\lgcd(B_\alpha),\nu(F^0_\alpha)), \]
and similarly for $F^1_\alpha$. In other cases set 
$\nu^\Tt_\alpha=\nu(E^\Tt_\alpha)$. We use $\nu^\Tt_\alpha$ as the exchange ordinal for determining 
$<_\Tt$.  Now if 
$(0,\alpha]_\Tt\inter\dropset^\Tt=\emptyset$ we will have 
$B^\Tt_\alpha=(M_\alpha,F^0_\alpha,F^1_\alpha)$ and
\[ \pi_\alpha:B^\Tt_\alpha\to C_\alpha=i^\Uu_{(0,0),(\alpha,0)}(B) \]
(the domain of $\pi_\alpha$ is $M_\alpha$)
which is a weak $0$-embedding as a map
\[ (M_\alpha,F^i_\alpha)\to(\univ{C_\alpha},F^i(C_\alpha)),\] and 
$\pi_\alpha(\lgcd(B_\alpha))=\lgcd(C_\alpha)$.
If $E^\Tt_\alpha=F^0_\alpha$, we set $k_\alpha=0$ and $E^\Uu_{\alpha 0}=i^\Uu_{(0,0),(\alpha,0)}(G^0)$.
We have 
$\pi_\alpha(\nu^\Tt_\alpha)=\nu(F^0(C_\alpha))$, by the proof of \cite[Theorem 5.3]{extmax}
and because 
$(N,F^0)$ is a premouse and by elementarity of $i^\Uu_{(0,0),(\alpha,0)}$.
This helps ensure that $<_\Tt$ 
matches $<_\Uu$ as in the proof of part \ref{item:Q-mouse_it}.

In the case that both $(N,F^0)$ and $(N,F^1)$ are type 3,
we squash and always use $\nu(E^\Tt_\alpha)$ as the exchange ordinal.
Suppose $E^\Tt_\alpha=F^0_\alpha$. Then $k_\alpha=0$. If $\pi_\alpha$ is non-$\nu$-low as a map
\[ (M_\alpha,F^0_\alpha)\to(\univ{C_\alpha},F^0(C_\alpha)),\]
we set $E^\Uu_{\alpha 0}=i^\Uu_{(0,0),(\alpha,0)}(G^0)$;
otherwise letting $\wt{\nu}=\psi_{\pi_\alpha}(\nu(F^0_\alpha))$,
we set $E^\Uu_{\alpha 0}=F^{Y_\alpha}_{\wt{\xi}}$
where $F(N_{\wt{\xi}}^{Y_\alpha})=F^0(C_\alpha)\rest\wt{\nu}$.
This again helps ensure that tree orders match.

This completes the proof of iterability.
\end{proof}

\subsection{Lifting and resurrection for essentially $m$-maximal trees}

We need to consider the following slight generalization of $m$-maximal trees:

\begin{dfn}\label{dfn:essentially_m-max}
 Let $M$ be an $m$-sound premouse. Let $\Tt$ be an iteration tree on $M$.
 We say that $\Tt$ is \emph{essentially $m$-maximal} iff (i) $\Tt$ satisfies the conditions for 
$m$-maximality excluding the monotone length condition (that is, that
\[ \lh(E^\Tt_\alpha)\leq\lh(E^\Tt_\beta)\text{ for all }\alpha+1<\beta+1<\lh(\Tt)\text{)},\]
and (ii) the monotone $\nu$ condition holds:
\[ \nu(E^\Tt_\alpha)\leq\nu(E^\Tt_\beta)\text{ for all }\alpha+1<\beta+1<\lh(\Tt).\]
We say that $\Tt$ is \emph{sse-$m$-maximal} (\emph{superstrong-essentially}) iff $\Tt$
is essentially $m$-maximal and for all $\alpha+2<\lh(\Tt)$,
if $\lh(E^\Tt_{\alpha+1})<\lh(E^\Tt_\alpha)$ then $E^\Tt_\alpha$ is superstrong.

An \emph{essentially-$(m,\alpha)$-iteration strategy} is a strategy which applies to essentially $m$-maximal trees of length $<\alpha$.
\emph{Essentially $(m,\alpha)$-iterability} requires the existence of such a strategy.
Likewise for \emph{sse} replacing \emph{essentially}.

Given an essentially $m$-maximal tree $\Tt$ and $\alpha+1<\lh(\Tt)$, we say that $E^\Tt_\alpha$, or $\alpha$, is \emph{$\Tt$-stable} iff 
$\lh(E^\Tt_\alpha)\leq\lh(E^\Tt_\beta)$ for all $\beta+1<\lh(\Tt)$ such that $\alpha<\beta$.
\end{dfn}

\begin{rem}
Note that for $\alpha+1<\beta+1<\lh(\Tt)$, we have $\lh(E^\Tt_\alpha)=\lh(E^\Tt_\beta)$ iff 
$\beta=\alpha+1$ and $E^\Tt_\alpha$ is superstrong type with $\crit(E^\Tt_\alpha)=\kappa$ and 
$M^{*\Tt}_{\alpha+1}$ is type 2 with $\kappa=\lgcd(M^{*\Tt}_{\alpha+1})$. So let $\Tt$ be 
essentially $m$-maximal. Given $\alpha+1<\lh(\Tt)$,
note there is a unique $n<\om$ such 
that
\begin{enumerate}[label=\tu{(}\roman*\tu{)}]
 \item  $\alpha+n+1<\lh(\Tt)$,
\item $\lh(E^\Tt_\beta)\geq\lh(E^\Tt_\gamma)$
for all $\beta,\gamma$ with $\alpha\leq\beta<\gamma\leq\alpha+n$,\footnote{So by the previous remarks,
if $\lh(E^\Tt_\beta)=\lh(E^\Tt_\gamma)$ then $\gamma=\beta+1$.}
\item if
$\alpha+n+2<\lh(\Tt)$ then $\lh(E^\Tt_{\alpha+n})\leq\nu(E^\Tt_{\alpha})$.
\end{enumerate}
Moreover, if $\Tt$ is sse-$m$-maximal then $E^\Tt_\beta$
is superstrong for all $\beta\in[\alpha,\alpha+n)$,
 and so if $\alpha\leq\beta<\gamma<\alpha+n$
then $\lh(E^\Tt_\beta)>\lh(E^\Tt_\gamma)$.
\end{rem}
\begin{lem}\label{lem:m-max_iter_equiv_ess_m-max_iter}
Let $M$ be an $m$-sound premouse. Then \tu{(}i\tu{)} $M$ is $(m,\om_1+1)$-iterable iff \tu{(}ii\tu{)} $M$ is 
essential-$(m,\om_1+1)$-iterable iff \tu{(}iii\tu{)} $M$ is sse-$(m,\om_1+1)$-iterable.
\end{lem}
\begin{proof}
It suffices to see that (i) implies (ii).
Let $\Tt$ be any putative essentially $m$-maximal tree. Note that for each $\alpha+1<\lh(\Tt)$ 
there is $n<\om$ such that either $\alpha+n+2=\lh(\Tt)$, or 
$\lh(E^\Tt_{\alpha+n})<\lh(E^\Tt_{\alpha+n+1})$, and hence, 
$\lh(E^\Tt_{\alpha+n})<\lh(E^\Tt_\beta)$ for all $\beta>\alpha+n$ with $\beta+1<\lh(\Tt)$. 
Let $X$ be the set of all $\alpha<\lh(\Tt)$ such that if $\alpha+1<\lh(\Tt)$ then 
$\alpha$ is $\Tt$-stable.
Let $\varphi:\zeta\to X$ be the increasing enumeration of $X$.
Note then that we get an $m$-maximal tree $\Tt'$ of length $\zeta$ with 
$E^{\Tt'}_\alpha=E^\Tt_{\varphi(\alpha)}$ and with $\varphi$ tree order, drop and degree structure, 
models, and embeddings. (That is, $M^{\Tt'}_\alpha=M^\Tt_{\varphi(\alpha)}$, etc.)
Using this it is easy to prove the lemma.
\end{proof}

\begin{dfn}\label{dfn:lift^Tt,Y}
Given a $(y,\om_1+1)$-strategy $\Sigma$,
$\Sigma^\sse$ denotes the sse-$(y,\om_1+1)$-strategy induced by $\Sigma$
via the procedure implicit in the proof above.

Adopt the hypotheses and notation of \ref{lem:iterability} part  \ref{item:general_mouse_it} or \ref{item:Q-mouse_it}.
Let $\Sigma_N$ be the $(m,\om_1+1)$-strategy for $N$ defined in the proof of \ref{lem:iterability}.
 (Note that $\Sigma_N$ depends on $\Sigma_Y$.)
Let $\Tt$ on $N$ be via $\Sigma_N$. Then $\lifttree^{\Tt,Y,y}$, abbreviated $\lifttree^{\Tt,Y}$,
denotes the corresponding tree $\Uu$ on $Y$ given in the proof of \ref{lem:iterability}.
(That is, we start with $M=N$ and $\pi=\id$.)

Let $\Lambda^\sse_N$ denote the following sse-$(z,\om_1+1)$-strategy for $N$.
A tree $\Tt=\Tt_0\conc\Tt_1$ with small component $\Tt_0$ is via $\Lambda^\sse_N$ iff
$\Tt_0$ is via $\Sigma_N$, and letting $\Uu_0=\lifttree^{\Tt_0,Y}$
 and $\pi:M^{\Tt_0}_\infty\to M^{\Uu_0}_\infty$ be the ultimate lifting map,
then $\Tt_1$ is given by copying (***with correct process; ***what does this mean exactly?
or should we just make $\Lambda^\sse_N=(\Sigma_N)^\sse$?) to a tree on $M^{\Uu_0}_\infty$
via the tail of $\Sigma^\sse_Y$.
 We write $\Uu=\lifttree^{\Tt,Y}$.
 
Suppose $\Tt$ is small and let $\Uu=\lifttree^{\Tt,Y}$.
 Let
 $\left<M_\alpha,m_\alpha,Y_\alpha,\xi_\alpha,N_\alpha,\pi_\alpha\right>_{\alpha<\lambda}$
be as in \ref{lem:iterability}.  Let $P=M^\Uu_{\alpha 0}$.
We write $\prodstage^{\Tt,\Uu}_\alpha=\xi_\alpha$\index{$\prodstage^{\Tt,\Uu}$} and $\pi^{\Tt,\Uu}_\alpha=\pi_\alpha$\index{$\pi^{\Tt,\Uu}$} and
(recall each of the following objects is a function with domain
$\segdegs(M_\alpha,m_\alpha)$):
\begin{enumerate}[label=--]
\item $\res^{\Tt,\Uu}_\alpha=\res^{P}_{\xi_\alpha m_\alpha\pi_\alpha}$,\index{$\res^{\Tt,\Uu}$}
\item $\resmap^{\Tt,\Uu}_\alpha=\resmap^P_{\xi_\alpha m_\alpha\pi_\alpha}$,\index{$\resmap^{\Tt,\Uu}$}
\item $\restree^{\Tt,\Uu}_{\alpha}=\restree^P_{\xi_\alpha m_\alpha\pi_\alpha}$,\index{$\restree^{\Tt,\Uu}$}
\item $\restree^{\Tt,\Uu,z}_\alpha=\restree^{P,z}_{\xi_\alpha m_\alpha\pi_\alpha}$,\index{$\restree^{\Tt,\Uu,z}$}
\item $\resprodstage^{\Tt,\Uu}_\alpha=\resprodstage^P_{\xi_\alpha m_\alpha\pi_\alpha}$,\index{$\resprodstage^{\Tt,\Uu}$}
\item $\resl^{\Tt,\Uu}_\alpha=\resl^P_{\xi_\alpha m_\alpha\pi_\alpha}$.\index{$\resl^{\Tt,\Uu}$}
\end{enumerate}
Similarly (a function with domain $\{e\mid e\ins M_\alpha\text{ and }e\text{ is active}\}$):
\begin{enumerate}[label=--]
 \item $\exitresadd^{\Tt,\Uu}_\alpha=\exitresadd^P_{\xi_\alpha m_\alpha\pi_\alpha}$.\index{$\exitresadd^{\Tt,\Uu}$}
\end{enumerate}

Let $\critres^{\Tt,\Uu}_\alpha$, etc, be defined analogously.

If $\Tt$ has successor length $\alpha+1$ then a lack of subscript abbreviates subscript $\alpha$,
so $\prodstage^{\Tt,\Uu}=\prodstage^{\Tt,\Uu}_\alpha$, etc.
\end{dfn}

The following lemma will be used in the proof of lower semicontinuity:

 \begin{lem}\label{lem:active_and_sub_projector_res_tree_agmt}Let $Y$ be $(x,\xi,m)$-good
 with 
 $t^Y_{x\xi}\neq 2$.
 Let
 $N=N^Y_{\xi m}$ and
 $\pi:M^\sq\to N^\sq$ a weak $m$-embedding
 such that if $\pi$ is $\nu$-low (hence $M$ is active type 3) then $\nu^M$ is $M$-singular.
 Let 
 $S\ins S'\pins R\ins M$ with $R$ active and
  $\rho_\om^S=\rho_\om^{S'}=\lgcd(R)$.
 Let $\Vv_R=\restree^Y_{\xi m\pi}(R)$
 and $\Vv_S=\restree^Y_{\xi m\pi}(S)$
 and $\Vv_{S'}=\restree^Y_{\xi m\pi}(S')$.
 Then  $\Vv_R\ins\Vv_S$ (so $\Vv_R\ins\Vv_{S'}$) and either:
 \begin{enumerate}[label=--]
  \item $\Vv_R,\Vv_S,\Vv_{S'}$ use the same non-empty extenders with the same indexing 
   ($\Vv_S$ might be properly longer than $\Vv_R$, but only by incorporating more padding), or
   \item 
 $R$ is type 2 and $\Vv_S=\Vv_{S'}=\Vv_R\conc\left<E\right>$
 where $E$ is an order $0$ measure
 ($\crit(E)$ is the relevant image of $\lgcd(R)$).
 \end{enumerate}
 \end{lem}
 \begin{proof}
 \begin{case} $R$ is type 3.

 \begin{scase}\label{scase:R=M}
$R=M$.

 So $\Tt_R$ is trivial; we claim that $\Tt_S$ uses no extenders. Let $\tau:N^\sq\to (N^Y_{x\xi})^\sq$ be the core map. Let $\psi=\psi_{\tau\com\pi}$.
 
 Suppose $\pi$ is $\nu$-low, so
 by assumption,
 $\nu^M$ is $M$-singular.
 Then $\tau\com\pi$ is $\nu$-low,
 and $\psi(\nu^M)=\rho_\om(\psi(S))$,
 which is $N^Y_{x\xi}$-singular,
 hence $Y$-singular, so this point is not measurable in $Y$,
 so $\Vv_S$ uses no extenders.
 
 Now suppose $\tau\com\pi$ is $\nu$-preserving. Then $\psi(\nu^M)=\nu(N^Y_{x\xi})=\rho_\om(\psi(S))$. Let $\zeta$ be such that $\psi(S)=\core_\om(N_{x\zeta}^Y)$.
 Suppose $t_{x\zeta}^Y=2$.
 Then $\kappa=\nu(F(N^Y_{x\xi}))$
 is $Y$-measurable.
 But then since $\rho_0(N^Y_{x\xi})=\kappa$,
 we have $\rho_\om(N^Y_{x\xi})=\kappa$,
 so $N=N_{x\xi n}^Y=N_{x\xi}^Y$,
 and $\tau=\id$.
 Since $t_{x\zeta}^Y=2$ and $N_{x\zeta}^Y|\kappa=N_{x\xi}^Y|\kappa$,
 note that $t_{x\xi}^Y=2$ also, contradicting our assumption.
 
 So $\tau\com\pi$ is $\nu$-high.
 So $\nu(F(N^Y_{x\xi}))<\psi(\nu^M)=\rho_\om(\psi(S))$. 
 Let $S'$ be the hull of $\psi(S)$
 used in the resurrection in this case,
 so $\rho_\om^{S'}=\nu(F(N^Y_{x\xi}))$.
 Since $\Vv_S$ uses an extender $E$,
 $\crit(E)=\rho_\om^{S'}$.
 So as before, $\rho_\om(N_{x\xi})=\rho_0(N_{x\xi})=\rho_\om(S')$ and $t^Y_{x\xi}=2$, contradiction.
 \end{scase}

\begin{scase}
 $R\pins M$ and $\rho_\om^R<\nu(F^R)=\rho_\om^S$.
  
  Easily, $\Tt_R\ins\Tt_S$.
  And since $\rho_\om^R<\nu(F^R)$,
  like in the previous subcase,
  $R$ lifts/resurrects to a type 3 $N_{x\alpha}^Y$ with $\nu(F(N_{x\alpha}^Y))$ non-$Y$-measurable, and this implies that $\Tt_S$ uses no extenders beyond those in $\Tt_R$.
  \end{scase}
  \begin{scase}
 $R\pins M$ and $\rho_\om^R=\nu(F^R)=\rho_\om^S$.
 
 If $M$ is type 3 and $\rho_\om^R=\rho_\om^S=\nu(F^M)$,
 then by Subcase \ref{scase:R=M},
 $\Tt_R,\Tt_S$ both use no extenders.
 If $\rho_\om^R=\rho_\om^S<\rho_0^M$
 and $\rho_\om^R=\rho_\om^S$
 is a cardinal of $M$,
 then $\Tt_R$ uses a measure $E$
 iff $\kappa=\psi(R)|\psi(\rho_\om^R)$
 is $Y$-measurable and $\psi(R)=N^{Y}_{x\kappa}$, and then $E=F_{\kappa 0}^Y$.
 Likewise for $S$,
 which suffices.
 
 Suppose $\rho_\om^R=\rho_\om^S$ is a non-$M$-cardinal.
Let $T\pins M$ be least such that $R\ins T$ and $\rho_\om^T<\rho_\om^R$;
so $T$ is also least such that $S\ins T$ and $\rho_\om^T<\rho_\om^S$.
Then $\Vv_T\ins\Vv_R$ and $\Vv_T\ins\Vv_S$.
If $\rho_\om^R=\rho_\om^S<\rho_0^T$
then we therefore $\Vv_R=\Vv_S$ basically like in the previous paragraph.
Otherwise, $T$ is type 3 and $\rho_\om^R=\rho_\om^S=\nu(F^T)$. But then
like before, $T$ lift/resurrects to some
type 3 stage $N_{x\alpha}^Y$ with $\nu(F(N^Y_{x\alpha}))$ non-$Y$-measurable,
and so $\Vv_R=\Vv_T\conc\left<\emptyset\right>\conc\Vv_T$.
\end{scase}\end{case}

\begin{case}
 $R$ is type 2.
 
 Using the previous case,
 this case is straightforward
 and left to the reader.
 (For example, we can use the previous case to understand the situation that $\rho_\om^R=\rho_\om^S=\nu(F^T)$
 where $T\ins M$, $T$ is type 3
 and either $T=M$ or $\rho_\om^R<\rho_\om^R$. Because $R$ is type 2, the resurrection map $\pi:R\to R^*$
 is used to shift $S,S'$,
 so $\rho^*=\rho_\om(\pi(S))=\rho_\om(\pi(S'))$
 and $\pi(S)|\rho^*=\pi(S')|\rho^*$, so $\Vv_S$ uses an extender $E$ beyond those in $\Vv_R$ iff $\Vv_{S'}$ does, and when they do use such an $E$, $E$ is
  just the order $0$ measure on $\rho^*$.)
\end{case}

\begin{case}
 $R$ is type 1.
 
 In this case, clearly there are no measures
 used in $\Vv_S$ whose critical point is an image of  $\rho_\om^S=\lgcd(R)=\mu^{+R}$ where $\mu=\crit(F^R)$. We omit further discussion.\qedhere
\end{case}
\end{proof}

\subsection{$\CC$ in Q-mice}

Recall that \emph{Q-mouse} and \emph{Q-degree} were introduced in  \ref{dfn:Q-prepair}
(and this was the special case of \emph{Q-pair} where $\delta=\delta^Y$).

\begin{lem}
Let $Y$ be a $\delta^Y$-sound Q-mouse with $\rho_{m+1}^Y\leq\delta=\delta^Y<\rho_m^Y$.
Let $x\in\RR\cap Y$
and $N=N_{x\Omega}^Y$ where $\Omega=\OR^Y$.
Then:
\begin{enumerate}\item\label{item:N_m-iterable_etc} $N$ is a
$\delta$-sound $m$-sound  $(m,\om_1+1)$-iterable non-small premouse
with $\rho_{m+1}^N\leq\delta=\delta^N$,
\item\label{item:N_is_Q-mouse} $N$ is a Q-mouse.
\end{enumerate}
In particular, if $\rho_{m+1}^N=\delta^N$ \tu{(}so $N$ is $(m+1)$-sound\tu{)} and $\delta^N$ is 
$\bfrSigma_{m+1}^N$-regular,
then $N$ is $(m+1,\om_1+1)$-iterable.
\end{lem}
\begin{proof}
 Part \ref{item:N_m-iterable_etc} is just by Lemma \ref{lem:iterability} and the material in \S\ref{sec:P-construction}.

 Part \ref{item:N_is_Q-mouse}:
If $\delta$ is $\bfrSigma_{m+1}^N$-singular then $(m,\om+1)$-iterability suffices.
 
 So suppose $\delta$ is $\bfrSigma_{m+1}^N$-regular, so in particular by Lemma \ref{lem:hull_cof_in_delta}, $\rho_{m+1}^N=\delta$.
 
 \begin{clm}$N$ is $(m+1,\om_1+1)$-iterable.\end{clm}
 \begin{proof}
 If $\delta$ is $\bfrSigma_{m+1}^Y$-regular then $\delta=\rho_{m+1}^Y$ and $Y$ is $(m+1,\om_1+1)$-iterable
 (as we are assuming that $Y$ is a Q-mouse,
 so $Y$ is $(q^Y,\om_1+1)$-iterable),
 which again by Lemma \ref{lem:iterability} implies that $N$ is $(m+1,\om_1+1)$-iterable. So suppose $\delta$ is 
$\bfrSigma_{m+1}^Y$-singular.

\begin{sclm}\label{sclm:cof^N_m+1(rho_m)}$\cof^{\bfrSigma_{m+1}^N}(\rho_m^N)=\delta$.\end{sclm}

\begin{proof} Let us officially assume $m>0$, and leave the other (similar but in the usual way slightly different) case to the reader.
Let $\mu<\delta$ and $x\in\OR^Y$ be such that
\begin{equation}\label{eqn:Hull_of_mu_cup_x_cof_in_delta} \delta\cap\Hull_{m+1}^Y(\mu\un\{x\})\text{ is cofinal in }\delta.\end{equation}

Let $\BB=\BB^N_\delta$. Let $G\sub\BB$ be the generic determined by $Y|\delta$.
We have that for each $\rSigma_{m+1}$ formula $\varphi$
and $\alpha<\mu$ and $\gamma<\delta$,
\[ Y\sats\varphi(x,\alpha,\gamma)\ \iff\ \exists\xi<\rho_m^N\ \exists p\in 
G\ \Big[ N\sats p\forces_0\psi_\varphi(\check{x},\check{\alpha},\check{\gamma},\dot{t}_\xi)\Big], \]
where $\forces_0$ is the $\rSigma_0$ forcing relation for $\BB$
and $\dot{t}_\xi$ is the canonical
$\BB$-name for $\Th_{\rSigma_m}^{N[\dot{G}]}(\xi\cup\{\pvec_m^N,x,\alpha,\gamma\})$
and $\psi_\varphi$ asserts
that ``$\dot{t}_\xi$ codes a witness to $\varphi(x,\alpha,\gamma)$''. Note that the statement ``codes a witness'' can be taken just $\Sigma_0$
in the language of set theory (we're assuming $m>0$); the  terminology \emph{codes a witness}
is like in \cite[\S5]{V=HODX_pub}.

For each $\xi<\rho_m^N$ let $\gamma_\xi$ be the supremum of all
$\gamma<\delta$ such that there is $p\in\BB$ and an $\rSigma_{m+1}$ formula $\varphi$ and $\alpha<\mu$ 
such that $\gamma$ is the unique $\sigma$ such that
\[ N\sats p\forces_0\psi_\varphi(\check{x},\check{\alpha},\check{\sigma},\dot{t}_\xi) \]
Note that for each individual $\xi$, this computation is made with sets in $N$, so  $\gamma_\xi<\delta$.
But $\delta=\sup_{\xi<\rho_m^N}\gamma_\xi$, by line (\ref{eqn:Hull_of_mu_cup_x_cof_in_delta}).
Since $\delta$ is $\bfrSigma_{m+1}^N$-regular, this completes the 
proof of the subclaim.
\end{proof}

By Subclaim \ref{sclm:cof^N_m+1(rho_m)} and very much as in the proof of \ref{lem:Q_extra_it}, every $\bfrSigma_{m+1}^{N}$ function $f:\kappa\to N$, with domain $\kappa<\delta$,
is in fact $\bfrSigma_m^N$, and so
$(m+1)$-maximal trees on $N$ are equivalent to $m$-maximal trees,
so $N$ is $(m+1,\om_1+1)$-iterable. This proves the claim.
\end{proof}

By the claim and Lemma \ref{lem:Q_extra_it},
$N$ is  $(k,\om_1+1)$-iterable,
where $k\leq\om$ is largest such that $\delta$ is $\bfrSigma_k^N$-regular.
So if $\rho_k^N<\delta$ for some $k<\om$,
then we are done, so suppose $N$ is $\om$-sound with $\rho_\om^N=\delta$ and $\delta$ is regular in $\J(N)$. So $N$ is $(\om,\om_1+1)$-iterable.
We just need that $\delta$ is not Woodin in $\J(N)$. But otherwise, by the proof of Lemma \ref{lem:Woodin_exactness}, $\delta$ is Woodin in $\J(N)[G]$,  contradicting that $Y$ is a Q-mouse.
\end{proof}

\section{Pointclasses}\label{sec:pointclasses}

We now introduce two kinds of pointclasses, for which we will later define scales. The first (corresponding to \emph{$1$-$\CC$-closed} mouse operators)
will include for example $\Pi^1_3$. The second (corresponding to \emph{exactly reconstructing mouse operators})
will include for example $\bigcup_{n<\om}(\Game(\omega n\text{-}\Pi^1_1))$ (this pointclass is  $\Gamma_{\mathscr{M},\mathrm{nice}}$ where $\mathscr{M}(x)=x^\#$).
\subsection{$1$-$\CC$-closed mouse operators}

\begin{dfn}
 Let $M,Y$ be premice, with $M$ over $x\in\RR$ and $M$ sound.
 We say that $\CC_x^Y$ \emph{properly reaches} $M$
 iff there is $\alpha<\OR^Y$ such that $Y$ is $(x,\alpha+1)$-iterability-good
 and $\core_\om(N^Y_{x\alpha})=\core_0(M)$.
\end{dfn}

\begin{dfn}
Let $\Ll^\passive_\pm$ denote the set of formulas in
the language of passive premice (so $\Ll^\passive_\pm\sub\Ll_\pm$)
and $\Ll^\passive_{\Sigma_1}$ denote
the set of $\Sigma_1$ formulas in $\Ll^\passive_\pm$.
We write $\elem_{1}^\passive$ for $\Sigma_1$-elementarity with respect to $\Ll^\passive_\pm$.
Given a formula $\psi\in\Ll_\pm$ and $x\in\RR$,
let $M^\psi_x$ be the least $\om$-mouse over $x$ satisfying $\psi$,
if any such exists. (Recall that $\Ll_\pm$ includes
a symbol $\dot{x}$ interpreted as the base real $x$.)

 Let $\Mmm:\RR\to V$ be a function. We say that $\Mmm$ is \emph{$1$-$\CC$-closed}
iff:
\begin{enumerate}
\item $\Mmm_x=\Mmm(x)$ is an $x$-premouse which is passive of limit ordertype (that is, 
$\Mmm_x\neq\J(N)$ for any $N$), $\Mmm_x\sats$``$V=\HC$''
and is $(\om,\om_1+1)$-iterable.
\item There is no $N\pins\Mmm_x$ such that $N\elem_{1}^\passive\Mmm_x$.
 \item (Turing invariance) For all $x,y\in\RR$, if $x\equiv_T y$
 then $\Mmm_x$ is the rearrangement of $\Mmm_y$ as an $x$-premouse.
 \item\label{item:closure_downward} (Downward $\CC$-closure) For all $x,z\in\RR$, if $M\pins\Mmm_z$, 
$x\in\RR^M$
 and $\CC_x^M$ properly reaches an $\om$-mouse $N$ (an $x$-premouse), then $\Th_{\Sigma_1}^{\J(N)}\sub\Th_{\Sigma_1}(\Mmm_x)$.
 \item\label{item:closure_upward} (Upward $\CC$-closure) For all $x,z\in\RR$ 
with $x\in\Mmm_z$, for all $\om$-mice $M\pins\Mmm_x$, there is $N\pins\Mmm_z$ such that $\CC^N_x$ 
properly reaches $M$.
\end{enumerate}
We say that $\Mmm$ is \emph{admissible} iff $\Mmm_x$ is admissible for every $x\in\RR$.
\end{dfn}
\begin{exm}
A very simple example of an admissible $1$-$\CC$-closed operator is $\Mmm(x)=\J_\delta(x)$
where $\delta$ is least such that $\J_\delta(x)\elem^\passive_1 L[x]$. For in this case downward 
$\CC$-closure is automatic, as every $N$ as there is a segment of $L[x]$. Upward $\CC$-closure is 
also easy: given $x\in\Mmm_z\inter\RR$ and $M\pins\Mmm_x$ such that $\rho_\om^M=\om$,
certainly there is some $N\pins L[z]$ such that $x\in N$ and $\CC^N_x$ properly reaches $M$.
Let $N$ be least such.
Because $M\pins\Mmm_x$, we may assume that there is a $\Sigma_1$ formula $\psi$
such that $\J(M)\sats\psi$ but $M\not\sats\psi$. But because $x\in\Mmm_z$, this implies 
that there is a $\Sigma_1$ formula $\varphi$ such that $\J(N)\sats\varphi$ but $N\not\sats\varphi$.
So $N\pins\Mmm_z$, and we are done.
\end{exm}
\begin{exm}
 A second example is $\Mmm(x)=M_1(x)|\delta$,
 where $\delta$ is least such that $M_1(x)|\delta\elem^\passive_1 M_1(x)$.
 Here downward $\CC$-closure is because if $x\in M_1(z)\inter\RR$
 and $x\in S\ins M_1(z)$ then $\RR\inter N_{x\infty}^S\sub M_1(x)$;
 this is because every model produced by $\CC^S_x$ is $1$-small,
 because $S$ is $1$-small and by \ref{cor:local_Woodins}. Upward $\CC$-closure is
 because $M_1(x)\inter\RR\sub N_{x\infty}^{M_1(z)}$ (given $x\in M_1(z)$),
 and as in the previous example.
\end{exm}

\begin{dfn}
 Let $\Mmm$ be $1$-$\CC$-closed.
 
 For $\psi\in\Ll^\passive_{\Sigma_1}$, write
 \[ A_{\Mmm\psi}=\{x\in\RR^n\ |\ 
1\leq n<\om\ \&\ \Mmm_x\sats\psi\}.\]
Where convenient we will suppress $n$ and imagine that $n=1$.
 We write $A_\psi=A_{\Mmm\psi}$ if $\Mmm$ is clear from context.
 
Define the pointclass
$\Gamma_\Mmm=\{A_{\Mmm\psi}\mid\psi\in\Ll^\passive_{\Sigma_1}\}$.
\end{dfn}
The first main theorem of the paper, to be proved later, is:
\begin{tm}\label{thm:1-C-closed_scale_prop}Let $\Mmm$ be $1$-$\CC$-closed. Then $\Gamma_\Mmm$ has the scale 
property.\end{tm}

Right now we can observe some simple properties of $\Gamma_\Mmm$.

\begin{dfn}
Let $\Mmm$ be $1$-$\CC$-closed. Let $x,z\in\RR$ and $x\in\Mmm_z$.
Then $S^{\Mmm_z}_x$, or just $S^z_x$, denotes the stack of all $\om$-premice $M$ such that
for some $N\pins\Mmm_z$, $\CC^N_x$ properly reaches $M$.
By downward and upward closure, $\Mmm_x\elem^\passive_1 S^z_x$.
\end{dfn}
\begin{rem}\label{rem:S^Z_x}
Suppose $x\in\Mmm_z$. Then $S^{\Mmm_z}_x$ really is a stack because
if $N\pins\Mmm_z$ and $N\sats$``$\CC^N_x$ properly reaches
 $M$'' and $\core_\om(M)$ is an $\om$-premouse,
 then $\core_\om(M)$ is an $\om$-mouse.
 Thus, $S^{\Mmm_z}_x$ is $\Sigma_1^{\Mmm_z}(\{x\})$, uniformly in $x,z$.
 By upward $\CC$-closure, $\Mmm_x\ins S^{\Mmm_z}_x$.
 Therefore $\Mmm_x$ is $\Sigma_1^{\Mmm_z}(\{x\})$, uniformly in $x,z$
 (it is the stack of all $\om$-premice $M\pins S^{\Mmm_z}_x$
 such that for some $\psi\in\Ll^\passive_{\Sigma_1}$, $\J(M)\sats\psi$
 but $M\not\sats\psi$).
\end{rem}

\begin{rem}
We recall some notions from
\cite{mosch} (we state the definitions here just for $\RR$). Let $\Gamma$ be a pointclass. Then 
$\Gamma$ is \emph{$\om$-parametrized}
iff for each $n$ there is $U_n\sub\om\cross\RR^n$
such that $U_n\in\Gamma$ and $U_n$ is universal for $\Gamma\inter\pow(\RR^n)$,
meaning that for each $A\in\Gamma\inter\pow(\RR^n)$, there is $m<\om$ such that $A=(U_n)_m$.
Given a partial function $f:\RR^n\to\RR$ and $P\sub\RR^n\cross\seq$,
we say that $P$ \emph{codes $f$ on its domain} iff for every $x\in\dom(f)$
and $s\in\seq$, we have $s\sub f(x)$ iff $(x,s)\in P$. Similarly for $f:\RR^n\to\RR^m$.
We say that $\Gamma$ has the \emph{$\RR$-Substitution Property} iff whenever
$P,Q\in\Gamma$ and 
$f:\RR^n\to\RR^m$ is a partial function such that $P$ codes $f$ on its domain,
and $Q\sub\RR^m$, then there is $R\in\Gamma$ such that for all $x\in\dom(f)$, we have
$x\in R$ iff $f(x)\in Q$.
Given $A\sub\RR^n$, $A^c=\{x\in\RR^n\bigm|x\notin A\}$ denotes the \emph{complement} of $A$
(of course, the $n$ is ambiguous if $A=\emptyset$,
but this doesn't cause a real problem).
We write $\stackrel{\smile}{\Gamma}=\{A^{\mathrm{c}}\bigm|A\in\Gamma\}$ for the 
\emph{dual} of $\Gamma$ (automatically including $\RR^n\in\stackrel{\smile}{\Gamma}$ for all $n<\om$, if $\emptyset\in\Gamma$).
We write $\Delta_{\Gamma}=\Gamma\cap\stackrel{\smile}{\Gamma}$.
For $A\sub\RR^{n+1}$, let $A_x=\{\vec{y}\in\RR^n\bigm|(\vec{y},x)\in A\}$.
For $x\in\RR$, $\Gamma(x)$ denotes the \emph{relativized \tu{(}to $x$\tu{)}} pointclass $\{A_x\bigm|A\in\Gamma\}$.
Note that $(\stackrel{\smile}{\Gamma})(x)={\stackrel{\smile}{\Gamma(x)}}$;
so we just write $\stackrel{\smile}{\Gamma}(x)$
for this pointclass.
Note that  $\Delta_{\Gamma(x)}=\Gamma(x)\cap\stackrel{\smile}{\Gamma}(x)$.\footnote{
The usual notational convention is to write $\Delta_\Gamma(x)=\Gamma(x)\cap\stackrel{\smile}{\Gamma}(x)$,
although this is notationally inconsistent,
since according to the earlier definitions,
$\Delta_\Gamma(x)$ should be $(\Gamma\cap\stackrel{\smile}{\Gamma})(x)$,
which can differ from $\Gamma(x)\cap\stackrel{\smile}{\Gamma}(x)$.
We do not follow this convention,
but also do not contradict it, as we just write $\Delta_{\Gamma(x)}$ instead.}
In some contexts,  $\Delta_{\Gamma(x)}$ 
can be considered as restricted to  relations $\sub\om$.
In this connection, we say that $\Gamma$ is \emph{closed under $\ex x\in\Delta_{\Gamma(y)}$}
iff for all $P\in\Gamma$ with $P\sub\RR^{2+n}$,
we have $Q\in\Gamma$ where
\[ Q(y,z)\ \iff\ \ex x\in\Delta_{\Gamma(y)}[P(x,y,z)].\]
\end{rem}

\begin{lem}
 Let $\Mmm$ be $1$-$\CC$-closed. Let $\Gamma=\Gamma_\Mmm$. Then $\Gamma$ is
closed under disjunction, conjunction, $\ex n<\om$, $\all i<m$ \tu{(}where $m$ is a 
natural number variable\tu{)},  recursive substitution, and is
 $\om$-parametrized. Moreover, for each $y\in\RR$, we have $\RR\inter\Mmm_y\sub\Delta_{\Gamma(y)}$.
 
Suppose $\Mmm$ is admissible. Then $\Gamma$ is closed under $\all n<\om$,
 has the $\RR$-Substitution Property, and is closed under $\ex x\in\Delta(y)$.
 Moreover, for each $y\in\RR$, we have $\RR\inter\Mmm_y=\RR\cap\Delta_{\Gamma(y)}$.
\end{lem}
\begin{proof}
Let $\varphi,\psi\in\Sigma_1$.
To see closure under disjunction, simply note that
$A_{\varphi}\un A_\psi=A_{\varphi\vee\psi}$.
To see closure under conjunction, let $\varrho$ be the formula ``Either (i) $\varphi$
and some initial segment of me satisfies $\psi$, or (ii) $\psi$ and some initial segment of me 
satisfies $\varphi$'', and note that
$A_\varrho=A_{\varphi}\inter A_\psi$. Closure under $\all i<m$ (where the variable $m$ ranges 
over integers) and $\ex n<\om$ are similar, though there is a small technical detail in relation to 
Turing invariance, which we address below regarding $\om$-parametrization.

Now consider closure under recursive substitution.
Let $\psi$ be a formula and $f:\RR\to\RR$ a recursive function.
We want $\varphi\in\Ll_{\Sigma_1}$ such that
$x\in A_\varphi$ iff $f(x)\in A_\psi$.
But we get such a $\varphi$ readily from \ref{rem:S^Z_x}.

For $\om$-parametrization, fix a recursive enumeration $\left<\psi_n\right>_{n<\om}$ of 
$\Ll_{\Sigma_1}$. Then
$(n,x)\in A_\varphi$ iff $\Mmm_x\sats\psi_n$,
where $\varphi$ is the formula ``Let $\dot{x}=(n,x^*)$;
then $L[\es,x^*]\sats\psi_n$''. This is because by Turing invariance,
$\Mmm_x$ is just the rearrangement of $\Mmm_{(n,x)}$, for any $n<\om$.

To see that $\Mmm_y\inter\RR\sub\Delta_{\Gamma(y)}$, use the fact that 
there is no $N\pins\Mmm_y$ such that $N\elem_1\Mmm_y$,
and so $\Mmm_y=\Hull_1^{\Mmm_y}(\emptyset)$.

Now suppose that $\Mmm$ is admissible.
Closure under $\all n<\om$ is an immediate corollary (again using Turing invariance).
For $y\in\RR$ we then easily have $\Delta_{\Gamma(y)}\sub\Mmm_y$,
and therefore $\Delta_{\Gamma(y)}=\RR\inter\Mmm_y$. So closure under $\ex x\in\Delta_{\Gamma(y)}$ is trivial.
Finally consider the Substitution Property, and let $f,P,Q$ be as there. Note that if
$x\in\dom(f)$ then $y=f(x)\in\Delta_{\Gamma(x)}\sub\Mmm_x$.
But then using the $\Sigma_1$ formulas defining $P$ and $Q$ (as in the Substitution Property)
and by \ref{rem:S^Z_x}, it is easy to verify the property.
\end{proof}
\subsection{Exactly reconstructing mouse operators}

We now introduce a second kind of pointclass, associated to which we will also define scales. These pointclasses are finer than those corresponding to $1$-$\CC$-closed operators. The main ideas for the scale construction, and the proofs of its properties,
are, however, very similar. So for the reader more interested in getting most quickly to those ideas, this section is safe to skip for the present.

\begin{dfn}\label{dfn:good_theory}
An \emph{ordered theory} is a sequence $\left<\varphi_n\right>_{n<\om}$ of formulas.

Let $T$ be an ordered theory consisting of $\Sigma_1$ sentences
over the language $\Ll_{\pm}\cup\{\dot{p}\}$,
where $\dot{p}$ is a new constant symbol. If $M$ is a premouse and $\varphi$ a formula over this language, we write
\[ M\sats\varphi\iff(M,p_1^M)\sats\varphi, \]
that is, where $\dot{p}$ is interpreted by $p_1^M$.
We say that $T$ is \emph{good} iff:
\begin{enumerate}[label=--]
 \item 
for some $i\in\{0,1,2\}$, $T$ asserts that the universe is type $i$ (as a premouse),
\item for some $k<\om$, $T$ has the formula ``$\dot{p}\in[\OR]^k$'',
\item for every real $x$ there is an $x$-premouse $M$ such that:
\begin{enumerate}[label=--]
  \item $M\sats T$ \tu{(}so $M$ is non-type 3\tu{)},
  \item $M=\Hull_1^M(\{p_1^M\})$,
 \item  $M$ is non-small and if $p_1^M\neq\emptyset$ then $\delta^M\leq\min(p_1^M)$,
 \item $M$ is $(0,\om_1+1)$-iterable.
\end{enumerate}
\end{enumerate}
Note that any such
$M$ is sound with $\rho_1^M=\om$.
Note that if $T$ is good and $x\in\RR$ then there is a least such $M$ for $x$,
which we denote by $M_{Tx}$ or $M_x$. \footnote{Since $T$ is not a complete theory, it seems that there might be $R\pins M_x$ such that $(R,p_1^R)\sats T$ (but $\rho_1^R>\om$) --
it is not clear that $\cHull_1^R(\{p_1^R\})$ is sound.}

If $T$ is good, let $\Mmm=\Mmm_T$ be the operator $\Mmm(x)=M_{Tx}$.
We say that $T,\Mmm$ are \emph{reconstructing}
iff for all reals $x,y$, if $x\in M_y$ then $M_x=\core_1(N^{M_y}_{x\alpha})$ for some $\alpha\leq\OR^{M_y}$.
We say that $T,\Mmm$ are \emph{exactly reconstructing} iff $M_x=\core_1(N^{M_y}_{\infty x})$ for all such $x,y$.
 We say that $T,\Mmm$ are \emph{uniformly $1$-solid}
iff there is a fixed $\Sigma_1$ term $w$ such that for every $x$, $w^{M_x}(p_1^{M_x})$ is a generalized $1$-solidity witness for $(M_x,p_1^{M_x})$.\footnote{Note that we only demand a \emph{generalized} $1$-solidity witness here, not necessarily a true $1$-solidity witness.}
\end{dfn}

\begin{rem}Many of the standard mouse operators are exactly reconstructing and  uniformly $1$-solid;
for example, $\Mmm(x)=M_n^\#(x)$, and many natural operators $\Mmm$ such that
$\Mmm(x)$ is the sharp for a minimal $x$-mouse with some fixed large cardinal property above a Woodin cardinal.
We next show that uniform $1$-solidity follows from exact reconstruction.

The demand that $T$ be such that $M_x=\Hull_1^{M_x}(\{p_1^{M_x}\})$
is artificial;
one should really generalize this, allowing that there is some (non-specific) $n<\om$ such that for all $x$, we have $M_x=\Hull_{n+1}^{M_x}(\{\pvec_{n+1}^{M_x}\})$. Also, the demand that $M_x$ be non-type 3 is similarly artificial.
But these restrictions allow us to avoid certain fine structural issues which are not particularly relevant to the scale construction. We intend to remove these restrictions in a later version of the paper.
\end{rem}
\begin{lem}\label{lem:p_1_above_lgcd} Let $T$ be good and $M_x=M_{Tx}$.
Then:
\begin{enumerate}[label=--]
\item  If $M_x$ is passive then $M_x$ has a largest cardinal $\gamma$; moreover, $p_1^{M_x}\cut\gamma\neq\emptyset$.
\item There is a fixed $\Sigma_1$ term $d$ such that $\delta^{M_x}=d^{M_x}(p_1^{M_x})$ for all $x$.
\end{enumerate}
\end{lem}
\begin{proof}
 Suppose $M_x$ is passive and let $\theta<\OR^{M_x}$ be any $M_x$-cardinal with $\om<\theta$ (and recall that such $\theta$ exists since $\delta^{M_x}$ exists).
 Then
  $M_x||\theta\elem_1 M_x$ by condensation, so
 $\Hull_1^{M_x}(\theta)=M_x|\theta$.
 Since $M_x=\Hull_1^{M_x}(\{p_1^{M_x}\})$, it follows that $p_1\cut\theta\neq\emptyset$,
 and (since $p_1^{M_x}$ is finite
 and $\delta^{M_x}<\OR^{M_x}$)
 that $M_x$ has a largest cardinal $\gamma$.
 We then get a term $d(u)$
 as desired by taking it to output the least $\delta$ which is Woodin in $L[\es]|\max(u)$,
 if there is such a $\delta$,
 and to output $\max(u)$ itself otherwise.
 
 On the other hand, if $M_x$ is active then since $\delta^{M_x}<\crit(F^{M_x})$, there is also an appropriate $d$ (which also works in case $p_1^{M_x}=\emptyset$).
\end{proof}

\begin{dfn}
 Let $N$ be a premouse
 and $\gamma\in\OR^N$ be a limit ordinal.
 The \emph{$(N,\gamma)$-internal-dropdown} is the longest sequence $\left<N_i\right>_{i\leq m}$
 such that $N_0=N|\gamma$,
 and for each $i<m$, $N_{i+1}$ is the least $N'$ such that $N_i\pins N'\pins N$ and $\rho_\om^{N'}<\rho_\om^{N_i}$.
\end{dfn}

So in particular, $\rho_\om^{N_m}=\card^N(\gamma)$.

\begin{lem}\label{lem:exact_implies_uniformly_1-solid}
Let $T$ be good and $\Mmm=\Mmm_T$ be
  exactly reconstructing. Let $x,y\in\RR$. Then:
 \begin{enumerate}
  \item\label{item:p_1_same_length} $\lh(p_1^{M_x})=\lh(p_1^{M_y})$,
 \item\label{item:p_1_match_construct_p_1} If $x\in M_y$ then $p_1(N^{M_y}_{x\infty})=p_1^{M_y}$.
 
  \item\label{item:p_1=sing_delta} 
 $\delta^{M_x}\in p_1^{M_x}$ iff $p_1^{M_x}=\{\delta^{M_x}\}$ iff $p_1^{M_y}=\{\delta^{M_y}\}$,
 \item\label{item:if_p_1=sing_delta} If $\delta^{M_x}\in p_1^{M_x}$ then $M_x$ is passive with largest cardinal $\delta^{M_x}$.
  \item\label{item:exact_to_unif_1-solid} $\Mmm$ is uniformly $1$-solid.
  \item\label{item:dropdown_invariance} If $M_x$ is passive
  then there is $\ell<\om$
  such that for each $z\in\RR$,
  the $(M_z,\max(p_1^{M_z}))$-internal-dropdown $\left<N_i\right>_{i\leq\ell}$ has length $\ell+1$, and $\rho_\om^{N_\ell}$ is the largest cardinal of $M_x$.
  \end{enumerate}
\end{lem}
\begin{proof}Part \ref{item:p_1_same_length}: This was just by definition;
the common length was denoted $k$ in Definition \ref{dfn:good_theory}.

Part \ref{item:p_1_match_construct_p_1}:
Let $N=N^{M_y}_{x\infty}$. We have $p_1^N\cut\delta^N=p_1^{M_y}\cut\delta^{M_y}$ by \ref{?}.
But $p_1^{M_y}\inter\delta^{M_y}=\emptyset$ and $p_1^{M_x}\inter\delta^{M_x}=\emptyset$ by assumption,
and $\pi(\delta^{M_x})=\delta^N=\delta^{M_y}$ and $\pi(p_1^{M_x})=p_1^N$ where $\pi:M_x\to N$ is the core map.

Part \ref{item:p_1=sing_delta}:
To see $p_1^{M_x}=\{\delta^{M_x}\}$ iff $p_1^{M_y}=\{\delta^{M_y}\}$, 
let $z\in\RR$ with $x,y\leq_Tz$.
It is enough to see that $p_1^{M_x}=\{\delta^{M_x}\}$
iff $p_1^{M_z}=\{\delta^{M_z}\}$.
But this follows from part \ref{item:p_1_match_construct_p_1}
and the preservation of $p_1$ under core maps.

Now suppose $\delta^{M_x}\in p_1^{M_x}\neq\{\delta^{M_x}\}$.
Since $p_1^{M_x}\cap\delta^{M_x}=\emptyset$, we can fix $\gamma\in p_1^{M_x}\cut(\delta^{M_x}+1)$.
But then $\delta^{M_x}\in\Hull_1^{M_x}(\{\gamma\})$,
contradicting the minimality of $p_1^{M_x}$.

Part \ref{item:if_p_1=sing_delta}:
Suppose $p_1^{M_x}=\{\delta^{M_x}\}$. It is easy to see that $M_x$ is passive. And if $p_1^{M_x}\neq\{\delta^{M_x}\}$,
so $p_1^{M_x}\cut(\delta^{M_x}+1)\neq\emptyset$, then note that $\delta^{M_x}\in\Hull_1^{M_x}(p_1^{M_x}\cut(\delta^{M_x}+1))$,
again a contradiction. By Lemma \ref{lem:p_1_above_lgcd}, it follows
that $M_x$ has largest cardinal $\delta^{M_x}$.

Part \ref{item:exact_to_unif_1-solid}:
If $p_1^{M_x}=\emptyset$ (for some, equivalently all, $x$), this is trivial, so suppose otherwise.
If $\delta^{M_x}\in p_1^{M_x}$ then $p_1^{M_x}=\{\delta^{M_x}\}$ and $M_x$ is passive with largest cardinal $\delta^{M_x}$,
and then $M_x|\delta^{M_x}\elem_1 M_x$, in which case the lemma is easy, so suppose otherwise.
So we can always refer to the parameter $\delta^{M_x}$ in a canonical manner.

Now fix $x$. Let $N=N^{M_x}_{\emptyset\infty}$. 
By exactness, $\core_1(N)=M_\emptyset$.
Let $\pi:M_\emptyset\to N$ be the core map.

Let $w$ be the $1$-solidity witness for $(M_\emptyset,p_1^{M_\emptyset})$,
considered as a finite tuple of transitive structures.
Let $u$ be an $\rSigma_1$ term
such that $w=u^{M_\emptyset}(p_1^{M_\emptyset})$.
Then $\pi(w)=u^{N}(p_1^{N})=u^N(p_1^{M_x})$
is a generalized 1-solidity witness for $N$.
Now $M_x|\delta^{M_x}$
is $N$-generic
for the $\delta$-generator
extender algebra of $N$,
and $M_x$ is the result of the canonical extension of $\es_+^N\rest(\delta^{N},\OR^N]$ above $M_x|\delta^{M_x}$.
Let $\wt{w}_x$ be the putative generalized 1-solidity
witness for $(M_x,p_1^{M_x})$
determined by $\pi(w)$ and $M_x|\delta^{M_x}$. (So for each component $\pi(w)_i$ of $\pi(w)$, the $(\wt{w}_x)_i$ is just the $x$-premouse given by the generic extension $\pi(w)_i[M_x|\delta^{M_x}]$.) Note that
$\wt{w}_x$ is in fact
a generalized 1-solidity
witness for $(M_x,p_1^{M_x})$.
Using the uniformity of these processes, uniform 1-solidity easily follows.

Part \ref{item:dropdown_invariance}: This follows from part \ref{item:p_1_match_construct_p_1}, Lemma \ref{lem:p_1_above_lgcd}, the level-by-level correspondence of fine structure 
between $M_x$ and $N_{\emptyset\infty}^{M_x}$,
and that core maps are $p_1$-preserving.
\end{proof}

\begin{dfn}\label{dfn:unif_cofinal}For $M$ a premouse and $\eta\leq\OR^M$, we define
\[ M\wr\eta=(M||\eta,\es^M\rest\eta,F^M\inter M||\eta), \]
where $M||\eta$ is defined using the $\Ss$-hierarchy if $\eta$ is a successor ordinal.
So $M\wr\OR^M=M$ and $M\wr\eta\in M$ for $\eta<\OR^M$,
\[ \eta_0<\eta_1\implies\Th_1^{M\wr\eta_0}\sub\Th_1^{M\wr\eta_1},\]
and
\[ \Th_1^M(M)=\bigcup_{\eta<\OR^M}\Th_1^{M\wr\eta}(M||\eta).\]
Given $\eta\leq\OR^M$ and a $\Sigma_1$ formula $\varphi$
over the language $\Ll_\pm\cup\{\dot{p}\}$ we write
\[ M\wr\eta\sats\varphi\iff (M\wr\eta,p_1^M)\sats\varphi. \]

Let $T$ be good and $\Mmm=\Mmm_T$ and $M_x=\mathscr{M}(x)$. Let $\eta_{xn}$ be the least $\eta$ such that
$M_x\wr\eta\sats\varphi_n$. We say that $T$ is \emph{cofinal} iff:
\begin{enumerate}[label=--]
\item $\sup_{n<\om}\eta_{xn}=\OR^{M_x}$,
\item $\eta_{xn}<\eta_{x,n+1}\text{ for each }n$, and
\item if $M_x$ is passive,
then letting $\gamma$ be the largest cardinal of $M_x$
\tu{(}cf.~Lemma \ref{lem:p_1_above_lgcd}\tu{)},
$\gamma$ is the largest cardinal of $M_x\wr\eta_{xn}$ for each $n$.
\end{enumerate}

 We say that $T$ is \emph{uniformly boundedly $1$-solid} iff there is $n<\om$
 and a  $\Sigma_1$ term $t$ such that for all $x$, $t^{M_x\wr\eta_{xn}}(p_1^{M_x})$ is a generalized $1$-solidity witness for $(M_x,p_1^{M_x})$.
We say that $T$ is \emph{uniform}
iff $\Mmm$ is exactly reconstructing and for all reals $x,y$ with $x\in M_y$,
we have $\pi(\eta_{xn})=\eta_{yn}$ for each $n<\om$, where
$\pi:M_x\to N^{M_y}_{x\infty}$
is the core map.
\end{dfn}

\begin{lem}\label{lem:good,unif_1-solid_to_cofinal}
 Let $T=\left<\varphi_n\right>_{n<\om}$ be good and exactly reconstructing.
  Then there is a cofinal,  uniformly boundedly $1$-solid, good $T'=\left<\varphi'_n\right>_{n<\om}$ with $\Mmm_{T'}=\Mmm_{T}$ and $T'$ recursive in $T$. Moreover, if $T$ is uniform
  then we can take $T'$ uniform.
\end{lem}
\begin{proof}
 Let $M_x=M^T_x$. By Lemma \ref{lem:exact_implies_uniformly_1-solid}, $T$ is uniformly $1$-solid.  Fix a $\Sigma_1$ term $t$ witnessing this.
  Let $\varphi_t$ be the $\Sigma_1$ formula asserting that $t(\dot{p})$ is defined.

  \begin{case}
   $M_x$ is passive (for some, equivalently all, $x$).
   
   Consider first the case in which we don't want to make $T'$ uniform. Let $\ell+1$ be the length of the $(M_x,\max(p_1^{M_x}))$-internal-dropdown;
  also by Lemma \ref{lem:exact_implies_uniformly_1-solid}, this is independent of $x$. Let $\varrho(\dot{p},\rho)$
  be the $\Sigma_1(\dot{p},\rho)$ formula asserting ``there is a sequence $\left<N_i\right>_{i\leq\ell}$
  of proper segments of me,
  defined as an internal-dropdown,
  with $\OR(N_0)=\max(\dot{p})$, and there is $j<\om$ such that $\rho=\rho_{j+1}^{N_\ell}$''.
 Given a $\Sigma_1(\dot{p})$ formula $\varphi$ and $m<\om$, let $\varphi'_m$ be the $\Sigma_1(\dot{p})$ formula
 \[ \exists\alpha,\rho,j,n\ \Big[\varrho(\dot{p},\rho)\wedge(\alpha\text{ is a limit ordinal})\wedge(j,n<\om)\wedge(\rho_{j+1}^{L[\es]|\alpha}=\rho)\wedge \text{ }L[\es]\wr(\alpha+n+m)\text{ exists and }
  L[\es]\wr(\alpha+n)\sats\varphi\Big]. \]
  Let $T_0=T\cup\{\varphi_t\}$ and let $T_1$ be the theory
  \begin{equation}\label{eqn:T_1=} T_1=\{\varphi'_m\mid\varphi\in T_0\text{ and }m<\om\} \end{equation}
  and let $\left<\psi_n\right>_{n<\om}$ be an enumeration of $T_1$ which is recursive in $T$.
 Then define $\left<\varphi_n'\right>_{n<\om}$ recursively by $\varphi'_0=\varphi_t$ and
 \[ \varphi'_{n+1}=\Big(\psi_n\ \wedge\ \exists\beta\in\OR\ \Big[L[\es]\wr\beta\text{ exists and }
 L[\es]\wr\beta\sats\varphi'_n\Big]\Big).\]

 We claim that $T'=\left<\varphi'_n\right>_{n<\om}$ works.
 For clearly $T'$ is logically at least as strong as $T$,
 and $M_x\sats T'$.
 Because $\varphi_0'=\varphi_t$,
 $T'$
 is easily uniformly boundedly $1$-solid (as witnessed by $0$).
 So we just need to see that $T'$ is cofinal. Suppose not and fix a counterexample $x$.
 Then note that there is a passive $N\pins M_x$ such that $p_1^{M_x}\in N$ and
 \[ \all n<\om\ \Big[N\sats \varphi'_n(p_1^{M_x})\Big];\]
 let $N$ be least such. Let $\bar{N}=\Hull_1^N(\{p_1^{M_x}\})$ and $\pi:\bar{N}\to N$ the uncollapse.
 Note that $w=t^{M_x}(p_1^{M_x})\in\rg(\pi)$, and since $w$ is a generalized $1$-solidity witness
 for $(M_x,p_1^{M_x})$, it is likewise for $(N,p_1^{M_x})$. Letting $\pi(\bar{p})=p_1^{M_x}$,
 it follows that $\bar{p}=p_1^{\bar{N}}$ and $\bar{N}=\Hull_1^{\bar{N}}(\{\bar{p}\})$
 and $(\bar{N},\bar{p})$ is $1$-solid, hence $\bar{N}$ is $1$-sound. But then also $(\bar{N},p_1^{\bar{N}})\sats T$,
 so $\bar{N}$ contradicts the minimality of $M_x$.
 
 Now suppose $T$ is uniform; we will modify the process a little and arrange that $T'$ is uniform. Set $\varphi_t=$``true'', and with this replacing the former $\varphi_t$, define the rest as before. It is straightforward to see that this works.  The uniformity of $T'$ results from that of $T$, together with the level-by-level correspondence of fine structure between the relevant models (the latter is important in seeing that the internal-dropdown sequences agree with one another across those different models). The fact that $T'$ is uniformly boundedly $1$-solid is established differently to before:  use the defining formula used in the proof of Lemma \ref{lem:exact_implies_uniformly_1-solid}, and note that because $T'$ is uniform, this definition witnesses uniformly bounded $1$-solidity.
 \end{case}
 \begin{case} $M_x$ is active (type 1 or 2).

 This case is similar to the passive one, but instead of incorporating formulas which ensure closure under the $\Ss$-operator,
 we incorporate formulas to ensure that the analogue of $\bar{N}$
 is an (active type 1 or 2, according to the type of $M_x$) premouse (not just a protomouse).
 Thus, for a type 1 or 2 premouse $M$ and $\beta<((\dot{\mu}^M)^+)^M$ (recall $\dot{\mu}^M=\crit(F^M)$),
 let
 \[ F^M_\beta=F^M\rest((M|\beta)\cross[\nu(F^M)]^{<\om}).\]
Then for a given $\Sigma_1$ formula $\varphi$,
 let $\varphi'$ be the formula
 \[\text{``}\ex\alpha,\beta\ [\beta<(\dot{\mu}^+)\text{ and }M\wr\alpha\sats\varphi\text{ and }
 F^M_\beta\notin F^M\wr\alpha]\text{''} \]
 (where $M$ above is the premouse our formula refers to);
 and given $\varphi$ of the form ``$\ex\beta[\beta<(\dot{\mu}^+)\ \&\ \psi(\beta)]$'', where $\psi$ is $\Sigma_1$,
 let $\varphi''$ be the formula
 \[ \text{``}\ex\alpha,\beta\ [\beta<(\dot{\mu}^+)\text{ and }\psi(\beta)\text{ and }F^M_\beta\in F^M\wr\alpha\text{ and }M||\alpha\sats\ZF^-]\text{''}. \]
 Let $T_0$ be the closure of $T\cup\{\varphi_t\}$ under $\varphi\mapsto\varphi'$ and $\varphi\mapsto\varphi''$,
 and then proceed much as before. Suppose the resulting $T'$ is non-cofinal.
 and fix a counterexample $x$. Let $\eta=\sup_{n<\om}\eta_{xn}$.
 Let
 \[ \bar{N}=\cHull_1^{M_x\wr\eta}(\{p_1^{M_x}\}) \]
 and $\pi:\bar{N}\to M_x$ be the uncollapse. 
Then by the closure under $\varphi\mapsto\varphi'$ and $\varphi\mapsto\varphi''$,
$\bar{N}$ is a premouse, and arguing as before, $\bar{N}$ is sound and $\rho_1^{\bar{N}}=\om$.
But assuming that $\eta<\OR^M_x$, then $\bar{N}\in M_x$, so we reach a contradiction as before. We leave the remaining details to the reader.\qedhere
\end{case}
\end{proof}

\begin{dfn}
 Let $T=\left<\varphi_n\right>_{n<\om}$ be good and cofinal.
 
 A $\Sigma_1$ formula $\varphi$ is \emph{$T$-uniformly bounded} iff there is a formula $\psi$ and
 $n<\om$ such that $\varphi$ asserts ``there is $\alpha$ such that $M\wr\alpha\sats\psi\wedge\neg\varphi_n$''.
 A $\Sigma_1$ term $t$ is \emph{\tu{(}$T$-\tu{)}uniformly bounded} iff $t=t_\varphi$ for some $T$-uniformly bounded $\varphi$.
 A $\Sigma_1$ term $t$ is \emph{$T$-nice} iff there is a $T$-uniformly bounded term $t'$
 such that $t$ is defined by: ``if $t'(u)$ is defined then output $t'(u)$; otherwise output $0$''. We may drop the prefix ``$T$-'' if $T$ is clear from context. 
 \end{dfn}

\begin{dfn}\label{dfn:A_psi_for_T_good}
Let $T$ be good, exactly reconstructing,  cofinal, uniformly boundedly $1$-solid.
Let $\Mmm=\Mmm_T$ and $M_x=\Mmm(x)$. 
For  $\rSigma_1(\dot{p})$ formulas $\psi$, let
\[ A_{\psi}=\Big\{x\in\RR\ \Big|\ (M_x,p_1^{M_x})\sats\psi\Big\}. \]
Define the pointclass
\[ \Gamma_{\mathscr{M}}=\Big\{A_\psi\ \Big|\ \psi\in\Ll_\pm(\dot{p})\text{ is }\rSigma_1(\dot{p})\Big\}.\]
If $T$ is also uniform, then we also define the pointclass
\[ \Gamma_{\mathscr{M},\mathrm{ubd}}=\Big\{A_\psi\ \Big|\ \psi\in\Ll_\pm(\dot{p})\text{ is }\rSigma_1(\dot{p})\text{ and uniformly bounded}\Big\}.\qedhere\]
\end{dfn}

Analogous to Theorem \ref{thm:1-C-closed_scale_prop}, we will also prove (with almost the same proof):

\begin{tm}\label{tm:exactly_recon_scales}
 Let $T$ be good, exactly reconstructing, cofinal, uniformly boundedly $1$-solid. Let $\mathscr{M}=\mathscr{M}_T$. Then:
 \begin{enumerate}
  \item 
$\Gamma_{\mathscr{M}}$ has the scale property.
\item  If $T$ is uniform then
  for every $A\in\Gamma_{\mathscr{M},\mathrm{ubd}}$
 there is a scale on $A$, all of whose norms are in $\Gamma_{\mathscr{M},\mathrm{ubd}}$, defined by by a sequence $\vec{\psi}=\left<\psi_n\right>_{n<\om}$ of $T$-uniformly bounded formulas $\psi_n$, with $\vec{\psi}$ recursive in $T$.
 \end{enumerate}
\end{tm}

One should also compare this result with Theorem \ref{tm:optimality}.

\section{Mouse scales}\label{sec:the_scale}

\subsection{Theory norms for $1$-$\CC$-closed mouse operators}\label{sec:theory_norms_1-CC-closed}

Let $\Mmm$ be a $1$-$\CC$-closed function. 
Let $A\in\Gamma_\Mmm$, and fix $\psi_0\in\Ll_{\Sigma_1}$ such that $A=A_{\Mmm\psi_0}=\{x\in\RR\bigm|\Mmm(x)\sats\psi_0\}$.
We now begin to define the scale on $A$ used to prove \ref{thm:1-C-closed_scale_prop}.

There are two basic kinds of norms we use for the scale: 
\emph{theory} norms and \emph{lifting} norms, generalizing those of \S\S\ref{subsec:Pi^1_1},\ref{sec:MS},\ref{sec:Pi^1_3}.
For $x\in A$ let $M_x$ be the least $M\pins\Mmm(x)$ such that $M\sats\psi_0$.
Note that $M_x=\J(R)$ for some $R$ with $\rho_\om^R=\om$, and $M_x=\Hull_1^{M_x}(\emptyset)$. So $R,M_x$ are both $\om$-mice.
The theory norms, which are easy to define, will compare reals $x,y\in A$ by comparing the $\Sigma_1$ theories of $M_x,M_y$.
(The lifting norms, which take a lot more work to define, will compare $x,y$ by comparing features of 
the process lifting finite trees on $M_x,M_y$ to corresponding trees on mice $P$ for which $M_x$ and/or $M_y$ are constructed
by $\CC^P_x,\CC^P_y$.)

\begin{dfn}[Theory norms for $1$-$\CC$-closed operators]\label{dfn:1-CC-closed_theory_norm} For each $\Sigma_1$ formula $\varrho$ of $\Ll_\pm$
we define a $0/1$ valued norm $\Phi_\varrho$
on $\RR$, where for $x\in\RR$,
\[ \Phi_\varrho(x)=1\ \iff\ M_x\sats\varrho.\qedhere\]
\end{dfn}

\begin{dfn}[$M_\infty$ for $1$-$\CC$-closed operators]\label{dfn:M_infty}
Suppose that $x_n\to x$ modulo the theory norms $\Phi_\varrho$. Let $T_{x_n}=\Th_1^{M_{x_n}}$. For each $\Sigma_1$ sentence $\varrho\in\Ll_\pm$ let
\[ B_\varrho=\{n\mid M_{x_n}\sats\varrho\}.\]
So $B_\varrho$ is either finite or cofinite. Let
\[ T_\infty=\{\varrho\mid B_\varrho\text{ is cofinite}\}.\]

Then $M_\infty$ denotes the unique \tu{(}up to isomorphism\tu{)} structure $M$ for $\Ll_\pm$
such that
\[ \Th_{1}^M=T_\infty\text{ and }M=\Hull_1^M(\emptyset).\]
Equivalently, $M_\infty$ is the ultraproduct of $\left<M_{x_n}\right>_{n<\om}$
modulo the cofinite filter, formed using functions $f:\om\to\bigcup_nM_{x_n}$
such that for some $\Sigma_1$ term $t$, we have
\[ f(n)=t^{M_{x_n}}\text{ for cofinitely many }n.\]
A more explicit description of $M_\infty$ is given in \ref{rem:M_infty}.\end{dfn}

\begin{rem}\label{rem:M_infty}
Of course $M_\infty$ depends on $\left<x_n\right>_{n<\om}$ and $\psi_0$ (not actually on $\Mmm$, as from $\psi_0$ and $x_n$ we can recover $M_{x_n}$),
but these things will be known by context, so we leave them out of the notation.
If $M_\infty$ is wellfounded, then we take it to be transitive (and then it really is unique).

Continuing with the notation from above, we give another description of $M_\infty$ and verify
that $\Th_1^{M_\infty}=T_\infty$ and $M_\infty=\Hull_1^{M_\infty}(\emptyset)$.

Fix a variable $v$, and let $\wt{M}$ be the set of $\Sigma_1$ formulas $\varrho(v)$ of the single free 
variable $v$ such that
\begin{enumerate}[label=--]
 \item  $\text{``}\ex v\varrho(v)\text{''}\in T_\infty$, but
\item $\text{``}\ex v,w[\varrho(v)\ \&\ \varrho(w)\ \&\ v\neq w]\text{''}\notin T_\infty$.
\end{enumerate}

Let $\approx$ be the equivalence relation on $\wt{M}$
specifying equality of the defined object; that is,
\[ \varrho\approx\varrho'\ \iff\ \text{``}\ex v,v'[\varrho(v)\ \&\ \varrho'(v')\ \&\  
v=v']\text{''}\in T_\infty.\]
Let $M_\infty$ be the structure for $\Ll_\pm$
with universe the set of $\approx$-equivalence classes, and membership relation $\dot{\in}$
and predicate $\dot{\es}$
induced by $T_\infty$ analogously to $\approx$, that is,
\begin{enumerate}[label=--]
 \item 
$[\varrho]\ \dot{\in}\ [\varrho']\ \iff\ \text{``}\ex v,v'[\varrho(v)\ \&\ \varrho'(v')\ \&\  
v\in v']\text{''}\in T_\infty$,
\item $\dot{\es}([\varrho])\ \iff\ \text{``}\ex v[\varrho(v)\ \&\  
\es(v)]\text{''}\in T_\infty$.
\end{enumerate}

We claim that for all $\Sigma_1$ formulas $\psi$
of free variables $v_1,\ldots,v_k$, and all $\varrho_1,\ldots,\varrho_k\in \wt{M}$, we have
\[ M_\infty\sats\psi([\varrho_1],\ldots,[\varrho_k])\iff
\text{``}\ex v_1,\ldots,v_k[\bigwedge_{1\leq i\leq k}\varrho_i(v_i)\ \&\ \psi(v_1,\ldots,v_k)]\text{''}\in T_\infty. \]
This is proved by a straightforward induction on quantifier complexity, essentially by the proof of Los theorem;
the main point is that for example if $\psi$ is $\Sigma_1$ and $\varrho\in \wt{M}$ and
\[ \text{``}\ex v[\varrho(v)\ \&\ \ex y[\psi(y,v)]]\text{''}\in T_\infty \]
then the usual computation of $\Sigma_1$ Skolem functions over premice yields some $\varrho'\in \wt{M}$
such that
\[ \text{``}\ex y,v[\varrho(v)\ \&\ \varrho'(y)\ \&\ \psi(y,v)]\text{''}\in T_\infty, \]
and so by induction, $M_\infty\sats\psi([\varrho'],[\varrho])$, so $M_\infty\sats\exists y[\psi(y,[\varrho])]$.

Applying this to $\Sigma_1$ sentences $\psi$, we have $\Th_{\Sigma_1}^{M_\infty}=T_\infty$. It also follows that $M_\infty=\Hull_1^{M_\infty}(\emptyset)$. Indeed, for each $\varrho\in \wt{M}$,
\[ M_\infty\sats\text{``}[\varrho]\text{ 
 is the unique }y\text{ such that }\varrho(y)\text{''}.\]
For $M_\infty\sats\varrho([\varrho])$ because
\[ \text{``}\exists v[\varrho(v)\ \&\ \varrho(v)]\text{''}\in T_\infty,\]
and if $M_\infty\sats\varrho([\varrho'])$ where $\varrho'\in \wt{M}$,
then
\[ \text{``}\exists v[\varrho'(v)\ \&\ \varrho(v)]\text{''}\in T_\infty,\]
so $\varrho\approx\varrho'$.

Note that $M_\infty\sats\psi$ for every $\Pi_2$ sentence $\psi$ satisfied by eventually all $M_{x_n}$.
For example if $\psi=\all x\psi'(x)$ where $\psi'$ is $\Sigma_1$,
and if $\varphi\in \wt{M}$, then note that $M_\infty\sats\psi'([\varphi])$
because
\[ \text{``}\exists v[\varphi(v)\ \&\ \psi'(v)]\text{''}\in T_\infty. \]
This generalizes easily to sentences $\all x_1,\ldots,x_k\ \psi'(x_1,\ldots,x_k)$
where $\psi$ is $\Sigma_1$. Since the passive premouse axioms are $\Pi_2$ (in fact Q-formulas),
it follows that \emph{if} $M_\infty$ is wellfounded then $M_\infty$ is a premouse.\footnote{Recall that
\cite{fsit} does not quite characterize type 3 premousehood so nicely.} Since $M_{x_n}=\Hull_1^{M_{x_n}}(\emptyset)$ for all $n$, it also follows that $M_\infty=\Hull_1^{M_\infty}(\emptyset)$, and so $M_\infty$ is sound with $\rho_1^{M_\infty}=\om$ and $p_1^{M_\infty}=\emptyset$. Also, for each $n<\om$, we have $\dot{x}^{M_\infty}(n)=x(n)$, so if $M_\infty$ is wellfounded then $\dot{x}^{M_\infty}=x$. Note finally that $M_\infty\sats\psi_0+\neg\psi_0'$,
where $\psi_0'$ asserts ``there is a proper segment of me satisfying $\psi_0$'';
this is because  $\psi_0\in T_\infty$ but $\psi_0'\notin T_\infty$. It follows that if $M_
\infty$ is wellfounded and $(0,\om_1+1)$-iterable then $M_\infty=M_x$.
\end{rem}

\subsection{Theory norms for exactly reconstructing mouse operators}\label{sec:theory_norms_exactly_reconstructing}

Let $T$ be good, exactly reconstructing, cofinal, uniformly boundedly $1$-solid. Let $\Mmm=\Mmm_T$.
Let $M_x=\Mmm(x)$.
Let $\psi_0$ be an $\rSigma_1$ formula of $\Ll(\dot{p})$ and $A=A_{\Mmm\psi_0}$.
We now begin to define the scale on $A$ used to prove \ref{tm:exactly_recon_scales}.
As before, there will be
\emph{theory} norms and \emph{lifting} norms,
and we first describe the theory norms.

\begin{dfn}[Theory norms for exactly reconstructing operators]
Associated to each $T$-uniformly bounded $\Sigma_1$ sentence $\varrho$ of $\Ll_{\pm}(\dot{p})$, we define a the $0/1$-valued \emph{$T$-uniformly bounded theory norm} $\Phi_\varrho$, where for $x\in\RR$, we set
\[ \Phi_\varrho(x)=1\iff M_x\sats\varrho;
\]
recall that here we interpret $\dot{p}$ in $\varrho$ with $p_1^{M_x}$. \tu{(}Note here that we are only interested in this norm assuming that $\varrho$ is $T$-uniformly bounded, even though one could have defined $\Phi_\varrho$ more generally.\tu{)}
\end{dfn}

\begin{dfn}
 [$M_\infty$ for exactly reconstructing operators]\label{dfn:M_infty_exactly_recon}
Let $x_n\to x$ modulo the $T$-uniformly bounded theory norms.
 Then we define $M_\infty$ as 
the ultraproduct of $\left<M_{x_n}\right>_{n<\om}$
modulo the cofinite filter, formed using functions $f:\om\to\bigcup_nM_{x_n}$
such that for some nice\footnote{Note that we are restricting to nice terms.} term $t$, we have
\[ f(n)=t^{M_{x_n}}\text{ for cofinitely many }n.\qedhere\]\end{dfn}

\begin{lem}\label{lem:M_infty_basics_exactly_recon}
 Let $x_n\to x$ modulo the $T$-uniformly bounded theory norms. 
Let $p=\dot{p}^{M_\infty}$. Then:
\begin{enumerate}
\item\label{item:Los_for_exactly_recon} Let $\psi$ be $\Sigma_1$. Then $(M_\infty,p)\sats\psi$
iff there is $m<\om$ such that for all sufficiently large $n$,
we have
\[ (M_{x_n}\wr\eta_{x_nm},p_1^{M_{x_n}})\sats\psi.\]
Thus, $(M_\infty,p)\sats T$, $T$ is witnessed cofinally in $\OR^{M_\infty}$, and $M_\infty=\Hull_1^{M_\infty}(\{p\})$.
\item\label{item:M_infty_sats_psi_0_if_psi_0_T-ubd_exactly_recon} If $\psi_0$ is $T$-uniformly bounded and $\{x_n\}_{n<\om}\sub A$  then
$(M_\infty,p)\sats\psi_0$.\footnote{We will need to incorporate further norms to ensure this conclusion in case $\psi_0$ is not $T$-uniformly bounded.}

\item\label{item:M_infty_fs_exactly_recon}If $M_\infty$ is wellfounded then $M_\infty$ is a sound premouse, $\dot{x}^{M_\infty}=x$, $\rho_1^{M_\infty}=\om$ and $p_1^{M_\infty}=p$.
\item\label{item:when_M_infty_iterable_exactly_recon} Suppose $M_\infty$ is wellfounded and $(0,\om_1+1)$-iterable. Then
$M_x\ins M_\infty$ and if $T$ is recursive then $M_x=M_\infty$.
\end{enumerate}
\end{lem}
\begin{proof}
Part \ref{item:Los_for_exactly_recon}: The first statement (regarding the $\Sigma_1$ theory of $(M_\infty,p)$) is a slight variant of Lo\'s' theorem. The rest is an easy consequence.

Part \ref{item:M_infty_sats_psi_0_if_psi_0_T-ubd_exactly_recon}: This is an immediate consequence of part \ref{item:Los_for_exactly_recon}.

Part \ref{item:M_infty_fs_exactly_recon}:
 Suppose $M_\infty$ is wellfounded. That $M_\infty$ is an real premouse
 with $\dot{x}^{M_\infty}=x$
 follows from part \ref{item:Los_for_exactly_recon} like in Remark \ref{rem:M_infty}. Because $T$ is uniformly boundedly $1$-solid,
 as witnessed by $t,n$,
 we get that $t^{M_\infty}(\dot{p}^{M_\infty})$ is a generalized $1$-solidity witness
 for $(M_\infty,\dot{p}^{M_\infty})$. Since  $M_\infty=\Hull_1^{M_\infty}(\emptyset)$,
 therefore $\dot{p}^{M_\infty}=p_1^{M_\infty}$ and
 $\rho_1^{M_\infty}=\om$ and $M_\infty$ is $1$-sound. 
 
 Part \ref{item:when_M_infty_iterable_exactly_recon}:
 Suppose $M_\infty$ is wellfounded and $(0,\om_1+1)$-iterable. Then  $M_x\ins M_\infty$
 by the minimality of $M_x$.
 Now suppose also that $T$ is recursive. Say that a premouse is \emph{$x$-good} iff it satisfies the model theoretic requirements of $M_x$.
 Then if $M_x\pins M_\infty$
 then $M_\infty\sats$``There is an $x$-good proper segment of me''.
 But then for all large $n$, $M_{x_n}\sats$``There is an $x_n$-good proper segment of me'',
 contradicting the minimality of $M_{x_n}$.
 \end{proof}

 \subsection{Lifting norms}

 In this section we will define the key notion in the scale construction, the lifting norms.
 We do this simultaneously for both $1$-$\CC$-closed operators and exactly reconstructing operators $\Mmm$. So we assume we are in either
 the case of \S\ref{sec:theory_norms_1-CC-closed} or \S\ref{sec:theory_norms_exactly_reconstructing},
 and define $M_x=\Mmm(x)$, etc, as there.
 
\begin{dfn}
 Let $M$ be a premouse. Define the \emph{ultrapower evaluation} function $\uev^M$, with domain $M$,
as follows. Let $x\in M$. If $x=(0,y)$ then $\uev^M(x)=y$. If $M$ is active 
type 3 and, letting $F=F^M$ and $\kappa=\crit(F)$, $x=(1,a,f)$ where
$a\in[\OR]^{<\om}$ and $f:[\kappa]^{<\om}\to M|\kappa^{+M}$, then 
\[ \uev^M(x)=[a,f]^M_F.\]
Otherwise $\uev^M(x)=x$.

Define the \emph{decoding} function $\dcd^M:\core_0(M)\to M$ (note $M$ is not squashed) by $\dcd^M(x)=\uev^M(x)$ if 
$\uev^M(x)\in M$, and $\dcd^M(x)=0$ otherwise. So $\dcd^M$ is a surjection.
\end{dfn}

\begin{dfn}[Min terms] We fix a recursive function $\varphi\mapsto\varphi'$ with the following 
property. Let $\varphi(\dot{x_0},\ldots,\dot{x}_{m-1},\dot{y})$ be an $\rSigma_{n+1}$ formula of $\Ll_{\pm}$ in the
$m+1$ free variables $\dot{x_0},\ldots,\dot{x}_{m-1},\dot{y}$. Write 
$\vec{\dot{x}}=(\dot{x_0},\ldots,\dot{x}_{m-1})$. Let $M$ be any $n$-sound premouse and
$\xvec\in\core_0(M)^m$. Then
$\varphi'(\dot{p},\vec{\dot{x}},\dot{y})$ is an $\rSigma_{n+1}$ formula of $\Ll_{\pm}$ in the free variables 
$\dot{p},\vec{\dot{x}},\dot{y}$, and
\[ \{(\xvec,y)\mid \core_0(M)\sats\varphi'(\pvec_n^M,\xvec,y)\} \]
uniformizes
\[ \{(\xvec,y)\mid \core_0(M)\sats\varphi(\xvec,y)\}.\]

Given $\varphi(\vecdotx,\dot{y})$ as above, we associate the
\emph{$(n+1)$-term} $\sigma_\varphi$. Given an $n$-sound $M$ and $\xvec\in\core_0(M)$,
we set
\[ \sigma_\varphi^M(\xvec)=\text{ unique }y\in\core_0(M)\text{ such that }
\core_0(M)\sats\varphi'(\pvec_n^M,\xvec,y) \]
if $M\sats\ex y\varphi(\xvec,y)$, and $\sigma_\varphi^M(\xvec)$ is undefined otherwise.

Given an $\rSigma_{n+1}$ formula $\varphi$, we associate the \emph{$(n+1)$-rep-term} $\tau_\varphi$.
Fix an $n$-sound $M$ and let $\xvec\in\core_0(M)$. We set
\[ \tau_\varphi^M(\xvec)=\dcd^M(\sigma_\varphi^M(\xvec)) \]
(so $\tau_\varphi^M(\xvec)$ is defined iff $\sigma_\varphi^M(\xvec)$ is defined).

Given some term or rep-term $t$ and $\xvec\in M$, we write $t^M(x)\downarrow$ iff $x\in\dom(t^M)$.
In general when we declare ``Let $y=t^M(\xvec)$'', it might be that $\xvec\notin\dom(t^M)$,
in which case $y$ does not have any meaning.
Following such a declaration, we write ``$y\downarrow$'' for ``$t^M(\xvec)\downarrow$''.
\end{dfn}

\begin{dfn}
Let $\Tt$ be an sse-$n$-maximal iteration tree on an $n$-sound premouse $M$, of successor length 
$\theta+1$.
Let $S\ins M^\Tt_\theta$. We say that $S$ is 
\emph{sse-$\Tt$-extending} iff $\Tt$ can be 
extended to a putative sse-$n$-maximal tree by next using $F^S$; that is, $S$ is active and 
either $\lh(E^\Tt_\alpha)\leq\OR^S$ for all $\alpha<\theta$,
or $\theta=\alpha+1$ and $E^\Tt_\alpha$ is superstrong and $\nu(E^\Tt_\alpha)<\OR^S$.
\end{dfn}

\begin{dfn}\label{dfn:M_x-good}
Let $x\in\RR$ and $M$ be a sound $(0,\om_1+1)$-iterable $x$-premouse with $\rho_1^M=\om$. 
 Let $z\in\RR$ and $P$ be a $(0,\om_1+1)$-iterable $z$-premouse 
 with $x\in\RR^P$.
 We say that $P$ is \emph{$M$-good}
 iff there is $\xi$ such that $P$ is $(x,\xi,0)$-good and $M=\core_{1}(N_\xi^P)$
 and either:
\begin{enumerate}[label=--]
 \item $\Mmm$ is $1$-$\CC$-closed and either:
 \begin{enumerate}[label=--]\item $\xi<\rho_0^P$, or
 \item $M,P$ are non-small and $\xi=\OR^P$,
 \end{enumerate}
 or
 \item $\Mmm$ is exactly reconstructing,
 $P$ is non-small, $\rho_1^P=\om$, $P$ is $\delta^P$-sound,
 $\xi=\OR^P$,  $p_1^P\cap\delta^P=\emptyset=p_1^M\cap\delta^M$.
\end{enumerate}
If $P$ is $M_x$-good, as witnessed by $\xi$, then let $\pi^P_x:\core_0(M_x)\to \core_0(N^P_\xi)$ be the core embedding.
\end{dfn}

\begin{dfn}[Norm description]\label{dfn:norm_desc}
Let $m<\om$. A \emph{norm description of depth $m+1$} is a sequence
 \[ \sigma=(\mathscr{N},(n_0,\nu_0,e_0,\tvec_0),\ldots,(n_{m-1},\nu_{m-1},e_{m-1},\tvec_{m-1}),(n_m,\nu_m,t)), \]
 where for each $i\leq m$, we have:
 \begin{enumerate}[label=--]
 \item $0\in\mathscr{N}\sub [0,m]$,
  \item 
$n_i<\om$,
\item if $i<m$ then $(\nu_i,e_i)\conc\tvec_i$ 
is a finite sequence of
$(n_i+1)$-rep-terms in variables $(\dot{p},\udotvec_0,\ldots,\udotvec_{i-1})$,
\item  $\nu_m$ is an
 $(n_m+1)$-rep-term
 and $t$ is an $(n_m+1)$-term\footnote{Note that
 $(\nu_i,s_i)\conc\tvec_i$
is a finite sequence of \emph{rep-}terms,
for $i<m$, and $\nu_m$ is a \emph{rep-}term,
whereas $t$ is a \emph{term}.}, all in variables $(\dot{p},\udotvec_0,\ldots,\udotvec_{m-1})$,
\end{enumerate}
where $\lh(\udotvec_j)=\lh(\tvec_j)$ for each $j<m$.

In case $\Mmm$ is exactly reconstructing,
 we call $\sigma$
 \emph{nice} iff
 \begin{enumerate}[label=--]
  \item
for each $i<m$, 
 if $i\in\mathscr{N}$
 then $n_i=0$ and the terms $\nu_i,s_i,\tvec_i$ are nice,
 and
 \item if $m\in\mathscr{N}$
 then $n_m=0$ and $\nu_m,t$ are nice.\qedhere
 \end{enumerate}
\end{dfn}

\begin{dfn}[Lifting norm]\label{dfn:lifting_norm}

Let $\sigma$ be a norm description of depth $m+1$
and fix notation as in Definition \ref{dfn:norm_desc}.
We define a relation $\leq_\sigma$ on $A$, through a certain lexicographic 
ordering.  
We will define a finite sequence  $\left<{\leq_i}\right>_{i\leq 4m+4}$ of approximations 
$\leq_i$, and set ${\leq_\sigma}={\leq_{4m+4}}$. We will show in Lemma \ref{lem:leq_sigma_is_a_pwo} that each $\leq_i$ is a prewellorder, in particular $\leq_\sigma$. We write $<_i$ for the strict part of $\leq_i$, and for $x,y\in A$, 
define $x\equiv_i y$ iff $x\leq_i y\leq_i x$. We will have that $\leq_{i+1}$ refines $\leq_i$; that 
is, if $x<_i y$ then 
$x<_{i+1}y$. 
For convenience we also define $\leq_{-i}$ to be the trivial prewellorder for $1\leq i\leq 4$; that is, we set 
$x\leq_{-i} y$ for all $x,y\in A$. For $i<4m+4$, we will also define relations 
$\equiv_i^{\dec}$ \tu{(}\emph{dec} for \emph{decided}\tu{)} and $\equiv_i^\undec$ \tu{(}\emph{undecided}\tu{)} on 
$A$. By setting $x\equiv_i^{\dec}y$ we are declaring that $x\equiv_j y$ for all $j\geq i$ \tu{(}hence for all $j$\tu{)}. We set 
$x\equiv_i^\undec y$ iff $x\equiv_i y$ and $x\not\equiv_i^\dec y$. So when computing whether 
$x\leq_{i+1}y$, we assume that $x\equiv_i^\undec y$.
We will have ${\equiv_i^\dec}\sub{\equiv_{i+1}^\dec}$.
We set $\equiv_{-i}^\dec=\emptyset$ (the 
trivial declaration) for $1\leq i\leq 4$.

Fix $x,y\in A$. Fix  $\Yback$ such 
that\begin{equation}\label{eqn:P_is_M_x-good_or_M_y-good} \Yback\text{ is }M_x\text{-good or }\Yback\text{ is }M_y\text{-good.}\end{equation}
We will determine whether $x\leq_ky$ and whether $x\equiv_k^\dec y$ by induction on $k$. This determination will be made referring to $P$, but we will  show later that the result is 
independent of the choice of $\Yback$.

We set $x\leq_0 y$ iff $P$ is $M_x$-good and if $P$ is $M_y$-good then
$\prodstage^\Yback(M_x)\leq\prodstage^\Yback(M_y)$.
We set $x\equiv_0^\dec y$ iff $x\equiv_0 
y$ and $n_0\neq 0$.\footnote{The parameter $n_0$ is pointless in the current setup, because we have demanded that $\rho_1^{M_x}=\om$. But this will be generalized in a later version.} If $x\equiv_0^\undec y$ then let 
$\Tt_{x,0},\Tt_{y,0}$ be the trivial $n_0$-maximal trees on $M_x,M_y$ respectively.
Note that if $x\equiv_0y$ as computed using $P$,
then $P$ is both $M_x$-good and $M_y$-good
\tu{(}using \tu{(}\ref{eqn:P_is_M_x-good_or_M_y-good}\tu{)}\tu{)}. Let $\pi_{x0}=\pi_x^P$ and $\pi_{y0}=\pi_y^P$ \tu{(}see Definition \ref{dfn:M_x-good}\tu{)}.

We maintain the following inductively on $i\in[0,m]$.
Assume $x\equiv_{4i}^\undec y$.
Then we have defined $\Tt_x=\Tt_{xi}$, $\Tt_y=\Tt_{yi}$, 
$\Avec_i=(\avec_0,\ldots,\avec_{i-1})$, $\Bvec_i=(\bvec_0,\ldots,\bvec_{i-1})$, $\pi_x=\pi_{xi}$, $\pi_y=\pi_{yi}$, $R_x$ and $R_y$ such that:
\begin{enumerate}[label=--]
\item $\Tt_x$ is an sse-$n_0$-maximal\footnote{Note we do not consider arbitrary essentially-$n_0$-maximal trees, but only sse-. In particular,
any non-$n_0$-maximal part of the tree is non-small.} tree on $M_x$ and $\Tt_y$ likewise on
$M_y$, both of length $i+1$,
with identical tree, drop and degree structures,  and for all $j\leq i$,
we have:
\begin{enumerate}[label=--]\item $j\in\mathscr{N}$
iff $[0,j]^{\Tt_x}\cap\dropset^{\Tt_x}_{\deg}=\emptyset$,
and
\item $\deg^{\Tt_x}(j)=n_j$.
\end{enumerate}
\item For all $j<i$, we have $\avec_j\in\nu(E^{\Tt_x}_j)^{<\om}$ and $\bvec_j\in\nu(E^{\Tt_y}_j)^{<\om}$ 
and $\lh(\avec_{j})=\lh(\tvec_j)=\lh(\bvec_j)$.
\item $\Uu\eqdef\lifttree^{\Tt_x,\Yback}=\lifttree^{\Tt_y,\Yback}$
and $\pi_x=\pi^{\Tt_x,\Uu}$ and 
$\pi_y=\pi^{\Tt_y,\Uu}$,
\item 
$\pi_x(\Avec_i)=\pi_y(\Bvec_i)$,
\item $R_x=M^{\Tt_x}_i$ and $R_y=M^{\Tt_y}_i$.
\end{enumerate}

\begin{case}
$i<m$.

Let
$l=\lh(\tvec_i)$ and
\[ \nu_x=\nu_i^{R_x}(\pvec_{n_i+1}^{R_x},\Avec_i), \]
\[ e'_x=e_i^{R_x}(\pvec_{n_i+1}^{R_x},\Avec_i), \]
\[ a_{ij}= 
t_{ij}^{R_x}(\pvec_{n_i+1}^{R_x},\Avec_i), \]
\[ \avec_i=(a_{i0},\ldots,a_{i,l-1}); \]
here each of $\nu_x,e'_x,a_{ij}$ may be undefined.
(Recall here that $\nu_i,e_i,t_{ij}$ are rep-terms, and thus, if $R_x$ is active type 3,
$\nu_x,e'_x,a_{ij}$ might be in $R_x\cut R_x^\sq$.)
Let $\nu_y,e'_y,b_{ij},\bvec_i$ be defined likewise, with $y$ replacing $x$
and $b$ replacing $a$ throughout. 
If $e'_x\downarrow$ and $e'_x=\emptyset$, let $e_x=R_x$ (in which case we say that 
$e_x\downarrow$); otherwise let $e_x=e'_x$. Likewise for $e_y$.

Recall that $x\equiv_{4i}^\undec y$. Say that $x$ is \emph{$(\sigma,i)$-well}
iff 
\begin{enumerate}[label=--]
 \item if $R_x$ is type 3 then $\nu_x\downarrow=\nu(R_x)$,
 \item $e_x\downarrow$ and $e_x$ is sse-$\Tt_x$-extending,
 \item $\all j<l\left[a_{ij}\downarrow\text{ and }a_{ij}\in\nu(F^{e_x} )\right]$.
\end{enumerate}

Set $x\leq_{4i+1} y$ iff, if $y$ is $(\sigma,i)$-well then $x$ is $(\sigma,i)$-well.
Set $x\equiv_{4i+1}^\dec y$ iff $x\equiv_{4i+1}y$ and $x$ is not $(\sigma,i)$-well.

Suppose $x\equiv_{4i+1}^\undec y$. Set $x\leq_{4i+2}y$ iff, if $e_y$ is small then $e_x$ is small.
Set ${\equiv_{4i+2}^\dec}={\equiv_{4i+1}^\dec}$.

Suppose $x\equiv_{4i+2}^\undec y$. We have a few subcases.

\begin{scase}\label{scase:S_x_small} $e_x$ is small (so $e_y$ is also small).
 
Let 
\[ k_x=\resl^{\Tt_x,\Uu}(e_x), \]
\[ \res^{\Tt_x,\Uu}(e_x)=(\Vv_x,\left<\Psi_{xj}\right>_{j\leq k_x},\psi_x), \]
\[\Psi_{xj}=(\beta_{xj},A_{xj},\psi_{xj},\alpha_{xj},d_{xj}), \]
\[ \exitresadd^{\Tt_x,\Uu}(e_x)=(\wt{\alpha}_x,\wt{\psi}_x). \]
Define:
\begin{enumerate}[label=--]
\item $\alpha_{xj}^*=\OR(\prodseg^{M^{\Vv_x}_j}(\alpha_{xj}))$ for $j\in[0,k_x]$,
 \item $\beta_{xj}^*=\OR(\prodseg^{M^{\Vv_x}_{j-1}}(\beta_{xj}))$ for $ j\in(0, k_x]$,
 \item $\params_{x0}=\langle\rho_\om^{A_{x1}},\beta_{x1},t^{M^{\Vv_x}_{0}}_{\beta_{x1}},\beta_{x1}^*\rangle$ if $0<k_x$,
 \item $\params_{xj}=\langle\alpha_{xj},t^{M^{\Vv_x}_{j}}_{\alpha_{xj}},\alpha_{xj}^*,\rho_\om^{A_{x,j+1}},\beta_{x,j+1},t^{M^{\Vv_x}_{j}}_{\beta_{x,j+1}},\beta_{x,j+1}^*\rangle$
 for $j\in(0,k_x)$, and
 \item $\params_{xk_x}=\langle\alpha_{xk_x},t^{M^{\Vv_x}_{k_x}}_{\alpha_{xk_x}},\alpha_{xk_x}^*\rangle$ (possibly $k_x=0$).
\end{enumerate}
Now let
\[\params_x=\params_{x0}\conc\params_{x1}\conc\ldots\conc\params_{xk_x}\conc\langle\wt{\alpha}_x,\lh(F^{M^{\Vv_x}_{k_x}}_{\wt{\alpha}_xx})\rangle. \]
Likewise with $y$ replacing  $x$.
Set
$x\leq_{4i+3} y$ iff
\[ 
\left<k_x\right>\conc\params_x\conc\langle\psi_x(\avec_i)\rangle\leq_\lex
\left<k_y\right>\conc\params_y\conc\langle\psi_y(\bvec_i)\rangle.\]
(Note that
$\avec_i\sub\nu(F^{e_x})\sub\dom(\psi_x)$ and $\bvec_i\sub\nu(F^{e_y})\sub\dom(\psi_y)$.)
\end{scase}

\begin{scase}\label{scase:S_x_non-small,Tt_x_small} $e_x$ is non-small but $\Tt_x$ is small.
 
Let
\[ k_x=\critresl^{\Tt_x,\Uu}(e_x,\delta^{e_x}), \]
\[ \critres^{\Tt_x,\Uu}(e_x,\delta^{e_x})=(\Vv_x,\left<\Psi_{xj}\right>_{j\leq k_x},\varrho_x), \]
\[\Psi_{xj}=(\beta_{xj},A_{xj},\psi_{xj},\alpha_{xj},d_{xj}). \]
Let
$\alpha_{xj}^*,\beta^*_{xj},\params_{xj}$ be defined by the equations given in the previous case. Let
 \[\params_x=\params_{x0}\conc\params_{x1}\conc\ldots\conc\params_{xk_x}\conc\left<\OR(\exitcopyseg^{\varrho_x}(e_x))\right>. \]
and $\psi_x=\exitcopymap^{\varrho_x}(e_x)$. Likewise with $y$ replacing  $x$.
Then 
$x\leq_{4i+3} y$ iff
\[ 
\left<k_x\right>\conc\params_x\conc\langle\psi_x(\avec_i)\rangle\leq_\lex
\left<k_y\right>\conc\params_y\conc\langle\psi_y(\bvec_i)\rangle.\]
\end{scase}
\begin{scase} $\Tt_x$ is non-small.
 
Let $\alpha_x=\OR(\exitcopyseg^{\pi_x}(e_x))$ and $\psi_x=\exitcopymap^{\pi_x}(e_x)$;
likewise for $y$.
Then $x\leq_{4i+3}y$ iff
$\langle\alpha_x,\psi_x(\avec_i)\rangle\leq_\lex
\langle\alpha_y,\psi_y(\bvec_i)\rangle$.
\end{scase}

This completes all subcases. In each subcase set ${\equiv_{4i+3}^\dec}={\equiv_{4i+2}^{\dec}}$.

Now suppose  $x\equiv_{4i+3}^\undec y$.
We define
\[ \Tt'_x=\Tt_{x,i+1}=\text{the sse-}n_0\text{-maximal tree given by }\Tt_x\conc\left<F^{e_x}\right>, 
\]
\[ \Tt'_y=\Tt_{y,i+1}=\text{the sse-}n_0\text{-maximal tree given by }\Tt_y\conc\left<F^{e_y}\right>; 
\]
note these trees make sense. Set $x\leq_{4i+4} y$ iff
\begin{enumerate}[label=]
 \item if
 \begin{enumerate}[label=]\item $\Big(i+1\in\mathscr{N}$ iff $[0,i+1]^{\Tt'_y}\cap\dropset^{\Tt'_y}_{\deg}=\emptyset\Big)$
 \end{enumerate}
 then
 \begin{enumerate}[label=]
 \item $\Big(i+1\in\mathscr{N}$ iff
 $[0,i+1]^{\Tt'_x}\cap\dropset^{\Tt'_x}_{\deg}=\emptyset\Big)$ and
$\Big($if
$\deg^{\Tt'_y}(i+1)=n_{i+1}$ then $\deg^{\Tt'_x}(i+1)=n_{i+1}\Big)$.
\end{enumerate}
\end{enumerate}
Set $x\equiv_{4i+4}^\dec y$ iff $x\equiv_{4i+4} y$ and 
\begin{enumerate}[label=]
 \item if $\Big(i+1\in\mathscr{N}$ iff $[0,i+1]^{\Tt'_x}\cap\dropset^{\Tt'_x}=\emptyset\Big)$
 then $\deg^{\Tt'_x}(i+1)\neq n_{i+1}$.
\end{enumerate}

If $x\equiv_{4i+4}^\undec y$ then we have established the inductive hypotheses.
In particular, $\lifttree^{\Tt'_x,\Yback}=\Uu'=\lifttree^{\Tt'_y,\Yback}$. For example in Subcase \ref{scase:S_x_small}, this is because $\lifttree^{\Tt_x,\Yback}=\Uu=\lifttree^{\Tt_y,\Yback}$,
the resurrections for $e_x$ and $e_y$ produced the same trees $\Vv_x,\Vv_y$,
the backgrounds $F^*_x,F^*_y$ for the lifts of $F^{e_x},F^{e_y}$ are identical,
and likewise the projecta $\rho_\om^{A_{xj}}=\rho_\om^{A_{yj}}$ for each $j\in[1,k_x=k_y]$,
as this information is encoded into $\params_x,\params_y$.
It follows that $\pred^{\Uu'}(i+1,0)$ is independent of whether we are considering $\Tt'_x$ or $\Tt'_y$,
as it is determined by $F^*$ and the projecta $\rho_\om^{A_{xj}}=\rho_\om^{A_{yj}}$,
along with earlier such projecta, which also match by induction. 
\end{case}

\begin{case} $i=m$.

Let
\[ \nu_x=\nu_i^{R_x}(\pvec_{n_m+1}^{R_x},\Avec_m), \]
\[ t_x=s^{R_x}(\pvec_{n_m+1}^{R_x},\Avec_m).\]
Let $\nu_y,t_y$ be defined likewise, with $y$ replacing $x$
and $b$ replacing $a$.

Recall that $x\equiv_{4m}^\undec y$. Say that $x$ is \emph{$(\sigma,m)$-well}
iff\[
\text{ if }R_x\text{ is type 3 then }\nu_x\downarrow=\nu(R_x).\]

Set $x\leq_{4m+1} y$ iff, if $y$ is $(\sigma,m)$-well then $x$ is $(\sigma,m)$-well.
Set $x\equiv_{4m+1}^\dec y$ iff $x\equiv_{4m+1}y$ and $x$ is not $(\sigma,m)$-well.

Set ${\leq_{4m+3}}={\leq_{4m+2}}={\leq_{4m+1}}$
and ${\equiv_{4m+3}^{\dec}}={\equiv_{4m+2}^{\dec}}={\equiv_{4m+2}^{\dec}}$.

Suppose $x\equiv_{4m+3}^\undec y$. Let
$\alpha_x=t^{R_x}(\pvec_{n_m+1}^{R_x},\avec_0,\ldots,\avec_{m-1})$
and define $\alpha_y$ analogously.
Then we set $x\leq_{4m+4}y$  iff
\begin{enumerate}[label=]
\item if $\alpha_y\downarrow$ and $\alpha_y\in\OR$ then [$\alpha_x\downarrow$ and $\alpha_x\in\OR$ and 
$\pi_x(\alpha_x)\leq\pi_y(\alpha_y)$].
 \end{enumerate}

(Recall that $t$ is an $(n_m+1)$-term, not a rep-term, so if 
$\alpha_x\downarrow$ then $\alpha_x\in\core_0(R_x)=\dom(\pi_x)$; likewise $\alpha_y$ and $\dom(\pi_y)$.)
\end{case}

Since ${\leq_\sigma}={\leq_{4m+4}}$, this completes the definition.
\end{dfn}

\begin{rem}
 Note that in Subcases \ref{scase:S_x_small} and \ref{scase:S_x_non-small,Tt_x_small},
 for each $j$, either:
 \begin{enumerate}[label=--]
  \item $\alpha_{xj}<\rho_0^{M^\Uu_\infty}$, or
 \item $j=0$ and $\alpha_{x0}=\OR^{M^\Uu_\infty}$ and there is no drop in $[0,\infty]_{\Tt_x}$.
 \end{enumerate}
For suppose $\rho_0^{M^\Uu_\infty}\leq\alpha_{xj}$.
Now $P$ is $M_x$-good.
Since $\alpha_{xj}\geq\rho_0^{M^\Uu_\infty}$,
we have $\prodstage^P(M_x)\geq\rho_0^P$.
Since $P$ is $M_x$-good, therefore $P$ is type 3 and $P$ is a Q-mouse and $\prodstage^P(M_x)=\OR^P$.
Since $\Tt_x$ is small, $[0,\infty]_{\Tt_x}$ cannot drop, as otherwise
$\alpha_{x0}<\rho_0^{M^\Uu_\infty}$.
So $\alpha_{x0}=\OR(M^\Uu_\infty)$.
So we may assume $j>0$ and $\alpha_{xj}<\OR(M^\Uu_\infty)$.
Since $\rho_0^{M^\Uu_\infty}\leq\alpha_{xj}$, $S_x$ is non-small. So we are in Subcase \ref{scase:S_x_non-small,Tt_x_small}
and $k_x\geq 1$.
If $\OR^{S_x}<\delta^{M^{\Tt_x}_\infty}$ then $\alpha_{x1}<\delta^{M^\Uu_\infty}<\rho_0^{M^\Uu_\infty}$,
contradiction. So $\delta^{M^{\Tt_x}_\infty}=\delta^{S_x}$. So $\delta^{S_x}$ is a cardinal of $M^{\Tt_x}_\infty$,
so
\[ k_x=\critresl^{\Tt_x,\Uu}(S_x,\delta^{S_x})=1=j\]
(note this need not equal $\critresl^{\Tt_x,\Uu}(S_x)$, which could be larger than $1$)
and $\rho_\om^{A_{x1}}=\delta^{M^\Uu_\infty}$,
and as $t_{\alpha_{x1}x}^{M^\Uu_\infty}=3$, therefore $\alpha_{x1}=\OR^{A_{x1}}<\rho_0^{M^\Uu_\infty}$, contradiction.
\end{rem}

\subsection{The mouse scale for $1$-$\CC$-closed operators}

\begin{dfn}
Let $\Mmm$ be a $1$-$\CC$-closed function.
 Let $A\in\Gamma^\Mmm$ and $\psi_0\in\Ll_{\Sigma_1}$ with $A=A_{\psi_0}^\Mmm$.
 The \emph{mouse scale}
 on $A$ is the (putative) very good scale  on $A$ given by folding all theory norms
 (as in Definition \ref{dfn:1-CC-closed_theory_norm}) and lifting norms
 (as in \ref{dfn:lifting_norm}) together,
 in a natural manner,
 and whose first norm is $\leq_0$
 (as in \ref{dfn:lifting_norm}).
\end{dfn}

\subsection{The mouse scale for exactly reconstructing operators}

 \begin{dfn}[Unbounded primary norm for exactly reconstructing] Let $\Mmm$ be an exactly reconstructing mouse operator. Let $\psi_0\in\Ll_{\Sigma_1}(\dot{p})$ and $A=A_{\psi}^\Mmm$. We define the \emph{unbounded primary prewellorder} $\leq_{\Mmm\psi}$
 on $A$
as follows: for $x,y\in A$, set $x\leq_{\Mmm\psi} y$ iff, letting $p=p_1^{M_{(x,y)}}$, we have
\[ M_{(x,y)}\sats\ex\alpha\ \Big[(N_{x\infty}\wr\alpha,p)\sats\psi\ \&\ \all\beta<\alpha\Big((N_{y\infty}\wr\beta,p)\sats\neg\psi_0\Big)\Big]\text{''}.\]
Write $\Phi_{\Mmm\psi}$ for the corresponding norm.\end{dfn}

 \begin{dfn}\label{dfn:mouse_scale_exactly_reconstructing}
Let $T$ be good, exactly reconstructing, cofinal, uniformly boundedly $1$-solid. Let $\Mmm=\Mmm_T$ and $\Gamma=\Gamma_M$. Let $\psi$ be a $\Sigma_1$ sentence of $\Ll_\pm(\dot{p})$.
Let $A=A_\psi$.

The \emph{mouse scale} on $A$ is the (putative)
very good scale on $A$ given by folding together, in a natural manner:
\begin{enumerate}[label=(\roman*)]
 \item $\Phi_{\Mmm\psi}$ (the unbounded primary norm),
 \item\label{item:theory_norms_in_mouse_scale} the $T$-uniformly bounded theory norms,
 \item\label{item:lifting_norms_in_mouse_scale} the lifting norms $\Phi_\sigma$ for nice norm descriptions $\sigma$.
\end{enumerate}

Now suppose that $T$ is also uniform.
Let $\psi$ be as before, but also uniformly bounded. The \emph{norm-uniformly-bounded mouse scale} on $A$ is the (putative) very good scale on $A$ given by folding in the norms in \ref{item:theory_norms_in_mouse_scale}
and \ref{item:lifting_norms_in_mouse_scale} above.
\end{dfn}

\section{The lifting norms are norms}\label{sec:lifting_norms_are_norms}

Our next goal is to verify that the lifting norms are in fact norms. We will also
establish that the mouse scales we have defined are appropriately definable.
The following notion will be a key tool in the proof; it is an instance of the more general notion of \emph{tree embedding} described in \cite{iter_for_stacks}.

\begin{dfn}[Coarse tree embedding]\label{dfn:coarse_tree_emb}\index{Coarse tree embedding}
 Let $Y$ be $n$-sound and
 $\Uu,\Xx$ be successor length, $n$-maximal trees on $Y$, nowhere dropping in model or degree,
 with $\lh(\Uu)$ finite and $E^\Uu_\alpha\in\es(\core_0(M^\Uu_\alpha))$ for all $\alpha+1<\lh(\Uu)$.
 \footnote{The last condition is not particularly important, but is a simplification we can assume.}
 Let $\Delta:\lh(\Uu)\to\lh(\Xx)$, and let 
$\delta_\alpha=\Delta(\alpha)$ and $\gamma_0=0$ and 
$\gamma_{\alpha+1}=\delta_\alpha+1$ for  $\alpha+1<\lh(\Uu)$.

We write $\Delta:\Uu\hookrightarrow_{\mathrm{c}}\Xx$ (the subscript ``$c$'' stands for ``coarse'') iff there is 
$\left<\pi_\alpha,\sigma_\alpha\right>_{\alpha<\lh(\Uu)}$ such that:
\begin{enumerate}
 \item $\Delta$ is strictly increasing and $\delta_{\lh(\Uu)-1}=\lh(\Xx)-1$; hence, $\lh(\Xx)$ is 
partitioned by the intervals $[\gamma_\alpha,\delta_\alpha]$.
\item $\gamma_\alpha\leq_\Xx\delta_\alpha$; let 
$I_\alpha=[\gamma_\alpha,\delta_\alpha]_\Xx$.
\item $\alpha <_\Uu\beta$ iff $\gamma_\alpha<_\Xx\gamma_\beta$.
\item For each $\alpha<\lh(\Uu)$,
$\pi_\alpha:M^\Uu_\alpha\to M^\Xx_{\gamma_\alpha}$
is a near $n$-embedding,\footnote{\label{ftn:near_embeddings_for_coarse_tree_embeddings}The fact that we can expect near $n$-embeddings here, instead of just weak $n$-embeddings, is via the argument in \cite{fs_tame}; this is also discussed in \cite{iter_for_stacks}.} and
\[ \sigma_\alpha=i^\Xx_{\gamma_\alpha,\delta_\alpha}\com\pi_\alpha:M^\Uu_\alpha\to M^\Uu_{\delta_\alpha}. \]
\item For all $\alpha<\beta<\lh(\Uu)$, we have $\sigma_\alpha\rest\nu(E^\Uu_\alpha)\sub\pi_{\beta}$ 
and if $E^\Uu_\alpha$ is not type 3 then $\sigma_\alpha\rest\lh(E^\Uu_\alpha)\sub\pi_\beta$.
\item For all $\alpha+1<\lh(\Uu)$, we have $E^\Xx_{\delta_\alpha}=\sigma_\alpha(E^\Uu_\alpha)$.
\item Let $\alpha+1<\lh(\Uu)$ and $\beta=\pred^\Uu(\alpha+1)$ and 
$\xi=\pred^\Xx(\gamma_{\alpha+1})$.
Then $\xi\in I_\beta$. Let $\kappa=\crit(E^\Uu_\alpha)$.
Then 
\[ \sigma_\alpha\rest\kappa^{+\exit^\Uu_\alpha}=i^\Xx_{\gamma_\beta,\xi}
\com\pi_\beta\rest\kappa^{+M^\Uu_\beta} \]
and
\[\pi_{\alpha+1}:M^\Uu_{\alpha+1}\to M^\Xx_{\delta_\alpha+1}=M^\Xx_{\gamma_{\alpha+1}} \]
is defined from $i^\Xx_{\gamma_\beta\xi}\com\pi_\beta$ and 
$\sigma_\alpha\rest\exit^\Uu_\alpha:\exit^\Uu_\alpha\to\exit^\Xx_{\delta_\alpha}$ as in the proof 
of the Shift Lemma.
\item If $\alpha<_\Uu\beta$ then $\pi_\beta\com 
i^\Uu_{\alpha\beta}=i^\Xx_{\gamma_\alpha\gamma_\beta}\com\pi_\alpha$.\footnote{This condition follows from the others.}
\end{enumerate}

Write $\delta^\Delta_\alpha=\delta_\alpha$, $\gamma^\Delta_\alpha=\gamma_\alpha$,
and using the observation that $\pi_\alpha,\sigma_\alpha$ are uniquely determined by $\Delta$,
write $\pi^\Delta_\alpha=\pi_\alpha$ and $\sigma^\Delta_\alpha=\sigma_\alpha$.
\end{dfn}
\begin{dfn}\label{dfn:alpha,x-compatible}
For an iterable premouse $R$ with $x\in R$, write $\Omega^R_x$ for the least $\alpha$
such that either $N_{\alpha+\om,x}^R$ is undefined, or $\core_\om(N_{\alpha x}^R)=M_x$.

Let $R,S$ be iterable premice with $x\in R\inter S$. For $\alpha\leq\min(\Omega^R_x,\Omega^S_x)$,
say that $R,S$ are \emph{$(\alpha,x)$-compatible} iff for all $\beta\leq\alpha$, we have
\begin{enumerate}[label=(\roman*)] \item $F\eqdef F_{\beta x}^R\neq\emptyset$ iff $G\eqdef F_{\beta x}^S\neq\emptyset$,
and
\item if $F\neq\emptyset$, then $\mu\eqdef\crit(F)=\crit(G)$,
and letting
\[ R'=R|\mu^{+(R|\lh(F))}\text{ and }S'=S|\mu^{+(S|\lh(G))},\]
then
$F\rest X=G\rest X$
where
$X=(R'\inter S')\cross[\min(\lh(F),\lh(G)]^{<\om}$.
\end{enumerate}

Say $R,S$ are \emph{$({<\alpha},x)$-compatible} iff they are $(\beta,x)$-compatible for all $\beta<\alpha$.
Say that $R,S$ are \emph{$M_x$-compatible} iff $R,S$ are $(\alpha,x)$-compatible
for all $\alpha\leq\min(\Omega^R_x,\Omega^S_x)$.
\end{dfn}
\begin{dfn}\label{dfn:CC-bar}
Given $\CC=\left<N_\alpha,t_\alpha,F_\alpha\right>_{\alpha<\lambda}$,
let $\bar{\CC}=\left<N_\alpha,t_\alpha\right>_{\alpha<\lambda}$.
\end{dfn}

\begin{lem}\label{lem:CCbar_match_from_comp}
Let $R,S$ be iterable.
Then:
\begin{enumerate}
 \item\label{item:compat_implies_CC-bar_agmt} Suppose that $R,S$ are
$(\alpha,x)$-compatible.
 Then $\bar{\CC}^R_x\rest(\alpha+1)=\bar{\CC}^S_x\rest(\alpha+1)$.
 \item\label{item:compat_through_min_of_Omegas_and_goodness_implies_Omegas_match} If $R,S$
 are $(\Omega,x)$-compatible
 where $\Omega=\min(\Omega^R_x,\Omega^S_x)$,
 and $R,S$ are both $M_x$-good,
 then $\Omega^R_x=\Omega^S_x$.
 \end{enumerate}
\end{lem}
\begin{proof}
 Part \ref{item:compat_implies_CC-bar_agmt}: This is a straightforward induction on $\alpha$, together with a quick observation.
 Write $\bar{\CC}=\bar{\CC}_x$ etc. Suppose that the claim fails, and
 $\alpha$ is least such. 
 It easily follows that $t_\alpha^R=2$ iff $t_\alpha^S\neq 2$.
 Suppose $t_\alpha^R=2$. By compatibility, then $t_\alpha^S=1$
 and $\kappa\eqdef\crit(F_\alpha^R)=\crit(F_\alpha^S)$. In particular,
 $\kappa$ is measurable in both $R$ and $S$. Let $F=F_{\kappa 0}^R$ and $G=F_{\kappa 0}^S$,
 and let $\beta$ be least such that $t_\beta^R=2$ and $\kappa=\crit(F_\beta^R)$. Then
 $\beta$ is least such that $t_\beta^S=2$ and $\kappa=\crit(F_\beta^S)$ (with $\gamma$ as in \ref{lem:least_t=2} with respect to $R$,
 then $N_\gamma^R=N^S_\gamma$ is the Q-structure for $N=N_\kappa^R=N_\kappa^S$,
 so $S|\gamma$ is the Q-structure for $S|\kappa$).
 So $F_{\beta}^R=F$ and $F_\beta^S=G$.
We have
 $\kappa^{+i^R_F(N)}<\lh(F)$ and $\kappa^{+i^S_G(N)}<\lh(G)$, so by $(\beta,x)$-compatibility,
 $\xi\eqdef\kappa^{+i^R_F(N)}=\kappa^{+i^S_G(N)}$. But then $\alpha<\xi$
 and $t_\alpha^S=2$, contradiction.
 
 Part \ref{item:compat_through_min_of_Omegas_and_goodness_implies_Omegas_match} is an immediate corollary.
\end{proof}

We now reach the first central argument of the paper, and the first key step in showing that the mouse scale is indeed  a scale:

\begin{lem}[Norm invariance]\label{lem:norm_invariance}
Let  $\sigma$ be a norm description. Let $x,y\in A$.
Let $P$ be either $M_x$-good or $M_y$-good,
and let $Q$ also be either $M_x$-good or $M_y$-good.
Then
$x\leq^P_\sigma y$ iff $x\leq^Q_\sigma y$.
\end{lem}
\begin{proof}
The proof is via a comparison modelled on that used in the proof of Lemma \ref{lem:first_comparison}.
The new features we must deal with are as follows.
Firstly, $\CC$ is somewhat more complex than the $L[\es]$-construction $\DD$ used in \ref{lem:first_comparison}.
The modifications needed to accommodate this are, however, quite straightforward.
During the comparison we use order $0$ measures corresponding to stages $\alpha$ where
at least one side has $t_\alpha=2$ and incompatibility (in the sense of \ref{dfn:alpha,x-compatible}) leads us to do so. Regarding extenders at stages $\alpha$ with $t_\alpha=3$,
we will in fact observe that no incompatibilities with such extenders can arise.
With just these two modifications to the proof, one can verify that the putative norms of depth $1$
are in fact norms. For the higher order norms,
one needs more care, and we discuss this now.

Let $m+1$ be the depth of $\sigma$. 
Let
$\leq^\Yback_i,\equiv_i^{\undec,\Yback}$ and $\leq^\Zback_i,\equiv_i^{\undec,\Zback}$
be as in \ref{dfn:lifting_norm}, as computed from $\Yback$ and $\Zback$ respectively.
We want to see that
\[ x\leq^\Yback_{4m+4} y\ \iff\ x\leq^\Zback_{4m+4} y.\]

So suppose not and let $\bar{m}\in[-1, m]$ be largest such that
[$x\leq^\Yback_{4\bar{m}} y$ iff $x\leq^\Zback_{4\bar{m}}y$].
Then $x\equiv^P_{4\bar{m}}y$ and $x\equiv^Q_{4\bar{m}}y$ by our contradictory hypothesis,
and therefore $x\equiv^{\undec,\Yback}_{4\bar{m}+2} y$ and $x\equiv^{\undec,\Zback}_{4\bar{m}+2}y$,
because the extra considerations are first-order over $M_x,M_y$,
hence independent of $\Yback,\Zback$. We begin by specifying some finite trees $\Tt_x,\Tt_y$
on $M_x,M_y$, and some partial resurrection trees $\Uu^\Yback,\Uu^\Zback$ on $\Yback,\Zback$.

\begin{case} $\bar{m}=-1$.
 
 Let $m'=0$ and $\Tt_x,\Tt_y,\Uu^\Yback,\Uu^\Zback$ be the trivial trees on $M_x,M_y,\Yback,\Zback$
 of the relevant degree
 (these trees are actually irrelevant in this case).
\end{case}

\begin{case} $\bar{m}\geq 0$.
  
Then note that $\deg(M_x)=\deg(M_y)=n_0$ and we have sse-$n_0$-maximal 
trees $\bar{\Tt}_x=\Tt_{x\bar{m}}$ and $\bar{\Tt}_y=\Tt_{y\bar{m}}$ on $M_x,M_y$, each of length $\bar{m}+1$.
If $\bar{\Tt}_x$ (equivalently $\bar{\Tt}_y$) is small let $m'=\bar{m}$;
otherwise let $m'<\bar{m}$ be least such that $\exit^{\bar{\Tt}_x}_{m'}$ (equivalently $\exit^{\bar{\Tt}_y}_{m'}$) is non-small.
Let $\Tt_x=\bar{\Tt}_x\rest(m'+1)$.
Then (note that)
\[ \Uu^\Yback_0\eqdef\lifttree^{\Tt_x\Yback}=\lifttree^{\Tt_y\Yback}\text{ and }
\Uu^\Zback_0\eqdef\lifttree^{\Tt_x\Zback}=\lifttree^{\Tt_y\Zback}\]
and $\Uu^\Yback_0,\Uu^\Zback_0$ are both small.
\begin{scase}\label{scase:m'<mbar}
$m'<\bar{m}$ ($\bar{\Tt}_x,\bar{\Tt}_y$ are non-small).

Let
$e_x=\exit^{\bar{\Tt}_x}_{m'}$ and $e_y=\exit^{\bar{\Tt}_y}_{m'}$.
 Let
\[\Uu^\Yback_{1}=\critrestree^{\Tt_x\Uu^\Yback_0}(e_x,\delta^{e_x})=
\critrestree^{\Tt_y\Uu^\Yback_0}(e_y,\delta^{e_y}).\]
\end{scase}

In the remaining subcases
(in which $m'=\bar{m}$),
if $\bar{m}<m$, we
adopt the notation from the definition $\leq_{4\bar{m}+3}$ of \ref{dfn:lifting_norm},
writing superscript $\Yback$ or $\Zback$ as required.
Note then that in this case, $e_x\downarrow$ and $e_y\downarrow$.

\begin{scase}\label{scase:m'=mbar_and_S_x_small} $m'=\bar{m}<m$ and $e_x,e_y$ are small.

Adopt also the notation from the definition of $\leq_{4\bar{m}+3}$ in its Subcase \ref{scase:S_x_small}
(in which $e_x$ is small).
Clearly $k_x=k_y$ (this is the first ordinal considered for $\leq_{4\bar{m}+3}$).
If there is $j\leq k_x$ such that $\params^\Yback_{xj}\neq\params^\Yback_{yj}$,
let $j^\Yback$ be the least such, and otherwise let $j^\Yback=k_x$. Let
$\wt{j}=\min(j^\Yback,j^\Zback)$. Let
\[ \Uu_1^\Yback=\restree^{\Tt_x\Uu_0^\Yback}(e_x)\rest(\wt{j}+1)=\restree^{\Tt_y\Uu_0^\Yback}(e_y)\rest(\wt{j}+1). \]
\end{scase}
\begin{scase}\label{scase:S_x_non-small} $m'=\bar{m}<m$ and $e_x,e_y$ are non-small.

Adopt the notation from Subcase \ref{scase:S_x_non-small,Tt_x_small} of the definition of $\leq_{4\bar{m}+3}$
(in which $e_x$ is non-small but $\Tt_x$ is small).
Again $k_x=k_y$. Let $\wt{j}$ be defined as above and
\[ \Uu_1^\Yback=\critrestree^{\Tt_x,\Uu_0^\Yback}(e_x,\delta^{e_x})\rest(\wt{j}+1)=
 \critrestree^{\Tt_y,\Uu_0^\Yback}(e_y,\delta^{e_y})\rest(\wt{j}+1).
\]
\end{scase}

\begin{scase}\label{scase:m'=mbar=m}
 $m'=\bar{m}=m$.
 
 Let $\Uu_1^P=\emptyset$.
\end{scase}

Now in each subcase let $\Uu_2^\Yback=\Uu_0^\Yback\conc\Uu_1^\Yback$.
Let $\Uu_1^\Zback,\Uu_2^\Zback$ be likewise.
We have $\lh(\Uu_2^\Yback)=\lh(\Uu_2^\Zback)$. Let $\Yback_{ij}=M^{\Uu_2^\Yback}_{ij}$
and $\Zback_{ij}=M_{ij}^{\Uu_2^\Zback}$.

If there is $(i,j)$ such that $(i,j+1)<\lh(\Uu_2^\Yback)$ and
$t_{\beta_{j+1}x}^{\Yback_{ij}}\neq t_{\beta_{j+1}'x}^{\Zback_{ij}}$,
where $\beta_{j+1}$ is as in the definition of $\leq_{4i+3}^\Yback$
and $\beta_{j+1}'$ likewise for $\leq_{4i+3}^\Zback$
(hence, $t_{\beta_{j+1}x}^{\Yback_{ij}}=2$ iff $t_{\beta_{j+1}'x}^{\Zback_{ij}}\neq 2$),
then let $(i_\dis,j_\dis)$ be the least such $(i,j)$; otherwise let $(i_\dis,j_\dis)$
be the largest index in $\dom(\Uu_2^\Yback)$
(\emph{dis} for \emph{distinction}).

Let $\wt{\Uu}^P=\Uu_2^P\rest(i_\dis,j_\dis+1)$ and $\wt{\Uu}^Q=\Uu_2^Q\rest(i_\dis,j_\dis+1)$.
Note that $\wt{\Uu}^P,\wt{\Uu}^Q$ have the same tree structure and padding (but  we could have, for example, $\crit(E^{\wt{\Uu}^P}_\alpha)\neq\crit(E^{\wt{\Uu}^Q}_\alpha)$).
Finally let $\Uu^P,\Uu^Q$ be the non-padded trees, indexed by ordinals,
equivalent to $\wt{\Uu}^P,\wt{\Uu}^Q$. Let
\[ I_2=\{\alpha\mid\alpha+1<\lh(\Uu^P)\ \&\ E^{\Uu^P}_\alpha\text{ is an order 0 measure}\}, \]
\[ I_1=\{\alpha\mid\alpha+1<\lh(\Uu^P)\}\cut I_2.\]
Note that $\alpha\in I_2$ iff $E^{\Uu^P}_\alpha$ is the measure corresponding
to some $\beta$ such that $t^{M^{\Uu^P}_\alpha}_{\beta x}=2=t^{M^{\Uu^P}_\alpha}_{\beta y}$.
\end{case}

This completes all cases. Directly from the definitions we have:

\begin{clm}\label{clm:Uu_small}
$\Uu^P,\Uu^Q$ are small. 
\end{clm}

If $\bar{m}\geq 0$ let $\mathscr{I}:I_1\to i_\dis$
be the order-preserving bijection and 
for $i<i_\dis$ let
\[ (\xi,\sigma)=\exitresadd^{\Tt_x,\Uu_0^P}_i, \]
\[ (\xi,\sigma')=\exitresadd^{\Tt_y,\Uu_0^P}_i, \]
and define
\begin{equation}\label{eqn:a^*,xi^*}\avec^{*\Yback}_i=\sigma(\avec_i)=\sigma'(\bvec_i)\text{ and }\xi^{*\Yback}_i=\xi.\end{equation}
Define $\avec^{*\Zback}_i,\xi^{*\Zback}_i$ analogously.

We will now define padded iteration trees $\Vv,\Ww$ on $\Yback,\Zback$ respectively, of successor length.
Write $P_\alpha=M^\Vv_\alpha$ and $Q_\alpha=M^\Ww_\alpha$.
We produce $(\Vv,\Ww)$ by a comparison similar to that used in the proof of  Lemma
\ref{lem:first_comparison}. However, we will modify the process,
arranging that there is also a coarse tree embedding \[\Delta:\Uu^P\hookrightarrow_{\mathrm{c}}\Vv,\]
with $\Delta(k)+1=\lh(\Vv)$ where $k+1=\lh(\Uu^{\Yback})=\lh(\Uu^{\Zback})$,
and also that \[\Delta:\Uu^Q\hookrightarrow_{\mathrm{c}}\Ww\]
(with the same $\Delta$). Using this $\Delta$, we will see that any distinction in the relevant calculations
between the last models of $\Uu^\Yback$ and $\Uu^\Zback$ will reflect to the last models of $\Vv,\Ww$.
But there will be no such distinction between the latter models, because the comparison process 
will remove all distinctions. Therefore there is no distinction between the calculations
in the last models of $\Uu^\Yback,\Uu^\Zback$. From there it will be easy enough to complete the proof.

We construct $(\Vv,\Ww)\rest\rest(\chi+1)$
by induction on ordinals $\chi$.
Along with $(\Vv,\Ww)\rest(\chi+1)$, we will also define a strictly increasing sequence of ordinals $\left<\beta_\alpha\right>_{\alpha<\chi}$. For each $\alpha<\chi$,
if $E^\Vv_\alpha\neq\emptyset$ then
$\pred^\Vv(\alpha+1)$ will be the least $\delta$ such that $\crit(E^\Vv_\alpha)<\beta_\delta$. If $E^\Vv_\alpha=\emptyset$ then $\pred^\Vv(\alpha+1)=\alpha$.  Likewise for $\Ww$.
The $\beta_\alpha$ which correspond to strengths of extenders, but note that the choice of $\beta_\alpha$ does not depend on whether we are considering $\Vv$ or $\Ww$.
We will also define $h_\chi<\lh(\Uu^P)$ and $\Delta_\chi$ such that
\[ \Delta_\chi:\Uu^\Yback\rest(h_\chi+1)\hookrightarrow_{\mathrm{c}}\Vv\rest(\chi+1),\]
\[ \Delta_\chi:\Uu^\Zback\rest(h_\chi+1)\hookrightarrow_{\mathrm{c}}\Ww\rest(\chi+1).\]
Write $\delta^\chi_j=\Delta_\chi(j)$ and $\gamma^\chi_j=\gamma_j^{\Delta_\chi}$ 
(as in \ref{dfn:coarse_tree_emb}). (Note that $\Delta_\chi,\gamma^\chi_j,\delta^\chi_j$ are independent of which pair $(\Uu^\Yback,\Vv)$ or $(\Uu^\Zback,\Ww)$ we are considering.
But it seems that $[\gamma^\chi_j,\delta^\chi_j]^\Vv$ might not equal $[\gamma^\chi_j,\delta^\chi_j]^\Ww$.)
Write $\pi^{\Vv\chi}_j=\pi_j^{\Delta_\chi}$ in the sense of
$(\Uu^\Yback,\Vv)$, and $\pi^{\Ww\chi}_j=\pi_j^{\Delta_\chi}$ in the sense of $(\Uu^\Zback,\Ww)$, etc.

At $\chi=0$ we set $h_0=0$ and $\Delta_0(0)=0$.

Now suppose we have defined $(\Vv,\Ww)\rest(\chi+1)$ and $\left<\beta_\alpha\right>_{\alpha<\chi}$ and $\left<h_\alpha,\Delta_\alpha\right>_{\alpha\leq\chi}$.
We first look for relevant ``least disagreements'' between $\CC_x^{P_\chi},\CC_x^{\Zback_\chi}$,
and between $\CC_y^{P_\chi},\CC_y^{\Zback_\chi}$.

If $P_\chi,\Zback_\chi$ are $M_x$-compatible then let $\alpha_{\chi x}=\infty$
and $G^\Vv_{\chi x}=\emptyset=G^\Ww_{\chi x}$.
Suppose otherwise and let $\alpha=\alpha_{\chi x}$ be the least $\alpha$ such that
$P_\chi,\Zback_\chi$ are $(\alpha,x)$-incompatible,
and let $G^\Vv_{\chi x}=F_{\alpha x}^{P_\chi}$ and $G^\Ww_{\chi x}=F_{\alpha x}^{\Zback_\chi}$.

Define $\alpha_{\chi y},G^\Vv_{\chi y},G^\Ww_{\chi y}$ likewise. Let $\alpha_\chi=\min(\alpha_{\chi x},\alpha_{\chi y})$.

\begin{clm}\label{clm:if_full_compatibility_then_basic_agreement}
 Suppose $\alpha_\chi=\infty$
 (that is,
  $P_\chi,Q_\chi$ are both $M_x$-compatible and $M_y$-compatible).
 Then $P,Q$ are both $M_x$-good and $M_y$-good, $\Omega^P_x=\Omega^P_y$ and $\Omega^Q_x=\Omega^Q_y$  (cf.~Definition \ref{dfn:alpha,x-compatible}),
 $P_\chi,Q_\chi$
 are both $M_x$-good and $M_y$-good,
 and 
$\Omega^{P_\chi}_x=\Omega^{Q_\chi}_x=\Omega^{P_\chi}_y=\Omega^{Q_\chi}_y$.
\end{clm}
\begin{proof}
 We know that $P$ and $Q$ are both either  $M_x$-good or $M_y$-good.
 Suppose $P$ is $M_x$-good
 but $Q$ is not $M_x$-good; so $Q$ is $M_y$-good.
 Then $P_\chi$ is $M_x$-good
 and $Q_\chi$ is $M_y$-good but not $M_x$-good. Let $\Omega=\min(\Omega^{P_\chi}_x,\Omega^{Q_\chi}_y)$.
 Suppose $\Omega=\Omega^{P_\chi}_x$.
 Then note that by $M_x$-compatibility,
 $Q_\chi$ is $M_x$-good,
 a contradiction. So $\Omega=\Omega^{Q_\chi}_y<\Omega^{P_\chi}_x$. Therefore $P_\chi$ is $M_y$-good and $\Omega^{P_\chi}_y=\Omega<\Omega^{P_\chi}_x$.
 But then $P$ is $M_y$-good
 and $\Omega^P_y<\Omega^P_x$.
 Therefore $y<^P_0x$.
 But since $Q$ is $M_y$-good and non-$M_x$-good, also $y<^Q_0x$, contradicting the
 disagreement between $P,Q$.
 
 So by symmetry, $P$ is $M_x$-good iff $Q$ is $M_x$-good, and $P$ is $M_y$-good iff $Q$ is $M_y$-good. But if $P,Q$ are $M_x$-good but non-$M_y$-good then $x<_0^Py$ and $x<_0^Qy$, again a contradiction. So $P,Q$ are both $M_x$- and $M_y$-good. Suppose that $\Omega^P_x<\Omega^P_y$. Then by   compatibility and elementarity of iteration maps,
 \[ \Omega^{Q_\chi}_x=\Omega^{P_\chi}_x<\Omega^{P_\chi}_y<\Omega^{Q_\chi}_y,\]
 so again by elementarity, $\Omega^Q_x<\Omega^Q_y$.
 So $x<_0^Py$ and $x<_0^Qy$, again a contradiction.
 
 So we get $\Omega^P_x=\Omega^P_y$ and $\Omega^Q_x=\Omega^Q_y$,
 and so by elementarity and compatibility,
 $\Omega^{P_\chi}_x=\Omega^{P_\chi}_y=\Omega^{Q_\chi}_x=\Omega^{Q_\chi}_y$.
\end{proof}

\begin{clm}\label{clm:alpha_chi<infty_implies_alpha_chi_leq_Omegas}
 Suppose $\alpha_\chi<\infty$.
 Then $\alpha_\chi\leq\min(\Omega_x^{P_\chi},\Omega_x^{Q_\chi},\Omega_y^{P_\chi},\Omega_y^{Q_\chi})$.
\end{clm}
\begin{proof}
 We may assume $\alpha_\chi=\alpha_{\chi x}$,
 and so $\alpha_\chi\leq\min(\Omega_x^{P_\chi},\Omega_x^{Q_\chi})$ by definition.
 Suppose $\min(\Omega_y^{P_\chi},\Omega_y^{Q_\chi})<\alpha_\chi$.
 Say $\Omega_y^{P_\chi}<\alpha_\chi$.
 Then $P_\chi$ is $M_y$-good,
 and so by compatibility, so is $Q_\chi$,
 with $\Omega_y^{Q_\chi}=\Omega_y^{P_\chi}$.
 It follows with elementarity and further compatibility
 that $P,Q$ are $M_y$-good,
 and if $P$ is $M_x$-good then $\Omega^P_y<\Omega^P_x$,
 and if $Q$ is $M_x$-good then $\Omega^Q_y<\Omega^P_x$.
 But then $y<_0^Px$ and $y<_0^Qx$, a contradiction. So instead,
 $\Omega_y^{Q_\chi}<\alpha_\chi\leq\Omega_y^{P_\chi}$.
 But then a similar argument again gives a contradiction.
\end{proof}

\begin{clm}\label{clm:t_alpha_chi_not_3} Let $\alpha=\alpha_\chi$. If $\alpha=\alpha_{\chi x}<\infty$ then $t_{\alpha x}^{P_\chi}\neq 3\neq t_{\alpha x}^{\Zback_\chi}$;
likewise for $y$.
\end{clm}
\begin{proof}
Suppose $t^{P_\chi}_{\alpha x}=3$.
Let $N=N^{P_\chi}_{\alpha x}$ and $N'=N^{Q_\chi}_{\alpha x}$.
Note that $N^\passive=(N')^\passive$. Let $\delta=\delta^N$.
So $\delta=\delta^{N'}$.
By \ref{lem:Woodin_exactness},  $t_{\alpha x}^{Q_\chi}=3$ and $\delta$ is the least Woodin of $P_\chi|\alpha$ and of $Q_\chi|\alpha$,
and is a limit cardinal of $P_\chi$ and of $Q_\chi$, and moreover, $P_\chi|\delta$
and $Q_\chi|\delta$ are both small.

Now there is $\gamma\leq\OR^P$ such that $P_\chi|\gamma$ is the unique above-$\delta$-iterable,
$\delta$-sound Q-structure for $P_\chi|\delta$.
For if $\delta$ is not Woodin in $P_\chi$,
this is immediate  (using that $P_\chi|\delta$  is small), so suppose $\delta$ is Woodin in $P_\chi$, and
hence $\delta=\delta^{P_\chi}$.
We just need to see that $P_\chi$ is $\delta$-sound, or equivalently, that $P$ is $\delta^P$-sound.
If $P$ is $M_x$-good
then since $\delta<\alpha$,
we have that $M_x=\core_\om(N_{x\infty}^P)$,
and $P$ is $\delta^P$-sound, by \ref{dfn:M_x-good}.
So suppose $P$ is not $M_x$-good;
so $P$ is $M_y$-good.
It can't be that $M_y=\core_\om(N_{y\xi}^P)$
for some $\xi<\delta^P$,
since $\delta^{P_\chi}<\alpha_\chi=\alpha_{x\chi}\leq\alpha_{y\chi}$.
So again by \ref{dfn:M_x-good},
$P$ is $\delta^P$-sound.

There is  likewise $\gamma'$
such that $Q_\chi|\gamma'$ is the $\delta$-sound Q-structure for $Q_\chi|\delta$.
This is again immediate 
 if $\delta<\delta^{Q_\chi}$,
 so suppose $\delta=\delta^{Q_\chi}$. Since $\delta\leq\alpha_\chi$ and $Q$ is 
 $M_x$-good or $M_y$-good,
 it follows that $Q$ is $\delta^Q$-sound,
 which again suffices.
 
 So $N^{P_\chi}_{\gamma x}$ is the iterable $\delta$-sound Q-structure for $N|\delta$,
and $N^{Q_\chi}_{\gamma' x}$ is the iterable $\delta$-sound Q-structure for $N'|\delta=N|\delta$. So in fact,
$N^{P_\chi}_{\gamma x}=N^{Q_\chi}_{\gamma' x}$
and $\alpha\leq\gamma=\gamma'$, so $N=N'$.

Now $N$ is the output of the P-construction of $P_\chi|\alpha$ over $N|\delta$.
Therefore $F^{P_\chi|\alpha}$ is just a small generic expansion of $F^N$,
and of course $F^{P_\chi|\alpha}\sub P_\chi|\alpha$. Likewise for $F^{Q_\chi|\alpha}$. Therefore
\[ F^{P_\chi|\alpha}\rest(X\cross[\alpha]^{<\om})=F^{Q_\chi|\alpha}\rest(X\cross[\alpha]^{<\om}) \]
where $X=(P_\chi|\alpha)\inter(Q_\chi|\alpha)$, 
and since $P_\chi,Q_\chi$ are already $({<\alpha},x)$-compatible,
therefore $P_\chi,Q_\chi$ are $(\alpha,x)$-compatible, contradiction.
\end{proof}

\begin{clm}\label{clm:compatibility_implies_Q-structures_match}
 Let $\delta<\min(\Omega^{P_\chi}_x,\Omega^{P_\chi}_y,\Omega^{Q_\chi}_x,\Omega^{Q_\chi}_y,\alpha_\chi)$
 be a $P_\chi$-cardinal,
 and suppose no $\delta'<\delta$
 is Woodin in $P_\chi$.
 If $\delta$ Woodin in $P_\chi$
  let $\gamma=\OR^{P_\chi}$,
  and otherwise let $\gamma=\gamma^{P_\chi}_\delta$.
  If $\delta$ is Woodin in $Q_\chi$
  let $\gamma'=\OR^{Q_\chi}$,
  and otherwise let $\gamma'=\gamma^{Q_\chi}_\delta$.
 Then:
 \begin{enumerate}\item\label{item:no_delta'<delta_Woodin} No $\delta'<\delta$ is Woodin in $Q_\chi$.
  \item\label{item:gamma=gamma'}  $\gamma=\gamma'$ and
  $N^{P_\chi}_{\gamma x}=N^{Q_\chi}_{\gamma x}$ is a Q-structure for $N^{P_\chi}_{\delta x}=N^{Q_\chi}_{\delta x}$.
  \item\label{item:if_delta_Woodin_in_Q} Suppose $\delta$ is Woodin in $P_\chi$; so $\gamma=\OR^{P_\chi}$. Then  $P,P_\chi,Q,Q_\chi$ are $M_x$-good and $M_y$-good, $\Omega_x^P=\OR^P=\Omega_y^P$,
   $\Omega_x^Q=\Omega_y^Q$
  (but maybe $\Omega_x^Q<\OR^Q$), and
  $\Omega_x^{P_\chi}=\Omega_y^{P_\chi}=\gamma=\Omega_x^{Q_\chi}=\Omega_y^{Q_\chi}$. Likewise if $\delta$ is Woodin in $Q_\chi$ instead of $P_\chi$.
  \end{enumerate}
\end{clm}
\begin{proof}
 Parts \ref{item:no_delta'<delta_Woodin}
 and \ref{item:gamma=gamma'} follow
 from Claim \ref{clm:t_alpha_chi_not_3} and similar calculations to those in its proof. Part \ref{item:if_delta_Woodin_in_Q}:
 Suppose $\delta$ is Woodin in $P_\chi$
 and $P,P_\chi$ are $M_x$-good. Since $\delta<\Omega_x^{P_\chi}$,
 it follows that $\Omega_x^{P_\chi}=\OR^{P_\chi}=\gamma$. But $\gamma=\gamma'$ and $N_{\gamma x}^{P_\chi}=N_{\gamma x}^{Q_\chi}$ by part \ref{item:gamma=gamma'}, so $Q_\chi$ is $M_x$-good
 with $\Omega_x^{Q_\chi}=\gamma=\Omega_x^{P_\chi}$. So
 $P,Q$ are both $M_x$-good.
 Also since $\delta<\Omega_y^{P_\chi},\Omega_y^{Q_\chi}$, like for $x$,
 we get that $P_\chi$ is $M_y$-good
 iff [$Q_\chi$ is $M_y$-good and $\Omega_y^{Q_\chi}\leq\gamma$],
 and in the case that $P_\chi$ is $M_y$-good,
 then $\Omega_y^{P_\chi}=\gamma=\Omega_y^{Q_\chi}$.
 
 We claim that $P_\chi$ is indeed $M_y$-good.
 For 
 suppose otherwise.
 So if $Q_\chi$ is $M_y$-good then $\Omega_x^{Q_\chi}=\gamma<\Omega_y^{Q_\chi}$.
 Then $P$ is not $M_y$-good,
 and if $Q$ is $M_y$-good
 then $\Omega_x^Q<\Omega_y^Q$.
 But then $x<_0^Py$ and $x<_0^Qy$,
 a contradiction.
 
 So $P_\chi$ is $M_y$-good,
 as is $Q_\chi$,
 and $\Omega_y^{P_\chi}=\gamma=\Omega_y^{Q_\chi}$.
 The conclusion of part \ref{item:if_delta_Woodin_in_Q}
 now follows in this case,
 and so by symmetry, in all cases.
\end{proof}

We now define some (possibly empty) extenders
$G^\Vv_\chi$ and $G^\Ww_\chi$.
These will help us to determine the next extenders $E^\Vv_\chi$ and $E^\Ww_\chi$
to use in $\Vv,\Ww$.

If $\alpha_\chi=\infty$ then define $G^\Vv_\chi=\emptyset=G^\Ww_\chi$.

If $\alpha_\chi=\alpha_{\chi x}<\alpha_{\chi y}$ then define $G^\Vv_\chi=G^\Vv_{\chi x}$ and $G^\Ww_\chi=G^\Ww_{\chi x}$.

Likewise if $\alpha_\chi=\alpha_{\chi y}<\alpha_{\chi x}$.

Suppose $\alpha_{\chi x}=\alpha_{\chi y}<\infty$. If $G^\Vv_{\chi x}\neq\emptyset$
or $G^\Vv_{\chi y}\neq\emptyset$ then let $G^\Vv_\chi=\es^{P_\chi}_\lambda$
where $\lambda=\min(\lh(G^\Vv_{\chi x}),\lh(G^\Vv_{\chi y}))$ where $\lh(\emptyset)=\infty$;
otherwise $G^\Vv_\chi=\emptyset$.
Define $G^\Ww_\chi$ likewise.

We now want to select $E^\Vv_\chi$ and $E^\Ww_\chi$, or terminate 
the comparison.

We say that $\chi$ is a \emph{copying stage} iff
\begin{enumerate}[label=--]
\item $h_\chi+1<\lh(\Uu^\Yback)=\lh(\Uu^\Zback)$; let $E^*=\sigma^{\Vv\chi}_{h_\chi}(E^{\Uu^\Yback}_{h_\chi})$ and 
$F^*=\sigma^{\Ww\chi}_{h_\chi}(E^{\Uu^\Zback}_{h_\chi})$,
\item $\lh(E^*)\leq\lh(G^\Vv_\chi)$ and $\lh(F^*)\leq\lh(G^\Ww_\chi)$ and
\begin{equation}\label{eqn:copying_compatibility}
E^*\inter(X\cross[\theta]^{<\om})=F^*\inter(X\cross[\theta]^{<\om})
\end{equation}
where $X=P_\chi\inter \Zback_\chi$ and $\theta=\min(\lh(E^*),\lh(F^*))$, and
\item If $h_\chi\in I_1$ then
$\sigma^{\Vv\chi}_{h_\chi}(\xi^{*\Yback}_{\mathscr{I}(h_\chi)},\avec_{\mathscr{I}(h_\chi)}^{*\Yback})=\sigma^{\Ww\chi}_{h_\chi}(\xi^{*\Zback}_{\mathscr{I}(h_\chi)},\avec^{*\Zback}_{\mathscr{I}(h_\chi)})$ (see line (\ref{eqn:a^*,xi^*})).
\end{enumerate}

If $\chi$ is a copying stage we set $E^\Vv_\chi=E^*$ and $E^\Ww_\chi=F^*$ (notation as above).

We terminate the comparison at stage $\chi$ iff $\chi$ is a non-copying stage
and $G^\Vv_\chi=\emptyset=G^\Ww_\chi$.

Otherwise we say that $\chi$ is an \emph{inflationary stage}, and set $E^\Vv_\chi=G^\Vv_\chi$ and $E^\Ww_\chi=G^\Ww_\chi$.

After defining $(\Vv,\Ww)\rest(\eta+1)$, we will make the following assumption, denoted 
$\varphi_\eta$, which we verify later: Let $\beta\leq\eta$. Then (i) there are only finitely 
many copying stages $\alpha$ such that $\alpha+1\leq_\Vv\beta$ and $\alpha+1\leq_\Ww\beta$.
It follows that (ii) if $\xi\leq\eta$ and $\eps\leq\eta$, 
then there are only finitely many copying stages $\alpha$ such that $\alpha+1\leq_\Vv\xi$ and 
$\alpha+1\leq_\Ww\eps$. (For otherwise letting $\left<\alpha_n\right>_{n<\om}$ enumerate the 
first $\om$-many, and $\beta=\lim_{n<\om}\alpha_n$,
we have $\beta\leq\eta$ and
$\alpha_n+1<_\Vv\beta$ and $\alpha_n+1<_\Ww\beta$
for all $n<\om$, 
contradicting (i).)

Now given $(\Vv,\Ww)\rest(\chi+2)$, we define $h_{\chi+1},\Delta_{\chi+1}$,
using $\varphi_\chi$.

If $\chi$ is a copying stage then $h_{\chi+1}=h_\chi+1$ and $\Delta_{\chi+1}=\Delta_\chi\un\{(h_{\chi+1},\chi+1)\}$).

Now suppose that $\chi$ is inflationary. Let $\xi=\pred^\Vv(\chi+1)$ (so $\xi=\chi$ if
$E^\Vv_\chi=\emptyset$) and 
$\eps=\pred^\Ww(\chi+1)$.
If there is any copying stage $\alpha$ such that $\alpha+1\leq_\Vv\xi$ and $\alpha+1\leq_\Ww\eps$,
then let $\alpha$ be the largest such (this exists by $\varphi_{\chi}$), and let 
$\beta=\alpha+1$. If there is no such $\alpha$, then let 
$\beta=0$. Let $h_{\chi+1}=h_\beta$.
Let $\delta^{\chi+1}_j=\delta^{\beta}_j$ for $j<h_\beta$, and 
$\delta^{\chi+1}_{h_\beta}=\chi+1$. This easily gives appropriate inflations.

This completes the definition of everything at stage $\chi+1$.

Now let $\eta<\lh(\Vv,\Ww)$ be a limit, and suppose we have defined $(\Vv,\Ww)\rest\eta$, and hence
$(\Vv,\Ww)\rest(\eta+1)$. If there is a copying stage $\alpha$ such that 
$\alpha+1\leq_\Vv\eta$ and $\alpha+1\leq_\Ww\eta$,
let $\alpha$ be largest such (which exists by $\varphi_\eta$) and let $\beta=\alpha+1$; otherwise 
let $\beta=0$.
Let $h_\eta=h_\beta$. Let $\delta^{\eta}_j=\delta^{\beta}_j$ and 
for $j<h_\eta$, and $\delta^{\eta}_{h_\beta}=\eta$.
Again this gives appropriate inflations.

This completes the construction at stage $\eta$, and hence, the full comparison,
given that $\varphi_\eta$ holds for each $\eta<\lh(\Vv,\Ww)$. We now verify this.
Write
\[ C_\eta=\{\alpha+1\mid \alpha\text{  
is a copying stage, }\alpha+1\leq_\Vv\eta\text{ and }\alpha+1\leq_\Ww\eta\},\]
\[ G_\eta=\{\gamma_{j+1}^\eta\mid j+1\leq_{\Uu^P} h_\eta\}=\{\gamma_{j+1}^\eta\mid j+1\leq_{\Uu^Q}h_\eta\}.\]
(Recall here that $\gamma_0^\eta=0$ always.)

\begin{clm}\label{clm:C_eta=G_eta}
 Let $\eta<\lh(\Vv,\Ww)$. Then:
 \begin{enumerate}[label=(\roman*)]
  \item\label{item:C_eta=G_eta} 
  $C_\eta=G_\eta$, and therefore
$\varphi_\eta$ holds, and
 
\item\label{item:inflation_coherence} for each $j\leq h_\eta$
 and each $\beta\in[\gamma^\eta_j,\delta^\eta_j]^{\Vv}\cap[\gamma^\eta_j,\delta^\eta_j]^{\Ww}$\footnote{It seems that maybe
 $[\gamma^\eta_j,\delta^\eta_j]^{\Vv}\neq [\gamma^\eta_j,\delta^\eta_j]^{\Ww}$. But certainly $\gamma^\eta_j,\delta^\eta_j\in[\gamma^\eta_j,\delta^\eta_j]^{\Vv}\cap[\gamma^\eta_j,\delta^\eta_j]^{\Ww}$.}, we have:
 \begin{enumerate}[label=--]
 \item  $h_\beta=j$,
 \item $\gamma^\beta_i=\gamma^\eta_i$ for each $i\leq j$, and so
 \item $\delta^\beta_i=\delta^\eta_i$
 for each $j<i$, and $\delta^\beta_j=\beta$.
 \end{enumerate}
 \end{enumerate}
\end{clm}
\begin{proof}
We prove the claim by induction on $\eta$.
(For part \ref{item:C_eta=G_eta},
since $G_\eta$ is finite, it suffices to see that $C_\eta=G_\eta$.)

 \begin{case} $\eta=0$.
  
  Part \ref{item:C_eta=G_eta} is immediate as $C_0=\emptyset=G_0$, and part \ref{item:inflation_coherence}  as $h_0=0$ and $\gamma^0_0=0=\delta^0_0$.
 \end{case}

\begin{case} $\eta=\chi+1$ where $\chi$ is a copying stage.

So $h_{\chi+1}=h_\chi+1$
and $\Delta_{\chi+1}=\Delta_\chi\cup\{(h_{\chi+1},\chi+1)\}$,
and so part \ref{item:inflation_coherence}
at stage $\chi+1$ follows immediately by induction, using that it holds at stage $\chi$.

Now let $\xi=\pred^\Vv(\chi+1)$ and $\eps=\pred^\Ww(\chi+1)$.
Let
\[ j=\pred^{\Uu^P}(h_\chi+1)=\pred^{\Uu^Q}(h_\chi+1).\]
 Because $\chi$ is a copying stage and
\[ \Delta_\chi:\Uu^\Yback\rest(h_\chi+1)\hookrightarrow\Vv\rest(\chi+1),\]
it is easy to see that $\xi\in[\gamma^\chi_j,\delta^\chi_j]_\Vv$. Similarly, 
$\eps\in[\gamma^\chi_j,\delta^\chi_j]_\Ww$.  We have
 $\gamma^{\chi+1}_k=\gamma^\chi_k$ for all $k\leq h_\chi$, and $\gamma^{\chi+1}_{h_{\chi+1}}=\delta^{\chi+1}_{h_{\chi+1}}=\chi+1$.
Let $\beta=\delta^\chi_j$. 
Then by part \ref{item:inflation_coherence} applied at $(\chi,\beta)$
(that is, with $\eta$ replaced by $\chi$,
and $\beta=\delta^\chi_j$) together with part \ref{item:C_eta=G_eta}
at $\beta$ (that is, with $\eta$ replaced by $\beta$), it is easy
to see that
\[ G_{\chi+1}=G_\beta\un\{\chi+1\}=C_\beta\cup\{\chi+1\}=C_{\chi+1},\]
giving part \ref{item:C_eta=G_eta}
at $\chi+1$.
\end{case}

\begin{case}\label{case:eta=chi+1_and_chi_inflationary} $\eta=\chi+1$ and $\chi$ is inflationary.

Let $\xi=\pred^\Vv(\chi+1)$ and $\eps=\pred^\Ww(\chi+1)$.

\begin{scase}\label{scase:alpha_exists}
There is a copying stage $\alpha$ 
such that $\alpha+1\leq_\Vv\xi$ and $\alpha+1\leq_\Ww\eps$.

By $\varphi_\chi$ there is a largest; fix this $\alpha$.
Then $h_{\chi+1}=h_{\alpha+1}$ and
\[ \Delta_{\chi+1}=(\Delta_{\alpha+1}\rest h_{\alpha+1})\un\{(h_{\alpha+1},\chi+1)\}. \]
It follows that $G_{\chi+1}=G_{\alpha+1}$.
But also easily  $C_{\chi+1}=C_{\alpha+1}$.
So part \ref{item:C_eta=G_eta} at $\chi+1$ follows by induction.
Now consider part \ref{item:inflation_coherence} at $\chi+1$.
By induction, it suffices to consider
$j=h_{\chi+1}$. With this $j$,
let $\beta\in[\gamma^{\chi+1}_{j},\delta^{\chi+1}_{j}]^{\Vv}\cap[\gamma^{\chi+1}_{j},\delta^{\chi+1}_{j}]^{\Ww}$. We must verify that
\begin{enumerate}[label=(\alph*)]
 \item\label{item:a_clause_1}
$h_\beta=j=h_{\chi+1}$, and
\item\label{item:a_clause_2}$\gamma^\beta_i=\gamma^{\chi+1}_i=\gamma^{\alpha+1}_i$ for each $i\leq j$.
\end{enumerate}
Since $\delta^{\chi+1}_{j}=\chi+1$,
we may clearly assume $\beta<\delta^{\chi+1}_j$.
We also have $\gamma^{\chi+1}_j=\alpha+1$,
since $C_{\alpha+1}=G_{\alpha+1}$.
But if $\beta=\alpha+1$ then clauses \ref{item:a_clause_1} and \ref{item:a_clause_2} hold by induction. So suppose $\gamma^{\chi+1}_j<\beta<\delta^{\chi+1}_j$.
Then $\beta\leq\min(\xi,\varepsilon)$,
and by choice of $\alpha$,
if $\beta=\alpha'+1$ then $\alpha'$ is inflationary, but then note that since $\alpha'+1\in(\alpha+1,\xi)^{\Vv}\cap(\alpha+1,\varepsilon)^{\Ww}$, 
clauses \ref{item:a_clause_1} and \ref{item:a_clause_2} hold just by the definition of $\Delta_{\alpha'+1}$.
And if $\beta$ is a limit ordinal,
then they similarly hold.
\end{scase}

\begin{scase}
There is no $\alpha$ as in Subcase \ref{scase:alpha_exists}.

Then $h_{\chi+1}=0$ and  $G_{\chi+1}=\emptyset=C_{\chi+1}$, and otherwise things are similar but simpler than the previous subcase.
\end{scase}
\end{case}
\begin{case} $\eta$ is a limit.

We claim that $C_\eta$ is finite. For suppose not. Let $\left<\alpha_n\right>_{n<\om}$ be a 
strictly increasing sequence such that $\alpha_n+1\in C_\eta$ for each $n$.
Note that if $m<n$ then $\alpha_m+1\in C_{\alpha_n+1}$.
Let $n>\lh(\Uu^P)$. Then $C_{\alpha_n+1}$ has cardinality $>\lh(\Uu^P)$.
By induction, $C_{\alpha_n+1}=G_{\alpha_n+1}$. But $G_{\alpha_n+1}$ has cardinality $<\lh(\Uu^P)$, 
contradiction.

Now let $\beta=\alpha+1=\max(C_\eta)$, if $C_\eta\neq\emptyset$, and $\beta=0$ otherwise.
Then by  induction,
$C_\eta=C_\beta=G_\beta=G_\eta$,
and the rest is much like in Case \ref{case:eta=chi+1_and_chi_inflationary}.\qedhere
\end{case}
\end{proof}

\begin{clm}$\Vv,\Ww$ are small.\end{clm}\label{clm:Vv,Ww_small}
\begin{proof}This follows easily from Claims \ref{clm:Uu_small} and \ref{clm:t_alpha_chi_not_3}.
\end{proof}

\begin{clm}\label{clm:when_min_disag_t=2}
 Let $R=P_\chi$ and $S=\Zback_\chi$.
 Suppose $\alpha_\chi=\alpha_{\chi x}<\infty$ and $t_{\alpha_\chi x}^R=2$ and $\kappa=\card^R(\alpha_\chi)$ and $F=F^R_{\kappa 0}=F^R_{\alpha_\chi x}$. Then:
 \begin{enumerate}
  \item\label{item:alpha_chi_after_Q} $\alpha_\chi=\xi^R_{\kappa x}$ (cf.~Definition \ref{dfn:xi^Y_kappa}),
  \item\label{item:t_y_also_2} $\alpha_{\chi y}=\alpha_\chi=\xi^R_{\kappa y}$, so $F^R_{\alpha_\chi y}=F$,
  \item\label{item:t_alpha_chi^S_0_or_2} $t_{\alpha_\chi x}^{S}=t_{\alpha_\chi y}^S\in\{0,2\}$, 
  \item\label{item:same_crit}  $t^S_{\alpha_\chi x}=t^S_{\alpha\chi y}=2$ iff $\kappa$ is measurable in $S$,
  and in this case,  $\alpha_\chi=\xi^S_{\kappa y}$ and $F^S_{\alpha_\chi x}=F^S_{\alpha_\chi y}=F_{\kappa 0}^S$.
 \end{enumerate}
\end{clm}
\begin{proof}
 Part \ref{item:alpha_chi_after_Q}: Let $\gamma=\gamma^R_\kappa$ (cf.~Definition \ref{dfn:gamma^Y_kappa}).
 By Lemma \ref{lem:least_t=2},
 we have $\gamma+\om=\xi^R_{\kappa x}\leq\alpha_\chi$, so
 $F_{\alpha_\chi x}^{R}=F=F_{\gamma+\om,x}^{R}$. By Claim \ref{clm:alpha_chi<infty_implies_alpha_chi_leq_Omegas},
$\alpha_\chi\leq\min(\Omega^{S}_x,\Omega^S_y)$.
 So also by
 Lemmas \ref{lem:least_t=2}
 and \ref{lem:CCbar_match_from_comp},
 and since $R,S$ are $({<\alpha_\chi},x)$-compatible,
 either:
 \begin{enumerate}[label=--]
\item $\gamma=\kappa$ and $t_{\gamma x}^{R}=0=t_{\gamma x}^S$, or
 \item $\gamma>\kappa$ and $t_{\gamma x}^{R}=3=t_{\gamma x}^S$.
 \end{enumerate}

 By Claim \ref{clm:compatibility_implies_Q-structures_match} and Lemma \ref{lem:meas_implies_least_t=2_exists},
 no $\delta'<\kappa$ is Woodin in $S$,
 and either
 \begin{enumerate}[label=(\roman*)]
 \item\label{item:kappa_Woodin_in_S} $\kappa$ is the least Woodin  of $S$
 (hence non-measurable in $S$)
 and $\gamma=\OR^S$,  $t^S_\gamma=3$
 and $N^S_x=N^S_{\gamma x}=N^R_{\gamma x}$,
 and $R,S$ are both $M_x$-good and $M_y$-good
 with $\Omega_x^R=\Omega_y^R=\gamma=\Omega_x^S=\Omega_y^S$, or
 \item\label{item:option_kappa_meas} $\kappa$ is measurable non-Woodin in $S$, $\gamma=\gamma^S_\kappa$,
 $\gamma+\om=\xi^S_{\kappa x}$ and $t^S_{\gamma+\om,x}=2$, or
 \item\label{item:option_kappa_non-meas} $\kappa$ is non-measurable  non-Woodin in $S$, $\gamma=\gamma^S_\kappa$,
 and $t^S_{\gamma+\om,x}=0$.
 \end{enumerate}
 
 But option \ref{item:kappa_Woodin_in_S} does not hold, because $\kappa$ is measurable in $R$, so non-Woodin in $R$,
 and $R|\gamma\pins R'=\Ult(R,F)$
 (recall $F=F_{\kappa0}^R$),
 so $N_{\gamma x}^{R'}=N_{\gamma x}^R$
 is given by P-construction
 over $N_{\kappa x}^{R'}=N_{\kappa x}^R$,
 and letting $j:R\to R'$ be the ultrapower map,
  $N_{\kappa x}^{R}\pins_{\mathrm{card}} N_{j(\kappa)x}^{R'}$, so $\rho_\om^{N^{R'}_{\gamma x}}=\rho_\om^{N^R_{\gamma x}}=\kappa$,
  contradicting the fact that
  $\rho_\om(M_x)=\om$
  and $M_x=\core_\om(N_{\gamma x}^R)$.
 
 Now suppose $\alpha_\chi>\gamma+\om$. Then $R,S$ are $(\gamma+\om,x)$-compatible.
 But then $G\eqdef F^S_{\gamma+\om,x}$ agrees with $F$ over $(R\inter S)\cross[\theta]^{<\om}$
 where $\theta=\min(\lh(F),\lh(G))$.
 So $\kappa=\crit(G)$ is measurable in $S$,
 so option \ref{item:option_kappa_meas} holds.
 So $t_{\gamma+\om,x}^S=2$ and $G=F^S_{\kappa 0}$.
 Now $\alpha_\chi<\kappa^{+R}$.
 We claim that $\alpha_\chi<\kappa^{+S}$. For otherwise
 $\kappa^{+S}\leq\alpha_\chi<\kappa^{+R}$, so $\kappa^{+S}$
 is not a cardinal in $i^R_F(N^R_{\kappa x})$,
 but then by the compatibility
 of $F,G$ (and as $N^R_{\kappa x}=N^S_{\kappa x}$ and $\kappa^{+S}<\theta$),
 it follows that $\kappa^{+S}$ is not a cardinal in $i^S_G(N^S_{\kappa x})$,
 a contradiction. So  similarly,
 $t_{\alpha_\chi x}^S=2$ and $F_{\alpha_\chi x}^S=F_{\kappa 0}^S$.
 But the compatibility of $F,G$ then contradicts the choice of $\alpha_{\chi x}=\alpha_\chi$.

 Part \ref{item:t_y_also_2}: By Lemma \ref{lem:meas_implies_least_t=2_exists},
 $N_{\gamma y}^R$ is the Q-structure
 for $N_{\kappa y}^R$
 and $\xi^{R}_{\kappa y}=\gamma+\om$, and $N_{\gamma+\om,y}^R=\J(N_{\gamma y}^R)\pins i^R_{F}(N_{\kappa y}^R)$,  $t_{\gamma+\om,y}^R=2$
 and $F_{\gamma+\om,y}^R=F$. But then if $\gamma+\om<\alpha_{\chi y}$ then we reach a contradiction to the fact that
 $\alpha_{\chi x}=\alpha_\chi$ as in the previous part.
 
 Parts \ref{item:t_alpha_chi^S_0_or_2},\ref{item:same_crit}: These follow because option \ref{item:option_kappa_meas}
 or option \ref{item:option_kappa_non-meas} above hold
(with respect to $x$),
 and by symmetry, also with $y$ replacing $x$.
\end{proof}

As an immediate corollary we have:
\begin{clm}\label{clm:when_min_disag_t=1}
Let $R=P_\chi$ and $S=\Zback_\chi$.
If $\alpha_\chi=\alpha_{\chi x}<\infty$ and $t_{\alpha_\chi x}^R=1$
 then $t_{\alpha_\chi x}^S\in\{0,1\}$, and if also $\alpha_\chi=\alpha_{\chi y}$
 then $t^R_{\alpha_\chi y},t^S_{\alpha_\chi y}\in\{0,1\}$.
\end{clm}

\begin{clm}\label{clm:beta_description}
Let $R=P_\chi$ and $S=\Zback_\chi$.
There are $\beta,\xi$ such that:
\begin{enumerate}[label=--]
\item if $E^\Vv_\chi\neq\emptyset$ then $\nu(E^\Vv_\chi)=\aleph_\beta^R$,
\item if $E^\Ww_\chi\neq\emptyset$ then $\nu(E^\Ww_\chi)=\aleph_\beta^S$,

\item if $\beta=\kappa+1$ then
 \begin{enumerate}[label=--]
  \item regarding $\Vv,R$, we have:
  \begin{enumerate}[label=--]
  \item $E^\Vv_\chi\neq\emptyset$ iff $\kappa$ is measurable in $R$,
  \item if $E^\Vv_\chi\neq\emptyset$
 then $\xi=\xi^R_{\kappa x}=\xi^R_{\kappa y}$ (so $t_{\xi x}^R=2=t^R_{\xi y}$) and $E^\Vv_\chi=F_{\xi x}^R=F_{\xi y}^R=F_{\kappa 0}^R$,
\end{enumerate}
 \item likewise for $\Ww,S$,
\end{enumerate}
and
\item if $\beta$ is a limit then $\xi=\beta+\om$ and
\begin{enumerate}[label=--]
 \item  regarding $\Vv,R$, we have:
\begin{enumerate}[label=--]
 \item $R|\aleph_\beta^R\not\sats\ZFC$,
 \item if $E^\Vv_\chi\neq\emptyset$ then either $t_{\xi x}^R=1$ and $E^\Vv_\chi=F_{\xi x}^R$,
 or $t_{\xi y}^R=1$ and $E^\Vv_\chi=F_{\xi y}^R$,\end{enumerate}
 \item likewise for $\Ww,S$.
\end{enumerate}
\end{enumerate}
\end{clm}
\begin{proof}
 If $\chi$ is a copying stage and $h_\chi\in I_1$ then $\nu(E^\Vv_\chi)=\aleph_\beta^R$
 and $\nu(E^\Ww_\chi)=\aleph_\beta^S$ where
 $\xi=\beta+\om=\sigma^{\Vv\chi}_{h_\chi}(\xi^{*\Yback}_{\mathscr{I}(h_\chi)})=\sigma^{\Ww\chi}_{h_\chi}(\xi^{*\Zback}_{\mathscr{I}(h_\chi)})$.
 Moreover, $t^R_{\xi x}=t^R_{\xi y}=1$ and $E^\Vv_\chi=F^R_{\xi x}=F^R_{\xi y}$, and likewise for $\Ww,S$.
 
If $\chi$ is a copying stage and  $h_\chi\in I_2$ then $\nu(E^\Vv_\chi)=\aleph_{\kappa+1}^R$
 and $\nu(E^\Ww_\chi)=\aleph_{\kappa+1}^S$ where $\kappa=\crit(E^\Vv_\chi)=\crit(E^\Ww_\chi)$
 (the compatibility of these extenders implies they have the same critical point),
 so $\kappa$ is measurable in both $R,S$,
 and by Lemma \ref{lem:lifting_t=2_proj_meas} and Claim \ref{clm:compatibility_implies_Q-structures_match}, (letting)
 $\xi=\xi^R_{\kappa x}=\xi^R_{\kappa y}=\xi^S_{\kappa x}=\xi^S_{\kappa y}$,
 then $E^\Vv_\chi=F^R_{\xi x}=F^R_{\kappa 0}=F^R_{\xi y}$ and $E^\Ww_\chi=F^S_{\xi x}=F^S_{\kappa 0}=F^S_{\xi y}$.

Suppose $\chi$  is  inflationary and $\alpha_{\chi x}=\alpha_\chi$
and $t_{\alpha_\chi x}^R=2$. So $\kappa\eqdef\card^R(\alpha_\chi)=\crit(E^\Vv_\chi)$.
By Claim \ref{clm:when_min_disag_t=2}, $\alpha_{\chi y}=\alpha_\chi$
  and $t_{\alpha_\chi y}^R=2$ and $t_{\alpha_\chi x}^S=t_{\alpha_\chi y}^S\in\{0,2\}$ and $\alpha_\chi=\xi^R_{\kappa x}=\xi^R_{\kappa y}$,
  and \[t_{\alpha_\chi x}^S=t^S_{\alpha_\chi y}=2\text{ iff }\kappa\text{ is measurable in  }S,\]
  and if $\kappa$ is measurable in $S$
  then $\alpha_\chi=\xi^S_{\kappa x}=\xi^S_{\kappa y}$. This suffices.
  
Finally if $\chi$ is inflationary and $\alpha_{\chi x}=\alpha_\chi$
and $t_{\alpha_\chi x}^R=1$, then by Claim \ref{clm:when_min_disag_t=1},
$\nu(E^\Vv_\chi)=\aleph_{\beta}^R$ where $\beta+\om=\alpha_\chi$, and if $E^\Ww_\chi\neq\emptyset$ then $\nu(E^\Ww_\chi)=\aleph_{\beta}^S$.
\end{proof}

Write $\beta_\chi$ for the unique $\beta$ as in the clam.

\begin{clm}\label{clm:Vv,Ww_normal}$\Vv,\Ww$ are normal.
In fact, $\beta_\chi<\beta_\eps$ for  $\chi+1<\eps+1<\lh(\Vv,\Ww)$.\end{clm}
\begin{proof}
Suppose not and let $(\eps,\chi)$ be the lexicographically least counterexample.
So $\beta_\chi\geq\beta_\eps$.
Note then that
\begin{equation}\label{eqn:CC_agreement_eps_chi} \CC^{P_\eps}_x\rest\beta_\chi=\CC^{P_\chi}_x\rest\beta_\chi
\text{ and }
 \CC^{\Zback_\chi}_x\rest\beta_\chi=\CC^{\Zback_\eps}_x\rest\beta_\chi,
 \end{equation}
 so
 \begin{equation}\label{eqn:M_chi_M_eps_compatibility} P_\eps,\Zback_\eps\text{ are }({\beta_\chi},x)\text{-compatible}
 \end{equation}
 (take $(\kappa+1,x)$-compatible to mean just  $(\kappa,x)$-compatible) and likewise for $y$.

Now we may assume that $E^\Vv_\eps\neq\emptyset$ and $E^\Vv_\eps=F_{\alpha_\eps x}^{P_\eps}$.

If $\beta_\chi=\kappa+1$ then $\beta_\eps<\kappa$.
For if $\beta_\eps=\kappa+1$ then by Claim \ref{clm:beta_description}, $\kappa$ is measurable in $P_\eps$ and hence in $P_\chi$,
so $E^\Vv_\chi=F_{\kappa 0}^{P_\chi}$, contradiction. And
since $\kappa$ is measurable in either $P_\chi$ or $Q_\chi$,
either $P_\chi|\aleph_{\kappa}^{P_\chi}\sats\ZFC$
or $Q_\chi|\aleph_\kappa^{Q_\chi}\sats\ZFC$,
so by Claim \ref{clm:beta_description}, $\beta_\eps\neq\kappa$.

\begin{case} $\eps$ is inflationary.

If $\beta_\chi=\kappa+1$ then as shown above, $\beta_\eps<\kappa$,
but by line (\ref{eqn:M_chi_M_eps_compatibility}), $\alpha_\eps>\kappa+1$.
So by Claim \ref{clm:beta_description}, $\beta_\eps=\mu+1$ (for if $\beta_\eps$ is a limit
then $\beta_\eps+\om=\alpha_\eps$, so $\beta_\eps\geq\beta_\chi$, contradiction).
But then by Claim \ref{clm:when_min_disag_t=2}, $\alpha_\eps=\gamma+\om$ where $N_{\gamma x}^{P_\eps}=N_{\gamma x}^{\Zback_\eps}$
is the Q-structure for $N_{\mu x}^{P_\eps}=N_{\mu x}^{\Zback_\eps}$,
so $\gamma>\kappa$, but then $\mu$ is Woodin in $N_{\kappa x}^{P_\eps}=N_{\kappa x}^{\Zback_\eps}$,
contradicting smallness and that $\kappa$ is measurable in either $P_\chi$ or $\Zback_\chi$.

So $\beta_\chi$ is a limit. By line (\ref{eqn:M_chi_M_eps_compatibility}), $\alpha_\eps>\beta_\chi$.

Suppose $\beta_\eps<\beta_\chi$. Then clearly $\beta_\eps=\mu+1$. By Claim \ref{clm:beta_description},
\[ 1\in\{t_{\beta_\chi+\om,x}^{P_\chi},t_{\beta_\chi+\om,x}^{\Zback_\chi},t_{\beta_\chi+\om,y}^{P_\chi},t_{\beta_\chi+\om,y}^{\Zback_\chi}\},\]
\[  2\in\{t_{\alpha_\eps x}^{P_\chi},t_{\alpha_\eps x}^{\Zback_\chi},t_{\alpha_\eps y}^{P_\chi},t_{\alpha_\eps y}^{\Zback_\chi}\},\]the latter
with $\mu$ being the corresponding measurable, and by our earlier assumption, in fact $t_{\alpha_\eps x}^{P_\chi}=2$.
Then by Claim \ref{clm:when_min_disag_t=2},
$\alpha_\eps=\gamma+\om$ where $\gamma=\gamma^{P_\eps}_\mu$,
so $\gamma=\gamma^{\Zback_\eps}_\mu$. Note then that $P_\chi|\gamma=P_\eps|\gamma$ and $Q_\chi|\gamma=Q_\eps|\gamma$.
But also
\[ \mu<\mu+1=\beta_\eps<\beta_\chi\leq\gamma<\gamma+\om=\alpha_\eps \]
since $\beta_\chi<\alpha_\eps$, contradicting smallness at stage $\chi$.

So $\beta_\eps=\beta_\chi$. Let $\beta=\beta_\chi$. So $\beta+\om=\alpha_\eps$.
Let $G=E^\Vv_\eps$.

\begin{scase} $E=E^\Vv_\chi\neq\emptyset$.
 
So
\[ \nu_E=\nu_G=\aleph_\beta^{P_\chi}=\aleph_\beta^{P_\eps}<\lh(G)<\lh(E)=\aleph_{\beta+1}^{P_\eps}. \]
We have $t_{\beta+\om,x}^{P_\eps}=1$ and $G=F^{P_\eps}_{\beta+\om,x}$. Note that
since $N_{\beta x}^{P_\chi}=N_{\beta x}^{P_\eps}$ and by coherence,
$t_{\beta+\om,x}^{P_\chi}=1$ and $G=F^{P_\chi}_{\beta+\om,x}$.
Note then that $\chi$ is inflationary, because if it were copying,
then we would have $F_{\beta+\om,x}^{P_\chi}=E^\Vv_\chi=F_{\beta+\om,y}^{P_\chi}$.
Because we set $E^\Vv_\chi=E$ but $\lh(E)>\lh(G)$, we must have that $t_{\beta+\om,x}^{\Zback_\chi}=1$
and letting $H=F^{\Zback_\chi}_{\beta+\om,x}$, that $G,H$ are compatible
(that is, $P_\chi,\Zback_\chi$ are $(\beta+\om,x)$-compatible).
Note that $t_{\beta+\om,y}^{P_\eps}=t_{\beta+\om,y}^{\Zback_\eps}=0$, because we removed the least possible backgrounding extender at stage $\chi$.
It follows that $\alpha_{\eps y}>\beta+\om=\alpha_{\eps x}$, so $H\notin\es(\Zback_\eps)$, so there is some $\gamma$ such that $\chi\leq\gamma<\eps$
and $E^\Ww_\gamma\neq\emptyset$ and $\lh(E^\Ww_\gamma)\leq\lh(H)$.
But then by the choice of $\chi,\eps$, we have $\gamma=\chi$.

\begin{sscase}\label{sscase:E^Ww_chi=H}
$E^\Ww_\chi=H$.

We have $\alpha_{\chi x}>\alpha_{\chi y}=\beta+\om$, and so $H=F_{\beta+\om,y}^{\Zback_\chi}$.
But
\[ G\rest X\cross[\theta]^{<\om}=H\rest X\cross[\theta]^{<\om} \]
where $X=P_\chi\inter \Zback_\chi$ and $\theta=\min(\lh(G),\lh(H))$,
and therefore $G$ induces the same extender over $N_{\beta y}^{P_\chi}=N_{\beta y}^{\Zback_\chi}$
as does $H$. But $\nu_G=\aleph_\beta^{P_\chi}$. Therefore the minimality of
the choice of $F^{P_\chi}_{\beta y}$ ensures that $\lh(F^{P_\chi}_{\beta y})\leq\lh(G)$,
whereas $E=F^{P_\chi}_{\beta y}$, contradiction.
\end{sscase}
\begin{sscase} $H\neq E^\Ww_\chi$,
 so $\lh(H)>\lh(E^\Ww_\chi)$.

 Let $\kappa=\crit(H)=\crit(G)$ and let
\[ A=\{\alpha<\kappa\mid t_{\alpha+\om,y}^{\Zback_\chi}=1\}=
 \{\alpha<\kappa\mid t_{\alpha+\om,y}^{P_\chi}=1\}. \]
 Since $\lh(F_{\chi y}^{\Zback_\chi})<\lh(H)$, we have $\beta\in i^{\Zback_\chi}_H(A)$.
 Since $\beta\leq\min(\aleph_\beta^{P_\chi},\aleph_\beta^{\Zback_\chi})<\theta$ and $A\in P_\chi\inter \Zback_\chi$,
compatibility gives that $\beta\in i^{P_\chi}_G(A)$. But then by coherence and since $\nu_G=\aleph_\beta^{P_\chi}$, there must be some suitable backgrounding extender
$E'\in\es(P_\chi)$ for $N_{\beta+\om,y}^{P_\chi}$ with $\lh(E')<\lh(G)$, again contradicting the minimality of $\lh(E)$.
\end{sscase}
\end{scase}

\begin{scase} $E^\Vv_\chi=\emptyset$.
 
So $E^\Ww_\chi\neq\emptyset$, so again, $\chi$ is inflationary.
We have $\aleph_{\beta+1}^{P_\chi}<\lh(E^\Vv_\gamma)$ for all $\gamma\in(\chi,\eps)$.
So like before, we get that $t_{\beta+\om,x}^{P_\chi}=1$ with $F^{P_\chi}_{\beta+\om,x}=G$,
and $t_{\beta+\om,x}^{\Zback_\chi}=1$ and with $H=F^{\Zback_\chi}_{\beta+\om,x}$,
we have $\lh(E^\Ww_\chi)\leq\lh(H)$ and $G,H$ are compatible. We also have $\alpha_{\chi y}=\beta+\om$
and $t_{\beta+\om,y}^{\Zback_\chi}=1$ and $E^\Ww_\chi=F^{\Zback_\chi}_{\beta+\om,y}$,
and since $E^\Vv_\chi=\emptyset$, therefore $t_{\beta+\om,y}^{P_\chi}=0$.
If $H=E^\Ww_\chi$ then the compatibility gives that $t_{\beta+\om,y}^{P_\chi}=1$ (a contradiction),
as $G,H$ background the same extender over $N_{\beta y}^{P_\chi}=N_{\beta y}^{\Zback_\chi}$. 
So $\lh(H)>\lh(E^\Ww_\chi)$, but then we reach a contradiction by using an argument much like before.
\end{scase}
\end{case}

\begin{case} $\eps$ is copying.
 
 Let $E^*=E^\Vv_\eps$ and $F^*=E^\Ww_\eps$, so $E^*\neq\emptyset\neq F^*$. We may assume that $E^\Vv_\chi\neq\emptyset$.

 Recall that $\eps=\delta^\eps_{h_\eps}$ and
$\sigma^{\Vv\eps}_{h_\eps}=i^\Vv_{\gamma^\eps_{h_\eps}\eps}\com\pi^{\Vv\eps}_{h_\eps}$ and
$E^*=\sigma^{\Vv\eps}_{h_\eps}(E^{\Uu^\Yback}_{h_\eps})$, so
$\beta_\eps\in\rg(i^\Vv_{\gamma^\eps_{h_\eps}\eps})$.
We have $\gamma^\eps_{h_\eps}\leq\chi$, because if
\[ \chi<\delta^\eps_i+1=\gamma^\eps_{i+1}\leq\eps \]
then since $\delta^\eps_i<\eps$ and by the minimality of $\eps$,
we have $\beta_\chi\leq\beta_{\delta^\eps_i}$, and because $\Delta_\eps$ is an inflation,
we have $\lh(E^\Vv_{\delta^\eps_i})<\lh(E^*)$, but then
\[ 
\aleph_{\beta_\chi}^{P_\eps}\leq
\aleph_{\beta_{\delta^\eps_i}}^{P_\eps}=\nu(E^\Vv_{\delta^\eps_i})<
\aleph_{\beta_{\delta^\eps_i+1}}^{P_\eps}=\lh(E^\Vv_{\delta^\eps_i})\leq
\nu(E^*)=\aleph_{\beta_\eps}^{P_\eps}, \]
so $\beta_\chi<\beta_\eps$, contradiction.

So $\gamma^\eps_{h_\eps}\leq\chi<\eps=\delta^\eps_{h_\eps}$,
and $\gamma^\eps_{h_\eps}<_\Vv\eps$. 
Let $\xi$ be least such that
$\xi+1\leq_\Vv\eps$ and $\chi\leq\xi$. Then $E^\Vv_\xi\neq\emptyset$
because $E^\Vv_\chi\neq\emptyset$. Let $\iota=\pred^\Vv(\xi+1)$. So $\gamma^\eps_{h_\eps}\leq^\Vv\iota$. Let
$\kappa=\crit(E^\Vv_\xi)=\crit(i^\Vv_{\iota,\xi+1})$. Then $\kappa<\aleph_{\beta_\chi}^{P_\eps}$ (so $\kappa<\beta_\chi$) and
\[ 
\beta_\eps\leq\aleph_{\beta_\eps}^{P_\eps}\leq\aleph_{\beta_\chi}^{P_\eps}<
i^\Vv_{\iota\eps}(\kappa), \]
(the last inequality is because $\Vv$ is small, so $E^\Vv_\xi$ is not superstrong).
But $\beta_\eps\in\rg(i^\Vv_{\gamma^\eps_{h_\eps}\eps})$, and so $\beta_\eps<\kappa$,
so $\aleph_{\beta_\eps}^{P_\eps}<\kappa$, so $\lh(E^*)<\kappa$.

Now by the lexicographic minimality of $(\eps,\chi)$, for all $\delta<\chi$ we have $\beta_\delta<\beta_\eps$,
so $\aleph_{\beta_\delta}^{P_\eps}<\kappa$. Therefore $\chi=\pred^\Vv(\xi+1)$.

Now suppose that $\Zback_\chi\neq\Zback_\eps$. Let $\zeta$ be least such that $\zeta+1\leq_\Ww\eps$ and $\chi\leq\zeta$.
So $E^\Ww_\zeta\neq\emptyset$. Let $\mu=\crit(E^\Ww_\zeta)$. As above, $\lh(F^*)<\mu$ and $\pred^\Ww(\zeta+1)=\chi$.
In particular,
\[ \chi\in[\gamma^\eps_{h_\eps},\eps]_\Vv\inter[\gamma^\eps_{h_\eps},\eps]_\Ww, \]
and note that by Claim \ref{clm:C_eta=G_eta}, we 
have $h_\chi=h_\eps$ and $\gamma^\chi_k=\gamma^\eps_k$ for all 
$k\leq h_\eps$.

But now since $\lh(E^*)<\kappa$ and $\lh(F^*)<\mu$, the fact that $\eps$ is a copying stage, using 
extenders $E^*,F^*$,
easily reflects to show that $\chi$ is a copying stage, using $E^*,F^*$, so
$E^*=E^\Vv_\chi$, contradiction.

If instead $Q_\chi=Q_\eps$ then $\chi<_\Ww\eps$ and $i^\Ww_{\chi\eps}=\id$,
so we reach a contradiction through a similar but slightly simpler argument.\qedhere
\end{case}
\end{proof}
\begin{clm}\label{clm:copying_compatible}
 Let $\gamma+1<\lh(\Vv,\Ww)$. Then $\gamma$ is a copying stage iff 
$E^\Vv_\gamma\neq\emptyset\neq E^\Ww_\gamma$ and
 \[ E^\Vv_\gamma\inter(X\cross[\theta]^{<\om})=E^\Ww_\gamma\inter(X\cross[\theta]^{<\om}) \]
 where $X=P_\gamma\inter\Zback_\gamma$ and $\theta=\min(\lh(E^\Vv_\gamma),\lh(E^\Ww_\gamma))$.
\end{clm}
\begin{proof}
If $\gamma$ is copying, the compatibility is just part of the definition.
Suppose $\gamma$ is inflationary but $E^\Vv_\gamma,E^\Ww_\gamma$ are compatible,
and in particular, both non-empty.
Then note that $\alpha_\gamma=\alpha_{\gamma x}=\alpha_{\gamma y}$ and either
\begin{enumerate}[label=--]\item $E^\Vv_\gamma=F^{P_\gamma}_{\alpha_{\gamma}x}\neq F^{P_\gamma}_{\alpha_{\gamma}y}$ and $E^\Ww_\gamma=F^{Q_\gamma}_{\alpha_{\gamma}y}\neq F^{Q_{\gamma}}_{\alpha_{\gamma}x}$, or 
 \item $E^\Vv_\gamma=F^{P_\gamma}_{\alpha_{\gamma}y}\neq F^{P_\gamma}_{\alpha_\gamma x}$ and $E^\Ww_\gamma=F^{Q_\gamma}_{\alpha_{\gamma}x}\neq F^{Q_{\gamma}}_{\alpha_\gamma y}$.
\end{enumerate}
So by Claims \ref{clm:when_min_disag_t=2} and \ref{clm:beta_description}, $\beta_\gamma$ is a limit.
But
then an argument like that used in Subsubcase \ref{sscase:E^Ww_chi=H} in the proof of Claim \ref{clm:Vv,Ww_normal} gives a contradiction.
\end{proof}

\begin{clm}
 The comparison terminates in countably many stages.
\end{clm}
\begin{proof}
Suppose not. Then we get $\lh(\Vv,\Ww)=\om_1+1$. Let $\pi:H\to \mathcal{H}_{\omega_2}$ be elementary with $H$ countable and transitive, and everything relevant in $\rg(\pi)$.
Let $\kappa=\crit(\pi)$. Let $\alpha+1=\min((\kappa,\om_1]^\Vv)$ and 
$\beta+1=\min((\kappa,\om_1]^\Ww)$. By the normality of $\Vv,\Ww$ we have
\[ E^\Vv_\alpha\inter(X\cross[\nu]^{<\om})=E^\Ww_\beta\inter(X\cross[\nu]^{<\om}) \]
where $X=P_\kappa\inter \Zback_\kappa$ and $\nu=\min(\nu(E^\Vv_\alpha),\nu(E^\Ww_\beta))$. Moreover, $\lh(E^\Vv_\alpha)<\crit(i^\Vv_{\alpha+1,\omega_1})$, since $\nu(E^\Vv_\alpha)$
is not measurable in $M^\Vv_{\alpha+1}$, considering the kinds of extenders used in $\Vv$.
Similarly, $\lh(E^\Ww_\beta)<\crit(i^\Ww_{\beta+1,\omega_1})$. So  in fact,
\[ E^\Vv_\alpha\inter(X\cross[\theta]^{<\om})=E^\Ww_\beta\inter(X\cross[\theta]^{<\om}) \]
where $\theta=\min(\lh(E^\Vv_\alpha),\lh(E^\Ww_\beta))$.

Now suppose $\alpha=\beta$.
By Claim \ref{clm:copying_compatible}, $\alpha$ is
copying. Note that $\alpha+1\in C_{\pi(\kappa)}$.
But by Claim \ref{clm:C_eta=G_eta}, $C_{\pi(\kappa)}$ is finite, but then
$C_{\pi(\kappa)}\sub\kappa$, contradiction.

So $\alpha\neq\beta$; we may assume $\alpha<\beta$.
We have $\beta_\alpha<\beta_\beta$.
Now $\beta_\alpha$ is a limit, because otherwise 
by Claim \ref{clm:beta_description}, $\kappa$ would not be measurable in $Q_\beta$.
If $\alpha$ is an inflationary stage then we reach a 
contradiction much as in the proof of \ref{thm:comparison}. So $\alpha$ is copying.
We have $\aleph_{\beta_\alpha}^{\Zback_\alpha}=\aleph_{\beta_\alpha}^{\Zback_\beta}$. Let
$F=E^\Ww_\beta\rest\aleph_{\beta_\alpha}^{\Zback_\beta}$.
Then $F\in\es^{\Zback_\alpha}$. And because
\[ E^\Vv_\alpha\inter(X\cross[\nu]^{<\om})=F\inter(X\cross[\nu]^{<\om}) \]
where $X=P_\alpha\inter \Zback_\alpha$ and
$\nu=\min(\aleph_{\beta_\alpha}^{P_\alpha},\aleph_{\beta_\alpha}^{\Zback_\alpha})$,
we have that $F$ is a candidate for $F_{\beta_\alpha+\om,x}^{\Zback_\alpha}$.
\footnote{And likewise with $y$ replacing $x$,
but we only need to consider one of them.}
But $E^\Ww_\alpha=F_{\beta_\alpha+\om,x}^{\Zback_\alpha}$,
and $\lh(F)<\lh(E^\Ww_\alpha)$, contradicting the minimality of 
$F_{\beta_\alpha+\om,x}^{\Zback_\alpha}$.
\end{proof}

So we get a successful comparison $(\Vv,\Ww)$ of successor length $\eta+1$.
Write $Y=P_\eta$ and $Z=Q_\eta$.
Now recall the notation from the beginning of the proof, leading to the definition of $\Uu^P,\Uu^Q$. Recall $\bar{\CC}_x$, etc,
was introduced in \ref{dfn:CC-bar}.
Let $\bar{\bar{\CC}}_x^Y=\bar{\CC}^Y_x\rest(\Omega_x^{Y}+1)$, etc.

\begin{clm}\label{clm:0th_norm_agmt}
We have:
\begin{enumerate}[label=(\roman*)]
\item\label{item:CC-bar-bars_line_up} {[}$\bar{\bar{\CC}}^Y_x\ins\bar{\bar{\CC}}^Z_x$ or $\bar{\bar{\CC}}^Z_x\ins\bar{\bar{\CC}}^Y_x$]
and [$\bar{\bar{\CC}}^Y_y\ins\bar{\bar{\CC}}^Z_y$ or $\bar{\bar{\CC}}^Z_y\ins\bar{\bar{\CC}}^Y_y$].
 \item  $x\equiv_0^P y$ iff $x\equiv_0^Q y$, and hence, $\bar{m}\geq 0$ ($\bar{m}$ was introduced at the beginning of the proof).
 \item\label{item:P,Q_both_M_x,M_y-good}  $P,Q$ are both $M_x$-good and $M_y$-good, as are $P_\alpha,Q_\alpha$ for all $\alpha\leq\eta$.
  \item\label{item:Omegas_match_at_end}$\Omega\eqdef\Omega^{Y}_x=\Omega^{Y}_y=\Omega^{Z}_y=\Omega^{Z}_x$,
  and $\Omega^P_x=\Omega^P_y$ and $\Omega^Q_x=\Omega^Q_y$.
 \end{enumerate}
\end{clm}
\begin{proof}
Parts \ref{item:P,Q_both_M_x,M_y-good}
and \ref{item:Omegas_match_at_end}
are
by Claim \ref{clm:if_full_compatibility_then_basic_agreement}.

According to the definition, $x\leq^P_0y$ iff
\begin{enumerate}[label=--]
 \item if $P$ is $M_y$-good then $P$ is $M_x$-good, and
 \item if $P$ is $M_x$-good and $M_y$-good then $(\Omega^P_x,\projdeg(M_x))\leq_\lex(\Omega^P_y,\projdeg(M_y))$.
\end{enumerate}
Likewise for $\leq^Q_0$. So by parts \ref{item:P,Q_both_M_x,M_y-good}
and \ref{item:Omegas_match_at_end},
it just remains to see that $\projdeg(M_x)=\projdeg(M_y)$.
But for example if $\projdeg(M_x)<\projdeg(M_h)$, then $x<_0^Py$ and $x<_0^Qy$, contradicting the disagreement of $P,Q$ about these things. 
\end{proof}

Recall that $\wt{\Uu}^P$
and $\Uu^P$ are essentially equivalent,
but $\wt{\Uu}^P$ is padded and indexed with pairs $(i,j)$, and $\Uu^P$ is non-padded,
and indexed with integers $h$.
For $(i,j)\in\dom(\wt{\Uu}^P)$ let
$h_{ij}$ be the unique $h<\lh(\Uu^P)$
such that $M^{\Uu^P}_h=M^{\wt{\Uu}^P}_{ij}$.
For $(i,j)\in\dom(\wt{\Uu}^P)$,
say $(i,j)$ is \emph{relevant}
if there is no $j'<j$ with $M^{\wt{\Uu}^P}_{ij'}=M^{\wt{\Uu}^P}_{ij}$
(equivalently, if $j>0$ then
 $E^{\wt{\Uu}^P}_{i,j-1}\neq\emptyset$).
Let $(i_{\mathrm{end}},j_{\mathrm{end}})$
be the unique relevant $(i,j)$ such that
$M^{\wt{\Uu}^P}_{ij}=M^{\Uu^P}_{h_\eta}$.
(Note it would be equivalent to use $Q$ instead of $P$ in these definitions.)
For $i<i_{\mathrm{end}}$
let \[\vec{A}_i=\sigma^{\Vv\eta}_{h_{ik_i}}(\vec{a}_i^{*\Uu^P})=\sigma^{\Ww\eta}_{h_{ik_i}}(\vec{a}_i^{*\Uu^Q})\]
(the second equality holds
by the definition of \emph{copying stage},
considered at stage $\alpha$ where $\alpha+1=\gamma_{i+1,0}^\eta$).

In the following claim, for $i<\lh(\Tt_x)$,
\[\pi_i^{\Tt_xP}:M^{\Tt_x}_i\to \core_{d_i}(N^{M^{\widetilde{\Uu}^P}_{i0}}_{\xi_i^{\Tt_xP}x})\]
denotes the $i$th copy map produced in the iterability
proof (Lemma \ref{lem:iterability})
for lifting $\Tt_x$ to $\widetilde{\Uu}^P$
(where $d_i=\deg^{\Tt_x}_i$)
and
\[ \psi_{ij}^{\Tt_xP}:M_{ij}\to C_{ij}^{\Tt_xP} \]
denotes the $j$th resurrection map
produced by the resurrection process
corresponding to $E^{\Tt_x}_i$ employed in the same proof (so $C_{i0}^{\Tt_xP}=\core_{d_i}(N^{M^{\widetilde{\Uu}^P}_{i0}}_{\xi_i^{\Tt_xP}x})$ and $\psi_{i0}^{\Tt_xP}=\pi_i^{\Tt_xP}$, and if $j>0$
then $C_{ij}^{\Tt_xP}=\core_\om(N_{\alpha_{ij}^{\Tt_xP}x}^{M^{\widetilde{\Uu}^P}_{ij}})$). Also recall that $\copyredd$ was specified in Definition \ref{dfn:copy_redd}.

\begin{clm}\label{clm:key_images_match}
Let $(i,j)\leq^{\wt{\Uu}^P}(i_{\mathrm{end}},j_{\mathrm{end}})$
(equivalently, $(i,j)\leq^{\wt{\Uu}^Q}(i_{\mathrm{end}},j_{\mathrm{end}})$)
be relevant. Let $\gamma=\gamma^\eta_{h_{ij}}$ (so $\gamma^\eta_{h_{i0}}\leq^\Vv\gamma\leq^\Vv\eta$ and $\gamma^{\eta}_{h_{i0}}\leq^\Ww\gamma\leq^\Ww\eta$). Then:
\begin{enumerate}
 \item\label{item:A-vecs<crit(i)} $\vec{A}_{0},\ldots,\vec{A}_{i-1}\sub\min(\crit(i^\Vv_{\gamma\eta}),\crit(i^\Ww_{\gamma\eta}))$,
 
 \item\label{item:j=0_images_of_stages_and_redds_match} If $j=0$ then:
 \begin{enumerate}[label=(\roman*)]
 \item\label{item:j=0_images_of_stages_match} $i^\Vv_{\gamma\eta}(\pi^{\Vv\eta}_{ij}(\xi_i^{\Tt_xP}))=i^\Ww_{\gamma\eta}(\pi^{\Ww\eta}_{ij}(\xi_i^{\Tt_xQ}))$,
\item\label{item:j=0_images_of_redds_match} if $i+1<\lh(\Tt_x)$ and $\exit^{\Tt_x}_i$ is small then \[i^\Vv_{\gamma\eta}(\pi^{\Vv\eta}_{ij}(\copyredd(\pi_i^{\Tt_xP},\exit^{\Tt_x}_i)))=i^\Ww_{\gamma\eta}(\pi^{\Ww\eta}_{ij}(\copyredd(\pi_i^{\Tt_xQ},\exit^{\Tt_x}_i))),\]
 \end{enumerate}
\item\label{item:general_images_of_stages_and_redds_match}
If $i+1<\lh(\Tt_x)$ and $\exit^{\Tt_x}_i$ is small then\begin{enumerate}[label=(\roman*)]\item\label{item:general_images_of_stages_match} 
  $i^\Vv_{\gamma\eta}(\pi^{\Vv\eta}_{ij}(\alpha_{ij}^{\Tt_xP}))=i^\Ww_{\gamma\eta}(\pi^{\Ww\eta}_{ij}(\alpha_{ij}^{\Tt_xQ}))$ and
  \item\label{item:general_images_of_redds_match}  $i^\Vv_{\gamma\eta}(\pi^{\Vv\eta}_{ij}(\copyredd(\psi_{ij}^{\Tt_xP},\exit^{\Tt_x}_i)))=i^\Ww_{\gamma\eta}(\pi^{\Ww\eta}_{ij}(\copyredd(\psi_{ij}^{\Tt_xQ},\exit^{\Tt_x}_i)))$,
  \end{enumerate}
\item\label{item:j>0_C-structures_match_and_not_moved} If $j>0$ (and therefore $i+1<\lh(\Tt_x)$) and $\exit^{\Tt_x}_i$ is small then:
\begin{enumerate}[label=(\roman*)]\item 
$\pi^{\Vv\eta}_{ij}(C_{ij}^{\Tt_xP})=\pi^{\Ww\eta}_{ij}(C_{ij}^{\Tt_xQ})$; let $\kappa$ be the projectum of this (fully sound) structure,
\item  $\gamma=\alpha+1$ for some copying stage $\alpha$, 
and $E^\Vv_{\alpha},E^\Ww_{\alpha}$
are order 0 measures with $\crit(E^\Vv_\alpha)=\crit(E^\Ww_\alpha)=\kappa$,\item 
$\pi^{\Vv\eta}_{ij}(\copyredd(\psi_{ij}^{\Tt_xP},\exit^{\Tt_x}_i))=\pi^{\Ww\eta}_{ij}(\copyredd(\psi_{ij}^{\Tt_xQ},\exit^{\Tt_x}_i))$,
and
\item $\OR(\pi^{\Vv\eta}_{ij}(C_{ij}^{\Tt_xP}))=\OR(\pi^{\Ww\eta}_{ij}(C_{ij}^{\Tt_xQ}))<\min(\crit(i^\Vv_{\gamma\eta}),\crit(i^\Ww_{\gamma\eta}))$.
\end{enumerate}
\item\label{item:xi^*,a^*_images_match} If $i+1<\lh(\Tt_x)$
and $\exit^{\Tt_x}_i$ is small
and $j$ is the largest $j'\leq k_i$
such that $(i,j')$ is relevant,
then
\[ i^\Vv_{\gamma\eta}(\pi_{ij}^{\Vv\eta}(\xi^{*\Yback}_{ix},\avec_{ix}^{*\Yback}))=i^\Ww_{\gamma\eta}(\pi_{ij}^{\Ww\eta}(\xi^{*\Zback}_{ix},\avec^{*\Zback}_{ix}))\]
($\xi_{ix}^{*P},\vec{a}_{ix}^{*P}$ were defined in line (\ref{eqn:a^*,xi^*})).
\end{enumerate}
Likewise with ``$y$'' replacing ``$x$'' in parts \ref{item:j=0_images_of_stages_and_redds_match}, \ref{item:general_images_of_stages_and_redds_match},
 \ref{item:j>0_C-structures_match_and_not_moved},
 \ref{item:xi^*,a^*_images_match}.
\end{clm}
\begin{proof}
 Part \ref{item:A-vecs<crit(i)}:
 This is just because if $i>0$, then for each $j<i$, letting \[F_x=\pi^{\Vv\eta}_{h_{i0}}(F(N^{M^{\wt{\Uu}^P}_{jk_j}}_{\xi^{*P}_jx})))=\pi^{\Ww\eta}_{h_{i0}}(F(N^{M^{\wt{\Uu}^Q}_{jk_j}}_{\xi^{*Q}_jx}))) \]
 and
 \[F_y=\pi^{\Vv\eta}_{h_{i0}}(F(N^{M^{\wt{\Uu}^P}_{jk_j}}_{\xi^{*P}_jy})))=\pi^{\Ww\eta}_{h_{i0}}(F(N^{M^{\wt{\Uu}^Q}_{jk_j}}_{\xi^{*Q}_jy}))),\]
 we have
 \[ \vec{A}_j\sub(\nu(F_x)\cap\nu(F_y))\leq(\nu(F_x)\cup\nu(F_y))\leq(\nu(E^\Vv_\alpha)\cap\nu(E^\Ww_\alpha))\]
 where  $\alpha+1=\gamma_{h_{i0}}^\eta$,
 and $\gamma_{h_{i0}}^\eta\leq^\Vv\gamma_{h_{ij}}\leq^\Vv\eta$
 and $\gamma_{h_{i0}}^\eta\leq^\Ww\gamma_{h_{ij}}^\eta\leq^\Ww\eta$.
 
 Parts \ref{item:j=0_images_of_stages_and_redds_match}, \ref{item:general_images_of_stages_and_redds_match},
 \ref{item:j>0_C-structures_match_and_not_moved}:
 These are proved simultaneously by induction on $(i,j)\leq^{\wt{\Uu}^P}(i_{\mathrm{end}},j_{\mathrm{end}})$.
 
First consider part  \ref{item:j=0_images_of_stages_and_redds_match} with $i=0$, so $\gamma=0$.
Then part \ref{item:j=0_images_of_stages_and_redds_match}\ref{item:j=0_images_of_stages_match}
just says that $\Omega^{Y}_x=\Omega^{Z}_x$, which we established in Claim \ref{clm:0th_norm_agmt}. (Note here that if $\xi_0^{\Tt_xP}=\Omega^P_x\geq\rho_0^P$ then $P$ is non-small and $\Omega_x^P=\OR^P$, in and $i^\Vv_{0\eta}(\pi^{\Vv\eta}_0(\OR^P))$ just means $\OR(M^\Vv_\eta)$.)
Part \ref{item:j=0_images_of_stages_and_redds_match}\ref{item:j=0_images_of_redds_match}:
because $1<\lh(\Tt_x)$ and $\exit^{\Tt_x}_0$ is small, we have $c^P=\copyredd(\pi_0^{\Tt_xP},\exit^{\Tt_x}_0)\in P|\rho_0^P$, so $i^\Vv_{0\eta}(\pi^{\Vv\eta}_0(c^P))$ makes sense, and likewise for $i^\Ww_{0\eta}(\pi^{\Ww\eta}_0(c^Q))$. Now $x,y$ are both $(\sigma,0)$-well (because of the assumed disagreement between $P,Q$).

Note that $\pi_0^{\Tt_xP}$ is non-$\nu$-low, since it is just $\pi_x^P$, which is either the identity or a core map.
Likewise for $\pi_0^{\Tt_xQ}$.

\begin{case} $\exit^{\Tt_x}_0=M_x$.

Then (since $\pi_0^{\Tt_xP},\pi_0^{\Tt_xQ}$ are non-$\nu$-low)
\[\copyredd(\pi_0^{\Tt_xP},\exit^{\Tt_x}_0)=\left<N^P_x\right> \]
(in particular, the dropdown sequence consists of only the one element)
and \[\copyredd(\pi_0^{\Tt_xQ},\exit^{\Tt_x}_0)=\left<N^Q_x\right>,\]
but $i^\Vv_{\gamma\eta}(\pi^{\Vv\eta}_\gamma(\left<N^P_x\right>))=\left<N^Y_x\right>=\left<N^Z_x\right>=i^\Ww_{\gamma\eta}(\pi^{\Ww\eta}_\gamma(\left<N^Q_x\right>))$, giving the desired equation.
\end{case}

\begin{case}
$\exit^{\Tt_x}_0\pins M_x$.

Then $\exit^{\Tt_x}_0=e_0^{M_x}(\pvec_{n+1}^{M_x})$ where
$(\nu_0,e_0,\vec{t}_0)$
is the first component of $\sigma$
and $n=\projdeg(M_x)$.

\begin{scase}
$\exit^{\Tt_x}_0\in\core_0(M_x)$.

Then
\[ \copyredd(\pi_0^{\Tt_xP},\exit^{\Tt_x}_0)=\redd(N^P_x,e_0^{N^P_x}(\pvec_{n+1}^{N^P_x})) \]
and likewise for $Q$,
and this is preserved by $i^\Vv_{\gamma\eta}\com\pi^{\Vv\eta}_\gamma$ and $i^\Ww_{\gamma\eta}\com\pi^{\Ww\eta}_\gamma$, so the desired equality follows again from the fact that $N^Y_x=N^Z_x$. \end{scase}

\begin{scase}
$\exit^{\Tt_x}_0\pins M_x$
but $\exit^{\Tt_x}_0\notin\core_0(M_x)$
(so $M_x$ is active type 3).

Since we assumed that $\exit^{\Tt_x}_0$ is small and by $M_x$-goodness, we have $\Omega_x^{P}<\rho_0^P$ and $\Omega_x^Q<\rho_0^Q$.

\begin{sclm}
$\pi_0^{\Tt_xP}$ is $\nu$-preserving
iff $\pi_0^{\Tt_xQ}$ is $\nu$-preserving,
and therefore $\pi_0^{\Tt_xP}$ is $\nu$-high iff $\pi_0^{\Tt_xQ}$ is $\nu$-high.
\end{sclm}
\begin{proof}
For $\nu(M_x)=\nu_0^{M_x}(\pvec_{n+1}^{M_x})$, so the following are equivalent:
\begin{enumerate}[label=--]
 \item 
$\pi_0^{\Tt_xP}$ is $\nu$-preserving
\item $\nu(N^P_x)=\nu_0^{N^P_x}(\pvec_{n+1}^{N^P_x})$,
\item $\nu(N^Y_x)=\nu_0^{N^Y_x}(\pvec_{n+1}^{N^Y_x})$
\item $\nu(N^Z_x)=\nu_0^{N^Z_x}(\pvec_{n+1}^{N^Z_x})$
\item $\nu(N^Q_x)=\nu_0^{N^Q_x}(\pvec_{n+1}^{N^Q_x})$
\item $\pi_0^{\Tt_xQ}$ is $\nu$-preserving.\qedhere
\end{enumerate}
\end{proof}

\begin{sscase}
$\pi_0^{\Tt_xP},\pi_0^{\Tt_xQ}$ are $\nu$-preserving.

We have $\exit^{\Tt_x}_0=e_0^{M_x}(\pvec_{n+1}^{M_x})$,
and\[\copyredd(\pi_0^{\Tt_xP},\exit^{\Tt_x}_0)=\redd(N^P_x,\psi_{\pi_0^{\Tt_xP}}(\exit^{\Tt_x}_0))=\redd(N^P_x,e_0^{N^P_x}(\pvec_{n+1}^{N^P_x})).\]
And since $\Omega_x^P<\rho_0^P$,
 \[i^\Vv_{\gamma\eta}(\pi^{\Vv\eta}_\gamma(\redd(N^P_x,e_0^{N^P_x}(\pvec_{n+1}^{N^P_X})))=\redd(N^Y_x,e_0^{N^Y_x}(\pvec_{n+1}^{N^Y_x})).\]
 By symmetry and since $N^Y_x=N^Z_x$,
 this suffices in this subsubcase.\end{sscase}
 
 \begin{sscase}
  $\pi^{\Tt_xP}_0,\pi^{\Tt_xQ}_0$ are $\nu$-high.
  
  Let $\redd(M_x,\exit^{\Tt_x}_0)=\left<D_i\right>_{i\leq k}$,
  so $k\geq 1$ and $D_0=M_x$ and $\rho_\om^{D_1}=\nu(M_x)$.
  Let $m=\projdeg(D_1)$.
  Let $\copyredd(\pi^{\Tt_xP}_0,\exit^{\Tt_x}_0)=\left<D'_i\right>_{i\leq k'}$,
  so $k'=k$ and $D'_0=N^P_x$. 
 Let $\widetilde{D}_1=\psi_{\pi^{\Tt_xP}_0}(D'_1)$
  and $\varphi:D'_1\to\widetilde{D}_1$
  be such that either:
  \begin{enumerate}[label=--]
   \item Subsubcase \ref{sscase:R_neq_R^+,q_geq_0}
   of the definition of $\copyredd$ \ref{dfn:copyseg} attains,
     \[D'_1=\cHull_{m+1}^{\widetilde{D}_1}(\nu(N^P_x)\cup\pvec_{m+1}^{\widetilde{D}_1})\]
  and  $\varphi:D'_1\to\widetilde{D}_1$
  is the uncollapse map, or
   \item Subsubcase \ref{sscase:R_neq_R^+,q=-1} of the definition of $\copyredd$ \ref{dfn:copyseg} attains, $k=1$
   and $D'_1\pins N^P_x$
   is the segment with $F^{\widetilde{D}_1}\rest\nu(F^P_x)$ active,
   and $\varphi:D'_1\to\widetilde{D}_1$
   is the inclusion map (as a map between the  squashes, as usual).
  \end{enumerate}
  Let $\psi:D_1\to D'_1$  be
  $\psi=\varphi^{-1}\com\psi_{\pi^{\Tt_xP}_0}\rest D_1$.
  Then $\left<D'_i\right>_{1<i\leq k}=\psi(\left<D_i\right>_{1<i\leq k})$.

  But note then that  $\copyredd(\pi_0^{\Tt_xP},\exit^{\Tt_x}_0)$ is determined by $N_x^P$ and the terms $e_0,\nu_0$ (the parameter $\nu(N_x^P)$ is used in defining $D'_1$ and $\varphi$, but we do not use $\nu_0$ to determine $\nu(N_x^P)$,
  as $\pi^{\Tt_xP}_0$ is $\nu$-high;
  $\nu(N_x^P)$ is simply determined by $N_x^P$ itself), via the process described above. This process is preserved by $i^{\Vv}_{\gamma\eta}\com\pi^{\Vv\eta}_\gamma$, since $\Omega_x^P<\rho_0^P$. It easily follows
  that the desired equation holds.
 \end{sscase}
\end{scase}
\end{case}
 
 This completes all (sub)(sub)cases, and hence the proof of part
\ref{item:j=0_images_of_stages_and_redds_match} when $i=0$.

Parts \ref{item:general_images_of_stages_and_redds_match}, \ref{item:j>0_C-structures_match_and_not_moved}: We inductively assume part \ref{item:j=0_images_of_stages_and_redds_match}
holds for $i$, by which we may assume $j>0$, as the conclusions of the $j=0$ instance of part \ref{item:general_images_of_stages_and_redds_match} is just a restatement of those of part \ref{item:j=0_images_of_stages_and_redds_match}. Since $(i,j)$ is relevant and $j>0$, $E^{\wt{\Uu}^P}_{i,j-1}$ and $E^{\wt{\Uu}^Q}_{i,j-1}$ are order $0$ measures. Also, $\sigma^{\Vv\eta}_{h_{i,j-1}}(E^{\wt{\Uu}^P}_{i,j-1})$ and $\sigma^{\Ww\eta}_{h_{i,j-1}}(E^{\wt{\Uu}^Q}_{i,j-1})$ are  compatible
in the sense required to produce a copying stage of the comparison,
and in particular they have a common critical point $\kappa$. 
Let $j'<j$ be largest such that $(i,j')$ is relevant. So $j'\leq j-1$ and part \ref{item:general_images_of_stages_and_redds_match} holds at $(i,j')$.
Let $\gamma'=\gamma^\eta_{h_{ij'}}$.
We have \[\gamma=\gamma^{\Vv\eta}_{h_{ij}}=(\gamma-1)+1=\delta^{\Vv\eta}_{h_{i,j-1}}+1,\]
and
since $E^\Vv_{\gamma-1}=\sigma^{\Vv\eta}_{h_{i,j-1}}(E^{\wt{\Uu}^P}_{i,j-1})$
and $E^\Ww_{\gamma-1}=\sigma^{\Ww\eta}_{h_{i,j-1}}(E^{\wt{\Uu}^Q}_{i,j-1})$ are order $0$ measures,
we have $\gamma-1=\pred^{\Vv\eta}(\gamma)=\delta^{\Vv\eta}_{h_{i,j-1}}=\delta^{\Ww\eta}_{h_{i,j-1}}=\pred^{\Ww\eta}(\gamma)$.
So $\gamma'<^\Vv\gamma-1=\pred^{\Vv}(\gamma)\leq^\Vv\eta$
and $\gamma'<^\Ww\gamma-1=\pred^{\Ww}(\gamma)<^\Ww\eta$.

Now let \begin{enumerate}[label=--]
         \item 
$\zeta_{ij'}=i^\Vv_{\gamma',\gamma-1}(\pi^{\Vv\eta}_{\gamma'}(\alpha_{ij'}^{\Tt_xP}))$,
\item $\hat{\zeta}_{ij'}=i^\Ww_{\gamma',\gamma-1}(\pi^{\Ww\eta}_{\gamma'}(\alpha_{ij'}^{\Tt_xQ}))$,
\item $\left<D_\ell\right>_{j'\leq\ell\leq k}=i^\Vv_{\gamma',\gamma-1}(\pi^{\Vv\eta}_{\gamma'}(\copyredd(\psi_{ij'}^{\Tt_xP},\exit^{\Tt_x}_i)))$,
\item $\left<\hat{D}_\ell\right>_{j'\leq\ell\leq k}=i^\Ww_{\gamma',\gamma-1}(\pi^{\Ww\eta}_{\gamma'}(\copyredd(\psi_{ij'}^{\Tt_xQ},\exit^{\Tt_x}_i)))$.
\end{enumerate}
So \[i^\Vv_{\gamma-1,\eta}\Big(\zeta_{ij'},\left<D_\ell\right>_{j'\leq\ell\leq k}\Big)=i^\Ww_{\gamma-1,\eta}\Big(\hat{\zeta}_{ij'},\left<\hat{D}_\ell\right>_{j'\leq\ell\leq k}\Big).\]

Let
\begin{enumerate}[label=--]
\item $d=\deg^{\Tt_x}_i$
if $j'=0$,
and $d=\om$ otherwise
(so $D_{j'}=\core_d(N_{\zeta_{ij'}x}^{M^\Vv_{\gamma-1}})$ and $\hat{D}_{j'}=\core_d(N_{\hat{\zeta}_{ij'}x}^{M^\Ww_{\gamma-1}})$),
\item $(\Tt',\left<\Psi_\ell\right>_{j'\leq\ell\leq j})=\modres^{M^\Vv_{\gamma-1}}_{\zeta_{ij'},d,\id}(D_j)$,
 \item $(\hat{\Tt}',\left<\hat{\Psi}_\ell\right>_{j'\leq\ell\leq j})=\modres^{M^\Ww_{\gamma-1}}_{\hat{\zeta}_{ij'},d,\id}(\hat{D}_j)$,
\item $\Psi_\ell=(\beta_\ell,A_\ell,\psi_\ell,\alpha_\ell,d_\ell)$,
\item $\hat{\Psi}_\ell=(\hat{\beta}_\ell,\hat{A}_\ell,\hat{\psi}_\ell,\hat{\alpha}_\ell,\hat{d}_\ell)$.
\end{enumerate}
So $\psi_j:D_j\to A_j$ and $\hat{\psi}_j:\hat{D}_j\to\hat{A}_j$ are both $\nu$-preserving, since $j'<j$.

Now $\rho_\om^{A_j}=\kappa=\rho_\om^{\hat{A}_j}$
and $t_{\beta_jx}^{M^{\Vv}_{\gamma-1}}=2=t_{\hat{\beta}_jx}^{M^{\Ww}_{\gamma-1}}$,
so the full resurrections in $M^{\Vv}_{\gamma-1},M^\Ww_{\gamma-1}$ of $D_j,\hat{D}_j$ (starting from $D_{j'},\hat{D}_{j'}$; i.e.~those specified just above) respectively involve
forming ultrapowers 
with critical points $\kappa$,
via the order $0$ measures
 $E^\Vv_{\gamma-1},E^\Ww_{\gamma-1}$.
(And for $j''\in[j',j)$, we have
$t_{\beta_{j''}}^{M^\Vv_{\gamma-1}}\neq 2\neq t_{\hat{\beta}_{j''}}^{M^\Ww_{\gamma-1}}$,
so the earlier steps of resurrection do not involve ultrapowers.)
 
 Now let us establish parts \ref{item:general_images_of_stages_and_redds_match} \ref{item:general_images_of_stages_match}, \ref{item:general_images_of_redds_match} and \ref{item:j>0_C-structures_match_and_not_moved} for $(i,j)$. Let
 \begin{enumerate}[label=--]
  \item 
$r=\pi_\gamma^{\Vv\eta}(\copyredd(\psi^{\Tt_xP}_{ij},\exit^{\Tt_x}_i))$
 and
 \item $\hat{r}=\pi^{\Ww\eta}_{\gamma}(\copyredd(\psi^{\Tt_xQ}_{ij},\exit^{\Tt_x}_i))$;
 \end{enumerate}
 these are discussed in part \ref{item:general_images_of_stages_and_redds_match}\ref{item:general_images_of_redds_match}. Note that the first element of $r$ is just $A_j$,
 and that $r\in M^{\Vv}_{\gamma-1}\cap M^\Vv_\gamma$.
 Let $A_j^*=i^\Vv_{\gamma-1,\eta}(A_j)$.
Since $A_j$ is sound with $\rho_\om^{A_j}=\kappa=\crit(i^\Vv_{\gamma-1,\eta})$, we have $A_j=\Hull^{A_j^*}(\kappa)$,
and the uncollapse map $\tau:A_j\to A_j^*$ is just $i^\Vv_{\gamma-1,\eta}\rest A_j$. Moreover,
$r=\tau^{-1}(i^\Vv_{\gamma-1,\eta}(r))$.
But by induction, $i^\Vv_{\gamma-1,\eta}(r)=i^\Ww_{\gamma-1,\eta}(\hat{r})$,
and in particular, $A_j^*=\hat{A}_j^*$,
where $\hat{A}_j^*=i^\Ww_{\gamma-1,\eta}(\hat{A}_j)$.
Since $\kappa$ is the critical point of both of these maps, the same procedure
recovers $A_j,\hat{A}_j$ on both sides,
so in fact $A_j=\hat{A}_j$.
Likewise, $\tau=\hat{\tau}$
(where $\hat{\tau}$ is defined like $\tau$), so $r=\hat{r}$.
But $A_j$ is sound and projects to $\kappa$, and $\kappa<\crit(i^\Vv_{\gamma\eta})$ and $\kappa<\crit(i^\Ww_{\gamma\eta})$
(as $\kappa=\crit(i^\Vv_{\gamma-1,\eta})=\crit(i^\Ww_{\gamma-1,\eta})$),
so $i^\Vv_{\gamma\eta}(A_j)=A_j=i^\Ww_{\gamma\eta}(A_j)$,
and $i^\Vv_{\gamma\eta}(r)=r=i^\Ww_{\gamma\eta}(r)$.
Finally, $\pi^{\Vv\eta}_{\gamma}(\alpha_{ij}^{\Tt_xP})$
(which is mentioned in part  \ref{item:general_images_of_stages_and_redds_match}\ref{item:general_images_of_stages_match}) is just the unique $\alpha$
such that $A_j=\core_\om(N_{\alpha x}^{M^\Vv_{\gamma}})$,
and since $i^\Vv_{\gamma\eta}(A_j)=i^\Ww_{\gamma\eta}(\hat{A}_j)$,
therefore $i^\Vv_{\gamma\eta}(\pi^{\Vv\eta}_{\gamma}(\alpha_{ij}^{\Tt_xP}))=i^\Ww_{\gamma\eta}(\pi^{\Ww\eta}_{\gamma}(\alpha_{ij}^{\Tt_xQ}))$. This establishes
parts \ref{item:general_images_of_stages_and_redds_match} and \ref{item:j>0_C-structures_match_and_not_moved} for $(i,j)$, as desired.

Part \ref{item:j=0_images_of_stages_and_redds_match} for $i>0$: We have $j=0$.
Let $(i',j')=\pred^{\wt{\Uu}^P}(i,j)$.
Let $j''\leq j'$ be largest such that $(i',j'')$ is relevant.
Part \ref{item:j=0_images_of_stages_and_redds_match} \ref{item:j=0_images_of_stages_match}
follows easily from part \ref{item:general_images_of_stages_and_redds_match}
applied to $(i',j'')$ and the definitions. (For example, if $j''=j'$ then \[\pi^{\Vv\eta}_{\gamma}(\xi_i^{\Tt_xP})=i^\Vv_{\gamma^\eta_{h_{i'j'}}\gamma}(\pi^{\Vv\eta}_{\gamma^\eta_{h_{i'j'}}}(\alpha_{i'j'}^{\Tt_xP}));\] if $j''<j'$ then there are some  resurrection steps between $j''$ and $j'$, which do not involve ultrapowers.) Part \ref{item:j=0_images_of_stages_and_redds_match} \ref{item:j=0_images_of_redds_match}:
This is mostly like the $i=0$ case;
there are just a couple of small differences, which we mention.
Firstly, the terms $\nu_i,e_i$ (and $\vec{t}_i$)
can now use the inputs $\vec{A}_0,\ldots,\vec{A}_{i-1}$.
But because these are independent of which side we are considering ($\Vv$ or $\Ww$), and are not moved by $i^\Vv_{\gamma\eta},i^\Ww_{\gamma\eta}$,
this ends up causing no problem. (Of course, it is key here that for an ordinal $\alpha$ to be considered a copying stage, the relevant $\vec{A}_\ell$'s they produce must agree on the two sides, but that is part of our present assumptions.) Secondly,
it can be that $\pi_i^{\Tt_xP}$ and/or $\pi_i^{\Tt_xQ}$ are/is $\nu$-low;
this was ruled out in case $i=0$.
So this case must be considered also.
But:

\begin{sclm}\label{sclm:nu-low_iff_nu-low}
$\pi_i^{\Tt_xP}$ is $\nu$-low iff $\pi_i^{\Tt_xQ}$ is $\nu$-low.\end{sclm}
\begin{proof}
Suppose $M^{\Tt_x}_i$ is type 3. Then \[\nu(M^{\Tt_x}_i)=\nu_i^{M^{\Tt_x}_i}(\vec{a}_0,\ldots,\vec{a}_{i-1},\pvec_{m+1}^{M^{\Tt_x}_i}),\]
where $\deg^{\Tt_x}_i=m$
and $\vec{a}_\ell=\vec{t}_\ell^{M^{\Tt_x}_\ell}(\vec{a}_0,\ldots,\vec{a}_{\ell-1},\pvec_{d+1}^{M^{\Tt_x}_\ell})$ where $d=\deg^{\Tt_x}_\ell$.
Suppose $\pi_i^{\Tt_xP}$ is $\nu$-low.
Then  \[\nu_i^{N}(\vec{a}'_0,\ldots,\vec{a}'_{i-1},\pvec_{m+1}^N)<\nu(N),\]
where $N=\core_m(N_{\xi_i^{\Tt_xP}x}^{M^{\wt{\Uu}^P}_{i0}})$,
and $\vec{a}'_\ell=\pi^{\Tt_xP}_i(\vec{a}_\ell)$. This fact is preserved by $i^\Vv_{\gamma\eta}\com\pi^{\Vv\eta}_\gamma$, 
and $\vec{A}_\ell=\pi^{\Vv\eta}_{\gamma}(\vec{a}'_\ell)$
(and $\max(\vec{A}_\ell)<\crit(i^\Vv_{\gamma\eta})$, by part \ref{item:A-vecs<crit(i)}), so that
\[ \nu_i^{N^*}(\vec{A}_0,\ldots,\vec{A}_{\ell-1},\pvec_{m+1}^{N^*})<\nu(N^*),\]
where $N^*=\core_m(N_{\xi^*x}^{M^\Vv_\eta})$
where $\xi^*=i^\Vv_{\gamma\eta}(\pi^{\Vv\eta}_\gamma(\xi_i^{\Tt_xP}))$.
But if $\pi_i^{\Tt_xQ}$
is non-$\nu$-low
then we get a contradiction, by employing similar considerations regarding $\Ww$, and since $i^\Ww_{\gamma\eta}$
is non-$\nu$-low (since it is an iteration map), and $\pi^{\Ww\eta}_\gamma$
is non-$\nu$-low (since it is a near $0$-embedding (see Definition \ref{dfn:coarse_tree_emb} and Footnote \ref{ftn:near_embeddings_for_coarse_tree_embeddings})).
\end{proof}

So we need to handle the new case,
that both $\pi_i^{\Tt_xP}$ and $\pi_i^{\Tt_xQ}$ are non-$\nu$-low.
In this case, if $\exit^{\Tt_x}_i=M^{\Tt_x}_i$, the parameter $\nu=\psi_{\pi_i^{\Tt_xP}}(\nu(M_x))$
is involved in the definition of $\copyseg(\pi_i^{\Tt_xP},\exit^{\Tt_x}_i)$
(indeed, this is defined as the initial segment of $N$
with $F^N\rest\nu$ active,
where $N=N_{\xi_i^{\Tt_xP}x}^{M^{\wt{\Uu}^P}_{i0}}$).
But as in the proof of Subclaim \ref{sclm:nu-low_iff_nu-low},
this parameter is named by the term $\nu_i$
(with parameters as above),
and so it has the same image on either side, so it does not cause a problem.

The remaining details are essentially
as when $i=0$, so we leave the verification to the reader.

This completes the induction,
and hence the proof of parts \ref{item:j=0_images_of_stages_and_redds_match}, \ref{item:general_images_of_stages_and_redds_match},
 \ref{item:j>0_C-structures_match_and_not_moved}.

Part \ref{item:xi^*,a^*_images_match}:
The fact that
$i^\Vv_{\gamma\eta}(\pi_{ij}^{\Vv\eta}(\avec_{i}^{*\Yback}))=i^\Ww_{\gamma\eta}(\pi_{ij}^{\Ww\eta}(\avec^{*\Zback}_{i}))$ is proven just like  parts \ref{item:j=0_images_of_stages_and_redds_match}, \ref{item:general_images_of_stages_and_redds_match},
 \ref{item:j>0_C-structures_match_and_not_moved} above.
 (Recall here that $\vec{a}_i\in[\nu(E^{\Tt_x}_i)]^{<\om}$,
 and the action of \[\psi_{ik_i}:\exit^{\Tt_x}_i\to N= N_{\xi_i^{*P}x}^{M^{\wt{\Uu}_{ik_i}}} \]
 is computed essentially via the process of how $\exit^{\Tt_x}_i$
 itself is lifted to $N$.)
 We leave the details to the reader.
 
Now let use verify   $i^\Vv_{\gamma\eta}(\pi_{ij}^{\Vv\eta}(\xi_i^{*\Yback}))=i^\Ww_{\gamma\eta}(\pi_{ij}^{\Ww\eta}(\xi^{*\Zback}_{i}))$. If $k_i>0$ then the maps $\psi_{ik_i}$ are $\nu$-preserving,
so the desired equality already follows from part \ref{item:general_images_of_stages_and_redds_match}.
So suppose $k_i=j=0$,
so $E^{\Tt_x}_i=F^{M^{\Tt_x}_i}$.
We may assume $M^{\Tt_x}_i$ is type 3.
By part \ref{item:j=0_images_of_stages_match}, $\xi=i^\Vv_{\gamma\eta}(\pi_{i0}^{\Vv\eta}(\xi_i^{\Yback}))=i^\Ww_{\gamma\eta}(\pi_{i0}^{\Ww\eta}(\xi^{\Zback}_{i}))$. Since $\exit^{\Tt_x}_i$ is small
and
like before, it follows that $\pi_i^{\Tt_xP}$ is $\nu$-low/preserving/high iff $\pi_i^{\Tt_xQ}$ is $\nu$-low/preserving/high.
So we may assume that $\pi_i^{\Tt_xP},\pi_i^{\Tt_xQ}$ are both $\nu$-low; that is,
\[\wt{\nu}^P= \psi_{\pi^{\Tt_xP}_i}(\nu(M^{\Tt_x}_i))<\nu(N_{\xi_i^Px}^{M^{\wt{\Uu}^P}_{i0}}),\]
\[\wt{\nu}^Q= \psi_{\pi^{\Tt_xQ}_i}(\nu(M^{\Tt_x}_i))<\nu(N_{\xi_i^Qx}^{M^{\wt{\Uu}^Q}_{i0}}).\]
But also  \[\nu(M^{\Tt_x}_i)=\nu_i(\vec{a}_0,\ldots,\vec{a}_{i-1},\pvec_{m+1}^{M^{\Tt_x}_i})
               \]
               where $m$ is as usual.
  Let $N=N_{\xi x}^Y=N_{\xi x}^Z$.
 Since $\pi^{\Vv\eta}_{i0}(\xi_i^P)<\rho_0(M^{\wt{\Uu}^P}_{i0})$ etc, 
  we also get
\[i^\Vv_{\gamma\eta}(\pi^{\Vv\eta}_{i0}(\wt{\nu}^P)))=\nu_i^{N}(\vec{A}_0,\ldots,\vec{A}_{i-1},\pvec_{m+1}^{N})=i^\Ww_{\gamma\eta}(\pi^{\Ww\eta}_{i0}(\wt{\nu}^Q)),\]
as desired;
here note that for $\ell<i$,
we have\[i^\Vv_{\gamma\eta}(\pi^{\Vv\eta}_{i0}(\pi_{i}^{\Tt_xP}(\vec{a}_\ell)))=\vec{A}_\ell=i^\Ww_{\gamma\eta}(\pi^{\Ww\eta}_{i0}(\pi_i^{\Tt_xQ}(\vec{a}_\ell))).\qedhere\]
\end{proof}

\begin{clm}\label{clm:fully_small_situation_does_not_arise}
 It is not the case that  $m'=\bar{m}<m$ and $e_x,e_y$ are small
 (that is, Subcase \ref{scase:m'=mbar_and_S_x_small}
 attains).
\end{clm}
\begin{proof}
Suppose otherwise. We have that $e_x=\exit^{\Tt_x}_{m'}$ and $e_y=\exit^{\Tt_y}_{m'}$ exist,
so there are corresponding dropdown sequences.

\begin{sclm}
 $(i_{\mathrm{end}},j_{\mathrm{end}})$
 is the largest relevant $(i,j)$ in $\dom(\wt{\Uu}^P)$; in other words,
 $M^{\wt{\Uu}^P}_{i_{\mathrm{end}}j_{\mathrm{end}}}=M^{\wt{\Uu}^P}_\infty=M^{\wt{\Uu}^P}_{i_{\mathrm{dis}}j_{\mathrm{dis}}}$ (but maybe there are further steps of padding between $(i_{\mathrm{end}},j_{\mathrm{end}})$
 and the last node $(i_{\mathrm{end}}=i_{\mathrm{dis}},j_{\mathrm{dis}})$ of $\wt{\Uu}^P$).
\end{sclm}
\begin{proof}
Suppose not. Let $(i,j)=(i_{\mathrm{end}},j_{\mathrm{end}})$. Let $\gamma$ be as in  Claim \ref{clm:key_images_match} for $(i,j)$.
Using that claim, and the full compatibility of the constructions at stage $\eta$, it is straightforward to show that $\eta$ is in fact a copying stage of the comparison, contradicting that it terminated there.
(If $M^{\wt{\Uu}^P}_{ij}=M^{\wt{\Uu}^P}_{ik_i}$, so all order $0$ measures
involved in the resurrection of $\exit^{\Tt_x}_i$ have already been successfully copied, then part \ref{item:xi^*,a^*_images_match} of 
Claim \ref{clm:key_images_match} is relevant in seeing that $\eta$ meets the requirements for a copying stage.)
\end{proof}

Recall that $(i_{\mathrm{dis}},j_{\mathrm{dis}})$ is the last index of $\wt{\Uu}^P$ and of $\wt{\Uu}^Q$. Recall  we defined
$\Uu_2^\Yback=\Uu_0^\Yback\conc\Uu_1^\Yback$
and $\wt{\Uu}^P=\Uu_2^P\rest((i_{\mathrm{dis}},j_{\mathrm{dis}}+1))$.

\begin{sclm}
$(i_{\mathrm{dis}},j_{\mathrm{dis}})$
is the last index of $\Uu_2^P$.
\end{sclm}
\begin{proof}
Suppose not. So  $(i,j+1)=(i_{\mathrm{dis}},j_{\mathrm{dis}}+1)\in\dom(\Uu_2^P)$
and
$t_{\beta_{j+1}x}^{M^{\wt{\Uu}^P}_{ij}}=2$
iff $t_{\beta_{j+1}'x}^{M^{\wt{\Uu}^Q}_{ij}}\neq 2$,
where $\beta_{j+1}$ is as in the definition of $\leq_{4i+3}^\Yback$
and $\beta_{j+1}'$ likewise for $\leq_{4i+3}^\Zback$. But since these $t$-values are preserved by the relevant maps, this easily contradicts Claim \ref{clm:key_images_match} part \ref{item:general_images_of_stages_match} combined with the compatibility of the constructions at stage $\eta$.
\end{proof}

So $i_{\mathrm{dis}}=\bar{m}$ and $j_{\mathrm{dis}}=\min(j^P,j^Q)$,
where $j^P,j^Q$ were introduced when defining $\Uu_1^P$ in Subcase \ref{scase:m'=mbar_and_S_x_small}. 
We follow the notation from there.

\begin{sclm}\label{sclm:some_j_leq_k_x_has_param_disagreement}
 There is $j\leq k_x$ such that either $\params^P_{xj}\neq\params^P_{yj}$ or $\params^Q_{xj}\neq\params^Q_{yj}$.
\end{sclm}
\begin{proof}
Suppose not; in  particular, $j_{\mathrm{dis}}=j^P=j^Q=k_x$.
Recall that $x\leq^P_{4\bar{m}+4}y$
iff $x\not\leq^Q_{4\bar{m}+4}y$,
and note that this implies
$x\leq^P_{4\bar{m}+3}y$
iff $x\not\leq^Q_{4\bar{m}+3}y$
(the extra considerations
for $\leq_{4\bar{m}+4}$
depend only on $M_x,M_y$).
So consider the definition of $\leq_{4\bar{m}+3}$
in our present case, which is
Subcase \ref{scase:S_x_small} of Definition \ref{dfn:lifting_norm}.
The first ordinals it compares are $k_x,k_y$, so $k_x=k_y$. Let $k=k_x=k_y$.
Adopt notation as there, with ``superscript $P$'' to indicate the $P$-version, and likewise for $Q$.
Since $\params^P_{xj}=\params^P_{yj}$
 for all $j\leq k$ (and the resurrection trees involved in the definition of $\leq_{4\bar{m}+3}$ have form the tail end of $\wt{\Uu}^P$),
\[x\leq_{4\bar{m}+3}^P y\iff
\langle\wt{\alpha}^P_x,\lh(F^{M^{\wt{\Uu}^P}_{\bar{m}k}}_{\wt{\alpha}^P_xx}),\psi^P_x(\avec_{\bar{m}})\rangle\leq_{\mathrm{lex}}\langle\wt{\alpha}^P_y,\lh(F^{M^{\wt{\Uu}^P}_{\bar{m}k}}_{\wt{\alpha}^P_yy}),\psi^P_y(\bvec_{\bar{m}})\rangle.\]

Now $\wt{\alpha}_x^P=\xi_{\bar{m}}^{*P}$
and $\psi^P_{x}(\vec{a}_{\bar{m}})=\vec{a}_{\bar{m}}^{*P}$ (just expressed in a different notation system), so
by part \ref{item:xi^*,a^*_images_match}
of Claim \ref{clm:key_images_match},
applied to $(i_{\mathrm{end}},j_{\mathrm{end}})=(\bar{m},j_{\mathrm{end}})$,
\[i^\Vv_{\gamma\eta}(\pi_{\bar{m}j_{\mathrm{end}}}^{\Vv\eta}(\wt{\alpha}_x^P))=i^\Ww_{\gamma\eta}(\pi_{\bar{m}j_{\mathrm{end}}}^{\Ww\eta}(\wt{\alpha}_x^Q)),\]
\[ i^\Vv_{\gamma\eta}(\pi_{\bar{m}j_{\mathrm{end}}}^{\Vv\eta}(\psi_x^P(\avec_{\bar{m}})))=i^\Ww_{\gamma\eta}(\pi_{\bar{m}j_{\mathrm{end}}}^{\Ww\eta}(\psi_x^Q(\avec_{\bar{m}}))).\]
Likewise regarding $y$.
Since $x\leq^P_{4\bar{m}+3}y\iff x\not\leq^Q_{4\bar{m}+3}y$,
therefore $\wt{\alpha}^P_x=\wt{\alpha}^P_y$ and $\wt{\alpha}^Q_x=\wt{\alpha}^Q_y$ (and write $\wt{\alpha}^P=\wt{\alpha}^P_x$ and $\wt{\alpha}^Q=\wt{\alpha}^Q_x$), and also
\[\lh(F^{M^{\wt{\Uu}^P}_{\bar{m}k}}_{\wt{\alpha}^Px})\leq\lh(F^{M^{\wt{\Uu}^P}_{\bar{m}k}}_{\wt{\alpha}^Py})\iff \lh(F^{M^{\wt{\Uu}^Q}_{\bar{m}k}}_{\wt{\alpha}^Qx})\not\leq\lh(F^{M^{\wt{\Uu}^Q}_{\bar{m}k}}_{\wt{\alpha}^Qy}). \]
So letting $\wt{\alpha}=i^\Vv_{\gamma\eta}(\pi^{\Vv\eta}_{\bar{m}j_{\mathrm{end}}}(\wt{\alpha}^P))=i^\Ww_{\gamma\eta}(\pi^{\Ww\eta}_{\bar{m}j_{\mathrm{end}}}(\wt{\alpha}^Q))$, we have
\[ \lh(F^{Y}_{\wt{\alpha}x})\leq\lh(F^Y_{\wt{\alpha}y})\iff\lh(F^Z_{\wt{\alpha}x})\not\leq\lh(F^Z_{\wt{\alpha}y}).\]
So we may assume $\lh(F^Y_{\wt{\alpha}x})\leq\lh(F^Y_{\wt{\alpha}y})$,
so $\lh(F^Z_{\wt{\alpha}y})<\lh(F^Z_{\wt{\alpha}x})$. Recall that $F^Y_{\wt{\alpha}x},F^Z_{\wt{\alpha}x}$ are compatible
(in the sense considered in the comparison), as are $F^Y_{\wt{\alpha}y},F^Z_{\wt{\alpha}y}$. In particular, letting $\kappa=\crit(F^Z_{\wt{\alpha}x})$, we have $\kappa=\crit(F^Y_{\wt{\alpha}x})$.
Let
\[ A=\{\alpha<\kappa\bigm|t^Z_{\alpha y}=1\}.\]
Let $j:Z\to U=\Ult(Z,F^Z_{\wt{\alpha}x})$ be the ultrapower map.
Then $\wt{\alpha}\in j(A)$, since $\CC^U_y\rest\wt{\alpha}=\CC^Z_y\rest\wt{\alpha}$,
and therefore in fact $\CC^U_y\rest(\wt{\alpha}+1)=\CC^Z_y\rest(\wt{\alpha}+1)$,
since $F^Z_{\wt{\alpha}y}\in\es^U$.
But
\[ A=\{\alpha<\kappa\bigm|t^Y_{\alpha y}=1\},\]
by the compatibility of the constructions. In particular, $A\in Y\cap Z$. And \[\wt{\alpha}<\theta=\min(\lh(F^Y_{\wt{\alpha}x}),\lh(F^Z_{\wt{\alpha}x})),\]
since $\wt{\alpha}=\beta+\om$
for some limit $\beta$ and\[\nu(F^Y_{\wt{\alpha}x})=\aleph_{\beta}^Y<\aleph_{\beta}^Y+\om<\lh(F^Y_{\wt{\alpha}x}),\]
and likewise for $Z$.
Since
\[ F^Y_{\wt{\alpha}x}\cap(X\cross[\theta]^{<\om})=F^Z_{\wt{\alpha}x}\cap(X\cross[\theta]^{<\om}) \]
where $X=Y\cap Z$,
therefore $\wt{\alpha}\in j'(A)$
where $j':Y\to U'=\Ult(Y,F^Y_{\wt{\alpha}x})$
is the ultrapower map,
so $t^{U'}_{\wt{\alpha}y}=1$.
But $\CC^{U'}_y\rest\wt{\alpha}=\CC^Y_y\rest\wt{\alpha}$
(since $\nu(F^Y_{\wt{\alpha}x})=\aleph_\beta^Y$),
so $N^Y_{\beta y}=N^{U'}_{\beta y}$.
But  $F^{U'}_{\wt{\alpha}y}\in\es^{U'}$ and $\nu(F^{U'}_{\wt{\alpha}y})=\aleph_\beta^{U'}=\aleph_\beta^Y$, so
\[ \aleph_\beta^Y=\nu(F^{U'}_{\wt{\alpha}y})<\lh(F^{U'}_{\wt{\alpha}y})<\aleph_{\beta+1}^{U'}=\lh(F^Y_{\wt{\alpha}x})\leq\lh(F^Y_{\wt{\alpha}y}).\]
By coherence then, $F^{U'}_{\wt{\alpha}y}\in\es^Y$, and $\lh(F^{U'}_{\wt{\alpha}y})<\lh(F^Y_{\wt{\alpha}y})$.
But since $N^{Y}_{\beta y}=N^{U'}_{\beta y}$,
$F^{U'}_{\wt{\alpha}y}$
is therefore a valid background at stage $\wt{\alpha}=\beta+\om$ of $\CC^Y_y$,
contradicting the minimality of $F^Y_{\wt{\alpha}y}$.
This completes the proof of the subclaim.
\end{proof}

By the subclaim, $j_{\mathrm{dis}}$ is the least $j$ witnessing the subclaim.
And $M^{\wt{\Uu}^P}_{\bar{m}j_{\mathrm{end}}}=M^{\wt{\Uu}^P}_{\bar{m}j_{\mathrm{dis}}}$. Continuing with $(P,x)$, using notation  as in  Subcase \ref{scase:m'=mbar_and_S_x_small}
(and noting that the resurrection data is computed starting from the model $M^{\wt{\Uu}^P}_{\bar{m}0}$), using superscript $P$ to indicate that the computation is on the $P$-side (though this notation differs from its usual interpretation, as it is actually $M^{\wt{\Uu}^P}_{\bar{m}0}$ in which we compute the resurrection), we have
\begin{enumerate}[label=--]
\item $\alpha_{xj}^{*P}=\OR(\prodseg^{M^{\wt{\Uu}^P}_{\bar{m}j}}(\alpha^P_{xj}))$ for $j\in[0,k]$,
 \item $\beta_{xj}^{*P}=\OR(\prodseg^{M^{\wt{\Uu}^P}_{\bar{m},j-1}}(\beta^P_{xj}))$ for $ j\in(0, k]$,
 \item $\params^P_{x0}=\langle\rho_\om^{A^P_{x1}},\beta^P_{x1},t^{M^{\wt{\Uu}^P}_{
 \bar{m}0}}_{\beta^P_{x1}},\beta_{x1}^{*P}\rangle$ if $0<k$,
 \item $\params^P_{xj}=\langle\alpha^P_{xj},t^{M^{\wt{\Uu}^P}_{
 \bar{m}j}}_{\alpha^P_{xj}},\alpha^{*P}_{xj},\rho_\om^{A^P_{x,j+1}},\beta^P_{x,j+1},t^{M^{\wt{\Uu}^P}_{\bar{m}j}}_{\beta^P_{x,j+1}},\beta^{*P}_{x,j+1}\rangle$
 for $j\in(0,k]$, and
 \item $\params^P_{xk}=\langle\alpha^P_{xk},t^{M^{\wt{\Uu}^P}_{\bar{m}k}}_{\alpha^P_{xk}},\alpha_{xk}^{*P}\rangle$ (possibly $k=0$).
\end{enumerate}

Likewise for $(P,y)$, $(Q,x)$ and $(Q,y)$. The following case distinction is really just to simplify notation a little;
the other cases are dealt with in essentially the same way.

\begin{case}\label{case:j_dis=0<k} $j_{\mathrm{dis}}=0<k$.

Using Claim \ref{clm:key_images_match}
(and the overall disagreement between $P,Q$), it is easy to see that $\rho_\om^{A^P_{x1}}=\rho_\om^{A^P_{y1}}$
and $\beta^P_{x1}=\beta^P_{y1}$
and  $t^{M^{\wt{\Uu}^P}_{\bar{m}0}}_{\beta^P_{x1}x}=t^{M^{\wt{\Uu}^P}_{\bar{m}0}}_{\beta^P_{y1}y}$,
and likewise for $Q$,
and therefore that
\begin{equation}\label{eqn:contradictory_beta^*_disagreement} \beta^{*P}_{x1}\leq\beta^{*P}_{y1}\iff\beta^{*Q}_{x1}\not\leq\beta^{*Q}_{y1}.\end{equation}
But we have $\beta_{x1}^P=\beta_{y1}^P$ and $\beta_{x1}^Q=\beta_{y1}^Q$.
So we should examine a little the distinction between these ordinals $\beta$ and the corresponding ordinals $\beta^*$. Also, by the usual considerations, \[i^\Vv_{\gamma\eta}(\pi^{\Vv\eta}_{\bar{m}0}((\rho_\om^{A^P_{x1}},\beta^P_{x1},t^{M^{\wt{\Uu}^P}_{\bar{m}0}}_{\beta^P_{x1}x})))=i^\Ww_{\gamma\eta}(\pi^{\Ww\eta}_{\bar{m}0}((\rho_\om^{A^Q_{x1}},\beta^Q_{x1},t^{M^{\wt{\Uu}^Q}_{\bar{m}0}}_{\beta^Q_{x1}x}))).\]
 In particular,
in fact  $t^{M^{\wt{\Uu}^P}_{\bar{m}0}}_{\beta^P_{x1}x}=t^{M^{\wt{\Uu}^Q}_{\bar{m}0}}_{\beta^Q_{x1}x}$.
Likewise for $y$.

If $t=3$ (that is, the associated $t$-value) then $\beta^*=\beta$ (directly by definition).
Moreover, in the present context, by the above remarks,  if any of the 4  $t$-parameters $=3$ then all 4 $t$-parameters  $=3$.
Putting these things together, by line (\ref{eqn:contradictory_beta^*_disagreement}), none of the 4 $t$-parameters in question are $=3$.

If the $t$-parameters  $=1$,
then $\beta^*=\lh(F)$ where $F$ is the associated background extender. But therefore line (\ref{eqn:contradictory_beta^*_disagreement}) leads to a contradiction just like in the proof of Subclaim \ref{sclm:some_j_leq_k_x_has_param_disagreement}.

If the $t$-parameters  $=2$,
then the ordinals $\beta^*$
are just $\lh(F)$ where $F$ is the order $0$ measure on $\rho_\om^{A}$,
and so $\beta^{*P}_{x1}=\beta^{*P}_{y1}$
and $\beta^{*Q}_{x1}=\beta^{*Q}_{y1}$,
contradicting line (\ref{eqn:contradictory_beta^*_disagreement}).

So the $t$-parameters $=0$.
So the 4 structures $A^P_{x1}$, etc, are passive.
Let $\eta^P_{x1}$ be the supremum of all $\eta<\OR(A^P_{x1})$
such that $A^P_{x1}|\eta$ is active.
Let $\bar{\beta}^P_{x1}$
be the supremum of all $\beta<\beta^P_{x1}$ such that $N^{M^{\wt{\Uu}^P}_{\bar{m}0}}_{\beta x}$ is active.

We claim that $t^{M^{\wt{\Uu}^P}_{\bar{m}0}}_{\bar{\beta}^P_{x1}x}\neq 3$. For suppose otherwise.
Then $t_{\bar{\beta}^P_{x1}y}^{M^{\wt{\Uu}^P}_{\bar{m}0}}=3$ also
(note that although we are talking about a model of $\CC_y^{M^{\wt{\Uu}^P}_{\bar{m}0}}$, the index referred to is $\bar{\beta}^P_{x1}$, not $\bar{\beta}^P_{y1}$).
Let $\delta$ be the least Woodin of the corresponding structures.
Since $t^{M^{\wt{\Uu}^P}_{\bar{m}0}}_{\beta^P_{x1}x}=0$,
the Q-structure for $\delta$ is reached
at some stage $\beta\in[\bar{\beta}^P_{x1},\beta^P_{x1})$. Let $\mu$ be the least measurable of $M^{\wt{\Uu}^P}_{\bar{m}0}$ with $\mu>\delta$.
Then note that $\beta^P_{x1}<\mu$,
and between stages $\bar{\beta}^P_{x1}$ and $\beta^P_{x1}$, $\CC_x^{M^{\wt{\Uu}^P}_{\bar{m}0}}$ only uses rud closure and coring down.
It follows that $\beta^{*P}_{x1}=\beta^P_{x1}$.
But $\beta^P_{y1}=\beta^P_{x1}<\mu$,
so between the same stages,
$\CC_y^{M^{\wt{\Uu}^P}_{\bar{m}0}}$
also just uses rud closure and coring down, so $\beta^{*P}_{y1}=\beta^P_{y1}$. We get the same picture on the $Q$-side,
so $\beta^{*P}_{x1}=\beta^{*P}_{y1}=\beta^{*Q}_{x1}=\beta^{*Q}_{y1}$,
contradiction.

\begin{scase}
 $A^P_{x1}|\eta^P_{x1}$ is active; that is, $A^P_{x1}$ has a largest active segment.
 
 Then 
$\bar{\beta}^P_{x1}<\beta^P_{x1}$, $N^{M^{\wt{\Uu}^P}_{\bar{m}0}}_{\bar{\beta}^P_{x1}x}$ is active 
and 
$N^{M^{\wt{\Uu}^P}_{\bar{m}0}}_{\beta x}$ is passive for all $\beta\in(\bar{\beta}^P_{x1},\beta^P_{x1}]$.

Suppose $t^{M^{\wt{\Uu}^P}_{\bar{m}0}}_{\bar{\beta}^P_{x1}x}=2$.
Let $\mu=\card^{M^{\wt{\Uu}^P}_{\bar{m}0}}(\bar{\beta}^P_{x1})$.
Since $t^{M^{\wt{\Uu}^P}_{\bar{m}0}}_{\bar{\beta}^P_{x1}x}=0$,
we have $\chi_{\mu x}^{M^{\wt{\Uu}^P}_{\bar{m}0}}\leq\beta^P_{x1}$.
But by a simple reflection, the fact that $N^{M^{\wt{\Uu}^P}_{\bar{m}0}}_{\bar{\beta}^P_{x1}x}$ is active implies that there are cofinally many
$\beta<\chi_{\mu x}^{M^{\wt{\Uu}^P}}$
such that $N^{M^{\wt{\Uu}^P}_{\bar{m}0}}_{\beta x}$ is active,
contradicting that
$\bar{\beta}^P_{x1}$ is the largest such which is ${<\beta^P_{x1}}$.

So $t^{M^{\wt{\Uu}^P}_{\bar{m}0}}_{\bar{\beta}^P_{x1}x}=1$
and $\beta^P_{x1}<\mu$,
where $\mu$ is the least measurable
of $M^{\wt{\Uu}^P}_{\bar{m}0}$
which is $>\bar{\beta}^P_{x1}$,
if there is such.
Thus in fact, $\beta^P_{x1}<\mu'$,
where $\mu'=\crit(E)$,
for $E\in\es(M^{\wt{\Uu}^P}_{\bar{m}0})$ least with $\lh(F^{M^{\wt{\Uu}^P}_{\bar{m}0}}_{\bar{\beta}^P_{x1}x})<\crit(E)$, if there is such. Since $F^{M^{\wt{\Uu}^P}_{\bar{m}0}}_{\bar{\beta}^P_{x1}x}$
is a total extender whose natural length is a limit cardinal of $M^{\wt{\Uu}^P}_{\bar{m}0}$,
$E$ is moreover least
in $\es(M^{\wt{\Uu}^P}_{\bar{m}0})$
with $\lh(F^{M^{\wt{\Uu}^P}_{\bar{m}0}}_{\bar{\beta}^P_{x1}x})<\lh(E)$,
if there is any such $E$.
So 
\[\max(\lh(F^{M^{\wt{\Uu}^P}_{\bar{m}0}}_{\bar{\beta}^P_{x1}x}),\beta^P_{x1})<\mu'=\crit(E)<\lh(E),\]
if $M^{\wt{\Uu}^P}_{\bar{m}0}$
has any (partial) extender in its sequence with index $>\lh(F^{M^{\wt{\Uu}^P}_{\bar{m}0}}_{\bar{\beta}^P_{x1}x})$ (and then $E$ is least such).  (Actually
it's not hard to see that $\beta^P_{x1}<\lh(F^{M^{\wt{\Uu}^P}_{\bar{m}0}}_{\bar{\beta}^P_{x1}x})$).
But then $\CC_x^{M^{\wt{\Uu}^P}_{\bar{m}0}}$ just uses rud closure and coring down in the interval $(\bar{\beta}^P_{x1},\beta^P_{x1})$,
so \begin{equation}\label{eqn:lh(F)<beta^*}\lh(F^{M^{\wt{\Uu}^P}_{\bar{m}0}}_{\bar{\beta}^P_{x1}x})<\beta^{*P}_{x1}\end{equation}
and also, if there is $E\in\es(M^{\wt{\Uu}^P}_{\bar{m}0})$ with $\lh(F^{M^{\wt{\Uu}^P}_{\bar{m}0}}_{\bar{\beta}^P_{x1}x})<\lh(E)$,  and $E$ is least such, then
\begin{equation}\label{eqn:beta^*<crit(E)} \beta^{*P}_{x1}<\crit(E).\end{equation}

\begin{sscase}
$A^P_{y1}|\eta^P_{y1}$ is also active.

Then $1=t^{M^{\wt{\Uu}^P}_{\bar{m}0}}_{\bar{\beta}^P_{x1}x}=t^{M^{\wt{\Uu}^P}_{\bar{m}0}}_{\bar{\beta}^P_{y1}y}$, and  the foregoing also holds with $y$ replacing $x$. Note that like for $\beta_{x1}^P$ etc, we have
\begin{equation}\label{eqn:beta-bar_order_agmt} \bar{\beta}^P_{x1}\leq\bar{\beta}^P_{y1}\iff\bar{\beta}^Q_{x1}\leq\bar{\beta}^Q_{y1}.\end{equation}
But if $\bar{\beta}^P_{x1}<\bar{\beta}^P_{y1}$
then $\lh(F^{M^{\wt{\Uu}^P}_{\bar{m}0}}_{\bar{\beta}^P_{x1}x})<\aleph_{\bar{\beta}^P_{y1}}^{M^{\wt{\Uu}^P}_{\bar{m}0}}<\lh(F^{M^{\wt{\Uu}^P}_{\bar{m}0}}_{\bar{\beta}^P_{y1}y})$,
which gives that $E$ as above (with respect to $x$) exists
and $\beta^{*P}_{x1}<\lh(E)<\aleph_{\bar{\beta}^P_{y1}}^{M^{\wt{\Uu}_{\bar{m}0}}}<\beta^{*P}_{y1}$. Similarly and using (\ref{eqn:beta-bar_order_agmt}),
we also get $\beta^{*Q}_{x1}<\beta^{*Q}_{x1}$, contradicting line (\ref{eqn:contradictory_beta^*_disagreement}).
So $\bar{\beta}^P_{x1}=\bar{\beta}^P_{y1}$ and $\bar{\beta}^Q_{x1}=\bar{\beta}^Q_{y1}$.
Write $\bar{\beta}^P=\bar{\beta}^P_{x1}=\bar{\beta}^P_{y1}$
and $\bar{\beta}^Q$ likewise.
Now like in the proof of Subclaim \ref{sclm:some_j_leq_k_x_has_param_disagreement},
it follows that
\[ \lh(F^{M^{\wt{\Uu}^P}_{\bar{m}0}}_{\bar{\beta}^Px})\leq\lh(F^{M^{\wt{\Uu}^P}_{\bar{m}0}}_{\bar{\beta}^Py})
 \iff \lh(F^{M^{\wt{\Uu}^Q}_{\bar{m}0}}_{\bar{\beta}^Qx})\leq\lh(F^{M^{\wt{\Uu}^Q}_{\bar{m}0}}_{\bar{\beta}^Qy}).
\]
But if $\lh(F^{M^{\wt{\Uu}^P}_{\bar{m}0}}_{\bar{\beta}^Px})<\lh(F^{M^{\wt{\Uu}^P}_{\bar{m}0}}_{\bar{\beta}^Py})$
then note there is $E$ as above
for $x$, and \[\lh(F^{M^{\wt{\Uu}^P}_{\bar{m}0}}_{\bar{\beta}^Px})<\beta^{*P}_{x1}<\lh(E)<\lh(F^{M^{\wt{\Uu}^P}_{\bar{m}0}}_{\bar{\beta}^Py})<\beta^{*P}_{y1}.\]
Likewise for $Q$, which together yields
a contradiction to line (\ref{eqn:contradictory_beta^*_disagreement}). So $F^{M^{\wt{\Uu}^P}_{\bar{m}0}}_{\bar{\beta}^Px}=F^{M^{\wt{\Uu}^P}_{\bar{m}0}}_{\bar{\beta}^Py}$
 and $F^{M^{\wt{\Uu}^Q}_{\bar{m}0}}_{\bar{\beta}^Qx}=F^{M^{\wt{\Uu}^Q}_{\bar{m}0}}_{\bar{\beta}^Qy}$.
But after using these background extenders, the constructions
just use rud closure and coring down until stage $\beta^P_{x1}$ etc,
which easily implies that $\beta^{*P}_{x1}\leq\beta^{*P}_{y1}$ iff $\beta^{*Q}_{x1}\leq\beta^{*Q}_{y1}$,
again a contradiction.
\end{sscase}

\begin{sscase}\label{scase:A^P_y1_passive} $A^P_{y1}|\eta^P_{y1}$ is passive.
 
 We remark that we will develop
 this subsubcase
 at first without using the present case assumption that $A^P_{x1}$ is active, so that in the end we can also apply what we do to the subcase that both are passive.
 
 We have that $A^P_{y1}|\eta^P_{y1}$ is a passive limit of active segments.
 We get 
$t^{M^{\wt{\Uu}^P}_{\bar{m}0}}_{\bar{\beta}^P_{y1}y}\notin\{2, 3\}$
just like before, and so 
$t^{M^{\wt{\Uu}^P}_{\bar{m}0}}_{\bar{\beta}^P_{y1}y}=0$.

We claim that $\bar{\beta}^P_{y1}$
is a limit of ordinals $\beta$ with $t_{\beta y}^{M^{\wt{\Uu}^P}_{\bar{m}0}}=1$.
If not, then it is straightforward to see that there is a measurable $\mu$
such that $\bar{\beta}^P_{y1}=\chi^{M^{\wt{\Uu}^P}_{\bar{m}0}}_{\mu y}=\eta^P_{y1}$.
Since $A^P_{y1}$ is sound and satisfies condensation,
we must therefore have $\rho_\om(A^P_{y1})>\mu$
(if $\rho_\om(A^P_{y1})\leq\mu$
then $\rho_\om(A^P_{y1})=\mu$,
since it is sound and $\mu$ is a cardinal; but then $A^P_{y1}$ is a proper segment of the stack at $\mu$,
contradicting the fact that $\bar{\beta}^P_{y1}=\chi^{M^{\wt{\Uu}^P}_{\bar{m}0}}_{\mu y}$). So $\rho_\om(A^P_{y1})\geq\eta^P_{y1}$.  But this contradicts the role of $A^P_{y1}$  in the resurrection of $e_y$, which is active (the current image of $e_y$ must have height $>\rho_\om(A^P_{y1})$). 

It follows
that $\bar{\beta}^P_{y1}$
is a limit of measurable cardinals
of $M^{\wt{\Uu}^P}_{\bar{m}0}$,
and is also the limit of the ordinals $\lh(F_{\beta y}^{M^{\wt{\Uu}^P}_{\bar{m}0}})$ over all $\beta<\bar{\beta}^P_{1}$,
and hence $\bar{\beta}^P_{y1}$ coincides with the height of the associated production segment;
that is, \[\bar{\beta}^P_{y1}=\OR(\prodseg^{M^{\wt{\Uu}^P}_{\bar{m}0}}(\bar{\beta}^P_{y1})).\]
It now easily follows that
\begin{equation}\label{eqn:beta-bar_leq_beta^*} \bar{\beta}^P_{y1}\leq\beta^{*P}_{y1} \end{equation}
and if there is any (partial)
in $E\in\es(M^{\wt{\Uu}^P}_{\bar{m}0})$
such that $\bar{\beta}^P_{y1}<\lh(E)$, then letting $E$ be such with $\lh(E)$ minimal, we have
\begin{equation}\label{eqn:beta^*_y<crit(E)} \beta^{*P}_{y1}<\crit(E).\end{equation}

From now on we continue to use the case assumptions regarding $x$.
We get the corresponding situation on the $Q$-side,
as these things are preserved by the various maps. For example,
the $Q$-versions
of lines  (\ref{eqn:lh(F)<beta^*}), (\ref{eqn:beta^*<crit(E)}),
(\ref{eqn:beta-bar_leq_beta^*}),
(\ref{eqn:beta^*_y<crit(E)})
hold. So $\beta^{*P}_{x1}$
is just slightly beyond
$\lh(F^{M^{\wt{\Uu}^P}_{\bar{m}0}}_{\bar{\beta}^P_{x1}x})$,
whereas $\beta^{*P}_{y1}$
is just slightly beyond
$\bar{\beta}^P_{y1}$, and the latter is a limit of measurables of $M^{\wt{\Uu}^P}_{\bar{m}0}$.
These things easily serve the separate $\beta^{*P}_{x1}$ and $\beta^{*P}_{y1}$;
that is,
note that either:
\begin{enumerate}[label=--]
 \item $\lh(F^{M^{\wt{\Uu}^P}_{\bar{m}0}}_{\bar{\beta}^P_{x1}x})<\bar{\beta}^P_{y1}$
 and $\beta^{*P}_{x1}<\beta^{*P}_{y1}$, or
  \item $\lh(F^{M^{\wt{\Uu}^P}_{\bar{m}0}}_{\bar{\beta}^P_{x1}x})>\bar{\beta}^P_{y1}$
 and $\beta^{*P}_{x1}>\beta^{*P}_{y1}$.\footnote{Note that if the limit
$\bar{\beta}^P_{y1}=\aleph^{M^{\wt{\Uu}^P}_{\bar{m}0}}_{\bar{\beta}^P_{y1}}$
of measurables in question
happens to be exactly $\nu(F^{M^{\wt{\Uu}^P}_{\bar{m}0}}_{\bar{\beta}^P_{x1}x})$,
then there is $E$ as above for $y$ (in particular as in line (\ref{eqn:beta^*_y<crit(E)})),
and $\lh(E)<\lh(F^{M^{\wt{\Uu}^P}_{\bar{m}0}}_{\bar{\beta}^P_{x1}x})$.}
\end{enumerate}
These orderings are also preserved by the relevant maps, and so we get $\beta^{*P}_{x1}\leq\beta^{*P}_{y1}$
iff $\beta^{*Q}_{x1}\leq\beta^{*Q}_{y1}$, a contradiction.
So we have ruled out this subsubcase and hence the current subcase.
\end{sscase}
\end{scase}

\begin{scase}
 $A^P_{x1}|\eta^P_{x1}$ is passive.

By the previous subcase and symmetry, $A^P_{y1}|\eta^P_{y1}$ is also passive.
The first part of the discussion in Subsubcase \ref{scase:A^P_y1_passive} holds, for both $y$ and $x$.
So much like before, it is easy to see that
\[ \bar{\beta}^P_{x1}\leq\bar{\beta}^P_{y1}\iff\bar{\beta}^Q_{x1}\leq\bar{\beta}^Q_{y1}.\]
But in the interval $[\bar{\beta}^P_{x1},\beta^P_{x1}]$,
$\CC^{M^{\wt{\Uu}^P}_{\bar{m}0}}_x$ just uses rud closure and coring down, and in particular, $\beta^{*P}_{x1}<\mu$ where $\mu$ is the least measurable of $M^{\wt{\Uu}^P}_{\bar{m}0}$ with $\bar{\beta}^P_{x1}<\mu$, if such exists. So if
$\bar{\beta}^P_{x1}<\bar{\beta}^P_{y1}$, and hence $\bar{\beta}^Q_{x1}<\bar{\beta}^Q_{y1}$,
it follows immediately that
$\beta^{*P}_{x1}<\beta^{*P}_{y1}$ and $\beta^{*Q}_{x1}<\beta^{*Q}_{y1}$,
a contradiction.
So $\bar{\beta}^P_{x1}=\bar{\beta}^P_{y1}$ and $\bar{\beta}^Q_{x1}=\bar{\beta}^Q_{y1}$. But now it easily follows that
\[ \beta^{*P}_{x1}\leq\beta^{*P}_{y1}\iff\beta^{*Q}_{x1}\leq\beta^{*Q}_{y1},\]
again a contradiction.
 \end{scase}
 
So in the present case
(Case \ref{case:j_dis=0<k}),
we have reached a contradiction overall.
 \end{case}

 \begin{case}Otherwise (it is not the case that $j_{\mathrm{dis}}=0<k$).
  
  This is a very slight variant of the previous case, so we omit further discussion.
 \end{case}

 Thus, we have completed the proof of Claim \ref{clm:fully_small_situation_does_not_arise}.
\end{proof}

\begin{clm}\label{clm:not_m'=m-bar<m_and_e_x,e_y_non-small}
It is not the case that  $m'=\bar{m}<m$ and $e_x,e_y$ are non-small
 (that is, Subcase \ref{scase:S_x_non-small} attains).
\end{clm}
\begin{proof}
 The proof of this claim is mostly very similar to that of Claim \ref{clm:fully_small_situation_does_not_arise}. There is just one situation where there is a little wrinkle, which is as follows.
 Suppose that $\Omega_x^P=\Omega_y^P=\OR^P$,
 so $P$ is a Q-mouse.
 Suppose $P$ is active type 3,
 that $[0,\bar{m}]^{\Tt_x}$
 does not drop in model,
 that $e_x\neq M^{\Tt_x}_{\bar{m}}$ and $e_x\notin (M^{\Tt_x}_{\bar{m}})^\sq$. Suppose
 also that $\Tt_y,\bar{m},e_y$ are likewise. Note that
 $k_x=\critresl^{\Tt_xP}(e_x,\delta^{e_x})=0$,
 so letting
 \[\critres^{\Tt_xP}(e_x,\delta^{e_x})=(\Vv_x,\left<\Psi_{xj}\right>_{j\leq k_x},\varrho_x),\]
 we just have $\varrho_x=\pi^{\Tt_xP}_{\bar{m}}$. We have
 \[ \params_{x0}=\params_{xk_x}=\left<\alpha_{x0},t^{M^{\wt{\Uu}^P}_{\bar{m}0}}_{\alpha_{x0}},\alpha^*_{x0}\right>=\left<\OR^{M^{\wt{\Uu}^P}_{\bar{m}0}},3,\OR^{M^{\wt{\Uu}^P}_{\bar{m}0}}\right>,\]
 so
 \[ \params_x=\left<\OR^{M^{\wt{\Uu}^P}_{\bar{m}0}},3,\OR^{M^{\wt{\Uu}^P}_{\bar{m}0}},\OR^{\exitcopyseg^{\pi^{\Tt_xP}_{\bar{m}}}(e_x)}\right>.\]
 Write $\psi_x=\copymap^{\pi^{\Tt_xP}_{\bar{m}}}(e_x)$. So then
 $x\leq_{4\bar{m}+3}^Py$
 iff
 \[ \left<\OR^{\exitcopyseg^{\pi^{\Tt_xP}_{\bar{m}}}(e_x)},\psi_x(\vec{a}_{\bar{m}})\right>\leq_{\mathrm{lex}}
  \left<\OR^{\exitcopyseg^{\pi^{\Tt_yP}_{\bar{m}}}(e_y)},\psi_y(\vec{b}_{\bar{m}})\right>
 \]
(using that $k_x=k_y=0$
and $\alpha_{x0}^*=\alpha_{x0}=\OR(M^{\wt{\Uu}^P}_{\bar{m}0})=\alpha_{y0}=\alpha_{y0}^*$).

 The kinds of arguments from before show that 
 \[ \pi^{\Tt_xP}_{\bar{m}}
\text{ is }
\nu\text{-low}
\iff\pi^{\Tt_xQ}_{\bar{m}}\text{ is }\nu\text{-low.}\] The case in which
$\pi^{\Tt_xP}_{\bar{m}},\pi^{\Tt_yP}_{\bar{m}}$ are $\nu$-low is straightforward, so suppose these maps are non-$\nu$-low
(hence likewise $\pi^{\Tt_xQ}_{\bar{m}}$ and $\pi^{\Tt_yQ}_{\bar{m}}$). So 
\begin{enumerate}[label=--]\item $\exitcopyseg^{\pi^{\Tt_xP}_{\bar{m}}}(e_x)=\copyseg^{\pi^{\Tt_xP}_{\bar{m}}}(e_x)$ and
 \item 
 $\exitcopymap^{\pi^{\Tt_xP}_{\bar{m}}}(e_x)=\psi_x=\copymap^{\pi^{\Tt_xP}_{\bar{m}}}(e_x)$.
 \end{enumerate}
 So letting $\psi_x=\copymap^{\pi^{\Tt_xP}_{\bar{m}}}(e_x)$ and  $\psi_y$ be likewise, we have \begin{equation}\label{eqn:when_leq^P_4m-bar+3}x\leq_{4\bar{m}+3}^Py\iff
 \big(\OR^{\copyseg^{\pi^{\Tt_xP}_{\bar{m}}}(e_x)},\psi_x(\vec{a}_{\bar{m}})\big)\leq_{\mathrm{lex}}\big(\OR^{\copyseg^{\pi^{\Tt_yP}_{\bar{m}}}(e_y)},\psi_y(\vec{b}_{\bar{m}})\big).\end{equation}
Likewise for $\leq^Q_{4\bar{m}+3}$.

 \begin{sclm}\label{sclm:order_OR(copyseg)_pres} The following are equivalent:
 \begin{enumerate}[label=(\roman*)]
 \item\label{item:OR_inequality_P} $\OR(\copyseg(\pi^{\Tt_xP}_{\bar{m}},e_x))\leq\OR(\copyseg(\pi^{\Tt_yP}_{\bar{m}},e_y))$
 \item\label{item:OR_inequality_Vv}$\OR(\copyseg(\pi^{\Vv\eta}_{\bar{m}0}\com\pi^{\Tt_xP}_{\bar{m}},e_x))\leq\OR(\copyseg(\pi^{\Vv\eta}_{\bar{m}0}\com\pi^{\Tt_yP}_{\bar{m}},e_y))$
  \item\label{item:OR_inequality_Ww}$\OR(\copyseg(\pi^{\Ww\eta}_{\bar{m}0}\com\pi^{\Tt_xQ}_{\bar{m}},e_x))\leq\OR(\copyseg(\pi^{\Ww\eta}_{\bar{m}0}\com\pi^{\Tt_yQ}_{\bar{m}},e_y))$
   \item\label{item:OR_inequality_Q} $\OR(\copyseg(\pi^{\Tt_xQ}_{\bar{m}},e_x))\leq\OR(\copyseg(\pi^{\Tt_yQ}_{\bar{m}},e_y))$
 \end{enumerate}
 \end{sclm}
\begin{proof}
Let $A_x\pins M^{\Tt_x}_{\bar{m}}$
be least such that $e_x\ins A_x$
and $\rho_\om^{A_x}=\nu(M^{\Tt_x}_{\bar{m}})$.
Let $\wt{A}_x^P=\psi_{\pi^{\Tt_xP}_{\bar{m}}}(A_x)$
and $A^P_x\pins N_x^{M^{\wt{\Uu}^P}_{\bar{m}0}}$ be least such that $\copyseg(\pi^{\Tt_xP}_{\bar{m}},e_x)\ins A^P_x$ and $\rho_\om(A^P_x)=\nu(M^{\wt{\Uu}^P}_{\bar{m}0})=\nu(N_x^{M^{\wt{\Uu}^P}_{\bar{m}0}})$.
So either:
\begin{enumerate}[label=(\roman*)]
 \item\label{item:rho_0=nu}$A_x=e_x$ is active type 3 with $\nu(e_x)=\nu(M^{\Tt_x}_{\bar{m}})$,
 \[ \copyseg(\pi^{\Tt_xP}_{\bar{m}},e_x)=A^P_x\text{ is active with }
F^{A^P_x}=F^{\wt{A}^P_x}\rest\nu(M^{\wt{\Uu}^P}_{\bar{m}0});\] let $\varphi^P_x:(A^P_x)^\sq\to(\wt{A}^P_x)^\sq$ be the inclusion map  and $\varrho^P_x:(A_x)^\sq\to (A^P_x)^\sq$ be $\varrho^P_x=(\varphi^P_x)^{-1}\com\psi_{\pi^{\Tt_xP}_{\bar{m}}}\rest (A_x)^\sq$, or
 \item\label{item:rho_0>nu} $\rho_0^{e_x}>\nu(M^{\Tt_x}_{\bar{m}})$ and letting $n<\om$ be such that $\rho_{n+1}^{A_x}=\nu(M^{\Tt_x}_{\bar{m}})<\rho_n^{A_x}$, we have
 \[A^P_x=\cHull_{n+1}^{\wt{A}^P_x}(\pvec_{n+1}^{\wt{A}^P_x}\cup\nu(M^{\wt{\Uu}^P}_{\bar{m}0})),\]
 and letting $\varphi^P_x:A^P_x\to\wt{A}^P_x$ be the uncollapse
 and $\varrho^P_x:A_x\to A^P_x$
 be $\varrho^P_x=(\varphi^P_x)^{-1}\com\psi_{\pi^{\Tt_xP}_{\bar{m}}}\rest A_x$, then $\varrho^P_x,\varphi^P_x$ are $\nu$-preserving and \[\copyseg(\pi^{\Tt_xP}_{\bar{m}},e_x)=\psi_{\varrho^P_x}(e_x).\]
\end{enumerate}
  Let $\hat{A}^P_x\pins M^{\wt{\Uu}^P}_{\bar{m}0}$ with $\OR(\hat{A}^P_x)=\OR(A^P_x)$. So $A^P_x$ is the output of the P-construction of $\hat{A}^P_x$ over $A^P_x|\delta^{A^P_x}$. Likewise
  for $\hat{\wt{A}}^P_x$.
  Note that, for example in Case \ref{item:rho_0>nu} above, 
  \[ \hat{A}^P_x=\cHull_{n+1}^{\hat{\wt{A}}^P_x}(\pvec_{n+1}^{\hat{\wt{A}}^P_x}\cup\nu(M^{\wt{\Uu}^P}_{\bar{m}0})) \]
  (and $\hat{\varphi}^P_x$ is the uncollapse map) and
 $\range(\varphi^P_x)=\range(\hat{\varphi}^P_x)\cap\wt{A}^P_x$,
 and in particular,
 $\OR\cap\range(\varphi^P_x)=\OR\cap\range(\hat{\varphi}^P_x)$.

 Define $\wt{A}^{\Vv}_x$, $A^{\Vv}_x$, $\varphi^\Vv_x,\varrho^\Vv_x$ analogously,
 from  $\pi^{\Vv\eta}_{\bar{m}0}\com\pi^{\Tt_xP}_{\bar{m}}$ (considered
as a map $M^{\Tt_x}_{\bar{m}}\to N_x^{M^{\Vv}_\eta}$).
 For example, $\wt{A}^{\Vv}_x=\psi_{\pi^{\Vv\eta}_{\bar{m}0}\com\pi^{\Tt_xP}_{\bar{m}}}(A_x)$.
 Define $\hat{A}^{\Vv}_x$, etc, likewise.

\begin{ssclm}\label{ssclm:pres_pins_proj_segs}
 If $\OR(A^P_x)<\OR(A^P_y)$ then
$\OR(A^{\Vv}_x)<\OR(A^{\Vv}_y)$.
\end{ssclm}
\begin{proof}
Recall that $\pi^{\Vv\eta}_{\bar{m}0}$ is non-$\nu$-low.
If it is $\nu$-preserving then the
subclaim is easier,
so suppose that it is $\nu$-high.
Suppose $\OR(A^P_x)<\OR(A^P_y)$;
equivalently,
  $\hat{A}^P_x\pins\hat{A}^P_y$.
 We show that $\OR(A^{\Vv}_x)<\OR(A^{\Vv}_y)$, or equivalently, that
 $\hat{A}^{\Vv}_x\pins\hat{A}^{\Vv}_y$.
 
\begin{case}
 Case \ref{item:rho_0=nu} attains
 for both $x$ and $y$.
 
 Since $\hat{A}^P_x\pins\hat{A}^P_y$, we can fix $(a,f)\in(\hat{A}^P_y)^\sq$
 such that $\hat{A}^P_x=[a,f]^{\hat{A}^P_y}_{F^{\hat{A}^P_y}}$.
 Then letting \[(\hat{A}'_x,\hat{A}'_y,a',f')=\psi_{\pi^{\Vv\eta}_{\bar{m}0}}(\hat{A}^P_x,\hat{A}^P_y,a,f),\]
we have $\hat{A}'_x\pins\hat{A}'_y$ and 
$\hat{A}'_x=[a',f']^{\hat{A}'_y}_{F^{\hat{A}'}_y}$. Moreover,
$a',f'\in\rg(\pi^{\Vv\eta}_{\bar{m}})$. Let $\nu=\nu(F^{M^{\Vv}_\eta})$. So $\nu<\nu(F^{\hat{A}'_x})=\nu(F^{\hat{A}'_y})$, since $\pi^{\Vv\eta}_{\bar{m}}$ is $\nu$-high. Note that $\nu(F^{\hat{A}^\Vv_x})=\nu$ and $F^{\hat{A}^\Vv_x}\rest\nu=F^{\hat{A}'_x}\rest\nu$;
likewise for $y$. So let
\[ k:\Ult(M^\Vv_\eta,F^{\hat{A}^\Vv_y})\to\Ult(M^\Vv_\eta,F^{\hat{A}'_y}) \]
be the factor map. Since $a',f'\in\rg(\pi^{\Vv\eta}_{\bar{m}})$, we have $\hat{A}'_x\in\rg(k)$. But letting $k(\bar{A})=\hat{A}'_x$,
we have $F^{\bar{A}}\rest\nu=F^{\hat{A}'_x}\rest\nu=F^{\hat{A}^\Vv_x}\rest\nu$, and so $F^{\hat{A}^\Vv_x}\in\Ult(M^\Vv_\eta,F^{\hat{A}^\Vv_y})$, and therefore
\[ \lh(F^{\hat{A}^\Vv_x})<\nu^{+\Ult(M^\Vv_\eta,F^{\hat{A}^\Vv_y})}=\lh(F^{\hat{A}^\Vv_y})\]
(where $\lh$ here literally refers to the trivial completion),
so $\hat{A}^\Vv_x\pins\hat{A}^\Vv_y$, as desired.
\end{case}

\begin{case}Case \ref{item:rho_0>nu}
 attains for both $x$ and $y$.
 
 Then $\rho_\om^{\hat{A}^P_x}=\nu(M^{\wt{\Uu}^P}_{\bar{m}0})<\rho_0(\hat{A}^P_y)$,
 so $\hat{A}^P_x\in(\hat{A}^P_y)^\sq$.
 Let
 \[ (\hat{A}'_x,\hat{A}'_y)=\psi_{\pi^{\Vv\eta}_{\bar{m}0}}(\hat{A}^P_x,\hat{A}^P_y).\]
 Note that $\hat{A}^\Vv_x=\cHull_{n+1}^{\hat{A}'_x}(\pvec_{n+1}^{\hat{A}'_x}\cup\nu(M^\Vv_\eta))$ and likewise for $y$.
 Let $\varphi:\hat{A}^\Vv_y\to\hat{A}'_y$ be the uncollapse map. Then \[\hat{A}'_x\in\psi_{\pi^{\Vv\eta}_{\bar{m}0}}``(\hat{A}^P_y)^\sq\sub\rg(\varphi).\]
 But then letting $\varphi(\bar{A})=\hat{A}'_x$,
 we get a natural map $k:\hat{A}^\Vv_x\to\bar{A}$, and $\bar{A}\pins\hat{A}^\Vv_y$,
 so $\OR(\hat{A}^\Vv_x)<\OR(\hat{A}^\Vv_y)$, so $\hat{A}^\Vv_x\pins\hat{A}^\Vv_y$.
\end{case}
\begin{case}
 Case \ref{item:rho_0=nu}
 attains for $x$ iff Case \ref{item:rho_0>nu} attains for $y$.
 
 This case is dealt with similarly to the previous cases, and we omit further discussion.
\end{case}
This completes the proof of Subsubclaim \ref{ssclm:pres_pins_proj_segs}.
\end{proof}

\begin{ssclm}\label{ssclm:pres_equality_proj_segs}
 If $\OR(A^P_x)=\OR(A^P_y)$ then $\OR(A^\Vv_x)=\OR(A^\Vv_y)$.
\end{ssclm}
\begin{proof}
This follows from some of the considerations which arose in the proof of the Subsubclaim  \ref{ssclm:pres_pins_proj_segs},
so we leave it to the reader.
\end{proof}

\begin{ssclm}\label{ssclm:proj_seg_order_P,Q_agmt} The following are equivalent:
\begin{enumerate}[label=--]
 \item  $\OR(A^P_x)\leq\OR(A^P_y)$
 \item $\OR(A^\Vv_x)\leq\OR(A^\Vv_y)$
 \item $\OR(A^\Ww_x)\leq\OR(A^\Ww_y)$
\item $\OR(A^Q_x)\leq\OR(A^Q_y)$.
\end{enumerate}
\end{ssclm}
\begin{proof}
Note that $A^\Vv_x=A^\Ww_x$
and $A^\Vv_y=A^\Ww_y$.
The subsubclaim follows from these equalities and Subsubclaims \ref{ssclm:pres_pins_proj_segs} and \ref{ssclm:pres_equality_proj_segs}
and symmetry.
\end{proof}

\begin{ssclm}\label{ssclm:if_proj_seg_pins_then_exit_pins}
 Suppose $\OR(A^P_x)<\OR(A^P_y)$. Then
 \[ \OR(\copyseg(\pi^{\Tt_xP}_{\bar{m}},e_x))\leq\OR(A^P_x)<\OR(\copyseg(\pi^{\Tt_yP}_{\bar{m}},e_y))\leq\OR(A^P_y),\]
 so in particular,
 \[\OR(\copyseg(\pi^{\Tt_xP}_{\bar{m}},e_x))<\OR(\copyseg(\pi^{\Tt_yP}_{\bar{m}},e_y)).\]
 Likewise for $A^\Vv_x,A^\Vv_y$
 replacing $A^P_x,A^P_y$, etc.
\end{ssclm}
\begin{proof}
We just need to see that $\OR(A^P_x)<\OR(\copyseg(\pi^{\Tt_yP}_{\bar{m}},e_y))$. But $A^P_y$
is the least $A\pins N_y^{M^{\wt{\Uu}^P}_{\bar{m}0}}$
such that $\copyseg(\pi^{\Tt_yP}_{\bar{m}},e_y)\ins A$
and $\rho_\om^A=\nu(M^{\wt{\Uu}^P}_{\bar{m}0})$.
So if $\OR(\copyseg(\pi^{\Tt_yP}_{\bar{m}},e_y))\leq\OR(A^P_x)$
then since $\rho_\om(A^P_x)=\nu(M^{\wt{\Uu}^P}_{\bar{m}0})$,
we would have $\OR(A^P_y)\leq\OR(A^P_x)$, contradiction.
(Note we are using the fine structural equivalence between $A^P_x$ and $\hat{A}^P_x$, etc.)
\end{proof}

Now let us establish the equivalence
of conditions \ref{item:OR_inequality_P}, 
 \ref{item:OR_inequality_Vv},
 \ref{item:OR_inequality_Ww}
 and \ref{item:OR_inequality_Q}.
By Subsubclaims \ref{ssclm:proj_seg_order_P,Q_agmt}
and \ref{ssclm:if_proj_seg_pins_then_exit_pins}, we may assume 
 $\OR(A^P_x)=\OR(A^P_y)$,
 and hence
 $\OR(A^\Vv_x)=\OR(A^\Vv_y)$,
 $\OR(A^\Ww_x)=\OR(A^\Ww_y)$ and $\OR(A^Q_x)=\OR(A^Q_y)$.
 So if we are in Case \ref{item:rho_0=nu} for $x$ or for $y$, then we are in that case for both $x$ and $y$, and we are done.
 So assume that we are in Case \ref{item:rho_0>nu} for both $x$ and $y$. Then some calculations
 like in the previous subsubclaims give the desired conclusion. This completes the proof of Subclaim \ref{sclm:order_OR(copyseg)_pres}.
\end{proof}
So by line (\ref{eqn:when_leq^P_4m-bar+3}) and Subclaim \ref{sclm:order_OR(copyseg)_pres},
we may assume that \begin{equation}\label{eqn:OR=}\OR(\copyseg(\pi^{\Tt_xP}_{\bar{m}},e_x))=\OR(\copyseg(\pi^{\Tt_yP}_{\bar{m}},e_y)),\end{equation}
and so 
\begin{equation}\label{eqn:when_leq^P_4m-bar+3_with_OR=}x\leq_{4\bar{m}+3}^Py\iff
 \psi_x(\vec{a}_{\bar{m}})\leq_{\mathrm{lex}}\psi_y(\vec{b}_{\bar{m}}).\end{equation}
But further calculations like those used above show then that $x\leq^P_{4\bar{m}+3}y$
iff $x\leq^Q_{4\bar{m}+3}y$, contradicting the choice of $\bar{m}$. This completes the proof of Claim \ref{clm:not_m'=m-bar<m_and_e_x,e_y_non-small}.
\end{proof}

\begin{clm}\label{clm:not_m'=mbar=m}
 It is not the case the $m'=\bar{m}=m$ (that is, Subcase \ref{scase:m'=mbar=m} attains).
\end{clm}
\begin{proof}
 This is easy to see with a little of the calculations used already to rule out other cases.
\end{proof}

\begin{clm}\label{clm:not_m'<mbar}
 It is not the case that $m'<\bar{m}$ (that is, Subcase \ref{scase:m'<mbar} attains).
\end{clm}
\begin{proof}
 Suppose otherwise.
 Recall that $\sigma$ is the norm description under consideration, $\Tt_x=\Tt^\sigma_x\rest(m'+1)$
 is the small part of $\Tt^
 \sigma_x$,
 likewise for $y$, and
 $\wt{\Uu}^P=\Uu_0^P\conc\Uu_1^P$
 where $\Uu_0^P=\lifttree^{\Tt_xP}=\lifttree^{\Tt_yP}$
 and 
 \[\Uu^\Yback_{1}=\critrestree^{Rc}_{\xi_xd\pi_x}(e_x,\delta^{e_x})=
\critrestree^{Rc}_{\xi_yd\pi_y}(e_y,\delta^{e_y})\]
where $R=M^{\Uu^P_0}_{m'0}$,
$c=\deg^{\Uu^P_0}_{m'0}$,
$\xi_x=\xi^{\Tt_xP}_{m'}$,
$d=\deg^{\Tt_x}_{m'}=\deg^{\Tt_y}_{m'}$, $\pi_x=\pi^{\Tt_xP}_{m'}$,
 $e_x=\exit^{\Tt_x^\sigma}_{m'}$,
and $\xi_y,\pi_y,e_y$ are likewise.

Let \[j+1=\lh(\Uu_1^P)=
\critresl^{R}_{\xi_x d\pi_x}(e_x,\delta^{e_x})+1.\]
So $(m',j)$ is the last index of $\wt{\Uu}^P$. Let $\Tt^{\uparrow}_x$
be the non-small part of $\Tt_x^\sigma\rest(\bar{m}+1)$, so $\Tt_x^\sigma\rest(\bar{m}+1)=\Tt_x\conc\Tt^{\uparrow}_x$.
 Let \[(L_x,\ell_x)=\psegdeg^{M^{\Tt_x}_{m'}}(e_x,\delta^{e_x}).\]
 So $L_x\ins M^{\Tt_x}_{m'}$, $\delta^{L_x}=\delta^{e_x}$, $\ell_x\geq 0$
 (since if $\ell_x=-1$ then
 $L_x$ is active type 3 with $\rho_{0}^{L_x}=\delta^{e_x}$, but this is impossible as $\delta^{e_x}$ is the least Woodin of $L_x$).
 We consider $\Tt^{\uparrow}_x$
 as an $\ell_x$-maximal above-$\delta^{e_x}$ tree on $L_x$.
 Let
\[ \zeta^P_x=\critresprodstage^{Rc}_{\xi_x d\pi_x}(e_x,\delta^{e_x}), \]
 \[ \psi^P_x=\critresmap^{Rc}_{\xi_xd\pi_x}(e_x,\delta^{e_x}).\]
So
$\psi^P_x:(L_x)^\sq\to (N^P_x)^\sq$ is a weak $\ell_x$-embedding
where $N^P_x=N_{\zeta_xx}^{M^{\wt{\Uu}^P}_{m'j}}$, and $\psi^P_x(\delta^{L_x})=\delta^{N^P_x}$.

Let $\gamma=\gamma^\eta_{m'j}$.
Let $\zeta^\Vv_x=i^\Vv_{\gamma\eta}\com\pi^{\Vv\eta}_{m'j}(\zeta^P_x)$
(where this means $\zeta^\Vv_x=\OR(M^\Vv_\eta)$ if $\zeta^P_x=\OR(M^{\wt{\Uu}^P}_{m'j})$). Let $\psi^\Vv_x=i^\Vv_{\gamma\eta}\com\pi^{\Vv\eta}_{m'j}\com\psi^P_x$; then
 $\psi^\Vv_x:(L_x)^\sq\to(N^\Vv_x)^\sq$ is a weak $\ell_x$-embedding where
$N^\Vv_x=N_{\zeta^\Vv_xx}^{M^\Vv_\eta}$.

Define the ``sub-$y$''
and ``super-$Q$'' and ``super-$\Ww$'' versions of everything analogously. Then $\ell_x=\ell_y$,
$\zeta^P_x=\zeta^P_y$,
$\zeta^Q_x=\zeta^Q_y$,
$\zeta^\Vv_x=\zeta^\Vv_y=\zeta^\Ww_x=\zeta^\Ww_y$.

Given an sse-essentially-$\ell_x$-maximal above-$\delta^{N^P_x}$ tree
$\Tt$ on $N^P_x$, let
\begin{equation}\label{eqn:def_Tt'} \Tt'=\text{
 the tree on }M^{\wt{\Uu}^P}_{m'j}|\zeta^P_x\text{ which is equivalent to }\Tt
 \end{equation}
(this is above-$\delta^{M^{\wt{\Uu}^P}_{m'j}|\zeta^P_x}$ and sse-essentially-$\ell_x$-maximal).

 Recalling the copying procedure specified in 
Definition \ref{dfn:neat_copy}.
Let $\Xx^P_x=\copyseg(\psi^P_x,\Tt^\uparrow_x)$ (this is an above-$\delta^{N^P_x}$ sse-essentially-$\ell_x$-maximal tree on $N^P_x$) and \[\psi^P_{xi}:M^{\Tt^\uparrow_x}_i\to M^{\Xx^P_x}_i \]
be the $i$th resulting copy map (starting with $\psi_{x0}^P=\psi^P_x$). Let
$\Xx^P_y=\copyseg(\psi^P_y,\Tt^\uparrow_x)$ and $\psi^P_{xi}$ be likewise. Then
\begin{enumerate}[label=--]\item $(\Xx^P_x)'=(\Xx^P_y)'$ (these trees are defined  in line (\ref{eqn:def_Tt'}), and
 \item $\psi^P_{xi}(\vec{a}_0,\ldots,\vec{a}_{m'+i-1})=\psi^P_{yi}(\vec{b}_0,\ldots,\vec{b}_{m'+i-1})$;
\end{enumerate}
the equalities hold since $\lh(\Tt_x\conc\Tt_x^\uparrow)=\bar{m}+1$, and we have full agreement before index $\bar{m}$. Likewise for $Q$.

Similarly let $\Xx^\Vv_x=\copyseg(\psi^\Vv_x,\Tt^\uparrow_x)$ (a tree on $N^\Vv_x$), and let
$\psi^\Vv_{xi}:M^{\Tt^\uparrow_x}_i\to M_i^{\Xx^\Vv_x}$
be the resulting $i$th copy map.

Also let $\Xx^{P\Vv}_x=\copyseg(i^\Vv_{\gamma\eta}\com\pi^{\Vv\eta}_{m'j},\Xx^P_x)$
(also on $N^\Vv_x$)
and
$\psi^{P\Vv}_{xi}:M^{\Xx^P_x}_i\to M^{\Xx^\Vv_x}_i$
be the  resulting $i$th  copy map.

\begin{sclm}\label{sclm:copying_comm}
 $\Xx^{P\Vv}_x=\Xx^\Vv_x$
 and $\psi^{\Vv}_{xi}=\psi^{P\Vv}_{xi}\com\psi^P_{xi}$ for each $i\leq\bar{m}-m'$. 
\end{sclm}
\begin{proof}
This is a calculation like some of those in the proof of Claim \ref{clm:not_m'=m-bar<m_and_e_x,e_y_non-small}.
One useful observation is that for each $i$, $\psi^{P\Vv}_{xi}$ is non-$\nu$-low; in fact,
$\psi^{P\Vv}_{xi}$ is a near $\deg^{\Xx^P_x}_i$-embedding.
This is because $\psi^{P\Vv}_{x0}$ is a near $\deg^{\Xx^P_x}_0$-embedding, and this is propagated through the copying construction used here, by the proof in \cite{fs_tame}. We leave the details to the reader.
\end{proof}

\begin{sclm} We have:
\begin{enumerate}[label=--]
 \item  
$\Xx^{\Vv}_x=\Xx^\Ww_x$, $\Xx^{\Vv}_y=\Xx^\Ww_y$,  $(\Xx^\Vv_x)'=(\Xx^\Vv_y)'$ and $(\Xx^\Ww_x)'=(\Xx^\Ww_y)'$,
and so these trees are all mutually equivalent (the respective $i$th models are equivalent modulo small generics,
 the trees have the same lengths and use extenders of the same indices, etc),
 and
 \item $\psi^\Vv_{xi}(\vec{a}_0,\ldots,\vec{a}_{m'+i-1})=\psi^\Vv_{yi}(\vec{b}_0,\ldots,\vec{b}_{m'+i-1})=\psi^\Ww_{xi}(\vec{a}_0,\ldots,\vec{a}_{m'+i-1})=\psi^\Ww_{xi}(\vec{b}_0,\ldots,\vec{b}_{m'+i-1})$ for all $i\leq\bar{m}-m'$.
 \end{enumerate}
\end{sclm}
\begin{proof}
The fact that  $\Xx^\Vv_x=\Xx^\Ww_x$ and that $\psi^\Vv_{xi}(\ldots)=\psi^\Ww_{xi}(\ldots)$ follows easily by induction on $i$ from 
the fact that \[\zeta^\Vv_x=\zeta^\Ww_x \text{ and }\psi_{x0}^\Vv(\vec{a}_0,\ldots,\vec{a}_{m'-1})=\psi_{x0}^\Ww(\vec{a}_0,\ldots,\vec{a}_{m'-1}),\]
 given how the later extenders and generators are defined by terms with the generators $\vec{a}_0,\ldots,\vec{a}_{m'-1}$ as arguments. But then since $(\Xx^P_x)'=(\Xx^P_y)'$
 and by Subclaim \ref{sclm:copying_comm}, it follows that $(\Xx^\Vv_x)'=(\Xx^\Vv_y)'$ etc.
\end{proof}

So we reach index $\bar{m}-m'$,
producing $\psi_{x,\bar{m}-m'}^P$, $\psi_{x,\bar{m}-m'}^\Vv$, etc.
But now computations like in the proof of Claim \ref{clm:not_m'=m-bar<m_and_e_x,e_y_non-small} show that $\bar{m}=m$,
and then we reach a contradiction like in the proof of Claim \ref{clm:not_m'=mbar=m}.
This proves Claim \ref{clm:not_m'<mbar}.
\end{proof}

With Claims \ref{clm:fully_small_situation_does_not_arise},
\ref{clm:not_m'=m-bar<m_and_e_x,e_y_non-small}, \ref{clm:not_m'=mbar=m}, \ref{clm:not_m'<mbar} together,
we have ruled out all cases,
and so we have reached a contradiction to our assumption that $x\leq^P_\sigma y\iff x\not\leq^Q_\sigma y$.
This completes the proof of Lemma \ref{lem:norm_invariance}.
\end{proof}

As a corollary, we easily get:

\begin{lem}\label{lem:leq_sigma_is_a_pwo}
 Let $\sigma$ be a norm description.
 Then $\leq_\sigma$ is a prewellorder on $A$.
\end{lem}
\begin{proof}
 Let us verify that $\leq^\sigma$ is wellfounded; the rest is easier. Let $\{x_n\}_{n<\om}\sub A$. Let $x\in\RR$ with $x_n\leq_T x$ for each $n$, and let $P$ be an $x$-mouse which is $M_{x_n}$-good for some $n$. The first distinction made by $\leq_\sigma$
 is that of $M_{x_i}$-goodness,
 with $x_i<_\sigma x_j$ if $P$ is $M_{x_i}$-good but not $M_{x_j}$-good. 
 So we may assume that $P$ is $M_{x_n}$-good for all $n$.
 Then by Lemma \ref{lem:norm_invariance}, $\leq^P_\sigma$
 agrees with $\leq_\sigma$ on $\{x_n\}_{n<\om}$. But then since $P$ iterable for finite length trees, it easily follows that $\leq_\sigma\rest\{x_n\}_{n<\om}$ is wellfounded, which suffices.
\end{proof}

***Definability of mouse scales corollaries (have to finish this...):

\begin{lem}
 Let $T$, $\Gamma$, etc, be as in Definition \ref{dfn:mouse_scale_exactly_reconstructing}. Then:
 \begin{enumerate}[label=--]\item each of the norms $\Phi$ of the mouse scale are in $\Gamma$,
 \item the sequence $\vec{\psi}$
 of $\Sigma_1$ formulas if $T$ is recursive then we can fold the norms into a very good sequence
 $\left<\Phi_n\right>_{n<\om}\in\Gamma$.
 \end{enumerate}
\end{lem}

\begin{dfn}\label{dfn:Delta_scale}
Let $T$ be good, exactly reconstructing, cofinal, uniformly boundedly $1$-solid.
Let $\Mmm=\Mmm_T$ and $M_x=\Mmm(x)$. 
Let $\psi$ be a nice formula
and $A=A_\psi$.
We define the \emph{nice-norm mouse scale} $\left<\Phi_n\right>_{n<\om}$ on $A$
by using all $\Phi_\varphi$ and $\Phi_\sigma$, for nice formulas $\varphi$ and nice norm descriptions $\sigma$,
folded together naturally into a very good sequence
(such that the sequence is coded with a sequence of formulas and terms which is recursive in $T$).
\end{dfn}
\begin{lem}
 Let $T$, $\psi$, etc, be as in Definition \ref{dfn:Delta_scale}. Then each norm $\Phi_n$ is in $\Gamma_{\mathscr{M},\mathrm{nice}}$.
\end{lem}

\section{Iterability of $M_\infty$}\label{sec:iterability_M_infty}

In this section we establish the key fact needed to see that our putative scale is in fact a semiscale: the limit model $M_\infty$ is iterable.
The proof will be a fairly direct elaboration on that for  sharps 
in \S\ref{sec:MS} (the variant of Martin-Solovay).

\subsection{Putative trees}
\begin{dfn}
 A \emph{putative premouse} is a structure $N$ satisfying the first-order requirements of a premouse (but $N$ might be illfounded).
 Given a putative premouse $N$, we define $N^\sq=\core_0(N)$ in the obvious manner.
 A \emph{putative squashed premouse} is a structure $L$ satisfying the first-order requirements of a squashed premouse;
 this includes the requirement that if $L$ is active then $\OR^L\in\Ult(L,F^L)$ (but $\OR^L$ might be illfounded). Given a putative premouse,
 $\nu^N$ denotes $\dot{\nu}^N$, and given a putative squashed premouse $L$,
 $\nu^L$ denotes $\OR^L$.
 
 A \emph{$3\nu$-wellfounded-premouse} is a putative premouse such that either (i) $N$ is a type 0, 1 or 2 premouse,
 or (ii) $N$ is ``type 3'' and $\nu^N$ is wellfounded (hence $\nu^N\in\wfp(N,F^N)$).
\end{dfn}
\begin{rem}
 Let $L$ be a putative squashed premouse.
 Then there is a unique putative premouse $N$ such that $L=N^\sq$; we write $N=L^\unsq$.
 We have $\nu^N=\nu^L$.
 We define \emph{$\nu$-preserving embedding} $\pi:M\to L$
 in the obvious manner.
\end{rem}

 \begin{lem}\label{lem:nu_in_wfp}
  Let $L$ be the direct limit of a directed system
 \[ \left<M^\sq,N^\sq;\pi_{MN}:M^\sq\to N^\sq\right> \]
 of squashed type 3 premice $M^\sq,N^\sq$, whose embeddings $\pi_{MN}$ are $\Sigma_0$-elementary and $\nu$-preserving.
 Then $L$ is a putative squashed premouse and
 each direct limit map $\pi_{ML}$ is $\nu$-preserving.
 Let $N=L^\unsq$. If $L$ is wellfounded then $N$ is a $3\nu$-wfd-pm.
\end{lem}
\begin{proof}
The maps $\pi_{MN}$ induce maps
\[ \psi_{MN}=\psi_{\pi_{MN}}:\Ult(M^\sq,F^M)\to\Ult(N^\sq,F^N) \]
with $\psi_{MN}(\nu^M)=\nu^N$, and these yield a directed system with direct limit $\Ult(L,F^L)$
and direct limit maps $\psi_{ML}=\psi_{\pi_{ML}}$. This easily suffices. 
\end{proof}

***Move the following definition earlier:
\begin{dfn}
 An iteration tree $\Tt$ on an $n$-sound premouse $N$ is called \emph{sse-essentially-$n$-maximal}
 iff it satisfies the usual requirements for $n$-maximality,
 except for the monotone length condition (that is, we do not demand that $\lh(E^\Tt_\alpha)\leq\lh(E^\Tt_\beta)$ for all $\alpha+1<\beta+1<\lh(\Tt)$),
 and:
 \begin{enumerate}[label=--]
  \item $\nu(E^\Tt_\alpha)\leq\nu(E^\Tt_\beta)$ for all $\alpha+1<\beta+1<\lh(\Tt)$.
  \item If $\alpha+2<\lh(\Tt)$
 and $\lh(E^\Tt_{\alpha+1})<\lh(E^\Tt_{\alpha})$ then $E^\Tt_\alpha$
 has superstrong type
 \tu{(}that is, $\nu(E^\Tt_\alpha)=\lambda(E^\Tt_\alpha)=i_{E^\Tt_\alpha}(\crit(E^\Tt_\alpha))$\tu{)}.
 \end{enumerate}
If $\Tt$ is sse-essentially-$n$-maximal and $\alpha+1<\lh(\Tt)$,
we say that $\alpha$ is \emph{$\Tt$-stable} iff $\lh(E^\Tt_\alpha)\leq\lh(E^\Tt_\beta)$ for all $\beta+1<\lh(\Tt)$ with $\alpha<\beta$.
\end{dfn}

\begin{rem}\label{rem:nu_in_wfp}
 Let $\Tt$ be a putative sse-essentially-$m$-maximal iteration tree of length $\eta+1$
 where $\eta$ is a limit.  Let
 \[ L=\dirlim_{\alpha_0\leq_\Tt\alpha\leq_\Tt\beta<_\Tt\eta}\left<\core_0(M^\Tt_\alpha),\core_0(M^\Tt_\beta);i^\Tt_{\alpha\beta}\right>,\]
 where $\alpha_0<_\Tt\eta$ is sufficiently large that $(\alpha_0,\eta]_\Tt$ does not drop
 and $i^\Tt_{\alpha\beta}$ is $\nu$-preserving for $\alpha,\beta$ such that $\alpha_0\leq\alpha<_\Tt\beta<_\Tt\eta$.
 So \ref{lem:nu_in_wfp} applies, and we (can) therefore define $M^\Tt_\eta=L^\unsq$,
 and hence also $\core_0(M^\Tt_\eta)=L$.
 We adopt this practice in the sequel, in order to define $M^\Tt_\eta$ without assuming wellfoundedness.
 \end{rem}

\begin{cor}\label{cor:nu_in_wfp_tree}
 Let $\Tt$ be a putative sse-essentially-$m$-maximal iteration tree of length $\eta+1$ where $\eta$
 is a limit. Suppose that $M^\Tt_\eta$ is type 3
 and $L=\core_0(M^\Tt_\eta)$  is wellfounded. Then $\nu^L=\OR^L\in\wfp(\Ult(L,F^L))$,
 so $M^\Tt_\eta$ is $(\nu^L+1)$-wellfounded,
 and  $i^\Tt_{\alpha\eta}$ is $\nu$-preserving for all sufficiently large $\alpha<_\Tt\eta$.
\end{cor}

\newcommand{\reps}{\mathrm{reps}}

\begin{dfn} Let $P$ be a putative premouse. Define the \emph{representations} function 
\[ \reps^P:P\to\core_0(P) \]
as follows. If $x\in\core_0(P)$, let $\reps^P(x)=\{x\}$, and for $x\in P\cut\core_0(P)$ (hence $P$ is active type 3) let
\[ \reps^P(x)=\{(a,f)\mid x=[a,f]^P_{F^P}\}.\]

Given a putative iteration tree $\Uu$ and $\beta<\lh(\Uu)$, let 
$\reps^{\Uu}_\beta=\reps^{M^\Uu_\beta}$.
\end{dfn}

 \begin{dfn}
  Let $M$ be an $m$-sound premouse.
  
  We say that $\Tt$ is a \emph{$3\nu$-putative sse-essentially-$m$-maximal tree 
on $M$}
iff $\Tt$ is a putative sse-essentially-$m$-maximal tree on $M$, and if $\lh(\Tt)=\eta+1$ 
then $\core_0(M^\Tt_\eta)$ is wellfounded.\footnote{So by \ref{rem:nu_in_wfp},
$M^\Tt_\eta$ is well-defined and $\nu(M^\Tt_\eta)\in\wfp(M^\Tt_\eta)$.}

A \emph{$3\nu$-sse-essentially $(m,\eta)$-iteration strategy}
  $\Sigma$ for $M$ is a strategy which applies to sse-essentially-$m$-maximal trees of length $<\eta$,
  every putative tree via $\Sigma$ of length $<\eta$ has wellfounded models,
  and if $\Tt$ via $\Sigma$ has length $\eta=\gamma+1$, then $\Tt$ is
$3\nu$-putative sse-essentially-$m$-maximal.

We say that $\Tt$ is a \emph{$3\nu$-puta-putative sse-essentially $m$-maximal tree}
iff either:
\begin{enumerate}[label=(\roman*)]\item $\Tt$ is $3\nu$-putative sse-essentially $m$-maximal,
or \item $\lh(\Tt)=\eta+2$ and $\Tt\rest\eta+1$ is $3\nu$-putative sse-essentially $m$-maximal,
$M^\Tt_\eta$ is type 3 (possibly illfounded),
\[ \Tt=\Tt\rest(\eta+1)\conc\left<F^{M^\Tt_\eta}\right>, \]
with $\pred^\Tt(\eta+1)$ etc determined as usual by $m$-maximality.
\end{enumerate}

We say that $\Tt$ is \emph{$m$-relevant}
 iff $\Tt$ is either putative sse-essentially $m$-maximal,
 or $3\nu$-puta-putative sse-essentially $m$-maximal.
\end{dfn}

\subsection{Finite hulls of trees}
\begin{dfn}[Finite Support]\label{dfn:finite_support}
Let $k\leq\om$ and $\Tt$ be $k$-relevant. Let $M_\alpha=M^\Tt_\alpha$ and 
$\reps_\alpha=\reps^\Tt_\alpha$ etc.

A \emph{finite selection of $\Tt$} is a finite sequence
$\FF=\left<\FF_\theta\right>_{\theta\in J}$ of non-empty finite sets $\FF_\theta$ with $J\sub\lh(\Tt)$ and 
$\FF_\theta\sub\core_0(M_\theta)$ for each $\theta\in J$. Write $I_\FF=J$.

A \emph{finite support of $\Tt$} 
is a finite selection $\supp=\left<\supp_\alpha\right>_{\alpha\in I}$ of $\Tt$
such that:
\begin{enumerate}
 \item $0\in I$.
\item\label{item:nu_in_rg} $\reps_\alpha(\nu^{M_\alpha})\inter\supp_\alpha\neq\emptyset$
for each $\alpha\in I$.
\footnote{Note that if $\nu^{M_\alpha}$ is defined and is an element of $M_\alpha$ even if $M_\alpha$ is illfounded;
cf.~\ref{rem:nu_in_wfp}.}
\item Let $\beta+1\in I$ and $E=E_\beta$ and $\gamma=\pred(\beta+1)$ and 
$m=\deg(\beta+1)$
and $M^*=M^{*}_{\beta+1}\ins M_\gamma$.
Then:
\begin{enumerate}
\item $\beta,\gamma,\gamma+1\in I$.
\item\label{item:x_output_of_term_gens} For each $x\in\supp_{\beta+1}$ there
is $(t,a,y)$ such that $t$ is an $\rSigma_m$ term,
$a\in[\nu_E]^{<\om}$,
$\reps_\beta(a)\inter\supp_\beta\neq\emptyset$, $y\in\core_0(M^*)$, $\reps_\gamma(y)\inter\supp_\gamma\neq\emptyset$, and
\[ x=[a,f_{t,y}^{M^*}]^{M^*,m}_{E}.\]
\item If $E\in M_\beta$ then $\reps_\beta(E)\inter \supp_\beta\neq\emptyset$.
\end{enumerate}
\item\label{item:support_limit_case} Let $\alpha\in I$ be a limit ordinal (so
$0\in I\inter\alpha$) and $\beta=\max(I\inter\alpha)$. Then 
$\beta<_\Tt\alpha$ and $\beta$ is a successor
and letting $\gamma=\pred^\Tt(\beta)$,
then $(\gamma,\alpha)_\Tt$ does not drop in model or degree
and $\supp_\alpha\sub i^\Tt_{\gamma\alpha}``\supp_\gamma$.
\end{enumerate}

Let $\supp,I$ be as above.
We write
\[ I'_\supp=\{\alpha\in I\mid \alpha=\max(I)\text{ or }\alpha+1\in I\}.\]
Given a finite selection $\FF$ of $\Tt$,
we say that $\supp$ is a \emph{finite support of $\Tt$ for $\FF$}
iff $I_\FF\sub I'_\supp$ and $\FF_\beta\sub\supp_\beta$ for each $\beta\in I_\FF$.
\end{dfn}

\begin{rem}
 The above definition is mostly like that of \cite[Definition 2.7]{hsstm}, with the following small 
differences. First, we use $\reps_\alpha$ in place of $\rep_\alpha$. Second, we have added condition \ref{item:nu_in_rg}, in order to ensure that we get 
$\nu$-preserving copy maps. 
Third, condition  \cite[2.7(g)]{hsstm}
has been modified because our mice can have superstrong extenders,
unlike in \cite{hsstm}, but we have  only required that $\Ttbar$ (a finite support tree) be \emph{sse-essentially}-$m$-maximal. The point  of \cite[2.7(g)(iv)]{hsstm}  was to ensure that finite support trees have the monotone length condition, which we cannot quite achieve here.
\end{rem}

\begin{lem}\label{lem:finite_support_exists}
Let $n\leq\om$ and $\Tt$ be $n$-relevant.
Let $\FF$ be a finite selection of $\Tt$. Then there is a finite support of $\Tt$ for $\FF$.
\end{lem}
\begin{proof}
This is a straightforward construction; a very similar argument is 
given for \cite[Lemma 2.8]{hsstm}.
\end{proof}

\begin{dfn}\label{dfn:pse}
Let $\Tt$ be $m$-relevant on $M$.
We say that
  \[ \Phi=(\Ttbar,\varphi,\left<\varrho_i\right>_{i<\lh(\Ttbar)})\]
  is a \emph{pre-simple tree embedding \tu{(}of $\Ttbar$\tu{)} into $\Tt$}, written
$\Phi:\Ttbar\hookrightarrow_\simple^\pre\Tt$,  iff:
 \begin{enumerate}
  \item $\Ttbar$ is an sse-essentially $m$-maximal tree on $M$, of finite length.
  \item $\varphi:\lh(\Ttbar)\to\lh(\Tt)$ and $\varphi$ preserves tree, drop and degree structure.
  More precisely:
  \begin{enumerate}
  \item $i<_\Ttbar j$ iff $\varphi(i)<_\Tt\varphi(j)$.
  \item $\varphi(0)=0$.
  \item If $\alpha+1\in(0,\varphi(j)]_\Tt\inter\dropset_\deg^\Tt$ then 
$\{\alpha,\alpha+1,\pred^\Tt(\alpha+1)\}\sub\rg(\varphi)$.
  \item $i+1\in\dropset^\Ttbar$ iff $\varphi(i+1)\in\dropset^\Tt$,
  \item $i+1\in\dropset_\deg^\Ttbar$ iff $\varphi(i+1)\in\dropset^\Tt_\deg$,
  \item $\deg^\Ttbar(i)=\deg^\Tt(\varphi(i))$.
  \end{enumerate}
  \item $\varrho_0=\id:\core_0(M)\to\core_0(M)$.
  \item\label{item:varrho_i_maps}  $\varrho_i:\core_0(M^\Ttbar_i)\to\core_0(M^\Tt_{\varphi(i)})$.
  \footnote{\label{ftn:core_0_when_illf}Cf.~\ref{dfn:...}.}
  \item\label{item:varrho_i_props} If $\varphi(i)+1<\lh(\Tt)$ then $\varrho_i$ is a 
$\nu$-preserving near $\deg^\Ttbar(i)$-embedding.\footnote{If $\Tt$
is 3$\nu$-puta-putative $m$-maximal and $\varphi(i)+2=\lh(\Tt)$
and $M^\Tt_{\varphi(i)}$ is illfounded, then ``$\nu$-preserving near $\deg^\Ttbar(i)$-embedding''
is interpreted in the obvious manner with respect to the (wellfounded) structure
$\core_0(M^\Tt_{\varphi(i)})$ (this structure is defined
as in footnote \ref{ftn:core_0_when_illf}; note that by 3$\nu$-puta-putative $m$-maximality,
$\nu(M^\Tt_{\varphi(i)})$ is in the wellfounded part of $M^\Tt_{\varphi(i)}$,
so ``$\nu$-preserving'' makes sense).}
\item $E^\Tt_{\varphi(i)}=\extcopy(\varrho_i,E^\Ttbar_i)$ for $i+1<\lh(\Ttbar)$.
\item Let $j=\pred^\Ttbar(i+1)$. Then
\[ \varphi(j)<_\Tt\varphi(i)+1\leq_\Tt\varphi(i+1)\text{ and }\varphi(j)=\pred^\Tt(\varphi(i)+1). 
\] (note that if also $i+1\in\dropset^\Ttbar_{\deg}$ then $\varphi(i)+1=\varphi(i+1)$).
  \item\label{item:pse_nu-pres_factor_it_map} $i^\Tt_{\varphi(i)+1,\varphi(i+1)}$ is $\nu$-preserving
  (note that $(\varphi(i)+1,\varphi(i+1)]_\Tt\inter\dropset^\Tt_{\deg}=\emptyset$).
\item\label{item:drop_star_matches} Let $i+1\in\dropset^\Ttbar$ and $j=\pred^\Ttbar(i+1)$.
Then
$\psi_{\varrho_j}(M^{*\Ttbar}_{i+1})=M^{*\Tt}_{\varphi(i)+1}$.
\item\label{item:varrho_agmt} If $j<k<\lh(\Ttbar)$
then $\psi_{\varrho_j}\rest\nu^\Ttbar_j\sub\varrho_k$,
and if $E^\Ttbar_j$ is type 1 or 2 then $\psi_{\varrho_j}\rest\lh(E^\Ttbar_j)\sub\varrho_k$
(and recall that $\varrho\sub\psi_\varrho$).
\item Let $j=\pred^\Ttbar(i+1)$. Then $\varrho_{i+1}=i^\Tt_{\varphi(i)+1,\varphi(i+1)}\com\varrho'$,
where $\varrho'$ is defined as in the Shift Lemma. That is, let
$n=\deg^\Ttbar(i+1)$ and $\bar{M}^*=M^{*\Ttbar}_{i+1}$ and 
$M^*=M^{*\Tt}_{i+1}$ and $\bar{E}=E^\Ttbar_i$ and $E=E^\Tt_{\varphi(i)}$.
Then for $a\in[\nu(\bar{E})]^{<\om}$ and $x\in\core_0(\bar{M}^*)$ and $t$ an $\rSigma_n$-term,
\[ \varrho'([a,f_{t,x}^{\bar{M}^*}]^{\bar{M}^*,n}_{\bar{E}})=
[\psi_{\varrho_i}(a),f_{t,\psi_{\varrho_j}(x)}^{M^*}]^{M^*,n}_{E}.\]
 \end{enumerate}
 
 We will actually verify later that conditions \ref{item:varrho_i_props},
 \ref{item:drop_star_matches} and \ref{item:varrho_agmt} follow automatically from the other conditions;
 we have stated them explicitly here for convenience.

Let $\Phi:\Ttbar\hookrightarrow_\simple^\pre\Tt$ with 
$\Phi=(\Ttbar,\varphi,\vec{\varrho}=\left<\varrho_i\right>_{i<\lh(\Ttbar)})$. Then we write 
$\varphi^\Phi=\varphi$ and $\varrho_i^\Phi=\varrho_i$.
Given $j\in(0,\lh(\Ttbar)]$,
$\Phi\rest j$ denotes $(\Ttbar\rest j,\varphi\rest j,\vec{\varrho}\rest j)$.
Let $k+1=\lh(\Ttbar)$.
We say that $\Phi$ is a \emph{simple embedding \tu{(}of $\Ttbar$\tu{)} into $\Tt$},
written $\Phi:\Ttbar\hookrightarrow_\simple\Tt$, iff $\Tt\rest(\varphi(k)+1)$ has wellfounded models and condition
\ref{item:varrho_i_props} 
holds without the hypothesis ``if $\varphi(i)+1<\lh(\Tt)$''.
\end{dfn}

\begin{lem}\label{lem:simple_pre-embedding_basic_props}
 Let $\Phi:\Ttbar\hookrightarrow_\simple^\pre\Tt$,
  let $\varphi=\varphi^\Phi$ and 
$\varrho_i=\varrho_i^\Phi$. Let $j<k<\lh(\Ttbar)=\ell+1$. Then:
 \begin{enumerate}
 \item\label{item:if_Tt_wfd} If $\Tt$ has wellfounded models then $\Phi:\Ttbar\hookrightarrow_\simple\Tt$.
\item Let $\Phi'=\Phi\rest(j+1)$ and $\Ttbar'=\Ttbar\rest(j+1)$.
 Then $\Phi':\Ttbar'\hookrightarrow_\simple^\pre\Tt$.

\item If $j<_\Ttbar k$ and $(j,k]_\Ttbar\inter\dropset^\Ttbar=\emptyset$ then 
$\varrho_k\com 
i^\Ttbar_{j,k}=i^\Tt_{\varphi(j),\varphi(k)}\com\varrho_j$.

\item\label{item:drop_comm} If $i+1\in\dropset^\Ttbar$ and $j=\pred^\Ttbar(i+1)$ then
\[ \varrho_{i+1}\com 
i^{*\Ttbar}_{i+1}=i^{\Tt}_{\varphi(i)+1,\varphi(i+1)}\com i^{*\Tt}_{\varphi(i)+1}\com\psi_{\varrho_j}\rest\core_0(M^{*\Ttbar}_{i+1}).\]
 \item\label{item:stability_reflects} If $j$ is $\Ttbar$-stable then $\varphi(j)$ is $\Tt\rest(\varphi(\ell)+1)$-stable\footnote{***Where is this defined? It means that $\lh(E^{\bar{\Tt}}_j)<\lh(E^{\bar{\Tt}}_k)$ for all $k>j$.}
 \tu{(}however, the converse can fail\tu{)}. Hence if $j+2<\lh(\Ttbar)$
 and $j$ is $\Ttbar$-stable then $\varphi(j)$ is $\Tt$-stable.
 \item\label{item:lh_pres} If $E^\Ttbar_j$ is type 1 or 2  then:
 \begin{enumerate}[label=--]
\item $\psi_{\varrho_k}\rest\lh(E^\Ttbar_j)=\varrho_k\rest\lh(E^\Ttbar_j)=\psi_{\varrho_j}\rest\lh(E^\Ttbar_j)$, and
 \item $\psi_{\varrho_k}(\lh(E^\Ttbar_j))=\varrho_k(\lh(E^\Ttbar_j))=\psi_{\varrho_j}(\lh(E^\Ttbar_j))=\lh(E^\Tt_{\varphi(j)})$.
 \end{enumerate}
 \item\label{item:agmt_below_at_nu} We have:
 \begin{enumerate}[label=(\roman*)]
  \item 
$\psi_{\varrho_k}\rest\nu^\Ttbar_j=\varrho_k\rest\nu^\Ttbar_j=\psi_{\varrho_j}\rest\nu^\Ttbar_j$,

 \item $\psi_{\varrho_k}(\alpha)=\varrho_k(\alpha)\geq\psi_{\varrho_j}(\alpha)$ for each $\alpha<\lh(E^\Ttbar_j)$,
   \item $\psi_{\varrho_j}(\nu^\Ttbar_j)=\nu^\Tt_{\varphi(j)}$,
 \item\label{item:inner_lh_pres} either:
 \begin{enumerate}[label=--]
 \item $\lh(E^{\Ttbar}_j)\in\dom(\varrho_k)$
 and $\varrho_k(\lh(E^{\Ttbar}_j))\geq\psi_{\varrho_j}(\lh(E^{\Ttbar}_j))=\lh(E^\Tt_{\varphi(j)})$, or
 \item $k=j+1$, $E^{\Ttbar}_j$ is of superstrong type, $\OR(M^{\Ttbar}_k)=\lh(E^{\Ttbar}_j)$, $M^{\Ttbar}_k$ is active type 2, and $\lh(E^{\Tt}_{\varphi(j)})\leq\OR(M^\Tt_{\varphi(k)})$.
 \end{enumerate}

 \end{enumerate}
 
\end{enumerate}
\end{lem}
\begin{proof}
Parts \ref{item:if_Tt_wfd}--\ref{item:drop_comm} mostly follow easily from the definitions.
However, regarding part \ref{item:if_Tt_wfd} and condition \ref{item:varrho_i_props} of \ref{dfn:pse}  in case $\varphi(i)+1=\lh(\Tt)$:
\begin{enumerate}[label=--]
 \item to see that $\varrho_i$ is $\rSigma_{\deg^{\Ttbar}_i+1}$-elementary, argue essentially as in \cite{fs_tame};
 \item to see that $\varrho_i$ is $\nu$-preserving, argue as in \S\ref{sec:nu-pres};
 in particular see Lemmas \ref{lem:nu-pres_it_map} and \ref{lem:nu-pres_copy_map}.
\end{enumerate}
Parts \ref{item:lh_pres}
and \ref{item:agmt_below_at_nu}
are quite straightforward.
Let us consider part
\ref{item:stability_reflects}. We may assume that $\lh(\Tt)=\varphi(\ell)+1$.
We have $j<\ell$. 
Suppose $j$ is $\bar{\Tt}$-stable but $\varphi(j)$ is not $\Tt$-stable.
So
\begin{enumerate}[label=--]
 \item if $j+1<\ell$ then $\lh(E^{\bar{\Tt}}_{j})\leq\lh(E^{\bar{\Tt}}_{j+1})$, but
 \item $\varphi(j)+1<\varphi(\ell)$ and $\lh(E^\Tt_{\varphi(j)+1})<\lh(E^\Tt_{\varphi(j)})$.
 \end{enumerate}

Suppose $j+1=\ell$. So $\varphi(j)+1<^\Tt\varphi(j+1)=\varphi(\ell)$ and
\[ (\varphi(j)+1,\varphi(j+1)]_\Tt\inter\dropset^\Tt=\emptyset.\]
But since $\varphi(j)$ is non-$\Tt$-stable, 
\[ \nu=\nu(E^\Tt_{\varphi(j)})=\lambda(E^\Tt_{\varphi(j)})=\lgcd(\exit^\Tt_{\varphi(j)})<\lh(E^\Tt_{\varphi(j)+1})<\lh(E^\Tt_{\varphi(j)}), \]
and since $\varphi(j)+1<_\Tt\varphi(j+1)$, we have $\nu\leq\crit(i^\Tt_{\varphi(j)+1,\varphi(j+1)})$,
but then easily $\Tt$ drops at $\successor^\Tt(\varphi(j)+1,\varphi(j+1))$, contradiction.

Now suppose $j+1<\ell$, so $\lh(E^{\bar{\Tt}}_j)\leq\lh(E^{\bar{\Tt}}_{j+1})$
but $\lh(E^\Tt_{\varphi(j)+1})<\lh(E^\Tt_{\varphi(j)})$. 
So by part \ref{item:agmt_below_at_nu}\ref{item:inner_lh_pres} for $k=j+1$, clearly $\varphi(j)+1<^\Tt\varphi(j+1)$.
But now we reach a contradiction like the previous case.
\end{proof}

\begin{lem}\label{lem:support_determines_pse}
 Let $M$ be $m$-sound and $\Tt$ be $m$-relevant on $M$. Then:
\begin{enumerate}
\item\label{item:I_determines_Phi} Let $I'\sub\lh(\Tt)$ be finite. Then there is at most one pair 
$(\Ttbar,\Phi)$ such that 
$\Phi:\Ttbar\hookrightarrow_\simple^\pre\Tt$ and $\rg(\varphi^\Phi)=I'$.
\item\label{item:Phi_existence}
 Let $\supp=\left<\supp_\alpha\right>_{\alpha\in I}$ be a finite support of $\Tt$.
 Then there is $(\Ttbar,\Phi)$ such that 
$\Phi:\Ttbar\hookrightarrow_\simple^\pre\Tt$
and $\rg(\varphi^\Phi)=I'_\supp$ \tu{(}recall that $I'_\supp\psub I_\supp$ if $I_\supp$ includes a limit ordinal\tu{)}. Moreover, 
$\supp_{\varphi^\Phi(i)}\sub\rg(\varrho_i)$ for each $i<\lh(\Ttbar)$.
\end{enumerate}
\end{lem}
\begin{proof}[Proof Sketch]
Part \ref{item:I_determines_Phi} is easy. For part \ref{item:Phi_existence}, most of the proof is 
as in 
\cite[Lemma 2.9]{hsstm}. (As in \cite{hsstm}, the fact that $\varrho_i$ is a near $\deg^{\Ttbar}(i)$-embedding (not just a weak $\deg^{\Ttbar}(i)$-embedding) is proved as in \cite{fs_tame};
we may assume that $\bar{\Tt},\Tt$ have the increasing length condition for this (that is, $\lh(E^{\Ttbar}_\alpha)\leq\lh(E^{\Ttbar}_\beta)$ for all $\alpha+1<\beta+1<\lh(\Ttbar)$, and likewise for $\Tt$), which
is the case in \cite{fs_tame},
by restricting our attention to stable nodes and applying Lemma \ref{lem:simple_pre-embedding_basic_props} part \ref{item:stability_reflects}. We
do  have one requirement
which was not present in \cite{hsstm}: that $\varrho_i$ be 
$\nu$-preserving. But note that by condition \ref{item:nu_in_rg} of \ref{dfn:finite_support}, 
$\varrho_i$ is not $\nu$-high, and by the $\rSigma_1$-elementarity of $\varrho_i$, $\varrho_i$ is 
not $\nu$-low.

The fact that $\Ttbar$ is sse-essentially-$m$-maximal is verified as follows.
Let $\alpha\in I=I_\supp$ be a limit ordinal
and $\beta+1=\max(I\inter\alpha)$ and $\gamma=\pred^\Tt(\beta+1)$.
So $(\gamma,\alpha)_\Tt$ does not drop in model or degree
and $\supp_\alpha\sub i^\Tt_{\gamma\alpha}``\supp_\gamma$.
We have $\gamma,\beta,\alpha\in\rg(\varphi^\Phi)$
and $\beta=\max(\rg(\varphi^\Phi)\inter\alpha)$.
Write $\varphi^\Phi(\bar{\gamma})=\gamma$ etc.
 So $\bar{\beta}+1=\bar{\alpha}$.
Suppose $\bar{\alpha}+1<\lh(\Ttbar)$ and $\lh(E^\Ttbar_{\bar{\alpha}})<\lh(E^\Ttbar_{\bar{\beta}})$.
Let $\sigma=(i^\Tt_{\beta+1,\alpha})^{-1}\com\varrho^\Phi_{\bar{\alpha}}$,
so $\sigma:M^\Ttbar_{\bar{\alpha}}\to M^\Tt_{\beta+1}$ results from the Shift Lemma.
Since $\varrho^\Phi_{\bar{\beta}}$ is $\nu$-preserving, therefore
$\lh(\sigma(E^\Ttbar_{\bar{\alpha}}))<\lh(E^\Tt_{\beta})$.
But $E^\Tt_\alpha\in\supp_\alpha\sub i^\Tt_{\gamma\alpha}``\supp_\gamma$,
so $\sigma(E^\Ttbar_{\bar{\alpha}})\in\rg(i^\Tt_{\gamma,\beta+1})$.
Since $i^\Tt_{\beta+1,\alpha}(\sigma(E^\Ttbar_{\bar{\alpha}}))=E^\Tt_\alpha$,
it follows that $\lambda(E^\Tt_{\beta})<\lh(\sigma(E^\Ttbar_{\bar{\alpha}}))$,
 so $E^\Tt_{\beta}$ is of superstrong type.
But $\varrho^\Phi_{\bar{\beta}}$ is $\nu$-preserving,
so $E^\Ttbar_{\bar{\beta}}$ is also superstrong, which easily suffices.
\end{proof}

The following lemma is a straightforward matter of diagram chasing, which we leave to the reader:
\begin{lem}\label{lem:factor_pse}
Let $M$ be $m$-sound and $\Tt$ be sse-essentially $m$-maximal on $M$.
 Let $\Phi:\Ss\hookrightarrow_\simple\Tt$ and $J=\rg(\varphi^\Phi)$.
 Let $I\sub J$ and $\bar{I}=(\varphi^\Phi)^{-1}(I)$.
Then
 \[ \ex\Rr,\Psi[\Psi:\Rr\hookrightarrow_\simple\Tt\text{ and }\rg(\varphi^\Psi)=I]\]
 \[ \iff\ex\Rr',\bar{\Psi}[\bar{\Psi}:\Rr'\hookrightarrow_\simple\Ss\text{ and 
}\rg(\varphi^{\bar{\Psi}})=\bar{I}].\] For such $\Rr,\Psi,\Rr',\bar{\Psi}$,
we have $\Rr=\Rr'$, $\varphi^{\Psi}=\varphi^{\Phi}\com\varphi^{\bar{\Psi}}$
and $\varrho^{\Psi}_i=\varrho^{\Phi}_{\varphi^{\bar{\Psi}}(i)}\com\varrho_i^{\bar{\Psi}}$.
\end{lem}

\begin{lem}\label{lem:simple_emb_iterate_end}
 Let $\Phi:\Ttbar\hookrightarrow_\simple^\pre\Tt$ and
$\varphi=\varphi^\Phi$. Suppose $\lh(\Ttbar)=k+2$ and let 
$\eta=\varphi(k+1)$.
Let $\beta<\lh(\Tt)$ be such that either
\begin{enumerate}[label=\tu{(}\roman*\tu{)}]
 \item\label{item:option_i}
$\eta\leq_\Tt\beta$ and $(\eta,\beta]_\Tt\inter\dropset_\deg^\Tt=\emptyset$
and $i^\Tt_{\eta\beta}$ is $\nu$-preserving, or
\item\label{item:option_ii}
$\varphi(k)+1\leq_\Tt\beta\leq_\Tt\eta$.
\end{enumerate}
Then there is \tu{(}a unique\tu{)} $\Phi'$ such that 
\[ \Phi':\Ttbar\hookrightarrow^\pre_\simple\Tt\text{ and }\rg(\varphi^{\Phi'})=\rg(\varphi)\un\{\beta\}\cut\{\eta\}.\]
Moreover, $\Phi'\rest(k+1)=\Phi\rest(k+1)$,
\[ \text{if 
}\beta\leq_\Tt\eta\text{ then 
}\varrho_{k+1}^{\Phi'}=(i^\Tt_{\beta\eta})^{-1}\com\varrho_{k+1}^\Phi, \]
\[ 
\text{if }\eta\leq_\Tt\beta\text{ 
then }\varrho_{k+1}^{\Phi'}=i^\Tt_{\eta\beta}\com\varrho_{k+1}^\Phi.\]
\end{lem}
\begin{proof}
We must have $\Phi'\rest(k+1)=\Phi\rest(k+1)$ by uniqueness.
Define $\varrho=\varrho_{k+1}^{\Phi'}$ as above.
We just verify that $\varrho$ is $\nu$-preserving,
assuming that $\beta+1<\lh(\Tt)$
(as needed for condition \ref{item:varrho_i_props}); the rest is clear.
If \ref{item:option_i} holds this is because both $\varrho_{k+1}^{\Phi}$ and $i^\Tt_{\eta\beta}$ are 
$\nu$-preserving. Suppose \ref{item:option_ii} holds.
We may assume $\beta<\eta$. Then $\Tt\rest(\beta+1)$ has wellfounded models,
and note that $\Phi':\Ttbar\hookrightarrow^\pre_\simple\Tt\rest(\beta+1)$.
So by Lemma \ref{lem:simple_pre-embedding_basic_props},
$\Phi':\Ttbar\hookrightarrow_\simple\Tt\rest(\beta+1)$,
so the conclusion of condition \ref{item:varrho_i_props}
holds also for $i=k+1$.
\end{proof}

\begin{dfn}
 Let $\Tt$ be an $m$-relevant tree on an $m$-sound premouse $M$.
 Given $\alpha+1<\lh(\Tt)$, let
 $\redd^\Tt_\alpha=\redd^{M^\Tt_\alpha}(\exit^\Tt_\alpha)$,
 and if $\lh(\Tt)=\alpha+1$ let $\redd^\Tt_{\alpha}=\left<M^\Tt_{\alpha}\right>$.
Write $M^\Tt_{\alpha j}=M_j$ and $k^\Tt_\alpha=k$ where $\left<M_i\right>_{i\leq k}=\redd^\Tt_\alpha$.

Write $\liftdom^\Tt=\{(\alpha,i)\mid\alpha<\lh(\Tt)\text{ and }i\leq k^\Tt_\alpha\}$;
as usual we order $\liftdom^\Tt$ by $<$ lexicographically. We also define the iteration tree order $<^\Tt_{\lifttree}$ on
$\liftdom^\Tt$ as follows:
\begin{enumerate}[label=--]
 \item The root of $<^\Tt_\lifttree$ is $(0,0)$,
 \item  $\pred^\Tt_\lifttree(\alpha,i+1)=(\alpha,i)$,
 \item $\pred^\Tt_\lifttree(\alpha+1,0)=(\pred^\Tt(\alpha+1),j)$ where $M^{*\Tt}_{\alpha+1}=M^\Tt_{\alpha j}$,
 \item if $\alpha$ is a limit and $\beta<_\Tt\alpha$ and $(\beta,\alpha]_\Tt$ does not drop in model
 then:
 \begin{enumerate}[label=--]\item $(\beta,0)<^\Tt_\lifttree(\alpha,0)$, and \item for each $\gamma,j$,
 \[ (\beta,0)<^\Tt_\lifttree(\gamma,j)<^\Tt_\lifttree(\alpha,0)\text{ iff }\beta<^\Tt\gamma<^\Tt\alpha\text{ and }j=0.\qedhere\]
 \end{enumerate}
\end{enumerate}
\end{dfn}

\begin{dfn}\label{dfn:(Phi,Psi)} Let $\Tt$ be $m$-relevant on $M$
and let $\Uu$ be $n$-maximal on $Y$ with ${<_\Uu}={<^\Tt_{\lifttree}}$.
Then we write
\[ (\Phi,\Psi):(\Ttbar,\Uubar)\hookrightarrow_\simple^\pre(\Tt,\Uu) \]
iff 
 $\Phi:\Ttbar\hookrightarrow^\pre_\simple\Tt$ and $\Psi:\Uubar\hookrightarrow^\pre_\simple\Uu$
and
\[ \rg(\varphi^\Psi)=\{(\varphi^\Phi(\alpha),j)\mid
\alpha<\lh(\Ttbar)\text{ and }j\leq k^\Ttbar_\alpha\}.\qedhere \]
\end{dfn}

Note that for each $\alpha+1<\lh(\Ttbar)$
we have $k^\Ttbar_\alpha=k^\Tt_{\varphi(\alpha)}$, and  if $\alpha+1=\lh(\Ttbar)$ then $k^\Ttbar_\alpha=0$.
Given $(\Phi,\Psi)$ as above, note that $\varphi^{\Psi}(\alpha,j)=(\varphi^\Phi(\alpha),j)$.

\begin{lem}\label{lem:simultaneous_pse}
Let $\Tt$ be $m$-relevant on $M$
and let $\Uu$ be $n$-maximal on $Y$ with ${<_\Uu}={<^\Tt_{\lifttree}}$.
Let $\FF$ be a finite selection of $\Tt$ and $\theta=\max(I_\FF)$. Then there are $\Ttbar,\Uubar,\Phi,\Psi$ such that
\[ (\Phi,\Psi):(\Ttbar,\Uubar)\hookrightarrow_\simple^\pre(\Tt,\Uu) \]
with $I_\FF\sub\rg(\varphi^\Phi)$ and $\theta=\max(\rg(\varphi^\Phi))$ and $\FF_{\varphi^\Phi(\alpha)}\sub\rg(\varrho^\Phi_\alpha)$
for $\alpha<\lh(\Ttbar)$.
\end{lem}
\begin{proof}
 This is an easy modification of the proof of \ref{lem:finite_support_exists}.
 One constructs finite supports $\supp^\Tt,\supp^\Uu$ of $\Tt,\Uu$ simultaneously,
 starting with maximum indices $\theta$ and $(\theta,0)$, such that when we let $\Phi,\Psi$ be the pre-simple embeddings determined by $\supp^\Tt,\supp^\Uu$
 (using \ref{lem:support_determines_pse}), then $\Phi,\Psi$ satisfy the required conditions.
 Note here that if $\alpha$ is a limit and $\alpha\in I_{\supp^\Tt}$
 and $\beta=\max(\alpha\inter I_{\supp^\Tt})$, and hence $\beta\notin\rg(\varphi^\Phi)$,
 then $(\beta,0)$ will be the largest element of $I_{\supp^\Uu}$
 below $(\alpha,0)$ (so just in this situation, when forming the supports, we do not put $(\beta,j)\in\supp^\Uu$
 for $0<j\leq k^\Tt_\beta$).
\end{proof}

\subsection{Abstract lifting algorithms}
Recall that if $\varrho:M\to N$
is a weak $0$-embedding and $E\in\es_+^M$,
$\extcopy^\varrho(E)=\extcopy(\varrho,E)$
was introduced in Definition \ref{dfn:copyseg};
if $E=F^M$ then this is just $F^N$,
and in the special case that $\varrho$
is $\nu$-preserving and $E\in\es^M$, it is just $\psi_\varrho(F^M)$.

\begin{dfn}\label{dfn:extcopy}
Let $M,N$ be $m$-sound and $\Tt$ be an $m$-maximal tree on $M$ of successor length
 and $E\in\es_+(M^\Tt_\infty)$. Let $\varrho:M\to N$ be a $\nu$-preserving near $m$-embedding.
 Then $\extcopy(\varrho,\Tt,E)$ denotes the pair $(\Uu,F)$
 where $\Uu$ is the $m$-maximal tree given by copying $\Tt$ with $\varrho$,
 and $F=\extcopy(\pi_\infty,E)$ where $\pi_\infty:M^\Tt_\infty\to M^\Uu_\infty$ is the final copy map,
 if these things exist; otherwise $\extcopy(\varrho,\Tt,E)$ is undefined.\footnote{By Lemma \ref{lem:nu-pres_copy_map},
 each copy map $\pi_\alpha:M^\Tt_\alpha\to M^\Uu_\alpha$ is a $\nu$-preserving near $\deg^\Tt(\alpha)$-embedding.
 So $E^\Uu_\alpha=\extcopy(\pi_\alpha,E^\Tt_\alpha)$
 for each $\alpha+1<\lh(\Tt)$.}
\end{dfn}

\begin{dfn}
Given a $\nu$-preserving map $\pi:M\to N$ between premice $M,N$, we write $\pi(\rho_0^M)=\rho_0^N$.
\end{dfn}

\begin{dfn}\label{dfn:abstract_lifting}
 Let $M$ be an $m$-sound $(m,\om)$-iterable premouse, $N$ an $n$-sound $(n,\om)$-iterable premouse, 
and $\lambda\geq\om$ be a limit ordinal or a successor thereof. An \emph{abstract $(m,n,\lambda)$-lifting strategy} 
for $(M,N)$ is a pair $(\Sigma,\frakL)$ with the following properties:
\begin{enumerate}
 \item\label{item:Sigma_strategy} $\Sigma$ is an sse-essentially-$(m,\lambda)$ strategy for 
$M$.
\end{enumerate}
Let $X$ be the set of all 3$\nu$-puta-putative trees $\Tt$ according to $\Sigma$ (of length 
$\leq\lambda$).
\begin{enumerate}[resume*]
\item\label{item:frakL_consists_of}
$\frakL=\left<\Uu^\Tt,\Psi^\Tt\right>_{\Tt\in X}$.
\end{enumerate}
Let $\Tt\in X$. Write $\Uu=\Uu^\Tt$ and $\Psi=\Psi^\Tt$.
\begin{enumerate}[resume*]
\item\label{item:Uu_is_lift_tree} $\Uu$ is an sse-essentially-$n$-maximal tree on $N$
with
$\dom(\Uu)=\liftdom^\Tt$ 
and ${<^{\Uu}}={<^\Tt_\lifttree}$.
 \item $\Psi=\left<\zeta^\Tt_\alpha,\wtpi^\Tt_\alpha,\eta^\Tt_\alpha,\abres^\Tt_\alpha\right>_{\alpha<\lh(\Tt)}$.
 \end{enumerate}
For $\alpha<\lh(\Tt)$ let $\zeta_\alpha=\zeta^\Tt_\alpha$, 
$\wtpi_\alpha=\wtpi^\Tt_\alpha$, $\eta_\alpha=\eta^\Tt_\alpha$, $\abres_\alpha=\abres^\Tt_\alpha$,
and $k_\alpha=k^\Tt_\alpha$.
\begin{enumerate}[resume]
\item Let $\alpha<\beta<\lh(\Tt)$ and $\Ss=\Tt\rest\beta$. Then  $\Psi^\Ss=\Psi\rest\beta$;
if $\beta$ is a limit then
$\Uu^\Ss=\Uu\rest(\beta,0)$, and otherwise $\Uu^\Ss=\Uu\rest(\beta-1,1)$.
\item $\eta_\alpha\leq\rho_0(M^\Uu_{\alpha 0})$ and either $\zeta_\alpha<\rho_0(M^{\Uu}_{\alpha 0})$ or $\zeta_\alpha=\OR(M^\Uu_{\alpha 0})$.
  \item\label{item:pi_order_pres} $\wtpi_\alpha:\rho_0(M^\Tt_\alpha)\to\eta_\alpha$ is order 
preserving.\footnote{So $\dom(\wtpi_\alpha)$ is an ordinal.
For this reason we write $\wtpi_\alpha$ instead of $\pi_\alpha$;
in one key case, $\wtpi_\alpha=\pi_\alpha\rest\OR$
where $\pi_\alpha$ is a lifting map resulting from the iterability proof \ref{lem:iterability}.}
  \item\label{item:sigma_props} Let $\alpha<\lh(\Tt)$ where $\alpha+\om\leq\lambda$. Then $\abres_\alpha$ is a function with domain $\es_+(M^\Tt_\alpha)$ and
  for each $E\in\es_+(M^\Tt_\alpha)$, we have $\abres_\alpha(E)=(\Vv,F)$ where:\footnote{$\abres$
  abbreviates \emph{abstract resurrection}.}
  \begin{enumerate}
   \item $\Vv$ is a finite length padded linear $\deg^\Uu(\alpha,0)$-maximal
tree on $M^\Uu_{\alpha 0}$ which does not drop in model or degree,
  \item $F\in\es_+(M^\Vv_\infty)$ is $\Vv$-extending, and
 \item if $E$ is sse-essentially-$\Tt\rest(\alpha+1)$-extending, then:
 \begin{enumerate}[label=--]
 \item $\Vv$ is above $\nu(E^\Uu_{(\beta,j)})$
for each $(\beta,j)<(\alpha,0)$ and
\item $\Vv\conc\left<F\right>$ is sse-essentially-$\Uu\rest(\alpha,1)$-extending.\footnote{$\Vv$ might be trivial.}
\end{enumerate}
  \end{enumerate}
\item\label{item:Uu_exts_are_lifts} $(\Uu\rest[(\alpha,0),(\alpha,k_\alpha)],E^\Uu_{\alpha k_\alpha})=\abres_\alpha(E^\Tt_\alpha)$ for $\alpha+1<\lh(\Tt)$.
\item\label{item:zeta_movement}Let $\alpha+1<\lh(\Tt)$ and 
$\beta=\pred^\Tt(\alpha+1)$ and $(\beta,j)=\pred^\Uu(\alpha+1,0)$. Then:
\begin{enumerate}
\item\label{item:zeta_semi-pres} $(\alpha+1,0)\in\dropset^\Uu$ or $\zeta_{\alpha+1}\leq 
i^\Uu_{(\beta,0),(\alpha+1,0)}(\zeta_\beta)$.
\item\label{item:if_drop_in_Tt_then_drop_options_in_Uu} If $\alpha+1\in\dropset^\Tt$ then either $(\alpha+1,0)\in\dropset^\Uu$ or  
$\zeta_{\alpha+1}<i^\Uu_{(\beta,0),(\alpha+1,0)}(\zeta_\beta)$.
\item\label{item:eventual_commutativity} If $\alpha+1\notin\dropset^\Tt$
and $(\alpha+1,0)\notin\dropset^\Uu$ and 
$\deg^\Tt(\alpha+1)=\deg^\Tt(\beta)$ and 
$\deg^\Uu(\alpha+1,0)=\deg^\Uu(\beta,0)$ and
$i^\Tt_{\beta,\alpha+1}$ and $i^\Uu_{(\beta,0),(\alpha+1,0)}$ are $\nu$-preserving,
then:
\begin{enumerate}[label=--]
\item $\eta_{\alpha+1}\leq i^\Uu_{(\beta,0),(\alpha+1,0)}(\eta_\beta)$
(note $\zeta_{\alpha+1}\leq i^\Uu_{(\beta,0),(\alpha+1,0)}(\zeta_\beta)$ by \ref{item:zeta_semi-pres}),
\item $\wtpi_{\alpha+1}\com i^\Tt_{\beta,\alpha+1}=i^\Uu_{(\beta,0),(\alpha+1,0)}\com\wtpi_\beta$,
\item $\all E\in\es_+(M^\Tt_\beta)\left[
\abres_{\alpha+1}(\extcopy(i^\Tt_{\beta,\alpha+1},E))=
\extcopy(i^\Uu_{(\beta,0),(\alpha+1,0)},\abres_\beta(E))\right]$ (note $\abres_\beta(E)$ is of the form $(\Vv,F)$; recall
that $\extcopy(i,\Vv,F)$ was defined in \ref{dfn:extcopy}).
\end{enumerate}\end{enumerate}
\item\label{item:finite_support_commutativity} Let
$(\Phi,\Psi):(\Ttbar,\Uubar)\hookrightarrow_\simple^\pre(\Tt,\Uu)$ and $\varphi=\varphi^\Phi$.
Then $\Uubar=\Uu^\Ttbar$ and for 
each $\alpha<\lh(\Ttbar)$:
\begin{enumerate}[label=--]
\item $\zeta_{\varphi(\alpha)}=\varrho_{\alpha 0}^{\Psi}(\zeta_\alpha^\Ttbar)\text{ and }\eta_{\varphi(\alpha)}=\varrho_{\alpha 0}^\Psi(\eta_\alpha^\Ttbar)$,
\item $\wtpi_{\varphi(\alpha)}\com\varrho_{\alpha}^{\Phi}=\varrho^{\Psi}_{\alpha 0}\com\wtpi^\Ttbar_\alpha$,
\item  $\all E\in\es_+(M^\Ttbar_\alpha)\left[ 
\abres_{\varphi(\alpha)}(\extcopy(\varrho^{\Phi}_\alpha,E))=
\extcopy(\varrho^{\Psi}_{\alpha 0},\abres^\Ttbar_\alpha(E))\right]$.
\end{enumerate}
\end{enumerate} 
 Let $(\Sigma,\frakL)$ be an abstract $(m,n,\lambda)$-lifting strategy for $(M,N)$,
 and $\Gamma$ an sse-essentially-$(n,\lambda)$-strategy for $N$. We say that $(\Sigma,\frakL)$ is 
\emph{induced by $\Gamma$} iff $\Uu^\Tt$ is according to $\Gamma$ for every $\Tt$ according to 
$\Sigma$.
 \end{dfn}

\begin{dfn}***Do I use this?
 $\nu$-essentially? I think it means that we can have 3$\nu$-puta-putative trees or something like this?
\end{dfn}

\begin{dfn}***Do I use this?
We also define \emph{$\nu$-abstract $(m,n,\lambda)$-lifting strategy}
by replacing conditions \ref{item:Sigma_strategy} and \ref{item:Uu_is_lift_tree} with the weaker 
requirements that $\Sigma$ be a $\nu$-essentially-$(m,\lambda)$-strategy, and $\Uu$ be a 
$\nu$-essentially-$n$-maximal tree. The notion \emph{$(\Sigma,\frakL)$ is induced by $\Gamma$}
is then defined just as before, for a $\nu$-abstract $(m,n,\lambda)$-lifting strategy
and a $\nu$-essentially-$(n,\lambda)$-strategy $\Gamma$.
\end{dfn}

\begin{lem}\label{lem:abstract_lifting_at_limit}
  Let $M$ be $m$-sound, $N$ be $n$-sound and $(\Sigma,\frakL)$
be an abstract $(m,n,\lambda)$-lifting strategy for $(M,N)$.
Fix a limit $\eta<\lambda$.
Let $\Tt$ be of length $>\eta$, according to $\Sigma$,
and $\Uu=\Uu^\Tt$.
Then for all sufficiently large $\beta<\eta$ with $(\beta,0)<_\Uu(\eta,0)$,
we have:
\begin{enumerate}
 \item\label{item:beta<_Tt_eta} $\beta<_\Tt\eta$ and $[(\beta,0),(\eta,0))_\Uu=\{(\gamma,0)\mid\gamma\in[\beta,\eta)_\Tt\}$,
 \item $\Tt$ does not drop in model or degree in $[\beta,\eta)_\Tt$,
 \item\label{item:Uu_non-dropping} $\Uu$ does not drop in model or degree in $[(\beta,0),(\eta,0))_\Uu$,
 \item\label{item:nu-pres_it_maps} $i^\Tt_{\beta\eta}$ and $i^\Uu_{(\beta,0),(\eta,0)}$ are $\nu$-preserving,
 \item\label{item:zeta,eta_pres} $i^\Uu_{(\beta,0),(\eta,0)}(\zeta_\beta,\eta_\beta)=(\zeta_\eta,\eta_\eta)$,
 \item $\wtpi_\eta\com i^\Tt_{\beta\eta}=i^\Uu_{(\beta,0),(\eta,0)}\com\wtpi_\beta$,
 \item\label{item:E-lift_pres} $\all E\in\es_+(M^\Tt_\beta)\left[
\abres_{\eta}(\extcopy(i^\Tt_{\beta\eta},E))=
\extcopy(i^\Uu_{(\beta,0),(\eta,0)},\abres_\beta(E))\right]$.
\end{enumerate}
\end{lem}
\begin{proof}
By induction on $\eta$.
Observe that there is $\beta$ such that $(\beta,0)<_\Uu(\eta,0)$,
properties \ref{item:beta<_Tt_eta}--\ref{item:nu-pres_it_maps} hold,
and for each $\gamma\in[\beta,\eta)_\Tt$,
properties \ref{item:zeta,eta_pres}--\ref{item:E-lift_pres} hold
with $\eta$ replaced by $\gamma$.
This follows easily from the induction hypothesis, condition \ref{dfn:abstract_lifting}(\ref{item:zeta_movement}),
that $[(0,0),(\eta,0))_\Uu\inter\dropset^\Tt$ is finite and $M^\Uu_{\eta 0}$ is wellfounded, and that
given a non-overlapping iteration tree $\Vv$ and a limit $\alpha<\lh(\Vv)$ and $\beta<_\Vv\alpha$
such that $(\beta,\alpha)_\Vv$ does not drop in model, there is at most one $\gamma$
such that $\gamma+1\in(\beta,\alpha)_\Vv$ and $i^{*\Vv}_{\gamma+1}$ is non-$\nu$-preserving (this is essentially by Lemma \ref{lem:nu-pres_it_map}).

Fix $\beta_0$ as above. Let $\beta_0\leq_\Tt\beta<_\Tt\eta$.
We claim that the statements of the lemma hold for $\beta,\eta$, which suffices.
For let $\xi<\rho_0(M^\Tt_\beta)$ and $E\in\es_+(M^\Tt_\beta)$.
It suffices to see that:
\begin{enumerate}[label=\tu{(}\alph*\tu{)}]
 \item\label{item:zeta,eta_pres_in_proof} $i^\Uu_{(\beta,0),(\eta,0)}(\zeta_\beta,\eta_\beta)=(\zeta_\eta,\eta_\eta)$,
 \item\label{item:pi_it_comm} $\wtpi_\eta(i^\Tt_{\beta\eta}(\xi))=i^\Uu_{(\beta,0),(\eta,0)}(\wtpi_\beta(\xi))$,
 \item\label{item:sigma_it_comm} $\abres_\eta(\extcopy(i^\Tt_{\beta\eta},E))=\extcopy(i^\Uu_{(\beta,0),(\eta,0)},\abres_\beta(E))$.
\end{enumerate}

Let $\Phi:\Ttbar\hookrightarrow_\simple^\pre\Tt$
and $\Psi:\Uubar\hookrightarrow^\pre_\simple\Uu$
be as in \ref{lem:simultaneous_pse} with $\eta=\max(\rg(\varphi^\Phi))$
and $\beta\in\rg(\varphi^\Phi)$ and letting $\varphi^\Phi(\bar{\beta})=\beta$, with
$\xi\in\rg(\varrho^\Phi_{\bar{\beta}})$ and if $E\neq F(M^\Tt_\beta)$,
also with $\reps^{M^\Tt_\beta}(E)\inter\rg(\varrho^\Phi_{\bar{\beta}})\neq\emptyset$.
We may assume that $\beta=\max(\rg(\varphi^\Phi)\inter[0,\eta)_\Tt)$;
this is by the choice of $\beta_0$.
So $\delta+1=\successor^\Tt(\beta,\eta)$ where $\delta=\max(\rg(\varphi^\Phi)\inter\eta)$.

Using \ref{lem:simple_emb_iterate_end}, let
$\Phi':\Ttbar\hookrightarrow_\simple^\pre\Tt$ with
\[ \rg(\varphi^{\Phi'})=\rg(\varphi^\Phi)\cup\{\delta+1\}\cut\{\eta\} \]
and $\Psi':\Uubar\hookrightarrow_\simple^\pre\Tt$ with\
\[ \rg(\varphi^{\Psi'})=\rg(\varphi^\Psi)\cup\{(\delta+1,0)\}\cut\{(\eta,0)\}.\]

Now condition \ref{dfn:abstract_lifting}(\ref{item:finite_support_commutativity}) applies to $(\Phi,\Psi)$,
and also to $(\Phi',\Psi')$. In particular, $\bar{\Uu}=\Uu^{\Ttbar}$.
Let $\varphi^\Phi(\bar{\beta})=\varphi^{\Phi'}(\bar{\beta})=\beta$ etc
(this is unambiguous, but $\bar{\eta}=\overline{\delta+1}=\bar{\delta}+1$).

Let us verify condition \ref{item:zeta,eta_pres_in_proof}.
We have
\begin{enumerate}[label=--]
 \item 
$i^\Uu_{(\beta,0),(\delta+1,0)}(\zeta_\beta)=\zeta_{\delta+1} 
\text{ and }\varrho^{\Psi'}_{\bar{\beta}0}(\zeta^\Ttbar_{\bar{\beta}})=\zeta_\beta
\text{ and }\varrho^{\Psi'}_{{\bar{\delta}+1},0}(\zeta^\Ttbar_{{\bar{\delta}+1}})=\zeta_{\delta+1}$, and
\item $\varrho^{\Psi'}_{{\bar{\delta}+1},0}\com i^\Uubar_{(\bar{\beta},0),({\bar{\delta}+1},0)}=i^\Uu_{(\beta,0),(\delta+1,0)}\com\varrho^{\Psi'}_{\bar{\beta}0}$,
\end{enumerate}
by choice of $\beta_0$, \ref{dfn:abstract_lifting}(\ref{item:finite_support_commutativity}) and
properties of $\Psi'$.  Therefore
\begin{enumerate}[label=--]
 \item $i^\Uubar_{(\bar{\beta},0),({\bar{\delta}+1},0)}(\zeta^\Ttbar_{\bar{\beta}})=\zeta^\Ttbar_{{\bar{\delta}+1}}$.
\end{enumerate}
Also (using \ref{lem:simple_emb_iterate_end} for the first equality below),
\begin{enumerate}[label=--]
 \item 
 $\varrho^\Psi_{\bar{\eta}0}=i^\Uu_{(\delta+1,0),(\eta,0)}\com\varrho^{\Psi'}_{{\bar{\delta}+1},0}$ (recalling ${\bar{\delta}+1}=\bar{\eta}$) and
 \item 
$\varrho^\Psi_{\bar{\eta}0}(\zeta^\Ttbar_{\bar{\eta}})=\zeta_\eta$ and
 $\varrho^{\Psi'}_{\bar{\delta}+1,0}(\zeta^\Ttbar_{\bar{\delta}+1})=\zeta_{\delta+1}$,
\end{enumerate}
so it follows that $i^\Uu_{(\beta,0),(\eta,0)}(\zeta_\beta)=\zeta_\eta$.
Likewise, $i^\Uu_{(\beta,0),(\eta,0)}(\eta_\beta)=\eta_\eta$, as required.

Condition \ref{item:pi_it_comm}: In the following diagram, the maps corresponding to vertical arrows are the iteration maps.
We want to see that the back\footnote{In case of visual ambiguity,
by ``back square'' we mean the square with vertices $M^\Tt_\beta,M^\Tt_\eta,M^\Uu_{(\beta,0)},M^\Uu_{(\eta,0)}$.} square commutes with respect to the point $\xi<\rho_0(M^\Tt_\beta)$.
Because $\xi\in\rg(\varrho^\Phi_{\bar{\beta}})$,
it easily suffices to see that each side of the diagram, excluding the back side, commutes:

\begin{tikzpicture}[->,shorten >=1pt,auto,node distance=2cm,
                    semithick]

  \node (Tbb)  {$M^\Ttbar_{\bar{\beta}}$};
  \node (mid1bb) [right of=Tbb] {};
  \node (mid2bb) [right of=mid1bb] {};
  \node (Ubb) [right of=mid2bb] {$M^\Uubar_{(\bar{\beta},0)}$};
  \node (Teb) [above of=Tbb] {$M^\Ttbar_{\bar{\eta}}$};
  \node (mid1eb) [right of=Teb] {};
  \node (mid2eb) [right of=mid1eb] {};
  \node (Ueb) [right of=mid2eb] {$M^\Uubar_{(\bar{\eta},0)}$};
  \node (Tb)  [above of=mid2eb] {$M^\Tt_{\beta}$};
  \node (mid1b) [right of=Tb] {};
  \node (mid2b) [right of=mid1b] {};
  \node (Ub) [right of=mid2b] {$M^\Uu_{(\beta,0)}$};
  \node (Te) [above of=Tb] {$M^\Tt_{\eta}$};
  \node (mid1e) [right of=Te] {};
  \node (mid2e) [right of=mid1e] {};
  \node (Ue) [right of=mid2e] {$M^\Uu_{(\eta,0)}$};

  \path
        (Tbb) edge node [below] {$\wtpi^\Ttbar_{\bar{\beta}}$} (Ubb)
        (Tbb) edge node {} (Teb)
        (Teb) edge node [below] {$\wtpi^\Ttbar_{\bar{\eta}}$} (Ueb)
        (Ubb) edge node {} (Ueb)
        (Tbb) edge node [near end] {$\varrho^\Phi_{\bar{\beta}}$} (Tb)
        (Teb) edge node [near end] {$\varrho^\Phi_{\bar{\eta}}$} (Te)
        (Tb) edge node {} (Te)
        (Ubb) edge node [near end] {$\varrho^\Psi_{(\bar{\beta},0)}$} (Ub)
        (Ueb) edge node [near end] {$\varrho^\Psi_{(\bar{\eta},0)}$} (Ue)
        (Ub) edge node {} (Ue)        
        (Tb) edge node [below] {$\wtpi_\beta$} (Ub)
        (Te) edge node [below] {$\wtpi_\eta$} (Ue)
;
\end{tikzpicture}

We have $\pred^\Ttbar(\bar{\eta})=\bar{\beta}$.
By the properties of $\Phi$, the left side square commutes,
and also $\bar{\eta}\notin\dropset_\deg^\Ttbar$ and $i^\Ttbar_{\bar{\beta}\bar{\eta}}$
is $\nu$-preserving (since $(\beta,\eta]_\Tt\inter\dropset^\Tt_\deg=\emptyset$
and $i^\Tt_{\beta\eta}$ is $\nu$-preserving). Likewise for $\Psi,\Uubar$ and the right side square.
So \ref{dfn:abstract_lifting}(\ref{item:eventual_commutativity}) applies to the front square,
and it therefore commutes.
And the bottom and top squares commute by \ref{dfn:abstract_lifting}(\ref{item:finite_support_commutativity}), as desired.

Condition \ref{item:sigma_it_comm}: This is basically like the previous condition,
although of course instead of literally having the models $M^\Uubar_{(\bar{\beta},0)}$, etc,
on the right side, we have some set of pairs $(\Vv,F)$
where $\Vv$ is a finite linear iteration tree on that model and $F\in\es_+(M^\Vv_\infty)$,
and instead of the maps $j=i^\Uubar_{(\bar{\beta},0),(\bar{\eta},0)}$ or $j=\varrho^\Psi_{\bar{\beta}}$, etc,
we have the function $\extcopy(j,\cdot)$. And instead of $\wtpi^\Ttbar_{\bar{\beta}}$, etc,
we have $\abres^\Ttbar_{\bar{\beta}}$, etc. (We also literally use functions of the form $\extcopy(i,\cdot)$
on the left side, with $i=i^\Ttbar_{\betabar,\etabar}$, etc, though there we are only copying extenders, not trees.)
The front, bottom and top squares commute much as before.
The maps on the left and right sides of the original diagram above
are $\nu$-preserving. It follows immediately that the left side of the new (invisible) diagram commutes.
Consider the right side. Fix some $(\bar{\Vv},\bar{F})$ in the bottom-front corner,
and let:
\begin{enumerate}[label=--]
 \item  $(\bar{\Vv}',\bar{F}')=\extcopy(i^{\Uubar}_{(\bar{\beta},0),(\bar{\eta},0)},(\bar{\Vv},\bar{F}))$ (in the top-front corner),\
 \item $(\Vv'_0,F'_0)=\extcopy(\varrho^\Psi_{(\bar{\eta},0)},(\bar{\Vv}',\bar{F}'))$ (at top-back),
 \item $(\Vv,F)=\extcopy(\varrho^\Psi_{(\bar{\beta},0)},(\bar{\Vv},\bar{F}))$ (at bottom-back),
 \item $(\Vv'_1,F'_1)=\extcopy(i^\Uu_{(\beta,0),(\eta,0)},(\Vv,F))$ (also at top-back).
\end{enumerate}
We need to see that $(\Vv'_0,F'_0)=(\Vv'_1,F'_1)$.

Let $\alpha<\lh(\bar{\Vv})$ and let
\[ \pi^{\bar{\Vv}\bar{\Vv}'}_\alpha:\core_0(M^{\bar{\Vv}}_\alpha)\to\core_0(M^{\bar{\Vv}'}_\alpha) \]
be the copy map induced by $i^\Uubar_{(\bar{\beta},0),(\bar{\eta},0)}$, and likewise for other pairs of trees replacing $(\bar{\Vv},\bar{\Vv}')$
(with copy maps induced by $\varrho^\Phi_{(\bar{\beta},0)}$,
$\varrho^\Phi_{(\bar{\eta},0)}$ and $i^\Uu_{(\beta,0),(\eta,0)}$).
Then by induction on $\alpha<\lh(\bar{\Vv})$, we get $M^{\Vv'_0}_\alpha=M^{\Vv'_1}_\alpha$ and
\[ \pi^{\bar{\Vv}'\Vv'_0}_\alpha\com\pi^{\bar{\Vv}\bar{\Vv}'}_\alpha=
 \pi^{\Vv\Vv'_1}_\alpha\com\pi^{\bar{\Vv}\Vv}_\alpha,
\]
and that each of these 4 maps are near embeddings and are $\nu$-preserving (see \S\ref{sec:nu-pres} for the latter point),
and if $\alpha+1<\lh(\bar{\Vv})$, that therefore $E^{\Vv'_0}_\alpha=E^{\Vv'_1}_\alpha$,
and if $\alpha+1=\lh(\bar{\Vv})$, that therefore $F'_0=F'_1$.
This yields the commutativity of the right side of the new diagram and
 completes the proof of the lemma.
\end{proof}

Given a pair $(M,N)$,
an abstract lifting strategy for $(M,N)$ which works for \emph{finite} trees,
and an iteration strategy for $N$,
we now show that these objects induce a natural iteration strategy for $M$.

\begin{lem}\label{lem:om-alg_to_lambda-alg}
 Let $M$ be $m$-sound, $N$ be $n$-sound. Suppose there is an 
abstract $(m,n,\om)$-lifting strategy for $(M,N)$.
Suppose $N$ is sse-essentially-$(n,\omega_1+1)$-iterable.
Then $M$ is sse-essentially-$(m,\omega_1+1)$-iterable.

In fact, let $\Gamma$ be an sse-essentially-$(n,\omega_1+1)$ strategy for $N$ and $(\Sigma,\frakL)$ 
an abstract $(m,n,\om)$-lifting 
strategy for $(M,N)$ \tu{(}so $\Sigma$ is trivial\tu{)}. Then there is a unique abstract 
$(m,n,\omega_1+1)$-lifting strategy $(\Sigma^+,\frakL^+)$ for $(M,N)$ such that 
$(\Sigma^+,\frakL^+)$ is induced by $\Gamma$ and $\frakL\sub\frakL'$.
\end{lem}
\begin{proof}
We will define $(\Sigma^+,\frakL^+)$ with the properties claimed,
and uniqueness will be straightforward to see.
In fact, we will define pairs $(\Sigma_\eta,\frakL_\eta)$,
for limit ordinals $\eta\leq\omega_1$, such that
$(\Sigma_\eta,\frakL_\eta)$
is an abstract $(m,n,\eta)$-lifting strategy,
$(\Sigma_\eta,\frakL_\eta)$ is induced by $\Gamma$ and $\frakL\sub\frakL_\eta$. This will be by recursion on $\eta$. We will  then
extend $(\Sigma_{\omega_1},\frakL_{\omega_1})$ to
the desired $(\Sigma^+,\frakL^+)$.

For $\eta=\om$ we (must) set $(\Sigma_\omega,\frakL_\omega)=(\Sigma,\frakL)$,
and this suffices.
 Now suppose we have defined $(\Sigma_\eta,\frakL_\eta)$,
 for some limit $\eta\leq\omega_1$. We
will define $(\Sigma',\frakL')=(\Sigma_{\eta+1},\frakL_{\eta+1})$, and observe that this is the only possible choice.
We must and do set $\Sigma_\eta\sub\Sigma'$ and $\frakL_\eta\sub\frakL'$.

Clearly $\Sigma'$ must now be determined as follows.
Let $\Tt$ be of length $\eta$, via $\Sigma_\eta$, let $\Uu=\Uu^\Tt$ and $b=\Gamma(\Uu)$.
Let $(\beta,0)\in b$ be such that $((\beta,0),(\eta,0)]_\Uu$ does not drop.
Then $\Sigma'(\Tt)$ must be the branch
\[ c=[0,\beta)_\Tt\cup\{\gamma<\eta\mid(\beta,0)<_\Uu(\gamma,0)\in b\}. \]
Note that $c$ is indeed $\Tt$-cofinal and drops only finitely often,
by \ref{dfn:abstract_lifting}(\ref{item:Uu_is_lift_tree}),(\ref{item:zeta_movement}).

So define $\Sigma'$ as above and fix $\Tt$ of length $\eta+1$ according to $\Sigma'$. We set 
\[ \Uu^{\Tt}=\Uu^{\Tt\rest\eta}\conc\left<\Gamma(\Uu^{\Tt\rest\eta})\right> \]
(as required). We must define
\[ \Psi^\Tt=\left<\zeta_\alpha,\wtpi_\alpha,\eta_\alpha,\abres_\alpha\right>_{\alpha<\lh(\Tt)},\]
with
$\Psi^\Tt\rest\eta=\Psi^{\Tt\rest\eta}$. For this, it just remains to define 
$\zeta_\eta,\wtpi_\eta,\eta_\eta,\abres_\eta$,
and to verify that the required properties hold.
Lemma \ref{lem:abstract_lifting_at_limit} indicates how
these objects must be defined, giving uniqueness.
So we just need to see that the objects are well-defined
and satisfy the conditions \ref{dfn:abstract_lifting}(\ref{item:Sigma_strategy})--(\ref{item:finite_support_commutativity}).
Just as in the proof of \ref{lem:abstract_lifting_at_limit},
we can pick $\beta_0$ with the same properties as there.
It easily follows that $\zeta_\eta,\wtpi_\eta,\eta_\eta,\abres_\eta$ are well-defined.

We now verify that $(\Sigma',\frakL')$ is indeed an abstract $(m,n,\eta+1)$-lifting strategy,
for which we must check the conditions of  \ref{dfn:abstract_lifting}.
Conditions \ref{item:Uu_is_lift_tree}--\ref{item:sigma_props} are straightforward.
(Here for condition \ref{item:sigma_props},
note that $\Gamma$ is in fact an sse-essentially-$(n,\omega_1+\om)$-strategy for $N$: if an illfounded  model $M^{\Tt'}_{\omega_1+n}$ of a  finite  extension $\Tt'$ of a tree $\Tt$ of length $\omega_1+1$ via $\Gamma$ is reached, then by taking a countable hull, one can produce a tree $\bar{\Tt}'$ via $\Gamma$
which also has an illfounded model, a contradiction.
Therefore, the finite (putative) trees $\Vv$
produced by $\abres_{\omega_1}$
have wellfounded models.)
Conditions \ref{item:Uu_exts_are_lifts},\ref{item:zeta_movement} are trivial
as $\eta$ is a limit. Condition \ref{item:Sigma_strategy}
follows from condition \ref{item:pi_order_pres} and \ref{lem:nu_in_wfp}.
It remains to verify condition
\ref{item:finite_support_commutativity}. Let
\[ \Phi:\Ttbar\hookrightarrow_\simple\Tt 
\text{ and }\Psi:\Uubar\hookrightarrow_\simple\Uu \]
be as there.
We may assume that $\eta\in\rg(\varphi^\Phi)$,
so\
\[ \eta=\max(\rg(\varphi^\Phi))\text{ and }(\eta,0)=\max(\rg(\varphi^\Psi)).\]
Let $\beta_0<_\Tt\eta$ be as before.
Let $\delta+1=\successor^\Tt(\beta_0,\eta)$ and
let $\Phi',\Psi'$ be as in the proof of \ref{lem:abstract_lifting_at_limit}
(hence, $\Phi',\Psi'$ witness \ref{lem:simple_emb_iterate_end})
with
\[ \rg(\varphi^{\Phi'})=\rg(\varphi^\Phi)\cup\{\delta+1\}\cut\{\eta\},\] etc.
So condition \ref{item:finite_support_commutativity}
applies to $\Phi',\Psi'$.
But then condition
\ref{item:finite_support_commutativity} for $\Phi,\Psi$ easily follows from
the commutativity given by \ref{lem:simple_emb_iterate_end},
the definitions of $\zeta_\eta$, $\wtpi_\eta$, $\eta_\eta$, $\abres_\eta$ and
condition \ref{item:finite_support_commutativity} applied to $\Phi',\Psi'$.

Now suppose we have defined $(\Sigma_{\eta+1},\frakL_{\eta+1})$ where $\eta$ is a limit ordinal
and $\eta+\om\leq\lambda$. We claim that $\Sigma_{\eta+1}$ extends (obviously uniquely)
to an essentially-$(m,\eta+\om)$-strategy $\Sigma_{\eta+\om}$,\footnote{Of course,
$\Sigma_{\eta+\om}$ chooses no further branches than those specified by $\Sigma_{\eta+1}$,
but we now also demand wellfoundedness for trees of length $<\eta+\om$.}
and $\frakL_{\eta+1}$ extends uniquely to $\frakL_{\eta+\om}$
such that the pair has the right properties.
In fact, we work inductively through analogous pairs $(\Sigma_{\eta+n},\frakL_{\eta+n})$
for $1\leq n<\om$; clearly this suffices.
So suppose we have $(\Sigma_{\eta+n},\frakL_{\eta+n})$ where $1\leq n<\om$.
Since $\Sigma_{\eta+n+1}=\Sigma_{\eta+n}$, we just need to define
$\frakL_{\eta+n+1}$ and verify the conditions.

So let $\Tt$ via $\Sigma_{\eta+n+1}$ have length $\eta+n+1$.
Note that $\Uu=\Uu^\Tt$ is determined by $\Uu^{\Tt\rest\eta+n}$
and condition \ref{item:Uu_exts_are_lifts}, and by condition
\ref{item:sigma_props}, $\Uu$ is essentially-$n$-maximal.
We need to define $\zeta_{\eta+n}$, $\wtpi_{\eta+n}$, $\eta_{\eta+n}$, $\abres_{\eta+n}$.
We define these objects as the unique ones determined by condition 
\ref{item:finite_support_commutativity}. Uniqueness here is because
we can form finite supports as in condition \ref{item:finite_support_commutativity}
covering any given finite selection of $\Tt$.
To see that the objects are well-defined, it suffices to see that if
$\Phi,\Psi$ and $\Phi',\Psi'$ are as in condition \ref{item:finite_support_commutativity}
with
\[ \max(\rg(\varphi^\Phi))=\eta+n=\max(\rg(\varphi^{\Phi'})) \] 
and $\rg(\varphi^\Phi)\sub\rg(\varphi^{\Phi'})$,
then the facts about $\zeta_{\eta+n}$, $\wtpi_{\eta+n}$, etc,
determined by $\Phi,\Psi$, agree with those determined by $\Phi',\Psi'$.
Say we have $\Phi:\Ttbar\hookrightarrow_\simple^\pre\Tt$
and $\Psi:\Uubar\hookrightarrow_\simple^\pre\Uu$;
then $\Uubar=\Uu^\Ttbar$ (for by condition \ref{item:finite_support_commutativity}
for $(\Sigma_{\eta+n},\frakL_{\eta+n})$,
the extenders used in $\Uubar$ are according to condition \ref{item:Uu_exts_are_lifts}).
Likewise with $\Phi',\Psi',\Ttbar',\Uubar'$.
By \ref{lem:factor_pse} there are
$\Phi^*:\Ttbar\hookrightarrow_\simple^\pre\Ttbar'$
and $\Psi^*:\Uubar\hookrightarrow_\simple^\pre\Uubar'$
with $\varphi^{\Phi'}``\rg(\varphi^{\Phi^*})=\rg(\varphi^\Phi)$
and $\varphi^{\Psi'}``\rg(\varphi^{\Psi^*})=\rg(\varphi^\Psi)$. But then note that condition \ref{item:finite_support_commutativity}
applies to $\Phi^*,\Psi^*$. Then much as in the proof of \ref{lem:abstract_lifting_at_limit},
the commutativity given by this and by \ref{lem:factor_pse} gives the required agreement between $(\Phi,\Psi)$ and $(\Phi',\Psi')$.

Finally we need to verify the conditions hold for $(\Sigma_{\eta+n+1},\mathscr{L}_{\eta+n+1})$.
***Condition \ref{item:...} is a slight issue as we don't yet have full wellfoundedness...
Conditions \ref{item:Uu_is_lift_tree}--\ref{item:Uu_exts_are_lifts} are straightforward,
and condition \ref{item:finite_support_commutativity} holds by construction;
this leaves only condition \ref{item:zeta_movement}, in the case that $\alpha+2=\lh(\Tt)$.
Let $\beta=\pred^\Tt(\alpha+1)$ and $(\beta,j)=\pred^\Uu(\alpha+1,0)$,
so $M^{*\Tt}_{\alpha+1}=M^\Tt_{\beta j}$. Given any $\xi<\rho_0(M^\Tt_{\beta j})$ and $E\in\es_+(M^\Tt_{\beta j})$,
we can fix $\Phi,\Psi$ as before with 
\[ \Phi:\bar{\Tt}\hookrightarrow_\simple^*\Tt\text{ and }\Psi:\bar{\Uu}\hookrightarrow_\simple\Uu \]
and $\beta,\alpha+1\in\rg(\varphi^\Phi)$ and capturing $\xi,E$, and reflecting all the pertinent properties
if $\beta$, $\alpha+1$, $\xi$, $E$ with respect to $\Tt$ to corresponding properties of corresponding objects with respect to $\bar{\Tt}$.
Then condition \ref{item:zeta_movement} holds for $\bar{\Tt}$ and $\bar{\Uu}=\Uu^{\bar{\Tt}}$.
Using condition \ref{item:finite_support_commutativity} applied to $\Phi,\Psi$,
clauses \ref{item:zeta_semi-pres}--\ref{item:if_drop_in_Tt_then_drop_options_in_Uu} of condition \ref{item:zeta_movement} easily follow.
For clause \ref{item:eventual_commutativity}, argue much as in the proof of \ref{lem:abstract_lifting_at_limit},
considering a diagram analogous to the one considered there, with $\eta$ there replaced by $\alpha+1$.
\end{proof}

\begin{dfn}\label{dfn:standard_alg}
 Let $Y$ be $(x,\xi,z)$-good, $N=N^Y_{\xi z}$, and $y<\om$ 
 be as in part \ref{item:general_mouse_it} (with $y=0$) or part \ref{item:Q-mouse_it} of the main iterability proof \ref{lem:iterability},
 with $Y$ being $(y,\om_1+1)$-iterable, as witnessed by strategy $\Gamma$.
We will define the \emph{$\Gamma$-standard} $(z,y,\om_1+1)$-lifting strategy $(\Sigma,\frakL)$
for $(N,Y)$, induced by $\Gamma$. We will verify in \ref{lem:standard_abstract_lifting} that it has the right properties.

Set $\Sigma=$ the essentially-$(z,\om_1+1)$-strategy for $N$ given by the proof of \ref{lem:iterability}.

We must define $\frakL=\left<\Uu^\Tt,\Psi^\Tt\right>_{\Tt\in X}$, where $X$ is the set of all trees on $M$ via $\Sigma$.
So let $\Tt\in X$. Then $\Uu^\Tt=\lifttree^{\Tt,Y}$ (see \ref{dfn:lift^Tt,Y}).
Fix $\alpha<\lh(\Tt)$; we must define $\Psi^\Tt_\alpha=\left<\zeta_\alpha,\wtpi_\alpha,\eta_\alpha,e_\alpha\right>$.
These objects are just those derived from the lifting and resurrection process of
\ref{lem:iterability}.
Let $\Tt=\Tt_0\conc\Tt_1$ where $\Tt_0$ is the small component of $\Tt$.
We adopt from the proof of \ref{lem:iterability} the notation
\[ 
\left<M_\alpha,m_\alpha,Y_\alpha,\xi_\alpha,N_\alpha,\pi_\alpha\right>_{\alpha<\lh(\Tt_0)}.\]
Let $\alpha<\lh(\Tt)$ and let $E\in\es_+(M_\alpha)$ and $P=M_\alpha|\lh(E)$ and $k+1=\lh(\redd^{M_\alpha}(P))$
and $d=\deg^{\Uu}(\alpha,0)$.

Suppose $\alpha<\lh(\Tt_0)$. We set
\[ \zeta_\alpha=\xi_\alpha\text{ and }\wtpi_\alpha=\pi_\alpha\rest\OR\text{ and }
\eta_\alpha=\rho_0^{N_\alpha},
\]
and set $\abres_\alpha(E)=(\Vv,F)$ where
 $\Vv=\restree^{\Tt_0\Uu_0d}(P)$ and $F=F^{M^\Vv_k}_{\beta}$ where $(\beta,\cdot)=\exitresadd^{\Tt_0\Uu_0}(P)$.

(***Should just define things so that this case is already covered above) Now suppose $\alpha\geq\lh(\Tt_0)$, so $\Tt\rest(\alpha+1)$ is non-small.
Let $\pi_\alpha:M_\alpha\to M^\Uu_{(\alpha,0)}$ be the copy map given by the proof of \ref{lem:iterability}.
We set
\[ \zeta_\alpha=\OR(M^\Uu_{(\alpha,0)})\text{ and }\wtpi_\alpha=\pi_\alpha\rest\OR\text{ and }
\eta_\alpha=\rho_0(M^\Uu_{(\alpha,0)}),\]
and set $\abres_\alpha(E)=(\Vv,F)$
where $\Vv$ is the length $k+1$ tree on $M^\Uu_{(\alpha,0)}$ which uses only padding,
and $F=\extcopy(\pi_\alpha,E)$.
\end{dfn}
\begin{lem}\label{lem:standard_abstract_lifting}
 Adopt the hypotheses and notation of \ref{dfn:standard_alg}.
 Then $(\Sigma,\frakL)$ is an abstract $(z,y,\om_1+1)$-lifting strategy  
for $(N,Y)$, induced by $\Gamma$.
\end{lem}
\begin{proof}
 We just discuss conditions \ref{item:eventual_commutativity} and \ref{item:finite_support_commutativity} of \ref{dfn:abstract_lifting},
 as the rest is clear.
 
Condition  \ref{item:eventual_commutativity}: Assume the hypotheses there.
Then:
\begin{enumerate}[label=--]
\item $\eta_{\alpha+1}=i^\Uu_{(\beta,0),(\alpha+1,0)}(\eta_\beta)$,
\item $\zeta_{\alpha+1}=i^\Uu_{(\beta,0),(\alpha+1,0)}(\zeta_\beta)$ and
\item  $\pi_{\alpha+1}\com i^\Tt_{\beta,\alpha+1}=i^\Uu_{(\beta,0),(\alpha+1,0)}\com\pi_\beta$ (where $\pi_{\alpha+1},\pi_\beta$ are defined in \ref{dfn:standard_alg}), so
\item $\wtpi_{\alpha+1}\com i^\Tt_{\beta,\alpha+1}=i^\Uu_{(\beta,0),(\alpha+1,0)}\com\wtpi_\beta$ and
\item $\abres_{\alpha+1}(\extcopy(i^\Tt_{\beta,\alpha+1},E))=
\extcopy(i^\Uu_{(\beta,0),(\alpha+1,0)},\abres_\beta(E))$, by \ref{lem:res_comm}.
\end{enumerate}
 
 Condition \ref{item:finite_support_commutativity}:
 Let $(\Phi,\Psi):(\Ttbar,\Uubar)\hookrightarrow_\simple^\pre(\Tt,\Uu)$.

 We verify by induction on $\alpha<\lh(\Ttbar)$ that $\Uubar\rest(\alpha+1,1)=\Uu^{\Ttbar\rest(\alpha+1)}$
and the desired commutativity holds, and in fact that
letting $\pi^\Ttbar_\alpha$ be the lifting map as in \ref{dfn:standard_alg} but with respect to $\Ttbar$,
\begin{enumerate}[label=\tu{(}\roman*\tu{)}]
\item\label{item:zeta,eta_pres+} $\zeta_{\varphi(\alpha)}=\varrho_{\alpha 0}^{\Psi}(\zeta_\alpha^\Ttbar)\text{ and }\eta_{\varphi(\alpha)}=\varrho_{\alpha 0}^\Psi(\eta_\alpha^\Ttbar)$,
\item\label{item:pi,varrho_comm} $\pi_{\varphi(\alpha)}\com\varrho_{\alpha}^{\Phi}=\varrho^{\Psi}_{\alpha 0}\com\pi^\Ttbar_\alpha$,
\item\label{item:copy_comm}  $\all E\in\es_+(M^\Ttbar_\alpha)\left[ 
\abres_{\varphi(\alpha)}(\extcopy(\varrho^{\Phi}_\alpha,E))=
\extcopy(\varrho^{\Psi}_{\alpha 0},\abres^\Ttbar_\alpha(E))\right]$
\end{enumerate}
(this is slightly stronger than the required commutativity,
as we deal $\pi_{\varphi(\alpha)}$, etc, as opposed to $\wtpi_{\varphi(\alpha)}$, etc).
Note that  property \ref{item:copy_comm} will follow from  \ref{lem:res_comm}
together with properties  \ref{item:zeta,eta_pres} and \ref{item:pi,varrho_comm}.

Everything is trivial when $\alpha=0$.
So suppose everything holds for all $\alpha\leq\beta$, where $\beta+1<\lh(\Ttbar)$.
We prove everything when $\alpha=\beta+1$. Let
\[ (\gamma,j)=\pred^\Uubar(\beta+1,0).\]
Then by properties of $(\Tt,\Uu)$ and $(\Phi,\Psi)$, the structures of $\Tt,\Ttbar$
and of $\Uu,\Uubar$ correspond, and
\[ \Uu\rest[(\varphi(\beta),0),(\varphi(\beta)+1,0)]=\copyseg^{\varrho^\Psi_{\beta0}}(\Uubar\rest[(\beta,0),(\beta+1,0)], \]
and so by property \ref{item:copy_comm} for $\alpha=\beta$, and since $E^\Tt_{\varphi(\beta)}=\copyseg^{\varrho^\Phi_\beta}(E^\Ttbar_\beta)$,
we get $\Uubar\rest(\beta+1,1)=\Uu^{\Ttbar\rest(\beta+2)}$. We also have
\begin{enumerate}[label=--]
 \item $(\varphi(\gamma),j)=\pred^\Uu(\varphi(\beta)+1,0)$,
 \item $\gamma=\pred^\Ttbar(\beta+1)$ and $\varphi(\gamma)=\pred^\Tt(\varphi(\beta)+1)$,
 \item $M^{*\Ttbar}_{\beta+1}=M^{\Ttbar}_{\gamma j}$ and $M^{*\Tt}_{\varphi(\beta)+1}=M^\Tt_{\varphi(\gamma)j}$
 \item $\deg^\Ttbar(\beta+1)=\deg^\Tt(\varphi(\beta)+1)$ and $\deg^\Uubar(\beta+1,0)=\deg^\Uu(\varphi(\beta)+1,0)$,
 \item 
 $(\varphi(\beta)+1,\varphi(\beta+1)]_\Tt\inter\dropset_{\deg}^\Tt=\emptyset$ and $i^\Tt_{\varphi(\beta)+1,\varphi(\beta+1)}$ is $\nu$-preserving,
 \item $\varrho^\Phi_{\beta+1}=i^\Tt_{\varphi(\beta)+1,\varphi(\beta+1)}\com\bar{\varrho}_{\beta+1}$,
 where $\bar{\varrho}^\Phi_{\beta+1}:M^\Ttbar_{\beta+1}\to M^\Tt_{\varphi(\beta)+1}$ is given by the Shift Lemma,
 \item and likewise for $i^\Uu_{(\varphi(\beta)+1,0),(\varphi(\beta+1),0)}$ and $\varrho^\Psi_{(\beta+1,0)}$.
 \end{enumerate}

 We adopt the notation from \ref{lem:res_comm} with $M=M^\Ttbar_\gamma$, $M'=M^\Tt_{\varphi(\gamma)}$,
 $Y=M^\Uubar_\gamma$, $Y'=M^\Uu_{\varphi(\gamma)}$, $\zeta=\zeta^\Ttbar_\gamma$,
 $\zeta'=\zeta^\Tt_{\varphi(\gamma)}$, $m=\deg^\Ttbar(\gamma)=\deg^\Tt(\varphi(\gamma))$,
 $\pi=\pi^\Ttbar_\gamma$, $\pi'=\pi^\Tt_{\varphi(\gamma)}$, $\varrho=\varrho^\Phi_\gamma$
 and $\varsigma=\varrho^\Psi_{\gamma0}$. So we have $\Psi_j,\alpha_j$, etc, as there.

 Let us verify property \ref{item:zeta,eta_pres}. If $\alpha_j<\rho_0(M^\Uubar_{\gamma j})$ then $\alpha_j'=\varrho^\Psi_{\gamma j}(\alpha_j)$
 by \ref{lem:res_comm}, and by definition
 \[ \zeta^\Ttbar_{\beta+1}=i^\Uubar_{(\gamma,j),(\beta+1,0)}(\alpha_j)\text{ and }
  \zeta^\Tt_{\varphi(\beta)+1}=i^\Uu_{(\varphi(\gamma),j),(\varphi(\beta)+1,0)}(\alpha'_j),
 \]
 so by commutativity,
 \[ \bar{\varrho}^\Psi_{(\varphi(\beta+1),0)}(\zeta^\Ttbar_{\beta+1})=
 \zeta^\Tt_{\varphi(\beta)+1},\]
 but due to a lack of dropping,
 \[ \zeta^\Tt_{\varphi(\beta+1)}=i^\Uu_{(\varphi(\beta)+1,0),(\varphi(\beta+1),0)}(\zeta^\Tt_{\varphi(\beta)+1}),
 \]
 and therefore
 \[ \varrho^\Psi_{(\varphi(\beta+1),0)}(\zeta^\Ttbar_{\beta+1})=\zeta^\Tt_{\varphi(\beta+1)}.\]
 
 If instead $\alpha_j=\OR(M^\Uubar_{\gamma j})$ then $\alpha'_j=\OR(M^\Uu_{\varphi(\gamma) j})$
 and it is similar.
 
 The analogous relationship between $\eta^\Ttbar_{\beta+1}$ and $\eta^\Tt_{\varphi(\beta+1)}$ is likewise.
 
 Property \ref{item:pi,varrho_comm}: To obtain the required commutativity,
 we verify some corresponding commutativity at earlier stages. For notational simplicity we assume that $j>0$,
 but the $j=0$ case is analogous. Let
 \[ n=\deg^\Ttbar(\beta+1)=\deg^\Tt(\varphi(\beta)+1).\]
 By \ref{lem:res_comm} we have
 \begin{equation}\label{eqn:pred_comm_in_proof} \tau^{M^\Uu_{\varphi(\gamma)j}}_{\alpha_j'\om n}\com\psi'_j\com\psi_{\varrho^\Phi_\gamma}\rest(M^\Ttbar_{\gamma j})^\sq=
 \varrho^\Psi_{\gamma j}\com\tau^{M^\Uubar_{\gamma j}}_{\alpha_j\om n}\com\psi_j. \end{equation}

 Using again notation as in \ref{lem:res_comm}, but with $\gamma$ replaced by $\beta$,
 and writing $M_{\beta}$, etc, in place of $M$, etc, and making the notationally simplifying assumption
 that $E^\Ttbar_\beta\neq F(M^\Ttbar_\beta)$, we also have:
 \begin{equation}\label{eqn:exit_comm_in_proof} \tau^{M^\Uu_{\varphi(\beta)k_\beta}}_{\alpha'_{k_\beta}\om 0}\com\psi'_{\beta k_\beta}\com\psi_{\varrho^\Phi_\beta}\rest(\exit^\Ttbar_\beta)^\sq=
  \varrho^\Psi_{\beta k_\beta}\com\tau^{M^\Uubar_{\beta k_\beta}}_{\alpha_{k_\beta}\om 0}\com\psi_{\beta k_\beta} \end{equation}
(note that with $E^\Ttbar_\beta\neq F(M^\Ttbar_\beta)$, we have $\wtalpha_\beta=\alpha_{\beta k_\beta}$, etc).

But since $\bar{\varrho}^\Phi_{\beta+1}$, $\bar{\varrho}^\Psi_{\beta+1,0}$, $\pi^\Ttbar_{\beta+1}$ and $\pi^\Tt_{\varphi(\beta)+1}$
are defined via (essentially) the Shift Lemma,
lines (\ref{eqn:pred_comm_in_proof}) and (\ref{eqn:exit_comm_in_proof}) yield easily that
\[ \pi^\Tt_{\varphi(\beta)+1}\com\varrho^\Phi_{\beta+1}=\varrho^\Psi_{\beta+1,0}\com\pi^\Ttbar_{\beta+1}. \]
But then since
$\pi^\Tt_{\varphi(\beta+1)}\com i^\Tt_{\varphi(\beta)+1,\varphi(\beta+1)}=
 i^\Uu_{(\varphi(\beta)+1,0),(\varphi(\beta+1),0)}\com\pi^\Tt_{\varphi(\beta)+1}$,
property \ref{item:pi,varrho_comm} follows.
\end{proof}

\begin{rem} ***I need to make this section correct/general. In the following definition, a \emph{term sequence} $\vec{t}$ is a sequence $(t_0,\ldots,t_{n-1})$
of terms whose interpretation $\vec{t}^P$ relative to a premouse $P$,
if defined, specifies a finite iteration tree $\Tt$ or some element of $M^\Tt_\infty$, etc,
as in \ref{dfn:lifting_norm}. The terms must have the appropriate degrees; in particular, $t_0$ is a $\Sigma_1$ term.
So for example, maybe $\Tt$ is a finite tree on $M_\infty$ and
\begin{enumerate}[label=--]
 \item $t_0$ is a $\Sigma_1$ term and $t_0^{M_\infty}(p_1^{M_\infty})=E^\Tt_0$,
 \item $t_1$ is a $\Sigma_1$ term and $t_1^{M_\infty}(p_1^{M_\infty})=\vec{a}_0\in[\nu(E^\Tt_0)]^{<\om}$,
 \item $t_2$ is a $\Sigma_{m+1}$ term and $t_2^{M^\Tt_1}(\avec_0,\pvec_{m+1}^{M^\Tt_1})=E^\Tt_1$ where $m=\deg^\Tt(1)$,
 \item etc.
\end{enumerate}
\end{rem}

Recall the mouse operator $\Mmm$ and the set $A\sub\RR$ and formula $\psi_0$
introduced at the beginning of \S\ref{sec:the_scale},
and the constructive scale on $A$, defined in \ref{dfn:con_scale}.

\begin{dfn}\label{dfn:limit_alg}
Let $x_n\to x$ modulo the constructive scale on $A$. Let $M_\infty$ be the term structure defined in \ref{dfn:M_infty}.
Let $Y$ be $(y,\om_1+1)$-iterable and $\Gamma$ be a $(y,\om_1+1)$-strategy for $Y$.
Suppose that $Y$ is $M_{x_n}$-good for all $n<\om$.
We will define the \emph{$\Gamma$-limit} abstract $(0,y,\om_1+1)$-lifting strategy  $(\Sigma,\frakL)$  
for $(M_\infty,Y)$, induced by $\Gamma$ (as verified in \ref{lem:limit_alg}).

Let $(\Sigma_n,\frakL_n)$ be the $\Gamma$-standard $(0,y,\om_1+1)$-lifting strategy for $(M_{x_n},Y)$
(see \ref{dfn:standard_alg}). Given a tree $\Tt_n$ according to $\Sigma_n$,
and $\alpha<\lh(\Tt_n)$, we write $\zeta^{\Tt_n}_\alpha$, etc, for the objects specified by $\frakL_n$.

Now we will define an abstract $(0,y,\om)$-lifting strategy $(\Sigma',\frakL')$,
and set $(\Sigma,\frakL)$ to be its unique extension to $(0,y,\om_1+1)$, induced by $\Gamma$, given by \ref{lem:om-alg_to_lambda-alg}.
We define $\Uu^\Tt,\Psi^\Tt$, for finite length essentially $0$-maximal trees $\Tt$ on $M_\infty$,
by induction on $\lh(\Tt)$.
Fix $\Tt$ and $\Uu=\Uu^\Tt$, with $\Tt$ of length $\alpha+1$.
We must define $\zeta^\Tt_\alpha,\wtpi^\Tt_\alpha,\eta^\Tt_\alpha,\abres^\Tt_\alpha$.

Fix a term sequence $\vec{t}$ such that $\Tt=\vec{t}^{M_\infty}$. 
For sufficiently large $n<\om$ let $\Tt_n=\vec{t}^{M_{x_n}}$.
We will verify in \ref{lem:limit_alg_well_def} that for sufficiently large $n$,
$\Tt_n$ is an essentially-$0$-maximal tree on $M_{x_n}$ of the same length,
and $\Uu^{\Tt_n}=\Uu$.

We set $\zeta^\Tt_\alpha=\lim_{n\to\om}\zeta^{\Tt_n}_\alpha$ and $\eta^\Tt_\alpha=\lim_{n\to\om}\eta^{\Tt_n}_\alpha$.

Let $\beta<\rho_0(M^\Tt_\alpha)$ and $\vec{u}$ be a term sequence
such that $\beta=\vec{u}^{M_\infty}$. Let $\beta_n=\vec{u}^{M_{x_n}}$.
We will have $\beta_n<\rho_0(M^{\Tt_n}_\alpha)$ for sufficiently large $n$.
We set
\[ \wtpi^\Tt_\alpha(\beta)=\lim_{n\to\om}\wtpi^{\Tt_n}_\alpha(\beta_n). \]

Similarly, let $E\in\es_+(M^\Tt_\alpha)$. Let $\vec{v}$ be such that $\vec{v}^{M_\infty}=E$ and
let $E_n=\vec{v}^{M_{x_n}}$. We will have $E_n\in\es_+(M^{\Tt_n}_\alpha)$ for sufficiently large $n$.
We set
\[ \abres^\Tt_\alpha(E)=\lim_{n\to\om}\abres^{\Tt_n}_\alpha(E_n).\qedhere \]
\end{dfn}
\begin{dfn}
Continuing with notation as in \ref{dfn:limit_alg}, we define also some auxiliary functions,
$\abrestree$ (\emph{abstract resurrection tree}), $\abresl$ (\emph{abstract resurrection length}),
$\abresprodstage$ (\emph{abstract resurrection production stage}),
$\abresprodseg$ (\emph{abstract resurrection production segment})
and $\abdirectprodseg$
(\emph{abstract direct production segment}) functions.

First suppose that $\lh(\Tt)<\om$ (so we have $\Tt_n$ for large $n$, etc).
Given $S\ins M^\Tt_\alpha$ (and corresponding $S_n\ins M^{\Tt_n}_\alpha$ for large $n$) we write
(***this is only for the non-small case? but maybe don't need it otherwise)
\begin{enumerate}[label=--]
 \item $\abrestree^\Tt_\alpha(S)=\lim_{n\to\om}\restree^{\Tt_n\Uu d}_{x_n}(S_n)$ where $d=\deg^\Uu(\alpha,0)$,
 \item $\abresl^\Tt_\alpha(S)=k$ where $k+1=\lh(\abrestree^\Tt_\alpha(S))$,
 \item letting $c=\deg^\Tt(\alpha)$ and
 $\pi_{n\alpha}:\core_0(M^{\Tt_n}_\alpha)\to\core_0(N_{x_n\zeta_\alpha c}^{M^\Uu_{\alpha 0}})$
 be the canonical lifting map, we set
 \benumdd
 \item $\abresprodstage^\Tt_\alpha(S)=\lim_{n\to\om}\resprodstage^{M^\Uu_{\alpha 0}}_{x_n,\zeta_\alpha,c,\pi_{n\alpha}}(S_n)$,
 \eenum
 \item letting $k=\abresl^\Tt_\alpha(S)$, we set
 \benumdd
 \item $\abresprodseg^\Tt_\alpha(S)=\lim_{n\to\om}\prodseg^{M^\Uu_{\alpha k}}_{x_n}(\abresprodstage^\Tt_\alpha(S))$
  \eenum
 (note that this is a segment of $M^\Uu_{\alpha k}$),
 \item letting $\Vv=\abrestree^\Tt_\alpha(S)$, we set
 \benumdd
  \item $\abdirectprodseg^\Tt_\alpha(S)=\abresprodseg^\Tt_\alpha(S)$ if $\all i[E^\Vv_i=\emptyset]$,
  \item $\abdirectprodseg^\Tt_\alpha(S)=M^\Uu_{\alpha 0}|\lh(E^\Vv_i)$ where $i$ least s.t.~$E^\Vv_i\neq\emptyset$
 \eenum
 (in both cases this is a segment of $M^\Uu_{\alpha 0}$).
\end{enumerate}

Now for the case that $\lh(\Tt)\geq\om$ we use simple tree embeddings,
to lift the preceding definitions just as we did for the data within an abstract lifting algorithm.\footnote{
Because $\Tt$ is infinite, we do not see how to produce useful (infinite) versions $\Tt_n$ of $\Tt$ on $M_{x_n}$.}
One verifies that this is well-defined via arguments as before.
\end{dfn}

\begin{lem}\label{lem:limit_alg_well_def}
Adopt the hypotheses and notation of \ref{dfn:limit_alg}.
Then:
\begin{enumerate}[label=\tu{(}\roman*\tu{)}]
 \item for all sufficiently large $n$, $\Tt_n$, $\beta_n$ and $E_n$ are well-defined,
$\Tt_n$ is an essentially-$0$-maximal tree on $M_{x_n}$, 
$\beta_n<\rho_0(M^{\Tt_n}_\alpha)$, and $E_n\in\es_+(M^{\Tt_n}_\alpha)$,
\item $\Uu=\Uu^{\Tt_n}$ for all sufficiently large $n$, 
\item the limits used in the definition exist \tu{(}in the ``eventually constant'' sense\tu{)},
and
\item the definitions are independent of the choices of term sequences $\vec{t},\vec{u},\vec{v}$.
\end{enumerate}
\end{lem}
\begin{proof}
For each term sequence $\vec{s}$, we have
\[ \vec{s}^{M_\infty}\downarrow\text{ and is a }0\text{-maximal tree }\Ss\text{ on }M_\infty \]
iff
\[ \all^* n\left[\vec{s}^{M_{x_n}}\downarrow\text{ and is a }0\text{-maximal tree }\Ss_n\text{ on }M_{x_n}\right],\]
and when these things hold,
\[ \all^*n\left[\Ss,\Ss_n\text{ have the same tree, drop and degree structure}\right].\]
Moreover, fixing such $\vec{s},\Ss,\Ss_n$, for each $\vec{r}$ and
$\rSigma_{m+1}$ formula $\varphi$ where $m=\deg^\Ss(\infty)$,
we have
\[ \vec{r}^{M_\infty}\downarrow\text{ and }M^\Ss_\infty\sats\varphi(\vec{r}^{M_\infty},\pvec_{m+1}^{M^\Ss_\infty}) \]
iff
\[ \all^*n\left[\vec{r}^{M_{x_n}}\downarrow\text{ and }M^{\Ss_n}_\infty\sats\varphi(\vec{r}^{M_{x_n}},\pvec_{m+1}^{M^{\Ss_n}_\infty})\right].\]
These statements follow from the fact that $x_n\to x$ modulo the theory norms,
using the uniform reduction of $\Sigma_{m'+1}^{M^{\Ss'}_\infty}$ to $\Sigma_1^P$,
where $\Ss'$ is finite essentially-$0$-maximal on $P$ and $m'=\deg^{\Ss'}(\infty)$,
 via a generalization of  \cite[3.18--3.28]{extmax} to the superstrong level (\cite{extmax} works formally with premice which do not have extenders of superstrong type on their sequence).
 
 Parts (i) and (iv) follow from these considerations,
 part (iv) because $\Tt_n,\beta_n,E_n$ are independent of the choice of terms for sufficiently large $n$.
  Parts (ii) and (iii) are because the relevant ordinals, including indices for the extenders of $\Uu$,
  are incorporated into the lifting norms.
\end{proof}
\begin{lem}
Adopt the hypotheses and notation of \ref{dfn:limit_alg}.
Then $M_\infty$ is wellfounded and $(0,\om_1+1)$-iterable, and in fact,
$(\Sigma,\frakL)$ is a $(0,y,\om_1+1)$-lifting algorithm induced by $\Gamma$.
\end{lem}
\begin{proof}
By \ref{lem:om-alg_to_lambda-alg}, we just need to see that $(\Sigma',\mathfrak{L}')$, as defined in \ref{dfn:limit_alg},
is a $(0,y,\om)$-lifting strategy.
Let $\Tt$ be a finite length essentially $0$-maximal tree on $M_\infty$
and $\Uu=\Uu^\Tt$, so $\Uu=\Uu^{\Tt_n}$ for all sufficiently large $n$,
whenever $\tvec$ is a term with $\tvec^{M_\infty}=\Tt$ and $\Tt_n=\tvec^{M_{x_n}}$.
For the purposes of illustration, we verify that
\[ \wtpi_{\varphi(\alpha)}\com\varrho_{\alpha}^{\Phi}=\varrho^{\Psi}_{\alpha 0}\com\wtpi^\Ttbar_\alpha \]
in the context of condition \ref{item:finite_support_commutativity}; the rest is similar.
So we have $\Phi,\Psi$ and $\Ttbar,\Uubar$ and $\Uubar=\Uu^\Ttbar$.
Fix a term $\vec{\bar{t}}$ such that $\Ttbar=\vec{\bar{t}}^{M_\infty}$.
Fix $\xibar<\rho_0(M^\Ttbar_\alpha)$.
We want to verify the commutativity with respect to $\xibar$. Let
\[ \xi=\varrho^\Phi_\alpha(\xibar).\]
Fix terms $\vec{\bar{u}},\uvec$ such that $\xibar=\vec{\bar{u}}^{M_\infty}$
and $\xi=\varrho_\alpha^\Phi(\xibar)=\uvec^{M_\infty}$.
For sufficiently large $n$, let $\Ttbar_n=\vec{\bar{t}}^{M_{x_n}}$,
and $\Phi_n:\Ttbar_n\hookrightarrow\Tt_n$ with $\rg(\varphi^{\Phi_n})=\rg(\varphi^\Phi)$,
and $\xibar_n=\vec{\bar{u}}^{M_{x_n}}$ and $\xi_n=\uvec^{M_{x_n}}$.
Note then that for all sufficiently large $n$,
we have $\varrho^{\Phi_n}_\alpha(\bar{\xi}_n)=\xi_n$,
as this equality is first-order over $M_\infty$ (***must generalize).
But
\[ \wtpi^\Ttbar_\alpha(\xibar)=\lim_{n\to\om}\wtpi^{\Ttbar_n}_\alpha(\xibar_n), \]
\[ \wtpi^\Tt_{\varphi(\alpha)}(\xi)=\lim_{n\to\om}\wtpi^\Tt_{\varphi(\alpha)}(\xi_n), \]
and since $(\Sigma_n,\mathfrak{L}_n)$ is a $(0,y,\om)$-lifting strategy,
 for large $n$ we have
 \[ \varrho^\Psi_{\alpha 0}(\pi^{\Ttbar_n}_\alpha(\xibar_n))=\pi^\Tt_{\varphi(\alpha)}(\varrho^{\Phi_n}_\alpha(\xibar_n)), \]
and so
\[ \varrho^\Psi_{\alpha 0}(\lim_{n\to\om}\wtpi^{\Ttbar_n}_\alpha(\xibar_n))=
 \lim_{n\to\om}\wtpi^\Tt_{\varphi(\alpha)}(\xi_n).
\]
Putting the equations together, we have the desired commutativity.
\end{proof}

We can now conclude:

\begin{lem}
Let $x_n\to x$ modulo the constructive scale on $A$. Let $M_\infty$ be the limit structure defined in \ref{dfn:M_infty}.
Then $M_x$ exists and $M_\infty=M_x$ \tu{(}that is, $M_\infty$ is the least $(\om,\om+1)$-iterable $x$-premouse
satisfying $\psi_0$\tu{)}.
\end{lem}
\begin{proof}
Let $T_\infty$ be as in \ref{dfn:M_infty}.
Since $T_\infty=\Th_1^{M_\infty}$ and $M_\infty$ is wellfounded, $M_\infty$ is a premouse
and $M_\infty\sats$``I am $M_x$'' (that is, $M_\infty\sats[\psi_0\&\neg\psi_0']$ where $\psi_0'$ asserts ``There is a proper segment of me satisfying $\psi_0$'').
Since $M_\infty=\Hull_1^{M_\infty}(\emptyset)$, it is sound.
Since $M_\infty$ is iterable, therefore $M_\infty=M_x$.
\end{proof}

\begin{rem}
We haven't yet demonstrated in general that the constructive scale is a semiscale,
because we haven't yet shown that $M_x\pins\Mmm(x)$. We will verify this in \ref{?}.
However, in certain special cases this is clear.

For example, if $\Mmm(x)$ is the least $P\elem_1 M_1(x)$,
then clearly $M_x\pins\Mmm(x)$, because $M_x$ is $1$-small (because $T_\infty=\Th_1^{M_\infty}$).
So by \ref{?}, we have established the semiscale property for $\Pi^1_3$, assuming the existence of $M_1^\#(x)$ for every $x$ (***we don't quite need this though)

One certain case is trivial: Recall that $\mathrm{Lp}(x)$  is the stack of all $\om$-mice over $x$.
If $\Mmm(x)\elem_1\mathrm{Lp}(x)$ for all $x$,
then clearly $M_x\pins\Mmm(x)$.
\end{rem}

\subsection{Limit lifting maps (to edit/delete)}

\begin{lem}
Let $\Tt$ be via $\Sigma_\infty$, 
$\Uu=\Uu^\Tt$, $\alpha<\lh(\Tt)$, $d=\deg^\Tt(\alpha)$.
Write $\eta_\alpha=\eta^\Tt_\alpha$, etc.
Let $N^\Tt_{\alpha n}=\core_d(N^{M^\Uu_{\alpha 0}}_{\zeta_\alpha x_n})$.
Then:
\begin{enumerate}
 \item\label{item:eta_is_limit_rho_0} $\all^*_n[\eta_\alpha=\rho_0(N^\Tt_{\alpha n})]$.
 \item\label{item:pi_alpha_is_limit_d-lifting} $\pi_\alpha$ is \emph{limit $d$-lifting on its domain}.
 That is, for all $\betavec\in[\rho_0(M^\Tt_\alpha)]^{<\om}$ and all $\rSigma_d$ formulas $\varphi$,
 \[ M^\Tt_\alpha\sats\varphi(\betavec)\implies\all^*n
  [N^\Tt_{\alpha n}\sats\varphi(\pi_\alpha(\betavec))]. \]
  \item\label{item:pi_alpha_is_limit_p-preserving} $\pi_\alpha$ is \emph{limit $\pvec_{d+1}$-preserving}: $\pi_\alpha(\pvec_{d+1}^{M^\Tt_\alpha})=\lim_{n\to\om}\pvec_{d+1}^{N^\Tt_{\alpha n}}$.
\item\label{item:pi_alpha_is_limit_c-preserving} $\pi_\alpha$ is \emph{limit c-preserving***this needs to be generalized}:
for all $\beta<\rho_0(M^\Tt_\alpha)$,
\[ M^\Tt_\alpha\sats\beta\in\Card\iff N_{\alpha n}\sats\pi_\alpha(\beta)\in\Card.
\]
[***the version if we have the weakened assumption earlier...
for all $\beta<\rho_0(M^\Tt_\alpha)$, if $\beta$ is an $M^\Tt_\alpha$-cardinal
but
\[ \neg\all^*_n[\pi_\alpha(\beta)\text{ is an }N^\Tt_{\alpha n}\text{-cardinal}],\]
then:
\begin{enumerate}[label=--]
\item $\Mmm$ is of case \ref{***?} (non-$1$-$\CC$-closed), with $M_\infty,M_{x_n}$ passive,
\item$\beta=\max(p_1^{M_\infty})=\lgcd(M_\infty)$,
\item $\lgcd(M_\infty|\beta)=\kappa$, and
\item  $\kappa$ is non-$M_\infty$-measurable.]
\end{enumerate}
\end{enumerate}
\end{lem}
\begin{proof}
First consider the case that $\alpha<\om$.

 We will prove the corresponding statements with $N_{\alpha n}$ replaced by $M^{\Tt_n}_\alpha$,
 where $\Tt_n$ is the tree on $M_{x_n}$ corresponding to $\Tt$ for large $n$, 
 and corresponding partial maps
 \[ \bar{\pi}_{\alpha n}:M_\infty\to M_{x_n}.\]
 This suffices, because
then these statements lift to the desired ones using the usual lifting maps.

 (***This should really be shifted back to the definitions of norms, as it's needed there:)

First suppose that $\alpha=0$.

 Part \ref{item:eta_is_limit_rho_0} is directly by the definition.

 Now $\left<\bar{\pi}_{0n}\right>_{n<\om}$ is limit $m$-lifting (where $m=\deg(M_\infty)$) because
we can convert the fact that $M_\infty\sats\varphi(\betavec)$ into a statement about certain terms,
which lifts to $M_{x_n}$ for all sufficiently large $n$.
The limit $\pvec_{d+1}$-preservation we have already established, and the limit c-preservation
is because of our current restrictions regarding $\max(p_1)$ when $\Mmm$ is non-$1$-$\CC$-closed
(and some of the preceding lemmas regarding $\max(p_1)$). These facts then lift from $M_{x_n}$ to $N_{0 n}$ 
via the core maps.

Now consider the case that $0<\alpha<\om$. We fix trees $\Tt_n$ on $M_{x_n}$ corresponding to $\Tt$ on $M_\infty$,
for large $n$. By induction, for large $n$ we get that $\Tt_n,\Tt$ have corresponding structure,
drops, degrees, etc (here we make definite use of limit c-preservation; without the assumptions giving this we need to modify this part). Again it suffices to consider lifting to $\left<M_{x_n}\right>_{n<\om}$ since we can lift from $M_{x_n}$
to $N_{\alpha n}$ with the standard lifting map, a weak embedding.
Again $\left<\bar{\pi}_{0n}\right>_{n<\om}$ is limit $\deg^\Tt(\alpha)$-lifting because
we can convert the fact that $M^\Tt_\alpha\sats\varphi(\betavec)$
into a statement about certain terms,
which lift to $M^{\Tt_n}$ for all sufficiently large $n$.
The limit $\pvec_{d+1}$-preservation and limit c-preservation also easily follow by induction
and commutativity. 

Now suppose $\alpha\geq\om$. Fix a particular $\betavec$ and $\beta$.
Let $\Phi:\bar{\Tt}\hookrightarrow_\simple\Tt$
and $\Psi:\bar{\Uu}\hookrightarrow_\simple\Uu$ be as in \ref{?},
with the relevant objects in $\rg(\varrho^\Phi_{\bar{\alpha}})$
and $\rg(\varrho^\Psi_{\bar{\alpha}0})$, where $\varphi^\Phi(\bar{\alpha})=\alpha$.
Let $\bar{\Tt}_n$ be the tree on $M_{x_n}$ corresponding to $\bar{\Tt}$, for large $n$.
Let $\bar{\Uu}=\Uu^{\bar{\Tt}}=\Uu^{\bar{\Tt}_n}$ for large $n$.
Then the previous cases already apply to $\bar{\Tt}$, and we have enough commutativity
and $\varrho^\Psi_{\bar{\alpha}0}$
is sufficiently preserving that we get the desired facts.
\end{proof}
\subsection{$M_\infty$ and $M_x$}
Since we have shown that $M_\infty$ is $(0,\om_1+1)$-iterable, we have:
\begin{cor}\label{cor:M_infty_and_M_x}
 If $\Mmm$ is $1$-$\CC$-closed then $M_\infty=M_x$. If $\Mmm$ is exactly reconstructing then $M_x\ins M_\infty$.
\end{cor}
\begin{proof}
 See Remark \ref{rem:M_infty}
 and Lemma \ref{lem:M_infty_basics_exactly_recon}.
\end{proof}

\section{The mouse scale is a $\Gamma$-scale}\label{sec:lower_semi}

\subsection{Lower semicontinuity}
We now begin the proof of the key lemma which will establish that mouse scales are lower semicontinuous. We would like to might  mimic the proof of lower semicontinuity used in the sharps case. 
We will begin with the depth $1$ norms. Deeper norms will be dealt with by combining the method
here with the methods used to prove that the deeper norms are in fact norms.

However, in the general case,
things are  more subtle.
At the level of sharps,
 the absorption of an iterate into the background construction can be performed by working directly with indiscernibles.
In the general case the author doesn't see how to arrange this as directly,
and our proof makes more serious use of the iterability of mice in $V$.

\subsubsection{Iterating to background}

We will first discuss 
a key fact (Theorem \ref{tm:iter_to_bkgd} below) which will be essential to the proof. 
It is a variant of a result  \cite[Theorem 3.2]{maxcore} due to Steel and Schimmerling;
the analogous fact for other fully backgrounded constructions has featured in
various places in the literature.

\begin{dfn}\label{dfn:Tt^Sigma_N}
Let $\Sigma$ be a $(k,\alpha)$-iteration strategy for a premouse $M$, and $N$ be a premouse.
We write $\Tt^\Sigma_N$ for the unique tree $\Tt$ 
via $\Sigma$
of successor length such that $N\ins M^\Tt_\infty$ and $\lh(E^\Tt_\beta)\leq\OR^N$ for all 
$\beta+1<\lh(\Tt)$, if such $\Tt$ exists.
\end{dfn}

\begin{rem}\label{rem:Tt^Sigma_N}
The uniqueness of $\Tt$ is a standard observation. If we know $\Tt\rest(\eta+1)$ but 
$N\nins M^\Tt_\eta$ then $E^\Tt_\eta$ causes the least disagreement between $M^\Tt_\eta$ 
and $N$, and $N|\lh(E^\Tt_\eta)$ is passive.
And if we know $\Tt\rest\eta$ and $\eta$ is a limit, then $\Tt\rest(\eta+1)$ is determined by 
$\Sigma$. Note that this turns out to be the comparison $(\Tt,\Uu)$ of $(M,N)$,
using $\Sigma$ to form $\Tt$, and $\Uu$ happens to be trivial.
\end{rem}

\begin{tm}[Iterating to background]\index{Iterating to background}\label{tm:iter_to_bkgd}
Let $Y$ be countable and $(x,\delta,0)$-good. Let $M$ be an $\om$-mouse over $x$ with $M\in Y$
but $M\npins N_\delta$.
Let $\Sigma=\Sigma_M$. Let $\mu=\aleph_\delta^Y$.
Suppose that either:
\begin{enumerate}[label=(\roman*)]
 \item\label{item:it_to_bkgd_non-small} $Y$ is non-small and $\delta=\delta^Y=\mu$ and for some $y<\om$,
$Y$ is $(y,\om_1+1)$-iterable,
and $\rho_{y+1}^Y\leq\delta<\rho_y^Y$, and $Y$ is $\delta$-sound,
\item\label{item:it_to_bkgd_small} $Y$ is small and $(0,\om_1+1)$-iterable and there are $Q,Z$ such that $Q\pins Z\ins Y$ such that $\mu\leq\OR^Q$ and $Q$ is a Q-structure
for $\mu$ and $Z$ is admissible.
\end{enumerate}
Let $\eta\leq\delta$. Then:

\begin{enumerate}
 \item\label{item:Tt_eta_exists} $\Tt_\eta=\Tt^\Sigma_{N_\eta}$ exists \tu{(}see \ref{dfn:Tt^Sigma_N}\tu{)}
and $\Tt_\eta$ is small.
\item\label{item:eta<delta_Tt_in_below_delta} If $\eta<\delta$ then $\Tt_\eta\in Y|\mu$.
\item\label{item:uniformity} The function $\eta\mapsto\Tt_\eta$, with domain $\delta$, is a class 
of $Y|\mu$.
\item \label{item:final_tree_small} If $Y$ is small and $Z$ witnesses \ref{item:it_to_bkgd_small} then $\Tt_\delta\in Z$.
\item\label{item:final_tree} If $Y$ is non-small then:
\begin{enumerate}[label=--]
\item $\lh(\Tt_\delta)=\delta+1$ and $\delta=\delta(\Tt_\delta\rest\delta)$ and 
$N_\delta=M(\Tt_\delta\rest\delta)$,
\item $\Tt_\delta\rest\delta$ and $M(\Tt_\delta\rest\delta)$ are classes of $Y|\delta$,
\item $\Tt^\Sigma_{N_\alpha}$ exists for each $\alpha\leq\OR^Y$.
\item $\Tt_\delta\notin Y$, and $\Tt^\Sigma_{N_\alpha}=\Tt_\delta$ for each $\alpha\in[\delta,\OR^Y]$,
\end{enumerate}
\item\label{item:strategy1} If $\eta<\delta$ and $t_\eta,t_{\eta+\om}\neq 3$ \tu{(}so either $N_\eta$ is unsound or $\J(N_\eta)$ is small\tu{)},
then $Q(\Tt_\eta\rest\lambda,[0,\lambda)_{\Tt_\eta})\ins N_\eta$ for each limit $\lambda<\lh(\Tt_\eta)$,
so $\Tt_\eta$ is determined by $N_\eta$ via Q-structures.
\item\label{item:strategy2} Suppose $\eta<\delta$ and either $t_\eta=3$ or $t_{\eta+\om}=3$.
Let $\lambda=\card^Y(\eta)$ and $Q\pins Y$ be the Q-structure for $\lambda$ and $\zeta=\OR^Q$.
Then $\Tt_\lambda=\Tt_\eta=\Tt_\zeta$ and $\lh(\Tt_\lambda)=\lambda+1$ and 
$N_\lambda=M(\Tt_\lambda\rest\lambda)$
 is small, hence $\Tt_\lambda$ is determined by $N_\zeta$ via Q-structures.
 \item\label{item:it_to_bkgd_Y_small_stage_delta,t_delta_neq_3} Suppose $Y$ is small.
 If $t_\delta\neq 3$ and $\J(Y|\delta)\sats$``$\delta$ is not Woodin'', then $Q(\Tt_\delta\rest\lambda,[0,\lambda)^{\Tt_\delta})\ins N_\delta$ for each limit $\lambda\leq\lh(\Tt_\delta)$,
 so $\Tt_\delta$ is determined by $N_\delta$ via Q-structures,
 \item\label{item:it_to_bkgd_Y_small_stage_delta,t_delta=3} Suppose $Y$ is small.
 If $t_\delta=3$ or $\J(Y|\delta)\sats$``$\delta$ is Woodin'',
 then letting $\lambda=\card^Y(\delta)$ and $Y|\zeta$ be the Q-structure for $\lambda$,
 then $\lh(\Tt_\delta)=\lambda+1$
 and $Q(\Tt_\delta\rest\lambda,[0,\lambda)^{\Tt_\delta})=N_\zeta$,
 so $\Tt_\delta$ is determined by $N_\zeta$ via Q-structures,
 \item\label{item:it_to_bkgd_Y_non-small_stage_delta} Suppose $Y$ is non-small and is a Q-structure for $\delta$.
 Then $Q(\Tt_\delta\rest\delta,[0,\delta)_{\Tt_\delta})=N_{\OR^Q}$,
 and therefore $\Tt_\delta$ is determined by $N_{\OR^Q}$ and Q-structures.
\end{enumerate}
\end{tm}
\begin{proof}
We prove by induction on $\zeta<\delta$ such that $t_{\zeta+\om}\in\{0,1\}$,
 the assertion $(*)_\zeta$: parts \ref{item:Tt_eta_exists}, 
\ref{item:eta<delta_Tt_in_below_delta}, \ref{item:strategy1} and \ref{item:strategy2} hold for all 
$\eta\leq\zeta$.
 In other words,
 in the induction we treat the two kinds of intervals below in single steps:
 \begin{enumerate}[label=\tu{(}\roman*\tu{)}]
  \item\label{item:Woodin}  $[\delta^Q,\OR^Q]$ where $Q$ is $Y$-cardinal-large and 
$Q$ is the Q-structure for $\delta^Q$,
\item\label{item:measurable} $[\kappa,\eta]$ where $\kappa$ is $Y$-measurable and 
$\eta=(\kappa^+)^{j(N_\kappa)}$ where $j=i^Y_{F^Y_{\kappa 0}}$.
\end{enumerate}

(We can have $Q$ which is $Y$-cardinal-large
and $\delta^Q$ is also $Y$-measurable, so there is some overlap between the two cases,
but that is fine.) Given that part \ref{item:Tt_eta_exists} holds for $\eta$,
note that part \ref{item:strategy1} follows immediately for $\eta$ from smallness and standard 
facts. And part \ref{item:eta<delta_Tt_in_below_delta} at $\eta$ will follow from  
parts \ref{item:strategy1}, \ref{item:strategy2} at $\eta$ and since $\mu$ is a limit cardinal of $Y$. So it just remains to prove parts 
\ref{item:Tt_eta_exists} and \ref{item:strategy2} at $\eta$.
(After completing the induction, part \ref{item:uniformity} and \ref{item:final_tree_small}
follow routinely, and part \ref{item:final_tree} is by the proof of part \ref{item:strategy2} below,
together with the fact that if $Y$ is non-small then $[0,\delta]_{\Tt_\delta}\notin 
Y$, because otherwise $\delta$ would be singularized. Parts \ref{item:it_to_bkgd_Y_small_stage_delta,t_delta_neq_3}, \ref{item:it_to_bkgd_Y_small_stage_delta,t_delta=3}, \ref{item:it_to_bkgd_Y_non-small_stage_delta} will follow from some further similar considerations, which we leave to the reader.) We now begin the induction.\\

We set $\Tt_\om$ to be the trivial tree on $M$, giving $(*)_\om$.

Suppose $\zeta<\delta$, $t_{\zeta+\om}\in\{0,1\}$, and $(*)_\zeta$ holds.
In particular we 
have $\Tt_\zeta\in Y$ and $N_\zeta\ins M^{\Tt_\zeta}_\infty$, etc. Let $\Tt=\Tt_\zeta$.
Note that  $t_{\zeta+\om+\om}=0$, so we have to prove $(*)_{\zeta+\om}$,
in particular defining $\Tt_{\zeta+\om}$.

\begin{case} $N_\zeta$ is not sound.

Note that $t_{\zeta+\om}=0$,
 $N_{\zeta+\om}=\J(\core_\om(N_\zeta))$,
and $N_\zeta=M^{\Tt}_\infty$.
If $M=\core_\om(N_\zeta)$ then $M\pins N_\alpha$ for all 
$\alpha>\zeta$, so we are done. Otherwise $b^{\Tt}$ drops in model,
and letting $\alpha+1\in b^{\Tt}$ be the last drop and
$\gamma=\pred^{\Tt}(\alpha+1)$, we have $\core_\om(N_\zeta)\pins M^{\Tt}_\gamma$.
So setting $\Tt_{\zeta+\om}=\Tt\rest(\gamma+1)$, we have $N_{\zeta+\om}\ins 
M^{\Tt_{\zeta+\om}}_\gamma$, as desired. The other requirements follow easily by construction of 
$\Tt_{\zeta+\om}$.\end{case}
\begin{case} $N_\zeta$ is sound and $t_{\zeta+\om}=0$.

So $N_{\zeta+\om}=\J(N_\zeta)$.
So if $N_\zeta=M$ then we are done.
Otherwise, because $N_\zeta$ is sound, note that $N_\zeta\pins M^{\Tt}_\infty$,
so $N_{\zeta+\om}\ins M^{\Tt_\zeta}_\infty$, so $\Tt_{\zeta+\om}=\Tt_\zeta$ works.
\end{case}
\begin{case} $N_\zeta$ is sound and $t_{\zeta+\om}=1$.

So $N_\zeta$ is passive and $N_{\zeta+\om}=(N_\zeta,E)$ is active.
Because $N_\zeta\ins M^{\Tt_\zeta}_\infty$, we just have to see that $E$
is used in $\Tt_\zeta$.
Let $F=F^Y_{\zeta+\om}$ and $\kappa=\crit(F)$ and $j=i^Y_F$ and
$k=i^Y_{F^Y_{\kappa0}}$. 
By \ref{lem:measurable_stack} and \ref{lem:limit_proj_across_meas},
we have $\chi^Y_\kappa<\zeta$, $N_\zeta|\kappa=N_\kappa$ and
\footnote{Because of this point, we do not have the consider the possibility
that $(\kappa^+)^{N_\zeta}<(\kappa^+)^{j(N_\kappa)}$.
For more standard $L[\es]$-constructions, one needs to deal with this possibility,
and for this one can use Theorem \ref{thm:ISC_for_submeasures}.}
\[ 
P\eqdef j(N_\kappa)|(\kappa^+)^{j(N_\kappa)}=
k(N_\kappa)|\chi^Y_\kappa
=N_\zeta|(\kappa^+)^{N_\zeta}=M^{\Tt_\zeta}_\kappa|(\kappa^+)^{M^{\Tt_\zeta}_\kappa}.\]
Let $\Uu=j(\Tt_\zeta\rest(\kappa+1))$. 
Calculations from the argument that comparison terminates
show that
$i^\Uu_{\kappa,j(\kappa)}\rest P=j\rest P$.
But $E\rest\nu_E\sub F\rest\nu_F$, and letting $\alpha+1=\min((\kappa,j(\kappa)]_\Uu)$, it follows 
that $E\rest\nu_E\sub E^\Uu_\alpha\rest\nu(E^\Uu_\alpha)$. So considering the ISC, either
\begin{enumerate}[label=(\roman*)]\item\label{item:E=E^Uu_alpha_it_to_bkgd_proof} $E=E^\Uu_\alpha$, or
 \item\label{item:E_in_E^Uu_alpha_it_to_bkgd_proof}
 $E\in\es^{M^\Uu_\alpha}$, or
 \item\label{item:E_in_Ult_it_to_bkgd_proof}letting $\rho=\nu_E$, $M^\Uu_\alpha|\rho$ is active with extender $G$
and $E\in\es^U$ where $U=\Ult(M^\Uu_\alpha|\rho,G)$.
\end{enumerate}

We have $N_\zeta|\kappa=M^{\Tt_\zeta}_\kappa|\kappa=M^\Uu_\kappa|\kappa$,
so if \ref{item:E=E^Uu_alpha_it_to_bkgd_proof} or \ref{item:E_in_E^Uu_alpha_it_to_bkgd_proof} above hold, $N_\zeta=M^\Uu_\alpha||\OR^{N_\zeta}$,
and if \ref{item:E_in_Ult_it_to_bkgd_proof} holds, $N_\zeta|\rho=M^\Uu_\alpha||\rho$ and $N_\zeta=U||\OR^{N_\zeta}$
(where $\rho,U$ are as there). Let $\lambda$ be the largest limit ordinal such that 
$\chi=\delta(\Tt_\zeta\rest\lambda)<\OR^{N_\zeta}$.
So $\lambda\geq\kappa$ and
by the remarks above, $N_\zeta|(\chi^+)^{N_\zeta}=M^\Uu_\alpha||(\chi^+)^{N_\zeta}$.
Because $N_\zeta$ is small, $\Tt_\zeta$ left all the Q-structures behind used to guide it, so $\Uu\rest(\lambda+1)=\Tt_\zeta\rest(\lambda+1)$,
so $\Uu\rest(\lambda+\om)$ is via $\Sigma$, and note that 
$\OR^{N_\zeta}<\delta(\Uu\rest(\lambda+\om))$. So if \ref{item:E=E^Uu_alpha_it_to_bkgd_proof} or \ref{item:E_in_E^Uu_alpha_it_to_bkgd_proof} hold
then $N_{\zeta+\om}\ins M^\Uu_{\lambda+n}$ for some $n<\om$, so $E$ is used in $\Tt_\zeta$, as desired.
And if instead \ref{item:E_in_Ult_it_to_bkgd_proof}  holds then the last two extenders used in $\Tt_\zeta$
are $G$ and $E$, as desired.
\end{case}

Now suppose $\xi$ is a limit of limits $\zeta$ such that $t_{\zeta+\om}\in\{0,1\}$,
and $(*)_\zeta$ holds at all such $\zeta<\xi$.
By parts \ref{item:strategy1} and \ref{item:strategy2} below $\xi$, we have 
$\left<\Tt_\zeta\right>_{\zeta<\xi}\in Y$. Note that $t_\xi=0$.
Let $\Tt=\liminf_{\alpha<\xi}\Tt_\alpha$.
That is, an extender $E$ is used in $\Tt$ iff $E$ is used in eventually all $\Tt_\alpha$,
for $\alpha<\xi$. Let $N=N_\xi$. Note that $N$ is passive.

\begin{case}
 $\lh(\Tt)$ is a successor.
 
 Note then that $N=M^\Tt_\infty||\OR^N$.
 So if $N\ins M^\Tt_\infty$ then we set $\Tt_\xi=\Tt$,
 and otherwise letting $E=F^{M^\Tt_\infty|\OR^N}$,
 set $\Tt_\xi=\Tt\conc\left<E\right>$.
\end{case}

\begin{case} $\lambda=\lh(\Tt)$ is a limit.

Note that $N=M(\Tt)$ and $N$ has no largest cardinal. Therefore $t_{\xi+\om}\neq 1$. We set 
$\Tt_\xi=\Tt\conc b$ where $b=\Sigma(\Tt)$, so we have $N\ins M^\Tt_\lambda$ as required.

\begin{scase} Either $N\not\sats\ZFC$ or $\J(N)\sats$``$\xi$ is not Woodin''.
 
 Then $N$ determines $b$, so $\Tt_\xi\in Y$, and we have $(*)_\xi$. 
 Note that $t_{\xi+\om}\neq 3$.
 If $t_{\xi+\om}=0$ then we are done. So suppose $t_{\xi+\om}=2$,
so $\xi$ is $Y$-measurable and 
we need to establish that $(*)_\eta$ holds for $\eta\in[\xi,\chi]$ where $\chi=\chi^Y_\xi$.
We have $\OR^N=\xi=\lambda$
 and $N\pins M^{\Tt_\xi}_\xi$. Now
 \[ N_\eta\ins i^Y_{\xi0}(N_\xi)|\chi=M^{\Tt}_\xi||(\xi^+)^{M^\Tt_\xi}.\]
So either:
\begin{enumerate}[label=--]\item $N_\eta\pins M^\Tt_\xi$,
in which case set $\Tt_\eta=\Tt_\xi$, or
\item $\eta=\chi=\OR(M^\Tt_\xi)$ and $M^\Tt_\xi$ is active with extender $E$,
in which case set $\Tt_\chi=\Tt_\xi\conc\left<E\right>$.
\end{enumerate}
 
\end{scase}

\begin{scase} $N\sats\ZFC$ and $\J(N)\sats$``$\xi$ is Woodin''.
 
By Lemma \ref{lem:Woodin_exactness}, and since $t_\xi=0$, then $\xi$ is a 
$Y$-cardinal and $\J(Y|\xi)\sats$``$\xi$ is Woodin''. Let $Q=Q^Y_\xi$ and $\zeta=\OR^Q$. Then
$N_\zeta$ is non-small,
$\lambda=\xi=\delta^Q=\delta^{N_\zeta}$, $N\pins N_\zeta$ and $N_\zeta$ is the above-$
\xi$-iterable
$\xi$-sound Q-structure for $\xi$.
Therefore $N_\zeta\ins M^{\Tt_\xi}_\xi$, as required. Set $\Tt_\alpha=\Tt_\xi$ for each 
$\alpha\in[\xi,\zeta]$.

Now if $\xi$ is not $Y$-measurable, then $t_{\zeta+\om}=0$, so we are done.
If instead $\xi$ is $Y$-measurable and $\chi$ is as above then the interval 
$(\zeta,\chi]$ is dealt with as before.
\end{scase}
\end{case}

This completes the induction below $\delta$.

As mentioned earlier, the remaining parts are established similarly, and the details are left to the reader.
\end{proof}

\subsubsection{Lower semicontinuity with respect to the $\Omega_x$-norm}\label{sec:lower_semicont_for_Omega_x}

We now proceed to describing the main ideas for the proof of lower semicontinuity, by considering the case of the ``$M_x$-goodness norm'' and the ``$\Omega_x$-norm'' (the norms that compare reals $x,y$ according to whether,
and if so, at what stage, $M_x,M_y$ get produced by $\CC$).

Let $M_{x_n}\to M_\infty$ modulo the mouse scale norms, where $x_n\to x$. We will write $\wt{M}_x=M_\infty$,
since $M_\infty$ is a candidate
for $M_x$ (and in some cases we  already know that $M_\infty=M_x$).
Fix a mouse $P$ over a real $y$ such that $\left<M_{x_n}\right>_{n<\om},\wt{M}_x\leq_Ty$
and $P$ is $M_{x_n}$-good for all sufficiently large $n$.
We may assume $P$ is $M_{x_n}$-good for all $n$.
We will  show that \begin{equation}\label{eqn:lower_semic_first_goal} P\text{ is } \wt{M}_x\text{-good and }\Omega^P_{\wt{M}_x}\leq\lim_{n\to\om}\Omega^P_{M_{x_n}}.\end{equation}
Let $\Omega=\lim_{n\to\om}\Omega^P_{M_{x_n}}$.

Using Theorem \ref{tm:iter_to_bkgd},
for $\alpha\leq\Omega$ let $\Tt^P_{x\alpha}$ be the tree on $\wt{M}_x$ iterating out to $N^P_{x\alpha}$.
So $N^P_{x\alpha}\ins  M_\infty^{\Tt_{x\alpha}}$,
and either:
\begin{enumerate}[label=(\roman*)]
 \item\label{item:lower_semi_c_P_is_Q-mouse} $P$ is a Q-mouse,
 $\Omega=\OR^P$, 
 $\left<\Tt^P_{x\alpha}\right>_{\alpha<\delta^P}\in P$,
 $\Tt^P_{x\delta^P}=\Tt^P_{x\alpha}$ for all $\alpha\in[\delta^P,\OR^P]$,
 $\lh(\Tt^P_{x\delta^P})=\delta^P+1$ and $\Tt^P_{x\delta^P}\rest\delta^P\in P$; in this case let $q$ be the Q-degree of $P$ (see Definition \ref{dfn:Q-degree}), or
 \item\label{item:lower_semi_c_maybe_not_Q-mouse} $\aleph_\Omega^P<\rho_0^P$
 and $\left<\Tt^P_{x\alpha}\right>_{\alpha\leq\Omega}\in P$; in this case let $q=0$.
\end{enumerate}

So  $P$ is $(q,\omega_1+1)$-iterable.
Let $m=\projdeg(\wt{M}_x)$. So
in case \ref{item:lower_semi_c_P_is_Q-mouse}, we have $m\leq q$.
Let $(\Sigma_\infty,\frakL_\infty)$ be the $\Gamma$-limit $(m,q,\om_1+1)$-lifting algorithm
for $(\wt{M}_x,P)$ given by \ref{lem:...}.
Let $\Tt$ on $\wt{M}_x$ be via $\Sigma_\infty$ and $\alpha<\lh(\Tt)$.
We write $\Uu^\Tt,\abliftprodstage^\Tt_\alpha(S)$ for $S\ins M^\Tt_\alpha$, etc, for the objects given by $\frakL_\infty$.

The plan is to form a  tree $\Tt$ on $\wt{M}_x$,  via $\Sigma_\infty$, of successor length $\varepsilon+1$, such that letting
$\Uu=\Uu^\Tt$ and $P'=M^\Uu_{\varepsilon 0}$, we have:
\begin{equation}
\Tt\text{ is small,
  so }\Uu\text{ is small and nowhere dropping in model/degree, and }\end{equation}
  \begin{equation}
  \Tt=\Tt^{P'}_{x\Omega'}
\text{ where }\Omega'=i^\Uu(\Omega)\text{ 
(where }i^\Uu(\Omega)\text{ denotes }\OR^{P'}\text{ if }\Omega=\OR^P\text{).}\end{equation}
 So letting
 $M'=M^\Tt_\varepsilon$, we will have $N^{P'}_{\Omega' x}\ins M'$.

We just need to explain how to select the extenders for $\Tt$, as together with the requirements that $\Uu=\Uu^\Tt$ and $\Tt$ be via $\Sigma_\infty$, this determines $\Tt,\Uu$ completely. 
We write  $k_\alpha=k^\Tt_\alpha$,
$P_\alpha=P_{\alpha 0}$, $P_{\alpha i}=M^\Uu_{\alpha i}$ (for $i\leq k_\alpha$) and $(E^\Tt_\alpha)^+=E^\Uu_{\alpha k_\alpha}$ (so $(E^\Tt_\alpha)^+$
 is the limit-lift of $E^\Tt_\alpha$ to a background extender in $P_{\alpha k_\alpha}$).

We proceed by induction on $\alpha$.  Suppose we have defined $\Tt\rest(\alpha+1)$,
and $\Tt\rest(\alpha+1)\in P_\alpha$.

Let $R\pins M^\Tt_\alpha$ and $\beta\in\OR$.
We say that $R$ is \emph{produced directly} at $\beta$ (with respect to $P_\alpha$)
iff $R=\core_\om(N_{x\beta}^{P_\alpha})$
(here $\rho_\om^R$ need not be a cardinal of $M^\Tt_\alpha$, but note that $R$ and $\beta$ determine one another uniquely).
(***I think the abstract notions here should be shifted to where the other definitions are...)
When this holds, write \[\prodstage^{P_\alpha}_x(R)=\beta_R=\beta.\]

We say that $R$ is \emph{produced late} iff $R$ is produced directly at $\beta$ and letting $\rho=\card^{M^\Tt_\alpha}(R)$
and $\beta'=\abliftprodstage^\Tt_\alpha(R)$\footnote{Might have to put these definitions formally earlier, if they're not there yet. It means that we do the full limit-lift/resurrection process for $R$,
back to the ``limit $N_{x_n\beta'}^{M^{\Vv}_\infty}$'', where $\Vv$ is the resulting limit resurrection tree on $P_\alpha$; $\beta'$ is the limit stage in $M^\Vv_\infty$ where $R$ is limit-embedded; and if $\lim_{n\to\om}t^{M^\Vv_\infty}_{\beta'}=1$,
then $\xi'=\lim_{n\to\om}\lh(F^{M^\Vv_\infty}_{\beta'})$; otherwise $\xi'=0$.},
either:
\begin{enumerate}
 \item the limit-resurrection
 of $R$ (with respect to $\Tt,\alpha,P_\alpha$) uses some order 0 measure, and letting $E$ be least such (so note that $\rho\leq\crit(E)$), either:
 \begin{enumerate}
 \item $\rho<\crit(E)\leq\beta$, or
 \item\label{item:crit(E)=rho_and_beta>beta'} $\crit(E)=\rho$ and $\beta>\beta'$,
\end{enumerate}
or
\item the limit-resurrection of $R$ does not use any measures,
and $\beta>\beta'$.
\end{enumerate}

Let us digress to make an illustrative remark regarding these objects that one should keep in mind.
Suppose that $R\pins\core_0(S)$ and $S\pins M_\alpha$ and $\rho_\om^R$ is an $S$-cardinal
and $\rho_\om^S<\rho_\om^R$
and $S$ is produced directly at $\gamma$ and $R$ is produced directly at $\beta$. Let
\[ \tau:\core_0(S)\to\core_0(N^{P'}_{x\gamma}) \]
be the core map. Note that $\beta<\gamma$ and in fact $\gamma$ is the least $\gamma'>\beta$
such that
\[ \rho_\om(N^{P'}_{x\gamma'})<\rho_\om^R,\]
so $R\pins N^{P'}_{x\gamma}$.
Let \[ \beta'=\prodstage^{P'}_x(\tau(R)).\]
Then note $\beta\leq\beta'$; in fact either $\crit(\tau)\leq\rho$ and $R\pins\tau(R)$\footnote{If $\crit(\tau)=\rho$ then easily $R\pins\tau(R)$. Suppose $\crit(\tau)<\rho$. A standard comparison argument
shows that $\crit(\tau)=\crit(E)$
for some $E\in\es_+^S$
which is $S$-total,
and therefore $\tau(\crit(\tau))$ is measurable in $P'$. But $S$ is sound with $\rho_\om^S<\rho_\om^R$, and therefore $\rho_\om^R$ is not a cardinal in $P'$,
so $\tau(\crit(\tau))>\rho_\om^R$, so $\tau(\rho_\om^R)>\rho_\om^R$,
and it follows that $R\pins\tau(R)$.} and $\beta<\beta'$,
or $\crit(\tau)>\rho$ and $R=\tau(R)$ and $\beta=\beta'$. So the stage $N_{x\beta}^{P'}$ is a kind of resurrection of $R$,
but may be earlier than a full resurrection which passes through the dropdown sequence.
However, $\abdirectprodseg^\Tt_\alpha(R)$ corresponds to full resurrection (in the limit lifting sense).

Let us now, in the case that $\alpha=0$, define $E^\Tt_0$, or declare that $\Tt$ is trivial. There are three cases. 
 Note that $\Tt^P_{x\Omega}$ is trivial iff $N^P_{x\Omega}\ins \wt{M}_x$.

\begin{case} $\Tt^P_{x\Omega}$ is trivial and no $R\pins \wt{M}_x$ with
$\beta_R<\Omega$ is produced late.

Then we set $\Tt$ to be trivial.
\end{case}

\begin{case}\label{case:E_0_no_late_R} $\Tt^P_{x\Omega}$ is non-trivial and no $R\pins \wt{M}_x$ with $\beta_R<\Omega$ is produced late.

Then we set $E^\Tt_0=E^{\Tt^P_{x\Omega}}_0$.
\end{case}

\begin{case}\label{case:E_0_late_R} Otherwise (some $R\pins \wt{M}_x$ with $\beta_R<\Omega$
 is produced late).

Fix $R$ least such and let $\beta=\beta_R$ and $\Tt_R=\Tt^P_{x\beta}$.
We will verify below that $\Tt_R$ is non-trivial,
and we set $E^\Tt_0=E^{\Tt_R}_0$.
\end{case}

\begin{clm}\label{clm:Tt_R_non-triv}
If Case \ref{case:E_0_late_R} above attains,
and $R$ is least as there, then
$\Tt_R$ (as there) 
 is non-trivial.
\end{clm}
\begin{proof}
 Suppose not. Then  $R=N^P_{x\beta}\ins\wt{M}_x$ is sound.
 Because $R\ins \wt{M}_x\in P$, it easily follows that $R$ is passive and $t^P_{x\beta}=0$.
 Let
 \[ C=\{S\mid S\pins R\text{ and }\rho_\om^S\text{ is an }R\text{-cardinal}\}. \]
 Because $R=N^P_{x\beta}$, each $S\in C$ is produced directly at some $\gamma$.
 
 \begin{case}$R=\J(S)$ for some $S\pins R$.
 
 Then note that $S\in C$ and
 \[ \prodstage^P_x(R)=\prodstage^P_x(S)+\om.\]
 By the minimality of $R$, either:
 \begin{enumerate}[label=(\roman*)]
  \item\label{item:order_0_meas_used_in_limit-lift(S)} an order $0$ measure $E$ is used in the limit-lift of $S$, and  $\prodstage^P_x(S)<\crit(E)$ where $E$ is least such, or
  \item\label{item:no_order_0_meas_used_in_limit-lift(S)} no order $0$ measure used in the limit-lift of $S$, and $\prodstage^P_x(S)\leq\abliftprodstage^\Tt_0(S)$.
 \end{enumerate}
Suppose \ref{item:order_0_meas_used_in_limit-lift(S)} holds. So $\prodstage^P_x(R)=\prodstage^P_x(S)+\om<\crit(E)$,
where $E$ is least as there. Therefore by the choice of $R$,
$E$ is not used in the limit-lift of $R$.
It follows that $\rho_\om^R<\rho_\om^S$,
and the $\crit(E)=\pi(\rho_\om^S)$
where $\pi$
is the limit-lift map on $\OR^R$.
But the limit-lift process for $R$ is just an initial segment of that for $S$,
and for large enough $n<\om$,
we then get $\xi$ such that $t^{P}_{x_n\xi}=2$
with $\card^P(\xi)=\crit(E)$ (the stage corresponding to the version of the limit-lift of $S$ prior to taking the $E$-ultrapower) but with $\rho_\om(N^P_{x_n,\xi+\om})<\crit(E)$ (this gives the limit-lift of $R$), which is impossible.

So \ref{item:no_order_0_meas_used_in_limit-lift(S)}
 holds. But $\rho_\om^R\leq\rho_\om^S$ and  the limit-lift process for $R$ is an initial segment of that for $S$, so no order $0$ measure is used within that for $R$. So \[\abliftprodstage^\Tt_0(R)=\abliftprodstage^\Tt_0(S)+\om,\]
 so \[\prodstage^P_x(R)\leq\abliftprodstage^\Tt_0(R),\]
 contradiction.
\end{case}
\begin{case}
$R$ has a largest cardinal $\rho$,
 but no largest proper segment.

Let $C_\rho$ be the set of all $S\in C$ such that $\rho_\om^S=\rho$.
Note that
 \[ \prodstage^P_x(R)=\sup_{S\in C_\rho}\prodstage^P_x(S).\]

Suppose an order $0$ measure is used in the limit-lift of $R$, and let $E$ be least such.
So $\beta_R=\prodstage^P_x(R)>\crit(E)$. Note
that $E$ is also used in the limit-lift of $S$,
for each $S\in C_\rho$,
and so $\prodstage^P_x(S)<\crit(E)$ for each such $S$. But then $\prodstage^P_x(R)<\crit(E)$, contradiction.

So there is no order $0$ measure used in the limit-lift of $R$. If there is also no order $0$ measure used in the limit-lift of $S$,
for $S\in C_\rho$,
then note that
\[ \prodstage^P_x(R)=\sup_{S\in C_\rho}\prodstage^P_x(S) \]
\[ \leq\sup_{S\in C_\rho}\abliftprodstage^\Tt_0(S) \]
\[ \leq\abliftprodstage^\Tt_0(R),\]
contradiction.
Note then that there is $E$ such
that for each $S\in C_\rho$,
$E$ is the least order $0$ measure
used in the limit-lift of $S$,
and $\crit(E)=\pi(\rho)$
where $\pi$ is the limit-lift map for $R$ on $\dom(\pi)=\OR^R$.
So $\prodstage^P_x(S)<\crit(E)$
for each such $S$.
But then
\[ \prodstage^P_x(R)=\big(\sup_{S\in C_\rho}\prodstage^P_x(S)\big)<\crit(E)<\abliftprodstage^\Tt_0(R), \]
contradiction.\end{case}

\begin{case}
 $R$ has no largest cardinal.
 
We have \[\prodstage^P_x(R)=\sup_{S\in C}\prodstage^P_x(S).\]
Suppose some order $0$ measure is used in the limit-lift of $R$, and let $E$ be least such;
so $\prodstage^P_x(R)>\crit(E)$.
But for all sufficiently large $R$-cardinals $\rho<\OR^R$, and all $S\in C$ with $\rho_\om^S=\rho$, $E$ is also the least used in the limit-lift of $S$. So $\prodstage^P_x(S)<\crit(E)$.
But then $\prodstage^P_x(R)<\crit(E)$, contradiction. So there is no such $E$ used in the limit-lift of $R$. But then for all sufficiently large $R$-successor-cardinals $\rho$,
if $S\in C$ and $\rho_\om^S=\rho$,
then there is no such $E$ used in the limit-lift of $S$. It follows that
\[ \prodstage^P_x(R)=\sup_{S\in C}\prodstage^P_x(S) \]
\[ \leq\sup_{S\in C}\abliftprodstage^\Tt_0(S)\]
\[\leq\abliftprodstage^\Tt_0(R),\]
contradiction.\end{case}

This completes all cases, and hence the proof.
\end{proof}

\begin{rem}
 In Case \ref{case:E_0_late_R}, if $\Tt^P_{x\Omega}$ is non-trivial then
 letting $R\pins \exit^{\Tt^P_{x\Omega}}_0$,
 so $\lh(E^\Tt_0)<\lh(E^{\Tt^P_{x\Omega}}_0)$.
 For otherwise $\exit^{\Tt^P_{x\Omega}}_0\ins R=\core_\om(N_{\beta_Rx}^P)$,
 but since $\beta_R<\Omega$, this is impossible.
\end{rem}

Now assuming that $\Tt$ is non-trivial (equivalently, either Case \ref{case:E_0_no_late_R} or Case \ref{case:E_0_late_R} above attained at stage $0$), set $\eta_1$ to be the unique $\eta$ such that $N_{x\eta}^P=\wt{M}_x||\lh(E_0^\Tt)$
(this $\eta$ exists as $\wt{M}_x||\lh(E^\Tt_0)$ is a cardinal segment of $N^P_{x\Omega}$
in Case \ref{case:E_0_no_late_R}, and is a cardinal segment of $N^P_{x\beta_R}$
in Case \ref{case:E_0_late_R}). Recall that since $E^\Tt_0$ was used in either $\Tt^P_{x\Omega}$
or $\Tt^P_{x\beta_R}$, and by Theorem \ref{tm:iter_to_bkgd}, $\exit^\Tt_0$ is small.

In general along with $\Tt\rest(\alpha+1)$,
we will also define
$\left<\eta_\beta\right>_{0<\beta\leq\alpha}$,
 and defining $\Omega_\beta=i^\Uu_{00,\beta0}(\Omega)$, we
will verify that ($*$)$_\alpha$ holds, which asserts:
\benum
\item\label{item:star_Tt_small} $\Tt\rest(\alpha+1)$ is  $\om$-maximal
and small.
\item\label{item:star_eta_str_inc_disc} $\left<\eta_\beta\right>_{0<\beta\leq\alpha}$ is strictly increasing and continuous.
\item\label{item:star_N_eta_beta} For each $\beta\in(0,\alpha]$, we have $\eta_\beta\leq \Omega_\beta=i^\Uu_{00,\beta0}(\Omega)$ and 
 \[N^{P_{\beta}}_{x\eta_\beta}=M^\Tt_\beta|\sup_{\gamma<\beta}\lh(E^\Tt_\gamma) \]
(so if $\beta=\gamma+1$
then \[N^{P_{\gamma+1}}_{x\eta_{\gamma+1}}=M^\Tt_{\gamma+1}|\lh(E^\Tt_\gamma)=M^\Tt_\gamma||\lh(E^\Tt_\gamma), \]
and if $\beta$ is a limit ordinal
then $N^{P_\beta}_{x\eta_\beta}=M(\Tt\rest\beta)$).
\item\label{item:star_coherence}  For each $\beta<\alpha$,  letting $\eta=\eta_{\beta+1}$ and $\eta'=\eta+\om$ and $i\leq k_{\beta}$, we have:
\benum
\item\label{item:eta<crit(F_kappa,0)} if $i<k_{\beta}$ and $E^\Uu_{\beta i}\neq\emptyset$ then $\eta<\crit(F_{\beta i})$ (hence $\aleph_{\eta}^{P_{\beta i}}<\crit(F_{\beta i})$),
\item\label{item:aleph_eta<nu(F)} $\aleph_{\eta}^{P_{\beta i}}\leq\nu(E^\Uu_{\beta k_\beta})$,
\item\label{item:bkgd_agmt_thru_aleph_eta}$P_\alpha|\aleph_{\eta}^{P_\alpha}=P_{\beta i}|\aleph_{\eta}^{P_{\beta i}}$
\item\label{item:construction_agmt_thru_eta+1} $\eta'<\Omega^{P_\alpha}$
and $\CC^{P_{\alpha}}_x\rest(\eta+1)=\CC^{P_{\beta i}}_x\rest(\eta+1)$ and $t^{P_{\alpha}}_{x\eta'}=0=t^{P_{\beta i}}_{x\eta'}$, so
\item\label{item:iterate_is_construction_through_eta} $N^{P_{\alpha}}_{x\eta}=N^{P_{\beta i}}_{x\eta}=M^\Tt_\beta||\lh(E^\Tt_\beta)=M^\Tt_\alpha|\lh(E^\Tt_\beta)$
is fully sound and $N^{P_{\alpha}}_{x,\eta'}=\J(N^{P_\alpha}_{x\eta})$.
\eenum
\item\label{item:star_no_late_below_delta_beta} No $R\ins M^\Tt_\alpha|\sup_{\beta<\alpha}\lh(E^\Tt_\beta)$ is produced late with respect to $\Tt,\alpha,P_\alpha$.
\eenum

Note that $(*)_0$ is trivial.

Clearly $(*)_\alpha$ implies that for all limits $\gamma\leq\alpha$,
setting $\eta=\eta_{\gamma}$ and $\eta'=\eta+\om$ and $\delta=\delta(\Tt\rest\gamma)$, we also have:
\benumdd
\item Either $\delta$ is a cardinal of $M^\Tt_\gamma$ and of $M^\Tt_\alpha$,
or $\alpha=\gamma$ and $\delta=\OR^{M^\Tt_\alpha}$,
\item $\CC^{P_\alpha}_x\rest(\eta+1)=\CC^{P_\gamma}_x\rest(\eta+1)=\bigcup_{\beta<\gamma}\CC^{P_\beta}_x\rest(\eta_\beta+\om+1)$,
\item $N^{P_\alpha}_{x\eta}=N^{P_\gamma}_{x\eta}=M(\Tt\rest\gamma)=M^\Tt_\gamma|\delta=M^\Tt_\alpha|\delta$,
\item $t^{P_\alpha}_{x\eta}=t^{P_\gamma}_{x\eta}=0$ and
$t^{P_\alpha}_{x,\eta'},t^{P_\gamma}_{x,\eta'}\neq 1$ (as $M(\Tt\rest\gamma)$ has no largest cardinal).
\item No $R\ins M(\Tt\rest\gamma)$ is produced late (with respect to $\Tt,\alpha,P_\alpha$).
\eenum
But it is possible that $t^{P_\alpha}_{x,\eta'}\in\{2,3\}$,
and also possible that $t^{P_\gamma}_{x,\eta'}=2$ but $t^{P_\alpha}_{x,\eta'}=0$.

We have:

\begin{clm}
$(*)_1$ holds.
\end{clm}
\begin{proof}
We omit this as it is almost the same as the proof of Claim \ref{clm:*_alpha+1_holds} below.
\end{proof}

Now suppose we have defined $\Tt\rest(\alpha+1)$ and $\left<\eta_\beta\right>_{0<\beta\leq\alpha}$
and ($*$)$_\alpha$ holds.
We define $\Tt\rest(\alpha+2)$ and $\eta_{\alpha+1}$.

Let  $\delta=\sup_{\beta<\alpha}\lh(E^\Tt_\beta)$ and $\eta=\eta_\alpha$.
So $\eta\leq\Omega_\alpha$ and $M^\Tt_\alpha|\delta=N^{P_\alpha}_{x\eta}$
and $t^{P_\alpha}_{\eta+\om,x}\neq 1$.
Let
\begin{equation}\label{eqn:def_theta_in_E^Tt_alpha_def} \theta=\text{least }\theta'\in[\eta,\Omega_\alpha]\text{ such that }\theta'=\Omega_\alpha\text{ or }N_{x\theta'}^{P}\text{ projects }<\delta.\end{equation}
Let $\Tt^*=\Tt^P_{x\theta}$.
So $\Tt\rest(\alpha+1)\ins\Tt^*$.
We have $M^{\Tt^*}_\infty=N_{x\theta}^P$,
because $N_{x\theta}^P\npins M^{\Tt^*}_\infty$, because if $\delta=\OR^{M^\Tt_\alpha}$ then $\theta=\eta_\alpha$ (as $\Tt$ is  on $\wt{M}_x$, so $M^\Tt_\alpha$ is unsound),
and if $\delta<\OR^{M^\Tt_\alpha}$ then $\delta$ is a cardinal of $M^\Tt_\alpha$ and of $M^{\Tt^*}_\infty$.

We now define $E^\Tt_\alpha$ or declare the process complete; again there are multiple cases:
\begin{case}  $\Tt\rest(\alpha+1)\notin P_\alpha$ or $\theta=\eta_\alpha$.

Then the process is complete, so $\Tt=\Tt\rest(\alpha+1)$.
\end{case}

Now suppose that $\Tt\rest(\alpha+1)\in P_\alpha$ and $\theta>\eta_\alpha$ (so $\delta<\OR^{M^\Tt_\alpha}$).

\begin{case}\label{case:some_R_late_prior_to_gamma} Some $R\pins M^\Tt_\alpha$ is produced late, with $\beta_R<\theta$.

Fix $R$ least such and let $\beta=\beta_R$.

\begin{clm}\label{clm:t^P_beta=0,1}
 $M^\Tt_\alpha|\delta\pins R$ and $t^{P_\alpha}_{x\beta}\in\{0,1\}$.
\end{clm}

\begin{proof}
The fact that $M^\Tt_\alpha|\delta\pins R$ is by $(*)_\alpha$ (and the remarks following it). This also gives that $\beta>\eta_\alpha$.

Suppose that $t_{x\beta}^{P_\alpha}=3$. Then note that $\delta^{P_\alpha|\beta}=\delta$ because $\Tt\rest(\alpha+1)\in P_\alpha$
and is $\Tt\rest(\alpha+1)$ small
and $\delta^{P_\alpha|\beta}$ is a cardinal of $P_\alpha$. But $\rho_\om^R\geq\delta$ (as $R\pins M^\Tt_\alpha$) 
and so because $t_{x\beta}^{P_\alpha}=3$, we have $R=N^{P_\alpha}_{x\beta}$; that is, $N^{P_\alpha}_{x\beta}$ is sound.
But then $\prodstage^{P_\alpha}_x(R)=\OR^R$,
from which it is easy to reach a contradiction.

Suppose $t_{x\beta}^{P_\alpha}=2$. So $\prodseg^{P_\alpha}_x(R)=P_\alpha|\lh(F_{\kappa 0}^{P_\alpha})$ where $\kappa=\card^P(\beta)$,
and $\card^{P_\alpha}(R)=\kappa$.
Since $\Tt\rest(\alpha+1)\in P_\alpha$
and $R\pins M^\Tt_\alpha|$,
clearly $\alpha\geq\kappa$.
Now $\eta_\kappa=\kappa$.
For note that
 $\delta(\Tt\rest\kappa)=\kappa$ and $M^\Tt_\alpha|\kappa=N^{P_\alpha}_{x\eta_\alpha}|\kappa$, but if $N^{P_\alpha}_{x\eta_\alpha}|\kappa\neq N^{P_\alpha}_{x\kappa}$
 then there is no $\xi>\eta_\alpha$
 with $\card^{P_\alpha}(\xi)=\kappa$
 and $t_{\xi x}^{P_\alpha}=2$.
 So $N^{P_\alpha}_{x\eta_\alpha}|\kappa=N^{P_\alpha}_{x\kappa}$, so $\eta_\kappa=\kappa$. We have $N^{P_\alpha}_{x\chi}=M^\Tt_{\kappa}|\kappa^{+M^\Tt_\kappa}$ where $\chi=\chi_{x\kappa}^{P_\alpha}$,
 by Lemma \ref{lem:iterate_up_to_N_x,kappa}.
 But since $t^{P_\alpha}_{x\beta}=2$,
 we have $R\pins N^{P_\alpha}_{x\chi}$,
 so $R\pins M^\Tt_\kappa|\kappa^{+M^\Tt_\kappa}$,
 and therefore $\alpha=\kappa$
 (otherwise $\kappa<\lh(E^\Tt_\kappa)\leq\OR^R$).
 Let $Q\pins M^\Tt_\kappa$
 be the Q-structure for $M^\Tt_\kappa|\kappa$.
 So $Q\pins R$, since $t^{P_\alpha}_{x\beta}=2$.
 Note that if $R=\J_\xi(Q)$
 for some $\xi\in\OR$ then $\beta=\OR^Q+\xi\leq\abliftprodstage^\Tt_\kappa(R)$,
 contradicting that $R$ was produced late;
 but otherwise, there is some $\xi>\OR^Q$ such that $\xi\leq\abliftprodstage^\Tt_\kappa(R)$ and
 $t^{\Ult(P_\alpha,F_{\kappa0}^{P_\alpha})}_{\xi x}=1$, and taking the least such $\xi$, 
 \[ \OR^R=\beta<\chi^{P_\alpha}_{\kappa0}\leq\kappa^{+P_\alpha}<\xi\leq\abliftprodstage^\Tt_\kappa(R),\]
again a contradiction.
\end{proof}

\begin{clm}\label{clm:late_next_ext} $R\neq N_{x\beta}^P$.\end{clm}
\begin{proof}
A slight variant of  the proof of Claim \ref{clm:Tt_R_non-triv}.
\end{proof}

Note then that $\Tt^P_{x\beta}$ properly extends $\Tt\rest(\alpha+1)$.
Set $E^\Tt_\alpha=E^{\Tt^P_{x\beta}}_\alpha$,
and set
\[ \eta_{\alpha+1}=\prodstage^P_x(M^\Tt_\alpha||\lh(E^\Tt_\alpha)) \]
(note that $N^P_{x\beta}|\lh(E^\Tt_\alpha)=M^\Tt_\alpha||\lh(E^\Tt_\alpha)$ is a cardinal proper segment of $N^{P_\alpha}_{x\beta}$, so $\eta_{\alpha+1}$ is well-defined).

This completes this case. Note that in this case,
$E^\Tt_\alpha$ is defined, $E^\Tt_\alpha\in\es_+^R$, and $R\pins M^\Tt_\alpha$; in particular, $E^\Tt_\alpha\neq F^{M^\Tt_\alpha}$.
\end{case}

\begin{case}\label{case:no_late_R_before_gamma}
 There is no $R$ as in Case \ref{case:some_R_late_prior_to_gamma}.
 
 \begin{scase}\label{scase:no_late_next_ext}
 $\Tt\rest(\alpha+1)\neq\Tt^P_{x\theta}$.

Set $E^\Tt_\alpha=E^{\Tt^P_{x\theta}}_\alpha$ and
 $\eta_{\alpha+1}=\prodstage^P_x(N_{x\theta}^P|\lh(E^\Tt_\alpha))$.
\end{scase}

\begin{scase}
\label{scase:no_late_no_next_ext}
$\Tt\rest(\alpha+1)=\Tt^P_{x\theta}$.

Then the process is complete; we set $\Tt=\Tt\rest(\alpha+1)$.
\end{scase}

\end{case}

\begin{clm}\label{clm:*_alpha+1_holds} $(*)_{\alpha+1}$ holds.\end{clm}
\begin{proof}
Condition \ref{item:star_Tt_small}: Since $E^\Tt_\alpha=E^{\Tt^P_{x\iota}}_\alpha$ for some $\iota$, this is because $\Tt^P_{x\iota}$ is small.

Conditions \ref{item:star_eta_str_inc_disc},
\ref{item:star_N_eta_beta} are straightforward.

Condition \ref{item:star_coherence}: Part \ref{item:eta<crit(F_kappa,0)}:  Suppose $i<k_{\alpha}$ and $i$ is least such that $E^\Uu_{\alpha i}\neq\emptyset$.
Then  (since $k_\alpha>0$) letting $\kappa=\crit(E^\Uu_{\alpha i})$, we have $\exit^\Tt_\alpha\pins M^\Tt_\alpha$
and $E^\Uu_{\alpha i}=F^{P_\alpha}_{\kappa 0}$ is used in the limit-lift/resurrection of $\exit^\Tt_\alpha$. Let $\rho=\lgcd(\exit^\Tt_\alpha)$.
Let $C$ be the set of all 
$S\pins\exit^\Tt_\alpha$
with $\rho_\om^S=\rho$.
Then note that for each $S\in C$, $F^{P_\alpha}_{\kappa 0}$ is also used in the limit-lift/resurrection of $S$. But no such $S$ is produced late, 
and so $\abliftprodstage^\Tt_\alpha(S)<\kappa$.
So
 \[\eta=\prodstage^{P_\alpha}_x((\exit^\Tt_\alpha)^{\passive})=\sup_{S\in C}\prodstage^{P_\alpha}(S)<\kappa,\]
 as desired.
 
 Part \ref{item:aleph_eta<nu(F)}: We must see that
 $\aleph_{\eta}^{P_{\alpha i}}\leq\nu(E^\Uu_{\alpha k_\alpha})$
 for each $i\in [0,k_\alpha]$.
 If there is $i$ as in part \ref{item:eta<crit(F_kappa,0)} then this follows immediately from that part
 (since then in fact $\crit(E^\Uu_{\alpha i})<\nu(E^\Uu_{\alpha k_\alpha})$).
 So suppose otherwise.
 Let $C$ be as above. So $P_{\alpha i}=P_\alpha$. By Lemma \ref{lem:active_and_sub_projector_res_tree_agmt} (applied to limit $x_n$),
 either:
 \begin{enumerate}[label=(\roman*)]
  \item\label{item:no_exts_in_S-res} for all $S\in C$,
  limit-lift resurrection process for $S$
  uses no extenders, or
  \item\label{item:just_E_kappa_in_S-res} there is $\kappa$ such that for all $S\in C$,
  the limit-lift resurrection process
  for $S$ uses only the order $0$ measure on $\kappa$ in $P_\alpha$.
 \end{enumerate}
In case \ref{item:no_exts_in_S-res}, note that
 \[ \eta=\sup_{S\in C}\prodstage^{P_\alpha}_x(S)\leq\sup_{S\in C}\abliftprodstage^\Tt_\alpha(S)\leq(\abliftprodstage^\Tt_\alpha(\exit^\Tt_\alpha)-\om)\eqdef\eta', \]
 and $\nu(E^\Uu_{\alpha k_\alpha})=
 \aleph_{\eta'}^{P_\alpha}$
 (the ``${}-\om$'' is because $t_{x_n\zeta}^{P_\alpha}=1$ for $\zeta=\abliftprodstage^\Tt_\alpha(\exit^\Tt_\alpha)$ and for large $n$, so $\zeta=\eta'+\om$). In case \ref{item:just_E_kappa_in_S-res},
 with $\kappa$ as there,
 $\prodstage^{P_\alpha}_x(S)<\kappa$
 for each $S\in C$, so for large $n$,
 \[ \eta<\kappa<\chi_{x_n\kappa}^{P_\alpha}\eqdef\eta'<\chi_{x_n\kappa}^{P_\alpha}+\om=\abliftprodstage^\Tt_\alpha(\exit^\Tt_\alpha),\]
and again $\nu(E^\Uu_{\alpha k_\alpha})=\aleph_{\eta'}^{P_\alpha}$.

Part 
 \ref{item:bkgd_agmt_thru_aleph_eta} just requires that $P_{\alpha+1}|\aleph_\eta^{P_{\alpha+1}}=P_{\alpha i}|\aleph_\eta^{P_{\alpha i}}$, but this follows immediately from part \ref{item:aleph_eta<nu(F)}.
 
 Part 
 \ref{item:construction_agmt_thru_eta+1}:
With $\eta=\eta_{\alpha+1}$ as above,
the fact that $\eta+\om<\Omega^{P_\alpha}$
is easy, and using parts \ref{item:aleph_eta<nu(F)},\ref{item:bkgd_agmt_thru_aleph_eta},  \[ \CC^{P_{\alpha+1}}_x\rest(\eta+1)=\CC^{P_{\alpha i}}_x\rest(\eta+1).\]
We have $t^{P_{\alpha}}_{x\eta'}=0$ since $E^\Tt_\alpha$ was used in $\Tt^{P_\alpha}_{x\beta}$ for the appropriate $\beta$. (That is, therefore $\lh(E^\Tt_\alpha)<\OR(N_{x\beta}^{P_\alpha})$ and there is no $\xi\in[\eta,\beta)$ such that $N_{x\xi}^{P_\alpha}$ projects ${<\lh(E^\Tt_\alpha)}$.
But if $t_{x,\eta+\om}^{P_\alpha}=1$ then $N_{x,\eta+\om}^{P_\alpha}$ projects ${<}{\lh(E^\Tt_\alpha)}=\OR(N_{x,\eta})$, so $\eta+\om=\beta$, but  $F(N_{x,\eta+\om})$ cannot be used in $\Tt_{x,\eta+\om}^{N_\alpha}$, a contradiction.)
If there is $i<k_\alpha$ such that $E^\Uu_{\alpha i}\neq\emptyset$
then as above, $\eta<\kappa=\crit(E^\Uu_{\alpha i})<\nu(E^\Uu_{\alpha k_\alpha})$,
and therefore $\CC^{P_\alpha}_x\rest\kappa=\CC^{P_{\alpha+1}}_x\rest\kappa$, and in particular $t_{x,\eta+\om}^{P_{\alpha+1}}=t_{x,\eta+\om}^{P_\alpha}=0$.
If there is no such $i<k_\alpha$,
then $P_{\alpha i}=P_\alpha$
and $\aleph_\eta^{P_\alpha}\leq\nu(E^\Uu_{\alpha k_\alpha})<\lh(E^\Uu_{\alpha k_\alpha})$,
and since $t^{P_\alpha}_{x,\eta+\om}=0$,
coherence gives that $t^{P_{\alpha+1}}_{x,\eta+\om}=0$ also.

 Part \ref{item:iterate_is_construction_through_eta}: This follows directly from the previous parts (the soundness
 of $M^\Tt_{\alpha+1}|\lh(E^\Tt_\alpha)$
 is just because this is the passivization of an active premouse).
\end{proof}

Now let $\alpha$ be a limit
and suppose we have defined $\Tt\rest\alpha$ and $\left<\eta_\beta\right>_{\beta<\alpha}$.
Then we extend to $\Tt\rest(\alpha+1)$ by applying our iteration strategy, and set $\eta_\alpha=\sup_{\beta<\alpha}\eta_\beta$.
We then get:

\begin{clm} $(*)_\alpha$ holds (for this limit $\alpha$).
\end{clm}
\begin{proof}
This is straightforward by induction.
(The fact that $N_{x\eta_\alpha}^{P_\alpha}=M(\Tt\rest\alpha)$ uses the fact that for $\beta<\alpha$, we have 
\[M(\Tt\rest\alpha)|\lh(E^\Tt_\beta)=N_{x\eta_{\beta+1}}^{P_{\beta i}}=N_{x\eta_{\beta+1}}^{P_{\beta+1}}\]
and
 $\aleph_{\eta_{\beta+1}}^{P_{\beta+1}}=\aleph_{\eta_{\beta+1}}^{P_{\beta i}}\leq\nu(E^\Uu_{\beta k_\beta})$,
and so if $\beta+1\in b^\Tt$,
and hence $(\beta+1,0)\in b^\Uu$,
then $\aleph_{\eta_{\beta+1}}^{P_{\beta+1}}\leq\crit(i^\Uu_{(\beta+1,0),(\alpha,0)})$,
and therefore $N_{x\eta_{\beta+1}}^{P_\alpha}=N_{x\eta_{\beta+1}}^{P_{\beta+1}}=M(\Tt\rest\alpha)|\lh(E^\Tt_\beta)$.)
\end{proof}

\begin{clm}\label{clm:copy_incompatible} For each $\alpha+1<\lh(\Tt)$ and $i\leq k_\alpha$, we have:
\begin{enumerate}[label=(\alph*)]
\item\label{item:aleph_eta_alpha_leq_nu} $\aleph_{\eta_{\alpha+1}}^{P_\alpha}\leq\nu(E^\Uu_{\alpha i})$, and
\item\label{item:E_alpha_incompat_F_alpha,i} $E^\Tt_\alpha\rest\nu(E^\Tt_\alpha)\not\sub E^\Uu_{\alpha i}$.
\end{enumerate}
\end{clm}
\begin{proof}
Part \ref{item:aleph_eta_alpha_leq_nu}
was just part of $(*)_{\alpha+1}$.
Part \ref{item:E_alpha_incompat_F_alpha,i}:
If $i<k_\alpha$ then also by $(*)_{\alpha+1}$, we have $\eta_{\alpha+1}<\crit(E^\Uu_{\alpha i})$,
so $\lh(E^\Tt_\alpha)\leq\eta_{\alpha+1}<\crit(E^\Uu_{\alpha i})$,
which suffices.
So suppose $i=k_\alpha$,
and suppose $E^\Tt_\alpha\rest\nu(E^\Tt_\alpha)\sub E^\Uu_{\alpha k_\alpha}$.
Then by coherence
and the remarks just mentioned (concerning $E^\Uu_{\alpha i}$ when $i<k_\alpha$), we have $F=E^\Uu_{\alpha k_\alpha}\rest\aleph_{\eta_{\alpha+1}}^{P_\alpha}\in\es^{P_\alpha}$.
But $F$ witnesses that $t^{P_\alpha}_{x\eta_{\alpha+1}}\neq 0$,
contradicting $(*)_{\alpha+1}$.
\end{proof}

\begin{clm} The process (defining $\Tt$) terminates in countably many stages.\end{clm}

\begin{proof} Suppose not. So we get $(\Tt,\Uu)$ with $\Tt$ of length $\om_1+1$.
Let $\pi:\bar{V}\to V_\theta$ be elementary with $\bar{V}$ countable and everything relevant in $\rg(\pi)$.
Let $\pi(\kappa)=\om_1$. Then $\kappa<_\Tt\om_1$ and $(0,\om_1)_\Tt\inter\dropset^\Tt\sub(0,\kappa)_\Tt$.
Therefore $(\kappa,0)<_\Uu(\om_1,0)$ and for all $(\alpha,i)$ such that
$(\kappa,0)<_\Uu(\alpha,i)<_\Uu(\om_1,0)$, we have $i=0$. Note that
\[ \pow(\kappa)\inter M^\Tt_\kappa\sub M^\Uu_{\kappa 0}, \]
so as usual, $i^\Tt_{\kappa\om_1}\rest\pow(\kappa)\sub i^\Uu_{(\kappa,0),(\om_1,0)}$.

Let $\beta+1=\min(\kappa,\om_1)_\Tt$. Then $(\beta+1,0)=\min((\kappa,0),(\om_1,0))_\Uu$,
and by the compatibility, $E^\Tt_\beta\rest\nu(E^\Tt_\beta)\sub E^\Uu_{\beta k_\beta}$.
This contradicts Claim \ref{clm:copy_incompatible}.
\end{proof}

Let $\lh(\Tt)=\varepsilon+1$.
Let $\delta_\infty=\OR(N_{x\eta_\varepsilon}^{P_\varepsilon})=\sup_{\beta<\varepsilon}\lh(E^\Tt_\beta)$.
If $\alpha>0$, let $\theta_\alpha$
be the ordinal $\theta$ defined in line (\ref{eqn:def_theta_in_E^Tt_alpha_def})
at stage $\alpha$;
and let $\theta_0=\Omega_0$.

\begin{clm}\label{clm:if_theta_eps<Omega_eps_then_M^T_eps=N}
 We have:
 \begin{enumerate}[label=--]
  \item $\theta_\varepsilon\leq\Omega_\varepsilon$,
\item $M^\Tt_\varepsilon|\delta_\infty\ins N^{P_\varepsilon}_{x\theta_\varepsilon}\ins M^\Tt_\varepsilon$,
and
\item if $\theta_\varepsilon<\Omega_\varepsilon$ then $N^{P_\varepsilon}_{x\theta_\varepsilon}=M^\Tt_\varepsilon$.
\end{enumerate}
\end{clm}
\begin{proof}
 The fact that $\theta_\varepsilon\leq\Omega_\varepsilon$ is directly by definition. We already know that $M^\Tt_\varepsilon|\delta_\infty\ins N^{P_\varepsilon}_{x\theta_\varepsilon}$, and $N^{P_\varepsilon}_{x\theta_\varepsilon}\ins M^\Tt_\varepsilon$ since the process terminated at stage $\varepsilon$.
 And if $\theta_\varepsilon<\Omega_\varepsilon$ then $N^{P_\varepsilon}_{x\theta_\varepsilon}$ projects ${<\delta_\infty}=\sup_{\beta<\varepsilon}\lh(E^\Tt_\beta)$. But  $\delta_\infty$ is an $M^\Tt_\varepsilon$-cardinal, so no proper segment of $M^\Tt_\varepsilon$ of extending $M^\Tt_\varepsilon|\delta_\infty$ projects ${<\delta_\infty}$.
\end{proof}

\begin{clm}\label{clm:theta_eps=Omega_eps} $\theta_\varepsilon=\Omega_\varepsilon$.\end{clm}

\begin{proof} Suppose otherwise.
 So $0<\varepsilon$ (so $\Tt$ is non-trivial) and $\theta=\theta_\varepsilon<\Omega_\varepsilon$.
Now if $b^\Tt\cap\dropset^\Tt=\emptyset$
then we are done,
since then $\core_\om(N^{P_\varepsilon}_{x\theta})=\wt{M}_x$,
so $P_\varepsilon$ is $\wt{M}_x$-good
and $\Omega_{\wt{M}_x}^{P_\varepsilon}<\Omega_\varepsilon$, but then the same reflects back to $P$,
as desired.

So suppose $b^\Tt\cap\dropset^\Tt\neq\emptyset$.
 Let $\gamma+1=\max(\mathscr{D}^\Tt\cap b^\Tt)$ and 
$\alpha=\pred^\Tt(\gamma+1)$. Let $\mu=\crit(E^\Tt_\gamma)$ and 
$\kappa=\crit(E^\Uu_{(\gamma,k_\gamma)})$,
so $\mu\leq\kappa$. Let $\delta=\sup_{\beta<\alpha}\lh(E^\Tt_\beta)$.

\begin{sclm}\label{sclm:mu<kappa} $\mu<\kappa$.\end{sclm}
\begin{proof} Suppose $\mu=\kappa$. Then $\kappa\leq\delta$ since $\card^{P_\alpha}(M^\Tt_\alpha)\leq\delta$.
But if $\kappa<\delta$, then since $\Tt$ is normal, we must have
$\delta=(\kappa^+)^{M^\Tt_\alpha}$ and $\alpha=\kappa+1$ and $\nu(E^\Tt_\kappa)=\kappa$,
but then $\gamma+1\notin\dropset^\Tt$, contradiction. So $\kappa=\delta$. Since $\kappa$ is measurable
in $P$, therefore $\alpha=\kappa$.

Let $\left<M_{\kappa i}\right>_{i\leq k_\kappa}=\redd^{M^\Tt_\kappa}(\exit^\Tt_\kappa)$. Since $\gamma+1\in\dropset^\Tt$ and $\kappa=\crit(E^\Tt_\gamma)$, we have $\lh(E^\Tt_\kappa)<\kappa^{+M^\Tt_\kappa}$ and
 $\rho_\om^{M_{\kappa 1}}=\kappa$ and $M_{\kappa 1}=M^{*\Tt}_{\gamma+1}$. So
$\pred^\Uu(\gamma+1,0)=(\kappa,1)$. We have  $E^\Uu_{\kappa 0}=\emptyset$, because otherwise $\crit(E^\Uu_{\kappa0})=\kappa$
and $E^\Uu_{\kappa 0}$ is an order $0$ measure, but then $\kappa<\crit(E^\Uu_{\gamma k_\gamma})=\kappa$, contradiction. We have $M^\Tt_\kappa|\kappa=N_{x\eta_\kappa}^{P_\kappa}$. So if $\eta_\kappa=\kappa$, then by Lemma  \ref{lem:iterate_up_to_N_x,kappa}, $M^\Tt_\kappa|\kappa^{+M^\Tt_\kappa}=N_{x\chi}^{P_\kappa}$ where $\chi=\chi_{x\kappa}^{P_\kappa}$,
so $\lh(E^\Tt_\kappa)<\chi=\kappa^{+M^\Tt_\kappa}$, but then $t_{x\eta_{\kappa+1}}^{P_\kappa}=2$, contradicting $(*)_{\kappa+1}$. So $\kappa<\eta_\kappa$
and
$M^\Tt_\kappa|\kappa\neq N_{x\kappa}^{P_\kappa}$, so
there is $R\pins M^\Tt_\kappa$ such that $\rho_\om^R$ is an $M^\Tt_\kappa$-cardinal
and
\[ \kappa<\prodstage^{P_\kappa}_x(R)<\eta_{\kappa}.\]
So by $(*)_\gamma$, also
\begin{equation}\label{eqn:R_produced_after_kappa} \kappa<\prodstage^{P_\kappa}_x(R)\leq\abliftprodstage^\Tt_\gamma(R).\end{equation}
Also,
for $\beta<\lh(\Tt)$, letting
$\xi_\beta=\abliftprodstage^\Tt_\beta(M^\Tt_\beta)$, we also have the limit-resurrection map
\[\wt{\pi}^\Tt_\beta:\rho_0(M^\Tt_\beta)\to\lim_{n\to\om}\rho_0(N^{P_\beta}_{x_n\xi_\beta}),\]
and for $\beta+1<\lh(\Tt)$,
letting $\zeta_\beta=\abliftprodstage^\Tt_\beta(\exit^\Tt_\beta)$, we have
the limit-exit-resurrection map
 \[ \wt{\psi}^\Tt_{\beta k_{\beta}}:\rho_0(\exit^\Tt_\beta)\to\lim_{n\to\om}\rho_0(N_{x_n\zeta_\beta}^{P_{\beta k_\beta}}).\]
 Now since $\crit(E^\Uu_{\gamma k_\gamma})=\kappa=\mu=\crit(E^\Tt_\gamma)$, we have $\wt{\psi}^\Tt_{\gamma k_\gamma}(\kappa)=\kappa$, so
 for all $\beta\in[\kappa,\gamma]$,
\[ \wt{\psi}^\Tt_{\gamma k_\gamma}``\kappa=\wt{\pi}^\Tt_\beta``\kappa=\wt{\pi}^\Tt_\kappa``\kappa\sub\kappa.\]
In particular, $\iota=\wt{\psi}^\Tt_{\gamma k_\gamma}(\OR^R)<\kappa$.
But then letting $\zeta_\gamma=\abliftprodstage^\Tt_\gamma(\exit^\Tt_\gamma)$,
for sufficiently  large $n<\om$,
we have
$N^{P_{\gamma k_\gamma}}_{x_n\zeta_\gamma}|\kappa\neq N^{P_{\gamma k_\gamma}}_{x_n\kappa}$ (as $N^{P_{\gamma k_\gamma}}_{x_n\zeta_\gamma}|\iota$
was produced at stage $\abliftprodstage^\Tt_\gamma(R)>\kappa$,
by line (\ref{eqn:R_produced_after_kappa})).
But then by Lemma \ref{lem:limit_proj_across_meas},
$\lim_{n\to\om}t^{P_{\gamma k_\gamma}}_{x_n\zeta_\gamma}\neq 1$,
contradicting that $E^\Tt_\gamma$
was limit-lifted to a limit-ancestral stage with $t=1$
(as $\Tt$ is small). This completes the proof of the subclaim.
\end{proof}

\begin{sclm}\label{sclm:mu<kappa-bar}$\mu<\bar{\kappa}\eqdef\crit(i^\Uu_{(\alpha,0),\varepsilon})$, hence $M^\Tt_\alpha\in P_\alpha|\bar{\kappa}$.\end{sclm}
\begin{proof} If not then by Subclaim \ref{sclm:mu<kappa}, there is $i<k_\alpha$ such that $E^\Uu_{\alpha i}\neq\emptyset$,
and letting $i$ be least such, we have 
$\bar{\kappa}=\crit(E^\Uu_{\alpha i})\leq\mu<\kappa$.
But this contradicts $(*)_{\kappa+1}$
part \ref{item:star_coherence}(\ref{item:eta<crit(F_kappa,0)}).
\end{proof}

Now consider the definition of $E^\Tt_\alpha$.
In Subclaims \ref{sclm:not_alpha>0,theta_alpha<Omega_alpha,Case_not_late_R_before_gamma}, \ref{sclm:rule_out_Cases_43,50} 
and \ref{sclm:rule_out_Cases_44,49} below we rule out each possible subcase thereof.
Let $\theta_\alpha$
be the $\theta$ defined in line (\ref{eqn:def_theta_in_E^Tt_alpha_def})
at stage $\alpha$.

\begin{sclm}\label{sclm:not_alpha>0,theta_alpha<Omega_alpha,Case_not_late_R_before_gamma}
 It is not the case that $\alpha>0$
 and $\theta_\alpha<\Omega_\alpha$
 and Case \ref{case:no_late_R_before_gamma} attains
 at stage $\alpha$.
\end{sclm}
\begin{proof}
 Suppose otherwise.
 So $\eta_\alpha\leq\theta_\alpha<\Omega_\alpha$ and
 $\theta_\alpha$ is the least $\theta\geq\eta_\alpha$ such that $N_{x\theta}^{P_\alpha}$ projects
 $<\OR^N$ where $N=N_{x\eta_\alpha}^{P_\alpha}$.
 Since $E^\Tt_\alpha$ is defined,
 in fact $\eta_\alpha<\eta_{\alpha+1}<\theta_\alpha$,
 and $\theta_\alpha$ is also
 the least $\theta\geq\eta_{\alpha+1}$
 such that $N_{x\theta}^{P_\alpha}$
 projects $<\OR^{N^+}=\lh(E^\Tt_\alpha)$ where $N^+=N_{x\eta_{\alpha+1}}^{P_\alpha}$.
 By Subclaim \ref{sclm:mu<kappa-bar},
 $\OR(N^+)=\lh(E^\Tt_\alpha)\leq\OR(M^\Tt_\alpha)<\crit(i^\Uu_{(\alpha,0),\varepsilon})$,
 so letting $\eta^*_{\alpha+1}=i^\Uu_{\alpha0,\varepsilon0}(\eta_{\alpha+1})$,
 we have that $N^{P_\varepsilon}_{x\eta^*_{\alpha+1}}=N^+$
 and $\theta^*_\alpha=i^\Uu_{(\alpha,0),\varepsilon}(\theta_\alpha)$
 is the least $\theta\geq\eta^*_{\alpha+1}$
 such that $N^{P_\varepsilon}_{x\theta}$ projects ${<\OR^{N^+}}$;
  moreover, $N^{P_\varepsilon}_{x\theta^*_\alpha}$ actually projects ${<\OR^N}$.
 But by Claim \ref{clm:if_theta_eps<Omega_eps_then_M^T_eps=N}, $M^\Tt_\varepsilon=N^{P_\varepsilon}_{x\theta_\varepsilon}$,
 and since $N^+$ is a cardinal segment
 of $M^\Tt_\varepsilon$,
 we have $\theta_\varepsilon\geq\eta^*_{\alpha+1}$.
 And $\rho_\om(M^\Tt_\varepsilon)=\mu<\OR^{N^+}$. So $\theta_\varepsilon=\theta^*_\alpha$.
 But $\OR^N\leq\mu$, a contradiction.
\end{proof}

\begin{sclm}\label{sclm:rule_out_Cases_43,50} We have:
\begin{enumerate}[label=(\alph*)]
 \item\label{item:rule_out_Case_43} It is not the case that $\alpha=0$
 and Case \ref{case:E_0_no_late_R} attains at stage $0$.
 \item\label{item:rule_out_Case_50} It is not the case that $\alpha>0$
 and $\theta_\alpha=\Omega_\alpha$
 and Case \ref{case:no_late_R_before_gamma}
 attains at stage $\alpha$.
 \end{enumerate}
\end{sclm}
\begin{proof}
We just discuss part \ref{item:rule_out_Case_50};  part \ref{item:rule_out_Case_43} is only slightly simpler.
Suppose \ref{item:rule_out_Case_50} fails.
Then there is no $\xi\in[\eta_{\alpha+1},\Omega_\alpha)$ such that $N_{x\xi}^{P_\alpha}$ projects ${<\lh(E^\Tt_\alpha)}=\OR(N_{x\eta_{\alpha+1}}^{P_\alpha})$.
This is preserved by $i^\Uu_{\alpha0,\varepsilon0}$ in the sense described in the proof of Subclaim \ref{sclm:not_alpha>0,theta_alpha<Omega_alpha,Case_not_late_R_before_gamma}, so 
letting $\eta^*_{\alpha+1}$ be like before, we have $N^{P_\varepsilon}_{x\eta^*_{\alpha+1}}=M^\Tt_\varepsilon|\lh(E^\Tt_\alpha)$ and there is no $\xi\in[\eta^*_{\alpha+1},\Omega_\alpha)$ such that $N^{P_\varepsilon}_{x\xi}$ projects ${<\lh(E^\Tt_\alpha)}$.
But $\rho_\om(M^\Tt_\varepsilon)=\mu<\lh(E^\Tt_\alpha)$ and $M^\Tt_\varepsilon=N^{P_\varepsilon}_{x\theta_\varepsilon}$, a contradiction.
\end{proof}

\begin{sclm}\label{sclm:rule_out_Cases_44,49} We have:
\begin{enumerate}[label=(\alph*)]
\item It is not the case that $\alpha=0$ and Case \ref{case:E_0_late_R} attains at stage $\alpha$.
 \item It is not the case that $\alpha>0$
 and Case \ref{case:some_R_late_prior_to_gamma} attains
 at stage $\alpha$.
 \end{enumerate}
\end{sclm}
\begin{proof}
 Suppose otherwise.
 Let $R$ be as in the case hypothesis.
 Let $\beta=\beta_R$.
 So $R=\core_\om(N^{P_\alpha}_{x\beta})$
 and $R'=N^{P_\alpha}_{x\beta}$ is non-sound,
 and $R$ is produced late (relative to stage $\alpha$).
 Let $N^+=N_{x\eta_{\alpha+1}}^{P_\alpha}$.
 Since $E^\Tt_\alpha\in\es_+^R$
 and $\lh(E^\Tt_\alpha)=\OR^{N^+}$ and $N^+\pins_{\mathrm{card}} R'$,
 note that $\beta$ is the least
 $\beta'\geq\eta_{\alpha+1}$
 such that $N^{P_\alpha}_{x\beta'}$ projects $<\OR^{N^+}$.
 Like in the proof of Subclaim \ref{sclm:not_alpha>0,theta_alpha<Omega_alpha,Case_not_late_R_before_gamma}, it follows
 that $M^\Tt_\varepsilon=i^\Uu_{\alpha0,\varepsilon0}(R')$ and $R=\core_\om(M^\Tt_\varepsilon)=M^{*\Tt}_{\gamma+1}$. So with respect to stage $\varepsilon$, $R$ is produced late,
 and $\beta^*=i^\Uu_{\alpha0,\varepsilon0}(\beta)$ is the corresponding stage,
 so $R=\core_\om(N^{P_\varepsilon}_{x\beta^*})$.
 But because the process terminated
 (at stage $\varepsilon$), no proper segment of $M^\Tt_\varepsilon$ is produced late relative to stage $\varepsilon$. 
 But the latter 
 implies (using some considerations
 like those in the proof of Claim \ref{clm:Tt_R_non-triv}) that
 relative to stage $\varepsilon$,
 $R$ is not produced late, a contradiction.
\end{proof}

With the preceding subclaims we have ruled out every possibility for defining $E^\Tt_\alpha$, so we have a contradiction, completing the proof of Claim \ref{clm:theta_eps=Omega_eps}.
\end{proof}

\begin{clm}\label{clm:b^Tt_no_drop}
 $M^\Tt_\varepsilon=N^{P_\varepsilon}_{x\Omega_\varepsilon}$ and $(0,\varepsilon]^\Tt\cap\dropset^\Tt=\emptyset$.
\end{clm}
\begin{proof}
  Let $C$ be the set of all $S\pins N^{P_\varepsilon}_{x\Omega_\varepsilon}$ such that $\rho_\om^S$ is an $N^{P_\varepsilon}_{x\Omega_\varepsilon}$-cardinal.
 Let $\zeta=\abliftprodstage^\Tt_\varepsilon(N^{P_\varepsilon}_{x\Omega_\varepsilon})$.
 We claim $\zeta=\Omega_\varepsilon$. For
if $t^{P_\varepsilon}_{x\Omega_\varepsilon}\neq 1$ then  
 \[ \theta_\varepsilon=\Omega_\varepsilon=\prodstage^{P_\varepsilon}_x(N_{x\Omega_\varepsilon})=\sup_{S\in C}(\prodstage^{P_\varepsilon}_x(S)+\om)
 \leq\sup_{S\in C}(\abliftprodstage^\Tt_\varepsilon(S)+\om)\leq\zeta\leq\Omega_\varepsilon,\]
 and  if instead $t^{P_\varepsilon}_{x\Omega_\varepsilon}= 1$ then  
 \[ \theta_\varepsilon=\Omega_\varepsilon=\prodstage^{P_\varepsilon}_x(N_{x\Omega_\varepsilon})=\big(\sup_{S\in C}\prodstage^{P_\varepsilon}_x(S)\big)+\om
 \leq\big(\sup_{S\in C}\abliftprodstage^\Tt_\varepsilon(S)\big)+\om\leq\zeta\leq\Omega_\varepsilon\]
 (note $t_{x_n\zeta}^{P_\varepsilon}=1$ for sufficiently large $n$).
 Since $\zeta=\Omega_\varepsilon$,
 the claim easily follows.
\end{proof}

We can now deduce semicontinuity with respect to the $\Omega_x$ norm.
Let $\Tt^P=\Tt^P_{x\Omega_0}$.
It suffices to see that
\begin{equation}\label{eqn:Omega_x_lower_semicon_main_goal}M^{\Tt^P}_\infty=N^{P}_{x\Omega_0}\text{ and }b^{\Tt^P}\cap\dropset^{\Tt^P}=\emptyset.\end{equation}
Recall that we are working in one of the two contexts \ref{item:lower_semi_c_P_is_Q-mouse}
or \ref{item:lower_semi_c_maybe_not_Q-mouse} mentioned at the start of \S\ref{sec:lower_semicont_for_Omega_x}.
If \ref{item:lower_semi_c_maybe_not_Q-mouse} holds (in particular that $\aleph_\Omega^P<\rho_0^P$ and $\Tt_{x\Omega}^P\in P$) then Claim \ref{clm:b^Tt_no_drop} clearly pulls back to $P$, giving line (\ref{eqn:Omega_x_lower_semicon_main_goal}) in this case
And if  \ref{item:lower_semi_c_P_is_Q-mouse} holds, so $P$ is a Q-mouse and $\Omega=\OR^P$, then it follows from Lemma \ref{lem:Q-mouse_branch_pres_under_it_maps}.

\begin{lem}\label{lem:Q-mouse_branch_pres_under_it_maps} Let $P$ be a Q-mouse and $q$ be the Q-degree of $P$, as witnessed by $(q,\om_1+1)$-strategy $\Sigma_P$ for $P$. Let $x\in\RR^P$. Let $M\in\HC^P$ be both an $x$-mouse and an $\om$-mouse.
 Let $\Tt\in P$ be an $\om$-maximal tree on $M$ of length $\delta^P$,
 via $\Sigma_M$. 
 Suppose that $\Tt$ is small, $\Tt$ is definable from parameters over $P|\delta^P$, and $P|\delta^P$ is generic for the $\delta^P$-generator extender algebra of $M(\Tt)$ at $\delta(\Tt)=\delta^P$.
 
 Let $\Uu$ be a successor length small $q$-maximal tree on $P$ via $\Sigma_P$, such that $b^\Uu$ does not drop.
 Let $P'=M^\Uu_\infty$ and $\Tt'=i^\Uu(\Tt)$. Suppose that $\Tt'$
 is also via $\Sigma_M$. Let $b=\Sigma_M(\Tt)$
 and $c=\Sigma_M(\Tt')$.
 (Note that $b\cap\alpha\in P$ for each $\alpha<\delta^P$.)
 Then:
 \begin{enumerate}\item $c=\bigcup_{\alpha<\delta^P}i^\Uu(b\cap\delta^P)$.
  \item $b$ drops in model iff $c$ drops in model.
  \item $Q(\Tt,b)=\mathscr{P}^P(M(\Tt))\ins M^\Tt_b$ and $Q(\Tt',c)=\mathscr{P}^{P'}(M(\Tt'))\ins M^{\Tt'}_c$,
  \item $\mathscr{P}^{P}(M(\Tt))=M^\Tt_b$ iff $\mathscr{P}^{P'}(M(\Tt'))=M^{\Tt'}_c$.
 \end{enumerate}
\end{lem}

and we just have  $\left<\Tt^P_{x\alpha}\right>_{\alpha<\delta^P}\in P$.
But therefore also $\Tt^P_{x\delta^P}\rest\delta^P\in P$, so $M(\Tt^P_{x\delta^P}\rest\delta^P)\in P$.
Let $b=b^{\Tt^P}$.
Let $\Tt^{P_\varepsilon}=\Tt^{P_\varepsilon}_{x\OR^{P_\varepsilon}}=\Tt^{P_\varepsilon}_{x\delta^{P_\varepsilon}}$ and
$c=b^{\Tt^{P_\varepsilon}}$.
So $b\sub\delta^P$ and $c\sub\delta^{P_\varepsilon}$
and $b\cap\alpha\in P$ for each $\alpha<\delta^P$,
and $c\cap\alpha\in P_\varepsilon$
for each $\alpha<\delta^{P_\varepsilon}$.
\begin{clm}
 $c=\bigcup_{\alpha<\delta^P}i^\Uu_{00,\varepsilon0}(b\cap\alpha)$.
\end{clm}
\begin{proof}
Letting $c'$ be the union specified in the claim, since $i^\Uu_{00,\varepsilon0}$ is continuous at $\delta^P$,
it is easy to see that $c'$ is a $\Tt^{P_\varepsilon}\rest\delta^{P_\varepsilon}$-cofinal branch. Since the Q-structure which determines $c$ is just $Q'=N^{P_\varepsilon}_{x\OR^{P_\varepsilon}}$,
it therefore suffices to see that $Q'\ins M^{\Tt^{P_\varepsilon}}_{c'}$.

The Q-structure $Q=Q(\Tt^P\rest\delta^P,b)$ is just $Q=N_{x\OR^P}^P$. 
Let $n<\om$ be such that $\rho_{n+1}^P\leq\delta^P<\rho_n^P$.
Let $\alpha\in b$ be large enough that:
\begin{enumerate}[label=--]\item $b\cut\alpha$ does not drop in model or degree, 
 \item 
if $Q\pins M^{\Tt^P}_{\infty}$ then $Q\in\rg(i^{\Tt^P}_{\alpha\infty`})$;
and
\item  $i^{\Tt^P}_{\alpha,b}$ is continuous at $\rho_n^{\bar{Q}}$, where $\bar{Q}=M^{\Tt^P}_\alpha$ if $Q=M^{\Tt^P}_{\infty}$,
and $i^{\Tt^P}_{\alpha\infty}(\bar{Q})=Q$ otherwise.
\end{enumerate}
By \ref{?}, $Q$ is $\delta^P$-sound,
 $\rho_n^Q=\rho_n^P$ and $\rho_{n+1}^Q\leq\delta^P$.
 For $\xi<\rho_n^Q$ let
 $t^Q_\xi$ be the set of pairs $(\varphi,\vec{x})$ such that $\varphi$ is $\rSigma_{n+1}$,
 $\vec{x}\in[\delta^P]^{<\om}$,
 and \[Q\sats\varphi(\vec{x},\pvec_{n+1}^Q\cut\delta^P),\]
 as witnessed in the codes by 
 \[\Th_{\rSigma_n}^{Q}(\xi\cup\{\pvec_n^Q\}),\]
 or if $n=0$, as witnessed by the usual $\xi$th approximation to $Q$ (see \cite[\S2]{fsit} for an explanation of all of this).
 So
 \[ \Th_{\rSigma_{n+1}}^Q(\delta^P\cup\{\pvec_{n+1}^Q\cut\delta^P\})=\bigcup_{\xi<\rho_n^Q}t^Q_\xi,\]
 and $t^Q_\xi\in Q$ (and hence $t_\xi\in P$)
 for each $\xi<\rho_n^Q$.
For $\kappa<\delta^P$,
write $t^Q_\xi\rest\kappa$
for the set of all $(\varphi,\vec{x})\in t_\xi$ with $\vec{x}\in[\kappa]^{<\om}$.
 For $\beta\in b\cut\alpha$,
 let $\bar{Q}_\beta=i^\Tt_{\alpha\beta}(\bar{Q})$, and for $\xi<\rho_n^{\bar{Q}_\beta}$,
 define $t_\xi^{\bar{Q}_\beta}$
 over $\bar{Q}_\beta$ analogously
 to how $t_{\xi'}$ was defined over $Q$
 (note that $\delta^P\in\rg(i^\Tt_{\alpha\beta})$).
 Likewise for $t^{Q'}_\xi$ for $\xi<\rho_n^{Q'}$.
 
 For each $\xi<\rho_n^Q=\rho_n^P$, since $i^\Uu$ is a $\pvec_{n+1}$-preserving $n$-embedding and by the correspondence of fine structure between $Q,P$ (and $Q',P_\varepsilon$), we have
\[ i^\Uu(t^Q_\xi)=t^{Q'}_{i^\Uu(\xi)}.\]

 Fix $\beta\in b\cut\alpha$. 
 Then since $i^\Tt_{\beta\infty}$
 is also a $\vec{p}_{n+1}$-preserving $n$-embedding, we have
 $i^\Tt_{\beta\infty}(t^{\bar{Q}_\beta}_\xi)=t^Q_{i^\Tt_{\beta\infty}(\xi)}$,
 so
 \[t^{\bar{Q}_\beta}_\xi\rest\kappa=t^Q_{i^\Tt_{\beta\infty}(\xi)}\rest\kappa\]
 where $\kappa=\crit(i^{\Tt^P}_{\beta\infty})$.

 But $\beta'=i^\Uu(\beta)\in c'$ and note that $\kappa'=i^\Uu(\kappa)=\crit(i^{\Tt^{P_\varepsilon}}_{\beta',c'})$,
 and 
 \[t^{i^\Uu(\bar{Q}_\beta)}_{i^\Uu(\xi)} i^\Uu(t^{\bar{Q}_\beta}_\xi\rest\kappa)=i^\Uu(t^Q_{i^\Tt_{\beta\infty}(\xi)}\rest\kappa),\]
\end{proof}

\subsubsection{$M_\infty=M_x$}

\begin{cor}
 $M_\infty=M_x$.
\end{cor}
\begin{proof}
Suppose not. Then by Corollary \ref{cor:M_infty_and_M_x},
$\Mmm$ is exactly reconstructing
and $M_x\pins M_\infty$.
But by line (\ref{eqn:lower_semic_first_goal})
in \S\ref{sec:lower_semicont_for_Omega_x},
this implies that $M_x=\core_\om(N_{\alpha x}^P)$
for some $\alpha<\OR^P$. This contradicts
the assumption that $\Mmm$ is exactly reconstructing.
\end{proof}

So from now on we can write $M_x$ instead of $\widetilde{M}_x$, since they are equal.
\subsubsection{Lower semicontinuity in general}

Lower semicontinuity with respect to an arbitrary lifting norms given by norm description $\sigma$ will be established via an elaboration of the foregoing proof, integrating the role of (initial segments of) the finite tree on $M_x$ determined by $\sigma$ into the construction of the tree $\Tt$ much as in the proof of Lemma \ref{lem:norm_invariance}.

***THE PROOF IN THIS SECTION STOPS (EARLIER THAN) MIDWAY. THE REST WILL BE ADDED IN A LATER VERSION OF THE DOCUMENT.

Let $\sigma$ be an arbitrary norm description 
of depth $m+1$. Suppose $\sigma$ determines a tree $\Tt$ on $M_x$ (of length $\leq m+1$), and let $\Rr'$
be the corresponding resurrection tree on $P$ (computed in the standard manner; that is, with respect to $\mathfrak{L}_x$). We will run a version of the proof of \S\ref{sec:lower_semicont_for_Omega_x},
but incorporating simple tree embeddings from the relevant initial segment of $\Rr'$,
much as in the proof of the Norm Invariance Lemma \ref{lem:norm_invariance}.
There will also be a little more work to handle further details of the norms which did not arise so far.

Let
$\leq^P_{i},\equiv_{i}^P$, etc,
 be as in  \ref{dfn:lifting_norm}, as computed with respect to $P$.
We want to see that $x\leq^P_{4m+4} x_n$ for all sufficiently large $n<\om$.
So suppose not and let $\bar{m}\in[-1, m]$ be largest such that
$x\leq^P_{4\bar{m}} x_n$ for sufficiently large $n$; so $x>^P_{4\bar{m}+4}x_n$ for large $n$.
By \S\ref{sec:lower_semicont_for_Omega_x}, in fact $\bar{m}\geq 0$.
And note that (using the contradictory hypothesis) $x\equiv^{\undec,P}_{4\bar{m}+2} x_n$ for all large $n$. Let $n_0=\projdeg(M_x)$.\footnote{***In the present version of the paper, $n_0=0$.} We now specify a finite sse-$n_0$-maximal trees $\Tt_x$
on $M_x$ and $\Tt_{x_n}$ on $M_{x_n}$ for all sufficiently large $n<\om$, and a partial resurrection tree $\Rr$ on $P$.

We have the sse-$n_0$-maximal 
tree $\bar{\Tt}_x=\Tt_{x\bar{m}}$ on $M_x$, and for all sufficiently large  $n$,
the sse-$n_0$-maximal tree $\bar{\Tt}_{x_n}=\Tt_{x_n\bar{m}}$ on $M_{x_n}$, each of length $\bar{m}+1$.
If $\bar{\Tt}_x$ (equivalently $\bar{\Tt}_{x_n}$ for large $n$) is small let $m'=\bar{m}$;
otherwise let $m'<\bar{m}$ be least such that $\exit^{\bar{\Tt}_x}_{m'}$ is non-small.
Let $\Tt_x=\bar{\Tt}_x\rest(m'+1)$
and $\Tt_{x_n}=\bar{\Tt}_{x_n}\rest(m'+1)$.
Then for all large $n$, $\Rr_0\eqdef\lifttree^{\Tt_x\Yback}=\lifttree^{\Tt_{x_n}\Yback}$
and $\Rr_0$ is small.
\begin{case}\label{case:m'<mbar_lsc}
$m'<\bar{m}$ ($\bar{\Tt}_x$ is non-small).

Let
$e_x=\exit^{\bar{\Tt}_x}_{m'}$ and $e_{x_n}=\exit^{\bar{\Tt}_{x_n}}_{m'}$.
 Let
\[\Rr_{1}=\critrestree^{\Tt_x\Rr_0}(e_x,\delta^{e_x})=\lim_{n\to\om}
\critrestree^{\Tt_{x_n}\Rr_0}(e_{x_n},\delta^{e_{x_n}}).\]
\end{case}

In the remaining cases
(in which $m'=\bar{m}$),
if $\bar{m}<m$, we
adopt the notation from the definition of $\leq^P_{4\bar{m}+3}$ of \ref{dfn:lifting_norm}.
Note then that in this case, $e_x\downarrow$ and $e_{x_n}\downarrow$ for large $n$.

\begin{case}\label{case:m'=mbar_and_S_x_small_lsc} $m'=\bar{m}<m$ and $e_x$ is small.

Adopt also the notation from the definition of $\leq_{4\bar{m}+3}$ in its Subcase \ref{scase:S_x_small}
(in which $e_x$ is small).
Clearly $k_x=\lim_{n\to\om}k_{x_n}$ (this is the first ordinal considered for $\leq_{4\bar{m}+3}$).
Let $\wt{j}$ be the least  $j\leq k_x$ such that either $j=k_x$ or $\params^\Yback_{xj}\neq\lim_{n\to\om}\params^\Yback_{x_nj}$. Let
\[ \Rr_1=\restree^{\Tt_x\Rr_0}(e_x)\rest(\wt{j}+1)=\lim_{n\to\om}\restree^{\Tt_{x_n}\Rr_0}(e_{x_n})\rest(\wt{j}+1). \]
\end{case}
\begin{case}\label{case:S_x_non-small_lsc} $m'=\bar{m}<m$ and $e_x,e_y$ are non-small.

Adopt the notation from Subcase \ref{scase:S_x_non-small,Tt_x_small} of the definition of $\leq_{4\bar{m}+3}$.
Again $k_x=\lim_{n\to\om}k_{x_n}$. Let $\wt{j}$ be defined as above and
\[ \Rr_1=\critrestree^{\Tt_x\Rr_0}(e_x,\delta^{e_x})\rest(\wt{j}+1)=
 \lim_{n\to\om}\critrestree^{\Tt_{x_n}\Rr_0}(e_{x_n},\delta^{e_{x_n}})\rest(\wt{j}+1).
\]
\end{case}

\begin{case}\label{case:m'=mbar=m_lsc}
 $m'=\bar{m}=m$.
 
 Let $\Rr_1=\emptyset$.
\end{case}

This completes the cases. In each case let $\wt{\Rr}=\Rr_0\conc\Rr_1$ and
 $\Yback_{ij}=M^{\wt{\Rr}}_{ij}$.
Let $(i_\dis,j_\dis)$
be the largest index in $\dom(\wt{\Rr})$.
Let $\Rr$ be the non-padded tree, indexed by ordinals,
equivalent to $\wt{\Rr}$. Let
\[ I_2=\{\alpha\mid\alpha+1<\lh(\Rr)\wedge E^{\Rr}_\alpha\text{ is an order 0 measure}\}, \]
\[ I_1=\{\alpha\mid\alpha+1<\lh(\Rr)\}\cut I_2.\]
Note that $\alpha\in I_2$ iff $E^{\Rr}_\alpha$ is the measure corresponding
to some $\beta$ such that $t^{M^{\Rr}_\alpha}_{\beta x}=2=\lim_{n\to\om}t^{M^{\Rr}_\alpha}_{\beta x_n}$. For $h<\lh(\Rr)$ let $(\theta_h,i_h)$ be the least $(\theta,i)\in\dom(\wt{\Rr})$ such that $M^{\wt{\Rr}}_{\theta i}=M^\Rr_h$.

Recall that $M_x=\core_\om(N_{x\Omega}^P)$
(***and in the present version of the paper, $\rho_1^{M_x}=\om$, so $M_x=\core_1(N_{x\Omega}^P)$). Let $\pi_x:M_x\to\core_0(N_{x\Omega}^P)$
be the core embedding. Also let $\pi_{x_n}:M_{x_n}\to\core_0(N_{x_n\Omega}^P)$
be the core map for those (sufficiently large)
$n$ such that $M_{x_n}=\core_\om(N_{x_n\Omega}^P)$.
Define $\Tt^P_{x\alpha}$ (the tree on $M_x$ iterating out to $N^P_{x\alpha}$), for $\alpha\leq\Omega$, as before.
Let $(\Sigma_\infty,\frakL_\infty)$ be as before. For $\Tt$ on $M_x$ via $\Sigma_\infty$ and $\alpha<\lh(\Tt)$, we use the notation $\Uu^\Tt,\abliftprodstage^\Tt_\alpha(S)$ for $S\ins M^\Tt_\alpha$, etc, as before.

Again we will form a small tree $\Tt$ on $M_x$, via $\Sigma_\infty$, and define  $\Uu=\Uu^\Tt$.
Like in \S\ref{sec:lower_semicont_for_Omega_x}, we write  $k_\alpha=k^\Tt_\alpha$,
$P_\alpha=P_{\alpha 0}$, $P_{\alpha i}=M^\Uu_{\alpha i}$ (for $i\leq k_\alpha$) and $(E^\Tt_\alpha)^+=E^\Uu_{\alpha k_\alpha}$.
But this time, we will also arrange that there is a coarse tree embedding $\Delta:\Rr\hookrightarrow_{\mathrm{c}}\Uu$,
with $\Delta(k)+1=\lh(\Uu)$ where $k+1=\lh(\Rr)$ (see Definition \ref{dfn:coarse_tree_emb}). Using this $\Delta$, we will see that the
distinction in ordinal calculations 
seen in $M^{\Rr}_\infty=M^{\wt{\Rr}}_{i_{\mathrm{dis}},j_{\mathrm{dis}}}$
(exhibiting the fact that $x>^P_{4\bar{m}+4}x_n$ for large $n$) reflects into $M^\Uu_\infty$, thus ensuring that the relevant ordinal $\zeta$ associated to $x$ in $M^\Uu_\infty$ is strictly larger than the ordinal $\xi$ associated to $x_n$ for large $n$.
But the fact that the comparison process terminates at $M^\Uu_\infty$ will ensure
that in fact $\xi\leq\zeta$, which will yield the final  contradiction.

We construct $\Tt\rest\rest(\chi+1)$
and $\Uu\rest(\chi,1)$ simultaneously by 
 recursion on ordinals $\chi$.
Along with $\Tt\rest(\chi+1)$ and $\Uu\rest(\chi,1)$,
we will also define
$\left<h_{\alpha k},\Delta_{\alpha k}\right>_{(\alpha,k)\in\mathscr{I}_\chi}$
where $\mathscr{I}_\chi=\dom(\Uu\rest(\chi,1))$,
with $h_{\alpha k}<\lh(\Rr)$ and
\[\Delta_{\alpha k}:\Rr\rest(h_{\alpha k}+1)\hookrightarrow_{\mathrm{c}}\Uu\rest(\alpha,k+1) \]
and $\Delta_{\alpha k}(h_{\alpha k})=(\alpha,k)$.
Write $\delta^{\alpha k}_j=\Delta_{\alpha k}(j)$, $\gamma^{\alpha k}_j=\gamma^{\Delta_{\alpha k}}_j$, etc
(as in \ref{dfn:coarse_tree_emb}, for $j\leq h_{\alpha k}$). The associated maps are
\[\pi^{\alpha k}_j=M^{\wt{\Rr}}_{\theta_ji_j}\to M^\Uu_{\gamma^{\alpha k}_{j}},\]
\[\sigma^{\alpha k}_j=M^{\wt{\Rr}}_{\theta_ji_j}\to M^\Uu_{\delta^{\alpha k}_{j}} 
,\]
and we have $\gamma^{\alpha k}_j\leq^\Uu\delta^{\alpha k}_j$ and $\sigma^{\alpha k}_j=i^\Uu_{\gamma^{\alpha k}_{j}\delta^{\alpha k}_{j}}\com\pi^{\alpha k}_j$ (because $\Tt$ will be small, $\Uu$ will not drop anywhere in model or degree).

If $(\alpha,k)\in\dom(\Uu)$
but $(\alpha,k)$ is not the largest element in $\dom(\Uu)$, 
then $(\alpha,k)$ will be either \emph{copying} or \emph{inflationary}. Suppose this is the case and let $\iota$ be the successor of $(\alpha,k)$ in the $\Uu$-indexing (so either $\iota=(\alpha,k+1)$
or $\iota=(\alpha+1,0)$).
If $(\alpha,k)$ is copying
then we will have $h_{\alpha k}+1<\lh(\Rr)$,
and $\Delta_{\iota}$ will be the one-step copying extension of $\Delta_{\alpha k}$
(that is, $h_{\iota}=h_{\alpha k}+1$ and $\Delta_{\iota}=\Delta_{\alpha k}\cup\{(h_\iota,\iota)\}$).
If instead $(\alpha,k)$ is inflationary
then $\Delta_\iota$ will be the $E^\Uu_{\alpha k}$-inflation of $\Delta_{\beta j}$,
where $(\beta,j)=\pred^\Uu(\iota)$
(that is, $h_{\iota}=h_{\beta j}$
and $\Delta_\iota=(\Delta_{\beta j}\rest h_{\beta j})\cup\{(h_\iota,\iota)\}$).

We set $h_{00}=0$ and $\Delta_{00}(0)=(0,0)$.

Now suppose we have defined $\Tt\rest(\chi+1)$,
and hence $\Uu\rest(\chi,1)$,
and suppose we have also defined $\left<h_{\alpha k},\Delta_{\alpha k}\right>_{(\alpha,k)\in\mathscr{I}_\chi}$.

Define the terminology \emph{produced directly \tu{(}at $\beta$\tu{)}}, notation $\beta_R$,
and terminology \emph{produced late} as in \S\ref{sec:lower_semicont_for_Omega_x}.

We now, in case $\alpha=0$, define $E^\Tt_0$.
There are two main cases. Recall $\pi:M_x\to\core_0(N_{\Omega x}^P)$ is the core map.

\begin{case}$0<m'$.

Let $e_{x0}=\exit^{\Tt_x}_0$ and $e_{x_n0}=\exit^{\Tt_{x_n}}_0$ for large $n$.
Let
 $\Vv=\restree^P_{\Omega \om\pi_x x}(e_{x0})$,
 $k=\resl^P_{\Omega \om\pi_x x}(e_{x0})$,
 $\psi=\modresmap^P_{\Omega \om\pi_x x}(e_{x0})$,
 $\zeta=\resprodstage^P_{\Omega \om\pi_x x}(e_{x0})$, and
$N=N^{M^\Vv_k}_{\zeta x}$.
So $\psi:\core_0(e_{x0})\to \core_0(N)$ is the final resurrection map. Note that since $0<m'$,
we also have
\begin{enumerate}[label=--]\item $\Vv=\lim_{n\to\om}\restree^P_{\Omega \om\pi_{x_n} x_n}(e_{x_n0})$,
 \item $k=\lim_{n\to\om}\resl^P_{\Omega \om\pi_{x_n} x_n}(e_{x_n0})$, and
 \item $\zeta=\lim_{n\to\om}\resprodstage^P_{\Omega \om\pi_{x_n} x_n}(e_{x_n0})=\abliftprodstage^\Tt_{0}(e_{x0})$.
 \end{enumerate}

 Note that $\Tt^{M^\Vv_k}_{x\zeta}$ is non-trivial (as $N$ is active and small, hence backgrounded with an $M^\Vv_k$-total extender,
 but $M_x\in\HC^{M^\Vv_k}$). So there are just the following two subcases;
 if $\Vv$ uses some extender
 then let $\theta=\crit(D)$ with $D$ the least such,
 and otherwise let $\theta=\xi$:

\begin{scase}\label{scase:E_0_no_late_R_lsc} No $R\pins M_x$ with $\beta_R<\theta$ is produced late (as defined over $P$).

Then we set $E^\Tt_0=E^{\Tt^{M^\Vv_k}_{x\xi}}_0$.
\end{scase}

\begin{scase}\label{scase:E_0_late_R_lsc} Otherwise (some $R\pins M_x$ with $\beta_R<\theta$
 is produced late (as defined over $P$)).

Fix $R$ least such and let $\beta=\beta_R$ and $\Tt_R=\Tt^P_{x\beta}$.
By the proof of Claim \ref{clm:Tt_R_non-triv}
in \S\ref{sec:lower_semicont_for_Omega_x},
$\Tt_R$ is non-trivial,
and we set $E^\Tt_0=E^{\Tt_R}_0$.
\end{scase}
\end{case}

\begin{case}
$0=m'<\bar{m}$ (so Case \ref{case:m'<mbar_lsc} attained in defining $\Tt_x$).

Let
$e_x$ and $e_{x_n}$ for large $n$ be as in Case \ref{case:m'<mbar_lsc}. 
Let
 $\Vv=\critrestree^P_{\Omega \om\pi_x x}(e_{x},\delta^{e_x})$,
 $k=\critresl^P_{\Omega \om\pi_x x}(e_{x},\delta^{e_x})$,
 $\psi=\critresmap^P_{\Omega \om\pi_x x}(e_{x},\delta^{e_x})$,
 $\zeta=\critresprodstage^P_{\Omega \om\pi_x x}(e_{x},\delta^{e_x})$, 
$N=N^{M^\Vv_k}_{\zeta x}$,
and $\delta=\delta^N$,
if $N$ is non-small,
and otherwise $\delta=\delta^{\psi(e_x)}$.
 Note that since $m'<\bar{m}$,
we also have
\begin{enumerate}[label=--]\item $\Vv=\lim_{n\to\om}\critrestree^P_{\Omega \om\pi_{x_n} x_n}(e_{x_n},\delta^{e_{x_n}})$,
 \item $k=\lim_{n\to\om}\critresl^P_{\Omega \om\pi_{x_n} x_n}(e_{x_n},\delta^{e_{x_{n}}})$, and
 \item $\zeta=\lim_{n\to\om}\critresprodstage^P_{\Omega \om\pi_{x_n} x_n}(e_{x_n},\delta^{e_{x_n}})$.
 \end{enumerate}

 Note that $\Tt^{M^\Vv_k}_{x\delta}$ is non-trivial. If $\Vv$ uses some extender
 then let $\theta=\crit(D)$ with $D$ the least such,
 and otherwise let $\theta=\delta$.
 
\begin{scase}\label{scase:E_0_no_late_R_lsc_0=m'<mbar} No $R\pins M_x$ with $\beta_R<\theta$ is produced late (as defined over $P$).

Then we set $E^\Tt_0=E^{\Tt^{M^\Vv_k}_{x\theta}}_0$.
\end{scase}

\begin{scase}\label{scase:E_0_late_R_lsc_0=m'<mbar} Otherwise (some $R\pins M_x$ with $\beta_R<\theta$
 is produced late (as defined over $P$)).

Fix $R$ least such and let $\beta=\beta_R$ and $\Tt_R=\Tt^P_{x\beta}$.
By the proof of Claim \ref{clm:Tt_R_non-triv}
in \S\ref{sec:lower_semicont_for_Omega_x},
$\Tt_R$ is non-trivial,
and we set $E^\Tt_0=E^{\Tt_R}_0$.
\end{scase}
\end{case}

\begin{case} $0=m'=\bar{m}<m$ and $e_x$ is small (so Case \ref{case:m'=mbar_and_S_x_small_lsc} attained
in defining $\Tt_x$).

Let $\wt{j}$ be as in Case \ref{case:m'=mbar_and_S_x_small_lsc}. Let
$(\Vv,\left<\Psi_i\right>_{i\leq k})=\modres^P_{\Omega\omega\pi_xx}(e_x)$ and
$\Psi_i=(\beta_i,A_i,\psi_i,\alpha_i,d_i)$.
If $\Vv\rest(\wt{j}+1)$ uses some extender
then let $\theta=\crit(D)$ where $D$ is least such; otherwise let $\theta=\alpha_{\wt{j}}$.

\begin{scase}\label{scase:E_0_no_late_R_lsc_0=m'<mbar_2} No $R\pins M_x$ with $\beta_R<\theta$ is produced late (as defined over $P$).

Then we set $E^\Tt_0=E^{\Tt^{M^\Vv_k}_{x\theta}}_0$.
\end{scase}

\begin{scase}\label{scase:E_0_late_R_lsc_0=m'=mbar<m} Otherwise (some $R\pins M_x$ with $\beta_R<\theta$
 is produced late (as defined over $P$)).

Fix $R$ least such and let $\beta=\beta_R$ and $\Tt_R=\Tt^P_{x\beta}$.
By the proof of Claim \ref{clm:Tt_R_non-triv}
in \S\ref{sec:lower_semicont_for_Omega_x},
$\Tt_R$ is non-trivial,
and we set $E^\Tt_0=E^{\Tt_R}_0$.
\end{scase}
\end{case}

\begin{case} $0=m'=\bar{m}<m$ and $e_x$ is non-small (so Case \ref{case:S_x_non-small_lsc} attained
in defining $\Tt_x$).

Let $\wt{j}$ be as in Case \ref{case:S_x_non-small_lsc}. Let
$(\Vv,\left<\Psi_i\right>_{i\leq k})=\modres^P_{\Omega\omega\pi_xx}(e_x)$ and
$\Psi_i=(\beta_i,A_i,\psi_i,\alpha_i,d_i)$.
If $\Vv\rest(\wt{j}+1)$ uses some extender
then let $\theta=\crit(D)$ where $D$ is least such; otherwise let $\theta=\alpha_{\wt{j}}$.

\begin{scase}\label{scase:E_0_no_late_R_lsc_0=m'<mbar_3} No $R\pins M_x$ with $\beta_R<\theta$ is produced late (as defined over $P$).

Then we set $E^\Tt_0=E^{\Tt^{M^\Vv_k}_{x\theta}}_0$.
\end{scase}

\begin{scase}\label{scase:E_0_late_R_lsc_0=m'=mbar<m_2} Otherwise (some $R\pins M_x$ with $\beta_R<\theta$
 is produced late (as defined over $P$)).

Fix $R$ least such and let $\beta=\beta_R$ and $\Tt_R=\Tt^P_{x\beta}$.
By the proof of Claim \ref{clm:Tt_R_non-triv}
in \S\ref{sec:lower_semicont_for_Omega_x},
$\Tt_R$ is non-trivial,
and we set $E^\Tt_0=E^{\Tt_R}_0$.
\end{scase}
\end{case}

(***THE REMAINDER OF THE PROOF OF LOWER SEMICONTINUITY IS TO BE ADDED.)

\section{Optimality of mouse scales for exactly reconstructing}

\begin{tm}\label{tm:optimality}
Let $T$ be good, exactly reconstructing, uniformly cofinal, uniformly boundedly $1$-solid. Let $\mathscr{M}=\mathscr{M}_T$. Then
 there is a set $B\in\Gamma_{\Mmm,\mathrm{ubd}}$ such that there is no
semiscale $\vec{\leq}$ on $B$ with $\vec{\leq}\in\undertilde{\Gamma}_{\Mmm,\mathrm{ubd}}$.\end{tm}
\begin{proof}
We use an easy adaptation of Martin's proof.
That is, let $B(x,y)$ iff $y\in\RR\cut\RR^{M_x}$. Then $B\in\Delta^\Mmm_0$,
because given $x,y\in\RR$, we have $y\notin M_x$ iff $M_{x,y}\sats$``$y\notin N_{x\infty}^\CC$''
iff $M_{x,y}\sats$``$y\notin N_{x\delta}^\CC$'', where $\delta=\delta^{M_{x,y}}$.

Now suppose there is a semiscale $\vec{\leq}$ on $B$ and $n<\om$
such that $\vec{\leq}$ is $\Delta^\Mmm_n(x)$.
Let $y$ be Cohen generic over $M_x$. Then $M_{x,y}$ is a reorganization of $M_x[y]$.
Let $B'=B\inter M_{x,y}$, let $\vec{\leq'}$ be the norms
$\vec{\leq}\rest M_{x,y}$,
and let $T$ be the tree of this semiscale.
Then $T\in M_{x,y}$. Moreover, $T$ is $\Delta_n^{M_{x,y}}(\{x\})$
(since $\es_+^{M_{x,y}}$ is computed easily from $\es_+^{M_x}$),
and by homogeneity of the forcing, therefore $T\in M_x$.
So $T_x\in M_x$. But $y\in p[T_x]$, so $T_x$ is illfounded in $M_x$,
so $p[T_x]\inter M_x\neq\emptyset$ (note that $T_x\in M_x|\delta^{M_x}$,
so this does not require any extra closure of $M_x$ above $\delta^{M_x}$).
Considering the definition of $B_0$, this is a contradiction.
\end{proof}

\section{Appendix: ISC etc}\label{sec:ISC_etc}
(***Proofs in this section are small variants on proofs that appears elsewhere.)

In this section: Proof of ISC for pseudo-premice, partial measures ISC, proof of solidity/universality.

\begin{lem}\label{lem:type_Z_limit_card}
 Let $M=(N,G)$ be a generalized-pseudo-premouse and $G\rest(\xi+1)$ be type Z.
 Suppose that all proper segments of $M$ satisfy condensation.
 Then $\xi$ is a limit cardinal of $M$.
\end{lem}
\begin{proof}
 Let $U=\Ult(N,G\rest\xi)$ and $U'=\Ult(N,G\rest(\xi+1))$ and $k:U\to U'$ be the factor map,
 so $\crit(k)=\xi$. Suppose $\xi$ is not a limit cardinal of $M$.
 Then $\xi=(\delta^+)^U$ for some $\delta$, so $k(\xi)=(\delta^+)^{U'}=(\xi^+)^{U'}$.
 Because $G\rest(\xi+1)$ is type Z,
 \[ (\delta^{++})^U=(\xi^+)^U=(\xi^+)^{U'}=(\delta^+)^{U'}. \]
 So by condensation for proper segments of $U'$, either:
 \begin{enumerate}[label=--]
  \item\label{item:pass} $U'|\xi$ is passive and $U|(\delta^{++})^U=U'||(\delta^{++})^U$, or
  \item\label{item:act} $U'|\xi$ is active and letting
  $R=\Ult(U'|\xi,\es^{U'}_\xi)$, then $U|(\delta^{++})^U=R||(\delta^{++})^U$.
 \end{enumerate}
 But both cases contradict the fact that $(\delta^{++})^U=(\delta^+)^{U'}$.
\end{proof}

\begin{rem}\label{rem:proof_gen-pseudo-pm} We now prove Theorem \ref{thm:pseudo-premice}:\end{rem}
\begin{tm}\label{tm:pseudo}
Let $M$ be a $(0,\om_1+1)$-iterable generalized-pseudo-premouse. Then either
$M$ is a premouse, or
\tu{[}$\nu(F^M)<\lgcd(M)=\delta$ and $(M|\delta,F^M\rest\delta)$ is a premouse\tu{]}.
\end{tm}

\begin{proof}
\begin{clm}
 If $M$ is a pseudo-premouse then $M$ is a premouse.
\end{clm}
\begin{proof} See below.\end{proof}

Now suppose that
 $\nu(F^M)<\delta=\lgcd(M)$. So assume this. Then $M$ is not a pseudo-premouse,
 so $\delta=(\rho^+)^M$ where $\rho$ is an $M$-cardinal, and $F^M\rest\rho\in\es^M$. 
 Note that $M'=(M|\delta,F^M\rest\delta)$ is a pseudo-premouse. Moreover, $M'$ is $(0,\om_1+1)$-iterable.
So $M'$ is a premouse, as required.

So we may assume that $\nu(F^M)\geq\delta$.
By the claim above, it suffices to prove:

\begin{clm} $F^M\rest\delta\in\es^M$.\end{clm}
\begin{proof} See below.\end{proof}

\begin{clm} $\delta$ is a limit of generators of $F^M$ and for every $\alpha<\delta$
 there is $\beta\in(\alpha,\delta)$ such that $F^M\rest\beta\in\es^M$.
\end{clm}

This completes the proof.
\end{proof}

\begin{proof}
(***Need to point out that we have standard condensation for proper segs.)
Let $M=(N,G)$ be a $(0,\om_1+1)$-iterable generalized-pseudo-premouse.
For notational simplicity we will just consider the case that $M$ is not a pseudo-premouse; the other case is easier. (It is not however quite standard,
as we assume only $(0,\om_1+1)$-iterability, not $(0,\om_1,\om_1+1)$-iterability.)
So $M$ has largest cardinal $(\delta^+)^M$ where $\delta$ is an $M$-cardinal,
and $G\rest\delta\in\es^M$. Therefore $G$ has a generator $\geq\delta$.
Note that if $\delta$ is the largest generator of $G$ then we are 
done, so we may assume that $G$ has a generator $\gamma>\delta$.
Let $\nu=\nu_G$.

Suppose that $G$ has no largest generator. Then $\nu$ is a cardinal of $\Ult(M,G)$,
hence $\nu=(\delta^+)^M$. But then by considering $G\rest\alpha$ for cofinally many $\alpha<\nu$,
we can reduce to the pseudo-premouse case.

So we may assume that $G$ has a largest generator $\gamma$,
and we may assume that $\gamma\geq(\delta^+)^M$, as otherwise we can again reduce to the pseudo-premouse case.
Note that $(\delta^+)^M$ is not a generator of $G$ (otherwise the factor map 
\[ k:\Ult(M,G\rest(\delta^+)^M)\to\Ult(M,G) \]
has $\crit(k)=(\delta^+)^M$, which is easily a contradiction). So $\gamma>(\delta^+)^M$.
By \ref{lem:type_Z_limit_card}, $G\rest\gamma$ is not type Z.
Let $\nu=\nu(G\rest\gamma)$. We need to verify that ($*$) either 
\begin{enumerate}[label=--]
 \item $G\rest\gamma\in\es^M$, or
 \item $M|\nu$ is active and $G\rest\gamma\in\es^{\Ult(M|\nu,\es^M_\nu)}$.
\end{enumerate}

Let $H=\cHull_1^M(\{\gamma,G\rest\delta\})$ and $\pi:H\to M$ be the uncollapse.
Here $\Sigma_1$ is in the language with symbols $(\dot{\in},\dot{\es},\dot{F})$;
we do not use the extra constants used in the standard premouse language.
Let $\pi(\delta^H,\gamma^H)=(\delta,\gamma)$.
Then $H$ is a generalized-pseudo-premouse with largest cardinal $((\delta^H)^+)^H$,
$\delta^H$ is an $H$-cardinal,
$F^H$ has largest generator $\gamma^H>((\delta^H)^+)^H$, and $\pi(F^H\rest\delta^H)=G\rest\delta^H$.
Clearly $H$ is $(0,\om_1+1)$-iterable, by copying with $\pi$.
Moreover,
\[ H=\Hull_1^H(\{\gamma^H,F^H\rest\delta^H\}). \]

As above, $F^H\rest\gamma^H$ is not  type Z.
So by the following claim,
the ISC applies to $F^H\rest\gamma^H$, and lifting this up with $\pi$,
we have ($*$), completing the proof:

\begin{clm} $H$ is a premouse. \end{clm}

\begin{proof}
We can mostly follow the standard comparison proof,
using the fine structural setup to replace the use of Dodd-Jensen.

Renaming, let $\delta=\delta^H$ and $\gamma=\gamma^H$ and $F=F^H$.
Let $U=\Ult(H,F\rest\gamma)$.
Then $\gamma=(\delta^{++})^U$ and $i^H_{F\rest\gamma}(\kappa)>\gamma$
where $\kappa=\crit(F)$.

Define the phalanx
\[ P=((H,0,{\leq\delta}),(U,0)). \]

Here the notation means that we form normal trees $\Uu$ on $P$ by setting $\pred^\Uu(\alpha+1)=-1$
when $\crit(E^\Uu_\alpha)\leq\delta$, and otherwise $\pred^\Uu(\alpha+1)\geq 0$.
Note that $U|\gamma=H||\gamma$, and we will have $\lh(E^\Uu_0)>\gamma$, so when $\pred^\Uu(\alpha+1)=-1$,
$M^{*\Uu}_{\alpha+1}=H$ and $\Uu$ does not drop at $\alpha+1$. Also,
note that $\crit(E^\Tt_\alpha)\notin(\delta,\gamma]$.
Now $P$ is $(\om_1+1)$-iterable, as we can lift normal trees on $P$ to trees on $H$.

Let $(\Tt,\Uu)$ be the result of comparing $H$ vs $P$.
The fine structural situation clearly prevents either $b^\Tt$ or $b^\Uu$ from dropping.

Let $R\pins H$ be least such that $\gamma\leq\OR^R$ and $\rho_\om^R=(\delta^+)^H$.
Because $\gamma<\OR^M$

Renaming, let $\delta=\delta^H$ and $\gamma=\gamma^H$ and $F=F^H$.
Let $U=\Ult(H,F\rest\gamma)$. Then $\gamma=(\delta^{++})^U$ and $i^H_{F\rest\gamma}(\kappa)>\gamma$
where $\kappa=\crit(F)$.Let $R\pins H$ be least such that 
$\gamma\leq\OR^R$ and $\rho_\om^R=(\delta^+)^H$ and let $r=\projdeg(R)$.

Define the phalanx
\[ P=((H,0,{<\delta}),(R,r,\delta),(U,0)). \]
Here the notation indicates that when forming iterates $\Tt$ of $P$,
if $E=E^\Tt_\alpha$ and $\crit(E)<\delta$ then $\pred^\Tt(\alpha+1)=H$ and 
$\deg^\Tt(\alpha+1)=0$ (so $M^\Tt_{\alpha+1}=\Ult_0(H,E)$),
if $\crit(E)=\delta$ then $\pred^\Tt(\alpha+1)=R$ and $\deg^\Tt(\alpha+1)=r$, etc.
Here $\Tt$ will be normal, and we will have $\gamma<\lh(E^\Tt_0)$,
so $\crit(E)\notin(\delta,\gamma]$, and $E$ will be total over $P$ or $R$ in the situations just 
mentioned. The case that $r=-1$ is just the \emph{anomalous} case in the terminology of 
\cite{deconstructing}.
In this case $R$ is type 3 with $\nu(F^R)=\delta=\lgcd(R)$,
and $\Ult_{-1}(R,E)$ denotes forming the standard degree $0$ ultrapower,
but of $R$ itself, not $R^\sq$ (which we must do as $\delta=\lgcd(R)$).
Likewise if $\pred^\Tt(\alpha+1)=\beta$ and $R<_\Tt\beta$ and 
$(R,\alpha+1)_\Tt\inter\dropset^\Tt=\emptyset$. If $R<_\Tt\beta$ and 
$(R,\beta)_\Tt\inter\dropset^\Tt=\emptyset$ and $E^\Tt_\beta=F(M^\Tt_\beta)$
then we use $\lgcd(M^\Tt_\beta)$ as the exchange ordinal instead of $\nu(E^\Tt_\beta)$
(it can be that $\nu(E^\Tt_\beta)<\lgcd(M^\Tt_\beta)$, and for the proof of iterability of $P$ it 
is more convenient to use $\lgcd(M^\Tt_\beta)$ as the exchange ordinal).

\begin{sclm} $P$ is $(\om_1+1)$-iterable with respect to trees $\Tt$ as described above.\end{sclm}
\begin{proof}[Proof Sketch]
Let $\pi_H:H\to H$ be the identity, $\pi_R:R\to R\pins H$ be the identity/inclusion,
and $\pi_U:U\to U_F=\Ult(H,F)$ be the factor. So 
\[ \crit(\pi_U)=\gamma=(\delta^+)^U=(\delta^+)^R \]
so $\pi_H\rest\delta\sub\pi_U$ and $\pi_R\rest\gamma\sub\pi_U$,
and $\pi_U(\gamma)=(\delta^+)^{U_F}=\OR^H$. So using these lifting maps, one can lift
the relevant trees $\Tt$ to normal trees $\Uu$ on the phalanx 
\[ P'=((H,0,<\delta),(H,0,\delta),(U_F,0,>\delta)), \]
where $\OR^H<\lh(E^\Uu_0)$ (the latter because $\gamma<\lh(E^\Tt_0)$ and $\pi_U(\gamma)=\OR^H$).
Note that this is just a $0$-maximal continuation of the tree on $H$ which uses only the extender 
$F^H$. This completes the sketch.
\end{proof}

Now let $(\Tt,\Uu)$ be the comparison of $(P,H)$, using our $(\om_1+1)$-strategies. Note here that 
$U|\gamma=H||\gamma$,
so if $\Tt$ is non-trivial then $\lh(E^\Tt_0)>\gamma$, as required.

\begin{sclm}\label{sclm:compatible_exts}
Let $\alpha,\beta$ be such that $E^\Tt_\alpha\rest\nu=E^\Uu_\beta\rest\nu$
where $\nu=\min(\nu(E^\Tt_\alpha),\nu(E^\Uu_\beta))$.
Then $r=-1$ and $\beta=0$ (so $E^\Uu_\beta=F^R=F^{R|\gamma}$ is type 3), $R<_\Tt\alpha$ and 
$E^\Tt_\alpha=F(M^\Tt_\alpha)$, and in particular, $\crit(E^\Tt_\alpha)=\crit(E^\Uu_\beta)<\delta$.
\end{sclm}
\begin{proof}***Fill in\end{proof}
\begin{sclm}
 The comparison terminates.
\end{sclm}
\begin{proof}
Otherwise the usual argument gives $\alpha,\beta$ such that 
$E^\Tt_\alpha\rest\nu=E^\Uu_\beta\rest\nu$ where $\nu=\min(\nu^\Tt_\alpha,\nu^\Uu_\beta)$,
where $\nu^\Tt_\alpha$ is the exchange ordinal associated to $E^\Tt_\alpha$,
and $\nu^\Tt_\beta$ is likewise, and $\kappa=\crit(E^\Tt_\alpha)>>\delta$. Here recall that
$\nu(E^\Tt_\alpha)\leq\nu^\Tt_\alpha$
and $\nu^\Tt_\beta=\nu(E^\Tt_\beta)$. This contradicts Subclaim \ref{sclm:compatible_exts}.
\end{proof}

The following is a standard calculation (for the fact that 
$F(M^\Tt_\alpha)$ is generated by $i^\Tt_{0\alpha}(\gamma)+1$ see 
example \cite{extmax}):
\begin{sclm}\label{sclm:pres_fsH}
 Let $0<_\Tt\alpha$ be such that $(0,\alpha]_\Tt$ does not drop in model.
 Then $i^\Tt_{0\alpha}(\gamma)$ is the largest generator of $F(M^\Tt_\alpha)$
 and $i^\Tt_{0\alpha}(\delta)=\lgcd(M^\Tt_\alpha)$ and
 and $i^\Tt_{0\alpha}(F^H\rest\delta)=F(M^\Tt_\alpha)\rest i^\Tt_{0\alpha}(\delta)$.
 
Likewise for $\Uu$ if $H<_\Uu\alpha$ and $(H,\alpha]_\Uu$ does not drop,
or if $U<_\Uu\alpha$ and $(U,\alpha]_\Uu$ does not drop.
\end{sclm}

\begin{sclm}
For each $\alpha+1<\lh(\Tt)$, $E^\Tt_\alpha$ is close to $M^{*\Tt}_{\alpha+1}$.
Likewise for $\Uu$.\end{sclm}
\begin{proof}[Proof Sketch]
For $\Uu$ this is by the proof of \cite[6.1.5]{fsit},
and for $\Tt$ is by that proof along with some extra observations.
That (inductive) proof works in the case that $\pred^\Tt(\alpha+1)\notin\{H,R\}$;
let us consider the other cases.
Let $E=E^\Tt_\alpha$.

Suppose $\pred^\Tt(\alpha+1)=R$, so $\crit(E)=\delta$.
If $E=F(M^\Tt_\alpha)$ then 
\[ (\phalroot^\Tt(\alpha),\alpha]_\Tt\inter\dropset^\Tt\neq\emptyset.\]
For suppose not. If $H<_\Tt\alpha$ then
\[ \crit(i^\Tt_{H,\alpha})<\crit(F(M^\Tt_\alpha))=\delta<\gamma<\lh(E^\Tt_0), \]
so $\crit(F(M^\Tt_\alpha))\notin\rg(i^\Tt_{H,\alpha})$, contradiction.
If $R<_\Tt\alpha$ then $\delta=\crit(i^\Tt_{\alpha,R})$, again a contradiction.
But $U\not<_\Tt\alpha$, because $\crit(F^U)>\delta$, because 
$F^H$ is not superstrong, because $F^H$ has generators $\geq\delta$.
Now arguing as in \cite[6.1.5]{fsit}, one can show that the component measures of $E$
are all in $U$.
But then by condensation applied to the factor map $\pi_U$, we get that either
$H|\gamma$ is passive and all of the component 
measures are in $R$, or $R=H|\gamma$ is active and all the component measures are in $\Ult(R,F^R)$.
Note here that condensation suffices here because it applies to all proper segments of $U$ 
which project to $\gamma=(\delta^+)^U$, and this includes 
all of the component measures of $E$. So $E$ is close to $R$.

Now suppose that $\pred^\Tt(\alpha+1)=H$, so $\crit(E)<\delta$.
If all component measures are in $U$, hence in $U|\gamma\sub H$, then we are done
(maybe $\gamma=(\kappa^{++})^U$ where $\kappa=\crit(E)$).
So we may assume that $E^\Tt_\alpha=F(M^\Tt_\alpha)$ and
\[ (\phalroot^\Tt(\alpha),\alpha]_\Tt\inter\dropset^\Tt=\emptyset.\]
Therefore $U\not<_\Tt\alpha$. If $R<_\Tt\alpha$
then by induction, all component measures of $E$ are definable over $R$,
hence in $H$. If $H<_\Tt\alpha$ then $\crit(E)<\crit(i^\Tt_{H,\alpha})$,
so by induction all component measures are close to $H$, so we are done.
\end{proof}

By the previous subclaim, we get standard preservation of fine structure under the iteration maps 
of $\Tt,\Uu$.

\begin{sclm}
$M^\Tt_\infty=M^\Uu_\infty$, $b^\Tt$ is above $U$, and neither $b^\Tt$ nor $b^\Uu$ drops.
\end{sclm}
\begin{proof}
The fact that $M^\Tt_\infty=M^\Uu_\infty$ is straightforward using standard fine structure
and the fact that $H=\Hull_1^H(\{\gamma,F^H\rest\delta\})$.
The fact that not both $b^\Tt,b^\Uu$ drop in model is also by standard fine structure.

Suppose $b^\Uu$ drops in model, so $b^\Tt$ does not. If $b^\Tt$ is above $H$ or $U$,
we get a contradiction using standard fine structure and the fact that $H=\Hull_1^H(\{x\})$ for 
some $x$. So $b^\Tt$ is above $R$. Because $b^\Uu$ drops in model,
$M^\Uu_\infty$ is a premouse, which implies that $r\geq 0$ (as if $r=-1$
then $F(M^\Tt_\infty)\rest\delta\notin M^\Tt_\infty$). But standard fine structure,
considering the core map determined by $M^\Tt_\infty=M^\Uu_\infty$,
gives that $\Tt,\Uu$ use compatible extenders having critical point $\geq\delta$,
contradicting Subclaim \ref{sclm:compatible_exts}.

So $b^\Uu$ does not drop in model.
If $b^\Tt$ drops in model or $b^\Tt$ is above $R$
then we get a contradiction as above (here if $r=-1$, use that 
the ISC fails for $F(M^\Tt_\infty)\rest\delta$
for $M^\Tt_\infty$, whereas by Subclaim \ref{sclm:pres_fsH}, it holds at $\delta$ for 
$M^\Uu_\infty$).
So $b^\Tt$ is above $H$ or $U$,
and does not drop in model.

Suppose $b^\Tt$ is above $H$.
Then $i^\Tt=i^\Uu$, by Subclaim \ref{sclm:pres_fsH}. So letting $\alpha+1=\min(b^\Tt\cut\{H\})$
and $\beta+1=\min(b^\Uu\cut\{0\})$, we have $E^\Tt_\alpha\rest\nu=E^\Uu_\beta\rest\nu$
where $\nu=\min(\nu(E^\Tt_\alpha),\nu(E^\Uu_\beta))$. So by Subclaim \ref{sclm:compatible_exts},
$r=-1$ and $\beta=0$ and $R<_\Tt\alpha$ and $E^\Tt_\alpha=F(M^\Tt_\alpha)$.
Note that $\delta$ is a generator of $E^\Tt_\alpha$, so 
$M^\Uu_\infty\neq M^\Uu_{\beta+1}$.
So let $\alpha'+1=\min(R,\alpha]_\Tt$ and $\beta'+1=\min(b^\Uu\cut\{0,\beta+1\})$.
Then $E^\Tt_{\alpha'}\rest\nu'=E^\Uu_{\beta'}\rest\nu'$ where 
$\nu'=\min(\nu(E^\Tt_{\alpha'}),\nu(E^\Uu_{\beta'}))$. But $\crit(E^\Tt_{\alpha'})=\delta$,
contradicting Subclaim \ref{sclm:compatible_exts}.
\end{proof}

So $b^\Tt$ is above $U$ and does not drop in model.
Again using Subclaim \ref{sclm:pres_fsH},
\[ i^\Tt=i^\Uu\com i^H_{F^H\rest\gamma}.\]
And $\gamma<\crit(i^\Uu)$ by construction. So letting $\alpha+1=\min(b^\Uu)$,
\[ F^H\rest\nu=E^\Uu_\alpha\rest\nu \]
where $\nu=\min(\gamma,\nu(E^\Uu_\alpha))$.
But $\gamma\leq\lh(E^\Uu_0)$.

Suppose that $\gamma=\lh(E^\Uu_\alpha)$, so $\alpha=0$.
Recall that $F^H\rest\gamma$ has a generator $\geq\delta$.
So if $\nu(E^\Uu_0)=\delta$ then note that $1<_\Uu\infty$
and that $\delta=\crit(i^\Uu_{1,\infty})$. But it follows that $\nu(F^H\rest\gamma)=\delta+1$
and $F^H\rest\gamma$ is type Z, a contradiction. So $\delta<\nu(E^\Uu_0)$,
so if $1<_\Uu\infty$ then $\gamma<\crit(i^\Uu_{1,\infty})$,
which gives that $F^H\rest\gamma=E^\Uu_\alpha$, as required.

Now suppose that $\gamma<\lh(E^\Uu_\alpha)$.
Then $\gamma$ is a cardinal of $\exit^\Uu_\alpha$,
which by Subclaim \ref{sclm:pres_fsH} implies that $\gamma\leq\nu(E^\Uu_\alpha)$
and if 
$\gamma<\nu(E^\Uu_\alpha)$
then $E^\Uu_\alpha\rest\gamma\in\exit^\Uu_\alpha$,
hence $E^\Uu_\alpha\rest\gamma\in M^\Uu_\infty$.
But by the compatibility above, this also gives that $F^H\rest\gamma=E^\Uu_\alpha\rest\gamma$,
and since $F^H\rest\gamma\notin M^\Tt_\infty$,
therefore $\gamma=\nu(E^\Uu_\alpha)$. It follows that either $\alpha=0$
and $E^\Uu_0=F^H\rest\gamma\in\es^H$,
or $\alpha=1$ and $\lh(E^\Uu_0)=\gamma$
and $E^\Uu_1=F^H\rest\gamma$, which completes the proof that $H$ is a premouse.
\end{proof}
This completes the proof of the theorem.
\end{proof}

\begin{tm}[ISC for submeasures]\label{thm:ISC_for_submeasures}
 Let $N$ be an active, $(0,\om_1+1)$-iterable premouse and $\kappa=\crit(F^N)$.
 Let $M$ be a type 1 premouse such that $M|\kappa=N|\kappa$
 and $(\kappa^+)^M=\xi<(\kappa^+)^N$ and
$F^M\rest\xi\sub F^N$.

 Then either \tu{(}i\tu{)} $M\pins N$ or \tu{(}ii\tu{)} $N|\xi$ is active and 
$M\pins U=\Ult(N|\xi,F^{N|\xi})$.
\end{tm}
This theorem can be proved directly in the style of the ISC (such as \cite[\S10]{fsit}
or \ref{tm:pseudo})
and such an argument is given in \cite{mim} (literally under the restriction
to premice without superstrongs) from very similar assumptions.\footnote{
In \cite{mim}, the term \emph{premouse} did not demand $(k+1)$-condensation
for all proper segments, as we take it do demand here. But $(k+1)$-condensation for proper 
segments follows from $(0,\om_1+1)$-iterability by \cite{premouse_inheriting}. In \cite[Theorem 4.15]{mim}, it 
is assumed that $\Ult_0(M,\mu)$ is $(0,\om_1+1)$-iterable (in the notation used there, so $\mu=F^M$). But a straightforward calculation shows that this is equivalent to the $(0,\om_1+1)$-iterability of $M$;
moreover, this is a simple instance of \cite[Theorem 9.6]{iter_for_stacks}.} Here we instead deduce the theorem 
from $1$-condensation (\cite[Theorem 5.2]{premouse_inheriting} with $k=0$;
\cite[Theorem 5.2]{premouse_inheriting} is a natural adaptation to Mitchell-Steel indexed mice of Zeman's 
general condensation theorem \cite[9.3.2]{imlc} for Jensen-indexed mice).
\begin{proof}
Suppose not. By the ISC, we may assume that $N$ is type 1.
Note that by weak amenability, $M\in N$ and the failure of the theorem with respect to $N$ is a 
$\Sigma_1^N$ assertion.
Therefore by replacing $N$ with $\cHull_1^N(\emptyset)$, we may assume that 
$N=\Hull_1^N(\emptyset)$.

Let $j=i^N_{F^N}$.
 Because $M|\kappa=N|\kappa$, we have $j(M)|j(\kappa)=j(N|\kappa)$,
 so $j(M)|\OR^N=N^\passive$.
  Let
  \[ k:\Ult(M,F^M)\to j(M) \]
  be the factor. Then $\crit(k)=\xi=(\kappa^+)^M$ and 
$k(\xi)=(\kappa^+)^N$ and $k(\OR^M)=\OR^N$. Let $\pi:M\to N$ be $\pi=k\rest M$.
Then $\pi$ is $\Sigma_0$-elementary; 
this follows readily from the fact that $F^M\rest\xi\sub F^N$.
Also, either $\rho_1^M=\xi$ or $\rho_1^M=\kappa$ (we can't have $\rho_1^M<\kappa$ as $M\in N$ and 
$N|\kappa\sub M$), and $p_1^M\sub\xi$ and $M$ is $\xi$-sound.
We have that $N$ is $1$-solid, in fact $1$-sound, as $N=\Hull_1^N(\emptyset)$.
So by $1$-condensation (\cite[Theorem 5.2]{premouse_inheriting}, with $k=0$),
the desired conclusion holds. For clearly cases 1 and 3(c)
of \cite[Theorem 5.2]{premouse_inheriting} fail. So does its case 3(b),
as otherwise we have some $G\in\es^N$ with $\crit(G)=\kappa$,
and some $Q,k$ such that $M=\Ult_k(Q,G)$. But  $\crit(F^M)=\kappa$, contradicting that because 
$\crit(G)=\kappa$.
The other cases each give the desired 
conclusion.
\end{proof}

\begin{rem}\label{rem:proof_bicephalus_comp}We now give the proof of Theorem \ref{thm:bicephali}.\end{rem}
\begin{tm*} Let $B=(N,F,G)$ be a $(0,\om_1+1)$-iterable bicephalus such that 
$N$ is small.
Then $F=G$.
\end{tm*}

\begin{proof}
If $B$ is not of mixed type (that is, either $F,G$ are both type 2,
or both not of type 2), then one can use the argument of \cite[\S9]{fsit}.
Otherwise, use the argument of \cite{extmax}. By that argument, it suffices to see that in the 
comparison of $B$ with itself, we never use any extender of superstrong type.
So let $(\Tt,\Uu)$ be the comparison and suppose that $E^\Tt_\alpha$ is of superstrong type.
Then $\exit^\Tt_\alpha\sats$``There is a Woodin cardinal''.
Let $\delta=\delta^{\exit^\Tt_\alpha}$,
 and $\beta$ be least such that $\delta<\lh(E^\Tt_\beta)$.
 Then because $N$ is small, there is $Q\ins M^\Tt_\beta$ which is a Q-structure for $\delta$,
 and $Q$ is a premouse, not a bicephalus, and because $\delta=\delta^Q$, 
 $Q$ is the unique iterable $\delta$-sound Q-structure for $\delta$ extending $Q|\delta$.
 But then $Q\ins M^\Uu_\beta$, so $\OR^Q<\lh(E^\Tt_\beta)$,
 so $\delta$ is not Woodin in $\exit^\Tt_\alpha$, contradiction.
\end{proof}

\printindex

\bibliographystyle{plain}
\bibliography{../bibliography/bibliography}

\begin{thebibliography}{10}

\bibitem{barwise}
Jon Barwise.
\newblock {\em Admissible sets and structures}.
\newblock Springer, 1975.

\bibitem{gale_stewart_1953}
David Gale and Frank~M Stewart.
\newblock Infinite games with perfect information.
\newblock In {\em Contributions to the {T}heory of {G}ames (AM-28), Volume II},
  volume~2, pages 2--16. Princeton University Press, 1953.

\bibitem{twms}
{J}ohn {R}.~{S}teel.
\newblock A theorem of {W}oodin on mouse sets.
\newblock In {A}lexander {S}.~{K}echris, {B}enedikt {L}\"owe, and {J}ohn
  {R}.~{S}teel, editors, {\em {O}rdinal {D}efinability and {R}ecursion
  {T}heory: {T}he {C}abal {S}eminar, {V}olume {III}}, pages 243--256.
  {C}ambridge {U}niversity {P}ress, 2016.
\newblock Cambridge Books Online. Preprint available at author's website.

\bibitem{kechris_cdst}
Alexander Kechris.
\newblock {\em Classical descriptive set theory}.
\newblock Springer, 1995.

\bibitem{kondo_1939}
Motokiti Kond\^{o}.
\newblock Sur l'uniformisation des compl\'{e}mentaires analytiques et les
  ensembles projectifs de la seconde classe.
\newblock {\em Japanese journal of mathematics: transactions and abstracts},
  15:197--230, 1939.

\bibitem{conjecture_mouse_order_for_weasels}
Jan Kruschewski and Farmer Schlutzenberg.
\newblock On a conjecture regarding the mouse order for weasels.
\newblock {\em The Journal of Symbolic Logic}, 90(1):364–390, 2025.

\bibitem{lusin_1927}
Nicolas Lusin.
\newblock Sur les ensembles analytiques.
\newblock {\em Fundamenta Mathematicae}, 10(1):1--95, 1927.

\bibitem{lusin_1930a}
Nicolas Lusin.
\newblock Sur le probl\`eme de {M}. {J}. {H}adamard d’uniformisation des
  ensembles.
\newblock {\em Comptes {R}endus {H}ebdomadaires des {S}\'{e}ances de
  {L}'{A}cademie des {S}ciences}, 190(6):349--351, 1930.

\bibitem{lusin_novikoff_1935}
Nicolas Lusin and Pierre Novikoff.
\newblock Choix effectif d'un point dans un compl\'ementaire analytique
  arbitraire donn\'e par un crible.
\newblock {\em Fundamenta Mathematicae}, 25:559--560, 1935.

\bibitem{projdet}
Donald Martin and John~R. Steel.
\newblock A proof of projective determinacy.
\newblock {\em Journal of the American Mathematical Society}, 2(1):71--125,
  January 1989.

\bibitem{martin_borel_det}
Donald~{A}. {M}artin.
\newblock Borel determinacy.
\newblock {\em Annals of {M}athematics}, 102(2):363--371, 1975.

\bibitem{fsit}
William~J. Mitchell and John~R. Steel.
\newblock {\em Fine structure and iteration trees}, volume~3 of {\em Lecture
  Notes in Logic}.
\newblock Springer-Verlag, Berlin, 1994.

\bibitem{mosch}
Yiannis~N. Moschovakis.
\newblock {\em Descriptive set theory}.
\newblock North-Holland, 1980.

\bibitem{mueller_sargsyan_2021}
Sandra M\"uller and Grigor Sargsyan.
\newblock {HOD} in inner models with {W}oodin cardinals.
\newblock {\em The Journal of Symbolic Logic}, 86(3):871–896, 2021.

\bibitem{neeman_talk}
Itay Neeman.
\newblock Mice and scales.
\newblock Slides for talk.

\bibitem{neeman_lightening}
Itay Neeman.
\newblock Mice and scales.
\newblock Unpublished notes, 2000.

\bibitem{novikoff_1931}
Pierre Novikoff.
\newblock Sur les fonctions implicites mesurables {B}.
\newblock {\em Fundamenta Mathematicae}, 17(1):8--25, 1931.

\bibitem{vm2_v2}
Grigor Sargsyan, Ralf Schindler, and Farmer Schlutzenberg.
\newblock Varsovian models {II}.
\newblock arXiv:2110.12051v2.

\bibitem{fs_tame}
E.~Schimmerling and J.~R. Steel.
\newblock Fine structure for tame inner models.
\newblock {\em J. Symbolic Logic}, 61(2):621--639, 1996.

\bibitem{maxcore}
E.~Schimmerling and J.~R. Steel.
\newblock The maximality of the core model.
\newblock {\em Trans. Amer. Math. Soc.}, 351(8):3119--3141, 1999.

\bibitem{sile}
Ralf Schindler and John Steel.
\newblock The self-iterability of {$L[E]$}.
\newblock {\em Journal of Symbolic Logic}, 74(3):751--779, 2009.

\bibitem{deconstructing}
Ralf-Dieter Schindler, John Steel, and Martin Zeman.
\newblock Deconstructing inner model theory.
\newblock {\em J. Symbolic Logic}, 67(2):721--736, 2002.

\bibitem{fsfni_v4}
Farmer Schlutzenberg.
\newblock Fine structure from normal iterability.
\newblock To appear in Journal of Mathematical Logic. Preprint
  arXiv:2011.10037v4.

\bibitem{odle_v2}
Farmer Schlutzenberg.
\newblock Ordinal definability in {$L[\es]$}.
\newblock arXiv:2012.07185v2.

\bibitem{recon_con}
Farmer Schlutzenberg.
\newblock Reconstructing copying and condensation.
\newblock Notes available at \url{\myurl}.

\bibitem{vmom_v2}
Farmer Schlutzenberg.
\newblock Varsovian models {$\om$}.
\newblock ar{X}iv:2212.14878v2.

\bibitem{mim}
Farmer Schlutzenberg.
\newblock {\em Measures in mice}.
\newblock PhD thesis, University of California, Berkeley, 2007.
\newblock arXiv:1301.4702v1.

\bibitem{hsstm}
Farmer Schlutzenberg.
\newblock Homogeneously \protect{S}uslin sets in tame mice.
\newblock {\em Journal of Symbolic Logic}, 77(4):1122--1146, 2012.

\bibitem{premouse_inheriting}
Farmer Schlutzenberg.
\newblock A premouse inheriting strong cardinals from {$V$}.
\newblock {\em Annals of Pure and Applied Logic}, 171(9), 2020.

\bibitem{iter_for_stacks}
Farmer Schlutzenberg.
\newblock Iterability for (transfinite) stacks.
\newblock {\em Journal of Mathematical Logic}, 21(2), 2021.

\bibitem{extmax}
Farmer Schlutzenberg.
\newblock The definability of $\es$ in self-iterable mice.
\newblock {\em Annals of Pure and Applied Logic}, 174(2), 2023.

\bibitem{V=HODX_pub}
Farmer Schlutzenberg.
\newblock The definability of the extender sequence {$\mathbb{E}$} from
  {$\mathbb{E}\upharpoonright\aleph_1$} in {$L[\mathbb{E}]$}.
\newblock {\em The Journal of Symbolic Logic}, 89(2):427--459, 2024.

\bibitem{gaps_as_derived_models_v2}
Farmer Schlutzenberg and John Steel.
\newblock {$\Sigma_1$} gaps as derived models and correctness of mice.
\newblock arXiv:2307.08856v2.

\bibitem{operator_mice_v3}
Farmer Schlutzenberg and N.~Trang.
\newblock The fine structure of operator mice.
\newblock ar{X}iv:1604.00083v3.

\bibitem{sierpinski_1930}
Wac{\l}aw Sierpi\'{n}ski.
\newblock Sur l'uniformisation des ensembles mesurables ({B}).
\newblock {\em Fundamenta Mathematicae}, 16(1):136--139, 1930.

\bibitem{projwell}
John~R. Steel.
\newblock Projectively well-ordered inner models.
\newblock {\em Annals of Pure and Applied Logic}, 74(1):77--104, 1995.

\bibitem{steel_games_and_scales}
John~R. Steel.
\newblock Games and scales.
\newblock In {\em Games, scales, and {S}uslin cardinals. {T}he {C}abal
  {S}eminar. {V}ol. {I}}, volume~31 of {\em Lect. Notes Log.}, pages 3--27.
  Assoc. Symbol. Logic, Chicago, IL, 2008.

\bibitem{ScalesK(R)}
John~R. Steel.
\newblock Scales in {${\rm K}(\mathbb R)$}.
\newblock In {\em Games, scales, and {S}uslin cardinals. {T}he {C}abal
  {S}eminar. {V}ol. {I}}, volume~31 of {\em Lect. Notes Log.}, pages 176--208.
  Assoc. Symbol. Logic, Chicago, IL, 2008.

\bibitem{outline}
John~R. Steel.
\newblock An outline of inner model theory.
\newblock In {\em Handbook of set theory. {V}ols. 1, 2, 3}, pages 1595--1684.
  Springer, Dordrecht, 2010.

\bibitem{imlc}
Martin Zeman.
\newblock {\em Inner models and large cardinals}, volume~5 of {\em de Gruyter
  Series in Logic and its Applications}.
\newblock Walter de Gruyter \& Co., Berlin, 2002.

\end{thebibliography}

\end{document}